\documentclass[reqno]{amsbook}
\usepackage{amsmath,amssymb,amsthm}

\theoremstyle{plain}

\newtheorem{theorem}{Theorem}[chapter]
\newtheorem{lemma}[theorem]{Lemma}
\newtheorem{proposition}[theorem]{Proposition}

\newtheorem{corollary}[theorem]{Corollary}

\numberwithin{section}{chapter}
\numberwithin{equation}{chapter}

\allowdisplaybreaks[1]

\theoremstyle{definition}

\newtheorem{definition}[theorem]{Definition}
\newtheorem{remark}[theorem]{Remark}
\newtheorem{example}[theorem]{Example}

\theoremstyle{remark}

\newcommand{\bin}{{\mu}}

\newcommand{\bA}{{\mathbf A}}
\newcommand{\bX}{{\mathbf X}}

\newcommand{\bL}{{\mathbf L}}

\newcommand{\Bo}{{\mathbf B}}
\newcommand{\Do}{{\mathbf D}}
\newcommand{\bPs}{{\mathbf \Psi}}

\newcommand{\To}{\Psi}

\newcommand{\bcD}{\boldsymbol{\mathcal D}}
\newcommand{\bD}{{\mathbf D}}

\newcommand{\bo}{{\boldsymbol \omega}}
\newcommand{\bgam}{{\boldsymbol \gamma}}

\newcommand{\by}{{\mathbf y}}
\newcommand{\bi}{{\mathbf i}}
\newcommand{\bj}{{\mathbf j}}

\newcommand{\bS}{{\mathbf S}}
\newcommand{\bT}{{\mathbf T}}
\newcommand{\bu}{{\mathbf u}}
\newcommand{\bU}{{\mathbf U}}
\newcommand{\bx}{{\mathbf x}}

\newcommand{\cA}{{\mathcal A}}

\newcommand{\cC}{{\mathcal C}}
\newcommand{\cD}{{\mathcal D}}
\newcommand{\cE}{{\mathcal E}}
\newcommand{\cF}{{\mathcal F}}
\newcommand{\cG}{{\mathcal G}}
\newcommand{\cH}{{\mathcal H}}

\newcommand{\cJ}{{\mathcal J}}
\newcommand{\cK}{{\mathcal K}}
\newcommand{\cL}{{\mathcal L}}
\newcommand{\cM}{{\mathcal M}}
\newcommand{\cN}{{\mathcal N}}
\newcommand{\cO}{{\mathcal O}}

\newcommand{\cQ}{{\mathcal Q}}
\newcommand{\cR}{{\mathcal R}}
\newcommand{\cS}{{\mathcal S}}

\newcommand{\cU}{{\mathcal U}}

\newcommand{\cX}{{\mathcal X}}
\newcommand{\cY}{{\mathcal Y}}
\newcommand{\cZ}{{\mathcal Z}}
\newcommand{\bbD}{{\mathbb D}}

\newcommand{\D}{{\mathbb D}}
\newcommand{\T}{{\mathbb T}}

\newcommand{\B}{{\mathbb B}}

\newcommand{\bcs}[1]{\left(\begin{smallmatrix} #1
                \end{smallmatrix}\right)}
\newcommand{\Ob}{\boldsymbol{{{\mathfrak O}}}}
\newcommand{\Gr}{{\boldsymbol{\mathfrak G}}}

\newcommand{\sbm}[1]{\left[\begin{smallmatrix} #1
		\end{smallmatrix}\right]}

\newcommand{\free}{{\mathbb F}^{+}_{d}}

\newcommand{\bzeta}{\overline{\zeta}}

\newcommand{\bmu}{{\boldsymbol \mu}}
\begin{document}


\title[Function theory, model theory and linear systems]
{Hardy-space function theory, operator model theory, and dissipative linear systems:
the multivariable, free-noncommutative, weighted Bergman-space setting}
\author[J.A. Ball]{Joseph A. Ball}
\address{Department of Mathematics,
Virginia Tech,
Blacksburg, VA 24061-0123, USA}
\email{joball@math.vt.edu}
\author[V. Bolotnikov]{Vladimir Bolotnikov}
\address{Department of Mathematics,
The College of William and Mary,
Williamsburg VA 23187-8795, USA}
\email{vladi@math.wm.edu}
\date{}


\maketitle

\tableofcontents

\chapter{Introduction} \label{S:Intro}
\section{Overview}  \label{S:Overview}
\setcounter{equation}{0}
For $\cX$ and $\cY$ any pair of Hilbert spaces, we use the notation
$\cL(\cX, \cY)$ to denote the space of bounded, linear operators
from $\cX$ to $\cY$, shortening the notation $\cL(\cX,\cX)$ to $\cL(\cX)$.  
We start with the classical discrete-time linear system
\begin{equation}
\Sigma(\bf U): \quad\left\{ \begin{array}{rcl}
x(k+1)&= & Ax(k)+Bu(k)\\
y(k)&= & Cx(k)+Du(k)\end{array} \right.
\label{1.1pre}
\end{equation}
with $x(k)$ taking values in the {\em state space} $\cX$, $u(k)$ taking
values in the {\em input-space} $\cU$ and $y(k)$ taking values in the
{\em output-space} $\cY$, where $\cU$, $\cY$ and $\cX$ are given Hilbert spaces
and where the {\em connection matrix} (sometimes also called {\em colligation matrix} or {\em system matrix}
of the system
$$
{\bf U}=\begin{bmatrix} A & B \\ C & D\end{bmatrix}: \, \begin{bmatrix} \cX \\ 
\cU\end{bmatrix}\to \begin{bmatrix} \cX \\ \cY\end{bmatrix}
$$
is a given bounded linear operator. If we let the system evolve on the
nonnegative  integers $n \in {\mathbb Z}_{+}$, then the whole
trajectory $\{u(n), x(n), y(n)\}_{n \in {\mathbb Z}_{+}}$ is
determined from the input signal $\{u(n)\}_{n \in {\mathbb Z}_{+}}$
and the initial state $x(0)=x$ according to the formulas
\begin{align}
x(k) & = A^{k}x + \sum_{j=0}^{k-1} A^{k-1-j} B u(j), \notag \\
y(k) & = C A^{k} x + \sum_{j=0}^{k-1} C A^{k-1-j} B u(k) + D u(k).
\label{1.2pre}
\end{align}
Application of the $Z$-transform
$$\{f(k)\}_{k \in {\mathbb Z}_{+}} \mapsto \widehat f(\lambda) =
{\displaystyle \sum_{k=0}^{\infty} f(k) \lambda^{k}}
$$
to the system equations \eqref{1.1pre} converts the
expressions \eqref{1.2pre} to the so-called frequency-domain formulas   
\begin{align}
           \widehat x(\lambda) & = (I - \lambda A)^{-1} x + \lambda
           (I - \lambda A)^{-1} B \widehat
           u(\lambda),   \notag \\
           \widehat y(\lambda) & = C(I - \lambda A)^{-1} x + [D + \lambda
           C(I - \lambda A)^{-1}B]
           \widehat u(\lambda) 
            = {\mathcal O}_{C,A} x + \Theta_{\bU}(\lambda) \widehat
           u(\lambda), \notag  
\end{align}
where
\begin{equation}
{\cO}_{C,A} \colon \; x \mapsto \sum_{k=0}^{\infty}(CA^k x)\, \lambda^k=C(I-\lambda 
A)^{-1} x
\label{1.4pre}
\end{equation}
is the observability operator and where
\begin{equation}
\Theta_{\bU}(\lambda) = D + \lambda C (I - \lambda A)^{-1} B
\label{1.5pre}
\end{equation}
is the {\em transfer function} of the system $\Sigma$ given by \eqref{1.1pre}.  
In particular, if the input signal $\{u(n)\}_{n \in {\mathbb Z}_{+}}$
is taken to be zero, the resulting output $\{y(n)\}_{n \in {\mathbb
Z}_{+}}$ is given by $y = {\mathcal O}_{C,A} x(0)$.
If $\cO_{C,A}$ is injective, i.e., if $(C,A)$ satisfies the so-called 
{\em observability condition}
\begin{equation}   \label{observable}
\bigcap_{k=0}^{\infty} {\rm Ker}\, C A^{k} = \{0\}, 
\end{equation}
we say that the output pair $(C,A)$ is {\em observable}.  In case ${\mathcal O}_{C,A}$ 
is bounded as an operator from $\cX$ into
the standard vector-valued Hardy space of the unit disk
$$
    H^2_{\cY} = \big\{ f(\lambda) = \sum_{k\ge 0} f_{k} \lambda^{k} \colon
    \sum_{k\ge 0} \|f_{k}\|_{\cY}^{2} < \infty \big\},
$$
we say that the pair $(C,A)$ is {\em output-stable}.

\smallskip

The case where the operator connection matrix ${\bf U}$ is {\em isometric}, or more generally just {\em contractive}, is of special 
interest. In system-theoretic terms the isometric property of ${\bf U}$ has the interpretation that the system 
$\Sigma({\bf U})$ is {\em conservative} in the sense that  the energy stored by the state at 
time $k$ ($\| x(k+1)\|^{2} - \| x(k) \|^{2}$) is exactly compensated by the
net energy put into the system from the outside environment
($\|u(k)\|^{2} - \| y(k)\|^{2}$).    In case ${\bf U}$ is contractive the system $\Sigma({\bf U})$ is said to be {\em dissipative}  in the sense
that the net energy stored by the state at time $k$ $\|x(k+1)\|^2 - \| x(k)\|^2$ is no more than the net
energy put into the system from the outside environment ($\|u(k)\|^{2} - \| y(k)\|^{2}$) at time $k$.
In case the system is dissipative (i.e., $\| {\bf U} \| \le 1$), the transfer function $\Theta_{\bf U}$ is in the 
 {\em Schur class} $\cS(\cU,\cY)$ (i.e., analytic on the open unit disk $\D$ and such that $\Theta(z)$ is a contraction in
$\cL(\cU,\cY)$ for every $z\in\D$), and moreover the observability operator $\cO_{C, A} \colon \cX \to H^2_\cY$ is contractive.
Conversely, if $\Theta$ is in the Schur class, then $\Theta$ has a realization as $\Theta = \Theta_{\bf U}$
as in \eqref{1.1pre} with $\Sigma({\bf U})$ dissipative (in fact, even conservative).  

\smallskip

If  $\bU$ is isometric 
and in addition the state space operator $A$ is {\em strongly stable} in
the sense that $\|A^n x \| \to 0$ as $n \to \infty$ for each $x \in\cX$, then the observability operator
is a partial isometry (even an isometry in case $(C,A)$ is observable) and the transfer function $\Theta_\bU$ is {\em inner}
(the boundary values $\Theta_\bU(\zeta)$ existing as strong radial limits from inside ${\mathbb D}$ for almost every $\zeta$ on the unit circle $\T$ are isometric operators 
from $\cU$ to $\cY$), and conversely:  {\em any inner function $\Theta$ arises in this way as $\Theta = \Theta_\bU$ with $\bU = \sbm{ A & B \\ C & D}$
isometric with $A$ strongly stable.}

\smallskip
\noindent
We say that a subspace 
$\cM \subset H^2_\cY$ is {\em shift-invariant} if $f \in \cM \Rightarrow S_\cY f \in \cM$ where $S_\cY$ is 
the shift operator given as the coordinate multiplication operator on $H^2_\cY$ 
\begin{equation}   \label{shift}
   S_\cY = M_\lambda \colon f(\lambda) \mapsto \lambda f(\lambda).
\end{equation}
Note that if $\Theta$ is inner, then $\cM : = M_\Theta H^2_\cU = \Theta \cdot H^2_\cU$ is a shift-invariant subspace for $S_\cY$;
the content of the Beurling-Lax theorem is that conversely, any such invariant subspace can be represented in this way.
Similarly we say that the subspace $\cN \subset H^2_\cY$ is {\em backward shift-invariant} if
$f \in \cN \Rightarrow S_\cY^* f \in \cN$ where the backward-shift operator $S_\cY^*$, the Hilbert-space 
adjoint of the forward-shift operator $S_\cY$, works out to be
$$
 S_\cY^* \colon f(\lambda) \mapsto [f(\lambda) - f(0)]/\lambda.
$$
The computation 
$$
S_\cY^* \colon C (I - \lambda A)^{-1} x \mapsto  z^{-1} [ C (I - \lambda A)^{-1} - C] x = C(I - \lambda A)^{-1} A x
$$
shows that, for any output-stable pair $(C, A)$, $\operatorname{Ran} \cO_{C, A}$ is $S_\cY^*$-invariant.  Conversely, 
 {\em if $\cM^\perp \subset H^2_\cY$ is $S_\cY^*$-invariant, then there is an output pair $(C, A)$} (with
$C^*C = I - A^*A$) {\em so that $\cM^\perp = \operatorname{Ran} \cO_{C, A}$} (see \cite{Agler1982} for additional background and generalizations).
Moreover the Sz.-Nagy--Foias characteristic function for a completely nonunitary contraction operator $T$ on a Hilbert space $\cH$
\begin{equation}  \label{NFcharfunc}
\Theta_T(\lambda) = [ -T + \lambda D_{T^*} (I - \lambda T^*)^{-1} D_T]|_{\cD_T} \colon \cD_T \to \cD_{T^*},
\end{equation}
where we use the standard notation
$$
  D_T = (I - T^* T)^{\frac{1}{2}}, \quad D_{T^*}= (I - T T^*)^{\frac{1}{2}}, \quad
  \cD_T = \overline{\operatorname{Ran}}  \, D_T, \quad   \cD_{T^*} = \overline{\operatorname{Ran}} \, D_{T^*},
$$ 
amounts to the transfer function $\Theta_\bU$ associated with the unitary connection matrix
$$
    \bU = \begin{bmatrix} T^* & D_T \\ D_{T^*} & -T \end{bmatrix} \colon \begin{bmatrix} \cH \\ \cD_{T} \end{bmatrix}
    \to \begin{bmatrix} \cH \\ \cD_{T^*} \end{bmatrix}.
$$
In summary 
the following themes have developed and matured over the last several decades connecting
vectorial Hardy-space function theory, theory of Hilbert-space contraction operators, and conservative discrete-time 
linear systems:
\begin{enumerate}
\item A backward-shift invariant subspace of $H^2_\cY$ arises as the range of some observability operator.

\item  A forward-shift invariant subspaces of $H^2_\cY$ has Beurling-Lax inner-function representations.

\item  In the case of a conservative linear system with strongly stable state operator $A$ (i.e., $\bU = \sbm{ A & B \\ C & D}$ is isometric
and $\| A^n x \| \to 0$ as $n \to \infty$ for each $x \in \cX$), the observability
operator $\cO_{C,A} \colon \cX \to H^2_\cY$  and the transfer-function multiplier operator $M_{\Theta_\bU}$ are isometric, and one has 
the orthogonal decomposition
$$
  H^2_\cY = \operatorname{Ran} \cO_{C,A} \oplus \operatorname{Ran} M_{\Theta_\bU},
$$
and conversely:  if $\cM \subset H^2_\cY$ is $S_\cY$-invariant (and hence $\cM^\perp \subset H^2_\cY$ is $S_\cY^*$-invariant),
then there is an isometric $\bU = \sbm{ A & B \\ C & D}$ with $A$ strongly stable such that $\cM^\perp = \operatorname{Ran} \cO_{C, A}$
and $\cM = \Theta_\bU \cdot H^2_\cU$.

\item Inner (and more generally, contractive) multipliers $M_\Theta \colon H^2_\cU \to H^2_\cY$ arise as the
Sz.-Nagy--Foias characteristic function for some completely nonunitary Hilbert-space contraction operator $T$
which in turn induces a canonical functional model for the operator $T$.
\end{enumerate}

Much work has been done to extend this set of ideas, particularly themes \#2 and \#4 above (the operator-model theory aspects without the 
system-theoretic connections) to more general settings, e.g.,

(i)  to Bergman-spaces and hypercontraction operators; see Agler \cite{aglerhyper}, M\"uller \cite{muller}, 
M\"uller-Vasilescu \cite{MV}, Hedenmalm-Korenblum-Zhu \cite{HKZ2000}, Duren-Schuster \cite{DS2004}), 

(ii) to the Drury-Arveson space and commutative row-contractive operator tuples; see Bhattacharyya-Eschmeier-Sarkar \cite{BES, BES2}, 
Bhattacha\-ryya-Sarkar \cite{BS}, Ball-Bolotnikov \cite{BB11}, 

(iii) to more general domains in ${\mathbb C}^d$ than the ball and associated
more general commutative operator tuples; see Athavale \cite{Athavale},  Curto-Vasilescu \cite{CV1993, CV1995}, Timotin \cite{Timotin},
Pott \cite{Pott}, Ambrozie-Engli\u{s}-M\"uller \cite{AEM}, Arazy-Engli\u{s} \cite{ArazyEnglis}.  

(iv) to the full Fock space and freely noncommutative row-contractive operator tuples, possibly also constrained to lie in a prescribed noncommutative
operator variety; see Bunce \cite{Bunce}, Frazho \cite{Frazho}, 
Popescu \cite{PopescuNF0, PopescuNF1, PopescuNF2, PopescuBL,  Popescu-var1, Popescu-var2},  

(v) to a more general formalism of representations of 
certain operator algebras based on tensor-algebra constructions; see Muhly-Solel \cite{MS99, MS08, MS05, MS16}), and

(vi) to noncommutative hypercontractive operator tuples modeled on noncommutative varieties (see Popescu \cite{PopescuJFA2013,
PopescuTAMS2016, PopescuMANN}) as well as a weighted version of the tensor-algebra context 
(see Muhly-Solel \cite{MS16}).

\smallskip

  Identification of a characteristic function 
defined by a formula of the Sz.-Nagy--Foias
type \eqref{NFcharfunc} (the main thrust of theme \#4 above) can be found (i) for the Bergman space setting only recently first in the work of 
Olofsson \cite{oljfa, olieot, olaa} and then followed up by the authors \cite{BBIEOT, BBSzeged} and Eschmeier \cite{Esch2018}, (ii) for the Drury-Arveson 
space setting earlier in the work of 
Bhattacharyya-et-al \cite{BES, BES2, BS}, (iv) for the full Fock space in the work of Popescu \cite{PopescuNF2}, Ball-Bolotnikov-Fang
 \cite{BBF1, BBF3} and Ball-Vinnikov \cite{Cuntz-scat}, for the tensor-algebra context in Muhly-Solel \cite{MS05}.

\smallskip

Let us note at this stage that theme \#4 simplifies considerably in case the completely nonunitary contraction operator $T$ is in the 
Sz.-Nagy--Foias class $C_{\cdot 0}$ (i.e., $T^*$ is {\em strongly stable} or equivalently, the contraction operator $T$ is {\em pure} in 
the terminology of some authors).  Indeed, $T \in C_{\cdot 0}$ is equivalent to $\Theta_T$ being inner, i.e., we are in the setting of 
item \#3 in the list of themes above:
$ H^2_{\cD_{T^*}} = \cN \oplus \cM$ where $\cN = \operatorname{Ran} \cO_{D_{T^*}, T^*}$ is backward-shift invariant and
$\cM = \Theta_T \cdot H^2_{\cD_T}$ is forward-shift invariant.  It then turns out that $T$ is unitarily equivalent to its Sz.-Nagy--Foias model 
operator $\bT$ given simply by $\bT : = P_\cN M_\lambda|_\cN$ where $M_\lambda \colon f(\lambda) \mapsto \lambda f(\lambda)$
is the coordinate multiplication operator on $H^2_{\cD_{T^*}}$.

\smallskip

In the multivariable context, with only a couple of exceptions (see \cite{BB11, Cuntz-scat}) the full completely nonunitary Sz.-Nagy--Foias model
theory has proved to be elusive.  A common compromise choice for moving past the $C_{\cdot 0}$ (or pure) case is to invoke an analogue of
the assumption that $T$ is {\em completely non-coisometric} (see \cite{BES2, PopescuNF2, MS05}).  In this case the observability operator
$\cO_{D_{T^*}, T^*}$ is no longer isometric but is still injective (i.e. the output pair $(D_{T^*}, T^*)$ is observable) and one can equip
$\operatorname{Ran} \cO_{D_{T^*}, T^*}$ with the lifted norm from $\cX$ (rather than the induced norm from containment in the ambient
Hardy space $H^2_{\cD_{T^*}}$) to view $T^*$ as unitarily equivalent to the restriction of the backward shift $S_{\cD_{T^*}}$
to the subspace $\operatorname{Ran} \cO_{D_{T^*}, T^*}$ which is now only contractively included in $H^2_{\cD_{T^*}}$.
Many authors construct an additional defect space $\overline{\Delta_{\Theta_T} H^2_{\cD_{T^*}}}$ so that $\operatorname{Ran} \cO_{D_{T^*}, T^*}$
(with lifted norm from $\cX$) can be identified isometrically with a subspace of the two-component space 
$$
\operatorname{Ran} \cO_{D_{T^*}, T^* } \, \oplus \, \overline{\Delta_{\Theta_T} H^2_{\cD_{T^*}  }} 
$$
(with the first component now taken with $H^2_{\cD_{T^*}}$-norm), but from our point of view this
is not necessary.  Then there is a Cholesky-factorization algorithm for constructing an operator $\sbm{ B \\ D } \colon \cU \to 
\sbm{ \cX \\ \cD_{T^*}}$ so that $\bU  = \sbm{ T^* & B \\ D_{T^*} & D} \colon \sbm{ \cH \\ \cD_T} \to \sbm{ \cH \\ \cD_{T^*}}$
is unitary (for the classical case one simply takes $\cU = \cD_T$, $B = D_T \colon \cD_T \to \cH$ and $D = -T|_{\cD_T}$),
and we can define the transfer function $\Theta_\bU$ associated with connection matrix $\bU$ to be the characteristic function of $T$.
Then the backward shift $S_{\cD_{T^*}}^*$ restricted to the subspace $\cN: = \operatorname{Ran} \cO_{D_{T^*}, T^*}$ considered with
lifted norm is the de Branges-Rovnyak model for the original contraction operator $T$. The link with the characteristic function $\Theta_T$
is due to the validity of the formula
$$
 D_{T^*} (I - \lambda T^*)^{-1} ( I - \overline{\zeta} T)^{-1} D_{T^*} + \frac{\Theta_T(\lambda) \Theta_T(\zeta)^*}{1 - \lambda \overline{\zeta}}
 = \frac{ I_{\cD_{T^*}} }{1 - \lambda \overline{\zeta}}
$$
(where $k(\lambda, \zeta) = \frac{1}{1 - \lambda \overline{\zeta}}$ is the {\em reproducing kernel} for $H^2$)
which can be interpreted as saying that the subspace $\cM: = \Theta_T \cdot H^2_{\cU} \subset H^2_{\cD_{T^*}}$ with lifted norm is 
the {\em Brangesian complement}
 of  the backward-shift invariant subspace $\cN = \operatorname{Ran} \cO_{D_{T^*}, T^*}$ (also contractively included in $H^2_{\cD_{T^*}}$);
see Section \ref{S:Brangesian} below.  The idea of the generalized Beurling-Lax Theorem whereby one represents contractively included
 forward-shift invariant subspaces as being of the form $\Theta \cdot H^2_\cU$ with $\Theta$ being in the Schur class rather than inner
 appears already in the work of de Branges-Rovnyak \cite{dBR1} and indeed the de Branges-Rovnyak model space $\cH(\Theta)$
 for the noninner case can be considered as a contractively included backward-shift invariant subspace $\cN$ as above.  
 
\smallskip

 Here we focus on a general setting of model spaces, forward and backward shift-operator tuples, and their joint invariant subspaces
 which simultaneously contain as special cases the Bergman setting (i) and the full Fock space setting (iv) mentioned above.  Before plunging
 into the most general setting, we next sketch how the system theory setup \eqref{1.1pre}, \eqref{1.2pre}, \eqref{1.4pre}, \eqref{1.5pre},
 \eqref{observable} adapts to these motivating special cases.

\section{Standard weighted Bergman spaces} \label{S:standard}
\noindent
For a Hilbert space $\cY$ and an integer $n\ge 1$, we denote by $\cA_{n,\cY}$ 
the Hilbert space of $\cY$-valued functions $f$  analytic
in the open unit disk $\bbD$ and with finite norm $\|f\|_{A_{n,\cY}}$:
$$
\cA_{n,\cY}=\big\{f(\lambda)={\displaystyle\sum_{j\ge 0}f_j \lambda^j} \colon \; 
\|f\|^2_{\cA_{n,\cY}}:={\displaystyle \sum_{j\ge 0}\bin_{n,j} \cdot
\|f_j\|_{\cY}^2}<\infty\big\},
$$
where the weights $\mu_{n,j}$'s are defined by
\begin{equation}
\bin_{n,j}:=\frac{1}{\binom{j+n-1}{j}} =
\frac{j!(n-1)!}{(j+n-1)!}.
\label{1.6pre} 
\end{equation}
The space $\cA_{n,\cY}$ can be alternatively characterized as 
the reproducing kernel Hilbert space with reproducing kernel $k_n(\lambda,\zeta)I_\cY$ where 
\begin{equation}   
k_n(\lambda,\zeta)=(1-\lambda\overline{\zeta})^{-n}.
\label{1.7pre}
\end{equation}
We introduce the function $R_n$ and its shifted counterparts $R_{n,k}$ by the formulas
\begin{equation}  
R_n(\lambda):=(1-\lambda)^{-n}=\sum_{j=0}^\infty \mu_{n,j}^{-1}\lambda ^j\quad\mbox{and}\quad 
R_{n,k}(\lambda)=\sum_{j=0}^\infty \mu_{n,j+k}^{-1}\lambda^j,
\label{1.8pre}
\end{equation}
so that $R_{n,0}=R_n$ and $k_n(\lambda,\zeta)=R_n(\lambda\overline{\zeta})$.
Observe that functions \eqref{1.8pre} satisfy the following relations:
\begin{align}   
 R_{n,k}(\lambda) &= \bcs{n+k-1 \\ k} + \lambda R_{n,k+1}(\lambda),\label{1.9pre} \\
 R_{n,k}(\lambda) & = \sum_{\ell=1}^{n} \bcs{\ell + k - 2 \\ \ell - 1} 
 R_{n-\ell + 1}(\lambda)\quad \text{for}\quad k \ge 1. \label{1.10pre}
\end{align}
Identity \eqref{1.9pre} follows directly from  definition 
\eqref{1.8pre} while the proof of \eqref{1.10pre}
can be found in \cite[Section 2]{BBIEOT}.
We also record that for any operator
$A\in\cL(\cX)$ with  spectral radius $\rho_A$, the operator-valued functions
\begin{equation}
 \label{1.12pre}
R_{n,k}(\lambda A)=\sum_{j=0}^\infty\mu_{n,k+j}^{-1}A^j\lambda^j
\end{equation}
are analytic on the disk $\{\lambda: \; |\lambda|<1/\rho_A\}$ for any $k\in{\mathbb 
Z}_+$.

\smallskip

In \cite{BBIEOT}, we considered the following discrete-time time-varying 
linear system:
\begin{equation}  \label{1.13pre}
 \Sigma_{n}\left( \left\{\left[ \begin{smallmatrix} A & B_{j} \\ C & D_{j}
 \end{smallmatrix}\right]  \right\}_{j\in{\mathbb Z}_+} \right) \colon    \left\{
\begin{array}{rcl}
 x(j+1) & = & \frac{j+n}{j+1}\cdot Ax(j)+\binom{j+n}{j+1}\cdot B_ju(j) \\ [3mm] 
 y(j) & = & C x(j)+\binom{j+n-1}{j}\cdot D_ju(j)
 \end{array} \right.
 \end{equation}
where
$A \in \cL(\cX), \; C \in \cL(\cX, \cY), \;
B_k \in \cL(\cU_k, \cX), \;  D_k \in \cL(\cU_k, \cY)$
are given bounded linear operators acting between given Hilbert spaces
$\cX$, $\cY$ and $\cU_k$ ($k\ge 0$).  We note that the case where $n=1$ and where the
operators $B_{k} =  B$ and $D_{k} = D$ are taken independent of the time parameter
 $k\in{\mathbb Z}_+$ reduces to the classical time-invariant case 
\eqref{1.1pre}. If we let the system \eqref{1.13pre} evolve on ${\mathbb Z}_{+}$, 
then the whole 
trajectory $\{u(j), x(j), y(j)\}_{j \in {\mathbb Z}_{+}}$ is
determined from the input signal $\{u(j)\}_{j \in {\mathbb Z}_{+}}$
and the initial state $x(0)$ according to the formulas
\begin{align}
x(j) &= \mu_{n,j}^{-1}\cdot \big(A^{j} x(0)+\sum_{\ell=0}^{j-1}A^{j-\ell-1} B_\ell
u(\ell)\big),
\label{1.14pre}\\
y(j) &= \mu_{n,j}^{-1}\cdot \big(CA^{j} x(0)+\sum_{\ell=0}^{j-1}CA^{j-\ell-1} B_\ell
u(\ell) +D_ju(j)\big).\label{1.15pre}
  \end{align}
Formula \eqref{1.14pre} is established by simple induction arguments, while \eqref{1.15pre} is
obtained by substituting \eqref{1.14pre} into the second equation in \eqref{1.13pre}.

\smallskip

To write the $Z$-transformed version of the system-trajectory formula \eqref{1.14pre}, 
we multiply both sides of \eqref{1.14pre} by $\lambda^{j}$ and sum over $j\ge 0$ to get, on account 
of \eqref{1.12pre},
\begin{align}
\widehat x(\lambda)= \sum_{j=0}^{\infty} x(j) \lambda^{j}
&=\bigg(\sum_{j=0}^\infty \mu_{n,j}^{-1}A^j\lambda^j\bigg)x(0)+\sum_{k=1}^{\infty}
\bigg(\sum_{j=k}^\infty \mu_{n,j}^{-1} A^{j-k}\lambda^j\bigg)B_{k-1}u(k-1)\notag\\
&=(I-\lambda A)^{-n}x(0)+\sum_{k=1}^{\infty}\lambda^k\bigg(\sum_{j=0}^\infty
\mu_{n,j+k}^{-1}A^j\lambda^j\bigg)B_{k-1}u(k-1)\notag\\
&=(I-\lambda A)^{-n}x(0)+\sum_{k=0}^{\infty}\lambda^{k+1}R_{n,k+1}(\lambda A)B_ku(k).\notag
\end{align}  
The same procedure applied to \eqref{1.15pre} gives
\begin{align}  
\widehat y(\lambda)&=C(I-\lambda A)^{-n}x(0)+\sum_{k=0}^{\infty}\lambda^k \left(\mu_{n,k}^{-1}D_k+
\lambda CR_{n,k+1}(\lambda A)B_k\right)u(k)\notag\\
&=\cO_{n,C,A}x(0)+\sum_{k=0}^{\infty}\lambda^k \Theta_{n,k}(\lambda)u(k), \label{1.16pre}
\end{align}  
where
 \begin{equation}
\cO_{n,C,A} \colon \; x \mapsto
\sum_{j=0}^\infty\left(\mu_{n,j}^{-1}CA^jx\right) \, \lambda^j=C(I-\lambda A)^{-n} x
\label{1.17pre}
\end{equation}
is the $n$-observability operator and where
$$
\Theta_{n,k}(\lambda)=\mu_{n,k}^{-1}D_k+
\lambda CR_{n,k+1}(\lambda A)B_k \qquad (k=0,1,\ldots)
$$
is the family of transfer functions. 

\smallskip
 
Note that observability of the output pair $(C,A)$ in the classical sense \eqref{observable} 
is also equivalent to the injectivity of the $n$-observability operator 
$\cO_{n,C,A}$ \eqref{1.17pre}.  Following \cite{BBIEOT}, we say that 
the output pair $(C,A)$ is {\em $n$-output stable} if $\cO_{n, C, A}$ 
is bounded as an operator from $\cX$ into $\cA_{n,\cY}$.

\smallskip

We note next that the transfer
function $\Theta_{n,k}(z)$ encodes the result of a pulse input-vector
$u$ being applied at time $j = k$:
$$
\widehat y(\lambda) = \Theta_{n,k}(\lambda) \cdot \lambda^{k} u \quad\text{if}\quad x(0) = 0
\quad\text{and}\quad u(j) = \delta_{j,k} u
$$
(where $\delta_{jk}$ stands for the Kronecker symbol).
In fact the functions $\Theta_{n,k}(\lambda)$ could have been derived in
this way and then one could arrive at input-output relation
\eqref{1.16pre} via superposition of all these time-$k$ impulse
responses.  There is a notion of {\em conservative} for a system of 
the form \eqref{1.13pre}  involving the connection matrix $\sbm{A 
& B_{j} \\ C & D_{j}}$ being unitary with respect to an appropriate 
choice of weights (see formulas (6.7) and (6.21) in \cite{BBIEOT}).
When these metric constraints are satisfied, the associated 
transfer-function family $\{\Theta_{n,k}\}$ serves as a representer 
of a shift-invariant subspace in the weighted Bergman space while 
the image space of an observability operator $\cO_{n,C,A}$ is the 
model for a backward shift invariant subspace in $\cA_{n,\cY}$ (see 
\cite{BBIEOT, BBSzeged}).

\smallskip

\section{The Fock space setting}   \label{S:Fock0}
The classical results on the system \eqref{1.1pre} admit nice extensions to a number of 
multivariable settings,
both commutative and noncommutative. In this section we recall the case where the Hardy 
space $H^2_\cY$
is replaced by the Fock space $H^2_\cY(\free)$.

\smallskip

To define the Fock space, we let $\free$ denote the unital free
semigroup (i.e., monoid)  generated by the set of $d$ letters $\{1, \dots, 
d\}$.
Elements of $\free$ are words of the form $i_{N}
\cdots i_{1}$ where $i_{\ell} \in \{1, \dots, d\}$ for each $\ell\in\{1,
\dots, N\}$ with multiplication given by concatenation. The unit element
of $\free$ is the empty word denoted by $\emptyset$.
For $\alpha = i_{N} i_{N-1} \cdots i_{1} \in {\mathcal
F}_{d}$, we let $|\alpha|$ denote the number $N$ of letters in $\alpha$ and we let
$\alpha^{\top} : = i_{1}  \cdots i_{N-1} i_{N}$ denote the {\em transpose}
of $\alpha$. We let $z = (z_{1}, \dots,z_{d})$ to be a
collection of $d$ formal noncommuting variables and let $\cY\langle\langle
z\rangle\rangle$ denote the set of noncommutative formal power series
$\sum_{\alpha \in \free} f_\alpha z^\alpha$ where $f_\alpha \in \cY$ and where
\begin{equation}
\label{1.19pre}
z^{\alpha} =z_{i_{N}}z_{i_{N-1}} \cdots
z_{i_{1}}\quad\mbox{if}\quad \alpha= i_{N}i_{N-1} \cdots i_{1}.
\end{equation}
The Fock space $H^2_{\cY}(\free)$ is then defined as
\begin{equation}   \label{Fock}
H_{\cY}^{2}(\free) = \bigg\{\sum_{\alpha \in \free} 
f_{\alpha}z^{\alpha}
\in \cY\langle\langle z\rangle\rangle\colon \;
\sum_{\alpha \in \free} \|f_{\alpha}\|_{\cY}^{2} <\infty \bigg\}.
\end{equation}
The Fock-space counterpart of \eqref{1.1pre} is the system
\begin{equation}
\Sigma(\bf U): \quad\left\{ \begin{array}{ccc}
x(1 \alpha) &= & A_{1} x(\alpha)+ B_1 u(\alpha) \\
\vdots & \vdots  & \vdots  \\
x(d \alpha) & = & A_d x(\alpha)+ B_d u(\alpha)  \\
y(\alpha) &  = & Cx(\alpha)+Du(\alpha)\end{array}\right.
\label{1.20pre}
\end{equation}
which evolves along the free semigroup $\free$, and, for
each $\alpha \in \free$, the state vector $x(\alpha)$, input signal 
$u(\alpha)$
and output signal $y(\alpha)$ take values in the {\em state space} $\cX$,
{\em input space} $\cU$ and {\em output space} $\cY$. The {\em
connection matrix} ${\bf U}$ has the form
\begin{equation}
   {\bf U} = \begin{bmatrix} A & B \\ C & D \end{bmatrix} = \begin{bmatrix}
   A_{1} & B_{1} \\ \vdots & \vdots \\ A_{d} & B_{d} \\ C & D
\end{bmatrix} \colon \begin{bmatrix} \cX \\ \cU \end{bmatrix} \to
\begin{bmatrix} \cX \\ \vdots \\ \cX  \\ \cY \end{bmatrix}.
\label{1.20}
\end{equation}
Such systems were introduced in \cite{Cuntz-scat} and with further elaboration
in \cite{BGM1} and \cite{BGM2}; following \cite{BGM1} we call this
type of system a {\em noncommutative Fornasini-Marchesini linear system}.

\smallskip

We extend the noncommutative functional calculus \eqref{1.19pre}
from noncommuting indeterminates $z = (z_{1}, \dots, z_{d})$ to a
$d$-tuple of operators ${\mathbf A} = (A_{1}, \dots, A_{d})$ by letting 
\begin{equation}         {\mathbf A}^{\alpha} := A_{i_{N}}A_{i_{N-1}} \cdots A_{i_{1}}\quad
\text{if}\quad   \alpha = i_{N} i_{N-1} \cdots i_{1} \in \free,
\label{1.21pre}
\end{equation}
where the multiplication is now operator composition. Letting
\begin{equation}
Z(z) = \begin{bmatrix} z_{1} & \cdots & z_{d}
\end{bmatrix}\otimes I_{\cX}, \quad
A=\begin{bmatrix} A_{1} \\ \vdots \\ A_{d} \end{bmatrix},
\quad B = \begin{bmatrix} B_{1} \\ \vdots \\ B_{d} \end{bmatrix},
\label{1.23pre}  
\end{equation}
we next observe that
\begin{equation}
(Z(z)A)^j={\displaystyle\bigg(\sum_{i=1}^dz_iA_i\bigg)^j}={\displaystyle\sum_{\alpha 
\in\free:
\, |\alpha|=j}{\bf A}^\alpha z^\alpha}\quad\mbox{for all}\quad j\ge 0
\label{1.24pre}
\end{equation}
and therefore,
$$
(I-Z(z)A)^{-1}=\sum_{j=0}^\infty (Z(z)A)^j=
\sum_{j=0}^\infty \sum_{\alpha \in\free: \, 
|\alpha|=j}{\bf A}^\alpha z^\alpha
=\sum_{\alpha \in\free}{\bf A}^\alpha z^\alpha.
$$
Application of the formal noncommutative $Z$-transform
\begin{equation}  \label{1.25pre}
\{f_{\alpha} \}_{\alpha \in \free} \mapsto \widehat f(z) = 
\sum_{\alpha \in \free} f_{\alpha} z^{\alpha}
\end{equation}
to the system \eqref{1.20pre} then gives
\begin{align}
\widehat x(z) &= (I - Z(z) A)^{-1} x(\emptyset) +(I - Z(z) A)^{-1}Z(z)B \widehat 
u(z),\notag\\
\widehat y(z) &= \cO_{C,{\bf A}} x(\emptyset) +\Theta_{\bf U} (z) \widehat 
u(z),\label{1.26pre}
\end{align}
where
\begin{equation}
\cO_{C,{\bf A}}: \, x\mapsto C(I - Z(z) A)^{-1} x=\sum_{\alpha 
\in\free}(C{\bf A}^\alpha x)z^\alpha
\label{1.27pre}
\end{equation}
is the observability operator of $(C,{\bf A})$ and where $\Theta_{\bf 
U}(z)\in\cL(\cU,\cY)\langle\langle z\rangle\rangle$ is given by
\begin{equation}
\Theta_{\bf U}(z) = D + C (I - Z(z) A)^{-1} Z(z) B=D+\sum_{\alpha \in\free}\sum_{j=1}^d 
C{\bf A}^\alpha B_j z^\alpha z_j.
 \label{1.28pre}
  \end{equation}
  Thus the initial state $x = x_{\emptyset}$ is uniquely determined 
  by the output signal $\widehat y(z)$ when the input signal $\widehat 
  u(z)$ is taken to be zero exactly when $\cO_{C,\bA}$ is injective; 
  when this is the case, we say that the  output pair $(C, \bA)$ is 
  {\em observable}.
The pair $(C,{\bf A})$ is called {\em output-stable} if $\cO_{C,{\bf A}}$ 
is bounded as an operator  from $\cX$ into $H^{2}_{\cY}(\free)$, 
and {\em exactly observable} if $\cO_{C, \bA}$ is bounded and bounded 
below.  As in the single-variable case, the connection matrix \eqref{1.20} being unitary 
corresponds to a notion of energy conservation; for details on this 
we refer to \cite{Cuntz-scat}. The $d$-tuple ${\mathbf A} =
(A_{1}, \dots, A_{d})$ is called {\em strongly stable} if
\begin{equation}
\label{bAstable}
           \lim_{N \to \infty} \sum_{\alpha \in \free \colon |\alpha| = N} \|{\mathbf
           A}^{\alpha} x\|^{2} \to 0 \quad \text{for all} \; \; x \in {\mathcal X}.
          \end{equation}
Again, as in the single-variable case, the condition \eqref{bAstable} 
in conjunction with $\sbm{  A \\ C }$ being isometric guarantees that 
$\cO_{C, \bA}$ is a partial isometry and that the operator 
$M_{\Theta_{\bU}} \colon f(z) \mapsto \Theta_{\bU}(z) f(z)$ of 
multiplication by the transfer function \eqref{1.28pre} 
acts isometrically as an operator from $H^{2}_{\cU}(\free)$ into 
$H^{2}_{\cY}(\free)$ (see \cite{Cuntz-scat,  BBF1, BBF3} for details).
The representation $\cM = \Theta_{\bU} \cdot H^{2}_{\cU}(\free)$ amounts to 
the Beurling-Lax representation for right-shift invariant subspaces 
(i.e., $f(z) \in \cM \Rightarrow f(z) \cdot z_j \in \cM$ for $j = 1, \dots, d$)
while backward  shift-invariant subspaces arise as the range of an observability 
operator $\cO_{C,\bA}$ (see \cite{PopescuNF1, Cuntz-scat, BBF1, BBF3}).

\section{Weighted Bergman-Fock spaces}  \label{S:Fock}
We introduce a family of weighted Bergman-Fock spaces as a multivariable noncommutative 
counterpart of standard weighted Bergman spaces; the system-theoretic 
point of view presented here combines the single-variable setting handled in 
\cite{BBIEOT, BBSzeged} with the unweighted multivariable setting from 
section \ref{S:Fock0}.
Given an integer $n\ge 1$, the free semigroup $\free$, and the coefficient
Hilbert space $\cY$, we let
\begin{equation}
\cA_{n,\cY}(\free) = \bigg\{\sum_{\alpha \in \free} 
f_{\alpha}z^{\alpha}
\in \cY\langle\langle z\rangle\rangle\colon \;
\sum_{\alpha \in \free} \mu_{n,|\alpha|}\cdot \|f_{\alpha}\|_{\cY}^{2} <\infty \bigg\}
\label{1.30pre}
\end{equation}   
where, according to \eqref{1.6pre}, $\mu_{n,|\alpha|}=\frac{|\alpha|! 
\, (n-1)!}{(n+|\alpha|-1)!}$.
We propose to consider the following  multidimensional system with evolution along the 
free semigroup $\free$:
\begin{equation}    \label{1.31pre}
\Sigma_{\{\bU_\alpha\}, n}\colon \left\{ \begin{array}{ccc}
x(1 \alpha) &= & 
\frac{n+|\alpha|}{|\alpha|+1}A_{1}x(\alpha)+\bcs{n+|\alpha| \\ 
|\alpha|+1} B_{1,\alpha} u(\alpha) \\
\vdots &\vdots & \vdots  \\
x(d \alpha) & = & 
\frac{n+|\alpha|}{|\alpha|+1}A_{d}x(\alpha)+\bcs{n+|\alpha| \\ 
|\alpha|+1} B_{d,\alpha} u(\alpha) \\
y(\alpha) &  = & Cx(\alpha)+\bcs{n+|\alpha|-1\\ |\alpha|} D_\alpha u(\alpha)\end{array}
\right.
\end{equation}
with the $d$-tuple of state space operators ${\bf A}=(A_1,\ldots,A_d)$ and the state-output 
operator $C\in\cL(\cX,\cY)$. Here in addition we have a family of connection matrices and the 
family of  input spaces indexed by $\alpha\in\free$:
\begin{equation}
{\bf U}_\alpha=\begin{bmatrix}A & \widehat B_\alpha \\ C & D_\alpha \end{bmatrix}: \;
\begin{bmatrix}\cX \\ \cU_\alpha\end{bmatrix}\to \begin{bmatrix}\cX^d \\ 
\cY\end{bmatrix},\quad\mbox{where}\quad
A=\begin{bmatrix}A_1 \\ \vdots \\ A_d\end{bmatrix},\;
\widehat B_\alpha=\begin{bmatrix}B_{1,\alpha} \\ \vdots \\ B_{d,\alpha}\end{bmatrix}
\label{coll}
\end{equation}
together with additional  $\alpha$-dependent weights in the system equations indexed by the natural number $n$.
Upon running the system  \eqref{1.31pre}  with a fixed
initial condition $x(\emptyset)=x\in\cX$ we get recursively
\begin{align}
x(\alpha)=&\mu_{n,|\alpha|}^{-1}\cdot\big(\bA^\alpha x 
+\sum_{\alpha^{\prime\prime}j \alpha ^\prime=\alpha}
{\bf A}^{\alpha^{\prime\prime}}B_{j,\alpha^\prime}u(\alpha^\prime)\big),\label{1.32pre}\\
y(\alpha)=&\mu_{n,|\alpha|}^{-1}\cdot\big(C\bA^\alpha x+ 
\sum_{\alpha^{\prime\prime}j\alpha^\prime=\alpha}
C{\bf A}^{\alpha^{\prime\prime}}B_{j,\alpha^\prime}u(\alpha^\prime)+D_\alpha u(\alpha)
\big).\label{1.33pre}
\end{align}
Making use of notation \eqref{1.23pre} and equality \eqref{1.24pre} we observe that 
\begin{equation}
(I-Z(z)A)^{-n}=\sum_{j=0}^\infty 
\mu_{n,j}^{-1}\cdot\sum_{\alpha\in\free: \, |\alpha|=j}{\bf A}^\alpha 
z^\alpha
=\sum_{\alpha \in\free} \mu_{n,|\alpha|}^{-1}{\bf A}^\alpha z^\alpha,\label{2.13pre}
\end{equation}
and then define $R_{n,k}(Z(z)A)$ via formal power series
\begin{equation}
R_{n,k}(Z(z)A)=\sum_{\alpha \in\free} \mu_{n,|\alpha|+k}^{-1} {\bf 
A}^\alpha z^{\alpha}.
\label{2.14pre}
\end{equation}
We next apply the noncommutative $Z$-transform \eqref{1.25pre} to \eqref{1.32pre}
and then invoke \eqref{2.13pre}, \eqref{2.14pre} to get
\begin{align}
\widehat x(z)=&\sum_{\alpha\in\free} \mu_{n,|\alpha|}^{-1}
\bigg(\bA^\alpha x +\sum_{\alpha^{\prime\prime}j\alpha^\prime=\alpha}
{\bf A}^{\alpha^{\prime\prime}}B_{j,\alpha^\prime}u(\alpha^\prime)\bigg)z^\alpha \notag\\
=&\sum_{\alpha\in\free}\big(\mu_{n,|\alpha|}^{-1}\bA^\alpha  x\big)z^\alpha\notag\\
&+\sum_{\alpha^\prime\in\free}\bigg(\sum_{\alpha^{\prime\prime}\in\free}
\mu_{n,|\alpha^{\prime\prime}|+|\alpha^\prime|+1}^{-1}
{\bf A}^{\alpha^{\prime\prime}}z^{\alpha^{\prime\prime}}\bigg)\bigg( \sum_{j=1}^d
z_jB_{j,\alpha^\prime}\bigg)z^{\alpha^\prime}u(\alpha^\prime)\notag\\
=&(I-Z(z)A)^{-n}x+\sum_{\alpha\in\free}R_{n,|\alpha|+1}(Z(z)A)Z(z)\widehat
B_{\alpha}z^{\alpha}u(\alpha).\notag
\end{align}
The same procedure applied to \eqref{1.33pre} now gives
\begin{align}
\widehat y(z)=&C(I-Z(z)A)^{-n}x\notag\\
&+\sum_{\alpha \in\free}\left(CR_{n,|\alpha|+1}(Z(z)A)Z(z)\widehat
B_{\alpha}+\mu_{n,|\alpha|}^{-1} D_\alpha \right)z^{\alpha}u(\alpha)\notag\\
=&\cO_{n,C,{\bf A}}(z) x+\sum_{\alpha\in\free}\Theta_{n, \bU_\alpha}(z)z^{\alpha}u(\alpha),
\label{1.36pre}
\end{align}
where the first term on the right presents the $n$-observability operator 
\begin{equation}   \label{ncobsop}
    \cO_{n,C,\bA}(z) x = C (I - Z(z) A)^{-n} x = \sum_{\alpha \in \free}
    \mu_{n,|\alpha|}^{-1} (C \bA^{\alpha} x) z^{\alpha}
\end{equation}
associated with the state space $d$-tuple ${\bf A}$ and the state-output operator $C$ and 
where
\begin{equation}
\Theta_{n, \bU_\alpha}(z)=\mu_{n,|\alpha|}^{-1}D_\alpha+CR_{n,|\alpha|+1}(Z(z)A)Z(z)\widehat 
B_{\alpha}
\label{1.37pre}
\end{equation}
is the family of transfer functions indexed by $\alpha \in \free$. One can see that 
the notion of the $n$-observability operator \eqref{ncobsop} generalizes the 
single-variable notion \eqref{1.17pre} as well as the unweighted multivariable one 
in \eqref{1.27pre}.
We say that the output pair $(C, \bA)$ is $n$-observable if 
$\cO_{n,C, \bA}$ is injective; from \eqref{ncobsop} we see that this 
is equivalent to $(C, \bA)$ being observable when viewed as an output 
pair for an unweighted system as in \eqref{1.27pre}, i.e., observability is equivalent to
\begin{equation}  \label{ncobs}
    \bigcap_{\alpha \in \free} {\rm Ker}\, C 
    \bA^{\alpha} = \{0\}.
\end{equation}
We say that the output pair $(C,\bA)$ is {\em $n$-output stable} if 
$\cO_{n,C, \bA}$ is bounded as an operator from $\cX$ into 
$\cA_{n,\cY}(\free)$ and {\em exactly $n$-observable} if also 
$\cO_{n,C, \bA}$ is bounded below.  

\smallskip

In parallel with the discussion 
at the end of Section \ref{S:standard}, we use the formula 
\eqref{1.36pre} to view the transfer function $\Theta_{n,\bU_\alpha}(z)$
as encoding the result of a pulse input vector $u$ being applied at 
position $\alpha \in \free$ with zero initial state:
$$
  \widehat y(z) = \Theta_{n, \bU_\alpha}(z) \cdot z^{\alpha} u \; \text{ if } \;  
  x_{\emptyset} = 0 \; \text{ and } \; u(\beta)  = \delta_{\alpha, \beta} u.
$$
A preliminary notion of {\em noncommutative $n$-Bergman conservative  
system} for systems of the form \eqref{1.31pre} will be developed in 
Section \ref{S:metric} below.  The associated {\em $n$-Bergman inner 
family} is the main ingredient for one version of a Beurling-Lax 
representation for forward-shift invariant subspaces for this setting 
(see Section \ref{S:BL7} below). We shall see that backward shift-invariant subspaces
in this setting arise as the range of an $n$-observability operator  
form $\cO_{n,C,\bA}$ for an appropriate choice of a state-space $d$-tuple $\bA$ and 
an state-output operator $C$. 

\smallskip

The main goal of this paper is to carry out the program outlined above in themes \#1--\#4 (including the refinements 
where $\Theta_T$ is allowed to be a non-inner Schur-class function and $T$ is not required to be pure but only completely
noncoisometric) for the setting where the system $\Sigma(\bU)$ \eqref{1.1pre} is replaced by the (time-varying) system
$\Sigma_{\{\bU_\alpha, n\}}$ (for a fixed $n \in {\mathbb N}$),  where the Hardy space $H^2_\cY$ is replaced by the weighted Bergman-Fock 
space $\cA_{n, \cY}(\free)$ \eqref{1.30pre}, where the observability operator $\cO_{C, A}$ \eqref{1.4pre} becomes the $n$-observability
operator $\cO_{n, C, \bA}$ \eqref{ncobsop}, where the transfer function $\Theta_\bU(\lambda)$ \eqref{1.5pre} becomes the
 family of formal-power-series transfer functions $\{ \Theta_{n, \bU_\alpha}(z) \}_{\alpha \in \free}$ \eqref{1.37pre}, where the shift 
 operator $S_\cY$ \eqref{shift} becomes the
right-shift operator tuple $S_{\bmu_n,R} = (S_{\bmu_n, R, 1}, \dots, S_{\bmu_n, R, d})$ on $\cA_{n,\cY}(\free)$ (see \eqref{bo-shift}),
and where the contraction operator $T$ becomes a $*$-$n$-hypercontractive operator tuple, i.e., an operator tuple 
$\bT = (T_1, \dots, T_n) \in \cL(\cH)^d$ such that 
\begin{equation}   \label{nhypertuple}
(I - B_\bT^*)^m[I]= \sum_{\alpha \in \free \colon |\alpha| \le m} (-1)^{|\alpha|} \left( \sbm{m \\ |\alpha|} \right) T^{\alpha^\top} T^{* \alpha} \succeq 0
\end{equation}
for $1 \le m \le n$, where in general $B_{\bT^*}[X] = \sum_{j=1}^d T_j X T_j^*$.

\smallskip

Our results are actually more general than the setting based on the reciprocal binomial coefficient weights $\bmu_n$. In section 
\ref{S:Hardy} we introduce a class of weights $\bo = \{ \omega_j\}_{j \ge 0}$ satisfying certain natural
admissibility conditions (automatically satisfied by $\bo = \bmu_n$ for all $n=1,2,\dots$) and eventually show how the whole
program \#1-- \#4 extends to this level of generality.
Let us point out that the setting $\bo = \bmu_n$ is also part of the setting studied in the recent paper of Popescu \cite{PopescuMANN}.
Many pieces of the program \#1--\#4 for the setting with $\bo = \bmu_n$ are also handled in Popescu's paper, but many of our results 
along with our point of view and approach are complementary.
The setting with a general admissible $\bo$ not equal to $\bmu_n$ however appears to be a strict generalization and is not covered by
Popescu's results.

We also discuss a generalization of the setting $\bo = \bmu_n$ whereby one fixes a free noncommutative function 
given by a formal power series in freely noncommuting arguments $z = (z_1, \dots, z_d)$
\begin{equation}   \label{p}
  p(z) = \sum_{\alpha \in \free} p_\alpha z^\alpha
 \end{equation}
 which is {\em regular}, meaning that 
 $$
 p_\emptyset = 0, \quad p_\alpha > 0 \; \text{ if } \; |\alpha| = 1, \quad p_\alpha \ge 0 \; \text{ for all } \; \alpha \in \free,
 $$
 the series \eqref{p} has positive radius of convergence $\rho > 0$:  {\em $p(\bA) = \sum_{\alpha \in \free} p_\alpha \bA^\alpha$ converges 
 absolutely whenever $\| \begin{bmatrix} A_1 & \cdots & A_d \end{bmatrix} \| < \rho$.}  We then consider the class of operator tuples 
 $\bT = (T_1, \dots, T_d)$ which are {\em  $*$-$(p,n)$-hypercontractive} in the sense that condition \eqref{nhypertuple} is replaced by
 \begin{equation}   \label{pnhypertuple}
 (1 - p)^m(B_{\bT^*})[I] \succeq 0 \; \text{ for } \; m=1, \dots, n.
 \end{equation}
 In Chapter 9 we sketch how to carry out the program \#1--\#4 for this setting.  This generalization originates in the work of Pott \cite{Pott} for the univariate
 case and is handled by  Popescu in \cite{PopescuMemoir2010} for the case $\bo = \bmu_1$ and in \cite{PopescuMANN} for the case $\bo= \bmu_n$.  
 The setting in \cite{PopescuMANN} has an additional layer of flexibility:
 there is incorporated a constraint on the freeness of the lack of commutativity in the components $T_1, \dots, T_d$ of the
 tuple $\bT$; in particular the set of constraints $T_i T_j = T_j T_i$ implies that one can pick up the commutative version of the setting
 as a special case.  Let us mention that the related Popescu paper \cite{PopescuTAMS2016} has still another layer of flexibility which allows one
 to pick up the commutative polydisk and the freely noncommutative ball as special cases.
 
 \smallskip
 
 This paper is organized as follows: 
 
 \smallskip
 
 After the present Introduction,  Chapter 2 recalls the notion of the 
noncommutative formal reproducing kernel Hilbert space (NFRKHS) from \cite{NFRKHS, Cuntz-scat} which will be a substantial
tool for the subsequent analysis here.  

\smallskip

The notion of contractive multiplier between two NFRKHSs works out well and is the topic of Chapter 3.
 We  specialize the general setting to the case of contractive multipliers from one Fock space 
into another as well as from a Fock space to a more general weighted Bergman-Fock spaces, with particular focus
on inner  multipliers (taken in various distinct senses).
Some of these results are obtained by using a noncommutative 
version of Leech theorem, \cite{Leech},  two different proofs of which are presented in Section \ref{S:Leech}.

\smallskip

Chapter 4 provides an elaboration of theme \#1 (backward-shift invariant subspaces and observability operators) 
in these more general settings. Much of the analysis hinges on the study of multivariable Stein equations and
connections with stability for discrete-time linear systems---a standard topic in classical linear system theory 
in the univariate setting (see e.g.\ the book of Dullerud--Paganini \cite{DulPag}).  Included here is an exposition of the basic
properties of the model shift-operator tuple on $H^2_{\bo, \cY}(\free)$ and a converse characterization of which
$\bo$-isometric operator tuples are unitarily equivalent to such a shift-operator tuple, as the shift part in the
Wold decomposition for a more general class of operator tuples which we call $\cC(\bo)$, thereby generalizing
results of Olofsson, Giselsson, and Wennman \cite{gisolof, olofw} and Eschmeier-Langend\"orfer \cite{EL}.
The analysis in this chapter builds on and clarifies
the material from \cite{Cuntz-scat, BBF1, BBF3}   treating the Fock-space case.

\smallskip

Chapter 5   moves on to an elaboration of theme \#2 (Beurling-Lax representation theorems for shift-invariant subspaces).  
For these more general settings Chapter 5 has results for (i) shift-invariant subspaces $\cM$
isometrically included in the noncommutative weighted Hardy space $H^2_{\bo, \cY}(\free)$ in terms of a McCullough-Trent
(McCT) inner multiplier $\Theta$ (a formal power series $\Theta$ for which the multiplication operator is a partial isometry), thereby
generalizing the result of McCullough-Trent \cite{MCT} to the multivariable noncommutative setting, 
(ii) for contractively-included shift-invariant subspaces $\cM$  in terms of a contractive multiplier $\Theta$,  thereby generalizing
results of de Branges-Rovnyak \cite{dBR1, dBR2} for the univariate case.  For both of these representations, we use as the
the Fock space $H^2_\cU(\free)$ as the input space for the multiplier (so $M_\Theta \colon H^2_\cU(\free) \to H^2_{\bo, \cY}(\free)$)
rather than the same weighted Bergman-Fock space 
but with a different coefficient space as the model for the input space (so $M_\Theta \colon H^2_{\bo, \cU}(\free) \to  H^2_{\bo, \cY}(\free)$)
as in some previous treatments (see \cite{BBIEOT, BBSzeged, PopescuTAMS2016}); this enables one to have a more natural
minimal set of hypotheses, as has already been seen in the nice abstract setup for commutative setting by Sarkar and coauthors
 \cite{sarkar, sarkarII, BEKS} and in
the noncommutative multivariable setting as here in work of the authors and Popescu \cite{BBCR, PopescuMANN}.

\smallskip

Chapter 6  also is concerned with theme \#2 but of a different flavor.  Namely, here we give our noncommutative multivariable version 
of the {\em quasi-wandering subspace} of Izuchi and others \cite{iziziz1, iziziz, chen}.

\smallskip

Chapter 7 is concerned exclusively with  Beurling-Lax representation results for
isometrically included subspaces $\cM$, but now in terms of a Bergman-inner multiplier $\Theta$ (i.e., $\Theta$ maps the constants
isometrically onto a wandering subspace for $\cM$---generalizing work of Aleman, Duren, Khavinson, Richter, Shapiro, Sundberg
\cite{ars, DKSS, DKSS1} to the noncommutative multivariable setting.   By using a whole family of Bergman inner functions rather than
a single Bergman inner function, we get a more orthogonal Beurling-Lax representation, closer to but more complicated than, the classical
case.  We also obtain analogues of the expansive multiplier property and some results on characterizations of Bergman-inner multipliers as 
extremal solutions of interpolation problems analogous to results of Duren, Hedenmalm, Khavinson, Shapiro, Sundberg, Vukoti\'c
\cite{heden1, DKSS, vuk} for the univariate case. 

\smallskip

Chapter 8 uses the results of Chapters 5 and 7 to flesh out themes \#3 and \#4 for our general noncommutative multivariable setting.
There are two distinct types of model theory depending on whether one uses the Beurling-Lax representation results of Chapter 5
or the Beurling-Lax representation results of Chapter 7.  The former includes model theory results for at least some classes of
completely noncoisometric $*$-$\bo$-hypercontractive tuples while the latter to this point is only for pure (or $*$-strongly stable)
$*$-$\bo$-hypercontractive operator tuples.  Our use of contractively included subspaces and Brangesian complementary spaces 
for the former  gives some complementary results and an alternative perspective to the results of Popescu in \cite{PopescuMANN}.

\smallskip

Chapter 9, as already mentioned, deals with the $(p,n)$ setting where the weighted Bergman-Fock space is adjusted to handle the model
theory for $*$-$(p,n)$-hypercontrac\-tive operator tuples $\bT$ as in \eqref{pnhypertuple}.  It turns out that, with the proper (not always
expected) adjustments, all the results for the $\bmu_n$ setting extend to this more general setting.
In this way we recover the operator-model theory results of Popescu \cite{PopescuMANN} for this setting.

\chapter{Formal Reproducing Kernel Hilbert Spaces}
\label{S:NFRKHS}

As was mentioned in Section \ref{S:standard}, the standard weighted
Bergman space $\cA_{n,\cY}$ can be viewed as a reproducing kernel
Hilbert space with reproducing kernel given by \eqref{1.7pre}.
It is useful to have a similar point of view for the weighted
Bergman-Fock spaces discussed in Section \ref{S:Fock}.
\section{Basic definitions}
In this section we review the notion of {\em formal reproducing kernel 
Hilbert space} developed in \cite[Section 3]{NFRKHS}.
\smallskip

Given a collection of freely noncommuting indeterminates $z = (z_{1}, 
\dots, z_{d})$, we suppose that we are given a Hilbert space $\cH$ 
whose elements are formal power series 
\begin{equation}   \label{elem}
f(z) = \sum_{\alpha \in 
\free} f_{\alpha} z^{\alpha} \in \cY \langle \langle z \rangle \rangle,\qquad f_{\alpha}\in\cY, 
\end{equation}
with coefficients from a coefficient Hilbert space $\cY$.  We say that $\cH$ is a NFRKHS ({\em 
noncommutative formal reproducing kernel Hilbert space}) if, for each 
$\beta \in \free$, the linear operator $\Phi_{\beta} \colon \cH 
\mapsto \cY$ defined by 
\begin{equation}   \label{phi}
\Phi_{\beta}: \; f(z) = \sum_{\alpha \in \free} f_{\alpha} 
z^{\alpha} \mapsto f_{\beta},
\end{equation}
is continuous.  As any such power series is completely 
determined by the list of its coefficients $\alpha \mapsto 
f_{\alpha}$ for $\alpha \in \free$, equivalently we can view the elements $f(z)$ as the 
functions $\alpha \mapsto f_{\alpha}$ on $\free$.  Hence, by the 
noncommutative Aronszajn 
theory of reproducing kernel Hilbert spaces (see e.g.~\cite[Theorem 
1.1]{NFRKHS}), there is a  positive kernel 
$K \colon \free \times \free \to \cL(\cY)$ so that $\cH$ is the 
reproducing kernel Hilbert space associated with $K$.  To spell this 
out,  in the present 
context we denote the value of $K$ at $(\alpha, \beta)$ by 
$K_{\alpha, \beta} \in 
\cL(\cY)$ rather than $K(\alpha, \beta)$.  Since we view an element $f \in \cH$ 
as a formal power series \eqref{elem} rather 
than as a function $\alpha \mapsto f_{\alpha}$ on $\free$, we write, for a 
given $\beta \in \free$ and $y \in \cY$, the element 
$\Phi_{\beta}^{*} y \in \cH$ as $\Phi_{\beta}^{*} y = K_{\beta}(\cdot) 
y$, where
\begin{equation}   \label{Kbeta}
K_{\beta}(z) y = \sum_{\alpha \in \free} K_{\alpha, \beta}y \, z^{\alpha}.
\end{equation}
Then the reproducing kernel property can be written as 
\begin{equation}  \label{reprod-prop}
    \langle f, \, K_{\beta}(\cdot) y \rangle_{\cH} =  \langle f, \, \Phi_{\beta}^{*} y \rangle_{\cH}
=\langle \Phi_{\beta}f, \, y \rangle_{\cY}=\langle f_{\beta}, \, y \rangle_{\cY}.
\end{equation}
We can make the notation more suggestive of the classical case as 
follows. Let $\bzeta = (\bzeta_{1}, \dots, \bzeta_{d})$ be a second $d$-tuple of 
noncommuting indeterminates.  
Given a coefficient Hilbert space $\cC$, we can
use the $\cC$-inner product to define pairings
$$
\langle \cdot, \cdot \rangle_{\cC \times \cC\langle \langle \bzeta \rangle
\rangle} \mapsto {\mathbb C} \langle \langle \zeta \rangle \rangle\quad\mbox{and}\quad
\langle \cdot, \cdot \rangle_{\cC\langle \langle \bzeta \rangle
\rangle\times\cC} \mapsto {\mathbb C} \langle \langle \zeta \rangle \rangle
$$
(where ${\mathbb C} \langle \langle \zeta \rangle \rangle$ is the
space of formal power series in the set of formal conjugate indeterminates
$\zeta = (\zeta_{1}, \dots, \zeta_{d})$ with coefficients in ${\mathbb C}$) by
$$
\big\langle c, \sum f_{\alpha} \bzeta^{\alpha}
\big\rangle_{\cC \times \cC\langle \langle \bzeta \rangle \rangle} = 
\sum\langle c, f_{\alpha} \rangle_{_{\cC}}\,  \zeta^{\alpha^{\top}},\quad
\big\langle \sum f_{\alpha} \zeta^{\alpha}, c
\big\rangle_{{\cC\langle \langle \zeta \rangle \rangle \times \cC}} = 
\sum\langle f_{\alpha}, c \rangle_{_{\cC}} \, \zeta^{\alpha}.
$$
These pairings can be seen as special cases of the more general pairing
\begin{equation}   \label{genpairing}
\bigg\langle \sum_{\alpha \in \free} f_{\alpha} \zeta^{\alpha}, \sum_{\beta 
\in \free} g_{\beta} \bzeta^{\beta} \bigg\rangle_{\cC\langle \langle 
\zeta \rangle\rangle \times \cC\langle \langle \bzeta \rangle \rangle} = 
\sum_{\alpha \in \free} \bigg[ 
\sum_{\beta, \gamma \colon \alpha =\gamma^{\top}  \beta  } \langle 
f_{\beta}, g_{\gamma} 
\rangle_{_{\cC}} \bigg] \zeta^{\alpha},
\end{equation}
which  can be written more suggestively as
\begin{equation}   \label{genpairing'}
\langle f(\zeta), g(\bzeta) \rangle_{\cC\langle \langle \zeta  \rangle \rangle \times \cC \langle \langle \bzeta \rangle \rangle}
= \big\langle \sum f_{\alpha} \zeta^{\alpha}, \, \sum
g_{\beta} \bzeta^{\beta} \big\rangle_{\cC\langle \langle \zeta \rangle \rangle \times \cC \langle \langle \bzeta \rangle \rangle}= 
g(\bzeta)^{*} f(\zeta)
\end{equation}
if we set
$$
  g(\bzeta)^{*} = \bigg( \sum g_{\beta} \bzeta^{\beta} \bigg)^{*} =
  \sum g_{\beta}^{*} \zeta^{\beta^{\top}},
$$
where we view $g_{\beta}^{*} \in \cL(\cC, {\mathbb C})$ as a linear 
functional on $\cC$ so that
$$
  g_{\beta}^{*} f_{\alpha} = \langle f_{\alpha}, g_{\beta}\rangle_{_{\cC}}\quad\mbox{for any}
  \quad f_{\alpha} \in \cC.
$$
Then, if $S(\zeta) \in \cL(\cU, \cY)\langle \langle \zeta \rangle \rangle$, $f(\zeta) \in \cU\langle 
\langle \zeta \rangle \rangle$ and $g(\bzeta)\in \cY\langle \langle 
\bzeta \rangle \rangle$, we see that
\begin{align}
 \langle S(\zeta) f(\zeta), \, g(\bzeta) \rangle_{\cY\langle \langle \zeta \rangle \rangle \times \cY \langle \langle \bzeta  \rangle \rangle} 
    & = g(\bzeta)^{*}\left( S(\zeta) f(\zeta) \right)
    = \left( g(\bzeta)^{*} S(\zeta) \right) f(\zeta)\notag  \\
   & = \langle f(\zeta), S(\zeta)^{*} g(\bzeta) \rangle_{\cU\langle 
    \langle \zeta \rangle \rangle \times \cU \langle \langle \bzeta  \rangle \rangle}.\label{MSadj}
\end{align}
The reproducing kernel property \eqref{reprod-prop} can be 
written more suggestively as
\begin{equation}   \label{reprod-prop'}
    \langle f, K(\cdot, \zeta) y \rangle_{\cH \times \cH\langle 
    \langle \bzeta \rangle \rangle} = \langle f(\zeta), y \rangle_{\cY\langle 
    \langle \zeta \rangle \rangle \times \cY},
\end{equation}
where we set
\begin{equation}   \label{formalkernel}
K(z,\zeta) = \sum_{\alpha, \beta \in \free} K_{\alpha, \beta}z^{\alpha} 
\bzeta^{\beta^{\top}} \in \cL(\cY)\langle \langle z, \bzeta \rangle 
\rangle.
\end{equation}
We note that, for each $y \in \cY$, the formal power series
$$
K(z,\zeta) y = \sum_{\alpha, \beta \in \free} K_{\alpha, \beta} y  \,
z^{\alpha} \bzeta^{\beta^{\top}} = \sum_{\beta \in \free}
\bigg[ \sum_{\alpha \in \free} K_{\alpha, \beta^{\top}}y \, z^{\alpha} \bigg] 
\bzeta^{\beta} 
$$
is an element of $ \cH\langle \langle \bzeta \rangle \rangle$.

\smallskip

The formal kernel \eqref{formalkernel} is of the 
so-called {\em hereditary form}, i.e., the formal power series 
\eqref{formalkernel} involves only noncommutative monomials in $z$ and 
$\bzeta$ of the form $z^{\alpha} \bzeta^{\beta^{\top}}$ (a monomial 
in $z$ followed by a monomial in $\bzeta$).  For our purposes here we define the space 
$\cL(\cY)\langle \langle z, \bzeta \rangle \rangle$ to consist only 
of such hereditary formal power series in the freely noncommuting 
indeterminates.  Here we are assuming that $z  = (z_{1}, \dots, 
z_{d})$ and $\bzeta = (\bzeta_{1}, \dots, \bzeta_{d})$ are each 
separately freely noncommuting indeterminates.  It will be convenient 
to assume in addition that $z_{k}$ commutes with each $\bzeta_{j}$
\begin{equation}   \label{hereditarynorm} \bzeta_{j} z_{k} = z_{k} \bzeta_{j}\quad 
 \text{for}\quad k,j = 1, \dots, d.
\end{equation}

The following summary (see Theorem 3.1 in \cite{NFRKHS}) of the structure of such NFRKHSs 
among other things  characterizes which formal power series $K(z,\zeta) \in \cL(\cY) 
\langle \langle z, \bzeta \rangle \rangle$ arise in this way from a NFRKHS.

\begin{theorem} \label{T:NFRKHS}
    Let $K(z, \zeta)$ of the form \eqref{formalkernel} be a given element of $\cL(\cY) 
    \langle \langle z, \bzeta \rangle \rangle$ where $z = (z_{1}, \dots, 
    z_{d})$ and $\bzeta = (\bzeta_{1}, \dots, \bzeta_{d})$ are $d$-tuples of 
    noncommuting indeterminates which pairwise commute with each 
    other.  Then the following are equivalent:
    \begin{enumerate}
	\item $K$ is the reproducing kernel for a unique NFRKHS denoted as $\cH(K)$.
	\item There is an auxiliary Hilbert space $\cH_{0}$ and a noncommutative formal power series 
$H(z) \in \cL(\cH_{0}, \cY)\langle \langle z \rangle \rangle$ such that
\begin{equation}   \label{Kol}
	  K(z,\zeta) = H(z) H(\zeta)^{*}
 \end{equation}
(Kolmogorov decomposition for $K$).  Here we use the conventions
$$
  (\zeta^{\beta})^{*} = \bzeta^{\beta^{\top}}, \quad H(\zeta)^*=\left( \sum H_{\beta} \zeta^{\beta}\right)^{*} = \sum 
  H_{\beta}^{*}\, \bzeta^{\beta^{\top}}.
$$

\item $K(z,\zeta)$ is a {\em positive formal kernel} in the sense that for all finitely supported $\cY$-valued functions $\alpha \mapsto
y_{\alpha} $ on $\free$,
\begin{equation}   \label{coef-pos}
   \sum_{\alpha, \beta \in \free} \langle K_{\alpha, \beta} \,
   y_{\alpha}, y_{\beta} \rangle_{_{\cY}} \ge 0.
\end{equation}
\end{enumerate}
Moreover, in this case the space $\cH(K)$ can be defined directly in 
terms of the formal power series $H(z)$ appearing in condition (2) via
$$
  \cH(K) = \{ H(z) h_{0} \colon h_{0} \in \cH_{0} \}
$$
with norm taken to be the ``lifted norm''
$$
  \| H(\cdot) h_{0} \|_{\cH(K)} = \| Q h_{0} \|_{\cH_{0}}
$$
where $Q$ is the orthogonal projection of $\cH_{0}$ onto the 
orthogonal complement of the kernel of the map $M_{H} \colon \cH_{0} 
\to \cY\langle \langle z \rangle \rangle$ given by $M_{H} \colon h_{0} 
\mapsto H(z) \cdot h_{0}$.
\end{theorem}

The more general pairings \eqref{genpairing} and \eqref{genpairing'} are involved 
in the formulation of the following useful more general form of the reproducing 
kernel property.

\begin{proposition}  \label{P:genreprod}
Let $\cH(K)$ be the NFRKHS associated with a positive formal kernel  $K\in\cL(\cY)
    \langle \langle z, \bzeta \rangle \rangle$. Then, for all $f\in \cH(K)$ and 
$g(\zeta) \in \cY\langle \langle \bzeta \rangle \rangle$, we have
\begin{equation}   \label{genreprod}
    \langle f,\, K(\cdot, \zeta) g(\zeta) \rangle_{\cH(K) \times 
    \cH(K) \langle \langle \bzeta \rangle \rangle} =
    \langle f(\zeta), g(\zeta) \rangle_{\langle \cY\langle \langle 
    \zeta \rangle \rangle \times \cY \langle \langle \bzeta \rangle 
    \rangle}.
\end{equation}
\end{proposition}

\begin{proof}
    Note first that the reproducing property \eqref{reprod-prop'}  
    can be rewritten as
 $$
 (K(\cdot, \zeta) y)^{*} f = y^{*} f(\zeta)\quad\mbox{for all}\quad y \in \cY,
$$
from which we deduce that $K(\cdot, \zeta)^{*} f = f(\zeta)$ for any $f \in \cH(K)$. Then we have
\begin{align*}
\langle f(\zeta), \, g(\zeta) \rangle_{\cY\langle \langle \zeta
     \rangle \rangle \times \cY \langle \langle \bzeta \rangle \rangle}&=
\langle K(\cdot, \zeta)^{*} f, \, g(\zeta) \rangle_{\cY\langle \langle \zeta
     \rangle \rangle \times \cY \langle \langle \bzeta \rangle \rangle}\\
&=\langle f, \, K(\cdot, \zeta) g(\zeta) \rangle_{\cH(K) \times
     \cH(K)\langle \langle \bzeta \rangle \rangle}
\end{align*}
for any $g(\zeta) \in \cY\langle \langle \bzeta \rangle \rangle$ which completes the proof.
\end{proof}

The following general principle (a variation of the last statement 
of Theorem \ref{T:NFRKHS}) will be useful in the sequel.

\begin{proposition}  \label{P:principle}
Suppose that $H(z) = \sum_{\alpha \in \free} H_{\alpha} z^{\alpha} \in \cL(\cX, 
\cY)\langle \langle z \rangle \rangle$ is a formal power series such that the set 
$$
\cM = \{ H( \cdot) x \colon x \in \cX \} \subset \cY \langle 
  \langle z \rangle \rangle
$$
is a Hilbert space with inner product $\langle \cdot, 
\cdot \rangle_{\cM}$ satisfying
\begin{equation} \label{H-principle}
    \langle H(\cdot) x, H(\cdot) x' \rangle_{\cM} = \langle G x, x' 
    \rangle_{\cX}
\end{equation}
where $G$ is an invertible positive definite operator on $\cX$.  
Then $\cM$ is a NFRKHS with formal reproducing kernel $K_{\cM}(z, \zeta)$ given by
 $$   K_{\cM}(z, \zeta) = H(z) G^{-1} H(\zeta)^{*}.$$
\end{proposition}
\begin{proof}
    For $x \in \cX$ and $y \in \cY$, we compute
\begin{align*}
    \langle H(\zeta)x, y \rangle_{\cY\langle \langle \zeta \rangle 
    \rangle\times \cY}  
  & = \langle Gx, G^{-1} H(\zeta)^{*} y \rangle_{\cX \times \cX 
  \langle \langle \bzeta \rangle \rangle} \\
&  = \langle H(\cdot) x, H(\cdot) G^{-1} H(\zeta)^{*} y 
  \rangle_{\cM \times \cM \langle \langle \bzeta \rangle \rangle}
\end{align*}
where the last step follows from the hypothesis \eqref{H-principle}. 
The result follows.
\end{proof}
\begin{remark}
{\rm Just as in the classical case, if $K$ is the 
reproducing kernel for the NFRKHS $\cH(K)$, then there is a canonical 
choice of $H(z)$ which yields the factorization \eqref{Kol} in part 
(2) of Theorem \ref{T:NFRKHS}, namely:  take $\cH_{0}$ equal to the space 
$\cH(K)$ itself and set $H(z) = \sum_{\beta \in \free} 
\Phi_{\beta} z^{\beta}$ where $\Phi_{\beta}$ is the
coefficient-evaluation functional \eqref{phi}.}
\label{R:fun}
\end{remark}
\begin{example}  \label{E:1.1}
The Fock space $H^{2}_{\cY}(\free)$ introduced in Section
    \ref{S:Fock0} is the NFRKHS $\cH(k_{\rm nc, Sz} I_{\cY})$
    where $k_{\rm nc,Sz}$ is the {\em noncommutative Szeg\H{o} kernel}
\begin{equation}
    k_{\rm nc, Sz}(z,\zeta) =\sum_{\alpha \in \free} z^{\alpha}
    \bzeta^{\alpha^{\top}}.
\label{ncsz}
\end{equation}
\end{example}

\begin{example}  \label{E:1.2}
More generally, the weighted Bergman-Fock space
$\cA_{n,\cY}(\free)$ \eqref{1.30pre} is the NFRKHS $\cH(k_{{\rm nc},n}
I_{\cY})$ where the {\em noncommutative weight-$n$ Bergman
kernel} $k_{{\rm nc },n}(z,\zeta)$ is defined by
\begin{equation}  \label{nc-n-Berg-ker}
    k_{{\rm nc }, n}(z,\zeta) = \sum_{\alpha \in \free}
   \mu_{n,|\alpha|}^{-1} z^{\alpha} \bzeta^{\alpha^{\top}},
\qquad\mu_{n,|\alpha|}=\frac{|\alpha|!(n-1)!}{(|\alpha|+n-1)!}.
\end{equation}
The reproducing kernel property will be verified in Section \ref{S:Hardy} for a more general setting.
We next point out two identities relating the kernels \eqref{ncsz} and \eqref{nc-n-Berg-ker}:
\begin{align}
&k_{\rm{nc}, n}(z,\zeta) - \sum_{j=1}^{d} \bzeta_{j} k_{\rm{nc}, n}(z,\zeta)
z_{j} = k_{\rm{nc}, n-1}(z, \zeta),\label{kern/n-1}\\
&k_{{\rm nc}, n}(z,\zeta)
= \sum_{\alpha \in \free} \mu_{n-1, |\alpha|}^{-1} \bzeta^{\alpha^{\top}} k_{\rm nc,
Sz}(z, \zeta) z^{\alpha}.\label{ncsz1}
\end{align}
To verify \eqref{kern/n-1} and \eqref{ncsz1}, it suffices to show that the coefficients of $z^{\alpha} \bzeta^{\alpha^{\top}}$ on both sides
are equal for each $\alpha\in\free$. Due to \eqref{nc-n-Berg-ker} and \eqref{ncsz}, the latter amounts to the respective equalities
$$
\mu_{n,|\alpha|}^{-1}-\mu_{n,|\alpha|-1}^{-1}=\mu_{n-1,|\alpha|}^{-1}\quad\mbox{and}\quad 
\mu^{-1}_{n,|\alpha|}=\mu^{-1}_{n-1,0}+\mu^{-1}_{n-1,1}+\ldots+\mu^{-1}_{n-1,|\alpha|}.
$$
By the definition \eqref{1.6pre} of $\mu_{n,|\alpha|}$, 	the first equality is the binomial-coefficient identity 
$\bcs{ |\alpha| + n -1\\ n-1}- \bcs{|\alpha|+n-2 \\ n-1}=\bcs{|\alpha|+n-2 \\ n-2}$, 
while the second follows from the first by an induction argument. We finally note that letting $n=1$ in \eqref{ncsz1} gives
\begin{equation}   \label{Sz-id}
 k_{\rm nc, Sz}(z, \zeta) - \sum_{k=1}^{d} \bzeta_{k} k_{\rm nc, Sz}(z, \zeta) z_{k} =1.
 \end{equation}
We shall use formal positive kernels \eqref{nc-n-Berg-ker} as a tool for
 obtaining our multivariable Beurling-Lax theorem for the space
 $\cA_{n, \cY}(\free)$ and the operator model theory for freely noncommuting
 $n$-hypercontractive operator $d$-tuples below.
 \end{example}

\begin{definition}
\label{kercon}
We say that the nc positive kernel $K(z,\zeta)$ in free indeterminates $z = (z_1, \dots, z_d)$ and
$\overline{\zeta} = (\overline{\zeta_1}, \dots, \overline{\zeta_d})$ is a {\em contractive kernel} if  
the space $\cH(K)$ is  invariant under the  right coordinate-variable multipliers 
$$
M^{r}_{z_j}: \, f(z)\mapsto f(z)\cdot z_j\quad\mbox{for}\quad j=1,\ldots,d
$$ 
and   the tuple ${\bf M}^{r}_z=(M^{r}_{z_1},\ldots,M^{r}_{z_d})$ is a row contraction, 
i.e., the block row matrix $M^r_z:=\begin{bmatrix} M^{r}_{z_{1}} & \cdots & 
M^{r}_{z_{d}} \end{bmatrix}$ is a contraction operator from $\cH(K)^{d}$ to $\cH(K)$: 
\begin{equation}  \label{dec22}
\|M^{r}_{z_1}f_1+ \cdots  + M^{r}_{z_d}f_d \| \le 
\|f_1\|^2+ \cdots + \|f_d\|_{\cM} ^2 \text{ \; for all }
f_1, \dots  f_d\in \cH(K).
\end{equation}
\end{definition}
\noindent
We have the following characterization of contractive formal positive kernels $K$.
\begin{proposition}   \label{P:referee}
A given positive formal kernel $K$ is a contractive formal positive kernel if and only if
$K$ has the form
    \begin{equation}   \label{kcm2}
 K(z, \zeta) = G(z)( k_{\rm nc, Sz}(z,\zeta) \otimes I_\cU) G(\zeta)^{*}
    \end{equation}
    for some $G\in \cL(\cU, \cY)\langle\langle z\rangle\rangle$
    and some auxiliary Hilbert space $\cU$.
\end{proposition}

    \begin{proof} We claim that condition \eqref{dec22} along with shift invariance of $\cH(K)$ is equivalent 
    to the condition
   \begin{equation}   \label{kernelL}
    L(z, \zeta) = K(z, \zeta)-\sum_{j=1}^d \bzeta_j K(z,  \zeta) z_j \quad\text{is a positive kernel,}
  \end{equation}
i.e., that $L(z, \zeta)$ has a Kolmogorov decomposition
  \begin{equation} \label{L-Koldecom}
  L(z, \zeta) = G(z) G(\zeta)^{*}
  \end{equation}
  for some $G\in \cL(\cU, \cY)\langle\langle z\rangle\rangle$. As a 
  first step toward verifying \eqref{L-Koldecom}, we wish to compute 
  the adjoint of $M^{r}_{z_{j}}$ acting on a kernel element:
  \begin{equation}  \label{adj-action}
\big((M^{r}_{z_{j}})^{*} \otimes I_{{\mathbb C}\langle \langle 
\overline{\zeta} \rangle \rangle} \big) K(z, \zeta) = 
\overline{\zeta}_{j} K(z, \zeta)\quad \text{for} \quad j = 1, \dots, d.
\end{equation}
To this end, for $y \in \cY$, write out $K(z, \zeta)y$ as 
$K(z, \zeta)y = \sum_{\alpha} K_{\alpha}(z)y \,\overline{\zeta}^{\alpha^{\top}}$ 
with $K_{\alpha}y \in \cH(K)$ for each $\alpha \in \free$,
and then, for $f(\zeta)=\sum_{\alpha}f_\alpha\zeta^\alpha \in \cH(K)$, compute
\begin{align*}
\langle f, \; \overline{\zeta}_{j} K(\cdot, \zeta) y
\rangle_{\cH(K) \times \cH(K) \langle \langle \overline{\zeta}
\rangle \rangle}
&= \langle f, \sum_{\alpha \in \free} K_{\alpha} y \, \overline{\zeta}_{j}
\overline{\zeta}^{\alpha^{\top}}\rangle_{\cH(K) \times \cH(K) \langle
\langle \overline{\zeta} \rangle \rangle} \\
& = \sum_{\alpha \in \free} \langle f, K_{\alpha} y
\rangle_{\cH(K)} \zeta^{\alpha} \zeta_{j} \\
& = \sum_{\alpha \in \free} \langle f_{\alpha}, y \rangle_{\cY}
\zeta^{\alpha} \zeta_{j}= \langle f(\zeta) \cdot \zeta_{j}, \; y
\rangle_{\cY\langle \langle \, \zeta \rangle \rangle \times \cY}.
\end{align*}
Since on the other hand,
\begin{align*}
\left\langle f, \;  (M^{r}_{z_{j}})^*K(\cdot, \zeta) y \right\rangle
_{\cH(K)\times \cH(K) \langle \langle \overline{\zeta} \rangle \rangle}
&=\langle M^{r}_{z_{j}} f, \; K(\cdot, \zeta) y \rangle_{\cH(K)
\times \cH(K) \langle \langle \overline{\zeta} \rangle \rangle}\\
& = \langle f(\zeta) \cdot \zeta_{j}, \; y
\rangle_{\cY\langle \langle \zeta \rangle \rangle \times \cY}, 
\end{align*}
equality \eqref{adj-action} follows.
To verify \eqref{L-Koldecom} we compute next, for $y,y' \in \cY$,
\begin{align*}
  \langle L(z, \zeta) y, \; y' \rangle_{\cY \langle \langle z, 
    \overline{\zeta} \rangle \rangle \times \cY}  =&
\langle K(\cdot, \zeta) y, \; K(\cdot, z) y' \rangle_{\cH(K) 
\langle \langle \overline{\zeta} \rangle \rangle \times \cH(K) 
\langle \langle \overline{z} \rangle \rangle}  \\
&
-\sum_{j=1}^{d} \langle M^{r *}_{z_{j}} K(\cdot, \zeta) y, 
M^{r*}_{z_{j}} K(\cdot, z) y'\rangle_{\cH(K) \langle \langle 
\overline{\zeta} \rangle \rangle \times \cH(K)\langle \langle 
\overline{z} \rangle \rangle} \\
= &\left\langle \left( \Gamma \otimes I_{{\mathbb C} \langle \langle 
\overline{\zeta} \rangle \rangle} \right)
K(\cdot, \zeta) y, K(\cdot, z) y' \right\rangle_{\cH(K) \langle \langle 
\overline{\zeta} \rangle \rangle \times \cH(K)\langle \langle 
\overline{z} \rangle \rangle}
\end{align*}
where we set $\Gamma$ equal to
$$
\Gamma =  I_{\cH(K)} - M^{r}_{z} (M^{r}_{z})^*,\quad\mbox{where}\quad 
 M^{r}_{z}=\begin{bmatrix} M^{r}_{z_{1}} & \cdots &
M^{r}_{z_{d}} \end{bmatrix}. 
$$
The assumption that $\| \begin{bmatrix} M^{r}_{z_{1}} & \cdots & 
M^{r}_{z_{d}} \end{bmatrix} \| \le 1$ implies that 
the operator $\Gamma$ is positive semidefinite and has a factorization of 
the form $\Gamma = X^{*} X$ for an operator $X \in \cL(\cH(K), \cU)$ 
(for a convenient auxiliary Hilbert space $\cU$).  We may then 
continue the calculation above to get
\begin{align*}
\langle L(z, \zeta) y, y' \rangle_{\cY \langle \langle z,
    \overline{\zeta} \rangle \rangle \times \cY}&= \langle (X \otimes I_{{\mathbb C}\langle \langle
  \overline{\zeta} \rangle \rangle}) K(\cdot, \zeta) y,
  (X \otimes I_{{\mathbb C}\langle \langle
  \overline{z} \rangle \rangle}) K(\cdot, z) y' \rangle_{\cU \langle \langle
  \overline{\zeta} \rangle \rangle  \times  \cU \langle \langle
  \overline{z} \rangle \rangle}\\
  &= \bigg\langle \sum_{\beta \in \free} X K_{\beta}y
  \overline{\zeta}^{\beta^{\top}}, \sum_{\alpha \in \free} X K_{\alpha}y'
\overline{z}^{\alpha^{\top}} \bigg\rangle_{\cU \langle \langle
  \overline{\zeta}\rangle \rangle  \times  \cU \langle \langle
  \overline{z} \rangle\rangle}.
\end{align*}
Let us define $G \in \cL(\cU, \cY)\langle \langle z \rangle \rangle$ by
\begin{equation}  \label{L-Kolfactor}
 G(z) = \sum_{\alpha \in \free} G_{\alpha}^{*} z^{\alpha},\quad\mbox{where}\quad 
 G_{\alpha} \colon y \mapsto X K_{\alpha}y.
 \end{equation}
Then the above calculation continues as
\begin{align*}
 \langle L(z, \zeta) y, y' \rangle_{\cY \langle \langle z, 
    \overline{\zeta} \rangle \rangle \times \cY}  &=
 \bigg\langle \sum_{\beta \in \free} G_{\beta} y \; 
 \overline{\zeta}^{\beta^{\top}}, \, \sum_{\alpha \in \free} 
 G_{\alpha} y' \, \overline{z}^{\alpha^{\top}}\bigg\rangle_{ \cU \langle 
 \langle \overline{\zeta} \rangle \rangle \times \cU \langle \langle 
 \overline{z} \rangle \rangle} \\
 & = \sum_{\alpha, \beta \in \free} \langle G_{\alpha}^{*} G_{\beta} y, y' \rangle_{\cY} 
 z^{\alpha} \overline{\zeta}^{\beta^{\top}} = \langle G(z) 
 G(\zeta)^{*} y, y' \rangle_{\cY\langle \langle z, \overline{\zeta} 
 \rangle \rangle \times \cY},
 \end{align*}
 and we arrive at the Kolmogorov decomposition \eqref{L-Koldecom} 
 with the factor $G(z)$ given by \eqref{L-Kolfactor} as wanted.
 We next rewrite the identity \eqref{kernelL} in the form
  $$
  K(z, \zeta) = G(z) G(\zeta)^{*} + \sum_{j=1}^{d} \bzeta_{j} 
  K(z, \zeta) z_{j}.
  $$
  Iteration of this identity $N-1$ times then gives
  $$
  K(z, \zeta) = \sum_{\alpha\in\free \colon |\alpha| \le N} G(z) 
  z^{\alpha} \bzeta^{\alpha^{\top}} G(\zeta)^{*} + \sum_{\alpha\in\free 
  \colon |\alpha| = N+1} \bzeta_{j}^{\alpha^{\top}} K(z, \zeta) 
  z^{\alpha}.
  $$
  Taking the limit as $N \to \infty$ in this expression then leaves 
  us with the desired formula \eqref{kcm2} for $K$.

\smallskip

Conversely, if $K$ is of the form \eqref{kcm2}, then any element of $\cH(K)$ is of the form 
$G(z)g(z)$ for some $g\in \cH(k_{\rm nc, Sz} I_{\cU})=H^{2}_{\cU}(\free)$ and since $H^{2}_{\cU}(\free)$
is invariant under $M_{z_j}^r$ for $j=1,\ldots,d$, it follows that $\cH(K)$ is $M_{z_j}^r$-invariant as well.
Representation \eqref{L-Koldecom} follows from \eqref{Sz-id}, \eqref{kcm2}, and \eqref{kernelL}.
Inequality \eqref{dec22} follows by combining \eqref{kernelL} and \eqref{adj-action}.
 \end{proof}
 
 \begin{remark}  \label{R:FRKHSvsNCRKHS}
 Given a formal power series $f(z) = \sum_{\alpha \in \free} 
f_{\alpha} z^{\alpha} \in \cY\langle \langle z \rangle \rangle$ in $d$ free indeterminates $z = (z_{1}, 
\dots, z_{d})$ with coefficients $f_{\alpha}$ coming from a Hilbert space $\cY$, one may 
substitute in a $d$-tuple 
$Z = (Z_{1}, \dots, Z_{d})$ of $n \times n$ matrices for the formal 
indeterminates $z = (z_{1}, \dots, z_{d})$ to obtain a $n \times n$ matrix
\begin{equation}   \label{series}
f(Z) = \lim_{N \to \infty} \sum_{\alpha \in \free\colon 0 \le 
|\alpha| \le N} f_{\alpha}  \otimes Z^{\alpha} \in \cY \otimes 
{\mathbb C}^{N \times N} \cong \cL({\mathbb C}^{n}, \cY^{n})
\end{equation}
on a domain corresponding 
to the set of $d$-tuples $Z$ when the series \eqref{series} converges
(here $Z^{\alpha} = Z_{i_{N}} \cdots Z_{i_{1}} \in {\mathbb C}^{N 
\times N}$ if $\alpha$ is the word of the form $\alpha = i_{N} \cdots 
i_{1}$ with each $i_{j} \in \{1, \dots, d\}$). This is a particular setting for what has come 
to be called a \textbf{noncommutative function} (see \cite{K-VV, 
AMc1, AMc2, BMV1}).  Similarly, if 
$$
K(z,w) = \sum_{\alpha, \beta \in \free} K_{\alpha, \beta} 
z^{\alpha} \overline{w}^{\beta^{\top}}  \in \cL(\cY)\langle \langle 
z, \overline{w} \rangle \rangle
$$
is a formal kernel, one may substitute in a  $d$-tuple $Z = (Z_{1}, 
\dots, Z_{d})$ along with a $d$-tuple $W = (W_{1}, \dots, W_{d})$ of 
$n \times n$ complex matrices to get an evaluation
$$
  K(Z,W) = \lim_{N \to \infty} \sum_{\alpha, \beta \in \free \colon 0 
  \le |\alpha|, |\beta| \le N}  K_{\alpha, \beta} \otimes 
  Z^{\alpha} (W^{*})^{\beta^{\top}} \in \cL(\cY) \otimes {\mathbb 
  C}^{n \times n} \cong \cL(\cY^{n} \otimes \cY^{n}),
$$
where the limit exists for $(Z,W)$ in some domain ${\displaystyle\bigcup_{n=0}^{\infty} \cD_{n} \times 
\cD_{n}} \subset  {\displaystyle\bigcup_{n=0}^{\infty} {\mathbb C}^{n \times n} \times {\mathbb C}^{n 
\times n}}$.  More generally, if $Z$ is a $d$-tuple of $n \times n$ 
matrices, $W$ is a $d$-tuple of $m \times m$ matrices, and $P$ is $n 
\times m$ matrix, one may define an evaluation 
$$
{\mathbf K}(Z,W)(P) = \sum_{\alpha, \beta \in \free} K_{\alpha, 
\beta}  \otimes Z^{\alpha}  P W^{* \beta^{\top}} \in \cL(\cY^{m}, 
\cY^{n})
$$
wherever the series converges.  In case $K$ is a positive formal 
kernel, the result is an object satisfying the definition of a 
\textbf{completely positive noncommutative kernel} in the sense of 
\cite{BMV1}.  Thus the theory of formal noncommutative power series 
and formal positive kernels can be identified as a particular setting 
for the theory of noncommutative functions and completely positive 
noncommutative kernels; the precise connections are discussed in 
\cite[Section 3.3]{BMV1}.  We leave open the question as to how one 
should define noncommutative Bergman spaces over more general 
noncommutative domains.
\end{remark}
 
\section{Weighted Hardy-Fock spaces} \label{S:Hardy} 
Let us say that the formal kernel $K(z, \zeta) =
 \sum_{\alpha, \beta} K_{\alpha, \beta} z^\alpha \overline{\zeta}^{\beta^\top}$
 is a {\em scalar radial} kernel if $K_{\alpha, \beta} = 0$ for $\alpha \ne \beta$ and furthermore
 $K_{\alpha, \alpha} = k_{\alpha} I_\cY$ for some scalar-valued function 
 ${\mathbf k} \colon \free \to {\mathbb C}$ (${\mathbf k} \colon \alpha \mapsto k_\alpha$).
  i.e., if $K(z, \zeta)$ has the  form
 $$  
 K(z, \zeta) = \sum_{\alpha \in \free}  (k_\alpha  I_\cZ) z^\alpha \overline{\zeta}^{\alpha^\top}
 $$
 where $\cY$ is a coefficient Hilbert space.   Often it is the case that $k_\alpha \ne 0$ for all 
$\alpha$ and it is more convenient to write
 $k_\alpha$ as $k_\alpha = \omega_\alpha^{-1}$ for a nonzero-valued scalar function
 $\bo \colon \free \mapsto {\mathbb C} \setminus \{0\}$.  In many examples the sequence 
$k_\alpha$ in fact depends only the
 length $|\alpha|$ of $\alpha$, in which case we say that $K$ is a {\em symmetrized radial kernel}.
 
\smallskip

 Specifically we shall consider symmetrized radial kernels arising in the following way.
 Starting with a positive sequence $\bo=\{\omega_j\}_{j\ge 0}$ and
a Hilbert space $\cY$, we denote by $H^2_{\bo,\cY}(\free)$ be the {\em weighted Hardy-Fock space}
\begin{equation}
H^2_{\bo,\cY}(\free) = \bigg\{\sum_{\alpha \in \free}
f_{\alpha}z^{\alpha}
\in \cY\langle\langle z\rangle\rangle\colon \; \|f\|^2:=
\sum_{\alpha \in \free} \omega_{|\alpha|}\cdot \|f_{\alpha}\|_{\cY}^{2} <\infty \bigg\}.
\label{18.1}
\end{equation}
The space $H^2_{\bo,\cY}(\free)$ is the NFRKHS with  noncommutative reproducing kernel
$k_{\bo}(z, \zeta) I_\cY$ with
\begin{equation} \label{kbo}
k_{ \bo}(z,\zeta)=\sum_{\alpha \in \free}\omega_{|\alpha|}^{-1}
z^{\alpha} \bzeta^{\alpha^{\top}}
\end{equation}
which is verified by the standard computation for $f(z) = \sum_{\alpha \in \free} 
f_{\alpha} z^{\alpha} \in H^2_{\bo,\cY}(\free)$ and $y \in \cY$:
\begin{align*}
    \langle f, k_{{\rm nc}, \bo}(\cdot, \zeta) y \rangle_{H^2_{\bo,\cY}(\free) \times
    H^2_{\bo,\cY}(\free)\langle \langle \bzeta \rangle \rangle} &
= \sum_{\alpha \in \free} \langle f(z), \,  \omega_{|\alpha|}^{-1} y
\, z^{\alpha} \rangle_{H^2_{\bo,\cY}(\free)} \zeta^{\alpha} \\
    & = \sum_{\alpha \in \free} \langle f_{\alpha}, y \rangle_{\cY} \zeta^{\alpha}
 = \langle f(\zeta), y \rangle_{\cY\langle \langle \zeta \rangle \rangle
    \times \cY}.
\end{align*}
In the specific context of the weighted Hardy-Fock space $H^2_{\bo,\cY}(\free)$, 
the tuple of the right coordinate-variable multipliers ${\bf M}^r_z$ (see \eqref{dec22}) will 
be denoted by ${\bf S}_{\bo,R}$:
\begin{equation}
{\bf S}_{\bo,R}=(S_{\bo,R,1},\ldots,S_{\bo,R,d}),\quad S_{\bo,R,j}: \, f(z)\mapsto f(z)z_j.
\label{18.2b}
\end{equation}
The standard inner-product computation gives the formula for adjoint operators
\begin{equation}   \label{Sboj*}
    S_{\bo,R,j}^{*} \colon \sum_{\alpha \in \free} f_{\alpha}
    z^{\alpha} \mapsto \sum_{\alpha \in \free}
    \frac{\omega_{|\alpha|+1}}{\omega_{|\alpha|}} f_{\alpha j} z^{\alpha}.
\end{equation}
We say that the weight sequence $\bo=\{\omega_j\}_{j\ge 0}$ is {\em admissible} if it is 
subject to the following conditions:
\begin{equation}
\omega_0=1 \quad\mbox{and}\quad 1\le
\frac{\omega_j}{\omega_{j+1}}\le M \quad\mbox{for all}\quad j\ge 0\quad\mbox{and some}\quad 
M\ge 1.
\label{18.2}
\end{equation}  
The first condition says that ${\mathbb C}$ embeds isometrically into the space of constant nc functions
in $\cH_{\bo, \cY}(\free)$, the second  condition means that the shift-operator tuple ${\bf S}_{\bo,R}$ 
is a row contraction and $S_{\bo,R,j}$ is left-invertible for $j=1,\ldots,d$. Indeed, it follows from \eqref{Sboj*} that
$$
\| S^{*}_{\bo,R,j} f \|^{2}_{H^2_{\bo,\cY}} = \sum_{\alpha \in \free} 
\frac{\omega_{|\alpha|+1}^{2}}{ \omega_{|\alpha|}} \| f_{\alpha 
j} \|^{2}
$$
and hence, for $f(z) = \sum_{\alpha \in \free} f_{\alpha} z^{\alpha}$, we have
\begin{align*}
 \sum_{j=1}^{d} \|  S_{\bo,R,j}^{*} f\|^{2}_{H^2_{\bo,\cY}} &  =
\sum_{\alpha \ne \emptyset}  \frac{ 
\omega_{\alpha}^{2}}{\omega_{|\alpha|-1}} \| f_{\alpha} \|^{2}
= \sum_{\alpha \ne \emptyset} \omega_{|\alpha|} \left( 
\frac{\omega_{|\alpha|}}{\omega_{|\alpha|-1}} \right)  \| f_{\alpha}\|^{2} 
\\  &  \le \sum_{\alpha \in \free} \omega_{|\alpha|} \| f_{\alpha} 
\|^{2} = \| f \|^{2}_{H^2_{\bo,\cY}},
\end{align*}
and similarly, 
$$
\| S_{\bo,R,j}f \|^{2}_{H^2_{\bo, \cY}} = 
\sum_{\alpha \in \free} \omega_{|\alpha| +1} \| f_{\alpha} \|^{2} \ge 
\frac{1}{M} \sum_{\alpha \in \free} \omega_{|\alpha|} \| f_{\alpha} \|^{2}=\frac{1}{M}\|f\|^2_{H^2_{\bo, \cY}}
$$
for $j=1,\ldots,d$. Combining the row-contractive property of ${\bf S}_{\bo,R}$ with 
Definition \ref{kercon} we arrive at the following.

\begin{remark}
\label{R:examples}
The kernel \eqref{kbo} based on an admissible weight sequence $\bo$ is a contractive nc positive 
kernel. Furthermore, the sequence $\bmu_n=\{\mu_{n,j}\}$ \eqref{1.6pre} of reciprocal binomial coefficients is 
an admissible weight sequence which is strictly decreasing for $n> 1$. 

\smallskip

It is worth mentioning the following example $\bmu_\rho$ with a continuous real parameter $\rho > 1$ 
which interpolates the weights $\bmu_n$ with discrete parameter $n=1, 2, \dots$.  Namely, for $\rho >1$ real,
we define  $\bo = \bmu_\rho = \{ \mu_{\rho, k} \}_{ k \ge 0}$ where we set
 \begin{equation}   \label{bmurho}
   \mu_{\rho, k} = \frac{k!}{ \rho (\rho + 1) \cdots (\rho + k - 1)} = 
   \frac{ k! \Gamma(\rho)}{\Gamma(\rho + k )}.
 \end{equation}
 One can check that $\bmu_\rho$ meets the admissibility conditions \eqref{18.2} and that indeed
 $\mu_{\rho, k} =  \mu_{n,k}$ in case $\rho = n$ is a positive integer.   Therefore the weight sequence
 \eqref{bmurho} generates a formal nc reproducing kernel Hilbert space $\cH(k_{\bmu_\rho})$ 
 with formal kernel given by
  \begin{equation}  \label{kbmurho}
  k_{\bmu_\rho}(z, \zeta) = \sum_{\alpha \in \free}  \mu_{\rho, |\alpha|}^{-1}   z^\alpha \overline{\zeta}^{\alpha^\top}.
  \end{equation}
  The associated weighted Hardy
 spaces $H^2_{\bmu_\rho}(\free)$ are mentioned as examples of the more general spaces
 $H^2_\bo(\free)$ for the single-variable case $d=1$, in which case we simply write
 $H^2_\bo$ rather than $H^2_\bo({\mathbb F}^+_1)$.
 
\smallskip

In particular, the choice $\bo = {\boldsymbol \mu}_n$ leads to the identification of a weighted Hardy-Fock 
space with a weighted Bergman-Fock space:  $H^2_{{\boldsymbol \mu}_n,\cY}(\free)= \cA_{n, \cY}(\free)$. In short, 
formulas \eqref{ncsz}, \eqref{nc-n-Berg-ker}, \eqref{kbmurho}, \eqref{kbo} give examples  of increasing 
generality of symmetrized radial kernels which are also contractive kernels. As for shift operator-tuples on weighted Bergman-Fock spaces, we shall write $\bS_{n,R}$ rather than $\bS_{\bmu_n,R}$ and we shall also write 
$\bS_{1,R}$ for the shift operator tuple on the unweighted Fock space $H^2_\cY(\free)$.
 \end{remark}
 
We conclude this section by introducing a multidimensional system with evolution along $\free$ and based
on an admissible weight $\bo=\{\omega_j\}_{j\ge 0}$
\begin{equation}    \label{1.31prehardy}
\Sigma_{\{\bU_\alpha\}, \bo} \colon
\left\{ \begin{array}{ccc}
x(1 \alpha) &= &
\frac{\omega_{|\alpha|}}{\omega_{|\alpha|+1}}A_{1}x(\alpha)+\frac{1}{\omega_{|\alpha|+1}}B_{1,\alpha} u(\alpha) \\
\vdots &   & \vdots  \\
x(d \alpha) & = &
\frac{\omega_{|\alpha|}}{\omega_{|\alpha|+1}}A_{d}x(\alpha)+\frac{1}{\omega_{|\alpha|+1}}B_{d,\alpha} u(\alpha) \\
y(\alpha) &  = & Cx(\alpha)+\frac{1}{\omega_{|\alpha|}}D_\alpha u(\alpha)\end{array}
\right.
\end{equation}
with the the same family of connection matrices $\{{\bf U}_\alpha\}_{\alpha\in\free}$ as in \eqref{coll} but with additional $\alpha$-dependent
weights in the system equations determined by the weight sequence $\bo$.
Upon running the system  \eqref{1.31prehardy}  with a fixed initial condition $x(\emptyset)=x\in\cX$ we get recursively
the formulas for $x(\alpha)$ and $y(\alpha)$ (the same as \eqref{1.32pre}, \eqref{1.33pre} but with $\omega_{|\alpha|}$
instead of $\mu_{n,|\alpha|}$). Subsequent application of the noncommutative $Z$-transform \eqref{1.25pre} eventually leads us to
the formulas
\begin{align}
& \widehat y(z)=  \sum_{\alpha\in\free} \omega_{|\alpha|}^{-1}
\bigg(C\bA^\alpha x +\sum_{\alpha^{\prime\prime}j\alpha^\prime=\alpha}
{\bf A}^{\alpha^{\prime\prime}}B_{j,\alpha^\prime}u(\alpha^\prime)+D_{\alpha}u(\alpha)\bigg)z^\alpha =  \notag\\
& \sum_{\alpha\in\free}\big(\omega_{|\alpha|}^{-1}C\bA^\alpha  x\big)z^\alpha
+\sum_{\alpha\in\free}\bigg(\omega_{|\alpha|}^{-1}D_\alpha+\sum_{j=1}^d\sum_{\alpha^\prime\in\free}
\omega^{-1}_{|\alpha|+|\alpha^\prime|+1}C\bA^{\alpha^\prime}B_{j,\alpha}z^{\alpha^\prime j}\bigg)z^\alpha u(\alpha)
\notag\\
&  =\cO_{\bo,C,{\bf A}}x+\sum_{\alpha\in\free}\Theta_{\bo, \bU_\alpha}(z)z^{\alpha}u(\alpha),\label{1.36prehardy}
\end{align}
where the first term on the right presents the $\bo$-observability operator
\begin{equation}   \label{ncobsophardy}
    \cO_{\bo,C,\bA} x =\sum_{\alpha \in \free} \omega_{|\alpha|}^{-1} (C \bA^{\alpha} x) z^{\alpha}
\end{equation}
associated with output pair $(C,\bA)$ and where
\begin{equation}
\Theta_{\bo, \bU_\alpha}(z)=\omega_{|\alpha|}^{-1}D_\alpha+\sum_{j=1}^d\sum_{\alpha^\prime\in\free}
\omega^{-1}_{|\alpha|+|\alpha^\prime|+1}C\bA^{\alpha^\prime}B_{j,\alpha}z^{\alpha^\prime j}
\label{thetahardy}
\end{equation}
is the family of transfer functions indexed by $\alpha \in \free$. Making use of power series
$$
R_\bo(\lambda)=R_{\bo,0}(\lambda)
=\sum_{j=0}^\infty \omega_j^{-1}\lambda^j\quad\mbox{and}\quad
R_{\bo,k}(\lambda)=\sum_{j=0}^\infty \omega_{j+k}^{-1}\lambda^j \; \; (k\ge 1)
$$
extending those in \eqref{1.8pre} and operators $A$ and $\widehat{B}_\alpha$ as in \eqref{coll}, 
one can write the formulas for \eqref{ncobsophardy} and \eqref{thetahardy} as 
\begin{align}
\cO_{\bo,C,\bA} x &=CR_\bo(Z(z)A),\label{obsreal}  \\
\Theta_{\bo, \bU_\alpha}(z)&=\omega_{|\alpha|}^{-1}D_\alpha+CR_{\bo,|\alpha|+1}(Z(z)A)Z(z)
\widehat{B}_\alpha,
\label{thetareal}
\end{align}
i.e., in the form similar to \eqref{ncobsop} and \eqref{1.37pre}, respectively.

\smallskip

Now we may introduce an $\bo$-output stable pair $(C,\bA)$ as one for which the $\bo$-observability operator
\eqref{ncobsophardy} is bounded from $\cX$ to $H^2_{\bo,\cY}(\free)$. In this case, we see from \eqref{18.2aa} and \eqref{18.1}
that for every $x\in\cX$, 
$$
\langle \cG_{\bo,C,\bA}x, \, x\rangle_{\cX}=\|\cO_{\bo,C,\bA}x\|^2_{H^2_{\bo,\cY}(\free)}=
\sum_{\alpha\in\free}\omega_{|\alpha|}^{-1}\cdot \|C\bA^{\alpha}x\|^2_{H^2_{\bo,\cY}(\free)}.
$$
which justifies the representation of the $\bo$-observability gramian
$$
\cG_{\bo,C,\bA}:=\cO_{\bo,C,\bA}^*\cO_{\bo,C,\bA}=
\sum_{\alpha\in\free}\omega_{|\alpha|}^{-1}\bA^{*\alpha^\top}C^*C\bA^{\alpha}
$$
by the weakly (and therefore, strongly) convergent series. Observability operators and their range spaces will be 
studied in detail in Chapter 4.

\chapter{Contractive multipliers}\label{C:contrmult}

\section{Contractive multipliers in general}   \label{S:gen-contr-mult}
Let $K'$ and $K$ be noncommutative reproducing kernels with values in 
$\cL(\cU)$ and $\cL(\cY)$ respectively, and
let $\cH(K')$ and $\cH(K)$ be the associated 
NFRKHSs.   As we shall have many notions of "inner multiplier", we always 
use the term with some adjective to specify the type of inner under
consideration.  We note that the term "inner" has been used in the literature
for what we are here calling "strictly inner" as well as "McCT-inner".

\begin{definition}   \label{D:mult}
 A formal power series $\Theta\in\cL(\cU,\cY)\langle \langle z\rangle\rangle$
is called a {\em multiplier from  $\cH(K')$ to $\cH(K)$}
if the operator $M_\Theta$ of multiplication by $\Theta$  
$$
        M_{\Theta} \colon f(z) \mapsto \Theta(z) \cdot f(z)=
        \sum_{v \in \free} \bigg( \sum_{\alpha, \beta \colon \alpha
        \beta = v} \Theta_{\alpha} f_{\beta} \bigg) z^{v}
$$
is bounded from $\cH(K')$ to $\cH(K)$. The multiplier $\Theta$ 
is called {\em contractive}, {\em McCT-inner} (suggesting the authors of the seminal paper
\cite{MCT}) or {\em strictly inner}
if the operator $M_\Theta \colon \cH(K')\to \cH(K)$ is a contraction,
a partial isometry or an isometry, respectively.  
\end{definition}

The following result for the special case where $K' = k_{\rm nc, Sz} 
\otimes I_{\cU}$ and $K = k_{\rm nc, Sz} \otimes I_{\cY}$ appears in 
\cite[Theorem 3.15]{NFRKHS} and in \cite[Theorem 3.1]{BBF3}.  We give 
a simple direct proof based on an adaptation of the proof of 
\cite[Proposition 2.10]{PapaBear} where a more complicated two-sided 
commutative setting is studied.   

\begin{proposition}  \label{P:2.3}
Let $K \in \cL(\cY)\langle \langle z, \bzeta \rangle \rangle$ and $K' \in \cL(\cU)\langle \langle z, 
\bzeta \rangle \rangle$ be two positive formal kernels and let $\Theta$ be a formal power series in 
$\cL(\cU,\cY)\langle \langle z\rangle\rangle$. Then: 
\begin{enumerate}
   \item  $\Theta$ is a
contractive multiplier from $\cH(K')$ to $\cH(K)$ if and only if
\begin{equation}
K_\Theta(z,\zeta)=K(z,\zeta)-\Theta(z)K'(z,\zeta)\Theta(\zeta)^*\in\cL(\cY)\langle \langle z, \bzeta 
\rangle \rangle
\label{jan1b}
\end{equation}
is a positive formal kernel.
\item $\Theta$ is a coisometric multiplier from $\cH(K')$ to $\cH(K)$ 
if and only if $K_{\Theta} = 0$.
\end{enumerate}
\end{proposition}
\begin{proof} If $\Theta$ is a contractive 
multiplier from $\cH(K')$ to $\cH(K)$, then for all $f \in \cH(K')$ and $y \in \cY$, we have
\begin{align*}
&   \langle f, (M_{\Theta}^{*} \otimes I_{{\mathbb C}\langle \langle 
    \bzeta \rangle \rangle}) K( \cdot, \zeta) y \rangle_{\cH(K') \times 
    \cH(K')\langle \langle \bzeta \rangle \rangle} \\
& = \langle \Theta(\cdot) f(\cdot),\, K(\cdot, \zeta) y \rangle_{\cH(K) \times \cH(K) 
  \langle \langle \bzeta \rangle \rangle} \\
  &= \langle \Theta(\zeta) f(\zeta),\,  y 
  \rangle_{\cY\langle \langle \zeta \rangle \rangle \times \cY} 
  \quad\text{(by \eqref{reprod-prop'})}\\
 & = \langle f(\zeta), \Theta(\zeta)^{*} y \rangle_{\cU \langle 
  \langle \zeta \rangle \rangle \times \cU \langle \langle \bzeta 
  \rangle \rangle} \quad\text{(by \eqref{MSadj})} \\
 & = \langle f,\, K'(\cdot, \zeta) \Theta(\zeta)^{*} y 
  \rangle_{\cH(K') \times \cH(K')\langle \langle \bzeta \rangle 
  \rangle} \quad\text{(by \eqref{genreprod}}
\end{align*}
from which we conclude that
\begin{equation}  \label{MTheta*act}
\big( M_{\Theta}^{*} \otimes I_{{\mathbb C} \langle \langle \bzeta \rangle 
\rangle} \big) K(\cdot, \zeta) y = K'(\cdot, \zeta) \Theta(\zeta)^{*}y.
\end{equation}
Equating coefficients of $\bzeta^{\beta^{\top}}$ in \eqref{MTheta*act} 
and using the notation \eqref{Kbeta} then gives
\begin{equation}   \label{MTheta*}
M_{\Theta}^{*} K_{\beta}(\cdot) y = \sum_{\gamma, \gamma' \colon \gamma' 
\gamma = \beta} K'_{\gamma}(\cdot) \Theta_{\gamma'}^{*} y.
\end{equation}
The fact that $\| M_{\Theta}^{*} \| \le 1$ then implies that
$$
 \big \|\sum_{j=1}^{N} K_{\beta_{j}}(\cdot) y_{j} \big\|^{2}_{\cH(K)} - 
  \big\| \sum_{j=1}^{N} M_{\Theta}^{*} K_{\beta_{j}}(\cdot) y_{j} 
  \big\|^{2}_{\cH(K')} \ge 0,
$$
or, in terms of notation \eqref{formalkernel},
\begin{equation}   \label{MScontr}
\sum_{i,j = 1}^{N} \bigg( \langle K_{\beta_{i},\beta_{j}} y_{j}, y_{i} 
\rangle_{\cY} - \bigg\langle
\sum_{{\scriptsize\begin{array}{c}
\gamma_{j}, \gamma_{j}'\colon \gamma_{j}' \gamma_{j} = \beta_{j},\\
 \gamma_{i}, \gamma_{i}' \colon \gamma_{i}' \gamma_{i} = \beta_{i}\end{array}}}
\Theta_{\gamma_{i}^\prime}K'_{\gamma_{i}, \gamma_{j}} \Theta^{*}_{\gamma_{j}'} y_{j}, 
y_{i} \bigg\rangle_{\cY} \bigg) \ge 0
\end{equation}
for all choices of $N\in\mathbb N$ and 
$\beta_{1}, \dots, \beta_{N} \in \free$, $y_{1}, 
\dots, y_{N} \in \cY$.  On the other hand, the 
formal kernel \eqref{jan1b} has  power series expansion
\begin{equation}   \label{KTheta}
    K_{\Theta}(z, \zeta) = 
    \sum_{\alpha, \beta} \bigg( K_{\alpha, \beta} - 
\sum_{\gamma, \gamma', \delta, \delta' \colon \gamma' \gamma = 
\alpha,\, \delta' \delta = \beta} \Theta_{\gamma'} K^\prime_{\gamma, \delta} 
\Theta^{*}_{\delta'} \bigg) z^{\alpha} \bzeta^{\beta^{\top}}.
\end{equation}
It now becomes clear that the condition \eqref{MScontr} is just the 
criterion \eqref{coef-pos} for the positivity of the formal kernel 
$K_{\Theta}$.  Note that $M_{\Theta}$ being coisometric means that 
$M_{\Theta}^{*}$ is isometric, so \eqref{MScontr} holds with equality.  
This then forces the kernel \eqref{KTheta} to be zero.  This completes 
the proof of  necessity in Proposition \ref{P:2.3}.

\smallskip

To prove sufficiency in part (1) of Proposition \ref{P:2.3}, proceed as 
follows.  The formula \eqref{MTheta*}  for the action of 
$M_{\Theta}^{*}$ on kernel functions suggests that we define an 
operator $T_{0}$ mapping the span of kernel functions in $\cH(K)$ to 
the span of kernel functions in $\cH(K')$ via
$$
  T_{0} \colon \sum_{j=1}^{N} K_{\beta_{j}}\, y_{j} \mapsto
  \sum_{j=1}^{N} \sum_{\gamma_{j}, \gamma_{j}' \colon \gamma_{j}' 
  \gamma_{j} = \beta_{j}} K'_{\gamma_{j}} \Theta_{\gamma'_{j}}^{*} 
  y_{j}.
$$
The computation in the first part of the proof read backwards tells us
that the assumption $K_{\Theta}$ is a positive formal kernel 
implies that $T_{0}$ is contractive.  Hence $T_{0}$ extends uniquely 
by continuity to a well-defined contraction operator from $\cH(K)$ into $\cH(K')$.  
Furthermore, the action of $T: = T_{0} \otimes I_{{\mathbb C} \langle 
\langle \bzeta \rangle \rangle}$ is given by \eqref{MTheta*act}.
We then compute, for $f \in \cH(K')$ and $y \in \cY$,
\begin{align*}
    \langle T_{0}^{*} f,\, K(\cdot, \zeta) y \rangle_{\cH(K) \times 
    \cH(K)\langle \langle \bzeta \rangle \rangle} & =
    \langle f, T K(\cdot, \zeta) y \rangle_{\cH(K^\prime) \times \cH(K^\prime) 
    \langle \langle \bzeta \rangle \rangle}  \\
 & = \langle f, K'(\cdot, \zeta) \Theta(\zeta)^{*} y \rangle_{\cH(K') 
 \times \cH(K') \langle \langle \bzeta \rangle \rangle} \\
 & = \langle f(\zeta),  \Theta(\zeta)^{*} y \rangle_{\cU\langle \langle 
 \zeta \rangle \rangle \times \cU\langle \langle \bzeta \rangle 
 \rangle} \quad\text{(by \eqref{genreprod})} \\
 & = \langle \Theta(\zeta) f(\zeta), y \rangle_{\cY\langle \langle 
 \zeta \rangle \rangle \times \cY} \quad\text{(by \eqref{MSadj})}
\end{align*}
from which we conclude that $T_{0}^{*} = M_{\Theta}$ and hence $\| 
M_{\Theta} \| \le 1$.
\end{proof}
We note that a homological-algebra Hilbert-module proof of
part (2) in Proposition \ref{P:2.3} has been given by Douglas-Misra-Sarkar for
the classical (non-formal) setting \cite[Theorem 1]{DMS}.

\begin{proposition}\label{P:2.4}
Let $\cH(K_j)$ be the NFRKHS with $\cL(\cU_j)$-valued formal reproducing kernel 
$K_j(z,\zeta)$ for $j=1,\ldots,n$ and let
\begin{equation}
K(z,\zeta)=\sum_{j=1}^nF_j(z)K_j(z,\zeta)F_j(\zeta)^*
\label{jan1c}
\end{equation}
where $F_j\in \cL(\cU_j,\cX)\langle \langle z\rangle\rangle$. 
Then $F(z)=\begin{bmatrix}F_1(z) & \cdots &F_n(z)\end{bmatrix}$ is a 
coisometric multiplier from $\oplus_{j=1}^n\cH(K_j)$ to $\cH(K)$.
\end{proposition}

\begin{proof}  We view $\oplus_{j=1}^{n} \cH(K_{j})$ as a NFRKHS with 
   block-diagonal reproducing kernel 
$$
K(z, \zeta) = \begin{bmatrix} K_{1}(z, \zeta) & & 0 \\ & 
    \ddots & \\ 0 & & K_{n}(z, \zeta)\end{bmatrix}.  
$$
As an application of part (2) of Proposition \ref{P:2.3}, we see that the identity 
    \eqref{jan1c} is the precise condition required for the map
    $$
    M_{F} \colon \sbm{ f_{1}(z) \\ \vdots \\ f_{n}(z) } \mapsto 
    \begin{bmatrix} F_{1}(z) & \cdots & F_{n}(z) \end{bmatrix}
	 \sbm{ f_{1}(z) \\ \vdots \\ f_{n}(z) }
$$
to be a partial isometry form $\oplus_{j=1}^{n} \cH(K_{j})$ onto $\cH(K)$.
\end{proof}
\noindent
Given a contractive multiplier $F$ from $\cH(K')$ to $\cH(K)$, we let $M_F\in\cL(\cH(K'),\cH(K))$ 
denote the multiplication operator
$$
  M_F \colon f(z) \mapsto F(z) \cdot f(z).
$$
We shall often view the linear space $\cM: = \operatorname{Ran} M_F \subset \cH(K)$ as a Hilbert space in its own right 
with its own norm
$$
\| M_F f \|_{\cM}= \inf \{ \| g \| \colon g \in \cH(K^\prime) \text{ such that } M_F g = M_F f \} = \| Q f \|_{\cH(K^\prime)},
$$
where $Q$ is the orthogonal projection of $\cH(K^\prime)$ onto $({\rm Ker} M_F)^\perp:=\cH(K^\prime)\ominus {\rm Ker} M_F$.

\subsection{Lifted-norm and pullback spaces and their Brangesian complements}   \label{S:Brangesian}

More generally, if $X \colon \cK' \to \cK$ is any contraction operator, let us write $\cH^\ell(X)$ for the {\em lifted norm space}
$$
 \cH^\ell(X) = \operatorname{Ran} X\quad\text{with}\quad \| X f \|_{\cH^\ell(X)} : = 
 \inf  \{ \| g \|_{\cK'} \colon X g = X f \}.
$$
It then follows that the inclusion map $\iota \colon \cM \to \cH(K)$ is contractive since
$$
   \| M_F f \|_{\cM} = \| M_F Q f \|_{\cM}  = \| Q f \|_{\cH(K')}  \ge \| M_F Q f ||_{\cH(K)} = \| M_F f \|_{\cH(K)}.
$$
Conversely, if $\cM$ is any Hilbert space contractively included in  an ambient Hilbert space $\cH$, then the inclusion map
$\iota \colon \cM \to \cK$ is contractive and we may view $\cM$ as equal to $\operatorname{Ran} \iota$.  Let 
$X \colon  \cK' \to \cK$ be any operator such that $X X^* = \iota \iota^*$.  Then $\operatorname{Ran} X
= \operatorname{Ran} \iota = \cM$ and we have
$$
  \| X f \|_\cM  = \inf \{  \| g \|_{\cK'} \colon X g = X f \}  \ge \| X f \|_\cK.
$$
In this case there is a complementary space, denoted by $\cM^{[\perp]}$ and called the {\em Brangesian complementary
space} (to $\cM$ relative to  $\cK$) defined as $\cH^\ell( (I - X X^*)^{\frac{1}{2}})$, i.e., 
$$
\cM^{[\perp]} =  \operatorname{Ran} (I - X X^*)^{\frac{1}{2}}
$$
with lifted norm 
$$
 \|(I - X X^*)^{\frac{1}{2}} f \|
  = \inf \{ \| g \|_{\cK} \colon (I - X X^*)^{\frac{1}{2}}g =( I - X X^*)^{\frac{1}{2}} f \}.
$$
We note that the space $\cM^{[\perp]}$ is independent of the choice of operator $X$ used to represent the Hilbert space
$\cM = \cH^\ell(X)$ since $\cM^{[\perp]}$ can be characterized intrinsically in terms of $\cM$ via the formula
$$
\cM^{[\perp]} = \left\{ g  \in \cK \colon  \| g \|_{\cM^{[\perp]}} : =
\sup \{ \| Xf + g \|^2_\cK - \| f\|^2_\cM \colon f \in \cM \} < \infty \right\}.
$$

It is sometimes more convenient to view a lifted-norm space $\cH^\ell(X)$ as a pullback space $H^p(\Pi)$ based
on the positive-semidefinite operator $\Pi = X X^*$ and defined as follows.  Given any positive semidefinite operator
$\Pi$ on $\cX$ (e.g., $\Pi = X X^*$), we define $\cH^p(\Pi)$ as the completion of $\operatorname{Ran} \Pi$ in the inner product
$$
   \langle \Pi x, \Pi y \rangle_{\cH^p(\Pi)} = \langle \Pi x, y \rangle_\cX.
$$
It is then not hard to show that the completion can be identified explicitly as 
$$
   \cH^p(\Pi) = \operatorname{Ran} \Pi^{\frac{1}{2}} = \operatorname{Ran} X \text{ (if $\Pi = X X^*$)}
$$
with $\operatorname{Ran} \Pi$ as a dense subset of $\cH^p(\Pi)$.  In particular for the computation
$$
 \langle X X^* f, XX^* g \rangle_{\cH^\ell(X)} = \langle X^*f, X^* g \rangle_\cX = \langle X X^* f, g \rangle_\cX
 = \langle X X^* f, X X^* g \rangle_{\cH^p(X X^*)}
$$
shows that $\cH^p(X X^*)$ is isometrically equal to $\cH^\ell(X)$.  In particular, if $\cM = \cH^\ell(X)$ is a contractively
included subspace of $\cX$, then its Brangesian complement $\cM^{[\perp]} = \cH^\ell((I - X X^*)^{\frac{1}{2}}$
is conveniently alternatively characterized as $\cM^{[\perp]} = \cH^p(I - X X^*)$.  Note that if $X$ is a partial isometry
and $\Pi = X X^*$ is a projection, then $\cM$ is isometrically included in $\cX$ and the Brangesian complement
$\cM^{[\perp]} = \cH^p(I - \Pi)$ collapses to the standard Hilbert-space orthogonal complement
$\cM^{[\perp]} = \cM^\perp$.
For more complete details on this material, we refer to \cite{BB-HOT, Sarason-deB} which has its origins in the work of
de Branges-Rovnyak \cite{dBR1, dBR2}.

\smallskip

Let us now return to the case where $\cK' = \cH(K')$ and $\cK = \cH(K)$ are noncommutative formal reproducing kernel
Hilbert spaces and the operator $X \colon \cH(K') \to \cH(K)$ is a contractive multiplication operator $X = M_F$.
Then we may view $\cM = \cH^\ell(M_F)$ as  being the lifted-norm Hilbert space induced by the contraction operator
$M_F$. In particular, $\cM$ is contractively included in $\cH(K)$, and we have the following result:

\begin{theorem}   \label{T:charNFRKHS}
Given a contractive multiplier $F \in \cL(\cU, \cY)\langle \langle z \rangle \rangle$ from
$\cH(K')$ to $\cH(K)$, set $\cM$ equal to the lifted norm space $\cM = \cH^\ell(M_F)$.  Then:
\begin{enumerate}
\item   $\cM$ is itself a NFRKHS with reproducing kernel $K_\cM(z, \zeta)$ given by
\begin{equation}  \label{K-M}
  K_\cM(z, \zeta) = F(z) K'(z, \zeta) F(\zeta)^*.
\end{equation}

\item The Brangesian complement $\cM^{[\perp]}$ is also a NFRKHS with reproducing kernel $K_{\cM^{[\perp]}}$ given by
\begin{equation}
K_{\cM^{[\perp]}}(z, \zeta)=  K(z, \zeta) - K_{\cM}(z, \zeta)= K(z, \zeta) - F(z) K'(z, \zeta) F(\zeta)^*.
  \label{K-Mperp}
\end{equation}
\end{enumerate}
\end{theorem}

\begin{proof}
As the elements of $\cM$ consist of formal power series, in order to show that $\cM$ is a NFRKHS 
it suffices to verify that the evaluation functional $\Phi_\alpha \colon f(z) \mapsto f_\alpha$ 
is continuous (i.e., bounded as a linear operator from $\cM$ to $\cY$) for each $\alpha \in \free$. Since 
$\Phi_\alpha$ is defined on all of $\cM$ which is a complete space in its norm, it is enough to show that 
$\Phi_\alpha$ is a closed operator. To this end, we take an arbitrary sequence $\{f_n\}\subset \cM$ converging to $f\in\cM$ 
and assume that $\Phi_{\alpha}f_n$ converges to $y\in\cY$. 
Since $\cM$ is contained contractively in $\cH(K)$ (since $\cM = \cH^\ell(M_F)$ and $F$ is a contractive multiplier), 
we have
$$
   \| f_n - f \|_{\cH(K)} \le \| f_n - f \|_\cM.
 $$
Since $\{f_n\}$ converges to $f$ in $\cM$, it also converges to $f$ in $\cH(K)$.  
As $\cH(K)$ is a NFRKHS, it follows that $\Phi_\alpha f_n$ converges to $f_\alpha\in\cY$.
By uniqueness of limits in $\cY$ it now follows that $y = f_\alpha$ meaning that the operator $\Phi_\alpha:\cM\to\cY$ is closed.
The fact that $\cM^{[\perp]}$ is also a NFRKHS follows from the fact that $\cM^{[\perp]}$ is also a lifted-norm space
induced by a contraction operator, namely 
$$
\cM^{[\perp]} = \cH^\ell((I - M_F M_F^*)^{\frac{1}{2}})
$$
 together with the
general principle established in the preceding paragraph:
{\em any Hilbert space $\cN$ contractively included in a NFRKHS is itself a NFRKHS}.   

\smallskip

It remains to verify the expressions \eqref{K-M} and \eqref{K-Mperp} for the reproducing kernels of $\cM$ and $\cM^{[\perp]}$.
As a consequence of the formula
\eqref{MTheta*act} and the discussion of the pullback space preceding the statement of the theorem and making use
of the fact that we also have $\cM= \cH^p(M_F M_F^*)$, we see that 
\begin{align*}
& \big\langle \big( (M_F M_F^* \otimes I_{{\mathbb C}\langle \langle \overline{\zeta'} \rangle \rangle})
K(\cdot, \zeta') y' \big)(\zeta), \, y \big\rangle_{\cY\langle \langle \zeta, \overline{\zeta'} \rangle \rangle
\times \cY} \\
& = \big\langle \big(M_F M_F^* \otimes I_{{\mathbb C} \langle \langle \overline{\zeta'} \rangle \rangle}\big) K(\cdot, \zeta') y',
\, K(\cdot, \zeta) y  \big\rangle_{\cH(K)\langle \langle \overline{\zeta'} \rangle \rangle \times \cH(K) 
\langle \langle \overline{\zeta} \rangle \rangle}  \\
& =  \big\langle \big(M_F M_F^* \otimes I_{{\mathbb C} \langle \langle \overline{\zeta'} \rangle \rangle}\big) K(\cdot, \zeta') y',
\, K(\cdot, \zeta) y \big\rangle_{\cH(K)\langle \langle \overline{\zeta'} \rangle \rangle \times \cH(K) 
\langle \langle \overline{\zeta} \rangle
\rangle}  \\
& =  \big\langle \big(M_F M_F^* \otimes I_{{\mathbb C} \langle \langle \overline{\zeta'} \rangle \rangle}\big)
K(\cdot, \zeta') y',
\,  \big(    M_F M_F^* \otimes I_{{\mathbb C} \langle \langle \overline{\zeta} \rangle \rangle}  \big)
K(\cdot, \zeta) y \big\rangle_{\cM \langle \langle \overline{\zeta'} \rangle \rangle \times 
\cM \langle \langle \overline{\zeta} \rangle \rangle}.
\end{align*}
Since elements of the form $( M_F M_F^* )K_\beta(\cdot) y' $ ($ \beta \in \free$, $y' \in \cY$) span a dense set in $\cM$, 
we may then conclude by \eqref{MTheta*act} again that
$$
K_{\cM}(z, \zeta) y = \big( ( M_F M_F^* \otimes I_{{\mathbb C}\langle \langle \overline{\zeta} \rangle \rangle})
K(\cdot, \zeta) y \big) (z)= F(z) K'(z, \zeta) F(\zeta)^*  y
$$
in agreement with \eqref{K-M}. The formula \eqref{K-Mperp} for $K_{\cM^{[\perp]}}$ now follows by exactly the same argument as for
$K_\cM(z, \zeta)$ with the substitution that $\cM^{[\perp]} = \cH^p(I - M_F M_F^*)$ rather than
$\cM = \cH^p(M_F M_F^*)$.
\end{proof}

\section{Contractive multipliers between Fock spaces}  \label{S:con-mult}
Contractive multipliers from $H^2_{\cU}(\free)$ to $H^2_{\cY}(\free)$ are well known.
To place these into the context of related results from \cite{NFRKHS, BBF3}, recall that 
an operator-tuple ${\bf A}=(A_1,\ldots,A_d)\in\cX^d$ is called {\em strongly stable} if 
$$
\lim_{N \to \infty} \sum_{\alpha \in \free \colon |\alpha| = N} \|{\mathbf
           A}^{\alpha} x\|^{2} \to 0 \quad \text{for all} \; \; x \in {\mathcal X},
$$
where $\bA^\alpha$ is defined by ${\mathbf A}^{\alpha}$ is defined according to \eqref{1.21pre}.
 Furthermore, a pair $(C,\bA)$ consisting of 
an operator $C\in\cL(\cY,\cX)$ and a operator-tuple ${\bf A}=(A_1,\ldots,A_d)$ is called {\em output-stable}
if the observability operator 
\begin{equation}
\cO_{C,{\bf A}}: \, x\mapsto \sum_{\alpha \in\free}(C{\bf A}^\alpha x)z^\alpha
\label{1.27preagain}
\end{equation}
is bounded from $\cX$ into $H^2_{\cY}(\free)$. The following realization result states that any such multiplier appears 
as the transfer function of a unitary (conserva\-tive) system \eqref{1.20pre};
see \cite[Theorem 3.11]{NFRKHS}, \cite[Proposition 4.1.3]{Cuntz-scat}, \cite[Proposition 3.2]{BBF3}.

\begin{theorem}  
\label{T:NC1} 
Given $G(z)$ a contractive multiplier from $H^{2}_{\cU}(\free)$ to $H^{2}_{\cY}(\free)$,
there exists a Hilbert space $\cX$ and a unitary connection operator $\bU$ of the form \eqref{1.20}
so that $G(z)$ can be realized as a formal power series in the form
\begin{equation}   
\label{NCrealization}
G(z)=D +\sum_{j=1}^{d} \sum_{\alpha \in\free} C \bA^{\alpha}B_{j} z^{\alpha j} = D + C (I - Z(z) A)^{-1} Z(z) B.
\end{equation}
Conversely, if $G(z)$ has a realization as in \eqref{NCrealization} with a contractive connection matrix $\bU$, then $G$ is a 
contractive multiplier from $H^{2}_{\cU}(\free)$ to $H^{2}_{\cY}(\free)$.
\end{theorem}

The next result gives some finer system-theoretic structure associated with a contractive connection matrix $\bU$.

\begin{theorem} \label{1.29pre}
Assume that the operator connection matrix ${\bf U}$ \eqref{1.20} is contractive. Then
(a) the pair $(C,\bA)$ is output-stable, (b) the observability operator $\cO_{C,{\bf A}}: \,
\cX\to H^{2}_{\cY}(\free)$ (see \eqref{1.27preagain}) is a contraction, and (c)
the transfer function  $\Theta_{\bf U}$ (see \eqref{1.28pre}) is a
contractive multiplier from $H^{2}_{\cU}(\free)$ to $H^{2}_{\cY}(\free)$. 
Moreover:
\begin{enumerate}
    
\item The operator $\cO_{C,{\bf A}}: \, \cX\to H^{2}_{\cY}(\free)$ is isometric if and
only if ${\bf A}$ is strongly stable and the pair $(C,\bA)$ is isometric in the sense that
\begin{equation}  \label{isoIO}
A_1^*A_1+\ldots +A_d^*A_d+C^*C=I_{\cX}.
\end{equation}

\item If  ${\bf A}$ is strongly stable and
${\bf U}$ is unitary, then $\Theta_{\bf U}$ is strictly inner  (i.e., $M_{\Theta_{\bf U}} \colon
H^2_\cU(\free) \to H^2_\cY(\free)$ is isometric) and the representation \eqref{1.26pre} is
orthogonal in $H^{2}_{\cY}(\free)$:
\begin{equation}   \label{orthogISO}
\|\widehat y\|_{ H^{2}_{\cY}(\free)}^2= \|\cO_{C,{\bf A}} x(\emptyset)\|_{H^{2}_{\cY}(\free)}^2+\|\Theta_{\bf U} 
\widehat u\|_{ H^{2}_{\cY}(\free)}^2=\|x(\emptyset)\|_\cX^2+\|\widehat u\|_{ H^{2}_{\cU}(\free)}^2.
\end{equation}
\end{enumerate}
\end{theorem}

\begin{proof}
Items (a) and (b) follow from \cite[Proposition 2.3]{BBF1} while item (c) follows from \cite[Theorem 3.1]{BBF1}.
Item (1) is a consequence of Theorem 2.10 in \cite{BBF1}.   

\smallskip

We now address item (2) and
assume now that $\bA$ is strongly stable and that $\bU$ is unitary.  To show that $M_{\Theta_\bU} \colon H^2_\cU(\free)
\to H^2_\cY(\free)$ is isometric, it is convenient to return to the system equations \eqref{1.20pre}.  From the second of equations
\eqref{1.26pre} with $x(\emptyset) = 0$, we see that $M_{\Theta_\bU}$ being an isometry is equivalent to the following:
{\em whenever $(\bu, \bx, \by)$ is a system trajectory of $\Sigma_\bU$ \eqref{1.20pre} initialized with $x(\emptyset) = 0$,
then necessarily}
 \begin{equation}  \label{isom'}
\| \by \|^2_{\ell^2_\cY(\free)}
 = \| \bu \|^2_{\ell^2_\cU(\free)}.
\end{equation}
 (Here we use the notation
 $\bu = \{ u(\alpha)\}_{\alpha \in \free}$,     $\bx = \{ x(\alpha)\}_{\alpha \in \free}$,   $\by = \{ y(\alpha)\}_{\alpha \in \free}$.)
The following argument shows that 
this conclusion requires only that $\bU$ is isometric.  Indeed,  if $\bU$ is isometric, from the system equations we see that
for each $N\ge 0$,
$$
\sum_{\alpha \colon |\alpha| = N+1} \| x(\alpha) \|^2 +  \sum_{\alpha \colon |\alpha| = N} \| y(\alpha) \|^2 =
\sum_{\alpha \colon |\alpha| = N} \| x(\alpha) \|^2  +  \sum_{\alpha \colon |\alpha| = N} \| u(\alpha) \|^2 
$$
which we rewrite in the telescoping form
\begin{equation}  \label{telescope}
\sum_{\alpha \colon |\alpha| = N+1} \| x(\alpha) \|^2 - \sum_{\alpha \colon |\alpha| = N} \| x(\alpha) \|^2 =
\sum_{\alpha \colon |\alpha| = N} \| u(\alpha) \|^2  -  \sum_{\alpha \colon |\alpha| = N} \| y(\alpha) \|^2. 
\end{equation}
If we sum from $N=0$ to $N=M$ and use that $x(\emptyset) = 0$, we arrive at
$$
\sum_{\alpha \colon |\alpha| = N+1} \| x(\alpha) \|^2 = \sum_{\alpha \colon |\alpha| \le N} \| u(\alpha) \|^2
  -  \sum_{\alpha \colon |\alpha| \le N} \| y(\alpha) \|^2.
$$
As a first approximation, let us restrict to input signals $\bu$ such that $u(\alpha) = 0$ once $|\alpha| > K$
for some large $K$.  After running the system up to words of length $K$, we may consider the trajectory
generated on words of length larger than $K$ as determined by initializing the state $x$ at words of length $K$
and then continuing to run the system with zero inputs for words of length larger than $K$.  The strong-stability
assumption on $\bA$ implies that 
$$
\sum_{\alpha \colon |\alpha| = N+1} \| x(\alpha)\|^2 \to 0\quad\mbox{as}\quad N \to \infty.
$$
Hence we may take the limit as $N \to \infty$ in \eqref{telescope} to conclude that \eqref{isom'} holds for the case where
$\bu$ has finite support.   An approximation argument then implies that the same result holds for a general 
$\bu \in \ell^2_\cU(\free)$.

\smallskip

We next address the string of equalities \eqref{orthogISO}.  Since $\bU$ is unitary, in particular $(C,\bA)$ is an isometric
output-pair (i.e. \eqref{isoIO} holds) and hence by item (1) $\cO_{C, \bA}$ is isometric.  By Proposition \ref{P:principle} (see also Theorem 2.10 in \cite{BBF1})
the space $\cM:= \operatorname{Ran} \cO_{C, \bA}$ is then isometrically equal to the formal reproducing kernel Hilbert space $\cH(K_{C, \bA})$
with kernel
\begin{equation}  \label{K-CA}
K_\cM(z, \zeta) =    K_{C, \bA}(z, \zeta) = C (I - Z(z) A)^{-1} (I - A^* Z(\zeta)^*)^{-1}C^*.
\end{equation}
As $\bU$ is unitary, it follows from Proposition 3.2 in \cite{BBF3} that
$$
   K_{C, \bA}(z, \zeta) = k_{\rm nc, Sz}(z, \zeta) I_\cY - \Theta_\bU(z) (k_{\rm nc, Sz}(z, \zeta) I_\cU) \Theta_\bU(\zeta)^*.
$$
An application of Theorem \ref{T:charNFRKHS} (for the case where $\cM$ and $\cM^{[\perp]} = \cM^\perp$ are
isometrically contained in $\cH(k_{\rm nc, Sz} I_\cY)$)  then gives us that $\cM^\perp = \Theta_\bU \cdot H^2_\cU(\free)$.
Thus the decomposition
in the second equation in \eqref{1.26pre} is orthogonal and we have verified the first equality in \eqref{orthogISO}.
The second follows from the fact that we have already observed that both $\cO_{C, \bA}$ and $M_{\Theta_\bU}$ are
isometric.   
\end{proof}

Theorem \ref{T:cm} below provides a useful procedure for constructing a contractive multiplier from 
$H^{2}_{\cU}(\free)$ to $H^{2}_{\cY}(\free)$ with prescribed output pair $(C, \bA)$ in its contractive realization, and furthermore
identifying when this contractive multiplier is McCT-inner or even strictly inner.
For future convenient reference let us first set forth our precise terminology concerning operator inequalities.

\begin{definition}  \label{D:op-ineq} Let $X$ be a selfadjoint operator on a Hilbert space $\cX$.  Then we
say that
\begin{enumerate}
 \item $X$ is {\em positive-semidefinite}, written as $X \succeq 0$ or $-X \preceq 0$, if $\langle X x, x \rangle_\cX  \ge 0$
 for all $x \in \cX$,
 \item $X$ is {\em positive-definite} if $\langle X x, x \rangle_\cX > 0$ for all nonzero $x \in \cX$, and
 \item $X$ is {\em strictly positive-definite}, written as $X \succ 0$ or $-X \prec 0$, if there is $\epsilon > 0$ so that
 $\langle X x, x \rangle_\cX \ge \epsilon^2 \| x \|^2$ for all $x \in \cX$.  
Note that {\em positive-definite} and {\em strictly positive-definite} are equivalent in case $\cX$ is finite-dimensional.  
\end{enumerate}
In case
 $X - Y \succeq 0$ (respectively $X - Y \succ 0$), we write $X \succeq Y$ or $-Y \preceq -X$ (respectively $X \succ Y$ or $-Y \prec -X$).
\end{definition}

\noindent
Let us also mention the following general fact: whenever $C \in \cL(\cX, \cY)$, $A \in \cL(\cX, \cX')$, $0 \prec H \in \cL(\cX)$
and  $0  \prec H' \in \cL(\cX')$, then
\begin{equation}  \label{fact}
 H - A^* H' A \succeq C^*C \; \; \Longleftrightarrow \; \;
 \begin{bmatrix} H^{\prime -1} & 0 \\ 0 & I_\cY \end{bmatrix} - \begin{bmatrix} A \\ C \end{bmatrix} H^{-1} \begin{bmatrix} A^* & C^* \end{bmatrix}
 \succeq 0.
 \end{equation}
 Indeed, the inequality $H - A^* H' A \succeq C^*C$ means that $\left\| \sbm{ H^{\prime \frac{1}{2}} A H^{-\frac{1}{2}} \\ C H^{-\frac{1}{2}} } \right\| \le 1$
 which implies that $\| \begin{bmatrix} H^{-\frac{1}{2}} A^* H^{\prime \frac{1}{2}} &  H^{-\frac{1}{2}} C^* \end{bmatrix} \| \le 1$,
 which in turn implies that 
 $$
    \begin{bmatrix} H^{\prime \frac{1}{2}} A H^{-\frac{1}{2}} \\ C H^{- \frac{1}{2}} \end{bmatrix}
   \begin{bmatrix} H^{-\frac{1}{2}}A^* H^{\prime \frac{1}{2}} & H^{-\frac{1}{2}} C^* \end{bmatrix}
   \preceq \begin{bmatrix} I_{\cX'} & 0 \\ 0 & I_\cY \end{bmatrix},
$$
which finally leads to the inequality on the right side of the implication \eqref{fact}, and conversely.

\begin{theorem}  \label{T:cm}
Given a tuple $\bA=(A_1,\ldots,A_d)\in\cL(\cX)^d$ and $C\in\cL(\cX,\cY)$,
let $H\in\cL(\cX)$ be a strictly positive definite operator such that 
\begin{equation}
H-\sum_{j=1}^dA_j^*HA_j  \succeq C^*C.
\label{jan3}
\end{equation}
Let $A$ and $Z(z)$ be defined as in \eqref{1.23pre}. As a consequence of \eqref{fact} we can choose a solution
$\sbm{B \\ D}: \, \cU\to \sbm{\cX^d \\ \cY}$ of the Cholesky factorization problem
\begin{equation}
\begin{bmatrix} B \\ D \end{bmatrix}\begin{bmatrix}
B^{*} & D^{*}\end{bmatrix}=\begin{bmatrix} H^{-1}\otimes I_d & 0 \\0 &
I_{\cY} \end{bmatrix}-\begin{bmatrix} A \\ C
\end{bmatrix}H^{-1}\begin{bmatrix} A^{*} & C^{*}\end{bmatrix}.
\label{10}
\end{equation}
\begin{enumerate}
\item
Then the pair $(C,\bA)$ is output-stable and the power series
\begin{equation}
G(z)=D+C(I-Z(z)A)^{-1}Z(z)B
\label{jan3a}
\end{equation}
is a contractive multiplier from $H^2_{\cU}(\free)$ to $H^2_{\cY}(\free)$. Moreover,
\begin{align}
&k_{\rm nc, Sz}(z,\zeta)I_{\cY}-G(z)(k_{\rm nc, Sz}(z,\zeta) I_\cU)G(\zeta)^*\notag\\
&=C(I-Z(z)A)^{-1}H^{-1}(I-A^*Z(\zeta)^*)^{-1}C^*.
\label{jan3b}
\end{align}

\item
If \eqref{jan3} holds with equality and $\bA$ is strongly stable, then $G$ is McCT-inner.
Conversely, any McCT-inner multiplier arises in this way.

\item
If \eqref{jan3} holds with equality, $\bA$ is strongly stable, and the solution $\sbm{ B \\ D }$ of \eqref{10}
is normalized to be injective, then $G$ is strictly inner as a multiplier from $H^2_\cU(\free)$ to $H^2_\cY(\free)$.
Conversely, any such strictly inner multiplier arises in this way.
\end{enumerate}
\end{theorem}

The proof relies on the following general observations which will also come up in later applications.

\begin{proposition}   \label{P:elementary'}  Suppose that $U \colon \cH \to \cK$ is a Hilbert-space contraction operator.  
Then there exists a Hilbert space $\cJ$ and an  operator 
$ V \colon \cJ \to \cK$ so that
\begin{equation}  \label{CholeskyFact0}
V V^*  = I_\cK  - U U^*
\end{equation}
and then $\begin{bmatrix} U & V \end{bmatrix}$ is a coisometry from $\sbm{ \cH \\ \cJ }$ to $\cK$. Moreover,
\begin{enumerate}
\item  If $V \colon \cJ \to \cK$ is normalized to be injective and 
if  $V' \colon \cJ' \to \cK$ is another injective operator satisfying \eqref{CholeskyFact0}, then there is a unitary operator
 $W$ from $\cJ$ onto $\cJ'$ so that $V' = V W$.
\item The operator $\begin{bmatrix} U & V \end{bmatrix}$ is unitary if and only if the originally given operator $U$ is isometric
and $V$ is normalized to be injective.
\end{enumerate}
\end{proposition}

\begin{proof}  
If $\| U \| \le 1$, then $I - U U^* \succeq 0$ and we can solve for $V$ so that $V V^*  = I - U U^*$.
Then it follows immediately that $\begin{bmatrix} U & V \end{bmatrix}$ is a coisometry.  
As this is the only freedom in the construction, the uniqueness statement
 follows as well for the case where $V$ is required also to be injective.
If $\begin{bmatrix} U & V \end{bmatrix}$ is unitary, then
$$
 \begin{bmatrix} U^* \\ V^* \end{bmatrix} \begin{bmatrix} U & V \end{bmatrix} = \begin{bmatrix} I_\cH  & 0 \\ 0 & I_\cJ \end{bmatrix},
$$
so in particular $U^* U = I_\cH$ or $U$ is an isometry, and $V^* V = I_\cJ$ so $V$ is injective.  Conversely, if $U$ is an isometry, 
then  $U U^* = P_{\operatorname{Ran}  U}$ is the orthogonal projection of $\cK$ onto the range of $U$.  
Then $V V^* = I_\cK - U U^* = P_{({\rm Ran}  \, U)^\perp}$ is the
 orthogonal projection onto the orthogonal complement $({\rm Ran} \, U)^\perp$ of the range of $U$
 in $\cK$.  Then $V$ is a partial isometry with final space equal to $(\operatorname{Ran}  U)^\perp$.
 If $V$ is injective, it follows that the initial space of $V$ is the whole space $\cJ$ and $V$ is an isometric
 embedding of $\cJ$ onto $({\rm Ran}\, U)^\perp$.  It then follows that
 $\begin{bmatrix} U & V \end{bmatrix}$ is a unitary transformation from $\cH \oplus \cJ$ onto  $\cK$.
 \end{proof}
 
Other ingredients needed for the proof Theorem \ref{T:cm} are the  shift operator-tuple
\begin{equation}
\bS_{1,R} = (S_{1,R,1}, \dots, S_{1,R,d}),\qquad
S_{1,R,j} \colon f(z)\to f(z)z_j
\label{march7}
\end{equation}
on $H_{\cY}^{2}(\free)$ (already mentioned in Section 2.2) and the empty-word-coefficient evaluation operator 
$E: \, \sum_{\alpha\in\free}f_\alpha z^\alpha\to f_{\emptyset}$ on $H^{2}_{\cY}(\free)$. 
The operators $S_{1,R,1},\ldots, S_{1,R,d}$ are isometries with mutually orthogonal ranges, while
the pair $(\bS_{1,R},E)$ is isometric in the sense that 
$$
\sum_{j=1}^d S_{1,R,j}S_{1,R,j}^*+EE^*=I_{H^{2}_{\cY}(\free)}.
$$
 \begin{proof}[Proof of Theorem \ref{T:cm}.]  
 
 \textbf{(1)} By replacing $A_j$ with $H^{\frac{1}{2}} A_j H^{-\frac{1}{2}}$ and $C$ with 
 $C H^{-\frac{1}{2}}$, we may assume without loss of generality that $H = I_\cX$.  With $G(z)$ defined as in \eqref{jan3a},
 the inequality \eqref{jan3} says that $(C, \bA)$ is a contractive output pair; it then follows that $(C, \bA)$ is output-stable
 by Proposition 2.3 in \cite{BBF1}.  Once $G$ is defined by \eqref{jan3a} with $\sbm{ A & B \\ C & D}$ coisometric, the identity
 \eqref{jan3b} follows, again by Proposition 3.2 of \cite{BBF3}.  The fact that the construction of $B, D$
 via \eqref{10} makes $\sbm{ A & B \\ C & D}$ a coisometry is a simple consequence of Proposition \ref{P:elementary'} with 
 $U = \sbm{ A \\ C}$ and $V = \sbm{ B \\ D}$ which then implies the identity \eqref{jan3b}.
 The validity of the identity \eqref{jan3b} in turn implies that $G$ is a contractive multiplier (see \cite[Theorem 3.15]{NFRKHS}).
  \smallskip
 
 \noindent
 \textbf{(2)}
  If \eqref{jan3} holds with equality and $\bA$ is strongly stable, then the observability operator
 $\cO_{C, \bA} \colon \cX \to H^2_\cY(\free)$ is isometric (see \cite[Proposition 2.3] {BBF1}), and then
 $\operatorname{Ran} \cO_{C, \bA}$ has reproducing kernel 
 $$
 K_{C, \bA}(z,\zeta) = C (I - Z(z) A)^{-1}( I - A^* Z(\zeta)^*)^{-1} C^*
 $$ 
 as a subspace of $H^2_\cY(\free)$ (see \cite[Theorem 2.10]{BBF1}).
 From the identity \eqref{jan3b}, we see that this kernel has the de Branges-Rovnyak form
 $$
K_{C, \bA}(z, \zeta) = k_{\rm nc, Sz} I_\cY - G(z) (k_{\rm nc, Sz}(z \zeta) I_\cU) G(\zeta)^*.
$$
 As a consequence of Theorem \ref{T:charNFRKHS}, it follows that $\cM^\perp = \cM^{[\perp]}$ is equal to the pullback
 space $\cH^p(M_G M_G^*)$.  Since $\cM^\perp$ is isometrically contained in $H^2_\cY(\free)$,  it follows that
 $M_G M_G^*$ is the projection onto $\cM^\perp$.  In particular, $M_G$ is a partial isometry, so $G$ is a 
 McCT-inner multiplier.
 
\smallskip

 Finally, suppose that $G$ is any McCT-inner multiplier from $H^2_\cU(\free)$ to $H^2_\cY(\free)$, and set
 $\cN = \operatorname{Ran} M_G$.  Then $\cN$ is isometrically contained in $H^2_\cY(\free)$ and also
 we can identify $\cN$ with the pullback space $\cH^p(M_G M_G^*)$ having reproducing kernel
 $$
K_\cN(z, \zeta) = G(z)( k_{\rm nc, Sz}(z, \zeta) I_\cU) G(\zeta)^*.
$$
  As $\cN = M_G \cdot H^2_\cU(\free)$ is
 clearly $\bS_{1, R}$-invariant, it follows that  $\cN^\perp$ is $\bS_{1, R}^*$-invariant.  Then we can represent
 $\cN^\perp$ as the range of an observability operator $\cN^\perp = \operatorname{Ran} \cO_{C, \bA}$ with
 $(C, \bA)$ the restricted model output-pair
 $$
   C = E|_{\cN^\perp} \text{ where } E \colon \sum_{\alpha \in \free} f_\alpha z^\alpha \mapsto f_\emptyset, \quad
   \bA = \bS_{1, R}^*|_{\cN^\perp}.
 $$
 Since the model output pair $(E, \bS_{1, R}^*)$ is an isometric pair with $\bS_{1, R}^*$ strongly stable,
 it follows that the same properties hold for $(C, \bA)$.  If we construct $\sbm{ B \\ D}$  from $\cU'$ to
 $\sbm{ \cN^\perp \\ \cY}$ according to the Cholesky factorization procedure \eqref{10} and define 
$$
G'(z) = D + C (I - Z(z) A)^{-1} B, 
$$
then by the first part of the proof  it follows that we also recover $\cN$ as
 $\cN = \cH^p(M_{G'} M_{G'}^*)$ and hence $M_{G'} M_{G'}^* = M_G M_G^*$, or in kernel form,
 $$
   G(z) (k_{\rm nc, Sz}(z, \zeta) I_\cU) G(\zeta)^* = G'(z) (k_{\rm nc, Sz}(z, \zeta) I_{\cU'}) G'(\zeta)^*.
 $$
 As a consequence of Proposition \ref{P:leech} to come, it follows that there is a partial isometry $W \colon \cU \to \cU'$
 so that $G(z) = G'(z) W$ and $G'(z) = G(z) W^*$.  
 Thus $G(z)$ has a realization of the form \eqref{jan3a} but with connection matrix
$\widetilde \bU$ given by
 $$
\widetilde \bU : =  \begin{bmatrix} \widetilde A & \widetilde B \\ \widetilde C & \widetilde D \end{bmatrix} =
 \begin{bmatrix} A & B W \\ C & D W  \end{bmatrix}.
 $$
 To show that the matrix $\widetilde \bU$ is coisometric (and hence $\sbm{ \widetilde B \\ \widetilde D}$
 arises as another solution of the Cholesky factorization procedure \eqref{10}), we need to show that
 $$
 \begin{bmatrix} \widetilde B \\ \widetilde D \end{bmatrix} \begin{bmatrix} \widetilde B^* & \widetilde D^* \end{bmatrix} =
 \begin{bmatrix} B \\ D \end{bmatrix} \begin{bmatrix} B^* & D^* \end{bmatrix},
 $$
 or equivalently,
 \begin{equation}  \label{need}
 \begin{bmatrix} B  \\ D \end{bmatrix} W W^* \begin{bmatrix} B^* & D^* \end{bmatrix} =
 \begin{bmatrix} B \\ D \end{bmatrix} \begin{bmatrix} B^* & D^* \end{bmatrix}.
 \end{equation}
 The equality $G(z) = G'(z) W$ when expressed in terms of power-series coefficients gives us
 \begin{equation}  \label{have}
     \widetilde D = D W, \quad C \bA^\alpha \widetilde B_j = C \bA^\alpha B_j W \quad\text{for all}\quad \alpha \in \free 
\; \text{ and } \; j=1, \dots, d.
 \end{equation}
 As $(C, \bA)$ is isometric and $\bA$ is strongly stable by the setup for item (2) in the theorem, it follows from
item (1) in Theorem \ref{1.29pre} that $\cO_{C, \bA}$ is isometric, so in particular $(C, \bA)$ is observable.
Thus the second equality in \eqref{have} implies that in fact $\widetilde B_j = B_j W$ for $j=1, \dots, d$, or 
$\widetilde B = B W$.  This combined with the first equality in \eqref{have} gives us the desired equality \eqref{need}.
 
 \smallskip
 
 \noindent
 \textbf{(3)} Suppose now that equality holds in \eqref{jan3}, $\bA$ is strongly stable and the solution 
 $\sbm{ B \\ D }$ of \eqref{10} is normalized to be injective,  Then the last part of Proposition \ref{P:elementary'}
 tells us that $\sbm{ A & B \\ C & D}$ so constructed is unitary.   Since we are also assuming that $\bA$ is strongly stable,
 item (2) in Theorem \ref{1.29pre} tells us that $G$ is strictly inner.  Conversely, if $G$ is strictly inner, then in particular
 $G$ is McCT-inner so the converse side of item (2) above tells us that $G$ has the form \eqref{jan3a} with
 $(C, \bA)$ isometric and $\bA$ strongly stable and $\sbm{ B \\ D }$ a solution of the Cholesky factorization problem
 \eqref{10}.  The assumption that $G$ is strictly inner precludes the possibility that $\sbm{ B \\ D}$ has a nontrivial kernel, i.e.,
 necessarily $\sbm{ B \\ D }$ is injective, and the realization for $G$ is of the form stated in item (3) of the theorem.
 In particular, as a consequence of the last part of Proposition \ref{P:elementary'}, $\sbm{ A & B \\ C & D}$ is unitary.
   \end{proof}
A close relation between McCT-inner and strictly inner multipliers is pointed out in Remark \ref{R:BBF3} below.
We start with a preliminary observation. 
\begin{proposition} \label{P:BBF3} Given any McCT-inner multiplier $G$ from
  $H^2_\cU(\free)$ to $H^2_\cY(\free)$, the space $\operatorname{Ker} M_G$ is reducing for $\bS_{1,R}$.
\end{proposition}
\begin{proof} Since $M_G$ intertwines $S_{1,R,j}$ on $H^2_\cU(\free)$ with $S_{1,R,j}$ on $H^2_\cY(\free)$, 
the computation 
$$
M_G S_{1,R,j} f = S_{1,R,j} M_G f = 0\quad\mbox{for}\quad f \in \operatorname{Ker} M_G
$$ 
shows that $\operatorname{Ker} M_G$ is invariant for $\bS_{1,R}$ on $H^2_\cU(\free)$.  To show that $(\operatorname{Ker} M_G)^\perp$ 
is also invariant for $\bS_{1,R}$, we use the fact that each $S_{1,R,j}$ is an isometry and that $M_G$ being a partial isometry 
  implies that $(\operatorname{Ker} M_G)^\perp$ is characterized by the property
  $$
   f \in (\operatorname{Ker} M_G)^\perp \Leftrightarrow \| M_G f \|  = \| f \|.
  $$
 Thus $f \in (\operatorname{Ker} M_G)^\perp$ implies that
 $$
   \| M_G S_{1,R,j} f \| = \| S_{1,R,j} M_G f \| = \| M_G f \| = \| f \|,
  $$
  from which we conclude that $S_{1,R,j} f \in (\operatorname{Ker} M_G)^\perp$ as well for $j=1, \dots, d$.
\end{proof}  

\begin{remark}
 \label{R:BBF3}
 {\em For any McCT-inner multiplier $G$ from
  $H^2_\cU(\free)$ to $H^2_\cY(\free)$, there is an orthogonal decomposition  $\cU = \sbm{ \cU_{si} \\ \cU_0}$ of the input space
  so that with respect to this decomposition, $G(z)$ has the form $G(z) = \begin{bmatrix} G_{\rm si}(z) & 0 \end{bmatrix}$ 
  with $G_{si}$ strictly inner. }    
One way to see this statement is as an application of Proposition \ref{P:BBF3}:  given a McCT-inner multiplier
  $G$, set $\cU_0 = {\rm Ker} M_G$ and $\cU_{si} = \cU \ominus \cU_0$.
An alternative state-space proof can be derived from  Theorem \ref{T:cm} as follows.
  By item (2) in Theorem \ref{T:cm},   McCT-inner multipliers $G$ are characterized by having a realization \eqref{jan3a} with
  $(C, \bA)$ isometric and observable,  $\bA$ strongly stable, and $B, D$ solving the Cholesky factorization problem
  \eqref{10} (so $\sbm{ A & B \\ C & D}$ is coisometric).  By item (3), strictly inner multipliers are characterized via the same
  conditions with the additional property that $\sbm{B \\ D }$ be injective.  We may therefore set $\cU_{\rm si} = 
 ( \operatorname{Ker} \sbm{B \\ D})^\perp$ and $\cU_0 = \operatorname{Ker} \sbm{B \\ D}$  to recover the decomposition
 $G = \begin{bmatrix} G_{\rm si}  & 0 \end{bmatrix}$ of $G$ with $G_{\rm si}$ strictly inner.
  \end{remark} 

  This reduction of a McCT-inner $G$ to a strictly inner $G_{\rm si}$ fails in the more general context
  of McCT-inner multipliers between weighted Bergman spaces, causing the Beurling-Lax representations for shift-invariant 
  subspaces isometrically included in the ambient weighted Bergman space to be discussed in Section  \ref{S:NC-appl}   
  to be formulated only with McCT-inner (rather than strictly inner) multipliers. 

\smallskip
    
  This same phenomenon is behind the incorrect statement of Theorem 4.6
  in \cite{BBF3} where the additional hypothesis {\em the output pair $(C, \bA)$ is isometric} should be inserted 
  and the conclusion   {\em $S$ is (strictly) inner} should be changed to {\em $S$ is McCT-inner}.  
  Also the conclusion of part (2) of Theorem \ref{T:cm}
  is stated incorrectly in part (4) of Proposition 3.3 in \cite{BBF3}:  the hypothesis 
  {\em the output pair $(C,\bA)$ is a contractive} should be changed to 
  {\em the output pair $(C, \bA)$ is isometric}.

\smallskip

 We recall that multiplication operators $M_{F} \colon
H^{2}_{\cU}(\free) \to H^{2}_{\cY}(\free)$ are characterized as the operators intertwining
$\bS_{1,R} \otimes I_{\cU}$ with $\bS_{1,R} \otimes I_{\cY}$: if the operator $X$ in
$\cL(H^{2}_{\cU}(\free), H^{2}_{\cY}(\free))$ satisfies the intertwining equalities
$$
  X S_{1,R,j} = S_{1,R,j} X
\;  \text{ for } \; j=1, \dots, d,
$$
then $X = M_{F}\colon u(z)\mapsto F(z) u(z)$ for a multiplier $F \in \cL(\cU, \cY)\langle \langle
z \rangle \rangle$ (see \cite{Popescu-multi}). We shall have use for
the following Commutant Lifting Theorem for this setting due to
Popescu (see \cite[Theorem 3.2]{PopescuNF1}). 

\begin{theorem}  \label{T:CLT}
Suppose that the subspaces $\cM \subset H^{2}_{\cY}(\free)$ and $\cN \subset H^{2}_{\cU}(\free)$
are invariant for the backward right-shift tuple  $(\bS_{1,R})^{*}$ acting on $H^{2}_{\cY}(\free)$ and $H^{2}_{\cU}(\free)$ 
respectively, and that the operator $X \in \cL(\cM, \cN)$ satisfies
 $$
 \| X \| \le 1, \quad X   (S_{1,R,j}|_{\cM})^{*} =  (S_{1,R,j}|_{\cN})^{*} X \quad \text{for} \quad j = 1, \dots, d.  
 $$
 Then there is a contractive multiplier $F \in \cL(\cU, \cY)\langle
 \langle z \rangle \rangle$ so that $\; \left(M_{F}\right)^{*}|_{\cM} = X$.
\end{theorem}

Strictly inner multipliers between Fock spaces are of particular interest as they serve 
as Beurling-Lax representers of shift invariant subspaces (see Chapter 5 for more details). 
The following Beurling-Lax type theorem was originally given by Popescu \cite{PopescuBL}; see also \cite{BBF3}.

\begin{theorem} \label{T:bl4a}
    Let $\cM$ be a closed ${\bf S}_{1,R}$-invariant subspace of
    $H^2_{\cY}(\free)$. Then there exist a Hilbert space $\cU$ and a strictly inner multiplier
 $G$ from $H^{2}_{\cU}(\free)$ to $H^2_{\cY}(\free)$ such that $\cM = G\cdot H^2_{\cU}(\free)$.
    \end{theorem}
Characterization of strictly inner multipliers between two Fock spaces in terms of state space realizations
\eqref{jan3a} was given in Theorem \ref{T:cm}. Another characterization of strictly inner multipliers is as 
follows.

\begin{lemma}
Let $F$ be a contractive multiplier from $H^2_{\cU}(\free)$ to $H^2_{\cY}(\free)$.
Then $F$ is a strictly inner multiplier if and only if $\|Fu\|_{H^2_{\cY}(\free)}=\|u\|_{\cU}$ for all 
$u\in\cU$.
\label{L:in}
\end{lemma}

\begin{proof}
The ``only if" part is self-evident. 

\smallskip

To prove the ``if" part, we 
recall that the operators $S_{1,R,1},\ldots,S_{1,R,d}$ \eqref{march7} are isometries on 
$H^2_{\cY}(\free)$ with mutually orthogonal ranges, and therefore,
$$
\|\bS_{1,R}^\alpha M_{F}u\|_{H^2_{\cY}(\free)}=\|M_{F} u\|_{H^2_{\cY}(\free)}=\|u\|_\cU
$$
for all $u\in\cU$ and $\alpha\in\free$. Moreover,
$$
\|\bS_{1,R}^\alpha M_{F}u+\bS_{1,R}^\beta M_{F}v\|_{H^2_{\cY}(\free)}^2
=\|u\|^2_{\cU}+\|v\|^2_{\cU}+2{\rm Re} \big\langle \bS_{1,R}^\alpha M_{F}u, \, 
\bS_{1,R}^\beta M_{F}v\big\rangle_{H^2_{\cY}(\free)}
$$
for any $u,v\in\cU$ and $\alpha,\beta\in\free$. On the other hand, as $M_{F}: \; 
H^2_{\cU}(\free)\to H^2_{\cY}(\free)$ is a contraction, we have for $\alpha\neq \beta$,
\begin{align*}
\|\bS_{1,R}^\alpha M_{F}u+\bS_{1,R}^\beta M_{F}v\|_{H^2_{\cY}(\free)}^2
&=\|M_{F}(\bS_{1,R}^\alpha u+\bS_{1,R}^\beta v)\|_{H^2_{\cY}(\free)}^2\\
&\le \|\bS_{1,R}^\alpha  u+\bS_{1,R}^\beta v\|_{H^2_{\cU}(\free)}^2=\|u\|_{\cU}^2+\|v\|_{\cU}^2.
\end{align*}
Combining the two latter relations we conclude that
$$
2 {\rm Re} \big\langle \bS_{1,R}^\alpha M_{F}u, \,
\bS_{1,R}^\beta M_{F}v\big\rangle_{H^2_{\cY}(\free)}\le 0 \quad\mbox{for all $u,v\in\cU$ and 
$\alpha\neq \beta$,}
$$
which is possible (since $u$ can be replaced by $-u$ and by $\pm i u$) only if 
$$
\big\langle \bS_{1,R}^\alpha M_{F}u, \, \bS_{1,R}^\beta M_{F}v\big\rangle_{H^2_{\cY}(\free)}= 0.
$$
If this is the case,  then
$\|Fp\|_{H^2_{\cY}(\free)}=\|p\|_{H^2_{\cU}(\free)}$ for every $\cU$-valued ``polynomial" 
$p\in\cU\langle\langle z\rangle\rangle$ and therefore,
for any $p\in H^2_{\cU}(\free)$. Therefore, $M_{F}$ is an isometry from $H^2_{\cU}(\free)$ to
$H^2_{\cY}(\free)$ and hence, $F$  is strictly inner.
\end{proof}

The next result shows that the class of coisometric multipliers 
from $H^2_{\cU}(\free)$ to $H^2_{\cY}(\free)$ is quite narrow.

\begin{lemma} 
Let $F(z)=\sum_{\alpha\in\free}F_\alpha z^\alpha$ be a coisometric multiplier from 
$H^2_{\cU}(\free)$ to $H^2_{\cY}(\free)$. Then $F(z)$ is a coisometric constant:
$F_{\emptyset}F_{\emptyset}^*=I_{\cY}$ and $F_\alpha=0$ for all 
$\alpha\in\free\backslash\{\emptyset\}$.
\label{L:cois}
\end{lemma}

\begin{proof}
By specializing the second part of Proposition \ref{P:2.3} to the case where 
$K(z,\zeta)=k_{\rm nc, Sz}(z,\zeta)I_{\cY}$ and $K^\prime(z,\zeta)=k_{\rm nc, Sz}(z,\zeta)I_{\cU}$
we conclude that $F(z)$ is a coisometric multiplier from
$H^2_{\cU}(\free)$ to $H^2_{\cY}(\free)$ if and only if 
$$
k_{\rm nc, Sz}(z,\zeta)I_{\cY}-F(z)k_{\rm nc, Sz}(z,\zeta)F(\zeta)^*=0.
$$
Substituting explicit power series formulas for $F$ and  \eqref{ncsz} for $k_{\rm nc, Sz}$ into the 
latter equality and equating the coefficients of $z^\alpha\overline{\zeta}^{\beta^\top}$ on both 
sides we get the desired conclusion. 
\end{proof}

\section{A noncommutative Leech theorem}  \label{S:Leech}

The next result is the noncommutative version of the Leech theorem  (see \cite{Leech}) solving the problem of producing
a contractive analytic operator function $S$ which solves the factorization problem $F = G S$ for two given
meromorphic operator functions $G$ and $F$.

\begin{theorem}  \label{T:Leech}
Given power series $G\in\cL(\cY,\cX)\langle\langle z\rangle\rangle$ and $F\in\cL(\cU,\cX)\langle\langle 
z\rangle\rangle$, the formal kernel
\begin{equation}
K_{G,F}(z, \zeta): = G(z)(k_{\rm nc, Sz}(z,\zeta)\otimes I_{\cY})G(\zeta)^*-F(z)(k_{\rm 
nc, Sz}(z,\zeta)\otimes I_{\cU})F(\zeta)^*
\label{jan3c}
\end{equation}
is positive if and only if there exists a contractive multiplier $S$ from 
$H^2_{\cU}(\free)$ to $H^2_{\cY}(\free)$ such that $F(z)=G(z)S(z)$.
\end{theorem}

\begin{proof}  If $F(z) = G(z) S(z)$ for a contractive multiplier 
    $S$, then
 $$
  K_{G,F}(z, \zeta) = G(z) \big( k_{\rm nc, Sz}(z, \zeta)\otimes 
  I_{\cY} - S(z) (  k_{\rm nc, Sz}(z, \zeta)\otimes I_{\cU} )
  S(\zeta)^{*} \big) G(\zeta)^{*}
 $$
 is a positive kernel as a consequence of the characterization of 
 contractive multipliers in Proposition \ref{P:2.3}.
  We provide two proofs of the converse which illustrate two distinct 
 approaches to metric-constrained interpolation theory for this 
 setting.  The second proof is more general in that the first proof 
 requires an additional hypothesis.
 
 \medskip
  \noindent
 \textbf{Proof I} (via \textbf{Commutant Lifting Theorem}):
For this proof we impose the additional hypothesis:

\smallskip

\noindent
\textbf{Additional Hypothesis:}  {\em Assume that both $G$ and $F$ 
are themselves bounded multipliers from $H^{2}_{\cY}(\free)$ to 
$H^{2}_{\cX}(\free)$ and from $H^{2}_{\cU}(\free)$ to 
$H^{2}_{\cX}(\free)$ respectively.}

\smallskip

\noindent
Suppose that $K_{G,F}$ is a positive formal kernel.
    By the computations in the proof of Proposition 
    \ref{P:2.3}, we see that the positivity of $K_{G,F}$ is equivalent to
 $$
 \| \sum_{j=1}^{N} M_{G}^{*} k_{{\rm nc, Sz}}(\cdot,\beta_{j}) x_{j} 
 \|^{2} - \| \sum_{j=1}^{N} M_{F}^{*}  k_{{\rm nc, Sz}}(\cdot,\beta_{j}) x_{j} \|^{2} \ge 0
 $$
 for all choices of $\beta_{1}, \dots, \beta_{N} \in \free$ 
 and $x_{1}, \dots, x_{N} \in \cX$.  Thus, the operator
 $$
  T_{0} \colon M_{G}^{*} f \mapsto M_{F}^{*} f\quad \text{for}\quad f \in 
  H^{2}_{\cX}(\free)
  $$
  extends to a contraction from $\cM=\overline{\rm Ran}\, 
  M_{G}^{*}\subset H^{2}_{\cY}(\free)$ into $\cN=\overline{\rm Ran}\, M_{F}^{*}\subset 
H^{2}_{\cU}(\free)$.  
Since 
  multiplication operators intertwine the right shift operators 
  $S_{1,R,j} \otimes I_{\cU}$ and $S_{1,R,j} \otimes I_{\cY}$, it is 
  easily verified that
  $$
    T_{0} \left((S_{1,R,j} \otimes I_{\cY})^{*}|_{\cM}\right) =\left((S_{1,R,j} \otimes 
    I_{\cU})^{*}|_{\cN}\right) T_{0} \quad\text{for}\quad j = 1, \dots, d.
  $$
  By the Commutant Lifting Theorem \ref{T:CLT}, we conclude 
  that there is a contractive multiplier $S \in \cL(\cU, \cY) \langle 
  \langle z \rangle \rangle$ such that $M_{S}^{*} |_{\cM} = T_{0}$, 
  i.e., $M_{S}^{*} M_{G}^{*} = M_{F}^{*}$.  Taking 
  adjoints gives $M_{G} M_{S} = M_{F}$, or $G(z) S(z) = F(z)$ as 
  wanted.
  
  \medskip 
  \noindent
  \textbf{Proof II} (via \textbf{Lurking Isometry method}):  Assume that $K_{G,F}$ 
  \eqref{jan3c} is a positive formal kernel.  Then by (3) 
  $\Rightarrow$ (2) in Theorem \ref{T:NFRKHS}, $K_{G,F}$ has a formal 
  Kolmogorov decomposition
  \begin{equation}   \label{KGF-Kol}
      K_{G,F}(z, \zeta) = H(z) H(\zeta)^{*}
 \end{equation}
 for a formal power series $H(z) \in \cL(\cX_{0}, \cX)$ ($\cX_{0}$ 
 is an auxiliary Hilbert space). Combining \eqref{jan3c} and \eqref{Sz-id} 
we see that the kernel $K_{G,F}(z, \zeta)$ satisfies
 $$
   K_{G,F}(z, \zeta) - \sum_{k=1}^{d} \bzeta_{k} K_{G,F}(z, \zeta) 
   z_{k} = G(z) G(\zeta)^{*} - F(z) F(\zeta)^{*}.
 $$
Applying the operation $X \mapsto X - \sum_{k=1}^{d} \bzeta_{k} X
 z_{k}$ to both sides of \eqref{KGF-Kol} and taking into account 
the last equality then yields
 $$
 G(z) G(\zeta)^{*} - F(z) F(\zeta)^{*} = H(z)H(\zeta)^{*}-
 \sum_{k=1}^{d}  H(z) z_{k} \bzeta_{k}H(\zeta)^{*}.
 $$
 Equate coefficients of $z^{\alpha} \bzeta^{\beta^{\top}}$ to get
 \begin{equation}   \label{lurking1}
   G_{\alpha} G_{\beta}^{*} - F_{\alpha} F_{\beta}^{*} = H_{\alpha} 
   H_{\beta}^{*} - \sum_{k=1}^{d} H_{\alpha k^{-1}}H_{\beta 
   k^{-1}}^{*}
 \end{equation}
 where we use the convention
 $$ \alpha k^{-1} = \begin{cases} \alpha', &\text{if } \alpha = 
 \alpha' k, \\
 \text{undefined,} &\text{otherwise,}
 \end{cases}
 $$
 and where by convention, $H_{\text{undefined}} = 0$. Rewrite \eqref{lurking1} as
 $$
   \sum_{k=1}^{d} H_{\alpha k^{-1}} H_{\beta k^{-1}}^{*} + 
   G_{\alpha} G_{\beta}^{*} = H_{\alpha} H_{\beta}^{*} + F_{\alpha} 
   F_{\beta}^{*}.
$$
Then the map 
\begin{equation}
  V \colon \begin{bmatrix} H_{\beta 1^{-1}}^{*} \\ \vdots \\ 
  H_{\beta d^{-1}}^{*} \\ G_{\beta}^{*} \end{bmatrix} x \mapsto
  \begin{bmatrix} H_{\beta}^{*} \\ F_{\beta}^{*} \end{bmatrix} x
\label{m1}
\end{equation}
extends to a well-defined isometry from the subspace
$$
  \cD = \overline{\rm span}_{\beta \in \free,\, x \in \cX} \left\{ 
  \begin{bmatrix} H_{\beta 1^{-1}}^{*} \\ \vdots \\ 
  H_{\beta d^{-1}}^{*} \\ G_{\beta}^{*} \end{bmatrix} x 
  \right\} \subset \begin{bmatrix} \cX_{0}^{d} \\ \cY \end{bmatrix}
$$
onto the subspace
$$
 \cR = \overline{\rm span}_{\beta \in \free,\, x \in \cX} \left\{ 
 \begin{bmatrix} H_{\beta}^{*} \\ F_{\beta}^{*} \end{bmatrix} x \right\}
   \subset  \begin{bmatrix} \cX_{0} \\ \cU \end{bmatrix} .
 $$
 Extend $V$ to a contraction (or even to a unitary at the possible 
 expense of enlarging the state space $\cX_{0}$) $\bU^{*} \colon 
 \sbm{ \cX_{0}^{d} \\ \cY } \to \sbm{  \cX_{0} \\ \cU}$.  Decompose 
 $\bU^{*}$ into block-matrix form
 $$
  \bU^{*} = \begin{bmatrix} A^{*} & C^{*} \\ B^{*} & D^{*} 
\end{bmatrix} = \begin{bmatrix} A_{1}^{*} & \cdots & A_{d}^{*} & 
C^{*} \\ B_{1}^{*} & \cdots & B_{d}^{*} & D^{*} \end{bmatrix} \colon 
\begin{bmatrix}  \cX_{0}^{d} \\ \cY \end{bmatrix} \to 
    \begin{bmatrix}  \cX_{0} \\ \cU \end{bmatrix}.
$$
Since $\bU^{*}$ is an extension of $V$, we get from \eqref{m1} the system of equations
$$
\sum_{k=1}^{d} A_{k}^{*} H_{\beta k^{-1}}^{*} + C^{*}
    G_{\beta}^{*} = H_{\beta}^{*}, \qquad
    \sum_{k=1}^{d} B_{k}^{*} H_{\beta k^{-1}}^{*} + D^{*}
   G_{\beta}^{*} = F_{\beta}^{*}.
$$
Actually the adjoint form of these equations will be more convenient:
$$
\sum_{k=1}^{d}H_{\beta k^{-1}} A_{k} + G_{\beta} C =
H_{\beta}, \qquad \sum_{k=1}^{d} H_{\beta k^{-1}} B_{k} + G_{\beta} D = F_{\beta}.
$$
Apply the formal $Z$-transform \eqref{1.25pre} to get the formal power 
series version of this system of equations:
\begin{align}
    H(z) Z(z) A + G(z) C & = H(z),  \label{1} \\
    H(z) Z(z) B + G(z) D & = F(z) \label{2}
 \end{align}
with $Z(z)$ as in \eqref{1.23pre} (but with $\cX_{0}$ in place of $\cX$).
Use \eqref{1} to solve for $H(z)$:
 $$
   H(z) = G(z) C (I - Z(z) A)^{-1}.
 $$
 Plug this expression for $H(z)$ back into \eqref{2} to get
 \begin{equation}   \label{GS=F}
   G(z) \left( C (I - Z(z)   A)^{-1} Z(z) B + D 
   \right) = F(z).
 \end{equation}
 This suggests that we set
 $$
   S(z) = D +  C (I - Z(z)   A)^{-1} Z(z) B.
 $$
Since $\bU$ is contractive/unitary, then $S$ is a contractive multiplier from $H^2_{\cU}(\free)$ to 
$H^2_{\cY}(\free)$, by Theorem \ref{T:cm}. Finally, it is seen from \eqref{GS=F} that $F(z)=G(z)S(z)$ which 
completes the proof.
\end{proof}

The following particular case of the Leech theorem is worth noting.

\begin{proposition}
Let $G\in\cL(\cY,\cX)\langle\langle z\rangle\rangle$ and $F\in\cL(\cU,\cX)\langle\langle
z\rangle\rangle$ satisfy the kernel identity
\begin{equation}
G(z)(k_{\rm nc, Sz}(z,\zeta) \otimes I_\cY)G(\zeta)^*=F(z)(k_{\rm nc, Sz}(z,\zeta) \otimes I_\cU) F(\zeta)^*.
\label{jan3ca}
\end{equation} 
Then there exists a partial isometry $W\in\cL(\cU,\cY)$ 
such that 
$F(z)=G(z)W$ and $G(z)=F(z)W^*$.
\label{P:leech}
\end{proposition}

\begin{proof}
The ``if" part is clear. To prove the ``only if" part, we 
follow the lines of the second proof of Theorem \ref{T:Leech}.
Identity \eqref{jan3ca} means that the kernel $K_{G,F}$ \eqref{jan3c} is zero, so that 
the formal power series $H(z)$ from the formal Kolmogorov decomposition \eqref{KGF-Kol} is also 
zero. Then the map $V: \, G_\beta^*x\to F_\beta^*x$ extends to a well-defined isometry from 
the subspace 
$\mathcal G=\overline{\rm span}_{\beta \in \free,\, x \in \cX}G_\beta^*x\subset \cY$  
onto the subspace $\mathcal R=\overline{\rm span}_{\beta \in \free,\, x \in \cX}
F_{\beta}^{*} x \subset \cU$. 
Extending $V$ to a partial isometry $W^*: \, \cY\to\cU$ by letting $W^*\vert_{\cY\ominus\mathcal 
G}=0$, we get $W^*G_\beta^*=F_\beta^*$ and $W F_{\beta}^{*} = 
G_{\beta}^{*}$ or equivalently, 
$$
F_\beta=G_\beta W\quad\mbox{and}\quad G_\beta=F_\beta W^*\quad \mbox
{for all $\beta\in\free$}.
$$
Now the equalities $F(z)=G(z)W$ and $G(z)=F(z)W^*$ follow.
\end{proof}

\section[Contractive multipliers from $H^2_{\cU}(\free)$ to
$H^2_{\bo,\cY}(\free)$]
{Contractive multipliers from $H^2_{\cU}(\free)$ to 
$H^2_{\bo,\cY}(\free)$ for admissible $\bo$}\label{S:Hardy-con-mult}
We now suppose that $\bo = \{ \omega_{j}\}_{j\ge 0}$  is an admissible weight sequence
as in \eqref{18.2} with associated weighted Hardy-Fock space $H^2_{\bo}(\free)$.
By Remark \ref{R:examples}, the associated kernel $k_{{\rm nc},\bo}$ is a {\em contractive kernel},
i.e., the right shift operator tuple $\bS_{\bo,R} = (S_{\bo,R,1} \cdots, S_{\bo, R,d})$ is a row contraction.
By Proposition \ref{P:referee}, the kernel $K=k_{\bo}$ admits a factorization \eqref{kcm2} for an appropriately 
chosen $G$. Our next goal is to construct this $G$ explicitly. Toward this end,  let us introduce the space
 \begin{equation}   \label{jan4a}
 \ell^2_\cY(\free) = \bigg\{ \{ f_\alpha \}_\alpha \in \free \colon \sum_{\alpha \in \free} \| f_\alpha \|^2_\cY < \infty \bigg\}
 \end{equation}
 which is isomorphic to the Fock space $H^2_\cY(\free)$ via the noncommutative $Z$-transform \eqref{1.25pre}.
Let us define the operator $\bPs_{\bo} \colon \ell^{2}_{\cY}(\free) 
 \to H^2_{\bo, \cY}(\free)$ (a weighted noncommutative $Z$-transform) 
 by
 $$
 \bPs_{\bo} = \operatorname{Row}_{\alpha \in \free} [ 
 \omega_{|\alpha|}^{-\frac{1}{2}} z^{\alpha} ] \colon
 \{ f_{\alpha} \}_{\alpha \in \free} \mapsto \sum_{\alpha \in \free} 
 \bo_{|\alpha|}^{-\frac{1}{2}} f_{\alpha} z^{\alpha}.
 $$
 We identify $\bPs_{\bo}$ with an operator-valued power series:
 \begin{equation}   \label{powerseries}
 \bPs_{\bo}(z) = \sum_{\alpha \in \free} \left( 
 \operatorname{Row}_{\beta \in \free}[  \delta_{\alpha, \beta}
( \omega_{|\alpha|}^{-\frac{1}{2}} \otimes I_{\cY}] \right) z^{\alpha}\colon 
 \{f_{\alpha}\}_{\alpha \in\free} \mapsto \sum_{\alpha \in \free} 
 \omega_{|\alpha|}^{-\frac{1}{2}} f_{\alpha} z^{\alpha}.
 \end{equation}
 Note that $\bPs_{\bo}(z)$ is the factor for the Kolmogorov 
 decomposition of the positive noncommutative kernel $k_{\bo}$ 
 \eqref{kbo}:
 \begin{equation}  \label{kbo-Koldecom}
     k_{\bo}(z, \zeta) =  \bPs_{\bo}(z) 
     \bPs_{\bo}(\zeta)^{*}.
 \end{equation}
 Assume now that $\bo$ is {\em strictly decreasing} ($\omega_j < \Omega_{j-1}$ for $j \ge 1$),
 and define a new weight $\bgam = \{ \gamma_{j} \}_{j \ge 0}$ by
\begin{equation}   \label{bgam}
    \gamma_{0} = 1, \quad \gamma_{j} = (\omega_{j}^{-1} - 
    \omega_{j-1}^{-1})^{-1} = \frac{ \omega_{j} 
    \omega_{j-1}}{\omega_{j-1} - \omega_{j}} \; \text{ for } \; j \ge 1.
\end{equation}
Since $\bo$ is strictly decreasing and positive, we see that $\bgam$ is a positive weight 
sequence.  It then follows from \eqref{kbo} and \eqref{bgam} that the kernel 
\begin{equation}
\widetilde k_{ \bo}(z, \zeta): = k_{ \bo}(z,
    \zeta) - \sum_{k=1}^{d} \overline{\zeta_{k}} k_{ \bo}(z,\zeta) z_{k} =\sum_{\alpha \in \free} \gamma_{|\alpha|}^{-1}
    z^{\alpha} \overline{\zeta}^{\alpha^{\top}} = k_{\bgam}(z, \zeta)\label{kbgam}
\end{equation}
is a positive noncommutative kernel.  The latter can also be seen as a consequence 
of the fact that the right shift operator tuple $\bS_{\bo, R}$ is a 
row contraction -- by applying item (1) in Proposition \ref{P:2.3} with 
$$
K(z, \zeta) = k_{\bo}(z, \zeta),\quad
\Theta(z) = \begin{bmatrix} z_1  & \cdots & z_d \end{bmatrix},\quad 
\widetilde K(z,\zeta) = k_{ \bo}(z, \zeta) \otimes I_d.
$$
One can iterate \eqref{kbgam} to solve for $k_{ \bo}$ in terms of $k_{\bgam}$ as follows:
\begin{align*}
 k_{ \bo}(z, \zeta) &=  k_{ \bgam}(z, \zeta) + 
  \sum_{k=1}^{d} \overline{\zeta_{k}} k_{ \bo}(z, \zeta) 
  z_{k} \\
& = k_{\bgam}(z, \zeta) + \sum_{k=1}^{d} 
\overline{\zeta_{k}} k_{\bgam}(z, \zeta) z_{k} + \sum_{\alpha 
\colon |\alpha| = 2} \overline{\zeta}^{\alpha^{\top}} k_{\bo}(z, \zeta) z^{\alpha} \\
& = \cdots = \sum_{\alpha \colon |\alpha| \le N} 
\overline{\zeta}^{\alpha^{\top}} k_{\bgam}(z, \zeta) 
z^{\alpha} + \sum_{\alpha \colon |\alpha| = N+1} 
\overline{\zeta}^{\alpha^{\top}} k_{\bo}(z, \zeta) 
z^{\alpha}.
 \end{align*}
 Taking a limit as $N \to \infty$ then gives
 \begin{equation}  \label{id1}
     k_{\bo}(z, \zeta) = \sum_{\alpha \in \free} 
     \overline{\zeta}^{\alpha^{\top}} k_{\bgam}(z, \zeta) 
     z^{\alpha}.
 \end{equation}
 If we now recall \eqref{kbo-Koldecom} (with $\bgam$ in place of 
 $\bo$), we can rewrite \eqref{id1} as
 $$
     k_{ \bo}(z, \zeta) = \bPs_{\bgam}(z) \left( k_{\rm 
     nc, Sz}(z, \zeta) \otimes I_{\ell^{2}_{\cY}(\free)} \right) 
     \bPs_{\bgam}(\zeta)^{*},
 $$
which is the expected factorization \eqref{kcm2} for the noncommutative positive contractive kernel
$k_{\bo}$.

\begin{remark}  \label{R:bergfact}
In case $\bo=\bmu_n$ (for $n>1$), the sequence $\bgam$ defined via formulas \eqref{bgam} turns out to be 
$\bgam=\bmu_{n-1}$.
Then the formula \eqref{id1} takes the form 
\begin{equation}
     k_{n}(z, \zeta) = \bPs_{n-1}(z) \left( k_{\rm
     nc, Sz}(z, \zeta) \otimes I_{\ell^{2}_{\cY}(\free)} \right)
     \bPs_{n-1}(\zeta)^{*},
\label{jan4c}
\end{equation}
where $\bPs_{n-1}(z)=\bPs_{\bgam}(z)$ is given by 
$$
\bPs_{n-1}(z) = \sum_{\alpha \in \free} \left({\rm Row}_{\beta \in \free}
[\delta_{\alpha, \beta} (\mu_{n-1,|\alpha|}^{-\frac{1}{2}} \otimes
I_{\cY})] \right) z^{\alpha} \in\cL(\ell^2_{\cY}(\free),\cY)\langle\langle
z\rangle\rangle.
$$  
\end{remark}

The single-variable ($d=1$) case of the next theorem is the main 
result of \cite{BBPAMS2017}.

 \begin{theorem}   \label{T:3.8}
Let $M_{\bPs_{\bgam}} \colon H^2_{\ell^2_{\cY}(\free)}(\free)\to H^2_{\bgam,\cY}(\free)$
be the multiplication operator associated with 
$\bPs_{\bgam}(z)$ defined as in \eqref{powerseries}. 
 Then
\begin{enumerate}
    \item $M_{\bPs_{\bgam}}|_{\ell^{2}_{\cY}(\free)}$ is an isometry 
    from the space of constant functions $\ell^{2}_{\cY}(\free)$ in $H^2_{\ell^2_\cY(\free)}(\free)$ into 
    $H^2_{\bgam, \cY}(\free)$.
    
\item  $M_{\bPs_{\bgam}}$ is a coisometry from 
$H^{2}_{\ell^{2}_{\cY}(\free)}(\free)$ 
onto $H^2_{\bo, \cY}(\free)$.
    
\item  A given formal power series  
$F(z)\in \cL(\cU,\cY)\langle\langle z\rangle\rangle$ is a contractive
multiplier from $H^2_{\cU}(\free)$ to $H^2_{\bo,\cY}(\free)$ if and only if it is of the form
\begin{equation}
F(z)=\bPs_{\bgam}(z)S(z)
\label{jan4d}
\end{equation}
for some contractive multiplier $S(z)$ from $H^2_{\cU}(\free)$ to
$H^2_{\ell^2_{\cY}(\free)}(\free)$. 

\item Any contractive multiplier $S$ from
$H^{2}_{\cU}(\free)$ into $H^{2}_{\ell^{2}_{\cY}(\free)}(\free)$ for
which \eqref{jan4d} holds has the form
$$
S(z)= S_{0}(z)+ \Theta(z) H(z)
$$
where $S_{0}$ is any particular contractive multiplier for which 
\eqref{jan4d} holds, where $\Theta(z) \in \cL(\widetilde \cU, \ell^{2}_{\cY}(\free))\langle
\langle z \rangle \rangle$ is the Beurling-Lax representer
for the $\bS_{1,R}$-invariant subspace ${\rm Ker}\, M_{\Psi_\gamma}
\subset H^{2}_{\ell^{2}_{\cY}(\free)}(\free)$ and
where $H(z)$ is a bounded multiplier from $H^{2}_{\cU}(\free)$ into
$H^{2}_{\widetilde \cU}(\free)$ chosen so that $S_{0} + \Theta H$ is 
still a contractive multiplier.
\end{enumerate}
\end{theorem}

\begin{proof}  For $\by = \{ y_{\alpha} \}_{\alpha \in 
    \free}$ in $\ell^{2}_{\cY}(\free)$, note that
 $$
 \| \bPs_{\bgam}(z) \by \|^{2}_{H^2_{\bgam, \cY}(\free)}  = 
 \sum_{\alpha \in \free} \bgam_{|\alpha|} \| \bgam_{\alpha|}^{-\frac{1}{2}} y_{\alpha} \|_{\cY}^{2} = \| \by 
 \|^{2}_{\ell^{2}_{\cY}(\free)},
 $$
 and statement (1) follows.  Statement (2) is an immediate consequence of the identity \eqref{id1} 
combined with 
 part (2) of Proposition \ref{P:2.3}.

\smallskip
     
 By Proposition \ref{P:2.3}, $F$ is a contractive multiplier from
$H^2_{\cU}(\free)$ to $H^2_{\bo,\cY}(\free)$ if and only if the formal kernel
$$
K_F(z,\zeta):=k_{\bo}(z,\zeta)\otimes I_{\cY}-F(z)(k_{\rm nc,
Sz}(z,\zeta)\otimes I_{\cU})F(\zeta)^*
$$
is positive. Due to \eqref{id1} we have
\begin{align*}
K_F(z,\zeta):=&\bPs_{\bgam}(z)\big(k_{\text{nc,
Sz}}(z,\zeta)\otimes I_{\ell^2_{\cY}(\free)}\big)\bPs_{\bgam}(\zeta)^*\\ 
&-F(z)(k_{\rm nc, Sz}(z,\zeta)\otimes I_{\cU})F(\zeta)^*,
\end{align*}
and by Theorem \ref{T:Leech}, the latter kernel is positive if and only if $F$ admits
the representation \eqref{jan4d} for some contractive multiplier $S(z)$ from
$H^2_{\cU}(\free)$ to $H^2_{\ell^2_{\cY}(\free)}(\free)$. Statement (3) now follows.

\smallskip

Suppose that $S_{0}$ is any particular contractive multiplier 
satisfying \eqref{jan4d} and $S$ is another such multiplier. Then
$M_{\bPs_{\bgam}}M_{S - S_{0}} = M_{\bPs_{\bgam}} (M_{S} - M_{S_{0}}) = 0$, 
i.e., ${\rm Ran}\, M_{S-S_{0}} \subset {\rm Ker}\, M_{\bPs_{\bgam}} = 
M_{\Theta} H^{2}_{\widetilde \cU}$.  It follows that $S - S_{0}$ has 
a factorization $S - S_{0} = \Theta H$ where $H$ is a multiplier from 
$H^{2}_{\widetilde \cU}(\free)$ to $H^{2}_{\cU}(\free)$ such that
$S_{0} + \Theta H$ is still a contractive multiplier.  As the 
converse direction is clear,  this completes the proof of statement 
(4).
\end{proof}

\begin{proposition}  \label{P:HardyBergmancoiso}
A contractive multiplier $F(z)$ from $H^2_{\cU}(\free)$ to $H^2_{\bo,\cY}(\free)$ is coisometric 
if and only if it is of the form $F(z)=\bPs_{\bgam}(z)$
for some coisometric operator $W\in\cL(\cU,\ell^2_{\cY}(\free))$.
\end{proposition}

\begin{proof}
By Proposition \ref{P:2.3}, $F(z)$ is a coisometric multiplier from 
$H^{2}_{\cU}(\free)$ to $H^2_{\bo,\cY}(\free)$ if and only if the 
associated kernel $K_{F}$ vanishes:
$$
 K_{F}(z, \zeta): = (k_{{\rm nc}, \bo}(z, \zeta) \otimes I_{\cY}) - 
 F(z) (k_{{\rm nc, Sz}}(z, \zeta) \otimes I_{\cU}) F(\zeta)^{*} = 0,
$$
or, equivalently by \eqref{id1},
\begin{equation}  \label{coiso-criterion}
\bPs_{\bgam}(z) \left( k_{\rm nc, Sz}(z, \zeta) \otimes 
I_{\ell^{2}_{\cY}(\free)} \right) \bPs_{\bgam}(\zeta)^{*}   = F(z) \left( k_{\rm nc, Sz}(z, 
\zeta) \otimes I_{\cU}\right) F(\zeta)^{*}.
\end{equation}

If $F(z) = \bPs_{\bgam}(z) W$ with $W W^{*} = 
I_{\ell^{2}_{\cY}(\free)}$, use the fact that $k_{\rm nc, Sz}(z, 
\zeta)$ has scalar coefficients to get
$$
W \left( k_{\rm nc, Sz}(z, \zeta) \otimes I_{\cU}\right) W^{*} =
\left( k_{\rm nc, Sz}(z, \zeta) \otimes I_{\ell^{2}_{\cY}(\free)} 
\right) W W^{*} =  k_{\rm nc, Sz}(z, \zeta) \otimes I_{\ell^{2}_{\cY}(\free)}
$$
to arrive at \eqref{coiso-criterion} and the sufficiency direction in 
Proposition \ref{P:HardyBergmancoiso} follows. 

\smallskip

To verify the necessity direction, assume that 
\eqref{coiso-criterion} holds. Use Proposition \ref{P:leech} to 
deduce that there is a partial isometry $W\in\cL(\cU,\ell^2_{\cY}(\free))$
such that 
\begin{equation}
F(z)=\bPs_{\bgam}(z)W\quad\mbox{and}\quad \bPs_{\bgam}(z) = F(z)W^*.
\label{march8c}
\end{equation}
It remains to show that $W \colon \cU\to \ell^2_{\cY}(\free)$ is a coisometry.
To this end, we compare the coefficients of $z^\alpha$ in \eqref{march8c} to get
$$
F_\alpha=\bgam_{|\alpha|}^{-\frac{1}{2}}W_\alpha,\quad
F_\alpha W_\alpha^*=\bgam_{|\alpha|}^{-\frac{1}{2}}I_{\cY},\quad
F_\alpha W_\beta^*=0 \; ; (\alpha\neq \beta).
$$
Therefore, 
$$
W_\alpha W_\alpha^*=\mu_{n-1,|\alpha|}^{\frac{1}{2}}F_\alpha W_\alpha^*=I_{\cY},\quad
W_\alpha W_\beta^*=\mu_{n-1,|\alpha|}^{\frac{1}{2}}F_\alpha W_\beta^*=0 \; \; (\beta\neq \alpha),
$$
so that $W=\{W_\alpha\}_{\alpha\in\free}$ is indeed a coisometry from $\cU$ to 
$\ell^2_{\cY}(\free)$.  
\end{proof}

\begin{proposition} 
There are no isometric multipliers from $H^2_{\cU}(\free)$ to $H^2_{\bo,\cY}(\free)$ for any strictly decreasing
admissible $\bo$.
\label{P:inn}
\end{proposition}

\begin{proof}
Let $F(z)=\sum_{\alpha\in\free}F_\alpha z^\alpha$ be an isometric multiplier from $H^2_{\cU}(\free)$ to
$H^2_{\bo,\cY}(\free)$. Then in particular,
\begin{equation}
\|Fu\|^2_{H^2_{\bo,\cY}(\free)}=\sum_{\alpha\in\free}\omega_{|\alpha|}\|F_\alpha 
u\|_{\cY}^2=\|u\|^2_{\cU}
\quad\mbox{for all}\quad u\in\cU.
\label{march8e}
\end{equation}
On the other hand, for the polynomial $p(z)=z_1 u$, we have $\|p\|_{H^2_{\cU}}=\|u\|_{\cU}$ and 
therefore,
\begin{equation}
\|Fp\|^2_{H^2_{\bo,\cY}(\free)}=\sum_{\alpha\in\free}\omega_{|\alpha|+1}\|F_\alpha u\|_{\cY}^2=
\|p\|_{H^2_{\cU}}=\|u\|^2_{\cU}.
\label{march8d}
\end{equation}
We then conclude from \eqref{march8e} and \eqref{march8d} that 
$$
0=\sum_{\alpha\in\free}(\omega_{|\alpha|}-\omega_{|\alpha|+1})\|F_\alpha u\|^2_{\cY},
$$
which is possible (since by assumption $\omega_{|\alpha|}>\omega_{|\alpha|+1}$ for all $\alpha\in\free$)
only if $F_\alpha=0$ for all $\alpha\in\free$. The latter contradicts \eqref{march8e}.
\end{proof}

In the sequel we shall be interested in more general contractive 
multipliers $M_{\Theta}$ from $H^2_{\bo', \cU}$ to $H^2_{\bo, \cY}$; note 
that Propositions \ref{P:HardyBergmancoiso} and \ref{P:inn} are 
concerned with the case $\omega'_n = 1$.  The following characterization of such 
contractive multiplication operators $M_{\Theta}$ among all operators 
between two weighted Bergman spaces, well known for the classical Hardy-space case 
where $d=1$, $\omega'_k = \omega_k = 1$ for all $k$ (see e.g.\  \cite{rosenrov}), will be useful in the sequel.
Note also that for the case of a general $d$ with $\omega'_k = \omega_k = 1$ for all $k$, the result can be seen as 
the special case of Theorem \ref{T:CLT} where $\cM = H^2_\cY(\free)$ and $\cN = H^2_\cU(\free)$.

\begin{proposition}   \label{P:RR}  Suppose that $G$ is an operator 
    from $H^2_{\bo', \cU}(\free)$ to $H^2_{\bo, \cY}(\free)$.  Then $G$ has 
    the form $G = M_{\Theta}$ for a contractive multiplier $\Theta$ 
    between $H^2_{\bo',\cU}(\free)$ and $H^2_{\bo, \cY}(\free)$ if and only if $\| G \| 
    \le 1$ and $G$ satisfies the intertwining conditions:
 \begin{equation}   \label{intertwine5}
     G S_{\bo', R, j} = S_{\bo,R,j}G \quad  \text{ for } j=1, \dots, d.
\end{equation}
\end{proposition}

\begin{proof}  If $G = M_{\Theta}$, then the intertwining conditions  \eqref{intertwine5} hold since $M_{\Theta}$ is a 
left-multiplication operator while each $S_{\bo',R,j}$ and $S_{\bo,R,j}$ is a right multiplication operator. 
   
Conversely, suppose that the operator $G: \, H^2_{\bo', \cU}(\free)\to  H^2_{\bo, \cY}(\free)$ 
satisfies the intertwining conditions \eqref{intertwine5}.  Define a power series 
$\Theta(z) \in \cL(\cU, \cY)\langle  \langle z \rangle \rangle$ by
 $$
 \Theta(z) u = (G u)(z)
 $$
 for $u \in \cU$,
 where here we view $u = u \cdot z^{\emptyset}$ as an element of 
 $H^2_{\bo', \cU}(\free)$.  From the intertwining conditions 
 \eqref{intertwine5}, we deduce that 
 $$
 (Gp)(z) = \Theta(z) \cdot p(z) = (M_{\Theta} p )(z)
 $$
 for any noncommutative polynomial $p \in \cU\langle z \rangle$.  
 Thus $G$ is a multiplication operator when restricted to the dense 
 set of polynomials in $H^2_{\bo', \cU}(\free)$.  By assumption $G$ is bounded; 
 an approximation argument then gives us that $\Theta$ is a 
 multiplier (contractive multiplier if $\| G \| \le 1$) and 
 necessarily $G = M_{\Theta}$ as an operator from $H^2_{\bo', \cU}(\free)$ to 
 $H^2_{\bo, \cY}(\free)$.
    \end{proof}

\section{$H^2_{\bo, \cY}(\free)$-Bergman-inner multipliers}  \label{S:Berg-inner-mult}

\begin{definition}
We say that a formal power series $\Theta \in \cL(\cU, \cY)\langle \langle z \rangle \rangle$ is a
{\em $H^2_{\bo, \cY}(\free)$-Bergman inner multiplier} if 
\begin{enumerate}
\item $\|\Theta u\|_{H^2_{\bo,\cY}(\free)}=\|u\|_{\cU}\; $ for all $\; u\in\cU$, and
\item $\Theta \cdot u \perp \Theta\cdot \bS_{1, R}^{\alpha} v\; $ in $\; H^2_{\bo, \cY}(\free)\; $ for all $\; u,v\in\cU$ and 
all nonempty $\alpha \in \free$.  
\end{enumerate} 
\label{D:defin}
\end{definition}

The next result says  in particular that $H^2_{\bo, \cY}(\free)$-Bergman-inner multipliers are automatically also
contractive multipliers.  For the case of single-variable weighted Bergman spaces, this result goes back to 
\text{Richter-et-al?}.  More generally, this result has been obtained for the commutative setting by 
Bhattacharjee-Eschmeier-Keshari-Sarkar \cite[Theorem 6.2]{BEKS} and for the general setting by Popescu 
\cite[Theorem 4.2]{PopescuMANN} by different proofs which do not lead to the estimate \eqref{march13}.

\begin{theorem}  \label{T:12.2}
Let $\Theta\in\cL(\cU,\cY)\langle\langle z\rangle\rangle$ be such that 
\begin{enumerate}
\item $\|M_{\Theta} u\|_{H^2_{\bo,\cY}(\free)}\le\|u\|_{\cU}\; $ for all $\; u\in\cU$, and
\item $M_{\Theta} u\perp M_{\Theta} {\bf S}_1^\alpha v$ for all $u,v\in\cU$ 
and all nonempty $\alpha\in\free$.
\end{enumerate}
Then $\Theta$ is a contractive multiplier from $H^2_\cU(\free)$ to $H^2_{\bo, \cY}(\free)$, as shown
by the following general estimate:
\begin{align}
\|M_{\Theta} f\|_{H^2_{\bo,\cY}(\free)}^2\le &
\|f\|^2_{H^2_{\cU}(\free)}\label{march13}\\
&-\sum_{\alpha\in\free}\sum_{j=1}^d
\left\|(I-S_{\bo,R,j}^*S_{\bo,R,j})^{\frac{1}{2}}
M_{\Theta}({\bf S}_{1,R}^{*})^{j\alpha^\top}f
\right\|^2_{H^2_{\bo,\cY}(\free)}\notag
\end{align}
for all $f\in H^2_{\cU}(\free)$. Moreover, if $\|M_{\Theta} u\|_{H^2_{\bo,\cY}(\free)}=\|u\|_{\cU}$ for 
all $u\in\cU$ (so $\Theta$ is $H^2_{\bo, \cY}(\free)$-Bergman-inner), then 
\eqref{march13} holds with equality for all $f\in H^2_{\cU}(\free)$.
\end{theorem}

\begin{proof}
We first verify \eqref{march13} for any $\cU$-valued ``polynomial"
\begin{equation}
f(z)=\sum_{\alpha \in\free: \, |\alpha|\le m} f_\alpha z^\alpha.
\label{jul21}
\end{equation}
Let ${\bf S}^*_{1,R}=(S^*_{1,R,1},\ldots,S^*_{1,R,d})$ be the right backward shift
tuple on $H^2_{\cU}(\free)$  so that for the polynomial $f$ as in \eqref{jul21},
\begin{equation}
S_{1,R,j}^* \colon  \sum_{\alpha \in\free: \, |\alpha|\le m}f_\alpha z^\alpha
\mapsto \sum_{\alpha\in \free: \, |\alpha|<m}f_{\alpha j} z^\alpha
\label{jul21a}
\end{equation}
for $j=1,\ldots,d$. By the assumptions of the theorem, and since the ranges of $S_{\bo,R,i}$ and
$S_{\bo,R,j}$ are orthogonal whenever $i\neq j$, we have
\begin{align}
\|M_{\Theta} f\|^2=&
\bigg\|\sum_{\alpha \in\free: \, |\alpha|\le m}{\bf
S}_{\bo,R}^{\alpha^\top} M_{\Theta} f_\alpha\bigg\|^2\label{nov10}\\
=&\|M_{\Theta} f_\emptyset \|^2+
\bigg\|\sum_{j=1}^d\sum_{\alpha\in\free: \, |\alpha|\le m-1}
 {\bf S}_{\bo,R}^{j\alpha^\top} M_{\Theta} f_{\alpha j} \bigg\|^2 \notag\\
\le & \|f_{\emptyset}\|^2+
\bigg\|\sum_{j=1}^d\sum_{\alpha \in\free: \, |\alpha|\le m-1}
 {\bf S}_{\bo,R}^{j\alpha^\top}M_{\Theta} f_{\alpha j} \bigg\|^2\notag\\
=& \|f_{\emptyset}\|^2+\sum_{j=1}^d\bigg\|
S_{\bo,R,j}M_{\Theta} \bigg(\sum_{\alpha \in\free: \,
|\alpha|\le m-1} f_{\alpha j}z^\alpha\bigg)
 \bigg\|^2\notag\\
=&\|f_{\emptyset}\|^2+\sum_{j=1}^d\left\|
S_{\bo,R,j}M_{\Theta} (S_{1,R,j}^*f) \right\|^2\notag\\
=&\|f_{\emptyset}\|^2+\sum_{j=1}^d\left\|M_{\Theta} (S_{1,R,j}^*f)
\right\|^2-\sum_{j=1}^d\left\|(I-S_{\bo,R,j}^*S_{\bo,R,j})^{\frac{1}{2}} M_\Theta (S_{1,R,j}^*f) \right\|^2.\notag
\end{align}
Observe that for any $\alpha \in\free$ with $|\alpha|\le m$,
\begin{equation}
{\bf S}_{1,R}^{*\alpha^\top}f=\sum_{\beta\in\free: |\beta|\le m-|\alpha|}f_{\beta\alpha}z^{\beta}
\label{nov11} 
\end{equation}
which can be seen upon iterating formula \eqref{jul21a}.  
It follows from \eqref{nov11} that
$$
({\bf S}_{1,R}^{*\alpha^\top}f)_{\emptyset}=f_{\alpha}.
$$
Replacing $f$ by ${\bf S}_{1,R}^{*\alpha^\top}f$ in \eqref{nov10} now gives
\begin{align}
\|M_{\Theta}{\bf S}_{1,R}^{*\alpha^\top}f\|^2
\le \|f_\alpha\|^2 &
+\sum_{j=1}^d\left\|M_{\Theta}({\bf S}_{1,R}^{*})^{j \alpha^\top}f \right\|^2 \notag\\
& -\sum_{j=1}^d\left\|(I-S_{\bo,R,j}^*S_{\bo,R,j})^{\frac{1}{2}}
M_{\Theta}({\bf S}_{1,R}^{*})^{j \alpha^\top}f \right\|^2
\label{8.3'}
\end{align}
for any $\alpha\in\free$.  Iterating the inequality \eqref{nov10} using
\eqref{8.3'} then gives, for any $m' = 1,2,\dots$,
\begin{align} 
&\|M_{\Theta} f\|_{H^2_{\bo,\cY}(\free)}^2  \le \sum_{|\alpha| < m'}\|f_\alpha \|^2_{\cU}  
+ \sum_{|\alpha| = m'} \left\|  M_\Theta (\bS_{1,R})^{\alpha^\top} f \right\|^2_{H^2_{\bo, \cY}(\free)}  \notag  \\
&  \quad  \quad \quad \quad
-\sum_{|\alpha|< m'}\sum_{j=1}^d   
\left\|(I-S_{\bo,R,j}^*S_{\bo,R,j})^{\frac{1}{2}}
M_{\Theta}({\bf S}_{1,R}^{*})^{j \alpha^\top}f
\right\|^2_{H^2_{\bo,\cY}(\free)}.\notag 
\end{align}
Since $f_\alpha = 0$ as well as $(\bS_{1,R}^*)^{\alpha^\top} f = 0$ once $|\alpha| \ge m$, it follows that, 
once $m' \ge m$, this last inequality collapses to
\begin{align}  
\| M_\Theta f \|^2& \le
\|f\|^2_{H^2_{\cU}(\free)}\notag\\
&-\sum_{\alpha \in\free:|\alpha|< m'}\sum_{j=1}^d
\left\|(I-S_{\bo,R,j}^*S_{\bo,R,j})^{\frac{1}{2}}
M_{\Theta}({\bf S}_{1,R}^{*})^{j \alpha^\top}f
\right\|^2_{H^2_{\bo,\cY}(\free)}.\label{march13a}
\end{align}
Letting $m' \to \infty$ in \eqref{march13a} now implies the validity
of \eqref{march13} for every $\cU$-valued polynomial $f\in H^2_{\cU}(\free)$.
We then get the result for a general $f$ in $H^{2}_{\cU}(\free)$ by 
approximating $f$ by finite truncations of its series 
representation $f(z) = \sum_{\alpha \in \free} f_{\alpha} z^{\alpha}$.  

\smallskip

If equality holds in condition (1) of the theorem (i.e., if $\|M_{\Theta} u\|_{H^2_{\bo,\cY}(\free)}=\|u\|_{\cU}$  holds for
all $u\in\cU$), then we have equalities throughout \eqref{nov10}, \eqref{8.3'} and therefore,
in \eqref{march13} as well. 
 \end{proof}

Part of the content of the next result is a converse of the previous result:  {\em a contractive multiplier from
$H^2_\cU(\free)$ to $H^2_{\bo, \cY}(\free)$ which acts isometrically on the coefficient space $\cU$ is
in fact a Bergman-inner multiplier.}

\begin{theorem}  \label{T:bergin}
Let $\Theta\in\cL(\cU,\cY)\langle\langle z\rangle\rangle$ be a contractive multiplier from $H^2_\cU(\free)$ to
$H^2_{\bo, \cY}(\free)$ which is isometric on constants:  
$$
\| M_\Theta u \|_{H^2_{\bo, \cY}(\free)} = \| u \|_\cU\quad\mbox{for all}\quad u \in \cU.  
$$
Then: 
\begin{enumerate}
\item There is a unique contractive multiplier $S$ from $H^2_{\cU}(\free)$ to 
$H^2_{\ell^{2}_{\cY}(\free)}(\free)$ such that 
\begin{equation}
\Theta(z) =\bPs_{\bgam}(z)S(z)
\label{march9}
\end{equation}
where $\bPs_{\bgam}(z)$ is defined as in \eqref{powerseries}. Moreover, this unique $S$ is a strictly inner multiplier 
from $H^2_{\cU}(\free)$ to $H^2_{\ell^{2}_{\cY}(\free)}(\free)$.
\item $\Theta u\perp \Theta {\bf S}_1^\alpha v$ for all $u,v\in\cU$ 
and all nonempty $\alpha\in\free$, i.e., $\Theta$ is a Bergman-inner multiplier.
\end{enumerate}
\end{theorem}

\begin{proof}
Since $\Theta$ is, in particular, a contractive multiplier from $H^2_{\cU}(\free)$ to 
$H^2_{\bo,\cY}(\free)$, it is (by Theorem \ref{T:3.8} (3)) of the form \eqref{march9}
for some contractive multiplier $S(z)$ from $H^2_{\cU}(\free)$ to 
$H^2_{\ell^2_{\cY}(\free)}(\free)$. By assumption we have
\begin{equation}
\|u\|_{\cU}= \|M_{\Theta} u\|_{H^2_{\bo,\cY}(\free)}=\|M_{\bPs_{\bgam}} M_{S} u\|_{H^2_{\bo,\cY}(\free)}
\label{march6}
\end{equation}
for all $u\in\cU$. Since the operator 
$$
M_{S} \colon \cU \to 
H^2_{\ell^2_{\cY}(\free)}(\free)
$$
is contractive while the operator
$$
 M_{\bPs_{\bgam}} \colon H^2_{\ell^2_{\cY}(\free)}(\free)\to  H^2_{\bo,\cY}(\free)
$$ 
is coisometric (by Theorem \ref{T:3.8} (2)), equality \eqref{march6} 
implies that
\begin{equation}
\|M_{S} u\|_{H^2_{\ell^2_{\cY}(\free)}(\free)}=\|u\|_{\cU}\quad \text{and}\quad
 M_{S} u\perp
{\rm Ker} \, M_{\bPs_{\bgam}}\quad \text{for all}\quad u\in\cU.
\label{march6a}
\end{equation}
The first relation in \eqref{march6a} implies that $S$ is a strictly inner 
multiplier (by Lemma \ref{L:in}).  Let us rewrite the factorization 
\eqref{march9} in multiplication-operator form
\begin{equation}   \label{march9'}
    M_{\Theta} = M_{\bPs_{\bgam}} M_{S} \colon H^{2}_{\cU}(\free) \to 
    H^2_{\bo, \cY}(\free).
\end{equation}
As $M_{\bPs_{\bgam}}$ is a coisometry,
the second relation in \eqref{march6a} implies that
\begin{equation}  \label{S=Psi*PsiS}
M_{S} = M_{\bPs_{\bgam}}^{*} M_{\bPs_{\bgam}} M_{S}.
\end{equation}
Then multiplication of relation 
\eqref{march9'} on the left by $M_{\bPs_{\bgam}}^{*}$ and making use of 
\eqref{S=Psi*PsiS} leaves us with 
$$
 M_{\bPs_{\bgam}^{*}} M_{\Theta} = M_{\bPs_{\bgam}}^{*} M_{\bPs_{\bgam}} 
 M_{S}  = M_{S},
$$
i.e., $S$ is uniquely determined from $\Theta$ via the relation
$$
 S(z) u = \big( M_{\bPs_{\bgam}}^{*} M_{\Theta} u \big)(z)\quad \text{for all}\quad u \in \cU.
$$
 This completes the proof of part (1) of the theorem.

\smallskip

Again making use of relations \eqref{march9'} and \eqref{S=Psi*PsiS}, we see that, for $u, 
v \in \cU$ and $\alpha$ any nonempty word in $\free$,
\begin{align*}
\langle M_{\Theta} u, \, M_{\Theta} (\bS_{1,R}^\alpha v) 
\rangle_{\cA_{n,\cY}(\free)}  
   &  =\langle M_{\bPs_{\bgam}} M_{S} u, \, M_{\bPs_{\bgam}}\bS_{1,R}^\alpha 
   M_{S} v\rangle_{\cA_{n,\cY}(\free)}\\
&=\langle M_{\bPs_{\bgam}}^*M_{\bPs_{\bgam}} M_{S} u, \, \bS_{1,R}^\alpha 
M_{S}v  \rangle_{H^2_{\ell^2_{\cY}(\free)}(\free)}\\
&=\langle M_{S}u, \, M_{S} \bS_{1,R}^\alpha  v\rangle_{H^2_{\ell^2_{\cY}(\free)}(\free)}=0
\end{align*}
where we use the fact that $S$ is strictly inner for the last step.  This completes the proof of part (2)
of the theorem.
\end{proof}

\begin{remark}  \label{R:McCTvsBergman}
 Bergman-inner multipliers and McCT-inner multipliers are closely related in still other ways besides that revealed by
Theorem \ref{T:bergin}.  Indeed,  given a shift-invariant subspaces $\cM \subset H^2_{\bo, \cY}(\free)$ with Beurling-Lax representation
$\cM = \Theta \cdot H^2_{\cU}(\free)$ for a McCT-inner multiplier $\Theta$,  if we set $\cF: = \cU \cap \operatorname{Ran} \Theta^*$ 
and $\Psi(z) = \Theta(z)|_\cF$ and view $\Psi$ as an
element of $\cL(\cF, \cY)\langle \langle z \rangle \rangle$,  then $\Psi$ is Bergman inner and the wandering subspace
$\cE = \cM \ominus \left( \bigoplus_{j=1}^d S_{\bo, j} \cM\right)$  for $\cM$ is given by $\cE = \Psi \cdot \cF$
(see \cite[Theorem 6.6]{BEKS} for the commutative setting and \cite[Theorem 4.3]{PopescuMANN} for the general setting).  
\end{remark}

\begin{remark} \label{R:innermultno}
If $S$ is a strictly inner multiplier from $H^2_{\cU}(\free)$ to $H^2_{\ell^{2}_{\cY}(\free)}(\free)$,
then $\Theta$ of the form \eqref{march9} is a contractive multiplier from $H^2_{\cU}(\free)$ 
to $H^2_{\bo,\cY}(\free)$ by Theorem \ref{T:3.8}. However, $\Theta$ does not have to be 
a Bergman-inner multiplier.  For example, let $d=1$ and $\bo = \mu_2=\{\frac{1}{j+1}\}_{j\ge 0}$.
The single-variable function
$S(z) = \frac{1}{\sqrt{2}}\sbm{ z \\ z \\ 0 \\ \vdots }$ is $H^{2}_{\ell^{2}}$-inner and 
therefore, the function 
$$
 \Theta(z) = \Psi_{1}(z) S(z)=
\frac{1}{\sqrt{2}}
 \begin{bmatrix} 1 & z & z^{2} & \cdots \end{bmatrix} \sbm{ z \\ z \\ 0 \\ \vdots }
     = \frac{1}{\sqrt{2}} (z + z^{2})
$$
is a contractive multiplier from $H^2$ to $\cA_{2}$. However, 
$\Theta$ is not $\cA_{2}$-Bergman inner.
One reason is that $\|\Theta\|^{2}_{\cA_{2}} = \frac{1}{2} (\frac{1}{2} +
\frac{1}{3}) \ne 1$, and another reason is that 
$M_{\Theta} \cdot 1 =
\frac{1}{\sqrt{2}} (z + z^{2})$ and $M_{\Theta} \cdot z =
\frac{1}{\sqrt{2}} (z^{2} + z^{3})$ are not orthogonal in $\cA_{2}$.
\end{remark}

\begin{remark} \label{R:3.22}
Due to Theorems \ref{T:12.2} and \ref{T:bergin}, $H^2_{\bo,\cY}(\free)$-Bergman-inner multipliers can be equivalently
defined as noncommutative formal power series $\Theta\in \cL(\cU, \cY)\langle\langle z\rangle\rangle$ 
(for some Hilbert space $\cU$) which are contractive multipliers from $H^2_{\cY}(\free)$ to 
$H^2_{\bo,\cY}(\free)$ which are isometric on $\cU$.
\end{remark}

\chapter[Stein relations and observability-operator range spaces]
{Stein relations and observability-operator range spaces}\label{S:Stein}
In this chapter we flesh out Theme 1 from the Introduction for the $\bo$-setting.
In the classical setting of the discrete-time linear system \eqref{1.1pre} (see e.g.\ \cite{DulPag}), the output stability
of the pair $(C,A)$, that is, the boundedness of the observability operator $\cO_{C,A}: \cX\to H^2_{\cY}$
(see \eqref{1.4pre}), can be expressed in terms of certain Stein inequality, the minimal 
positive-semidefinite solution of which turns out to be equal to the {\em observability gramian}
$$
{\mathcal G}_{C, A}= \sum_{n=0}^\infty A^{*n}C^*CA^n.
$$
Precise statements are recorded next; here the notational conventions from Definition \eqref{D:op-ineq} are in place.

\begin{theorem}
\label{T:bbf1}
An output pair $(C,A)\in\cL(\cX,\cY)\times \cL(\cX)$ is output-stable (i.e., the observability operator $\cO_{C,A}$ is bounded as an operator
from $\cX$ into $H^2_\cY$) if and only if the Stein inequality
\begin{equation}
H-A^*HA \succeq C^*C
\label{bbfStein-ineq}
\end{equation}
has a positive semidefinite solution $H\in \cL(\cX)$. In this case, 
\begin{enumerate} 
\item The observability gramian $H={\mathcal G}_{C, A}$ is the minimal positive semidefinite solution of 
\eqref{bbfStein-ineq} and satisfies the Stein equation
\begin{equation}  \label{bbf1.5}
H - A^{*}HA = C^{*}C.
\end{equation}
\item ${\mathcal G}_{C, A}$ is a unique solution to this equality if  $A$ is strongly stable.
\item If $A$ is a contraction, then the Stein equation \eqref{bbf1.5} has a unique positive semidefinite solution 
if and only if $A$ is strongly stable.
\end{enumerate}
\end{theorem}
The range of the observability operator of an output-stable pair turns out to be a generic backward-shift
invariant subspace of $H^2_{\cY}$ that in turn serves as the functional-model space for a Hilbert space contraction 
operator and on the other hand is used to fairly explicitly construct the Beurling-Lax representer of a shift-invariant subspace
of $H^2_{\cY}$. 

\smallskip

Observability operators \eqref{1.17pre} appearing in the context of the time-varying linear system \eqref{1.13pre}
and multivariable observability operators \eqref{1.27pre} arising from a noncommutative Fornasini-Marchesini 
linear system \eqref{1.20pre} produce two quite different generalizations of the above observations and 
eventually lead to the functional-model theory of $n$-hypercontractions and $d$-row contractions, respectively. 
Both generalizations naturally merge into the framework 
of the linear system \eqref{1.31pre} which in  turn leads to the functional-model theory for row 
$n$-hypercontractions and to Beurling-Lax type theorems in the setting of weighted Bergman-Fock spaces. 

\smallskip

In this chapter, we consider stability questions and observability operator range spaces in the setting of the weighted 
Hardy-Fock space $H^2_{\bo,\cY}(\free)$ with an admissible weight sequence $\bo$; in addition we take care 
to spell out the special case $\bo=\bmu_n$ (i.e., of the standard weighted Bergman-Fock space) 
whenever the results can be simplified or detailed to a larger extent.

\section{Observability operators and gramians}
Let $\bo=\{\omega_j\}_{j\ge 0}$ be an admissible  weight sequence (in the sense of \eqref{18.2}). 
Let us make the following definition.
\begin{definition}   \label{D:bo-output-stable}
We say that the output pair $(C,\bA)\in\cL(\cX,\cY)\times\cL(\cX)^d$ is 
 {\em $\bo$-output stable} if the $\bo$-observability operator
\begin{equation}
\cO_{\bo,C,\bA}: \, x\mapsto
\sum_{\alpha\in\free}(\omega_{|\alpha|}^{-1}C\bA^{\alpha}x)z^\alpha
\label{18.2aaa}
\end{equation}
is bounded from $\cX$ to $H^2_{\bo,\cY}(\free)$ or equivalently, the $\bo$-observability gramian
\begin{equation}
\cG_{\bo,C,\bA}:=\cO_{\bo,C,\bA}^*\cO_{\bo,C,\bA}=
\sum_{\alpha\in\free}\omega_{|\alpha|}^{-1}\bA^{*\alpha^\top}C^*C\bA^{\alpha}
\label{18.2aa}
\end{equation}
is bounded on $\cX$. 
\end{definition}
We next recall the power series 
\begin{equation}
R_\bo(\lambda)=R_{\bo,0}(\lambda)
=\sum_{j=0}^\infty \omega_j^{-1}\lambda^j\quad\mbox{and}\quad
R_{\bo,k}(\lambda)=\sum_{j=0}^\infty \omega_{j+k}^{-1}\lambda^j \; \; (k\ge 1)
\label{1.6g}
\end{equation}
used in realization formulas \eqref{obsreal} and \eqref{thetareal} for $\cO_{\bo,C,\bA}$ and $\Theta_{\bo,\bU_\alpha}$.
Due to conditions \eqref{18.2}, the series \eqref{1.6g} converge in the open unit disk $\D$ and do not vanish there.
Hence, the reciprocal power series
\begin{equation}
\sum_{j=0}^\infty c_j\lambda^j:=\frac{1}{R_{\bo}(\lambda)}
=\bigg(\sum_{j=0}^\infty \omega_j^{-1}\lambda^j\bigg)^{-1}
\label{1.7g}
\end{equation}
converges on $\D$; for reasons that will become clear shortly, we assume that
the latter power series belongs to the Wiener class $W^+$, that is, the coefficients
$\{c_{j}\}_{j \ge 0}$ appearing in \eqref{1.7g} are absolutely summable:
\begin{equation}
\text{if } c_0=1 \text{ and recursively } c_k=-\sum_{j=0}^{k-1}c_j\omega_{k-j}^{-1}, \;
\text{ then } \; \sum_{j=0}^\infty |c_j|<\infty.
\label{1.8g}
\end{equation}
Meanwhile we record a useful identity satisfied by the sequences $\{\omega_j\}$ and $\{c_j\}$.
\begin{lemma}   \label{L:useful}
The sequence $\bo$ and the derived sequence $\{c_j \}_{j \ge 0}$ \eqref{1.8g} satisfy
the
equality
\begin{equation}
\sum_{r=0}^j\bigg(\frac{1}{\omega_{j-r}}\cdot
\sum_{\ell=1}^k\frac{c_{\ell+r}}{\omega_{k-\ell}}\bigg) =- \frac{1}{\omega_{k+j}}\quad \mbox{for}\quad j,k\ge 0.
\label{3.18u}
\end{equation}
\end{lemma}

\begin{proof}
By the recursion \eqref{1.8g}, we have
\begin{equation}
0=\sum_{\ell=0}^{k+r}\frac{c_\ell}{\omega_{k+r-\ell}}=\sum_{\ell=0}^{r}\frac{c_\ell}{\omega_{k+r-\ell}}+
\sum_{\ell=1}^{k}\frac{c_{r+\ell}}{\omega_{k-\ell}}\qquad (k+r\ge 1).
\label{3.18v}
\end{equation}
Substituting the latter equality into the left side of \eqref{3.18u}, changing the order of summation
and then again using \eqref{3.18v} (with $r$ instead of $k+r$) gives
\begin{align*}
\sum_{r=0}^j\bigg(\frac{1}{\omega_{j-r}}\cdot
\sum_{\ell=1}^k\frac{c_{\ell+r}}{\omega_{k-\ell}}\bigg) &=- \sum_{r=0}^j\bigg(\frac{1}{\omega_{j-r}}\cdot
\sum_{\ell=0}^{r}\frac{c_\ell}{\omega_{k+r-\ell}}\bigg)\\
& =-\sum_{r=0}^j\bigg(\frac{1}{\omega_{k+j-r}}\cdot\sum_{\ell=0}^r\frac{c_\ell}{\omega_{r-\ell}}\bigg)=
-\frac{1}{\omega_{k+j}}\cdot \frac{c_0}{\omega_{0}}=-\frac{1}{\omega_{k+j}}.
\end{align*}
and identity \eqref{3.18u} follows.
\end{proof}
We remark that the weights \eqref{1.6pre} meet the assumption \eqref{1.8g} as in this case,
$R_{\bmu_n}(\lambda):=R_n(\lambda)=(1-\lambda)^{-n}$,
and hence the power series \eqref{1.7g} amounts to a polynomial. The assumption \eqref{1.8g}
is imposed in order to define the operatorial maps \eqref{1.11g} and \eqref{3.5ag} below,
which in turn, will allow us to introduce the notion of an $\bo$-hypercontraction that 
extends the notion of a row-hypercontraction. 
It is not clear at the moment how much the assumption \eqref{1.8g} can be weakened to still
produce meaningful extensions of a hypercontractive fashion. 

\smallskip

As suggested by the Agler hereditary functional calculus (in Ambrozie-Engli\v s-M\"uller 
transcription \cite{AEM}) we introduce the operator $B_\bA: \, \cL(\cX)\to \cL(\cX)$ associated 
with the operator tuple $\bA = (A_1, \dots, A_d) \in \cL(\cX)^d$ via
\begin{equation}   \label{4.4}
B_\bA \colon X \mapsto B_\bA[X]: = \sum_{j=1}^d A_j^* X A_j.
\end{equation}
One easily sees from the definition that $B_\bA$ is a {\em positive map}, i.e.,
\begin{equation}  \label{BsubApos}
X \succeq 0 \Rightarrow B_{\bA}[X] \succeq 0.
\end{equation}
In fact, $B_\bA$  satisfies the stronger property of being a {\em completely positive map}
(see e.g.\ \cite{Pau}), but we shall not have need of this fact. Iterating \eqref{4.4} gives
\begin{equation} \label{4.5}
B_{\bA}^k \colon X \mapsto \sum_{\alpha \in\free: \,
|\alpha|=k}\bA^{*\alpha^\top} X \bA^{\alpha}\quad \mbox{for}\quad k\ge 1.
\end{equation}
The following general principle will be useful in the sequel.
\begin{lemma}  
\label{P:welldefined}
Given a function $f(\lambda)=\sum_{j\ge 0} f_j \lambda^j
\in W^+$, the operatorial map
\begin{equation}
f(B_\bA):=\sum_{j=0}^\infty f_j B_\bA^j: \;  X \mapsto 
\sum_{j=0}^\infty f_j \bigg(\sum_{|\alpha|=j}\bA^{*\alpha^\top}X\bA^{\alpha}\bigg)=
\sum_{\alpha\in\free} f_{|\alpha|}\, \bA^{*\alpha^\top}X\bA^{\alpha}
\label{1.10g}
\end{equation}
is well defined  for any $X\in\cL(\cX)$ subject to inequalities
\begin{equation}   \label{1.9g}
X\succeq \sum_{j=1}^d A_j^*XA_j\succeq 0.
\end{equation}
\end{lemma}
\begin{proof}
Indeed, iterating \eqref{1.9g}  yields that 
$$
X\succeq \sum_{\alpha\in\free:|\alpha|=j}\bA^{*\alpha^\top}X\bA^{\alpha}\succeq 0
\quad\mbox{for all}\quad j \ge 0,
$$
which together with \eqref{1.10g} implies
$\|f(B_\bA)[X]\|\le {\displaystyle\sum_{j=0}^\infty |f_j|\|X\|}=\|f\|_{_{W^+}}\cdot 
\|X\|$.
\end{proof}
As a consequence of Lemma \ref{P:welldefined},  the operatorial map 
\begin{equation}   
\label{1.11g}
\Gamma_{\bo,\bA}:=\frac{1}{R_{\bo}}(B_\bA)[X] =   \sum_{\alpha\in\free}
c_{|\alpha|}\, \bA^{*\alpha^\top }X\bA^\alpha
\end{equation}
is well defined  for any $X\in\cL(\cX)$ subject to inequalities \eqref{1.9g}. 
Note that the assumption \eqref{18.2} implies that $1\le \omega_j^{-1}\le \omega_{j+1}^{-1}$
forcing that $R_\bo(\lambda)$ (see \eqref{1.6g})
cannot be in the Wiener class $W^+$; nevertheless
 the assumption \eqref{1.8g} guarantees that for each $k\ge 0$, the function 
$$
\frac{R_{\bo,k}}{R_{\bo}}(\lambda)=\bigg(\sum_{j=0}^\infty
\frac{\lambda^j}{\omega_{k+j}}\bigg)\cdot\bigg(\sum_{j=0}^\infty c_j\lambda^j\bigg)
$$
does belong to $W^+$; furthermore one can make use of the recursion 
\eqref{1.8g} to see that
$$
\bigg(\sum_{j=0}^\infty
\frac{\lambda^j}{\omega_{k+j}}\bigg)\cdot\bigg(\sum_{j=0}^\infty c_j\lambda^j\bigg)
= \sum_{j=0}^\infty c_j^{\bo, k} \lambda^j, \;  \text{ where } \; c_j^{\bo,k} = - \sum_{\ell=1}^k \frac{c_{j+\ell}}{\omega_{k-\ell}}
$$
where the Taylor coefficients $c_j^{\bo, k}$ satisfy
$$
\sum_{j=0}^\infty | c_j^{\bo,k} | \le \bigg( \sum_{\ell = 1}^k \omega_{k-\ell} \bigg) \cdot \bigg(\sum_{j=1}^\infty |c_j|  \bigg) < \infty
$$
(see \cite[Proposition 3.3]{BBSzeged} for details). Again by Lemma \ref{P:welldefined}, 
the map
\begin{align}   
\Gamma^{(k)}_{\bo,\bA}:=   \frac{R_{\bo,k}}{R_{\bo}}(B_\bA)  \colon &  X \mapsto 
\sum_{j=0}^\infty c_j^{\bo, k} B_\bA^j [X]\notag\\
&=-\sum_{\alpha\in\free}\bigg(\sum_{\ell=1}^{k}\frac{c_{|\alpha|+\ell}}{\omega_{k-\ell}}\bigg)
\, \bA^{*\alpha^\top}X\bA^\alpha\quad (k\ge 1)
\label{3.5ag}
\end{align}
is well defined for any operator $X\in\cL(\cX)$ subject to inequalities \eqref{1.9g}.

\smallskip

The next result establishes connections between $\bo$-output stability, $\bo$-obser\-vability gramians, 
and solutions of associated Stein equations and inequalities.  Here we refer to the 
relations in \eqref{3.16g}  and \eqref{3.17g} as
the {\em Stein inequalities} and the {\em Stein equation} respectively. 

\begin{theorem}
\label{T:7.2}
Let us assume that the weight sequence $\bo$ meets conditions \eqref{18.2}, \eqref{1.8g} and
let $(C,\bA)\in\cL(\cX,\cY)\times\cL(\cX)^d$ be an output pair. Then:
\begin{enumerate}
\item The pair $(C,\bA)$ is $\bo$-output-stable if and only if there is
an $H\in \cL(\cX)$ subject to inequalities 
\begin{equation}
H\succeq \sum_{j=1}^d A_j^*HA_j\succeq 0,\qquad
\Gamma^{(k)}_{\bo,\bA}[H]\succeq 0\quad\mbox{for all}\quad k\ge 1
\label{3.15g}
\end{equation}
and the Stein inequality
\begin{equation}
\Gamma_{\bo,\bA}[H]\succeq C^*C.
\label{3.16g}
\end{equation}
\item If $(C,\bA)$ is a $\bo$-output-stable pair, then the gramian
$H={\mathcal G}_{\bo, C, \bA}$ is the  minimal positive semidefinite solution of the system \eqref{3.15g},
\eqref{3.16g} and it also satisfies the Stein equation
\begin{equation}
\Gamma_{\bo,\bA}[H]= C^*C.
\label{3.17g}
\end{equation}
\end{enumerate}
           \end{theorem}
\begin{proof}
Suppose first that $(C,\bA)$ is $\bo$-output-stable.
Then the series in \eqref{18.2aa} converges strongly to the operator
$H={\mathcal G}_{\bo,C,\bA}\succeq 0$, which, as will be now verified, satisfies relations 
\eqref{3.15g} and \eqref{3.17g}. Indeed, we have from \eqref{4.4}
$$
\cG_{\bo,C,\bA}- \sum_{j=1}^d A_j^*\cG_{\bo,C,\bA}A_j=C^*C+\sum_{\alpha\neq \emptyset}
\big(\omega_{|\alpha|}^{-1}-\omega_{|\alpha|-1}^{-1}\big) \bA^{*\alpha^\top}C^*C\bA^\alpha,
$$
which implies, since $\bo$ is non-increasing (by \eqref{18.2}) and since ${\mathcal G}_{\bo,C,\bA}\succeq 0$, that
\begin{equation}
\cG_{\bo,C,\bA}\succeq \sum_{j=1}^d A_j^*\cG_{\bo,C,\bA}A_j \succeq 0.
\label{3.18g}
\end{equation}
By Lemma \ref{P:welldefined}, the operators $\Gamma_{\bo,\bA}[\cG_{\bo,C,\bA}]$
and $\Gamma^{(k)}_{\bo,\bA}[\cG_{\bo,C,\bA}]$ are well-defined and the series
\eqref{1.11g}, \eqref{3.5ag} (with $X=\cG_{\bo,C,\bA}$) converge absolutely
when applied to any $x\in\cX$. Hence, we are now in position to apply the following
\textbf{general principle} due to Mertens (see \cite[Theorem 3.50]{Rudin} for the case of numerically-valued sequences):

\smallskip
\noindent
{\em If the series ${\displaystyle\sum_{n\ge 0}{\bf a}_n}$ converges absolutely and if $\; {\displaystyle\sum_{n\ge 0}{\bf a}_n}={\bf a}\; $ 
and $\; {\displaystyle\sum_{n\ge 0}{\bf b}_n}={\bf b}$, then }
\begin{equation}
\sum_{n=0}^{\infty}\bigg(\sum_{k=0}^n{\bf a}_k {\bf b}_{n-k}\bigg)=\bigg(\sum_{n= 0}^{\infty}{\bf a}_n\bigg)\cdot 
\bigg(\sum_{n=0}^{\infty}{\bf b}_n\bigg) ={\bf a}\cdot{\bf b}.
\label{GenPrin}
\end{equation}
Applying this \textbf{General Principle} with the choice
$$
{\mathbf a}_n = c_n B_\bA^n \in \cL(\cL(\cX)), \quad {\mathbf b}_n = \omega_n B_\bA^n[C^*C] \in \cL(\cX)
$$
then leads to
\begin{align}
\Gamma_{\bo,\bA}[\cG_{\bo,C,\bA}]&=  
\bigg( \sum_{n=0}^\infty c_n B_\bA^n \bigg) \bigg[ \sum_{n=0}^\infty \omega_n^{-1} B_\bA^n[C^* C]  \bigg] 
 = \sum_{n=0}^\infty \bigg( \sum_{k=0}^n  c_k \omega_{n-k}^{-1}  \bigg) B_\bA^n[C^* C]  \notag \\
&=\sum_{\alpha\in\free}\bigg(\sum_{j=0}^{|\alpha|}
c_{j}\omega_{|\alpha|-j}^{-1}\bigg)\bA^{*\alpha^\top}C^*C\bA^\alpha=C^*C
\label{3.19g}
\end{align}
where the last step follows from the recursion in \eqref{1.8g}. 

\smallskip

Similarly,  identity \eqref{3.18u} together with \eqref{18.2aa} and \eqref{3.5ag} leads us to
\begin{align}
\Gamma^{(k)}_{\bo,\bA}[\cG_{\bo,C,\bA}]   & =  \bigg( \sum_{n=0}^\infty c_n^{\bo,k} B_\bA^n \bigg)
\bigg[\sum_{n=0}^\infty \omega_n^{-1} B_\bA^n[C^*C] \bigg]  \notag \\
& = \sum_{n=0}^\infty \bigg( \sum_{j=0}^n   c_j^{\bo,k} \omega_{n-j}^{-1} \bigg) B_\bA^n[C^*C] \notag \\
& = \sum_{n=0}^\infty \bigg( \sum_{j=0}^n \sum_{\ell=1}^k \frac{c_{j+\ell}}{\omega_{k-\ell} \omega_{n-j}} \bigg) B_\bA^n [C^* C]
\notag \\
&=-\sum_{\alpha\in\free}\bigg(\sum_{j=0}^{|\alpha|}\sum_{\ell=1}^k
\frac{c_{\ell+j}}{\omega_{k-\ell}\omega_{|\alpha|-j}}\bigg)\bA^{*\alpha^\top}C^*C\bA^\alpha\notag\\
&=\sum_{\alpha\in\free}\omega_{|\alpha|+k}^{-1}\cdot 
\bA^{*\alpha^\top}C^*C\bA^\alpha \succeq 
0.\label{3.19gg}
\end{align}
We now see from \eqref{3.18g}, \eqref{3.19g} and the last inequality that 
$H={\mathcal G}_{\bo,C,\bA}$ satisfies relations \eqref{3.15g}, \eqref{3.17g} and hence, also \eqref{3.16g}.

\smallskip

Conversely, suppose that \eqref{3.15g}, \eqref{3.16g} hold for some $H\in\cL(\cX)$.
From definitions \eqref{1.6g}, it follows that for every $j\ge 0$,
\begin{equation}
R_{\bo,j}(\lambda)=\lambda R_{\bo,j+1}(\lambda)+\omega_j^{-1}.
\label{1.6g'}
\end{equation}
Dividing both sides by $R_\bo(\lambda)$ and applying the $B_\bA$ functional calculus 
to the resulting identity and to the chosen operator $H$ (this can be done thanks to Lemma \ref{P:welldefined}) 
gives the operator equality
$$
\Gamma^{(j)}_{\bo,\bA}[H]=B_\bA\Gamma^{(j+1)}_{\bo,\bA}[H]+\omega_j^{-1}
\Gamma_{\bo,\bA}[H]
=\sum_{k=1}^d A_k^*\Gamma^{(j+1)}_{\bo,\bA}[H]A_k+\omega_j^{-1}\Gamma_{\bo,\bA}[H]
$$
where $\Gamma^{(j)}_{\bo,\bA}[H]$ is simply $H$ for the case $j=0$.
Application of $(B_{\bA})^{j}$ to both sides of this last identity 
then gives us
$$
B_{\bA}^{j} \Gamma_{\bo, \bA}^{(j)}[H] - B_{\bA}^{j+1} \Gamma_{\bo, 
\bA}^{(j+1)}[H] = \omega_{j}^{-1} B_{\bA}^{j} \Gamma_{\bo, \bA}[H]
$$
or, more explicitly,
\begin{align}
&\sum_{\alpha\in\free:|\alpha|=j}\bA^{*\alpha^\top}\Gamma^{(j)}_{\bo,\bA}[H]\bA^\alpha-
\sum_{\alpha\in\free:|\alpha|=j+1}\bA^{*\alpha^\top}\Gamma^{(j+1)}_{\bo,\bA}[H]\bA^\alpha\notag\\
&=\omega_j^{-1}\sum_{\alpha\in\free:|\alpha|=j}\bA^{*\alpha^\top}\Gamma_{\bo,\bA}[H] \bA^\alpha\succeq 0.
\label{1.20g}
\end{align}
Define a sequence of operators  $\{\Omega_{k}\}$ ($k=0,1,2, \dots$) by
$$
\Omega_k=\sum_{\alpha\in\free:|\alpha|=k}\bA^{*\alpha^\top}\Gamma^{(k)}_{\bo,\bA}[H]\bA^\alpha.
$$
From \eqref{1.20g} and the fact that $\Gamma_{\bo, \bA}[H] \succeq 0$ by 
hypothesis \eqref{3.16g}, we see that $\Omega_{k}$ is a decreasing 
operator sequence.  As a consequence of hypothesis \eqref{3.15g}, we 
also have $\Omega_{k} \succeq 0$ for all $k$.  Therefore, there exists the strong limit
\begin{equation}
\Delta_{\bA,H}=\lim_{k\to\infty}\Omega_k.
\label{1.21g}
\end{equation}
Summing up equalities in \eqref{1.20g} for $j=0,\ldots,k-1$ and taking into account that 
$\Gamma^{(0)}_{\bo,\bA}[H]=H$, we get 
\begin{align}
\sum_{\alpha\in\free:|\alpha|< k}\omega_{|\alpha|}^{-1}\bA^{*\alpha^\top}\Gamma_{\bo,\bA}[H] 
\bA^\alpha&=H-\sum_{\alpha\in\free:|\alpha|=k}\bA^{*\alpha^\top}\Gamma^{(k)}_{\bo,\bA}[H]\bA^\alpha\notag\\
&=H-\Omega_k.
\label{1.21ga}
\end{align}
Combining \eqref{1.21g} with \eqref{3.16g} gives
\begin{align}
\sum_{\alpha\in\free:|\alpha|\le k}\omega_{|\alpha|}^{-1}A^{*\alpha^\top}C^*CA^\alpha
\preceq \sum_{\alpha\in\free:|\alpha| < k}\omega_{|\alpha|}^{-1}A^{*\alpha^\top}\Gamma_{\bo,\bA}[H]
A^\alpha =H-\Omega_k
\label{3.19ggg}
\end{align}
for all $k\ge 1$. By letting $k\to \infty$ in \eqref{3.19ggg} we conclude that the sum on the
left-hand side  converges to a bounded positive-semidefinite operator, which is
$\cG_{\bo,C,\bA}$ by \eqref{18.2aa}. Thus, passing to the limit in \eqref{3.19ggg} as $k\to \infty$ gives
\begin{equation}  \label{lim-ineq}
\cG_{\bo,C,\bA} \preceq H-\Delta_{\bA,H} \preceq H,
\end{equation}
where $\Delta_{\bA,H}\succeq 0$ is the limit defined in
\eqref{1.21g}. Therefore, the pair $(C,\bA)$ is $\bo$-output-stable
and  $\cG_{\bo,C,\bA}$ is indeed the minimal positive-semidefinite solution to the
system \eqref{3.15g}, \eqref{3.16g}.
\end{proof}

\begin{remark}  \label{R:extension}
Let us consider Theorem \ref{T:7.2} for the classical Hardy-space case where $d=1$ and $\omega_j=1$ for all 
$j\ge 0$.  Then
$\bo$-output stability amounts to classical output stability for the output pair $(C,A)$.
Furthermore we have
$$
\Gamma_{\bo,A}[H]=H-A^*HA\quad\mbox{and}\quad \Gamma^{(k)}_{\bo,A}[H]=I_\cX\quad\mbox{for all $k\ge 1$}.
$$ 
Therefore,
relations \eqref{3.16g} and \eqref{3.17g} amount to \eqref{bbfStein-ineq} and \eqref{bbf1.5},
respectively, while the only non-redundant inequality in the system of equations in the second part of 
\eqref{3.15g} is $H\succeq 0$.
Therefore, Theorem \ref{T:7.2} can be viewed as a  canonical extension of Theorem \ref{T:bbf1} to the 
$\bo$-setting, apart from the uniqueness statements (2) and (3) in Theorem \ref{T:bbf1}.
Observe that Theorem \ref{T:7.2} does not address these uniqueness issues as the latter require some 
suitable notion of {\em strong stability}.  See Remark \ref{R:extension2} below 
for further discussion on this point.    
\end{remark}

By Lemma \ref{P:welldefined}, the operatorial maps $\Gamma^{(k)}_{\bo,\bA}$ and 
$\Gamma^{(k)}_{\bo,\bA}$ are well defined on operators $X\in\cL(\cX)$ subject to inequalities 
\eqref{1.9g}. In particular (by choosing $X=I_{\cX}$), the operators $\Gamma_{\bo,\bA}[I_{\cX}]$ and 
$\Gamma^{(k)}_{\bo,\bA}[I_{\cX}]$ are well defined if the tuple $\bA$ is contractive in the sense that
\begin{equation}  \label{4.22}
A_{1}^{*}A_{1} + \cdots + A_{d}^{*}A_{d} \preceq I_{\cX},
\end{equation}
that is, $\bA^*=(A_1^*,\ldots,A_d^*)$ is a row-contraction.

\begin{definition} \label{D:7.3}
The tuple ${\mathbf A}=(A_1,\ldots,A_d)\in\cL(\cX)^d$ is called {\em $\bo$-contractive} if
it is contractive (i.e., \eqref{4.22} holds) and 
$$
\Gamma_{\bo,\bA}[I_\cX]=\sum_{\alpha\in\free}c_{|\alpha|}\bA^{*\alpha^\top}\bA^\alpha\succeq 
0.
$$
The tuple $\bA$ is called {\em $\bo$-hypercontractive} if it is $\bo$-contractive and 
$$
\Gamma^{(k)}_{\bo,\bA}[I_{\cX}]:=
-\sum_{\alpha\in\free}\bigg(\sum_{\ell=1}^{k}\frac{c_{|\alpha|+\ell}}{\omega_{k-\ell}}\bigg)
\, \bA^{*\alpha^\top}\bA^\alpha \succeq 0 \quad\mbox{for all}\quad k\ge 1.
$$
\end{definition}

\begin{definition}
\label{D:7.4}
An $\bo$-hypercontractive operator tuple $\bA\in\cL(\cX)^d$ will be called {\em $\bo$-strongly stable}
if the limit \eqref{1.21g} with $H=I_{\cX}$ equals zero, i.e.,
$$
\Delta_{\bA,I_{\cX}}=\lim_{k\to \infty}\sum_{\alpha\in\free:|\alpha|=k}
\bA^{*\alpha^\top}\Gamma^{(k)}_{\bo,\bA}[I_{\cX}]\bA^\alpha=0,
$$
or, equivalently (see \eqref{3.5ag}),
$$\lim_{k\to \infty}\sum_{\alpha\in\free:|\alpha|\ge k}
\bigg(\sum_{\ell=1}^{k}\frac{c_{|\alpha| -k  +\ell}}{\omega_{k-\ell}}\bigg)
\|\bA^{\alpha}x\|^2=0\quad\mbox{for all}\quad x\in\cX.
$$\end{definition}
As we will see in Remark \ref{R:stab} below, $\bo$-strong stability of an $\bo$-hypercontraction implies 
its strong stability in the usual sense \eqref{bAstable}.

\smallskip

Taking advantage of relations \eqref{3.15g}--\eqref{3.17g} characterizing $\bo$-output stability,
we introduce the notions of weighted contractive and isometric pairs.

\begin{definition} \label{D:7.5}
A pair $(C,\bA)$ (with $C\in\cL(\cX,\cY)$ and $\bA\in\cX^d$) will be called
$\bo$-{\em contractive output pair} if inequalities \eqref{3.15g},
\eqref{3.16g} hold with $H = I_{{\mathcal X}}$, i.e., if $\bA$ is $\bo$-hypercontractive and 
$$
\Gamma_{\bo,\bA}[I_{_\cX}]:=\sum_{\alpha\in\free}c_{|\alpha|}\bA^{*\alpha^\top}\bA^\alpha \succeq C^*C.
$$
The pair $(C,\bA)$ will be called $\bo$-{\em isometric} if  $\bA$ is
$\bo$-hypercontractive and
\begin{equation}
\Gamma_{\bo,\bA}[I_{\cX}]=\sum_{\alpha\in\free}c_{|\alpha|}\bA^{*\alpha^\top}\bA^\alpha=C^*C.
\label{1.23g}
\end{equation}
It is clear from \eqref{18.2aaa} that the gramian $\cG_{\bo,C,\bA}$ is positive-definite 
if and only if the pair $(C,\bA)$ is observable in the sense of \eqref{ncobs}.
We will say that  the pair $(C,\bA)$ is {\em exactly $\bo$-observable} if the
$\bo$-gramian $\cG_{\bo,C,\bA}$ is bounded and is strictly positive-definite.
\end{definition}

\begin{lemma} (1) Suppose $(C,\bA)\in\cL(\cX,\cY)\times\cL(\cX)^d$ 
is a $\bo$-contractive pair. Then $(C,\bA)$ is
$\bo$-output stable and $\cG_{\bo,C,\bA}\preceq I$ and hence $\cO_{\bo,C,\bA}: \, \cX\to
H^2_{\bo,\cY}(\free)$ is a contraction.

\smallskip

(2) Suppose $(C,\bA)$ is a $\bo$-isometric pair.
Then $\cG_{\bo,C,\bA}=I_{\cX}$ if and only if $\bA$ is $\bo$-strongly stable.
In particular, if $(C, \bA)$ is $\bo$-isometric and $\bA$ is $\bo$-strongly stable, then
$\cO_{\bo, C, \bA} \colon \cX \to H^2_{\bo, \cY}(\free)$ is isometric, and hence also the 
pair $(C, \bA)$ is exactly $\bo$-observable.
\label{L:7.6}
\end{lemma}

\begin{proof}
Suppose that $(C, \bA)$ is an $\bo$-contractive pair.  Then by Definition \ref{D:7.5} $H = I_\cX$ 
satisfies \eqref{3.15g} and \eqref{3.16g}.  Hence by part (1) of Theorem \ref{T:7.2} we conclude that 
$(C, \bA)$ is $\bo$-output stable, and  the inequality \eqref{lim-ineq} specified for the case $H = I_\cX$ 
gives us also that $\cG_{\bo, C, \bA}$ is contractive. This completes the proof of (1).

\smallskip

We now suppose that $(C, \bA)$ is $\bo$-isometric and $\bA$ is $\bo$-strongly stable.
Specifying \eqref{1.21g} to the present case where $H = I_\cX$, we conclude that the limit
$$
\Delta_{\bA,I}=\lim_{k\to \infty}\sum_{\alpha\in\free:|\alpha|=k}
\bA^{*\alpha^\top}\Gamma^{(k)}_{\bo,\bA}[I_{\cX}]\bA^\alpha
$$
exists. Setting $H=I_{\cX}$ in \eqref{1.21ga} and letting $k\to \infty$ we see that
the series below converges in the strong operator topology and satisfies
\begin{equation}
\sum_{\alpha\in\free}^\infty \omega_{|\alpha|}^{-1}A^{*\alpha^\top}\Gamma_{\bo,\bA}[I_{\cX}]\bA^\alpha=
I_{\cX}-\Delta_{\bA,I}.
\label{ref5g}
\end{equation}
Combining \eqref{ref5g} and \eqref{1.23g} now gives
$$
\sum_{\alpha\in\free}\omega_{|\alpha|}^{-1}\bA^{*\alpha^\top}C^*CA^\alpha= I-\Delta_{A,I}.
$$
Since $\Delta_{A,I}=0$ and since the series on the left converges to $\cG_{\bo,C,\bA}$, we conclude
$\cG_{\bo,C,\bA}=I_\cX$  from which it follows by definitions that $\cO_{\bo, C, \bA}$ is isometric and $(C, \bA)$ is exactly $\bo$-observable.

\smallskip

Conversely, suppose that $(C, \bA)$ is an $\bo$-isometric pair such that $\cG_{\bo, C, \bA} = I_\cX$.
Then we still arrive at the identity \eqref{ref5g}, where the series on the left is just the definition of
$\cG_{\bo, C, \bA}$.  The assumption that $\cG_{\bo, C, \bA} = I_\cX$ then forces  $\Delta_{\bA, I} = 0$,
i.e., that $\bA$ is $\bo$-strongly stable.  This  completes the proof.
\end{proof}

\begin{definition}
Let us say that  {\em the pair $(C, \bA)\in\cL(\cX,\cY)\times(\cL(\cX))^d$ is similar to the pair
$(\widetilde C, \widetilde \bA)\in\cL(\widetilde{\cX},\cY)\times(\cL(\widetilde{\cX}))^d$ } if there is an invertible operator 
$T:\cX\to \widetilde{\cX}$ so that
\begin{equation}
         \widetilde C = C T^{-1}, \qquad \widetilde A_{j} = T A_{j}T^{-1}
         \quad\text{for} \; \; j = 1, \dots, d.
\label{sep1}
\end{equation}
{\rm Then we have the following characterization of pairs
$(C, {\mathbf A})$ which are similar to an $\bo$-contractive or to an $\bo$-isometric pair.}
\label{D:sim}
\end{definition}

\begin{proposition}  \label{P:similar}
The pair $(C, \bA)$ is similar to an $\bo$-contractive (to an $\bo$-isometric) pair
$(\widetilde C, \widetilde \bA)$ if and only if there exists a
bounded, strictly positive-definite solution $H$ to the system \eqref{3.15g}, \eqref{3.16g}
(respectively, the system \eqref{3.15g}, \eqref{3.17g}).
         \end{proposition}

\begin{proof}
Suppose that $H\succ 0$ is a solution to the system \eqref{3.15g}, \eqref{3.16g}.
Factor $H$ as $H = T^{*}T$ with $T$ invertible and define $\widetilde{C}$ and $\widetilde{A}_j$
by formulas \eqref{sep1}. Then the pair $(\widetilde C, \widetilde \bA)$ is similar to $(C,\bA)$
and it is readily seen from \eqref{1.11g}, \eqref{3.5ag} that
$$
T^*\Gamma_{\bo,\widetilde{\bA}}[I_{\cX}]T=\Gamma_{\bo,\bA}[I_{\cX}]\quad\mbox{and}\quad
T^*\Gamma^{(k)}_{\bo,\widetilde{\bA}}[I_{\cX}]T=\Gamma^{(k)}_{\bo,\bA}[I_{\cX}]
\quad\mbox{for}\quad k\ge 1.
$$
Thus, inequalities \eqref{3.15g}, \eqref{3.16g} can be equivalently written in terms of
$\widetilde{C}$ and $\widetilde{A}$ as
$$
I\succeq \widetilde{A}_1^*\widetilde{A}_1+\ldots+\widetilde{A}_d^*\widetilde{A}_d,\quad
\Gamma^{(k)}_{\bo,\widetilde{\bA}}[I_{\cX}] \succeq 0 \; \; (k\ge 1),\quad
\Gamma_{\bo,\widetilde{\bA}}[I_{\cX}]\succeq \widetilde C^{*}\widetilde C,
$$
and hence, the pair $(\widetilde C, \widetilde \bA)$ is $\bo$-contractive.

\smallskip

Conversely, if $(\widetilde C, \widetilde \bA)$ given by \eqref{sep1} is
$\bo$-contractive, then $H = T^{*}T$ is bounded and strictly positive-definite and
satisfies  inequalities \eqref{3.15g}, \eqref{3.16g}. The statement concerning 
$\bo$-isometric pairs follows in a similar way.
\end{proof}

\subsection{$\bmu_n$-hypercontractions}  \label{S:bo-bmun}
We now examine how the above general results specialize to 
the Bergman weights $\bmu_n$.  Within this setting, we will often write $n$ instead of $\bo=\bmu_n$.
Since $R_{\bmu_n}(\lambda)=R_n(\lambda)=(1-\lambda)^{-n}$, the formula 
\eqref{1.11g} amounts to 
\begin{equation}
\Gamma_{n, \bA}:=\Gamma_{\bmu_n,\bA}\colon X \mapsto  (I - B_\bA)^n[X]=
\bigg( \sum_{\ell=0}^n (-1)^\ell \left( \begin{smallmatrix} n \\ \ell \end{smallmatrix} \right)  B_\bA^\ell \bigg) [X] 
\label{Gamma-k}
\end{equation}
or more explicitly, on account of \eqref{4.5}, to 
\begin{equation}
\Gamma_{n, \bA}\colon X \mapsto  \sum_{ \alpha \in \free \colon |\alpha| \le n} (-1)^{|\alpha|} 
\left( \begin{smallmatrix} n \\ |\alpha|
\end{smallmatrix} \right) \bA^{* \alpha^\top} X \bA^\alpha,
\label{Gamma-ex}
\end{equation}
and makes sense for all $n\ge 0$. Upon applying the identity
$$
(I - B_{\bA})^{k}=(I - B_{\bA})^{k-1}-B_{\bA}(I - B_{\bA})^{k-1}
$$
to an operator $H\in\cL(\cX)$ and making use of definition \eqref{Gamma-k} we get 
\begin{equation} 
\Gamma_{k,\bA}[H] = \Gamma_{k-1,\bA}[H] - \sum_{j=1}^d A_j^{*} \Gamma_{k-1,\bA}[H]A_j
\label{4.8}
\end{equation}
for all $k\ge 1$. Dividing both sides of \eqref{1.10pre} by $R_n$ and applying the $B_{\bA}$ calculus to the resulting
identity we specialize \eqref{3.5ag} to the case $\bo = \bmu_n$:
$$
\Gamma_{n,\bA}^{(k)}= \frac{R_{n,k}}{R_{n}}(B_{\bA})= \sum_{\ell = 1}^{n} \bcs{\ell + k - 2 \\ \ell - 1}
    \frac{R_{n-\ell +1}}{R_{n}}(B_{\bA}) = \sum_{\ell = 1}^{n} \bcs{\ell + k - 2 \\ \ell - 1}
    (I-B_\bA)^{\ell-1}.
$$
Making use of formula \eqref{Gamma-k} for all $\ell=0,\ldots,n-1$, we conclude:
\begin{equation}   \label{Gamma-bmun}
\Gamma^{(k)}_{n, \bA} = \sum_{\ell = 0}^{n-1} \left( \begin{smallmatrix} \ell + k - 1 \\ 
\ell \end{smallmatrix} \right)\Gamma_{\ell, \bA}.
\end{equation}
Before turning to $\bmu_n$-contractions and $\bmu_n$-hypercontractions (obtained via specializing Definition \ref{D:7.3} 
to the case $\bo=\bmu_n$) let us recall a more elementary notion of {\em contractivity} and {\em hypercontractivity} 
indexed by a discrete parameter $n = 1,2, \dots$ defined as follows.
\begin{definition}  \label{D:4.1}
The tuple ${\mathbf A}=(A_1,\ldots,A_d)\in\cL(\cX)^d$ is called {\em $n$-contractive} if
$\Gamma_{n,\bA}[I_{\cX}]\succeq 0$, and it is called {\em $n$-hypercontractive} if 
$\Gamma_{k,\bA}[I_{\cX}] \succeq 0$ 
for all $1\le k\le n$.
\end{definition}

It was shown in \cite{aglerhyper} (see also \cite{muller} as well as
\cite{MV} for a commutative multivariable version) that inequalities
$\Gamma_{1,\bA}[I_{\cX}]\succeq 0$ and $\Gamma_{n,\bA}[I_{\cX}]\succeq 0$
imply that $\bA$ is an $n$-hypercontraction. This result extends to our free noncommutative
setting, even with $I_{\cX}$ replaced by an arbitrary $H \succeq  0$, as follows.

\begin{lemma}  \label{L:squeeze}
Let  us assume that the operators $H$ and $A_1,\ldots,A_d$ in $\cL(\cX)$ are such that
\begin{equation}
H \succeq \sum_{j=1}^d A_j^*HA_j\succeq 
0\quad\mbox{and}\quad\Gamma_{n,\bA}[H]\succeq 0
\label{4.12}
\end{equation}
for some integer $n\ge 3$. Then
\begin{equation}   
\Gamma_{k,\bA}[H] \succeq 0\quad\mbox{for all}\quad k=1,\ldots,n-1.
\label{4.13} 
\end{equation}
\end{lemma}

\begin{proof}
Observe that the leftmost inequalities in \eqref{4.12} mean that $H$ and 
$\Gamma_{1,\bA}[H]$ are both positive-semidefinite. In light of the positivity property 
\eqref{BsubApos}, applying the map $B_{\bA}$ to both sides of the first inequality in 
\eqref{4.12} and subsequent iterating lead us to
\begin{equation}
H \succeq \sum_{\alpha \in\free: \, 
|\alpha|=j}\bA^{*\alpha^\top}H\bA^{\alpha}\quad\mbox{for all}\quad j\ge 0.
\label{4.14}
\end{equation}  
Let us introduce the Hermitian operators
\begin{equation}
S_{m,k}:=\sum_{\alpha\in\free: \, |\alpha|=k}\bA^{*\alpha^\top} 
\Gamma_{m,\bA}[H] \bA^{\alpha}
=B_\bA^k(I-B_\bA)^{m}[H]
\label{4.15}
\end{equation}
for $k\in{\mathbb Z}_+$ and $m=0,1,\ldots,n$
(observe that the second equality in \eqref{4.15} follows from
\eqref{4.5} and \eqref{Gamma-k}). Let us show that
\begin{equation}
-2^m\cdot H \preceq S_{m,k} \preceq 2^m\cdot H\qquad\mbox{for all $\; k\in{\mathbb Z}_+\; $ 
and $\; m=0,1,\ldots,n$.}
\label{4.16}
\end{equation}
Indeed, since $H$ is positive semidefinite, combining \eqref{Gamma-ex} and \eqref{4.14} gives
\begin{align*}
S_{m,k}&\preceq \sum_{\alpha \in\free: \, |\alpha|=k} \bA^{*\alpha^\top}\bigg(
\sum_{\beta\in\free: \, |\beta|\le m}  (-1)^{|\beta|}   \bcs{m \\ |\beta|}\cdot 
\bA^{*\beta^\top}H\bA^{\beta}\bigg)\bA^{\alpha}\\
&=\sum_{\alpha \in\free: \, k\le |\alpha|\le m+k} 
(-1)^{|\alpha|-k}     \bcs{m \\ |\alpha|-k}\cdot \bA^{*\alpha^\top}H\bA^{\alpha}   \\
&=\sum_{j=k}^{m+k}   (-1)^{j-k}   \bcs{m \\ j-k}\cdot \sum_{\alpha\in\free: \, 
|\alpha|=k} \bA^{*\alpha^\top}H\bA^{\alpha}
\preceq \sum_{j=k}^{m+k}  \bcs{m \\ j-k}\cdot H=2^m \cdot H,
\end{align*}
thus proving the right inequality in \eqref{4.16}. The left inequality follows
in much the same way. We next observe that due to \eqref{4.8} and the second inequality in 
\eqref{4.12},
$$
\Gamma_{n-1,\bA}[H] \succeq \sum_{j=1}^d A_j^{*}
\Gamma_{n-1,\bA}[H]A_j=B_\bA\left[\Gamma_{n-1,\bA}[H]\right].
$$
Therefore, on account of \eqref{BsubApos}, we have
\begin{equation}
S_{n-1,j}=B_\bA^j\left[\Gamma_{n-1,\bA}[H]\right] \succeq B_\bA^{j+1}\left[\Gamma_{n-1,\bA}[H]\right]=
S_{n-1,j+1}.  
\label{4.17}
\end{equation}  
Also, it follows from definitions \eqref{4.15} and \eqref{Gamma-k} that for any  $N\ge 1$,
\begin{align}
\sum_{j=0}^N S_{n-1,j}&=
\sum_{j=0}^N B_\bA^j(I-B_\bA)^{n-1}[H]\notag\\
&=(I-B_\bA^{N+1})(I-B_\bA)^{n-2}[H]\notag\\
&=\Gamma_{n-2,\bA}[H]-\sum_{\alpha \in\free: \, 
|\alpha|=N+1}\bA^{*\alpha^\top} \Gamma_{n-2,\bA}[H] \bA^{\alpha} \notag\\  
&=S_{n-2,0}-S_{n-2,N+1}.\notag
\end{align}
Taking the inner product of both parts in the latter equality against $x\in\cX$ and 
then making use of \eqref{4.16} gives
\begin{equation}
\big|\sum_{j=0}^N \left\langle S_{n-1,j}x, \, x\right\rangle\big|=
\left| \left\langle S_{n-2,0} x\, x\right\rangle-
 \left\langle S_{n-2,N+1} x, \, x\right\rangle\right|
\le 2^{n-1} \left\langle Hx, \, x\right\rangle.
\label{4.18}
\end{equation}
Due to \eqref{4.17}, on the left-hand side of \eqref{4.18} we have the partial sum
of a non-increasing sequence and, since the partial sums are uniformly bounded
(by $2^{n-1}\cdot\left\langle Hx, \, x\right\rangle$), it follows that all the terms in the 
sequence
are nonnegative. In particular,
$$
\left\langle S_{n-1,0}x, \, x\right\rangle=
\left\langle \Gamma_{n-1,\bA}[H]x, \, x\right\rangle\ge 0\quad\mbox{for every $\; 
x\in\cX$},
$$
and hence, $\Gamma_{n-1,\bA}[H]\succeq 0$. 
We then obtain recursively all the desired inequalities in \eqref{4.13}. 
\end{proof}
\begin{proposition}  \label{P:equiv}
For Bergman weights $\bo={\boldsymbol\mu}_n=\{\bin_{n,j}\}_{j\ge 0}$), the classes of 
${\boldsymbol\mu}_n$-hypercontractions, 
of ${\boldsymbol\mu}_n$-contractions, and of $n$-hypercontractions are identical.
\end{proposition}
\begin{proof} The statement follows from three implications below.

\smallskip
\noindent
$1$. {\em If a tuple $\bA\in\cL(\cX)^d$ is ${\bmu}_n$-contractive, then it is $n$-hypercontractive}.
Indeed, if $\bA$ is ${\bmu}_n$-contractive, then (see \eqref{4.22}) 
$$
\Gamma_{1, \bA}[I_\cX]=I_{\cX}-A_1^*A_1-\ldots-A_d^*A_d\succeq 0 
$$
and $\Gamma_{n,\bA}[I_{\cX}]:=\Gamma_{{\boldsymbol\mu}_n,\bA}[I_{\cX}]\succeq 0$.
Then it follows by Lemma \ref{L:squeeze} that 
$\Gamma_{\ell,\bA}[I_{\cX}]\succeq 0$ for $\ell=0,\ldots,n-1$,
so that $\bA$ is $n$-hypercontractive.

\smallskip
\noindent
$2$. {\em If a tuple $\bA\in\cL(\cX)^d$ is $n$-hypercontractive, then it is ${\bmu}_n$-hypercontractive}. Indeed,
if $\Gamma_{\ell,\bA}[I_{\cX}]\succeq 0$ for $\ell=0,\ldots,n-1$, then upon applying the formula \eqref{Gamma-bmun}
to the identity operator we see that
$$
\Gamma_{{\boldsymbol\mu}_{n},\bA}^{(k)}[I_{\cX}]= \sum_{\ell=0}^{n-1}
\bcs{\ell + k -1 \\ \ell} \Gamma_{\ell, \bA}[I_{\cX}]\succeq 0
$$
for all $k \ge 1$. The latter inequalities along with $\Gamma_{1, \bA}[I_\cX]\succeq 0$ and 
$\Gamma_{n, \bA}[I_\cX]\succeq 0$ imply that $\bA$ in ${\bmu}_n$-hypercontractive.
Since any $\bmu_n$-hypercontraction is clearly $\bmu_n$-contractive, the statement follows.  
\end{proof}

\subsection{$\bmu_n$-output-stability}  \label{S:bmun-stability}
 A crucial feature of the output-stability indexed by integers, as opposed to the
general $\bo$-output stability, is that $n$-output-stability implies the $k$-output-stability 
for any $0\le k<n$.
We shall say that an output pair $(C,\bA)$ is {\em $n$-output-stable} if it is the case that the series
$\sum_{\alpha\in\free}\mu_{n,|\alpha|}^{-1}\bA^{*\alpha^\top}C^*C\bA^{\alpha} $ is convergent
in the strong (or equivalently in the weak) operator topology of $\cL(\cX)$, in which case we define
the $n$-observability gramian $\cG_{n,C, \bA}$ as the sum of the series:
\begin{equation}
{\mathcal G}_{n,C, \bA}=:{\mathcal G}_{\bmu_n,C, \bA}=
\sum_{\alpha\in\free}\mu_{n,|\alpha|}^{-1}\bA^{*\alpha^\top}C^*C\bA^{\alpha} =: (I - B_\bA)^{-n}[C^* C].
\label{4.3}
\end{equation}
For the case where $n=0$ there are no convergence issues and we define
\begin{equation}
{\mathcal G}_{0,C, \bA}=C^*C. 
\label{ogo3}
\end{equation}

\begin{proposition}
\label{P:identities}
If the output pair $(C,\bA)$ is $n$-output stable, then it is also $k$-output stable for
$k=1,\dots, n-1$ and 
\begin{align}
 &  \cG_{k,C,\bA} - \sum_{j=1}^d A_j^* \cG_{k,C,\bA} A_j = \cG_{k-1, C, \bA},  \label{4.10}  \\
   &  \Gamma_{k,\bA}[\cG_{n,C,\bA}]  = \cG_{n-k,C,\bA} \quad\text{for}\quad k=0,\dots, n.
        \label{4.11}
  \end{align}
  Moreover we then have the chain of inequalities
  \begin{equation} \label{chain-ineq}
  C^* C \preceq \cG_{1, C, \bA} \preceq \cdots \preceq \cG_{n,C,\bA}.
  \end{equation}
 \end{proposition}
  
  \begin{proof}
   Note that $\Gamma_{k,\bA} = (I - B_\bA)^k = \sum_{j=0}^k \left( \begin{smallmatrix} k \\ j \end{smallmatrix} \right) B_\bA^j$ is a polynomial in 
  $B_\bA$ (and hence absolutely convergent) while by the stability assumption
    $$
  \cG_{k, C, \bA} = (I - B_\bA)^{-k}[C^* C] = \sum_{j=0}^\infty  \mu_{n,j+k}^{-1} B_\bA^j [C^* C]
  $$
  is convergent.    Therefore by the general principle \eqref{GenPrin} applied with ${\mathbf a}_j\in \cL(\cL(\cX))$ and 
${\mathbf b}_j\in \cL(\cX)$ given by
 $$ {\mathbf a}_j =  \begin{cases} (-1)^j  \left( \begin{smallmatrix} k \\ j \end{smallmatrix} \right) B_\bA^j &\text{ for } 0 \le j \le k, \\
\quad   0 & \text{ for }j \ge k,\end{cases}\quad {\mathbf b}_j = \mu_{n,j+k}^{-1} B_\bA^j[C^*C] \; \; \mbox{for}\; \; j\ge 0, 
   $$
it follows (the general principle justifies the second equality in the computation below) that
$$
\Gamma_{k,\bA}[\cG_{n,C,\bA}] = (I - B_\bA)^k[ (I - B_\bA)^{-n}[C^* C]] 
= (I - B_\bA)^{-n+k}[C^* C] = \cG_{n-k, C, \bA}
$$
which verifies \eqref{4.11}. By  making use of \eqref{4.3} and \eqref{4.4}, a similar application of the general principle \eqref{GenPrin}
can be used to justify the calculation
\begin{align*}
\cG_{k, C, \bA} - \sum_{j=1}^d A_j^* \cG_{k,C,\bA} A_j &= (I - B_\bA)[\cG_{k,C,\bA}]=
(I - B_\bA) \big[(I - B_\bA)^{-k} [C^* C]\big] \\&=  (I - B_\bA)^{-(k-1)}[C^*C] = \cG_{k-1, C, \bA},
  \end{align*}
for $k=1, 2, \dots, n$, which verifies \eqref{4.10}.  As a consequence of \eqref{4.10} and the identity \eqref{ogo3}, 
we then get the chain of inequalities \eqref{chain-ineq} which in turn, implies that $n$-output stability for $(C, \bA)$ implies $k$-output stability
for $(C, \bA)$ for $0\le k<n$.
\end{proof}

\begin{remark}   \label{R:extension2}
As we pointed out in Remark \ref{R:extension}, Theorem \ref{T:7.2} is a partial extension of Theorem \ref{T:bbf1}
to the general $\bo$-setting. Although Definition \ref{D:7.4} provides a useful notion of $\bo$-strong stability 
in some contexts (see its role in the proof of statement (2) of Lemma \ref{L:7.6}), we do not know
how to make use of this notion to prove the uniqueness of a solution to the system \eqref{3.15g}, \eqref{3.17g}.
It is still quite remarkable that Theorem \ref{T:bbf1} admits a full-extent generalization to the standard-weight multivariable 
Bergman-Fock setting---as given in the next result.
\end{remark} 

\begin{theorem}
\label{T:2-1.1}
Given $C\in\cL(\cX,\cY)$ and  $\bA=(A_1,\ldots,A_d)\in\cL(\cX)^d$,
the pair $(C,\bA)$ is $n$-output-stable (see \eqref{4.3}) if and only if there exists
an operator $H\in \cL(\cX)$ satisfying the system of inequalities
\begin{equation}
H\succeq \sum_{j=1}^d A_j^*HA_j\succeq
0\quad\mbox{and}\quad\Gamma_{n,\bA}[H]\succeq  C^*C.
\label{4.20}
\end{equation}
If this is the case (i.e., if $(C,\bA)$ is $n$-output-stable), then:
\begin{enumerate}
\item The gramian $H = {\mathcal G}_{n, C, \bA}$  is the minimal positive semidefinite solution
of the system \eqref{4.20} and actually solves \eqref{4.20} in the stronger form
\begin{equation}
H\succeq \sum_{j=1}^d A_j^*HA_j\succeq  0 \quad\mbox{and}\quad \Gamma_{n,\bA}[H]= C^*C
\label{4.21}
\end{equation}
\item If $\bA$ is strongly stable (see \eqref{bAstable}), then
$H = \cG_{n,C,\bA}$ is the unique solution of the system \eqref{4.21}.
\item If ${\mathbf A}$ is contractive (see \eqref{4.22}), the solution of the Stein equation in \eqref{4.21} is unique
if and only if ${\mathbf A}$ is strongly stable.
           \end{enumerate}
           \end{theorem}
           
\begin{proof}
If $H$ satisfies \eqref{4.20}, then $\Gamma_{\ell,\bA}[H]\succeq 0$ for $\ell=1,\ldots,n-1$, by Lemma \ref{L:squeeze}.
Applying identity \eqref{Gamma-bmun} to the operator $H$ then implies 
$$
\Gamma^{(k)}_{\bmu_n, \bA}[H] = \sum_{\ell = 0}^{n-1} \left( \begin{smallmatrix} \ell + k-1 \\
\ell \end{smallmatrix} \right)\Gamma_{\ell, \bA}[H]\succeq 0\quad\mbox{for all}\quad k\ge 1.
$$
Therefore, the inequalities $\Gamma^{(k)}_{\bmu_n, \bA}[H]\succeq 0$ in \eqref{3.15g} are redundant when $\bo=\bmu_n$.
Hence, all statements in Theorem \ref{T:2-1.1} except parts (2) and (3) follow by specializing Theorem \ref{T:7.2}
to the case $\bo=\bmu_n$.

To prove part (2),  suppose that $\bA$ is strongly stable and that $H$ solves the system \eqref{4.20}.
It is known from \cite[Theorem 2.2]{BBF1} that {\em if $\bA=(A_1,\ldots,A_d)$ is strongly stable and
$Q\in\cL(\cX)$ is positive  semidefinite, then, for $P \in \cL(\cX)$,} 
$$
P-\sum_{j=1}^d A_j^*PA_j=Q  \Leftrightarrow P=\sum_{\alpha \in\free} \bA^{*\alpha^\top}Q\bA^\alpha 
$$
In terms of the operator $B_\bA \colon H \mapsto \sum_{j=1}^d A_j^* H A_j$ on $\cL(\cX)$,  we can rephrase the latter statement as saying
that the operator $I_{\cL(\cX)} - B_\bA$ is invertible with the inverse given by
$$
  ( I_{\cL(\cX)} - B_\bA)^{-1} \colon Q \mapsto \sum_{\alpha \in \free} \bA^{*\alpha^\top} Q \bA^\alpha.
$$
If  $I_{\cL(\cX)} - B_\bA$ is invertible, so also is $(I_{\cL(\cX)} - B_\bA)^n$ with inverse given by
$$
\left( (I_{\cL(\cX)} - B_\bA)^n \right)^{-1} = \left((I_{\cL(\cX)} - B_\bA)^{-1} \right)^n.
$$
Indeed that formula for $( I - B_\bA)^{-n}$ works out to be
$$
  (I - B_\bA)^{-n} [Q] = \sum_{\alpha \in \free}  \mu^{-1}_{n, |\alpha|} \bA^{* \alpha^\top} Q \bA^\alpha,
$$
i.e., the series \eqref{4.3} converges with $C^*C$ replaced by any $Q \in \cL(\cX)$, not just for a special $Q = C^* C$.
Next note that the Stein equation in \eqref{4.21} can be viewed as the equation
$$
   (I - B_\bA)^n[H] = C^* C.
 $$
 Hence $H = (I - B_\bA)^{-n}[C^*C]$ necessarily is the unique solution.  As $\cG_{n, C, \bA}$ is given by the same formula  \eqref{4.3}
 we conclude that $\cG_{n, C, \bA}$ is the unique solution.    
 
\smallskip

To complete the proof of part (3) of the theorem it remains to  show that if $\bA$ is a
contraction and the system \eqref{4.21}  admits a unique solution (which necessarily is
$H={\mathcal G}_{n,C,\bA}$), then the tuple $\bA$ is strongly stable. We prove the contrapositive:
{\em if $\bA$ is not strongly stable, then the solution of \eqref{4.21} is not unique.}

\smallskip

Due to assumption \eqref{4.22}, the sequence of operators
$$
\Delta_{N} = B_{\bA}^N[I_{\cX}]=\sum_{\alpha \in \free \colon |\alpha| =
N} {\mathbf A}^{*\alpha^{\top}}
{\mathbf A}^{\alpha},\qquad N=1,2, \dots
$$
is decreasing and therefore has a strong limit $\Delta={\displaystyle\lim_{N \to\infty} \Delta_{N}}\succeq 0$. 
Since $\bA$ is assumed not to be strongly stable, this limit  $\Delta$ is not zero. 
Observe that
$$
\sum_{j=0}^k \bigg(\sum_{\alpha \in\free: \, |\alpha|=j}(-1)^j\bcs{k
\\ j}\bA^{*\alpha^\top}\Delta_{N+k-j}\bA^\alpha\bigg)
=\sum_{j=0}^k(-1)^j\bcs{k \\ j}\Delta_{N+k}
$$
for all $N\ge 0$ and $k\in\{1,\ldots,n\}$. 
Taking the limit in the last equality as $N\to \infty$ 
for a fixed $k$ we get zero on the right side (as  $\sum_{j=0}^{k} (-1)^{k} \bcs{k \\ j}=0$,
by the binomial theorem), while the left side expression tends to
\begin{align*}
\sum_{j=0}^k \bigg(\sum_{\alpha \in\free: \, |\alpha|=j}(-1)^j\bcs{k
\\ j}\bA^{*\alpha^\top}\Delta\bA^\alpha\bigg)&=
\sum_{\alpha \in\free: \, |\alpha|\le k}(-1)^j\bcs{k
\\ |\alpha|}\bA^{*\alpha^\top}\Delta\bA^\alpha\\
&=\Gamma_{k,\bA}[\Delta],
\end{align*}
by \eqref{Gamma-ex}. Thus, $\Gamma_{k,\bA}[\Delta]=0$ for $k=1,\ldots,n$ and therefore, 
the operator $H={\mathcal G}_{n,C,\bA}+\Delta$
(as well as ${\mathcal G}_{n,C,\bA}$) satisfies the system \eqref{4.21}
which therefore has more than one positive-semidefinite solution.
\end{proof}

In Lemma \ref{L:7.6} we saw that a $\bo$-isometric pair $(C,\bA)$ with $\bo$-strongly stable 
tuple $\bA$ is necessarily exactly $\bo$-observable. In the case of $\bo=\bmu_n$ this statement is immediate: 
if $(C,\bA)$ is $n$-isometric, then relations \eqref{4.21} hold for $H=I_{\cX}$ (by the definition of an 
$n$-isometric pair) as well as for $H=\cG_{n,C,\bA}$, by part (1) in Theorem \ref{T:2-1.1}. By the uniqueness
assertion in part (2) of the same theorem, $\cG_{n,C,\bA}=I_{\cX}$ and the pair $(C,\bA)$ is exactly $n$-observable. 
We do not know whether, conversely, the exact $\bo$-observability of a $\bo$-isometric pair $(C,\bA)$ implies the  
$\bo$-strong stability of $\bA$ in the general case. However, for $\bo=\bmu_n$ it does. We have the following
(even stronger) result.

\begin{proposition} \label{P:2-1.1'converse}
           Suppose that the pair $(C, {\mathbf A})$ is $n$-output-stable and
           exactly $n$-observable.  Then ${\mathbf A}$ is strongly stable.
        \end{proposition}
\begin{proof} 
We first verify that for any operator $H\in \cL(\cX)$ and any integers $n\ge 1$ and $N\ge 0$,
\begin{align}
&  H = \sum_{\alpha \in\free: \, |\alpha|\le N} \bcs{n+|\alpha|-1\\
|\alpha|} \bA^{*\alpha^\top} \Gamma_{n,\bA}[H] \bA^{\alpha}\notag\\
&\qquad\qquad + \sum_{\alpha\in\free: \, N+1\le |\alpha|\le N+n}\bcs{N+n\\
|\alpha|}\bA^{*\alpha^\top}\Gamma_{n+N-|\alpha|,\bA}[H]\bA^{\alpha}.
\label{4.9}
\end{align}
To this end,  we use the identity
$$
\sum_{j=0}^N\bcs{n+j-1 \\ j}z^j=
\frac{1}{(1-z)^n}-\sum_{j=1}^{n}\bcs{N+n\\
N+j}\frac{z^{N+j}}{(1-z)^{j}}
$$
giving the explicit formula for the truncation of the infinite series representation for
$(1-z)^{-n}$ (see \cite[Section 2]{BBIEOT} for the proof).
Multiplying both parts in the latter identity by $(1-z)^n$ we get
$$
1=\sum_{j=0}^N\bcs{n+j-1 \\ j}z^j(1-z)^n+\sum_{j=1}^{n}\bcs{N+n\\ N+j}z^{N+j}(1-z)^{n-j},
$$
which in turn implies the operator identity
\begin{equation}  \label{4.11'}
I_{\cL(\cX)}=\sum_{j=0}^N\bcs{n+j-1 \\ j}B_\bA^j(I-B_\bA)^n+\sum_{j=1}^{n}\bcs{N+n\\
N+j}B_\bA^{N+j}(I-B_\bA)^{n-j}.
\end{equation}
Upon applying this latter identity to an
operator $H\in\cL(\cX)$ and making use of \eqref{4.5} and of  definition \eqref{Gamma-ex}
we get \eqref{4.9}:
\begin{align*}
H=&\sum_{j=0}^N\bcs{n+j-1 \\ j}B_\bA^j\left[\Gamma_{n,\bA}[H]\right]+\sum_{j=1}^{n}\bcs{N+n\\
N+j}B_\bA^{N+j}\left[\Gamma_{n-j,\bA}[H]\right]\notag\\
=&\sum_{j=0}^N\bcs{n+j-1 \\ j}\cdot \bigg(\sum_{\alpha \in\free: \,
|\alpha|=j}\bA^{*\alpha^\top}\Gamma_{n,\bA}[H] \bA^{\alpha}\bigg)\notag\\
&+\sum_{j=1}^{n}\bcs{N+n\\ N+j}\cdot\bigg(\sum_{\alpha\in\free: \,
|\alpha|=N+j}
\bA^{*\alpha^\top}\Gamma_{n-j,\bA}[H]\bA^{\alpha}\bigg)\notag\\
=&\sum_{\alpha\in\free: \, |\alpha|\le N}\bcs{n+|\alpha|-1 \\
|\alpha|}\bA^{*\alpha^\top}\Gamma_{n,\bA}[H] \bA^{\alpha}\notag\\
&+\sum_{\alpha \in\free: \, N+1\le |\alpha|\le N+n}\bcs{N+n\\
|\alpha|}\bA^{*\alpha^\top}\Gamma_{n+N-|\alpha|,\bA}[H]\bA^{\alpha}.
\end{align*}
Since the pair $(C,\bA)$ is $n$-output-stable, the $k$-observability gramians $\cG_{k,C,\bA}$ 
are bounded operators for all $0\le k\le n$, by Proposition \ref{P:identities}.
Plugging $H={\mathcal G}_{n,C,\bA}$ into \eqref{4.9} and making use of \eqref{4.11} 
(we note that $\Gamma_{n,\bA}[{\mathcal G}_{n,C,\bA}]=\cG_{0,C,\bA}=C^*C$, by \eqref{ogo3})
we get
\begin{align}
{\mathcal G}_{n,C,\bA}=&\sum_{\alpha \in\free: \, |\alpha|\le
N}\bcs{n+|\alpha|-1\\ |\alpha|}\bA^{*\alpha^\top}C^*C\bA^\alpha\notag\\
&+\sum_{\alpha \in\free: \, N+1\le |\alpha|\le N+n}\bcs{N+n\\
|\alpha|}\bA^{*\alpha^\top}\Gamma_{n+N-|\alpha|,\bA}[{\mathcal
G}_{n,C,\bA}]\bA^{\alpha}\notag\\
=&\sum_{\alpha \in\free: \, |\alpha|\le N}\bcs{n+|\alpha|-1\\
|\alpha|}\bA^{*\alpha^\top}C^*C\bA^\alpha \notag\\
&+\sum_{k=1}^n\bigg(\sum_{\alpha \in\free: \, |\alpha|=N+k}
\bcs{N+n\\ N+k}\bA^{*\alpha^\top}\cG_{k,C,\bA}\bA^\alpha\bigg).
\label{4.28}
\end{align}
From the representation  \eqref{4.3} for ${\mathcal G}_{n,C,\bA}$, taking limits as
$N\to\infty$ in \eqref{4.28} gives
$$
\lim_{N \to \infty}\sum_{k=1}^n\bigg(\sum_{\alpha \in\free: \, |\alpha|=N+k} \bcs{N+n\\
N+k}\bA^{*\alpha^\top}\cG_{k,C,\bA}\bA^\alpha\bigg)=0
$$
which is equivalent, since all the terms on the left are positive semidefinite, to the system of
equalities
$$
\lim_{N \to \infty}\sum_{\alpha \in\free: \, |\alpha|=N+k} \bcs{N+n\\
N+k}\bA^{*\alpha^\top}\cG_{k,C,\bA}\bA^\alpha=0\quad\mbox{for}\quad k=1,\ldots,n.
$$
Letting $k=n$ gives
\begin{equation}
\lim_{N \to \infty}\sum_{\alpha \in\free: \, |\alpha|=N+n}
\bA^{*\alpha^\top}\cG_{n,C,\bA}\bA^\alpha =0\quad\mbox{for}\quad k=1,\ldots,n.
\label{4.29}
\end{equation}
The strict positive-definiteness of ${\mathcal G}_{n,C, {\mathbf A}}$
tells us that there is an $\varepsilon > 0$ so that
\begin{equation}  \label{4.30}
        \varepsilon  \| x \|^{2} \le \langle {\mathcal G}_{n,C, {\mathbf A}} x,
        x \rangle\quad  \text{for all} \; \; x \in {\mathcal X}.
\end{equation}
In particular, from \eqref{4.30} with
${\mathbf A}^{\alpha}x$ in place of $x$ combined with \eqref{4.29} we get
$$
        \varepsilon \cdot\sum_{\alpha \in \free \colon |\alpha| = N+n} \|
        {\mathbf A}^{\alpha} x \|^{2}
          \le \sum_{\alpha \in\free \colon |\alpha|=N+n}
        \langle {\mathcal G}_{n,C, {\mathbf A}}{\mathbf A}^{\alpha}
        x, {\mathbf A}^{\alpha}x
        \rangle \begin{array}{c} \\ \longrightarrow\\ {\scriptstyle N\to\infty}\end{array}0
$$
for all $x \in {\mathcal X}$, and we conclude that ${\mathbf A}$ is
strongly stable as asserted.
 \end{proof}
As we will see Section \ref{S:n-obs} below (specifically Remark \ref{R:stabmu}), for the weight sequence
$\bo={\boldsymbol\mu}_n=\{\bin_{n,j}\}_{j\ge 0}$,
strong stability of a ${\boldsymbol\mu}_n$-hypercontraction $\bA$ is equivalent to the usual
strong stability of $\bA$ (see \eqref{bAstable}).

\section{Shifted $\bo$-gramians}   \label{S:shifted-grammian}

We now formally introduce the shifted versions of the observability operator
$\cO_{\bo,C,\bA}$ \eqref{18.2aaa} and
observability gramian $\cG_{\bo,C,\bA}$ \eqref{18.2aa} which will be needed
in the sequel. Observe that the operators $\cO_{\bo,C,\bA}$ and $\cG_{\bo,C,\bA}$
can be expressed as
$$
\cO_{\bo,C,\bA}: \; x\to C R_{\bo}(Z(z)A)x\quad\mbox{and}\quad \cG_{\bo,C,\bA} = R_{\bo}(B_{\bA})[C^{*}C]
$$
where $R_{\bo}$ is the power series \eqref{1.6g} associated with a given admissible weight $\bo$, 
where $B_{\bA}$ is the operator on $\cL(\cX)$ given by
\eqref{4.4} and where $Z(z)$ and $A$ are defined as in \eqref{1.23pre}.
One gets the shifted variants of these objects by simply
replacing  the function $R_\bo$ in the two formulas above by its shifts $R_{\bo,k}$ (see \eqref{1.6g}):
\begin{align}
&{\Ob}_{\bo,k,C,\bA}: \;  x \to C R_{\bo,k}(Z(z)A)x=\sum_{\alpha \in\free}
\left(\omega_{|\alpha|+k}^{-1} C\bA^\alpha x\right) \, z^\alpha, \label{4.31}\\
&{\Gr}_{\bo,k,C,\bA}:= R_{\bo,k}(B_{\bA})[C^{*}C]=
\sum_{\alpha \in\free}  \omega_{|\alpha|+k}^{-1}\bA^{*\alpha^\top}C^*C\bA^\alpha
\label{4.32}
\end{align}
for $k\ge 1$. Letting $k=0$ in \eqref{4.31}, \eqref{4.32} we conclude
$$
{\Ob}_{\bo,0,C,\bA}={\cO}_{\bo,C,\bA}\quad\mbox{and}\quad {\Gr}_{\bo,0,C,\bA}=\cG_{\bo,C,\bA}.
$$
Note that as a consequence of identity \eqref{3.19gg}
we have the identity
\begin{equation}   \label{3.19gg'}
    \Gamma_{\bo, \bA}^{(k)}[ \cG_{\bo, C, \bA}] = \Gr_{\bo, k, C,
    \bA}.
\end{equation}

\begin{remark}  \label{R:bo=bmun}
It is instructive to note the following elementary verification of the identity \eqref{3.19gg'} for the case where
$\bo = \bmu_n$ for some $n=1,2,3,\dots$.  In this case, by making use of the identity 
\eqref{Gamma-bmun} along with its more fundamental companion \eqref{1.10pre} and the definitions of the
various quantities involved, we have
\begin{align*}
\Gamma_{n,\bA}^{(k)}[\cG_{n,C,\bA}] & =
\sum_{\ell = 0} ^{n-1}  \left( \begin{smallmatrix} \ell + k - 1 \\ \ell \end{smallmatrix} \right)
(I - B_\bA)^\ell (I - B_\bA)^{-n} [C^* C]  \\
& = \sum_{\ell = 0}^{n-1} \left( \begin{smallmatrix} \ell + k - 1 \\ \ell \end{smallmatrix} \right) 
(I - B_\bA)^{-(n-\ell)}[C^*C]  \\
& = \sum_{\ell=0}^{n-1} \left(  \begin{smallmatrix} \ell + k - 1 \\ \ell \end{smallmatrix} \right) R_{n-\ell}(B_\bA)[C^*C]  \\
& = R_{n,k}(B_\bA)[C^*C] = \Gr_{n,k,C,\bA},
\end{align*}
and \eqref{3.19gg'} follows for the case $\bo = \bmu_n$.
\end{remark}

\begin{proposition}
\label{P:wghtSteinid} 
If the weight sequence $\bo$ is admissible and 
the pair $(C,\bA)$ is $\bo$-output-stable, then the operator 
$\Gr_{\bo,k,C,\bA}$ is bounded for all $k\ge 1$, and the weighted Stein identity
    \begin{equation}
\sum_{j=1}^d A_j^*\Gr_{\bo,k+1,C,\bA}A_j+\omega_k^{-1}\cdot C^*C=\Gr_{\bo,k,C,\bA}
\label{4.34}
\end{equation}
holds for all $k\ge 0$. Furthermore,
\begin{equation}
\Gr_{\bo,k+1,C,\bA}\succeq \Gr_{\bo,k,C,\bA}\succeq \Gr_{\bo,0,C,\bA}
=\cG_{\bo,C,\bA}\quad \mbox{for all}\quad k\ge 1.
\label{grmonotone}
\end{equation}
\end{proposition}

\begin{proof}
Since $\frac{\omega_j}{\omega_{j+1}}\le M$ for all $j\ge 0$ (see
\eqref{18.2}), we have
$\frac{\omega_{j}}{\omega_{j+k}}\le M^{k}$ for all $k,j\ge 0$,
and then it follows from \eqref{4.32} and \eqref{18.2aa} that
\begin{align*}
{\Gr}_{\bo,k,C,\bA}&=\sum_{\alpha\in\free}^\infty
\frac{\omega_{|\alpha|}}{\omega_{|\alpha|+k}}\omega_{|\alpha|}^{-1}\bA^{*\alpha^\top}C^*C
\bA^{\alpha}\\
&\preceq M^k \sum_{\alpha\in\free}^\infty
\omega_{|\alpha|}^{-1}\bA^{*\alpha^\top}C^*C\bA^{\alpha}=M^k \cdot\cG_{\bo,C,\bA}.
\end{align*}
Since the pair $(C,\bA)$ is $\bo$-output-stable, the $\bo$-gramian $\cG_{\bo,C,\bA}$ is bounded and hence,
${\Gr}_{\bo,k,C,\bA}$ is bounded as well.

\smallskip

We next apply the $B_{\bA}$ calculus to the identity \eqref{1.6g'} to get the operator equality
$$
R_{\bo,k}(B_\bA) - B_\bA\circ R_{\bo,k+1}(B_\bA)= \omega_k^{-1}\cdot I_{\cL(\cX)}
$$
which in turn being applied to the operator $C^*C$ gives, on account of \eqref{4.32},
\begin{align*}
\left( R_{\bo,k}(B_{\bA}) - B_{\bA}\circ R_{\bo,k+1}(B_{\bA}) \right)[C^{*}C]
&=\Gr_{\bo,k,C,\bA} - \sum_{j=1}^d A_j^{*} \Gr_{\bo,k+1,C,\bA} A_j\\
&=\omega_k^{-1}\cdot C^*C,
\end{align*}
and we arrive at \eqref{4.34} as wanted. Inequalities \eqref{grmonotone} follow from power series representation
\eqref{4.32} due to the fact that the sequence $\bo$ is non-increasing.
\end{proof}
For our future purposes, it is convenient to identify conditions which guarantee 
that all the shifted gramians are invertible, i.e., bounded below.
\begin{proposition}   \label{P:shift-gram-inv}
Suppose that the $\bo$-output stable pair $(C, \bA)$ is furthermore exactly $\bo$-observable (as is the
case for example if $(C, \bA)$ is $\bo$-isometric and $\bA$ is $\bo$-strongly stable by part (2) of
Lemma \ref{L:7.6}).  Then the shifted gramians $\Gr_{\bo, k, C, \bA}$ are all bounded below, and hence
invertible.
\end{proposition}

\begin{proof}  Exact $\bo$-observability of $(C, \bA)$ by definition means that $\cG_{\bo, C, A}$ is
bounded below.  The result now follows from the chain of inequalities \eqref{grmonotone}.
\end{proof}

\section{The model shift-operator tuple on $H^2_{\bo, \cY}(\free)$}  \label{S:bo-shift}

In this section we pursue a deeper study of the model operator tuples 
$$
{\bf S}_{\bo,R}=(S_{\bo, R,1},\ldots, S_{\bo, R,d})\quad\mbox{and}\quad
{\bf S}^*_{\bo,R}=(S^*_{\bo, R,1},\ldots, S^*_{\bo, R,d})
$$
on the Hardy-Fock space $H^2_{\bo, \cY}(\free)$ introduced in Section 2.2 by the formulas
\begin{equation}  \label{bo-shift}
  S_{\bo, R, j } \colon f(z) \mapsto f(z) z_j, \quad S^*_{\bo, R, j} \colon \sum_{\alpha \in \free} f_\alpha z^\alpha
  \mapsto \sum_{\alpha \in \free} \frac{\omega_{|\alpha| +1}}{\omega_{|\alpha|}} f_{\alpha j} z^\alpha.
 \end{equation}
We define $E \colon H^2_{\bo, \cY}(\free) \to \cY$ to be the free-coefficient evaluation operator
and we say that the pair
 \begin{equation}  \label{modelE}
(E, \bS^*_{\bo, R})\quad\mbox{with}\quad {\bf S}^*_{\bo,R}=(S^*_{\bo, R,1},\ldots, S^*_{\bo, R,d}),\quad
E  \colon \sum_{\alpha \in \free} f_\alpha z^\alpha \mapsto f_\emptyset
 \end{equation}
is the {\em $\bo$-model output pair}.

\begin{proposition}  \label{P:bo-model}
Let $\bo$ be an admissible weight sequence (in the sense of \eqref{18.2}) and let 
$(E, \bS^*_{\bo, R})$ be the $\bo$-model output pair defined as in \eqref{modelE}.
 \begin{enumerate}
 \item
${\bf S}_{\bo,R}^*$ is strongly stable in the sense that
$$
\lim_{N \to \infty}   \sum_{ \alpha \in \free \colon |\alpha| = N} \| \bS_{\bo,R}^{* \alpha^\top} f \|^2 = 0
\quad\text{for each}\quad f \in H^2_{\bo, \cY}(\free).
$$

\item The $\bo$-observability operator $\cO_{\bo,E, {\bf S}_{\bo,R}^*}$ equals $I_{H^2_{\bo,\cY}(\free)}$ 
and hence the pair $(E, {\bf S}_{\bo,R}^*)$ is $\bo$-output stable and exactly $\bo$-observable.

\item The action of the shifted $\bo$-observability operators  $\Ob_{\bo,k, E, \bS_{\bo,R}^*}$ and the shifted 
$\bo$-gramians ${\mathfrak G}_{\bo,k,E,{\bf S}_{\bo,R}^*}$ on $f\in H^2_{\bo, \cY}(\free)$ are given by the formulas
\begin{equation}
{\mathfrak G}_{\bo,k,E,{\bf S}_{\bo,R}^*}=\Ob_{\bo,k, E, \bS_{\bo,R}^*} \colon \sum_{\alpha \in \free}
    f_{\alpha} z^{\alpha} \mapsto \sum_{\alpha \in \free}
    \frac{\omega_{|\alpha|}}{\omega_{k+|\alpha|}} f_{\alpha}z^{\alpha}  \label{grer}  
\end{equation}
for all $k\ge 0$. Furthermore, for any $\beta\in\free$,
\begin{equation}
{\bf S}_{\bo,R}^{*\beta}{\bf S}_{\bo,R}^{\beta^\top}\Ob_{\bo,|\beta|, E, \bS_{\bo,R}^*}=
{\bf S}_{\bo,R}^{*\beta}{\bf S}_{\bo,R}^{\beta^\top}\Gr_{\bo,|\beta|, E, \bS_{\bo,R}^*}=I_{H^2_{\bo,\cY}(\free)}.
\label{gramss}
\end{equation}

\item For every $f\in H^2_{\bo,\cY}(\free)$,
\begin{align}
f(z)-f_{\emptyset}&=\sum_{j=1}^d (S_{\bo,R,j}{\mathfrak G}_{\bo,1,E,{\bf S}_{\bo,R}^*}S_{\bo,R,j}^*f)(z)\notag\\
&=\sum_{j=1}^d(S_{\bo,R,j}^*{\mathfrak G}_{\bo,1,E,{\bf S}_{\bo,R}^*}f)(z)\cdot z_j.
\label{4.13a}
\end{align}
\item The operator $\Upsilon: \; (H^2_{\bo,\cY}(\free))^d\oplus \cY\to H^2_{\bo,\cY}(\free)$ defined by
\begin{equation}
\Upsilon=\begin{bmatrix}S_{\bo,R,1} {\mathfrak G}_{\bo,1,E,{\bf S}_{\bo,R}^*}^{\frac{1}{2}} & \ldots & 
S_{\bo,R,1} {\mathfrak G}_{\bo,1,E,{\bf S}_{\bo,R}^*}^{\frac{1}{2}} &  E^*\end{bmatrix}
\label{ups}
\end{equation}
is unitary.  In particular the following operator identity holds:
\begin{equation} \label{upscoiso}
    \sum_{j=1}^{d} S_{\bo,R,j} {\mathfrak G}_{\bo,1,E,{\mathbf S}_{\bo,R}^{*}}
    S_{\bo,R,j}^{*} + E^{*} E = I_{H^2_{\bo,\cY}(\free)}.
 \end{equation}
\end{enumerate} 
\end{proposition}

\begin{proof}[Proof of (1)]
Observe that if $f(z)={\displaystyle \sum_{\beta\in\free}f_\beta z^\beta}$
belongs to $H^2_{\bo,\cY}(\free)$, then
\begin{equation}
 \lim_{N \to \infty} \sum_{\beta\in\free:|\beta|\ge N}\omega_{|\beta|}\|f_\beta\|^2_{\cY} = 0
\label{ref8g}
\end{equation}
by \eqref{18.1}.   On the other hand, iteration of the second formula in \eqref{bo-shift} gives
\begin{equation}   \label{Sbo*-iterated}
( (\bS^*_{\bo, R})^{\alpha}f)(z) = \sum_{\beta \in \free} \frac{ \omega_{|\beta| + |\alpha|}}{\omega_{|\beta|}} f_{\beta \alpha} z^\beta
\end{equation}
and hence
$$
  \| \bS_{\bo,R}^{* \alpha} f \|^2_{H^2_{\bo, \cY}(\free)}  = \sum_{\beta \in \free}  \frac{ \omega_{|\beta| + |\alpha|}^2}{\omega_{|\beta|}}
  \| f_{\beta \alpha} \|^2_\cY.
$$
We now conclude from this last identity combined with the observation \eqref{ref8g} that
\begin{align*}
 \sum_{\alpha \in \free \colon |\alpha| = N} \| \bS_{\bo, R}^{* \alpha} f \|^2_{H^2_\bo(\free)}
& = \sum_{\alpha \in \free \colon |\alpha| = N} \frac{ \omega^2_{|\beta| + N}}{\omega_{|\beta|}} \| f_{\beta \alpha} \|^2_\cY  \\
& \le
\sum_{\alpha \in \free \colon |\alpha| = N} \sum_{\beta \in \free} \omega_{|\beta| + N} \| f_{\beta \alpha} \|^2_\cY  \\
& = \sum_{\gamma \in \free \colon |\gamma| \ge N} \omega_{|\gamma|} \| f_\gamma \|^2_\cY \to 0\quad 
\text{as}\quad N \to \infty,
\end{align*}
which proves the strong stability of $\bS_{\bo, R}^*$.
\end{proof}
\begin{proof}[Proof of (2)]
To verify that the $\bo$-observability operator $\cO_{\bo, E, \bS_{\bo, R}}$ is the identity operator on $H^2_{\bo, \cY}(\free)$,
note that by the formulas \eqref{Sbo*-iterated} and \eqref{modelE},
\begin{equation}
  E (\bS^*_{\bo, R})^\alpha  f = \omega_{|\alpha|}f_\alpha
\label{esbo}
\end{equation}
and therefore,
$$
\cO_{\bo, E, \bS^*_{\bo, R}} f  = \sum_{\alpha \in \free} ( \omega^{-1}_{|\alpha|} E \bS^{* \alpha}_{\bo, R} f) z^\alpha 
= \sum_{\alpha \in free} \omega^{-1}_{|\alpha|}\omega_{|\alpha|} f_\alpha z^\alpha = \sum_{\alpha \in \free} f_\alpha z^\alpha = f
$$
for all $f \in H^2_{\bo, \cY}(\free)$, i.e., $\cO_{\bo, E, \bS^*_{\bo, R}} = I_{H^2_{\bo, \cY}(\free)}$ as asserted.
\end{proof}

\begin{proof}[Proof of (3)]
To verify \eqref{grer}, we first combine \eqref{4.31} and \eqref{esbo} to compute
$$
\Ob_{\bo,k,E,{\bf S}_{\bo,R}^{*}} f =\sum_{\alpha\in\free}\omega_{k+|\alpha|}^{-1} \left( E
{\bf S}_{\bo,R}^{*\alpha}f\right) z^{\alpha} = \sum_{\alpha\in\free}
\frac{\omega_{|\alpha|}}{\omega_{k+|\alpha|}} f_{\alpha} z^{\alpha}.
$$
Since $\omega_0=1$, the adjoint operator $E^*: \, \cY\to H^2_{\bo,\cY}(\free)$ (see \eqref{modelE}) amounts to 
the inclusion of $\cY$ into $H^2_{\bo,\cY}(\free)$ which along with \eqref{esbo} leads us to
$$
{\bf S}_{\bo,R}^{\alpha\top}E^*E{\bf S}_{\bo,R}^{*\alpha}f=
{\bf S}_{\bo,R}^{\alpha\top}E^*(\omega_{|\alpha|}f_\alpha)=\omega_{|\alpha|}f_\alpha z^\alpha.
$$ 
Combining the latter equality with \eqref{4.32} we get
\begin{equation}
{\mathfrak G}_{\bo,k,E,{\bf S}_{\bo,R}^*}f=
\sum_{\alpha \in\free}\omega_{|\alpha|+k}^{-1}{\bf S}_{\bo,R}^{\alpha\top}E^*E{\bf S}_{\bo,R}^{*\alpha}f=
\sum_{\alpha \in\free}\frac{\omega_{|\alpha|}}{\omega_{|\alpha|+k}}f_\alpha z^\alpha,
\label{gogo}
\end{equation}
which completes the verification of \eqref{grer}. By iterating formulas \eqref{bo-shift} we have 
\begin{equation}
{\bf S}^{*\beta}_{\bo,R}{\bf S}_{\bo,R}^{\beta^{\top}}: \; \sum_{\alpha\in\free}
f_\alpha z^\alpha\to \sum_{\alpha\in\free}
\frac{\omega_{|\alpha|+|\beta|}}{\omega_{|\alpha|}}\, f_\alpha z^\alpha
\label{again43}
\end{equation}
for any $\beta\in\free$ which being combined with \eqref{grer} (with $k=|\beta|$) leads us to 
equalities \eqref{gramss}. We note that specialization of the equalities \eqref{gramss} 
to the case where $\beta = \emptyset$ amounts to the content of part (2) of the proposition.
\end{proof}

\begin{proof}[Proof of (4)] Combining \eqref{bo-shift} and \eqref{grer} (with $k=1$) gives
$$
S_{\bo,R,j} {\mathfrak G}_{\bo,1,E,{\mathbf S}_{\bo,R}^{*}} S_{\bo,R,j}^{*}: \, f(z)=\sum_{\beta\in\free}f_\beta z^\beta
\mapsto \sum_{\beta\in\free}f_{\beta j}z^{\beta j} 
$$
for $j=1,\ldots,d$. Therefore, we have for any $f\in H^2_{\bo,\cY}(\free)$,
\begin{align*}
\bigg(\sum_{j=1}^{d} S_{\bo,R,j} {\mathfrak G}_{\bo,1,E,{\mathbf S}_{\bo,R}^{*}} S_{\bo,R,j}^{*}\bigg)f
&=\sum_{j=1}^d \sum_{\beta\in\free}f_{\beta j} z^{\beta j}\\
&=\sum_{\beta\in\free: \beta\neq\emptyset}f_\beta z^\beta=f(z)-f_{\emptyset} 
\end{align*}
which verifies the identity \eqref{4.13a}.
\end{proof}

\begin{proof}[Proof of (5)] Equality \eqref{upscoiso} is just the operatorial form of the identity 
\eqref{4.13a}. With $\Upsilon$ defined as in \eqref{ups}, \eqref{upscoiso} amounts to the statement that $\Upsilon$
is a coisometric.  We next show that $\Upsilon$ is also isometric. 

\smallskip

Since $E^*$ is the inclusion map of $\cY$ into $H^2_{\bo,\cY}(\free)$, we have
$EE^*=I_{\cY}$. It is readily seen from formulas \eqref{bo-shift} and \eqref{modelE} that
$ES_{\bo,R,j}=0$ for $j=1,\ldots,d$ and $S_{\bo,R,j}^*S_{\bo,R,k}=0$, whenever $j\neq k$. 
On the other hand, the operator $S_{\bo,R,j}^*S_{\bo,R,j}$ acts on $H^2_{\bo,\cY}(\free)$ as follows:
\begin{equation}
S_{\bo, R, j}^*S_{\bo, R, j}: \; \; \sum_{\alpha \in \free} f_{\alpha} z^{\alpha} \mapsto
\sum_{\alpha \in \free}
\frac{\omega_{|\alpha|+1}}{\omega_{|\alpha|}}\cdot f_{\alpha} z^{\alpha},
\label{jan11}
\end{equation}
which together with \eqref{grer} implies that 
\begin{equation}
S_{\bo,R,j}^*S_{\bo,R,j}={\mathfrak G}_{\bo,1,E,{\bf S}_{\bo,R}^*}^{-1}\quad \mbox{for}\quad j=1,\ldots,d.
\label{jan11a}
\end{equation}
Summarizing we have the set of equalities
$$
EE^*=I_{\cY}, \quad ES_{\bo,R,j}=0, \quad {\mathfrak G}_{\bo,1,E,{\bf S}_{\bo,R}^*}^{\frac{1}{2}}S_{\bo,R,j}^*
S_{\bo, R,k}{\mathfrak G}_{\bo,1,E,{\bf S}_{\bo,R}^*}^{\frac{1}{2}}=\delta_{kj} I
$$
for $k,j=1,\ldots,d$ (where $\delta_{kj}$ is the Kronecker symbol) implying that the operator $\Upsilon$ is isometric.
Thus $\Upsilon$ is unitary.
\end{proof}

We emphasize that in Proposition \ref{P:bo-model} we only assumed that the weight sequence $\bo$ is admissible.
Now we impose the extra condition \eqref{1.8g}.
\begin{proposition}  \label{P:bo-model2} 
Let us assume that the weight sequence $\bo$ meets conditions \eqref{18.2}, \eqref{1.8g}. Then the 
operator tuple ${\bf S}_{\bo,R}^*$ on $H^2_{\bo,\cY}(\free)$ is a $\bo$-strongly stable 
$\bo$-hypercontraction, while the $\bo$-model output pair $(E, \bS^*_{\bo, R})$ is $\bo$-isometric. 
Moreover,
\begin{equation}
\Gamma_{\bo,{\bf S}_{\bo,R}^*}[I_{H^2_{\bo,\cY}(\free)}]=E^*E\quad\mbox{and}\quad\
\Gamma^{(k)}_{\bo,{\bf S}_{\bo,R}^*}[I_{H^2_{\bo,\cY}(\free)}]=\Gr_{\bo,k,E,{\bf S}_{\bo,R}^*}
\label{3.1e}
\end{equation}
for all $k\ge 1$, or equivalently,
\begin{align}
\sum_{\alpha\in\free} c_{|\alpha|}\|{\bf S}_{\bo,R}^{*\alpha}f\|_{H^2_{\bo,\cY}(\free)}^2&=
\|f_{\emptyset}\|^2_{\cY},
\label{3.1g}\\  
- \sum_{\alpha\in\free}\bigg(\sum_{\ell=1}^{k}\frac{c_{|\alpha|+\ell}}{\omega_{k-\ell}}\bigg)
\, \|{\bf S}_{\bo,R}^{*\alpha}f\|_{H^2_{\bo,\cY}(\free)}^2&= \sum_{\alpha\in\free}
\frac{\omega_{|\alpha|}^{2}}{\omega_{k+|\alpha|}}\|f_\alpha\|^2_{\cY}
\qquad (k\ge 1)\label{3.1h}
\end{align}
for all $f\in H^2_{\bo,\cY}(\free)$, where $c_j$'s are given in \eqref{1.8g}.
\end{proposition}

\begin{proof}
By part (2) in Proposition \ref{P:bo-model}, $\cG_{\bo,E,{\bf S}_{\bo,R}^*}=I$ and hence, 
equalities \eqref{3.1e} follow from general 
formulas \eqref{3.19g} and \eqref{3.19gg'}. 
Since $EE^*\succeq 0$ and $\Gr_{\bo,k,E,{\bf S}_{\bo,R}^*}\succeq 0$, we conclude from \eqref{3.1e} 
that ${\bf S}_{\bo,R}^*$ is $\bo$-hypercontractive. Next we note that the first equality in \eqref{3.1e} is equivalent 
to the quadratic-form identity
$$
\langle \Gamma_{\bo,{\bf S}_{\bo,R}^{*}}[I_{H^2_{\bo,\cY}(\free)}]f, f   
\rangle_{H^2_{\bo,\cY}(\free)} = \langle E^{*} E f, f \rangle_{H^2_{\bo,\cY}(\free)}\quad\mbox{for all}\quad
f \in H^{2}_{\bo}(\cY),
$$
which in turn is equivalent to \eqref{3.1g}.    The equivalence of the
second identity in \eqref{3.1e} with the quadratic-form
identity \eqref{3.1h} follows from  the equality
$$ \langle\Gamma^{(k)}_{\bo,{\bf S}_{\bo,R}^{*}}[I_{H^2_{\bo,\cY}(\free)}]f,\, f
\rangle_{H^2_{\bo,\cY}(\free)} =
-\sum_{\alpha\in\free}\bigg(\sum_{\ell=1}^{k}\frac{c_{|\alpha|+\ell}}{\omega_{k-\ell}}\bigg)
\, \|{\bf S}_{\bo,R}^{*\alpha}f\|_{H^2_{\bo,\cY}(\free)}^2
$$
which is immediate from the definitions, and the identity
\begin{equation}  \label{Grk-quadformg}  
\langle \Gr_{\bo, k,E, {\bf S}_{\bo,R}^{*}}f, \, f
\rangle_{H^2_{\bo,\cY}(\free)} =
\sum_{\alpha \in \free}\frac{\omega_{|\alpha|}^{2}}{\omega_{k+|\alpha|}}
\|f_{\alpha}\|^{2}_{\cY}
\end{equation}
which follows from \eqref{gogo} and the definition of the inner product in $H^2_{\bo,\cY}(\free)$.
Combining \eqref{Grk-quadformg} and \eqref{Sbo*-iterated} gives
\begin{align*}
&\sum_{\alpha\in\free:|\alpha|=k}
\langle \Gr_{\bo,k, E, {\bf S}_{\bo,R}^{*}} 
{\bf S}_{\bo,R}^{*\alpha}f, \, {\bf S}_{\bo,R}^{*\alpha}f\rangle_{H^{2}_{\bo}(\cY)}\\
&\qquad =\sum_{\beta,\alpha\in\free:|\alpha|=k} \frac{\omega_{|\beta|}^{2}}{\omega_{k+|\beta|}}
\|({\bf S}_{\bo,R}^{*\alpha}f)_{\beta}\|^{2}_{\cY}\\
&\qquad=\sum_{\beta,\alpha\in\free:|\alpha|=k} 
\frac{\omega_{|\beta|}^{2}}{\omega_{k+|\beta|}}
\bigg\|\frac{\omega_{|\beta|+|\alpha|}}{\omega_{|\beta|}}f_{\beta\alpha}\bigg\|^{2}_{\cY}\\
&\qquad=\sum_{\beta,\alpha\in\free:|\alpha|=k}
\omega_{|\beta|+|\alpha|}\cdot\|f_{\beta\alpha}\|^2_{\cY}
=\sum_{\beta\in\free: |\beta|\ge k}\omega_{|\beta|}\cdot\|f_{\beta}\|^2_{\cY}.
\end{align*}
This together with \eqref{ref8g} and the second equality in \eqref{3.1e} implies that
\begin{align*}
& \sum_{\alpha\in\free:|\alpha|=k}\langle {\bf S}_{\bo,R}^{\alpha^\top}
\Gamma_{\bo,k,S_{\bo}^*}[I_{\cH}]{\bf S}_{\bo,R}^{*\alpha}f, 
\, f\rangle_{H^{2}_{\bo}(\cY)}  \\
 & \quad \quad \quad \quad  =\sum_{\alpha\in\free:|\alpha|=k}\langle \Gr_{\bo,k, E, {\bf S}_{\bo,R}^{*}} 
{\bf S}_{\bo,R}^{*\alpha}f, \,
{\bf S}_{\bo,R}^{*\alpha}f\rangle_{H^{2}_{\bo}(\cY)}  \\
&\quad \quad \quad \quad  =\sum_{\beta\in\free: |\beta|\ge k}\omega_{|\beta|}\cdot\|f_{\beta}\|^2_{\cY}\to
0\quad\mbox{as}\quad k\to\infty.
\end{align*}
This finally verifies $\bo$-strong stability of $ {\bf S}^*_{\bo,R}$ and completes the proof.
\end{proof}

\section[Characterization of $\bS_{\bo,R}$]{A characterization of the model shift-operator tuple on $H^2_{\bo,\cY}(\free)$}\label{S:olof}
Let us note from the results of Section \ref{S:bo-shift}  that the model shift-operator tuple 
$\bS_{\bo,R}$ on $H^2_{\bo, \cY}(\free)$ enjoys the following properties:
\begin{align}
 & (S_{\bo,R,j}^*S_{\bo,R,j})^{-1}=\Gamma^{(1)}_{\bo,{\bf S}_{\bo,R}^*}[I_{H^2_{\bo,\cY}(\free)}]
\quad\mbox{for}\quad j=1,\ldots,d,
\label{ol-a}   \\
& \operatorname{Ran} S_{\bo, R, j} \perp \operatorname{Ran} S_{\bo, R, k}\quad \text{for}\quad j \ne k,  \label{ol-b} \\
& \bigcap_{N\ge 0}\bigvee_{\alpha\in\free:|\alpha|=N} {\bf S}_{\bo,R}^\alpha H^2_{\bo,\cY}(\free) = \{0\}
\label{ol-c}
\end{align}
(where we use the notation $\bigvee$ to denote {\em closed linear span}).
Indeed \eqref{ol-a} follows by combining \eqref{jan11a} and the second equality in \eqref{3.1e} (for $k=1$),
\eqref{ol-b} is clear from the form of the inner product on $H^2_{\bo, \cY}(\free)$, and \eqref{ol-c}
is clear from the observation that
$$
\bigcap_{N\ge 0}\bigvee_{\alpha\in\free \colon |\alpha|=N} {\bf S}_{\bo,R}^\alpha H^2_{\bo,\cY}(\free) 
\subset \bigcap_{N \ge 0} \bigvee_{ \alpha \in \free \colon |\alpha| = N} \cY\langle \langle z \rangle \rangle \cdot z^\alpha = \{0\}.
$$
The goal of this section is to prove a remarkable converse: {\em if  $\bT = (T_1, \dots, T_d)$ is any operator $d$-tuple 
on a Hilbert space $\cX$ satisfying conditions \eqref{ol-a}, \eqref{ol-b}, \eqref{ol-c}, then $\bT$ is 
unitarily equivalent to $\bS_{\bo, R}$ on $H^2_{\bo, \cY}(\free)$, where the coefficient Hilbert space $\cY$ is chosen so that 
$\dim \cY = \dim (\cX \ominus \bigvee_{j=1, \dots, d} \operatorname{Ran} T_j)$. }

\smallskip

We therefore let $\bo$ to be the weight sequence subject to conditions \eqref{18.2}, \eqref{1.8g}.  Let us define
a class of operators $\cC(\bo)$ to consist of {\em all operator-tuples  $\bT=(T_1,\ldots,T_d)$
of left-invertible operators $T_j\in\cL(\cX)$ with  mutually orthogonal ranges
and satisfying the additional identity}
\begin{equation}
(T_j^*T_j)^{-1}=\Gamma^{(1)}_{\bo,\bT^*}[I_{\cX}].
\label{ol2}
\end{equation}
One can check that for the case $\bo = \b1$ (the weight sequence consisting of all $1$'s), the class $\cC(\b1)$ consists of 
the set of all row isometries; we resist using the term {\em $\bo$-isometries} in general for the class $\cC(\bo)$ since this term is reserved for
 the class of all operator tuples $\bT = (T_1, \dots, T_d)$ such that $\Gamma_{\bo, \bT}(I_\cX) = 0$ (compatible with the terminology
 {\em $\bo$-contraction} for the class of all operator tuples $\bT$ with $\Gamma_{\bo, \bT}(I_\cX) \succeq 0$).
Given an operator-tuple ${\bf T} \in \cC(\bo)$ we associate the subspaces
\begin{equation}
\cE=\cX\ominus\bigg(\bigoplus_{j=1}^d T_j\cX\bigg)\quad\mbox{and}\quad \cX_0=
\bigcap_{N\ge 0}\bigvee_{\alpha\in\free:|\alpha|=N}{\bf T}^\alpha\cX.
\label{ol4}
\end{equation}
of $\cX$. Letting $A_j=T_j^*$, $X=I_{\cX}$ and $k=1$ in \eqref{1.11g} and \eqref{3.5ag} and taking into account that 
$\omega_0=c_0=1$, we get explicit formulas
\begin{equation}
\Gamma_{\bo,\bT^*}[I_{\cX}]=\sum_{\alpha\in\free}c_{|\alpha|}{\bf T}^{\alpha^\top}{\bf T}^{*\alpha},\quad
\Gamma^{(1)}_{\bo,\bT^*}[I_{\cX}]=    -\sum_{\alpha\in\free}c_{|\alpha|+1}{\bf T}^{\alpha^\top}{\bf T}^{*\alpha},
\label{ol2a}
\end{equation}
from which we see that
\begin{equation}  \label{ol2a'}
I_{\cX}-\sum_{j=1}^d T_j\Gamma^{(1)}_{\bo,\bT^*}[I_{\cX}]T_j^*=\Gamma_{\bo,\bT^*}[I_{\cX}],
\end{equation}
which, on account of \eqref{ol2}, can be written as
\begin{equation}
I_{\cX}-\sum_{j=1}^d T_j(T_j^*T_j)^{-1}T_j^*=\Gamma_{\bo,\bT^*}[I_{\cX}].
\label{ol3'}
\end{equation}
Since the ranges of $T_1,\ldots,T_d$ are  mutually orthogonal, it follows from \eqref{ol3'} that
$\Gamma_{\bo,\bT^*}[I_{\cX}]$ is an orthogonal projection. From \eqref{ol4} we see that the range 
space of this projection is the space $\cE$ given by \eqref{ol4} and  hence  we arrive at the more complete
version of \eqref{ol3'}:
\begin{equation}  \label{ol3}
I_{\cX}-\sum_{j=1}^d T_j(T_j^*T_j)^{-1}T_j^*=\Gamma_{\bo,\bT^*}[I_{\cX}] = P_\cE.
\end{equation}
The Cauchy dual tuple ${\bf L}=(L_1,\ldots,L_d)$ of ${\bf T}$ is defined by
\begin{equation}
L_j=T_j(T_j^*T_j)^{-1}\quad\mbox{for}\quad j=1,\ldots,d,
\label{ol5}
\end{equation}
and it is readily seen from \eqref{ol2} that
\begin{equation}
L_j^*T_j=I_\cX\quad\mbox{and}\quad L_j^*T_i=0\quad\mbox{for all}\quad i\neq j.
\label{ol6}
\end{equation}
\begin{proposition}
Let ${\bf T}$ be a tuple of left-invertible operators satisfying conditions \eqref{ol2}, let ${\bf L}$ be its Cauchy dual, and let
$\cE$ be defined as in \eqref{ol4}. Then
\begin{equation}
P_{\cE}{\bf L}^{*\beta}=\omega_{|\beta|}^{-1}P_{\cE}{\bf T}^{*\beta}\quad\mbox{for all}\quad \beta\in\free.
\label{ol7}
\end{equation}
\label{P:ol1}
\end{proposition}
\begin{proof} We prove \eqref{ol7} by induction in $|\beta|$. The basis case $|\beta|=0$ is trivial. Assuming
that \eqref{ol7} holds for all $\beta$ with $|\beta|\le N$ we will verify \eqref{ol7} for $\beta=\alpha j$ for
a fixed $\alpha$ ($|\alpha|=N$). To this end, we first observe from \eqref{ol5} and \eqref{ol2} that
\begin{align}
P_{\cE}{\bf L}^{*\alpha j}=P_{\cE}{\bf L}^{*\alpha}L_j^*&=P_{\cE}{\bf L}^{*\alpha}(T_j^*T_j)^{-1}T_j^*\notag\\
&=P_{\cE}{\bf L}^{*\alpha}\Gamma^{(1)}_{\bo,\bT^*}[I_{\cX}]T_j^*\notag\\
&=-P_{\cE}{\bf L}^{*\alpha}\bigg(\sum_{\beta\in\free}c_{|\beta|+1}{\bf T}^{\beta^\top}{\bf T}^{*\beta}\bigg)T_j^*,\label{ol8}
\end{align}
where we used the second formula from \eqref{ol2a} for the last step.
We next observe from \eqref{ol4} and \eqref{ol6} that ${\bf L}^{*\alpha}{\bf T}^{\beta^\top}\neq 0$ if and only if either
$\beta=\delta\alpha$ (for some non-empty $\delta\in\free$) or $\alpha=\gamma\beta$ (for some $\gamma\in\free$).
In the first case,
$$
P_{\cE}{\bf L}^{*\alpha}{\bf T}^{\beta^\top}=P_{\cE}{\bf L}^{*\alpha}{\bf T}^{\alpha^\top\gamma^\top}=
P_{\cE}{\bf T}^{\gamma^\top}=0,
$$
while in the second case, we have
\begin{align*}
P_{\cE}{\bf L}^{*\alpha}{\bf T}^{\beta^\top}{\bf T}^{*\beta}&=P_{\cE}{\bf L}^{*\gamma\beta}{\bf T}^{\beta^\top}{\bf T}^{*\beta}\\
&=P_{\cE}{\bf L}^{*\gamma}{\bf T}^{*\beta}\\
&=\omega_{|\gamma|}^{-1}P_{\cE}{\bf T}^{*\gamma}{\bf T}^{*\beta}=\omega_{N-|\beta|}^{-1}P_{\cE}{\bf T}^{*\alpha},
\end{align*}
where we have used the induction hypothesis for the third equality. Thus, we have only $N+1$ (nonzero) terms on the
right side of \eqref{ol8}:
$$
P_{\cE}{\bf L}^{*\alpha j}=-P_{\cE}\bigg(\sum_{i=0}^N c_{i+1}\omega_{N-i}^{-1}\bigg){\bf T}^{*\alpha}T_j^*.
$$
By the equality \eqref{3.18u} (for $k=1$), ${\displaystyle\sum_{i=0}^N c_{i+1}\omega_{N-i}^{-1}}=-\omega_{N+1}^{-1}$, which
being combined with the last equality gives
$$
P_{\cE}{\bf L}^{*\alpha j}=\omega_{N+1}^{-1}P_{\cE}{\bf T}^{*\alpha }T_j^*=\omega_{N+1}^{-1}P_{\cE}{\bf T}^{*\alpha j},
$$
and the induction argument completes the proof.
\end{proof}
\begin{proposition}
For the tuples ${\bf T}$ and ${\bf L}$ as above,
\begin{equation}
\cX_0:=\bigcap_{N\ge 0}\bigvee_{\alpha\in\free:|\alpha|=N}{\bf T}^\alpha\cX=
\bigcap_{\alpha\in\free}{\rm Ker} \, P_{\cE}{\bf L}^{*\alpha}=
\bigcap_{\alpha\in\free}{\rm Ker} \, P_{\cE}{\bf T}^{*\alpha}.
\label{ol9}\end{equation}
\label{P:ol2}
\end{proposition}

\begin{proof}
Since $P_{\cE}=I_{\cX}-T_1 L_1^*-  \cdots - T_d L_d^*$, it follows that for any $N\ge 0$,
$$
\sum_{\alpha\in\free:|\alpha|\le N-1}{\bf T}^{\alpha^\top}P_{\cE}{\bf L}^{*\alpha}=I_{\cX}-
\sum_{\alpha\in\free:|\alpha|=N}{\bf T}^{\alpha^\top}{\bf L}^{*\alpha}.
$$
Hence, if $x\in\cX$ belongs to ${\rm Ker} \, P_{\cE}{\bf L}^{*\alpha}$ for all $\alpha\in\free$, then
for each $N\ge 0$,
\begin{align}
x&=\sum_{\alpha\in\free:|\alpha|\le N-1}{\bf T}^{\alpha^\top}P_{\cE}{\bf L}^{*\alpha}x+
\sum_{\alpha\in\free:|\alpha|=N}{\bf T}^{\alpha^\top}{\bf L}^{*\alpha}x\notag\\
&=\sum_{\alpha\in\free:|\alpha|=N}{\bf T}^{\alpha}{\bf L}^{*\alpha^\top}x\in\bigvee_{\alpha\in\free:|\alpha|=N}{\bf T}^\alpha\cX,
\label{ol9a}
\end{align}
which verifies the inclusion
\begin{equation}
\bigcap_{\alpha\in\free}{\rm Ker} \, P_{\cE}{\bf L}^{*\alpha}\subseteq \cX_0.
\label{ol10}
\end{equation}
For the converse inclusion, take arbitrary $x\in\cX_0$ and  $\alpha\in\free$. In particular, $x$ can be represented as
$$
x=\sum_{\beta\in\free:|\beta|=|\alpha|+1}{\bf T}^{\beta^\top} x_\beta\quad \mbox{for some}\quad x_\beta\in\cX.
$$
As we have seen in the proof of Proposition \ref{P:ol1}, the inequality
${\bf L}^{*\alpha}{\bf T}^{\beta^\top}\neq 0$ may occur only if $\beta=j\alpha$. Hence, from the latter representation for $x$ we conclude
that
$$
P_{\cE}{\bf L}^{*\alpha}x=P_{\cE}\bigg(\sum_{j=1}^d T_j x_{j\alpha}\bigg) =0
$$
for any $\alpha\in\free$ confirming the reverse inclusion in \eqref{ol10} and hence, verifying the second equality in \eqref{ol9}.
The third equality holds due to \eqref{ol7}.
\end{proof}
\begin{proposition}
Let ${\bf T}$ be a tuple of left-invertible operators satisfying conditions \eqref{ol2} and let $\cE$ and $\cX_0$ be defined as in
\eqref{ol4}. Then $\cX_0$ reduces ${\bf T}$ (i.e., $\cX_0$ is invariant under ${\bf T}$ and ${\bf T}^*$) and furthermore,
\begin{equation}
\cX_0\oplus \cX_s
=\cX \text{ where } \cX_s : = \bigvee_{\alpha \in \free} \bT^\alpha \cE.
\label{ol11}
\end{equation}
\label{P:ol4}
\end{proposition}
\begin{proof} If ${\bf L}$ is the Cauchy dual of ${\bf T}$, then ${\bf T}$ is the Cauchy dual of ${\bf L}$. Then it follows
from representations \eqref{ol9} that $\cX_0$ can be represented as
\begin{equation}
\cX_0=\bigcap_{N\ge 0}\bigvee_{\alpha\in\free:|\alpha|=N}{\bf L}^\alpha\cX.
\label{ol12}
\end{equation}
Let $x$ be an arbitrary vector in $\cX_0$. Then for any $N\ge 1$, we can represent $x$ as
$$
x=\sum_{\alpha \in \free:|\alpha|=N}{\bf L}^\alpha x_\alpha\quad\mbox{for some}\quad x_\alpha \in \cX.
$$
Due to relations \eqref{ol6} and the assumption that $|\alpha| \ge 1$, for a fixed letter $j$, the inequality
$T_j^*{\bf L}^{\bf \alpha}\neq 0$ holds only if $\alpha=j \beta$.
Applying $T_j^*$ to the vector $x$ as above, we therefore have
$$
T_j^*x=\sum_{\alpha\in\free:|\alpha|=N}T_j^*L_j{\bf L}^\beta x_{j\beta}=\sum_{\beta\in\free:|\beta|=N-1}{\bf L}^\beta x_{j\beta}.
$$
We now conclude that $T_j^*x$ belongs to $\bigvee_{\alpha\in\free:|\beta|=N-1}{\bf L}^\alpha\cX$ for every $N\ge 1$ and therefore,
it belongs to $\cX_0$, due to representation \eqref{ol12}. Thus, $\cX_0$ is ${\bf T}^*$-invariant. Its ${\bf T}$-invariance
follows from the very definition \eqref{ol4} of $\cX_0$.

\smallskip

To prove \eqref{ol12}, note that $x\in\cX$ is orthogonal to ${\bT}^\alpha\cE$ for all $\alpha\in\free$ if and only if
${\bT}^{*\alpha}x$ is orthogonal to $\cE$ for all $\alpha\in\free$, i.e., that $P_{\cE}{\bT}^{*\alpha}x=0$ for
all $\alpha\in\free$. By \eqref{ol9}, the latter means that $x\in\cX_0$.
\end{proof}

It is straightforward to check that the operator-theoretic properties defining the class $\cC(\bo)$ are invariant upon
restriction to a reducing subspace; hence if $\bT \in \cC(\bo)$ on $\cX$ with $\cX$ having the decomposition
$\cX_0 \oplus \cX_s$ 
as in \eqref{ol11}, then the restricted operator-tuples
$\bT_0 : = \bT|_{\cX_0}$ and $\bT_s:= \bT|_{\cX_s}$ are again in the class $\cC(\bo)$. 
This discussion suggests that we introduce two subclasses of the class $\cC(\bo)$, namely:
\begin{equation}
\cC_0(\bo) = \{ \bT \in \cC(\bo) \colon \cX_0 = \cX\}\quad\mbox{and}\quad 
\cC_s(\bo) = \{ \bT \in \cC(\bo) \colon \cX_s= \cX \},
\label{subcl}
\end{equation}
where the subscript $s$ is to suggest a {\em shift} operator tuple.
From the definitions and Proposition \ref{P:ol4} it follows that the 
general element $\bT$ in $\cC(\bo)$ has the form $\bT  = \bT_s \oplus \bT_0$ where 
$\bT_s \in \cC_s(\bo)$ and
$\bT_0 \in \cC_0(\bo)$.  Our next goal is to get a more intrinsic characterization of each of the 
classes $\cC_0(\bo)$ and $\cC_s(\bo)$.

\begin{theorem}   \label{T:olof1}
The operator-tuple $\bT = (T_1, \dots, T_d)$ in the class $\cC(\bo)$ is in the subclass $\cC_0(\bo)$ if and only if
the subspace $\bigoplus_{j=1}^d T_j \cX$ is the whole space $\cX$.  The class $\cC_0(\bo)$ can be equivalently characterized as consisting of
all operator-tuples $T = (T_1, \dots, T_d)$ of left-invertible operators with orthogonal ranges such that $\bT^*$ is an $\bo$-isometry, i.e., such that  
$$
   \Gamma_{\bo, \bT^*}[I_\cX] = 0.
$$
\end{theorem}

\begin{proof} By \eqref{ol11} and the definition \eqref{subcl}, a tuple $\bT\in\cC(\bo)$ is in the subclass $\cC_0(\bo)$ exactly when
the subspace $\cE:  = \cX \ominus \left( \bigoplus_{j=1}^d T_j \cX \right)$ is trivial, which is the same as to say that
$\cX = \bigoplus_{j=1}^d  T_j \cX$.  From the identity \eqref{ol3} we also see that $\cE = \{0\}$ implies that $\Gamma_{\bo, \bT^*}[I_\cX] = 0$,
i.e., $\bT^*$ is a $\bo$-isometry. Hence any operator-tuple $\bT$ in $\cC(\bo)$ meets the criteria in the second characterization.

\smallskip

Conversely, suppose that $\bT$ is a operator-tuple of left-invertible operators with pair-wise orthogonal ranges 
 such that $\Gamma_{\bo, \bT^*}[I_\cX] = 0$.  To show that $\bT$ is in the class $\cC_0(\bo)$, by the first characterization of
 $\cC_0(\bo)$ it remains only to show that 
 \begin{enumerate}
 \item[(i)]  $\bigvee_{j=1, \dots, d} \operatorname{Ran} T_j = \cX$, and 
 \item[(ii)] the identity \eqref{ol2} holds for $j=1, \dots, d$.   
 \end{enumerate}
 The hypothesis that $\Gamma_{\bo, \bT^*}[I_\cX] = 0$ combined with identity \eqref{ol2a'} gives us
$$
  I_\cX = \sum_{j=1}^d T_j \Gamma^{(1)}_{\bo, \bT^*}[I_\cX] T_j^*.
$$
From this identity we read off that indeed item (i) holds. 
Multiply this same identity on the left and on the right by the projection $P_{T_k \cX}$ onto $T_k \cX$ and use the pairwise-orthogonality of the ranges
of the operators $T_k$ to deduce that 
$$
   I_{\operatorname{Ran} T_k} = T_k \Gamma^{(1)}_{\bo, \bT^*}[I_\cX] T_k^* |_{\operatorname{Ran} T_k}\quad\mbox{for}\quad
k = 1, \dots, d.
$$
Now use that the operators 
$T_k$ is invertible as an operator from $\cX$ to $\operatorname{Ran} T_k$,  that
$T_k^*|_{\operatorname{Ran} T_k}$ is invertible as an operator from $\operatorname{Ran} T_k$ to $\cX$, 
and that $\Gamma^{(1)}_{\bo, \bT^*}[I_\cX]$ is an operator on $\cX$.  
Multiplying the last identity on the left by $T_k^{-1}$ (viewed as an operator
from $\operatorname{Ran} T_k$ to $\cX$) and on the right
by $T_k^{*-1}$ (viewed as an operator from $\cX$ to $\operatorname{Ran} T_k$) leads us to the identity
$$
  T_k^{-1} T_k^{*-1} =  \Gamma^{(1)}_{\bo, \bT^*}[I_\cX].
$$
It is now a matter of verifying that  
$$
T_k^{-1} T_k^{*-1} (T_k^* T_k) = T_k^{-1}  (T_k^{*-1} T_k^*) T_k = T_k^{-1} (I_{\operatorname{Ran} T_k} )T_k
= T_k^{-1} T_k = I_\cX
$$
to conclude that  item (ii) holds as well.
\end{proof}

Now we will take a closer look at the second component in the orthogonal decomposition \eqref{ol11}.
It turns out that the observability gramian
$$
\cO_{P_{\cE},{\bf L}^*}: \; x\to P_{\cE}\big(I_{\cX}-z_1L_1^*-\ldots-z_dL_d^*\big)^{-1}x
$$
and the $\bo$-observability gramian
\begin{equation}   \label{ol14}
\cO_{\bo, P_{\cE},{\bf T}^*}: \; x\to P_{\cE}R_{\bo}(z_1T_1*+\ldots+z_dT_d)x=
\sum_{\alpha\in\free}\omega_{|\alpha|}^{-1}P_{\cE}{\bf T}^{*\alpha}x z^\alpha
\end{equation}
are equal. Indeed, combining \eqref{ol7} and \eqref{1.6g} gives
$$
\cO_{P_{\cE},{\bf L}^*}x=\sum_{\alpha\in\free}P_{\cE}{\bf L}^{*\alpha}x z^\alpha=
\sum_{\alpha\in\free}\omega_{|\alpha|}^{-1}P_{\cE}{\bf T}^{*\alpha}x z^\alpha=\cO_{\bo, P_{\cE},{\bf T}^*}x.
$$
It follows from \eqref{ol14} and \eqref{ol9} that
\begin{equation}
{\rm Ker} \, \cO_{\bo, P_{\cE},{\bf T}^*}=
\bigcap_{\alpha\in\free}{\rm Ker} \, P_{\cE}{\bf T}^{*\alpha}=\cX_0,
\label{ol16}
\end{equation}
so the pair $(P_{\cE}, {\bf T}^*)$ is observable if and only if $\cX_0=\{0\}$.

\begin{proposition}
Let ${\bf T}$ be a tuple of left-invertible operators satisfying conditions \eqref{ol2}
and let ${\bf L}$ be its Cauchy dual. Then
\begin{enumerate}
\item The following intertwining relations hold for $j=1,\ldots,d$:
\begin{equation}
S_{\bo,R,j}\cO_{\bo, P_{\cE},{\bf T}^*}=\cO_{\bo, P_{\cE},{\bf T}^*}T_j,\quad
S_{{\bf 1},R,j}^*\cO_{\bo, P_{\cE},{\bf T}^*}=\cO_{\bo, P_{\cE},{\bf T}^*}L_j^*.
\label{ol15}
\end{equation}
\item $\cO_{\bo, P_{\cE},{\bf T}^*}$ is a partial isometry from $\cX$ into the Hardy-Fock space
$H^2_{\bo,\cE}(\free)$ with initial space equal to  $\bigvee_{\alpha \in \free} {\bf T}^\alpha \cE$.
\end{enumerate}
\label{P:ol5}
\end{proposition}

\begin{proof} Recalling that ${\bf L}^{*\alpha}T_j\neq 0$ only if $\alpha=\beta j$, we have
\begin{align*}
\cO_{\bo, P_{\cE},{\bf T}^*}T_jx&=\cO_{P_{\cE},{\bf L}^*}T_jx=\sum_{\alpha\in\free} P_{\cE}{\bf L}^{*\alpha}T_jxz^\alpha
=\sum_{\beta\in\free} P_{\cE}{\bf L}^{*\beta}L_j^*T_jxz^{\beta j}\\
&=\sum_{\beta\in\free} P_{\cE}{\bf L}^{*\beta}xz^{\beta j}
=\left(\cO_{P_{\cE},{\bf L}^*}x\right)z_j= S_{\bo,R,j}\cO_{\bo, P_{\cE},{\bf T}^*}x.
\end{align*}
Making use of the explicit formula \eqref{jul21a} for $\bS_{{\bf 1},R,j}^*$, we get
\begin{align*}
S_{{\bf 1},R,j}^*\cO_{\bo, P_{\cE},{\bf T}^*}x&=S_{{\bf 1},R,j}^*\cO_{P_{\cE},{\bf L}^*}x
\\&=\sum_{\alpha\in\free} P_{\cE}{\bf L}^{*\alpha j}xz^\alpha=\cO_{P_{\cE},{\bf L}^*}L_j^*x=\cO_{\bo, P_{\cE},{\bf T}^*}L_j^*x,
\end{align*}
which completes the proof of \eqref{ol15}.

\smallskip

By \eqref{ol16}, $\cO_{\bo, P_{\cE},{\bf T}^*}: \, \cX_0\to 0$. Due to \eqref{ol11}, statement (2) in Proposition \ref{P:ol5}
follows once we show that $\cO_{\bo, P_{\cE},{\bf T}^*}$
maps
$\cX \ominus \cX_0=\bigvee_{\alpha\in\free}{\bf T}^{\alpha}\cE$
isometrically into $H^2_{\bo,\cE}(\free)$. To this end, take $x\in \cX\ominus\cX_0$ which is of the form
\begin{equation}
x=\sum_{\alpha\in\free \colon |\alpha| \le M}{\bf T}^{\alpha}x_\alpha, \quad x_\alpha\in\cE,
\label{ol18}
\end{equation}
for some $M < \infty$. Then we have
\begin{align}
\|x\|^2_{\cX}=\sum_{\alpha,\beta\in\free}\langle {\bf T}^{\alpha}x_\alpha, \, {\bf T}^{\beta}x_\beta \rangle_{\cX}
&=\sum_{\alpha,\beta\in\free}\langle {\bf T}^{*\beta^\top}{\bf T}^{\alpha}x_\alpha, \, x_\beta \rangle_{\cX}\notag\\
&=\sum_{\alpha,\beta\in\free}\langle P_{\cE}{\bf T}^{*\beta^\top}{\bf T}^{\alpha}x_\alpha, \, x_\beta \rangle_{\cX}\notag\\
&=\sum_{\alpha,\beta\in\free}\omega_{|\beta|}\langle P_{\cE}{\bf L}^{*\beta^\top}{\bf T}^{\alpha}x_\alpha, \, x_\beta \rangle_{\cX},
\label{ol19}
\end{align}
where the two last equalities hold  due to \eqref{ol7} since $x_\beta\in\cE$.
As we have seen in the proof of Proposition \ref{P:ol1}, the inequality  $P_{\cE}{\bf L}^{*\beta^\top}{\bf T}^{\alpha}\neq 0$ may
occur only if $\beta=\alpha\gamma$ for some $\gamma\in\free$, in which case
$$
P_{\cE}{\bf L}^{*\beta^\top}{\bf T}^{\alpha}x_\alpha=P_{\cE}{\bf L}^{*\gamma^\top}x_\alpha.
$$
Since  $\cE \perp \operatorname{Ran} T_j$ for all $j\in\{1,\ldots,d\}$ we have for any $e \in \cE$ and $x \in \cX$,
$$
\langle L_j^* e, \, x \rangle = \langle e, \, L_j x \rangle = \langle e, \, T_j (T_j^* T_j)^{-1} x \rangle = 0
$$
and therefore, $L_j^*|_\cE = 0$  (the same computation shows that $T_j^*|_\cE = 0$ as well)
for $j= 1, \dots, d$.  Hence, if $\gamma \neq \emptyset$, then ${\bf L}^{*\gamma^\top}x_\alpha=0$.
 Therefore, all the terms on the right side of \eqref{ol19} with $\alpha\neq \beta$ are equal zero, and hence,
\begin{align}
\|x\|^2_{\cX}&=\sum_{\alpha\in\free}\omega_{|\alpha|}
\langle P_{\cE}{\bf L}^{*\alpha^\top}{\bf T}^{\alpha}x_\alpha, \, x_\alpha \rangle_{\cX}\notag\\
&=\sum_{\alpha\in\free}\omega_{|\alpha|}\langle P_{\cE}x_\alpha, \, x_\alpha \rangle_{\cX}=
\sum_{\alpha\in\free}\omega_{|\alpha|}\|x_\alpha\|^2_{\cX}.\label{ol17}
\end{align}
On the other hand, since $T_j^*|_\cE = 0$, we have for any $e\in\cE$
$$
\cO_{\bo, P_{\cE},{\bf T}^*}e=\sum_{\alpha\in\free}\omega_{|\alpha|}^{-1}P_{\cE}{\bf T}^{*\alpha}e z^\alpha=
\omega_0^{-1}P_{\cE}e=e.
$$ 
Combining this latter observation with the first intertwining relation in \eqref{ol15}, we have for $x$ of the form \eqref{ol18},
$$
\cO_{\bo, P_{\cE},{\bf T}^*}x=\sum_{\alpha\in\free}x_\alpha z^{\alpha^\top}
$$
and hence,
$$
\|\cO_{\bo, P_{\cE},{\bf T}^*}x\|_{H^2_{\bo,\cE}(\free)}^2=\sum_{\alpha\in\free}\omega_{|\alpha|}\cdot\|x_\alpha\|^2_{\cX}.
$$
Comparing the latter equality with \eqref{ol17} we conclude that $\|\cO_{\bo, P_{\cE},{\bf T}^*}x\|_{H^2_{\bo,\cE}(\free)}=\|x\|_{\cX}$
for all $x \in \cX \ominus \cX_0$ of the form \eqref{ol18}.  As the set of all such elements form a dense subset of $\cX \ominus \cX_0$
(by \eqref{ol11}),  we  conclude by continuity that $\cO_{\bo, P_\cE, \bT^*}$ is isometric on all of  $\cX \ominus \cX_0$.
\end{proof}

\begin{theorem}  \label{T:ol1}  The operator tuple $\bT = (T_1, \dots, T_d)$ in the class $\cD(\bo)$ acting on the Hilbert space $\cX$
is in the subclass $\cC_s(\bo)$ if and only if $\bT$ satisfies any one of the following additional conditions:
\begin{enumerate}
\item $\bigcap_{N \ge 0} \bigvee_{\alpha \in \free \colon |\alpha| = N} \bT^\alpha \cX = \{0\}$;

\item $\bigvee_{\alpha \in \free} \bT^\alpha \cE = \cX$ where $\cE = \cX \ominus \left( \bigoplus_{1 \le j \le d} T_j \cX \right)$;

\item There is a coefficient Hilbert space $\cE$ so that $\bT$ is unitarily equivalent to the model shift operator-tuple
$\bS_{\bo, R}$ acting on $H^2_{\bo, \cE}(\free)$.
\end{enumerate}
\end{theorem}

\begin{proof} 
By the general decomposition \eqref{ol11} and the definition \eqref{subcl}, the tuple $\bT \in \cC(\bo)$ is in the subclass 
$\bT \in \cC_s(\bo)$ if and only if $\cX_s=\cX$, which is the same as condition (2). 
That this is equivalent to condition (1) is seen from the first formula for $\cX_0$ in \eqref{ol9}.
Next note that by Proposition \ref{P:ol5}, $\bT|_{\cX_s}$ is unitarily equivalent to $\bS_{\bo, R}$.
\end{proof}

In the case $\bo = \b1$ and $d=1$,  the class $\cC(\b1)$ consists of the isometries and  the decomposition \eqref{ol11} 
for $\bT$ is  referred to as the  {\em Wold decomposition} for $\bT$ (see \cite{rosenrov}).
The next result gives an indication of the extent to which this criterion 
generalizes to the general setting $\cC_s(\bo)$-operator tuples.

\begin{theorem}  Suppose that $\bT = (T_1, \dots, T_d)$ is in the class $\cC(\bo)$ for some admissible weight $\bo$.

\begin{enumerate}
\item  If $\bT$ is in the class $\cC_s(\bo)$, then $\bT^*$ is strongly stable.

\item  Suppose that $\| \begin{bmatrix} T_1 & \cdots & T_d \end{bmatrix} \| \le 1$ and ${\bf T} \in \cC_0(\bo)$.
Then for every $x \in \cX$, it is the case that
$$
   \liminf_{N \to \infty} \sum_{\alpha \in \free \colon |\alpha| = N} \| \bL^{*\alpha^\top} x \|^2 > 0,
$$
 where $\bL$ is the Cauchy dual of $\bT$. \end{enumerate}
\end{theorem}

\begin{proof}  \textbf{(1)}  If $\bT$ is in $\cC_s(\bo)$, then $\bT$ is unitarily equivalent to $\bS_{\bo, R}$ on some weighted Hardy space
$H^2_{\bo, \cE}(\free)$.  We have seen that $\bS_{\bo, R}^*$ is strongly stable (item (1) in Proposition \ref{P:bo-model}).
Hence $\bT^*$ is strongly stable.  In particular $\bL^*$ (rather than $\bT^*$) is not strongly stable.

\smallskip

\textbf{(2)}  
Suppose $\bT$ is in $\cC_0(\bo)$, so $\cX_s = \{0\}$.  Choose $0 \ne x \in \cX_0$.
Then  $x$ belongs to ${\rm Ker} \, P_{\cE}{\bf L}^{*\alpha}$ for all $\alpha\in\free$ (by characterization \eqref{ol9}) and then
the computation \eqref{ol9a} shows that 
$$
x =\sum_{\alpha \colon |\alpha| = N} \bT^\alpha \bL^{* \alpha^\top} x
 $$
for all $N = 0,1,2,\dots$. As $\bT$ is a row contraction, it follows that each $T_j$ is a contraction and hence, for
$\alpha = i_N \cdots i_1 \in \free$, 
$\bT^\alpha = T_{i_N} \cdots T_{i_1}$ is a contraction, and we can compute
$$
\| x \|^2 = \sum_{\alpha \in \free \colon |\alpha| = N} \| \bT^\alpha \bL^{* \alpha^\top}  x \|^2
\le \sum_{\alpha \in \free \colon |\alpha| = N} \| \bL^{* \alpha^\top} x \|^2.
$$
It follows that
$$
\liminf_{N \to \infty} \sum_{\alpha \in \free \colon |\alpha| = N} \| \bL^{* \alpha^\top} x \|^2 \ge \| x \|^2 > 0
$$
 In particular, when $\bT \in \cC(\bo_0)$, it is  $\bL^*$  (rather than $\bT^*$)
that  is guaranteed to be not strongly stable.\end{proof}

\begin{remark}  In case $\bo = \b1$,  it is automatic that  $T$ is a contraction and that the Cauchy dual $L$ of $T$ is equal to $T$.
Thus Theorem \ref{T:ol1} recovers the result for the case $\bo = \b1$, $d=1$ mentioned immediately preceding the theorem.  Let us also mention
that it is known that, in the single-variable case ($d=1$), there are many surjective $n$-isometries which are not contractions once $n \ge 3$
(see \cite{AS1, AS2, AS3}).
\end{remark}

\section{Observability-operator range spaces}  \label{S:NC-Obs-g}
The main character of this section is the range of the $\bo$-observability operator
$$
\operatorname{Ran}{\mathcal O}_{\bo,C,\bA} =
\{ CR_\bo(Z(z)A)x \colon \; x \in {\mathcal X}\}.
$$
We recall that $R_\bo$ is the power series \eqref{1.6g} and 
$Z(z)$ and $A$ are defined as in \eqref{1.23pre}.
Throughout this section, we assume that the weight sequence $\bo$ is admissible and meets 
the condition \eqref{1.8g}. This material fleshes out theme \#1 mentioned at the end of Section
\ref{S:Overview} for the general $\bo$-setting.

\begin{theorem}
\label{T:2-1.2}
Suppose that $(C,\bA)$ is a $\bo$-output-stable pair. Then:

\smallskip

\noindent
{\rm (1)} The intertwining relation
\begin{equation}
\label{4.8aga}
S_{\bo,R,j}^{*} {\mathcal O}_{\bo,C,\bA}x ={\mathcal O}_{\bo,C, \bA} A_{j} x\quad (x\in\cX)
\end{equation}
holds for every backward-shift operator $S_{\bo,R,j}^*$ defined in \eqref{Sboj*}
and hence $\operatorname{Ran}{\mathcal O}_{\bo,C, \bA}$
is ${\bf S}_{\bo,R}^*$-invariant.

\smallskip

\noindent
{\rm (2)}  Let $H\in\cL(\cX)$ be a solution of the system \eqref{3.15g}, \eqref{3.16g}
and let ${\mathcal X}'$ be the completion of ${\mathcal X}$ with
$H$-inner product $\| [x]\|_{{\mathcal X}'}^{2} = \langle H x, \, x\rangle_{{\mathcal X}}$.
Then $A_{j}$ and $C$ extend to define bounded operators
$$
A'_{j} \colon {\mathcal X}'\to{\mathcal X}'\; (j = 1, \dots, d)\quad\mbox{and}\quad C'
\colon {\mathcal X}' \to \cY,
$$
so that the $\bo$-observability operator 
${{\mathcal O}}_{\bo,C', {\mathbf A}'}: \, {\mathcal X}'\to H^2_{\bo,\cY}(\free)$ is a contraction. 

\smallskip

\noindent
{\rm (3)} If $H$ is subject to  \eqref{3.15g}, \eqref{3.16g}
and the linear manifold ${\mathcal M}:=\operatorname{Ran}\,
{\mathcal O}_{\bo,C,{\bf A}}$ is given  the lifted norm
\begin{equation}  \label{M-norm'}
  \|{\mathcal O}_{\bo,C,{\bf A}} x \|_{{\mathcal M}}^{2} =
{\displaystyle\inf_{y \in {\mathcal X} \colon{\mathcal O}_{\bo,C,{\mathbf A}}y ={\mathcal O}_{\bo,C,
{\mathbf A}}x} \langle H y, y \rangle_{{\mathcal X}}},
\end{equation}
then:
\begin{itemize}
           \item[(a)]
       ${\mathcal M}$ can be completed to ${\mathcal M}'=
\operatorname{Ran} {\mathcal O}_{\bo,C',{\bf A}'}$ with
contractive inclusion in $H^2_{\bo,\cY}(\free)$:
\begin{equation}   \label{3a-ineq}
\| f \|^{2}_{H^2_{\bo,\cY}(\free)} \le
\| f \|^{2}_{{\mathcal M}'}\quad\text{for all} \quad f\in {\mathcal M}'.
\end{equation}
Furthermore, ${\mathcal M}'$ is isometrically equal to the FNRKHS
with reproducing kernel
\begin{equation}   \label{NCkernel'}
    K_{C', \bA'}(z, \zeta)= C' R_\bo(Z(z) A')R_\bo(Z(\zeta) A')^*C^{\prime *}.
\end{equation}
In case $H$ is invertible, the reproducing kernel can be written
directly in terms of $(C, \bA)$:
\begin{equation}  \label{NCkernel}
K_{C, \bA, H}(z,\zeta) = C R_\bo(Z(z) A)H^{-1}R_\bo(Z(\zeta) A)^*C^*.
\end{equation}

\item[(b)] The restriction $(E\vert_{\cM},{\bf S}^*_{\bo,R}\vert_{\cM})$ of the $\bo$-model
output pair \eqref{modelE} to $\cM$ is $\bo$-contractive. Moreover, it is 
$\bo$-isometric if and only if \eqref{3.17g} holds. 
\end{itemize}
\smallskip

\noindent
{\rm (4)} Conversely, let ${\mathcal M}$ be a Hilbert space
included in $H^2_{\bo,\cY}(\free)$ (not necessarily isometrically or even contractively) such that
\begin{enumerate}
    \item[(i)] $\cM$ is invariant under the backward-shift tuple ${\bf S}_{\bo,R}^*$,
    \item[(ii)] the pair $(E_\cM, \bS^*_{\bo, R}|_\cM)$ is an $\bo$-contractive output pair, i.e., the inequalities
\begin{align}  \label{ineqs}
& \sum_{k=1}^d \| S^*_{\bo, R, k} f \|^2_\cM \le \| f \|^2_\cM, \quad
\sum_{\alpha \in \free} c_{|\alpha|} \|  \bS_{\bo, R}^{* \alpha} f  \|^2_\cM  \ge  \| f_\emptyset \|^2, \\
& \sum_{\alpha \in \free} \bigg( \sum_{\ell = 1}^k \frac{ c_{|\alpha| + \ell}}{\omega_{k-\ell} } \bigg)
 \| \bS_{\bo, R}^{* \alpha} f \|^2_\cM \ge 0 \quad\text{for all}\quad k \ge 1
 \notag 
\end{align}
hold  for all $f \in \cM$.
 \end{enumerate}
Then it follows that $\cM$ is contractively included in $H^2_{\bo,\cY}(\free)$ and
there exists an $\bo$-contractive pair $(C,\bA)$ such that
$$
{\mathcal M} = {\mathcal H}(K_{C,{\bf A},I}) = \operatorname{Ran}{\mathcal O}_{\bo,C,{\bf A}}
$$
isometrically.  One can take $(C, \bA)$ to be an $\bo$-isometric output pair if and only if the second of the 
inequalities \eqref{ineqs} holds with equality.
For example, one may take $\cX = \cM$, $C=E\vert_\cM$ and $\bA={\bf S}^*_{\bo,R}\vert_{\cM}$.
\end{theorem}
\begin{proof}[Proof of (1):] Making use of \eqref{Sboj*} and \eqref{18.2aaa} we get for any $j\in\{1,\ldots,d\}$,
\begin{align*}
S_{\bo,R,j}^*{\mathcal O}_{\bo,C, \bA}x&=S_{\bo,R,j}^*\sum_{\alpha\in\free}\omega_{|\alpha|}^{-1}
(C\bA^{\alpha}x)z^\alpha\\
&=\sum_{\alpha\in\free}\omega_{|\alpha|}^{-1}(C\bA^{\alpha j}x)z^\alpha={\mathcal O}_{\bo,C,\bA}A_jx
\end{align*}
which proves \eqref{4.8aga}.\end{proof}
\begin{proof}[Proof of (2):]
The Stein system  \eqref{3.15g}, \eqref{3.16g} amounts to the statement that $(C, {\mathbf A})$ is $\bo$-contractive
and well-defined on the dense subset $[{\mathcal X}]$ of ${\mathcal X}'$ (where $[x]$ is the
equivalence class containing $x$) and hence  extends to a $\bo$-contractive pair $(C', {\mathbf A}')$ on all of
${\mathcal X}'$. By part (2) in Theorem \ref{T:7.2}, $\cG_{\bo,C,\bA} \preceq H$  and therefore,
\begin{equation}
\| \cO_{\bo,C,\bA}x\|^{2}_{H^2_{\bo,\cY}(\free)} = \langle \cG_{\bo,C,\bA}x,\, x\rangle_{\cX} \le
 \langle H x, \, x \rangle_{\cX},
\quad\mbox{for all}\quad x\in\cX,
\label{again6}
\end{equation}
so $\cO_{\bo,C,\bA}$ is contractive from $\cX$ (with the $H$-pseudo-inner product) to $H^2_{\bo,\cY}(\free)$ and hence also
${\mathcal O}_{\bo,C',{\mathbf A}'}$ is contractive from ${\mathcal X}'$ to $H^2_{\bo,\cY}(\free)$.\end{proof}
\begin{proof}[Proof of (3a):]  The definition of the $\cM$-norm makes
    $\cO_{\bo,C,\bA}$ a partial isometry from $[\cX]$ (with the
    $\cM$-norm \eqref{M-norm'}) onto $\cM$.  In case $\cM$ is not
    complete in its norm, we work instead with $\cO_{\bo,C', \bA'}$
    which is a partial isometry from $\cX'$ onto
    $\cM' = {\rm Ran} \cO_{\bo, C', \bA'}$. It suffices to verify the
    inequality \eqref{3a-ineq} for the special case where $f =
    \cO_{\bo, C, \bA} x$ for some $x \in \cX$. In this case, \eqref{3a-ineq} amounts to   
   $\| \cO_{\bo, C,\bA} x \|^{2} \le \langle H x, x \rangle_{\cX}$, which holds true, by \eqref{again6}.

\smallskip

To identify the formal reproducing kernel for $\cM'$, we  compute, for $x\in\cX'$,
    \begin{align*}
& \langle (\cO_{\bo, C', \bA'}x)(\zeta), y \rangle_{\cY\langle \langle \zeta \rangle \rangle \times \cY}  = 
\langle C'R_\bo(Z(\zeta) A')x, \,  y \rangle_{\cY\langle \langle \zeta \rangle \rangle \times \cY}  \\
 & \quad = \langle x, \, R_\bo(Z(\zeta) A')^*C^{\prime *} y \rangle_{\cX' \times \cX'\langle \langle \bzeta \rangle \rangle}  \\
 & \quad  = \langle  (\cO_{\bo, C', \bA'}x)(z), C'R_\bo(Z(z) A')R_\bo(Z(\zeta) A')^*C^{\prime *} y \rangle_{\cM'
 \times \cM'\langle \langle \bzeta \rangle \rangle}.
 \end{align*}
 This verifies that the kernel \eqref{NCkernel'} is the formal reproducing kernel for $\cM'$.
If $H$ is invertible, then $\cX = \cX'$ as sets and we can simply quote Proposition
\ref{P:principle} to   conclude that $\cM$ has formal reproducing
kernel $K_{C,\bA,H}(z, \zeta)$ defined as in \eqref{NCkernel}.
\end{proof}

 \begin{proof}[Proof of (3b):]
It follows from the definition of the range norm and from the intertwining relations
\eqref{4.8aga} that for $f$ of the form $f=\cO_{\bo,C,\bA}x$,
$$
\|f\|^2_{\cM}=\langle Hx, \, x\rangle_{\cX}\quad\mbox{and}\quad
S_{\bo,R,j}^*f=\cO_{\bo,C,\bA}A_jx\quad(j=1,\ldots,d).
$$
Therefore,
\begin{align}
\|f\|^2_{\cM}-\sum_{j=1}^d\|S_{\bo,R,j}^*f\|^2_{\cM}
&=\langle Hx, \, x\rangle_{\cX}-\sum_{j=1}^d \langle HA_jx, \, A_jx\rangle_{\cX}\notag\\
&=\langle (H-\sum_{j=1}^d A_j^*HA_j)x, \; x\rangle_{\cX}.\label{dec1}
\end{align}
Due to the first relation in \eqref{3.15g}, the expression on the right side of \eqref{dec1}
is nonnegative for every $x\in\cX$. Therefore, the expression  on the left side of
\eqref{dec1} is nonnegative for every $f\in\cM$ which means that ${\bf S}^*_{\bo,R}\vert_{\cM}$
is a contractive tuple. Iterating \eqref{4.8aga} gives
\begin{equation}
{\bf S}_{\bo,R}^{*\alpha} {\mathcal O}_{\bo,C, {\mathbf A}}x ={\mathcal
O}_{\bo,C, {\mathbf A}} \bA^{\alpha^\top} x \quad \text{for all}\quad x\in\cX
\text{ and } \alpha\in\free.
\label{dec2}
\end{equation}
For $f=\cO_{\bo,C,\bA}x$ we then have
\begin{equation}
\| {\bf S}_{\bo,R}^{*\alpha} f \|^{2}_{\cM} =
\langle H{\bf A}^{\alpha^\top} x, {\bf A}^{\alpha^\top} x \rangle_{\cX}\quad\mbox{and}\quad
f_{\emptyset} = Ef=C x.
\label{again7}
\end{equation}
With these substitutions, we see that 
\begin{align*}
\big\langle \Gamma_{\bo,{\bf S}^*_{\bo,R}\vert_{\cM}}[I_{\cM}]f, \, f\big\rangle_{\cM}&=
\sum_{\alpha\in\free}c_{|\alpha|}\|{\bf S}^{*\alpha}_{\bo,R}f\|^2_{\cM} \text{ (by definition \eqref{1.11g})}\notag\\
&=\sum_{\alpha\in\free}c_{|\alpha|}\langle H{\bf A}^{\alpha^\top} x, {\bf A}^{\alpha^\top} x \rangle_{\cX}
\text{ (by \eqref{again7})}\notag\\
&=\bigg\langle \sum_{\alpha\in\free}c_{|\alpha|}\bA^{*\alpha^\top}H\bA^{\alpha}x, \, x\bigg\rangle_{\cX}
\text{ (by substitution $\alpha\mapsto \alpha^\top$)}\notag\\ 
&=\big\langle\Gamma_{\bo,\bA}[H]x, \, x\big\rangle_{\cM}\text{ (by definition \eqref{1.11g})},
\end{align*}
(the substitution $\alpha\mapsto \alpha^\top$ is justified by the fact that the latter sum is taken over {\em all} 
elements of $\free$ and that $|\alpha|=|\alpha^\top|$)
and subsequently, in view of the second equality in \eqref{again7} 
\begin{align}
\big\langle \Gamma_{\bo,{\bf S}^*_{\bo,R}\vert_{\cM}}[I_{\cM}]f, \, f\big\rangle_{\cM}-\|Ef\|^2_{\cY}&=
\big\langle\Gamma_{\bo,\bA}[H]x, \, x\big\rangle_{\cM}-\|Cx\|^2_{\cY}\notag\\
&=\big\langle (\Gamma_{\bo,\bA}[H]-C^*C)x, \, x\big\rangle_{\cX}.
\label{again8}
\end{align}
A similar computation relying on the definition \eqref{3.5ag} and relations \eqref{again7} shows that
for any fixed $k\ge 1$, 
\begin{align}
\big\langle\Gamma^{(k)}_{\bo,{\bf S}^*_{\bo,R}\vert_{\cM}}[I_{\cM}]f, \, f\big\rangle_{\cM}&=
-\sum_{\alpha\in\free}\bigg(\sum_{\ell=1}^{k}\frac{c_{|\alpha|+\ell}}{\omega_{k-\ell}}\bigg)
\|{\bf S}^{*\alpha}_{\bo,R}f\|^2_{\cM}\notag\\
&=-\sum_{\alpha\in\free}\bigg(\sum_{\ell=1}^{k}\frac{c_{|\alpha|+\ell}}{\omega_{k-\ell}}\bigg)
\langle H{\bf A}^{\alpha^\top} x, {\bf A}^{\alpha^\top} x \rangle_{\cX}\notag\\
&=\bigg\langle -\sum_{\alpha\in\free}\bigg(\sum_{\ell=1}^{k}\frac{c_{|\alpha|+\ell}}{\omega_{k-\ell}}\bigg)
\bA^{*\alpha}H\bA^{\alpha^\top}x, \, x\bigg\rangle_{\cX}\notag\\
&=\bigg\langle -\sum_{\alpha\in\free}\bigg(\sum_{\ell=1}^{k}\frac{c_{|\alpha|+\ell}}{\omega_{k-\ell}}\bigg)  
\bA^{*\alpha^\top}H\bA^{\alpha}x, \, x\bigg\rangle_{\cX}\notag\\
&=\big\langle\Gamma^{(k)}_{\bo,\bA}[H]x, \, x\big\rangle_{\cX}.
\label{again9}
\end{align}
Since $H$ satisfies inequalities \eqref{3.15g}, \eqref{3.16g}, the right hand side expressions
 in \eqref{again8} and \eqref{again9}
are nonnegative for all $x\in\cX$. Hence, the left sides of
\eqref{again8} and \eqref{again9} are nonnegative for all $f\in\cM$ meaning that 
${\bf S}^*_{\bo,R}\vert_{\cM}$ is an $\bo$-hypercontraction
and that the pair $(E\vert_{\cM},{\bf S}^*_{\bo,R}\vert_{\cM})$ is $\bo$-contractive. Equality \eqref{3.17g}
 is equivalent to both sides 
in \eqref{again8} vanish, i.e., that the pair $(E\vert_{\cM},{\bf S}^*_{\bo,R}\vert_{\cM})$ is $\bo$-isometric.
This completes the verification of part (3).
\end{proof}

\begin{proof}[Proof of (4):]
Suppose that ${\mathcal M}$ is a Hilbert space included in $H^2_{\bo,\cY}(\free)$ which is
invariant under ${\bf S}_{\bo,R}^*$ and satisfies the inequalities \eqref{ineqs} for all $f \in \cM$.
Set ${\mathcal X} ={\mathcal M}$ and let $C=E\vert_{\cM}$ and ${\bf A}={\bf S}^*_{\bo,R}\vert_{\cM}$. 
The import of conditions \eqref{ineqs} is that then the output pair 
$(C,\bA)$ is a $\bo$-contractive output pair (recall Definitions \ref{D:7.5} and \ref{D:7.3}).  Since ${\mathcal O}_{\bo,C, {\mathbf A}}= I_{{\mathcal M}}$ by 
the first part of statement (2) in Proposition \ref{P:bo-model} (a purely algebraic statement independent 
of the choice of norm on $\cM$), it follows that for each $f \in {\mathcal M}$ we have
$$
\| f \|^2_{H^2_\cY(\free)} = \| \cO_{\bo, C, \bA} f \|^2_{H^2_\cY(\free)}
= \langle \cG_{\bo, C , \bA} f, f \rangle_\cM \le \| f \|^2_\cM
$$
where the last step follows from part (1) of Lemma \ref{L:7.6} since we have already observed that
$(C, \bA)$ is a $\bo$-contractive output pair.  We conclude that $\cM$ is contractively included
in $H^2_{\bo, \cY}(\free)$ as claimed.

\smallskip

Since $\cO_{\bo,C, \bA}$ is just the identity operator on $\cM$, we also have
$$
   \| \cO_{\bo,C, \bA} f \|^2_\cM = \| f \|^2_\cM.
$$
Thus we can view $\cM$ as $\cM = \operatorname{Ran}\, \cO_{\bo,C, \bA}$ (with $\cO_{\bo,C, \bA}$ viewed as an operator
from $\cM$ into $H^2_{\bo, \cY}(\free)$).  We can then  follow the same argument as used in the proof of part (3)
of Theorem \ref{T:2-1.2} to identify the reproducing kernel $K_\cM$ for $\cM$ as $K_\cM(z, \zeta) =
K_{C, \bA, I}(z, \zeta)$, as claimed.
 \end{proof}
         
\begin{definition}  \label{D:modsub}
Let us call a subspace $\cM$ contained in $H^2_{\bo,\cY}(\free)$ (not necessarily isometrically or even
contractively)  an {\em $\bo$-model subspace} if it is  ${\bf S}_{n,R}^*$-invariant 
and the restricted $\bo$-model pair $(E\vert_{\cM},{\bf S}^*_{\bo,R}\vert_{\cM})$ is $\bo$-contractive. 
\end{definition}

As explained by part (4) of Theorem \ref{T:2-1.2}, for purposes of study of contractively included $\bo$-model subspaces of
$H^2_{\bo,\cY}(\free)$ without loss of generality we may suppose at the start that we are working with
${\mathcal X}'$ as the original state space ${\mathcal X}$ and with
the solution $H$ of Stein inequalities \eqref{3.15g}, \eqref{3.16g} to be
normalized to $H = I_{{\mathcal X}}$. Then certain simplifications
occur in parts (1)-(3) of Theorem \ref{T:2-1.2} as explained in the next result.  Note that part (4) of Theorem \ref{T:2-1.2}
does not involve a choice of $H$ in the hypotheses and indeed gets $H = I_\cX$ as a conclusion.

         \begin{theorem}  \label{T:2-1.2'}
Suppose that $(C, {\mathbf A})$ is an $\bo$-contractive pair with state space ${\mathcal X}$ and output
         space $\cY$.  Then:

\smallskip
\noindent
{\rm (1)}  $(C, {\mathbf A})$ is $\bo$-output-stable and  $\operatorname{Ran}\cO_{\bo, C,{\mathbf A}}$ is
${\bf S}_{\bo,R}^*$-invariant (by \eqref{4.8aga}).

\smallskip
\noindent
{\rm (2)}  The operator ${\mathcal O}_{\bo,C,{\mathbf A}}$ is a
         contraction from ${\mathcal X}$ into $H^2_{\bo,\cY}(\free)$.
         
         \smallskip
         
\noindent
 {\rm (3)} If the linear manifold ${\mathcal M}:= \operatorname{Ran}
         {\mathcal O}_{\bo,C,{\mathbf A}}$ is given the lifted norm
         \begin{equation}  \label{lifted'}
           \| {\mathcal O}_{\bo,C, {\mathbf A}}x\|_{{\mathcal M}} = \| Q x \|_{{\mathcal X}}
         \end{equation}
         where $Q$ is the orthogonal projection of ${\mathcal X}$ onto
         $(\operatorname{Ker}{\mathcal O}_{\bo,C, {\mathbf A}})^{\perp}$, then
\begin{enumerate}
\item[(a)]  ${\mathcal O}_{\bo,C, {\mathbf A}}$ is a coisometry of ${\mathcal X}$ onto ${\mathcal M}$, and implements a unitary equivalence
between $\bS_{\bo, R}^*|_\cM$ and $Q \bA|_{\operatorname{Ran} Q}$.

\item[(b)] ${\mathcal M}$ is contained contractively in $H^2_{\bo,\cY}(\free)$
         and is isometrically equal to the FNRKHS with
         reproducing kernel given by
         $$
           K_{C, {\mathbf A}}(z,\zeta) = CR_\bo(Z(z)A)R_\bo(Z(\zeta)A)^*C^*.
         $$
         
\item[(c)] The pair $(E\vert_{\cM},{\bf S}^*_{\bo,R}\vert_{\cM})$ is $\bo$-contractive and 
the orthogonal projection $Q$ in \eqref{lifted'} satisfies Stein inequalities \eqref{3.15g}, \eqref{3.16g}.
Moreover, $(E\vert_{\cM},{\bf S}^*_{\bo,R}\vert_{\cM})$ is $\bo$-isometric
         if and only if $Q$ satisfies \eqref{3.17g}.   
         
\item[(d)] In particular, if $(C,A)$ is $\bo$-observable, then $\bA$ in unitarily equivalent to $\bS^*_{\bo,R}|_\cM$ and $(C, \bA)$
is an $\bo$-contractive pair.  If $Q = I$ satisfies \eqref{3.17g}, then $(C, \bA)$ is an $\bo$-isometric output pair.
                   
\item[(e)]  If $(C, \bA)$ and $(\widetilde C, \widetilde \bA)$ are two
      $\bo$-output-stable, observable pairs (with $C\in\cL(\cX,\cY)$ and $\widetilde{C}\in\cL(\widetilde{\cX},\cY)$
realizing the same positive kernel
       \begin{align}
       K_{C, \bA}(z,\zeta)&:=CR_\bo(Z(z)A)R_\bo(Z(\zeta)A)^*C^* \notag \\
       &=\widetilde CR_\bo(Z(z)\widetilde A)R_\bo(Z(\zeta)\widetilde A)^*\widetilde C^{*}=: K_{\widetilde
       C,\widetilde\bA}(z,\zeta),   \label{ker-id}
       \end{align}
then $(C,\bA)$ and $(\widetilde C,\widetilde \bA)$ are {\em unitarily equivalent}, i.e.,
equalities \eqref{sep1} hold for a {\em unitary} operator $T \colon \cX \to \widetilde \cX$.
\end{enumerate}

\end{theorem}
         
\begin{proof}  As for parts (1), (2), (3a), (3b), (3c) of the theorem,  all statements but the last part of statement (3c) 
concerning $Q$ are direct specializations to the case $H =I_{{\mathcal X}}$ of the corresponding results
in Theorem \ref{T:2-1.2}. To complete the verification of the last part of (3c), observe  from the intertwining relations 
\eqref{4.8aga} and \eqref{dec2} that inequalities
\begin{align}
&\sum_{j=1}^d \|S_{\bo,R,j}^*f\|_{\cM}^2\le \|f\|^2_{\cM},\quad 
\sum_{\alpha\in\free}c_{|\alpha|}\|{\bf S}^{*\alpha}_{\bo,R}f\|^2_{\cM}\ge \|f_\emptyset\|^2_{\cY},\label{again11}\\
&-\sum_{\alpha\in\free}\bigg(\sum_{\ell=1}^{k}\frac{c_{|\alpha|+\ell}}{\omega_{k-\ell}}\bigg)
\|{\bf S}^{*\alpha}_{\bo,R}f\|^2_{\cM}  \ge   0 \quad \mbox{for all}\quad f\in\cM \text{ and }  k\ge 1\notag
\end{align}
explicitly stating the $\bo$-contractivity of the pair $(E\vert_{\cM},{\bf S}^*_{\bo,R}\vert_{\cM})$
for a generic element $f = {\mathcal O}_{\bo,C,\bA}x \in\cM$ mean that
\begin{align*}
\sum_{j=1}^d\|{\mathcal O}_{\bo,C,\bA}A_jx\|^2_{\cM}\le \|{\mathcal O}_{\bo,C,\bA}x\|^2_{\cM},\quad
\sum_{\alpha \in\free} c_{|\alpha|}\cdot
\|{\mathcal O}_{\bo,C,\bA}\bA^{\alpha^\top}x \|^{2}_{\cM}\ge \|Cx\|^2_{\cY},\\
\sum_{\alpha\in\free}\bigg(\sum_{\ell=1}^{k}\frac{c_{|\alpha|+\ell}}{\omega_{k-\ell}}\bigg)
\|{\mathcal O}_{\bo,C,\bA}\bA^{\alpha^\top}x \|^{2}_{\cM}\le 0\quad \mbox{for all}\quad x\in\cX \text{ and }  k\ge 1.
\end{align*}
By definition \eqref{lifted'} of the $\cM$-norm, the latter relations can be
written as
\begin{align}
\|Qx\|^2_{\cM}&\ge \sum_{j=1}^d\|QA_jx\|^2_{\cM},\notag\\
\|Cx\|^2_{\cY}&\le \sum_{\alpha \in\free} c_{|\alpha|}\cdot\|Q\bA^{\alpha^\top}x \|^{2}_{\cM}=
\sum_{\alpha \in\free} c_{|\alpha|}\cdot\|Q\bA^{\alpha}x \|^{2}_{\cM},\label{again13}\\
0&\ge \sum_{\alpha\in\free}\bigg(\sum_{\ell=1}^{k}\frac{c_{|\alpha|+\ell}}{\omega_{k-\ell}}\bigg)
\|Q\bA^{\alpha^\top}x \|^{2}_{\cM}=
\sum_{\alpha\in\free}\bigg(\sum_{\ell=1}^{k}\frac{c_{|\alpha|+\ell}}{\omega_{k-\ell}}\bigg)
\|Q\bA^{\alpha}x \|^{2}_{\cM},\notag
\end{align}
where we used the substitution $\alpha\mapsto \alpha^\top$ in the two last formulas. 
Since $Q=Q^{\frac{1}{2}}$ and since $x\in\cX$ is arbitrary, the latter relations can be written as 
$$
Q \succeq \sum_{j=1}^d A_j^*QA_j,\quad \Gamma_{\bo,\bA}[Q]\succeq C^*C, \quad 
\Gamma^{(k)}_{\bo,\bA}[Q]\succeq 0 \; \; (k\ge 1)
$$
telling us that $Q$ satisfies Stein inequalities \eqref{3.15g}, \eqref{3.16g}. 
Note next that the pair $(E\vert_{\cM},{\bf S}^*_{\bo,R}\vert_{\cM})$ being $\bo$-isometric is equivalent to 
equality in the second relation in \eqref{again11}
(for all $f\in\cM$) which is equivalent to equality in \eqref{again13} (for all $x\in\cX$), which in turn is
 equivalent to the equality
$\Gamma_{\bo,\bA}[Q]=C^*C$, which completes the proof of parts (1)-(3c) of the theorem.
Statement (3d) amounts to the specialization of the last part of (3a) to the case where $Q = I_\cX$.

\smallskip

As for statement (3e), suppose that the output pairs $(C, \bA)$ and $\widetilde C, \widetilde \bA)$ generate the same kernels
as in \eqref{ker-id}.  Equating coefficients of $z^\alpha \overline{\zeta}^{\alpha^\top}$ gives us the system
of equations
$$
\omega_{|\alpha|}^{-1} \omega_{|\beta|}^{-1}  C \bA^\alpha \bA^{* \beta^\top} C^* =
 \omega_{|\alpha|}^{-1}  \omega_{|\beta|}^{-1}       \widetilde C \widetilde \bA^{\alpha} 
 \widetilde \bA^{*\beta^\top} \widetilde C^*\quad
\text{for all}\quad \alpha, \, \beta \in \free, 
$$
or more simply, after cancellation of the common factor  $\omega_{|\alpha|}^{-1} \omega_{|\beta|}^{-1}$,
$$
   C \bA^\alpha \bA^{* \beta^\top} C^*  =   \widetilde C \widetilde \bA^{\alpha} 
 \widetilde \bA^{*\beta^\top} \widetilde C^*\quad
\text{for all}\quad \alpha, \, \beta \in \free.
$$
We conclude that the operator $U$ defined by
\begin{equation}  \label{def-U}
    U \colon \bA^{* \beta^\top} C^* y \mapsto \widetilde \bA^{* \beta^\top} \widetilde C^* y
\end{equation}
extends by linearity and continuity to an isometry from its domain space
$$
 \cD_U = \bigvee  \{ \bA^{* \beta^\top} C^* y \colon \beta \in \free, \, y \in \cY\}
$$
onto its range space
$$
 \cR_U = \bigvee  \{ \widetilde \bA^{* \beta^\top} \widetilde C^* y 
 \colon \beta \in \free, \, y \in \cY\}.
 $$
 The observability assumptions imply that $\cD_U$ is all of $\cX$ and $\cR_U$ is all of $\widetilde \cX$, and
 hence $U \colon \cX \to \widetilde \cX$ is unitary.    From the formula \eqref{def-U} we can read off
 the intertwining relations
 $$
 U C^* = \widetilde C^*, \quad U A_j^* = \widetilde A_j^* U \quad\text{for}\quad j = 1, \dots, d.
 $$
 Since $U$ is unitary we then also get
 $$
 \widetilde C U = C, \quad \widetilde A_j U  = U A_j\quad \text{for}\quad j = 1, \dots, d
 $$
 and we conclude the pairs $(C, \bA)$ and $(\widetilde C,  \widetilde \bA)$ are unitarily equivalent
 as claimed.
 
\smallskip

 We remark that this proof is essentially the same as that of Theorem 2.13 in \cite{BBF1} where the special
 case $\bo = \bmu_1$ is handled.
\end{proof}

We now turn to ${\bf S}_{\bo,R}^*$-invariant subspaces of $H^2_{\bo,\cY}(\free)$ that are 
{\em isometrically} included in $H^2_{\bo,\cY}(\free)$.

\begin{theorem} \label{T:1.2g}
If the pair $(C,\bA)$ is an $\bo$-isometric pair with $\bA$ strongly $\bo$-stable, then the 
observability operator $\cO_{\bo, C, \bA} \colon \cX \to H^2_{\bo,\cY}(\free)$ is an isometry onto a backward-shift-invariant
subspace $\cN \subset H^2_{\bo,\cY}(\free)$ and $\bA$ is unitarily equivalent to $\bS_{\bo, R}^*|_\cN$. Moreover
$\cN$ is the NFRKHS with reproducing kernel
\begin{equation}   \label{again1}
  K_{\bo, C, \bA}(z, \zeta) = C R_\bo(Z(z) A) R_\bo(Z(\zeta) A)^* C^*.
\end{equation}
Conversely, if $\cN$ is a Hilbert space isometrically contained in $H^2_{\bo, \cY}(\free)$ which is invariant
under $\bS^*_{\bo, R}$, then there is an $\bo$-isometric output pair $(C, \bA)$ with $\bA$ strongly $\bo$-stable
so that $\cN = \operatorname{Ran}  \cO_{\bo,C, \bA}$ and $\bA$ is unitarily equivalent to $\bS_{\bo, R}^*|_\cN$, and
the space $\cN$ is isometrically equal to the NFRKHS with reproducing kernel as in \eqref{again1}.
In fact, one can choose
\begin{equation} \label{again1-model}
   C = E|_\cN, \quad \bA = \bS_{\bo, R}^*|_\cN,
	\end{equation}
where $E$ is the model output map on $H^2_{\bo, \cY}(\free)$ given by 
$\; E \colon \sum_{\alpha \in \free} f_\alpha z^\alpha \mapsto f_\emptyset$.
\end{theorem}

\begin{proof} 
Assume that $(C, \bA)$ is an $\bo$-isometric output pair with $\bA$ strongly $\bo$-stable.  Then part (2) of Lemma \ref{L:7.6} 
tells us that $\cO_{\bo, C, \bA} \colon \cX \to H^2_{\bo, \cY}(\free)$ is isometric, and hence
$\cO_{\bo, C, \bA}$ can be viewed as a unitary operator from $\cX$ onto its range spaces
$\cN: = \operatorname{Ran} \cO_{\bo, C, \bA}$.  From the general intertwining relation \eqref{4.8aga}
we see that $\cN$ is $\bS_{\bo, R}^*$-invariant and that $\bA$ is unitarily equivalent to $\bS^*_{\bo, R}|_\cN$
via the unitary transformation $\cO_{\bo, C, \bA} \colon \cX \to \cN$.  The fact that $\cN$ has reproducing kernel
of the form \eqref{again1} is a direct consequence of Proposition \ref{P:principle}.

\smallskip

Conversely, suppose that $\cN$ is a backward-shift-invariant subspace of $H^2_{\bo, \cY}(\free)$ isometrically
contained in $H^2_{\bo, \cY}(\free)$.  It suffices to check that $(C, \bA)$ as in \eqref{again1-model} meets all
the requirements of the theorem.  Proposition \ref{P:bo-model} part (2) tells us that
$\cO_{E, \bS^*_{\bo, R}} = I_{H^2_{\bo, \cY}(\free)}$ and hence $\cO_{C, \bA} =
\cO_{E, \bS^*_{\bo, R}}|_\cN = I_\cN$ is an isometry from $\cX = \cN$ to $\cN \subset H^2_{\bo, \cY}(\free)$.
To show that $\bA$ is a $\bo$-strongly stable $\bo$-hypercontraction,
it suffices to note that $\bS^*_{\bo, R}$ is (by Proposition \ref{P:bo-model2}), and then observe 
that restriction of $\bS^*_{\bo, R}$ to an invariant
subspace preserves these properties.  Similarly, by Proposition \ref{P:bo-model2},
the pair $(E, \bS^*_{\bo, R})$ is $\bo$-isometric;  it then suffices to observe that the $\bo$-isometric property 
is preserved upon restriction of $(E, \bS^*_{\bo,R})$ to a $\bS^*_{\bo, R}$-invariant subspace.  
\end{proof}
\begin{theorem} \label{T:beta-stablemodel}  Suppose that the Hilbert-space operator tuple $\bA \in \cL(\cX)^d$ is an
$\bo$-strongly stable $\bo$-hypercontraction.
Let $\cY$ be a coefficient Hilbert
    space with $\operatorname{dim} \cY = \operatorname{rank}
    \Gamma_{\bo,\bA}[I_{\cX}]$.  Then there is a subspace $\cN
    \subset H^{2}_{\bo,\cY}(\free)$ invariant under ${\bf S}_{\bo,R}^{*}$
    so that $\bA$ is unitarily equivalent to $\bS_{\bo,R}^{*}\vert_{\cN}$.
\end{theorem}

\begin{proof} Since $\bA$ is $\bo$-hypercontractive, $\Gamma_{\bo, \bA}[I_{\cX}]$ is positive semidefinite.
We choose $C \in \cL(\cX, \cY)$ so that $C^{*}C =
    \Gamma_{\bo, \bA}[I_{\cX}]$.  Then $(C,\bA)$ is an $\bo$-isometric pair.
Since $\bA$ is $\bo$-strongly stable, the operator $\cO_{\bo,C,\bA} \colon \cX \to H^{2}_{\bo,\cY}(\free)$ is
isometric, by Lemma \ref{L:7.6}. We set $\cN = \operatorname{Ran} \cO_{\bo, C,\bA}
    \subset H^{2}_{\bo,\cY}(\free)$.  Then the intertwining property
    \eqref{4.8aga} leads to the conclusion that
    $\bA$ is unitarily equivalent to $\bS_{\bo,R}^{*}\vert_{\cN}$ via the
    unitary similarity transformation  $\cO_{\bo,C,\bA} \colon \cX \to \cN$.
\end{proof}

\begin{remark}
Since unitary similarity preserves stability properties of operator tuples,
as a consequence of Theorem \ref{T:beta-stablemodel} combined with the fact that
$\bS_{\bo,R}^{*}$ is strongly stable (by part (1) in Proposition \ref{P:bo-model}), we conclude
that $\bo$-strong stability of an $\bo$-hypercontraction $\bA$ implies its strong stability in the usual sense
\eqref{bAstable}.   The converse direction (strong stability in the usual sense implies of $\bo$-strong stability
for a $\bo$-hypercontraction) is true at least for the special case where $\bo = \bmu_n$ for some $n \in {\mathbb N}$
(see Remark \ref{R:stabmu} below).
\label{R:stab}
\end{remark}

\subsection{Shifted $\bo$-observability operator range spaces}  
 For our subsequent constructions, we will also need the range spaces 
associated with shifted $\bo$-observability operators \eqref{4.31}:
$$
\operatorname{Ran}{\Ob}_{\bo,k,C,\bA}=\{ CR_{\bo,k}(Z(z)A)x \colon \; x \in {\mathcal X}\}.
$$
We first observe a factorization of the shifted $\bo$-gramian \eqref{4.32} in terms of the corresponding 
shifted $\bo$-observability operator.
\begin{proposition} \label{R:2.12}
Let $(C,\bA)$ be a $\bo$-output-stable pair and let
${\Ob}_{\bo,k,C,\bA}$ and ${\Gr}_{\bo,k,C,\bA}$ be defined as in \eqref{4.31}, \eqref{4.32}. Then
\begin{equation}
\|{\bf S}_{\bo,R}^{\alpha\top}\Ob_{\bo,|\alpha|,C,\bA} x\|^2_{H^2_{\bo,\cY}(\free)}=\left\langle
\Gr_{\bo,|\alpha|,C,\bA}x, \, x\right\rangle_{\cX}
\label{st10}
\end{equation}
for every $\alpha\in\free$ and $x\in\cX$, and hence we have the
operator factorization
$$    \Gr_{\bo,|\alpha|,C,\bA} = \Ob_{\bo,|\alpha|,C,\bA}^{*}
    \bS_{\bo,R}^{*\alpha}
    \bS_{\bo,R}^{\alpha^{\top}} \Ob_{\bo,|\alpha|,C,\bA}.
$$
\end{proposition}

\begin{proof}
By definition \eqref{4.31}, we have
$$
{\bf S}_{\bo,R}^{\alpha^\top}{\Ob}_{\bo,|\alpha|,C,\bA}x=
\sum_{\alpha^{\prime}\in\free}\omega_{|\alpha^\prime|+|\alpha|}^{-1}
(C\bA^{\alpha^\prime}x) \, z^{\alpha^\prime \alpha}.
$$
Making use of the latter formula and invoking the definition of the inner product in $H^2_{\bo,\cY}(\free)$, 
we verify \eqref{st10} as follows: 
\begin{align}
\|{\bf S}_{\bo,R}^{\alpha^\top}{\Ob}_{\bo,|\alpha|,C,\bA}x\|^2_{H^2_{\bo,\cY}(\free)}
&=\sum_{\alpha^\prime\in\free}\omega_{|\alpha^\prime|+|\alpha|}^{-1}\big\langle
C\bA^{\alpha^\prime}x, \, C\bA^{\alpha^\prime }x\big\rangle_{\cX}\notag\\
&=\bigg\langle \sum_{\alpha^\prime\in\free} \omega_{|\alpha^\prime|+|\alpha|}^{-1}
\bA^{*\alpha^{\prime\top}}C^*C\bA^{\alpha^\prime}x, \, x \bigg\rangle_{\cX}\notag\\
&=\left\langle \Gr_{\bo,|\alpha|,C,\bA} x, \, x\right\rangle_{\cX},\notag
\end{align}
where the last equality is clear from the definition \eqref{4.32} of $\Gr_{\bo,|\alpha|,C,\bA}$.
\end{proof}

We next represent the range of a $k$-shifted observability operator as a noncommutative formal
reproducing kernel Hilbert space.

\begin{proposition}  \label{T:RanOb=NFRKHS}
Suppose that the pair $(C,\bA)$ is a $\bo$-output-stable and exactly $\bo$-observable.
Then, for any $\beta \in \free$ the subspace 
\begin{equation}
{\mathbf S}^{\beta^{\top}}_{\bo, R} 
\operatorname{Ran} \Ob_{\bo,|\beta|,C,\bA}\subset H^2_{\bo, \cY}(\free)
\label{again4}
\end{equation}
(with inner product induced by $H^2_{\bo, \cY}(\free)$) is a NFRKHS
   with reproducing kernel $\boldsymbol{\mathfrak K}_{\beta}(z, \zeta)$ given by
   \begin{equation}   \label{deffrakk}
\boldsymbol{\mathfrak{K}}_{\beta}(z,\zeta)=
CR_{\bo,|\beta|}(Z(z)A) \big(z^{\beta} \bzeta^{\beta^{\top}}
\Gr_{\bo,|\beta|,C,\bA}^{-1}\big)  R_{\bo,|\beta|}(Z(\zeta)A)^*C^* .
\end{equation}
 \end{proposition}
\begin{proof}
Due to equality \eqref{st10}, a direct application of Proposition \ref{P:principle} 
shows that the formal reproducing kernel for the space \eqref{again4} is given by
$$
( S^{\beta^{\top}}_{\bo, R} \Ob_{\bo, |\beta|, C, \bA})(z)
\Gr_{\bo,|\beta|,C,\bA}^{-1} \big(( S^{\beta^{\top}}_{\bo,R}
\Ob_{\bo,|\beta|,C,A})(\zeta) \big)^{*}.
$$
It is now straightforward to check that this expression agrees exactly with
$\boldsymbol{\mathfrak K}_{\beta}(z, \zeta)$ given by \eqref{deffrakk}.
\end{proof}

\subsection{Range spaces of $n$-observability operators}  \label{S:n-obs}
In this section we will specialize the previous general results to the case $\bo=\bmu_n$, 
where certain simplifications occur. The only result that will be established independently 
(and, presumably, does not hold for general weights) is the characterization of ${\bf S}_{n,R}^*$-invariant
spaces isometrically included in $\cA_{n,\cY}(\free)$. Specializing part (4) of Theorem \ref{T:2-1.2}
to the present case, we see that without loss of generality we may start with the solution $H$ of 
Stein inequalities \eqref{4.20} to be normalized to $H = I_{{\mathcal X}}$.   Recall from 
\cite[Definition 4.7]{BBIEOT} that we say that the output pair $(C,\bA)$ is {\em $n$-contractive} if
(i) $\bA$ is $n$-hypercontractive and in addition $\Gamma_{n,\bA}[I_\cX] \succeq C^*C$.

        \begin{theorem}  \label{T:2-1.2'n}
Suppose that $(C, {\mathbf A})$ is an $n$-contractive pair with state space ${\mathcal X}$ and output space $\cY$.  
Then:

\smallskip

$(1) \;$ $(C, {\mathbf A})$ is $n$-output-stable and $\operatorname{Ran}\cO_{n, C,{\mathbf A}}$ is ${\bf S}_{n,R}^*$-invariant.

\smallskip

$(2) \;$ The operator ${\mathcal O}_{n,C,{\mathbf A}}$ is a
         contraction from ${\mathcal X}$ into $\cA_{n,\cY}(\free)$. Moreover, ${\mathcal O}_{n,C, {\mathbf A}}: \, \cX\to \cA_{n,\cY}(\free)$
is an isometry if and only if $(C, {\mathbf A})$ is an $n$-isometric pair and ${\mathbf A}$ is strongly stable.

\smallskip

$(3) \; $ If the manifold ${\mathcal M}:= \operatorname{Ran}
         {\mathcal O}_{n,C,{\mathbf A}}$ is given the lifted norm
         $\| {\mathcal O}_{n,C, {\mathbf A}}x\|_{{\mathcal M}} =
           \| Q x \|_{{\mathcal X}}$
         where $Q$ is the orthogonal projection of ${\mathcal X}$ onto
         $(\operatorname{Ker}{\mathcal O}_{n,C, {\mathbf A}})^{\perp}$, then
\begin{itemize}
\item[(a)] ${\mathcal M}$ is contained contractively in $\cA_{n,\cY}(\free)$
         and is isometrically equal to the FNRKHS with
         reproducing kernel 
         $$
           K_{C, {\mathbf A}}(z,\zeta) = C (I - Z(z) A)^{-n} (I -
           A^{*}Z(\zeta)^{*})^{-n} C^{*}.
         $$
\item[(b)] The $d$-tuple ${\bf S}_{n,R}^*\vert_{\cM}$ is contractive on $\cM$ and the inequality
\begin{equation}
\sum_{\alpha\in\free: \, |\alpha|\le n}(-1)^{|\alpha|}\bcs{n\\ |\alpha|}\cdot
\|{\bf S}_{n,R}^{*\alpha}f \|^{2}_{\cM}\ge \|f_{\emptyset}\|^{2}_{\cY}
\label{dif-quot'}
\end{equation}
holds for all $f \in\cM$. Moreover, \eqref{dif-quot'} holds with equality
         if and only the orthogonal projection $Q$ of ${\mathcal X}$ onto
         $(\operatorname{Ker} {\mathcal O}_{n,C, {\mathbf A}})^{\perp}$
       is subject to relations
$$
Q\succeq \sum_{j=1}^d A_j^*QA_j\quad\mbox{and}\quad
\Gamma_{n,\bA}[Q]=C^*C.
$$
\end{itemize}

$(4) \; $ Conversely, let ${\mathcal M}$ be a Hilbert space
included in $\cA_{n,\cY}(\free)$ such that
\begin{enumerate}
    \item[(i)] $\cM$ is invariant under the backward shift ${\bf S}_{n,R}^*$,
    \item[(ii)] the restricted tuple ${\bf S}_{n,R}^*\vert_{\cM}$ is contractive, and
the inequality \eqref{dif-quot'} holds for all $f\in\cM$.
\end{enumerate}
Then it follows that $\cM$ is contractively included in $\cA_{n,\cY}(\free)$ and
there exists an $n$-contractive pair $(C,{\mathbf A})$ such that
$$
{\mathcal M} = {\mathcal H}(K_{C,{\bf A}}) = \operatorname{Ran}{\mathcal O}_{n,C,{\bf A}}
$$
isometrically.  If \eqref{dif-quot'} holds with
equality, then $(C,{\mathbf A})$ can be taken to be an $n$-isometric pair. For example, one may take
$\cX = \cM$, $\bA=\bS^{*}_{n,R}\vert_{\cM}$, $C=E\vert_{\cM}$.
\end{theorem}

\begin{proof}
We will show that ${\mathcal O}_{n,C, {\mathbf A}}: \, \cX\to \cA_{n,\cY}(\free)$
is an isometry if and only if $(C, {\mathbf A})$ is an $n$-isometric pair and ${\mathbf A}$ is strongly stable.
Indeed, since the pair $(C,\bA)$ is $n$-contractive, the operator $H=I_{\cX}$ satisfies 
the Stein system \eqref{4.20}. Thus, equalities \eqref{4.9} hold for all $N\ge 0$ and for $H=I_{\cX}$.
It follows from \eqref{4.9} that operators 
\begin{equation}  \label{3.7u'}
\Lambda_{N,\bA}=\sum_{\alpha \in\free: \, N+1\le |\alpha|\le N+n}\bcs{N+n\\
|\alpha|}\bA^{*\alpha^\top}\Gamma_{n+N-|\alpha|,\bA}[I_{\cX}]\bA^{\alpha}
\end{equation}
satisfy $\Lambda_{N,\bA}\succeq \Lambda_{N+1,\bA}\succeq 0$ for all $N\ge 0$ and therefore, the limit
\begin{equation} \label{3.7u}
\Delta_{\bA} = \lim_{N \to\infty} \Lambda_{N,\bA}\succeq 0
\end{equation}
exists in the strong sense. Since $C^*C\preceq \Gamma_{n,\bA}[I_{\cX}]$ and $B_{\bA}$ is a positive map, we have 
\begin{align*}
\sum_{\alpha \in\free: \, |\alpha|\le N}\bcs{n+|\alpha|-1\\
|\alpha|}\bA^{*\alpha^\top}C^*C\bA^\alpha 
&\preceq\sum_{\alpha \in\free: \, |\alpha|\le N}\bcs{n+|\alpha|-1\\
|\alpha|}\bA^{*\alpha^\top}\Gamma_{n,\bA}[I_{\cX}]\bA^\alpha\\
&=I_{\cX}-\Lambda_{N,\bA},
\end{align*}
where the last equality holds due to \eqref{4.9}. The latter relations are equivalent to 
\begin{align}
0 &\preceq  \sum_{\alpha \in\free: \, |\alpha |\le N}\bcs{n+|\alpha|-1\\
|\alpha|}\bA^{*\alpha^\top}\left(\Gamma_{n,\bA}[I_{\cX}]-C^*C\right)
\bA^v\notag\\
&= I_{\cX}-\Lambda_{N,\bA}-\sum_{\alpha \in\free: \, |\alpha |\le N}\bcs{n+|\alpha|-1\\
|\alpha|}\bA^{*\alpha^\top}C^*C\bA^\alpha.\label{through}
\end{align}
Letting $N\to\infty$ in \eqref{through} we get, on account \eqref{4.3} and \eqref{3.7u},
\begin{equation}
0 \preceq   \sum_{\alpha \in\free}\bcs{n+|\alpha|-1\\
|\alpha|}\bA^{*\alpha^\top}\left(\Gamma_{n,\bA}[I_{\cX}]-C^*C\right)\bA^\alpha
\preceq  I_{\cX} -{\mathcal G}_{n,C, {\mathbf A}} -\Delta_{{\mathbf A}}.
\label{ineq-chain}
\end{equation}
By definition, ${\mathcal O}_{n,C,{\mathbf A}} \colon{\mathcal X}
\to \cA_{n,\cY}(\free)$ being an isometry means
that ${\mathcal G}_{n,C, {\mathbf A}} = I_{\cX}$ in which case \eqref{ineq-chain} implies
$\Delta_{{\mathbf A}} = 0$ and equalities throughout \eqref{ineq-chain}. Since all the
terms on the right side of \eqref{3.7u'} are positive semidefinite, condition
$\Delta_{{\mathbf A}} = 0$ implies, in particular, that
$$
\lim_{N \to \infty}\sum_{\alpha \in\free: \,
|\alpha|=N+n}\bA^{*\alpha^\top}\Gamma_{0,\bA}[I_{\cX}]\bA^\alpha=
\lim_{N \to \infty}\sum_{\alpha \in\free: \,
|\alpha|=N+n}\bA^{*\alpha^\top}\bA^\alpha=0
$$
which just means that ${\mathbf A}$ is strongly stable. Since all
the terms in the series in \eqref{ineq-chain} are
positive semidefinite, we conclude that each term equals zero. The term corresponding to
the index $\alpha=\emptyset$ is $\Gamma_{n,\bA}[I_{\cX}]-C^*C$. Hence $\Gamma_{n,\bA}[I_{\cX}]=C^*C$ and thus
the Stein equation in \eqref{4.21} holds.

\smallskip

Conversely, by reversing the steps of the preceding argument, we see that ${\mathbf A}$ being
strongly stable and the Stein equation in \eqref{4.21} holding leads to ${\mathcal G}_{n,C,
{\mathbf A}} = I_{\cX}$, i.e., to ${\mathcal O}_{n,C, {\mathbf A}}: {\mathcal X}\to \cA_{n,\cY}(\free)$
being an isometry.
\end{proof}

As we know from Remark \ref{R:stab}, a $\bo$-strongly stable $\bo$-hypercontraction is necessarily
strongly stable in the usual sense. We do not know if the converse holds true in general. However,
in the case $\bo=\bmu_n$ it does!

\begin{remark} \label{R:stabmu}  We note here that 
{\em strong stability and ${\boldsymbol\mu}_{n}$-strong stability are equivalent for 
${\boldsymbol\mu}_{n}$-hypercontractions.}  Indeed, one direction (${\boldsymbol\mu}_n$-strong stability implies
the usual strong stability) holds even for a general admissible weight $\bo$ by Remark \ref{R:stab}.
Thus it remains only to verify that any strongly stable $\bmu_n$-hypercontraction $\bA$ 
(which is an $n$-hypercontraction by Proposition \ref{P:equiv}) is also $\bmu_n$-strongly stable. 
To this end, note that 
by part (2) in Theorem \ref{T:2-1.2'n}, the  strong stability of an $n$-hypercontraction $\bA$ 
implies that $\bA$ is unitarily equivalent to the restriction of 
${\bf S}_{n,R}^{*}$ to an invariant subspace  $\cN \subset H^{2}_{{\boldsymbol\mu}_n,\cY}(\free) =\cA_{n,\cY}(\free)$  
for a suitable coefficient Hilbert space $\cY$ 
as in Theorem \ref{T:beta-stablemodel}.  But by Proposition \ref{P:bo-model2}, $\bS_{n,R}^*$ is $\bmu_n$-strongly stable.
Putting all this together, we see that strong stability for a $\bmu_n$-hypercontraction implies
$\bmu_n$-strong stability.
\end{remark}

\begin{remark}  \label{R:1.2g}
For the standard-weight case, Theorem \ref{T:2-1.2} avoided the exact
observability statement in part (2) of the statement and had the
conclusion that the space $\cM$ with lifted norm is contractively
included in the Bergman space $\cA_{n,\cY}(\free)$ rather than being
isometrically included in $\cA_{n,\cY}(\free)$. For this reason Theorem
\ref{T:1.2g} is only a partial $\bo$-analog of Theorem \ref{T:2-1.2}.
\end{remark}

\chapter{Beurling-Lax theorems based on contractive multipliers}  \label{S:BL}
In Section \ref{S:NC-Obs-g} we characterized $S_{\bo,R}^*$-invariant spaces which are contractively or,
more specifically, isometrically included in $H^2_{\bo,\cY}(\free)$. In this and in the next two chapters we focus
on the spaces that are invariant under the forward-shift operator-tuple ${\bf S}_{\bo,R}$, thereby
fleshing out theme \#2 from the end of Section \ref{S:Overview} for the $\bo$-setting. The material from Chapter
\ref{S:Stein} on backward-shift-invariant subspaces and ranges of observability operators is relevant here in
at least two respects:

\smallskip

$\bullet$ Given $\cM$ contractively included in $H^2_{\bo, \cY}(\free)$ which is invariant under the forward
shift-tuple $\bS_{\bo, R}$,  we may set $\bA = (\bS_{\bo, R}|_\cM)^*$ and (under appropriate hypotheses)
choose $C \colon \cM \to \cU$ so that $(C, \bA)$ is an $\bo'$-isometric output pair (for some weight sequence $\bo'$;
interesting special cases include $\bo'=\bo$ and $\bo'=\bmu_1$). If $\bA$ is also $\bo'$-strongly
stable, then the observability operator $\cO_{\bo',C, \bA}$ is isometric from $\cM$ into $H^2_{\bo', \cU}(\free)$
and implements a unitary equivalence between $\bA$ and
$\bS_{\bo', R}^*|_{\operatorname{Ran} \cO_{\bo', C, \bA}}$,  since
$$
S_{\bo', R, j}^* \cO_{\bo',C, \bA} = \cO_{\bo',C, \bA} A_j \text{ for } j=1, \dots, d
$$
(see Theorem \ref{T:1.2g}). For our application here we prefer to view the codomain of 
$\cO_{\bo',C, \bA}$ as all of $H^2_{\bo', C, \bA}$ rather than only $\operatorname{Ran} \cO_{\bo', C, \bA}$.
Taking adjoints in the above intertwining relation
gives 
$$
\cO_{\bo',C, \bA}^* S_{\bo', R, j} =     A_j^*\cO_{\bo',C, \bA}^*,\quad\mbox{where}\quad A_j^* = S_{\bo, R, j}|_\cM.
$$
Thus  $(\cO_{\bo',C, \bA})^*$ intertwines the respective shift tuples forcing it to have the form $M_\Theta$
for $\Theta$ a contractive multiplier from $H^2_{\bo', \cU}(\free)$ to $H^2_{\bo, \cY}(\free)$.
The conclusion is that $\Theta$ so constructed serves as the Beurling-Lax representer for the
forward-shift-invariant subspace $\cM$.  A finer analysis shows that $M_\Theta$ considered as an operator
from $H^2_{\bo', \cU}(\free)$ into $H^2_{\bo, \cY}(\free)$ is a partial isometry in case $\cM$ is isometrically included
in $H^2_{\bo, \cY}(\free)$.
For complete details, see Theorems \ref{T:NC-BL}, \ref{T:NC-BLisom} and \ref{T:NC-BL'} below.

\smallskip

$\bullet$ If $\cM$ is a forward-shift-invariant subspace isometrically contained in $H^2_{\bo, \cY}(\free)$,
then its orthogonal complement $\cM^\perp$ is a backward-shift-invari\-ant subspace isometrically contained
in $H^2_{\bo, \cY}(\free)$ for which results of Chapter \ref{S:Stein} apply.  Specifically, if we choose
$(C, \bA)$ to be an $\bo$-isometric output pair with $\bA$ $\bo$-strongly stable such that
$\cM^\perp = \operatorname{Ran}\, \cO_{\bo,C, \bA}$ as in the converse direction of Theorem \ref{T:1.2g}, then
the reproducing kernel $k_{\cM^\perp}$ for $\cM^\perp$ has the explicit form
$$
k_{\cM^\perp}(z, \zeta) = C R_\bo(Z(z)A) (R_\bo(Z(\zeta) A))^* C^*.
$$
Then by standard calculus of reproducing
kernels, we see that the reproducing kernel for $\cM$ is given by
\begin{equation}
 k_\cM(z, \zeta) = k_{\bo, {\rm nc}}(z, \zeta) \otimes I_\cY - C R_\bo(Z(z) A) R_\bo(Z(\zeta) A)^* C^*.
\label{kMa}
\end{equation}
In the case where $\bo = \bmu_1$ (the Fock-space case), the Cholesky-factorization construction yields
a pair of operators $B \colon \cU \to \cX$ and $D \colon \cY \to \cY$ so that $U:=
\sbm{ A & B \\ C & D}$ is unitary from which it follows that we get the factorization
$$
k_\cM(z, \zeta)  I_\cY = \Theta(z) (k_{\rm Sz,\, nc}(z, \zeta) I_\cU) \Theta(\zeta)^*\quad\mbox{with}\quad
\Theta(z) = D + C (I - z A)^{-1} B.
$$
It then can be shown that $\Theta$ is a strictly inner multiplier which
serves as the Beurling-Lax representer:   $\cM =  \Theta \cdot H^2_\cU(\free)$. In the case 
where $\bo = \bmu_1$ (the standard weighted Bergman-Fock-space case), a certain iterative 
procedure leads to the representation 
$$
k_\cM(z, \zeta) = \sum_{j=1}^{n} F_{j}(z)k_{{\rm nc }, j}(z,\zeta)F_{j}(\zeta)^{*}
$$
with $F_j$ being contractive multipliers from $\cA_{j,\cU_j}(\free)$ to $\cA_{n,\cY}(\free)$ 
giving rise to Beurling-Lax representations for ${\bf S}_{n,R}$-invariant subspaces of 
$\cA_{n,\cY}(\free)$ with model space equal to the direct sum of $\cA_{j,\cU_j}(\free)$ for $j=1,\ldots,n$;
details are given in Section \ref{S:BL3}. Further elaboration of all these ideas will lead us (in Chapter \ref{S:BLorthog} below)
to the state-space realization formulas for the $H^2_{\bo, \cY}(\free)$-Bergman-inner family 
(see Definition \ref{D:innerfuncfam} to come) representing a given forward-shift-invariant subspace of $H^2_{\bo, \cY}(\free)$.

\section[Beurling-Lax theorem via contractive multipliers]{Beurling-Lax theorem via McCT-inner multipliers and 
more general contractive multipliers}  \label{S:NC-appl}

Beurling-Lax representations based on partially isometric multipliers (for the case of isometric inclusion) or more general
contractive multipliers (for the case of contractive inclusion) appeared first in the work of
Arveson \cite{ArvesonIII} and in a more systematic operator-theoretic framework in McCullough-Trent
\cite{MCT} in the context of the Drury-Arveson space (the commutative version of $H^2_{\bmu_1}$).
In this section, we present noncommutative versions of such representations for shift-invariant subspaces in 
a weighted Hardy-Fock space $H^2_{\bo, \cY}(\free)$ for an admissible weight $\bo$ as in \eqref{18.2}.
We shall in fact prove a more general version for contractively included (rather than isometrically 
included) subspaces of $H^2_{\bo,\cY}(\free)$ due in the classical setting to de Branges (see \cite{dBR1})
and appearing in the noncommutative Fock-space setting (i.e., for $n=1$ and $\bo  = \bmu_1$) in \cite{BBF1}.       
           
\begin{theorem} \label{T:NC-BL}
A Hilbert space $\cM$ is such that
\begin{enumerate}
\item $\cM$ is contractively included in $H^2_{\bo,\cY}(\free)$,
\item $\cM$ is ${\bf S}_{\bo,R}$-invariant,
\item the $d$-tuple $\bT =(T_1,\ldots, T_d)$ where $T_j=(S_{\bo,R,j}|_{\cM})$ for
$j=1,\ldots,d$ is a row contraction 
\end{enumerate}
if and only  if there is a coefficient Hilbert space $\cU$ and a
contractive multiplier $\Theta$ from $H^2_{\cU}(\free)$ to $H^2_{\bo,\cY}(\free)$ so that
\begin{equation}
\cM = M_\Theta H^2_{\cU}(\free)
\label{Mform}
\end{equation}
        with lifted norm
        \begin{equation}  \label{M-norm}
        \| M_\Theta  f \|_{\cM} = \| Q f \|_{H^2_{\cU}(\free)}
        \end{equation}
        where $Q$ is the orthogonal projection onto $(\operatorname{Ker} M_{\Theta})^\perp$.
        
 If the subspace $\cM$ is written more compactly as a pullback space
 $\cH^p(\Pi)$ for an operator $\Pi$ on $H^2_{\bo, \cY}(\free)$ satisfying $0 \preceq \Pi \preceq I$, then conditions (1), (2), (3)
 in Theorem \ref{T:NC-BL} can be encoded in the set of operator inequalities
 \begin{equation}   \label{pullback-ineq}
 0 \preceq \Pi \preceq I, \quad \Pi - \sum_{j=1}^d S_{\bo, R, j}  \Pi S_{\bo, R, j}^* \succeq 0.
 \end{equation}
 \end{theorem}
	
\begin{proof}  To verify sufficiency, let us
suppose that $\cM$ is of the form \eqref{Mform} for a
          contractive multiplier $\Theta$ with $\cM$-norm given by
          \eqref{M-norm}. Since $\|M_{\Theta}\| \le 1$, it follows that for any $f \in H^2_{\cU}(\free)$,
\begin{equation}  \label{ineq1}
\|\ M_\Theta f\|_{H^2_{\bo,\cY}(\free)} = \|M_\Theta Q f\|_{H^2_{\bo,\cY}(\free)} \le 
\|Qf\|_{H^2_{\cU}(\free)}=\|M_\Theta f\|_{\cM},
\end{equation}
i.e., (1) holds. Property (2) follows from intertwining relations 
$$
S_{\bo,R,j} M_{\Theta} = M_{\Theta}S_{1,R,j}\quad\mbox{for}\quad j=1,\ldots,d. 
$$
The latter relations also imply $M_\Theta S_{1,R,j}\vert_{\operatorname{Ker}M_\Theta}=0$, 
which can be written equivalently as $QS_{1,R,j}(I-Q)=0$ where $Q$ the orthogonal projection from \eqref{M-norm}.
Thus, 
\begin{equation}  \label{9.2}
Q S_{1,R,j} = Q S_{1,R,j}Q\quad \text{and}\quad S_{1,R,j}^{*} Q=
QS_{1,R,j}^{*}Q\quad\mbox{for}\quad  j=1,\ldots,d.
\end{equation}
Furthermore, for every $f, \, g \in H^2_{\cU}(\free)$ and $j\in\{1,\ldots,d\}$, we have
\begin{align*}
\langle M_\Theta g, \, T_j^* M_\Theta f\rangle_{\cM}=&\langle S_{\bo,R,j} M_\Theta g, \,  M_\Theta f\rangle_{\cM}\\
=&\langle M_\Theta S_{1,R,j} g, \, M_\Theta f\rangle_{\cM}\\
=&\langle Q S_{1,R,j} g, \,
f\rangle_{H^2_{\cU}(\free)}\\
=&\langle Q S_{1,R,j} Qg, \, f\rangle_{H^2_{\cU}(\free)}\\
=&\langle Q g, \, S_{1,R,j}^*Qf\rangle_{H^2_{\cU}(\free)}=
\langle M_\Theta g, \,M_ \Theta S_{1,R,j}^* Qf\rangle_{\cM},
\end{align*}
which implies that $T_j^*  \colon M_\Theta f  \mapsto M_\Theta S_{1,R,j}^* Qf$. Iterating the latter formula 
and making use of relations \eqref{9.2} we get
\begin{equation}
\bT^{* \alpha} \colon  M_\Theta f  \mapsto M_\Theta ({\bf S}_{1,R}^*)^\alpha Qf\quad\mbox{for}\quad \alpha\in\free.
\label{9.4}
\end{equation}  

It then follows that, for $f \in H^2_{\cU}(\free)$,
\begin{align}
\sum_{j=1}^d \| T_j^* M_\Theta f \|^2_\cM & = \sum_{j=1}^d  \| M_\Theta S^*_{1,R,j} Q f \|_{\cM}^2\notag\\
&= \sum_{j=1}^d \| Q S^*_{1,R,j} Q f \|^2_{H^2_\cU(\free)} \notag \\
& \le \sum_{j=1}^d \| S^*_{1,R,j} Q f \|^2_{H^2_{\cU}(\free)} 
\le \| Q f \|^2_{H^2_\cU(\free)}= \| M_\Theta f \|^2_\cM\label{again20}
\end{align}
(inequalities above hold since $Q$ is a projection and $(\bS_{1,R})$ is a row contraction)
which shows that $\bT$ is a row contraction.  This completes the proof of sufficiency.

\smallskip

Suppose now that the Hilbert space $\cM$ satisfies conditions (1), (2), (3) in the
statement of the theorem.  By Theorem \ref{T:charNFRKHS} it follows that $\cM$ is a NFRKHS
(noncommutative formal reproducing kernel Hilbert space)
in its own right.  As $\cM \subset H^2_{\bo, \cY}(\free)$, the elements of $\cM$ 
are formal power series of the form $f(z) = \sum_{\alpha \in \free} f_\alpha z^\alpha$
with $f_\alpha \in \cY$.

\smallskip

Thus $\cM$ is a NFRKHS which is invariant under the right coordinate multipliers $T_j=S_{\bo, R, j}|_\cM$ 
($1\le j\le d$) by assumption 
(2) and moreover, the tuple $\bT = (T_1, \dots, T_d)$ is a row contraction, by assumption (3).
 We are therefore in a position to apply Proposition \ref{P:referee} to conclude that the reproducing kernel 
$k_\cM(z, \zeta)\in \cL(\cY)\langle \langle z, \overline{\zeta} \rangle \rangle$ for $\cM$ has a factorization
$$
   k_\cM(z, \zeta) = G(z) (k_{\rm nc, Sz}(z, \zeta) \otimes I_\cU) G(\zeta)^*
$$
 for some coefficient Hilbert space $\cU$ and some power series $G(z) \in \cL(\cU, \cY)\langle \langle z \rangle \rangle$.
 Item (2) in Proposition \ref{P:2.3} now tells us that $G$ is a coisometric multiplier from $H^2_\cU(\free)$ onto $\cM$.
 Let us set $\Theta = G$ but consider $\Theta$ as a multiplier from $H^2_\cU(\free)$ into $H^2_{\bo, \cY}(\free)$. Then
 $M_\Theta = \iota M_G$ where $\iota \colon \cM \to H^2_{\bo, \cY}(\free)$ is the inclusion map.  Since $\cM$ is contained in
 $H^2_{\bo, \cY}(\free)$ contractively, $\| \iota \| \le 1$ and hence  $\| M_\Theta \| \le 1$, i.e., $\Theta$ is a contractive
 multiplier from $H^2_\cU(\free)$ to $H^2_{\bo, \cY}(\free)$.  Moreover, as $M_G H^2_\cU(\free) = \cM$, it is still the case that
 $M_\Theta H^2_\cU(\free) = \cM$.  The fact that $M_G \colon H^2_\cU(\free) \to \cM$ is a coisometry can be interpreted as
 saying that the $\cM$-norm is given by the lifted-norm criterion \eqref{M-norm}.
 
We now suppose that the subspace $\cM$ is presented as the pullback space $\cH^p(\Pi)$ for some operator
$\Pi$ on $H^2_{\bo, \cY}(\free)$ satisfying conditions \eqref{pullback-ineq}.
The first condition in \eqref{pullback-ineq} then implies that $\cM = \cH^p(\Pi)$ is contractively included in $H^2_{\bo, \cY}(\free)$.  By the Douglas
 lemma \cite{douglas}, the second condition \eqref{pullback-ineq} implies that there is a row contraction $\bX = (X_1, \dots, X_d)$ on 
 $H^2_{\bo, \cY}(\free)$ so that 
\begin{equation}   
S_{\bo, R, j} \Pi^{\frac{1}{2}} = \Pi^{\frac{1}{2}} X_j\quad\mbox{for}\quad j=1,\ldots,d.
\label{pullback-ineqa}
 \end{equation}
In particular it follows that $\cH^p(\Pi) = \operatorname{Ran} \Pi^{\frac{1}{2}}$ is invariant under the tuple $\bS_{\bo, R}$.
 Furthermore we may redo the computations leading up to \eqref{9.2} and \eqref{9.4} for the present context as follows.
From \eqref{pullback-ineqa} we have $\Pi^{\frac{1}{2}} X_j|_{\operatorname{Ker} \Pi} = 0$ which can be written as 
$Q X_j (I - Q) = 0$, where $Q = P_{(\operatorname{Ker} \Pi)^\perp}$. Thus, $Q X_j = Q X_j Q$ and then for 
$T_j = S_{\bo, R, j}|_{\cH^p(\Pi)}$ and $f,g\in H^2_{\bo, \cY}(\free)$,
we have (again due to \eqref{pullback-ineqa})
 \begin{align*}
\langle \Pi^{\frac{1}{2}} g, \, T_j^*\Pi^{\frac{1}{2}} f \rangle_{\cH^p(\Pi)}
&=\langle S_{\bo, R, j} \Pi^{\frac{1}{2}} g, \, \Pi^{\frac{1}{2}} f \rangle_{\cH^p(\Pi)} \\
&=\langle P^{\frac{1}{2}} X_j g, \, P^{\frac{1}{2}} f \rangle_{\cH^p(\Pi)} \\
& = \langle Q X_j g, \, f \rangle_{H^2_{\bo, \cY}(\free)} \\
& = \langle Q X_j Q g, \, f \rangle_{H^2_{\bo, \cY}(\free)}\\
&= \langle Q g, X_j^* Q f \rangle_{H^2_{\bo, \cY}(\free)} 
  = \langle \Pi^{\frac{1}{2}} g,  \Pi^{\frac{1}{2}} X_j^* Q f \rangle_{\cH^p(\Pi)}
\end{align*}
from which we conclude that
$T_j^* \Pi^{\frac{1}{2}}f = \Pi^{\frac{1}{2}} X_j^* Q f$. Hence the computations as in the derivation of
\eqref{again20} adapt to the present context as follows:   for
$f \in H^2_{\bo, \cY(\free)}$, we have
\begin{align*}
\sum_{j=1}^d \| T_j^* \Pi^{\frac{1}{2}} f \|^2_{\cH^p(\Pi)} & =
\sum_{j=1}^d \| \Pi^{\frac{1}{2}} X_j^* Q f \|_{\cH^p(\Pi)}\\
&= \sum_{j=1}^d \| Q X_j^* Q f \|^2_{H^2_{\bo, \cY}(\free)}  \\
& \le \sum_{j=1}^d \| X_j^* Q f \|^2_{H^2_{\bo, \cY}(\free)}  \le \| Q f \|^2_{H^2_{\bo, \cY}(\free)}  = \| \Pi^{\frac{1}{2}} f \|^2_{\cH^p(\Pi)}.
\end{align*}
We conclude that $\bT=(T_1,\ldots,T_j)$ is indeed a row contraction.  The converse assertion (that 
$\bT$ being a row contraction implies the last inequality in \eqref{pullback-ineq}) follows by reversing the argument.
\end{proof}

We next address the uniqueness issue of a representer $\Theta$ for a ${\bf S}_{\bo,n}$-invariant subspace $\cM$ in Theorem \ref{T:NC-BL}.  

 \begin{proposition}  \label{P:BLunique}
The power series $\Theta' \in \cL(\cU', \cY)\langle \langle z \rangle \rangle$ is another representer for a ${\bf S}_{\bo,n}$-invariant subspace $\cM$ 
as in Theorem \ref{T:NC-BL} if and only if
 \begin{equation}  \label{BLunique1}
 \Theta_\beta \Theta_\alpha^* = \Theta'_\beta \Theta_\beta'^* \; \text{ for all } \; \beta \in \free,
 \end{equation}
 or equivalently, if and only if  there is a partial isometry $U \colon \cU \to \cU'$ such that 
 \begin{equation}  \label{BLunique2}
 \Theta'(z) U = \Theta(z),\quad
 \bigvee_{\alpha \in \free} \operatorname{Ran} \Theta_\alpha^* \subset (\operatorname{Ker} U)^\perp, \quad
 \bigvee_{\alpha \in\free} \operatorname{Ran} \Theta_\alpha'^* \subset  \operatorname{Ran} U.
 \end{equation}
\end{proposition}
\begin{proof}
If $\Theta$ and $\Theta'$ are both Beurling-Lax representers for $\cM$ as in Theorem \ref{T:NC-BL}, then we can compute the 
kernel function $K_\cM(z, \zeta)$ for $\cM$ in two ways:
$$
  K_\cM(z, \zeta) = \Theta(z) \left(k_{\text{nc, Sz}}(z, \zeta)I_\cU\right) \Theta(\zeta)^* = 
      \Theta('z) \left(k_{\text{nc, Sz}}(z, \zeta)I_\cU\right) \Theta'(\zeta)^*.
 $$
 If we apply the transformation $X \mapsto X - \sum_{j=1}^d \overline{\zeta}_j X z_j$ to both sides of the last equality, we get simply
 $\Theta(z) \Theta(\zeta)^* = \Theta'(z) \Theta'(\zeta)^*$ which, upon equating coefficients of $z^\alpha \overline{\zeta}^{\beta^\top}$,
 implies equalities \eqref{BLunique1}. Hence there is a unitary map 
 $$
U_0 \colon \cR \to \cR',\quad\mbox{where}\quad  \cR = \bigvee_{\beta \in \free} \operatorname{Ran} \Theta_\beta^*, \quad
  \cR' = \bigvee_{\beta \in \free} \operatorname{Ran} \Theta_\beta'^*
 $$
 so that
 $\; \Theta_\beta^* y = U_0^* \Theta_\beta'^* y \text{ for all } \beta \in \free, \, y \in \cY$.
 Let $U \colon \cU \to \cU'$ be any partially isometric extension of $U_0$.  Then $U$ meets the last two conditions in \eqref{BLunique2}. From
 $$
   \Theta'_\alpha u = \Theta'_\alpha u_0 = \Theta_\alpha \; \text{ for all } \; \alpha \in \free
$$
we see that $\Theta'(z) U = \Theta(z)$, i.e., the first condition in \eqref{BLunique2} holds as well.

\smallskip

Conversely, if $U \colon \cU \to \cU'$ is a partial isometry satisfying all of \eqref{BLunique2}, one can work back to \eqref{BLunique1}
and conclude that $\cM = \Theta \cdot H^2_\cU(\free)$ and $\cM'  = \Theta' \cdot H^2_\cU(\free)$, both considered with lifted norm as
contractively included subspaces of $H^2_{\bo, \cY}(\free)$, have the same kernel function $K_\cM(z, \zeta) = K_{\cM'}(z, \zeta)$
and hence $\cM = \cM'$. The latter uniqueness criterion has been stated with only the first condition in \eqref{BLunique2} as a necessary condition 
for $\Theta$ and $\Theta'$ to be representers of the same contractively included subspace $\cM$ (see \cite[Theorem 4.1]{MCT} and \cite[Corollary 4.4]{BEKS}).
\end{proof}
 
The special case of Theorem \ref{T:NC-BL} where $\cM$ is isometrically contained in $H^2_{\bo, \cY}(\free)$ is of particular interest.
\begin{theorem} \label{T:NC-BLisom}
A Hilbert space $\cM$ is isometrically included in $H^2_{\bo,\cY}(\free)$ and is ${\bf S}_{\bo,R}$-invariant if and only if 
there is a Hilbert space $\cU$ and a
McCT-inner multiplier $\Theta$ from $H^2_{\cU}(\free)$ to $H^2_{\bo,\cY}(\free)$ so that
\begin{equation}
\cM = \Theta \cdot H^2_{\cU}(\free).
\label{Mform''}
\end{equation}
 \end{theorem}
 
\begin{proof}
The "if" direction is immediate. To get the "only if" part as a consequence of Theorem \ref{T:NC-BL}, let us show that 
condition (3) in Theorem \ref{T:NC-BL} is automatic if a ${\bf S}_{\bo,R}$-invariant subspace
$\cM$ is isometrically included in $H^2_{\bo, \cY}(\free)$. 
Indeed, if this is the case, then 
$$
T^*_j=\left(S_{\bo,R,j}\vert_\cM\right)^*=P_\cM S_{\bo, R, j}^*\vert_{\cM}\quad\mbox{for}\quad j=1,\ldots,d,
$$
and therefore, we have
\begin{align*}
 \sum_{j=1}^d \| T_j^* f \|^2_\cM & = \sum_{j=1}^d \| P_\cM (S_{\bo, R, j})^* f \|^2_{H^2_{\bo, \cY}(\free)}\\
&  \le \sum_{j=1}^d \| (S_{\bo, R, j})^* f \|^2_{H^2_{\bo, \cY}(\free)}\le \| f \|^2_{H^2_{\bo, \cY}(\free)}
 \end{align*}
which confirms that condition (3) in the statement of Theorem \ref{T:NC-BL} holds and we indeed have \eqref{Mform''}. 
Furthermore, the fact that the operator $M_G$ in the proof of Theorem \ref{T:NC-BL} is a coisometry translates to
 $M_\Theta$ is a partial isometry, i.e., $\Theta$ is a McCT-inner multiplier from $H^2_\cU(\free)$ to $H^2_{\bo, \cY}(\free)$.
\end{proof}

\begin{remark}  \label{R:BBF3'}  Let us consider the special case of Theorem \ref{T:NC-BLisom} where $\bo = \bmu_1$ 
so $H^2_{\bo, \cY}(\free)$ is equal to the Hardy-Fock space $H^2_\cY(\free)$ as in Section \ref{S:Fock0}.  Then Theorem 
\ref{T:NC-BLisom} gives the existence of a McCT-inner multiplier from some $H^2_\cU(\free)$ to $H^2_\cY(\free)$ to 
represent a closed $\bS_{1, R}$ invariant subspace.  However it is known that in this case one can implement
such a representation via a strictly inner multiplier (see \cite{PopescuBL} and also \cite[Theorem 2.14]{BBF1} and
\cite[Theorems 5.1 and 5.2]{BBF3} for later proofs from other points of view).  We note here that this finer result 
also follows easily from Theorem \ref{T:NC-BLisom}
if we recall Remark \ref{R:BBF3} to the effect that any McCT-inner multiplier $\Theta$ between Hardy-Fock spaces
decomposes as $\Theta = \begin{bmatrix} \Theta_{\rm si} & 0 \end{bmatrix}$ with respect to an appropriate decomposition
$\cU = \sbm{ \cU_{\rm si} \\ \cU_0}$ of the input space.  Replacing the McCT-inner $\Theta$ with the strictly inner 
$\Theta_{\rm si}$ generates the same shift-invariant subspace 
$\cM = \Theta \cdot H^2_\cU(\free) = \Theta_{\rm si} \cdot H^2_{\cU_0}(\free)$.
\end{remark}

We now present an explicit formula for the Beurling-Lax representer for at least a class of shift-invariant
subspaces $\cM$ contained only contractively in $H^2_{\bo, \cY}(\free)$ for the case where the weight sequence
$\bo$ is strictly decreasing. In this case we can define the associated weight sequence $\bgam$ as in \eqref{bgam}. 
If an output pair $(C,\bA)$ is $\bo$-output stable, it is also $\bgam$-output stable, since $\gamma_j^{-1}<\omega_j^{-1}$ 
for all $j\ge 1$ (by \eqref{bgam}) and therefore,
$$
\cG_{\bgam,C,\bA}=\sum_{\alpha\in\free}\gamma_{|\alpha|}^{-1}\bA^{*\alpha^\top}C^*C\bA^\alpha \preceq
\sum_{\alpha\in\free}\omega_{|\alpha|}^{-1}\bA^{*\alpha^\top}C^*C\bA^\alpha=\cG_{\bo,C,\bA}.
$$
It is readily seen from the above power series representations that the gramians $\cG_{\bo,C,\bA}$ and $\cG_{\bgam,C,\bA}$ 
satisfy the Stein equation
\begin{equation}
\cG_{\bo,C,\bA}-\sum_{j=1}^d A_j^*\cG_{\bo,C,\bA}A_j=\cG_{\bgam,C,\bA}.
\label{again21}
\end{equation}
Alternatively, we can apply the $B_\bA$ calculus to the identity
\begin{equation}  \label{Res-id}
(1-\lambda)R_\bo(\lambda)=R_{\bgam}(\lambda)
\end{equation}
and then apply the resulting operator equality to the operator $C^*C$.   Another consequence of \eqref{Res-id}
which we shall have several occasions to use is the identity
\begin{equation}  \label{Res-id-op}
R_\bo(Z(z) A) (I - Z(z) A) = R_\bgam(Z(z)A).
\end{equation}
We next recall the space $\ell^2_{\cY}(\free)$ 
defined in \eqref{jan4a} and the formal power series $\bPs_\bo(z)$ (as well as $\bPs_\bgam(z)$) defined by
\eqref{powerseries} which implement the factorizations
\begin{equation}   \label{factor}
k_{\bo}(z, \zeta) = \bPs_\bo(z) \bPs_\bo(\zeta)^* = \bPs_\bgam(z)
\left( k_{\rm nc, Sz}(z, \zeta) I_{\ell^2(\free)} \right) \bPs_\bgam(\zeta)^*.
\end{equation}
Let us also introduce a normalized time-domain version of the $\bgam$-observability operator 
$\widehat\cO_{\bgam,C,\bA}\colon \cX\to \ell^2_{\cY}(\free)$ defined by
$$
\widehat{\cO}_{\bgam,C,\bA}: \; x\mapsto
\{\gamma_{|\alpha|}^{-\frac{1}{2}}C\bA^\alpha x\}_{\alpha\in\free}.
$$
The equality
$\|\widehat{\cO}_{\bgam,C,\bA}x\|_{\ell^2_{\cY}(\free)}^2=\langle\cG_{\bgam,C,\bA}x, \, x\rangle_{\cX}$
holds for all $x\in\cX$ and justifies the factorization 
\begin{equation}
\cG_{\bgam,C,\bA}=\widehat{\cO}_{\bgam,C,\bA}^*\widehat{\cO}_{\bgam,C,\bA}.
\label{12.2ag}
\end{equation}
Yet another easily verified useful identity is
\begin{equation}    \label{bPs-Obs}
\bPs_\bgam(z) \widehat \cO_{\bgam, C, \bA} = C R_\bgam(Z(z) A) = C R_\bo(Z(z) A) (I - Z(z) A).
\end{equation}
\begin{theorem}  \label{T:explicit2}
Suppose that the admissible weight sequence $\bo$ is strictly decreasing.

\smallskip

{\rm (1)} Let $\cM$ be a Hilbert space contractively included in $H^2_{\bo, \cY}(\free)$.
Assume that the Brangesian complement $\cM^{[\perp]}$ (defined as in Section \ref{S:Brangesian}) satisfies the hypotheses
of part (4) Theorem \ref{T:2-1.2}
and hence has a  representation as
$$
    \cM^{[\perp]} = \operatorname{Ran}  \cO_{\bo, C, \bA}
$$
for some $\bo$-contractive output pair $(C, \bA)$.  Let us impose the additional hypothesis that $(C, \bA)$ satisfies
the inequality
\begin{equation}  \label{extra-hy'}
\sum_{j=1}^d A_j^* (I - \cG_{\bo, C, \bA}) A_j \preceq I - \cG_{\bo, C, \bA}.
\end{equation}
Then the related inequality
\begin{equation}  \label{extra-hy}
\begin{bmatrix} A^* & \widehat \cO_{\bgam, C, \bA}^*  \end{bmatrix}
\begin{bmatrix} A \\ \widehat \cO_{\bgam, C, \bA} \end{bmatrix}
\preceq I_\cX.
\end{equation}
holds and as a consequence of the observation \eqref{fact} we can choose a solution
\begin{equation}
\begin{bmatrix} B \\ D \end{bmatrix}=\begin{bmatrix}  B \\ {\rm col}_{\alpha \in \free}
[D_{\alpha}] \end{bmatrix}\colon \cU \to \begin{bmatrix}\cX^{d} \\
\ell^{2}_{\cY}(\free)\end{bmatrix}
\label{bd}
\end{equation}
of the Cholesky factorization problem
\begin{equation}  \label{Cho-big1}
\begin{bmatrix} B \\ D \end{bmatrix} \begin{bmatrix} B^* & D^* \end{bmatrix} =
\begin{bmatrix}  I_{\cX^d} & 0 \\ 0 & I_{\ell^2_\cY(\free)}\end{bmatrix}
- \begin{bmatrix} A \\ \widehat \cO_{\bgam, C, A} \end{bmatrix}
\begin{bmatrix} A^* & \widehat \cO_{\bgam, C, \bA}^* \end{bmatrix}.
\end{equation}
Define formal power series
\begin{equation}  \label{Theta1}
\mathfrak D(z) = \sum_{\alpha \in \free} \bgam_{|\alpha|}^{-\frac{1}{2}} D_\alpha z^\alpha\quad\mbox{and}
\quad  \Theta(z) = {\mathfrak D}(z) + C R_\bo(Z(z) A) Z(z) B.
\end{equation}
Then $\Theta$ is a Beurling-Lax representer for $\cM$, i.e., $\Theta$ is a contractive multiplier
from $H^2_{\cU}(\free)$ to $H^2_{\bo, \cY}(\free)$ such that $\cM = \Theta \cdot H^2_\cU(\free)$.

\smallskip

{\rm (2)}  Suppose that $\cM$ is  isometrically included subspace of $H^2_{\bo,\cY}(\free)$ which is $\bS_{\bo, R}$-invariant
with $\bS_{\bo, R}^*$-invariant orthogonal complement $\cM^\perp$
represented as $\cM^\perp = \operatorname{Ran} \cO_{\bo,C,\bA}$ for a $\bo$-isometric pair $(C,\bA)$ with
$\bA$ being $\bo$-strongly stable as in Theorem \ref{T:1.2g}.  Then the additional hypothesis \eqref{extra-hy'} is automatic
and a McCT-inner Beurling-Lax representer $\Theta$ for $\cM$ can be constructed explicitly via the procedure given by \eqref{bd}, \eqref{Cho-big1},
\eqref{Theta1}.
\end{theorem}

\begin{proof} To see that \eqref{extra-hy'} implies \eqref{extra-hy}, rewrite \eqref{extra-hy'} as
$$
  A^* A + \big(\cG_{\bo, C, \bA} - \sum_{j=1}^d A_j^* \cG_{\bo, C, \bA} A_j  \big)\preceq I
$$
which easily converts into \eqref{extra-hy} upon noting \eqref{again21} and \eqref{12.2ag}.

\smallskip

To verify part (1), i.e., to show that $\Theta$ \eqref{Theta1} is a Beurling-Lax representer for $\cM$, by Theorem \ref{T:charNFRKHS}, 
it suffices to show that $\cM$
has kernel function $K_\cM$ given by 
$$
K_\cM(z, \zeta) = \Theta(z) (k_{\rm nc, Sz}(z, \zeta) I_\cU) \Theta(\zeta)^*.
$$
and furthermore, that the kernel function for the Brangesian complement $\cM^{[\perp]}$ of $\cM$ as a subspace of $H^2_{\bo, \cY}(\free)$
is given by
$$
 K_\cM^{[\perp]}(z, \zeta) = k_\bo(z, \zeta) I_\cY - K_\cM(z, \zeta).
$$
On the other hand, by part (4) of Theorem \ref{T:2-1.2} we know that the kernel function for the Brangesian complement
$\cM^{[\perp]}$ is given by
$$
  K_{\cM^{[\perp]}}(z, \zeta) = C R_\bo(Z(z)A) R_\bo(Z(\zeta) A)^* C^*.
$$
Thus, to show that $\Theta$ is a Beurling-Lax representer for $\cM$ it suffices to demonstrate the kernel decomposition
\begin{equation}   \label{ker-decom1}
 k_{\bo}(z, \zeta) I_\cY - \Theta(z) (k_{\rm nc, Sz}(z, \zeta) I_\cU)  \Theta(\zeta)^* = C R_\bo(Z(z)A) R_\bo(Z(\zeta) A)^* C^*.
 \end{equation}
 To simplify notation let us set $C_0 = \widehat \cO_{\bgam, C, \bA}$.  With $\sbm{ B \\ D }$ constructed as in \eqref{Cho-big1}
 we set 
 $$
   \Theta_0(z) = D + C_0 (I - Z(z) A)^{-1} Z(z) B.
 $$
 Then as a consequence of Theorem \ref{T:cm} (specifically the identity \eqref{jan3b}), we have the kernel decomposition
 \begin{align}  
& k_{\rm nc, Sz}(z, \zeta) I_{{\ell^2_\cY}(\free)} - \Theta_0(z) (k_{\rm nc, Sz}(z, \zeta) I_\cU) \Theta_0(\zeta)^*\notag\\
& \quad= C_0 (I - Z(z) A)^{-1} (I - A^* Z(\zeta)^*)^{-1} C_0^*.\label{ker-decom0}
 \end{align}
 Next we  verify that the formula
\begin{equation}  \label{formula}
   \Theta(z) = \bPs_\bgam(z) \Theta_0(z).
\end{equation}
defines the same power series $\Theta$ as in \eqref{Theta1}. Indeed, $\bPs_\bgam(z) D = {\mathfrak D}(z)$, 
by the formula \eqref{powerseries} for $\bPs_\bgam(z)$, and on the other hand, 
$$
\bPs_\bgam(z) C_0 = \bPs_\bgam(z) \widehat \cO_{\bgam, C, \bA} = C R_\bgam(Z(z) A),
$$
by \eqref{bPs-Obs}.
With these two latter substitutions combined with \eqref{Res-id-op} we have
\begin{align*}
 \Theta(z) & = {\mathfrak D}(z) + C R_\bgam(Z(z) A) (I - Z(z) A)^{-1} Z(z) B  \\
 & = {\mathfrak D}(z) + C R_\bo(Z(z) A) Z(z) B  
\end{align*}
in agreement with \eqref{Theta1}.  Then also
\begin{align*}
& k_{{\rm nc}, \bo}(z, \zeta) I_\cY - \Theta(z) (k_{\rm nc, Sz}(z, \zeta) I_\cU) \Theta(\zeta)^*  \\
& = \bPs_\bgam(z) \left(k_{\rm nc, Sz}(z, \zeta) I_{\ell^2_\cY(\free)} -
\Theta_0(z) (k_{\rm nc, Sz}(z, \zeta) I_\cU) \Theta_0(\zeta)^* \right) \bPs_\bgam(\zeta)^*  \\
& \quad \quad \quad \quad 
\text{ (by \eqref{factor} and \eqref{formula})}  \\
& = \bPs_\bgam(z) C_0 (I - Z(z) A)^{-1} (I - A^* Z(\zeta)^*)^{-1} C_0^*\bPs_\bgam(\zeta)^*
\text{ (by \eqref{ker-decom0})}  \\
& = C R_\bgam(Z(z) A) (I - Z(z) A)^{-1}  (I - A^* Z(\zeta)^*)^{-1} R_\bgam(Z(\zeta)A)^* C^* \\
& = C R_\bo(Z(z) A) R_\bo(Z(\zeta) A)^* C^*
\text{ (by \eqref{Res-id-op})}
\end{align*}
and \eqref{ker-decom1} follows.  This completes the verification of part (1).

\smallskip

With regard to part (2),   since $\cM$ is isometrically contained in $H^2_{\bo, \cY}(\free)$, the Brangesian complement is just the
standard Hilbert-space orthogonal complement $\cM^\perp$.  As $\cM$ is $\bS_{\bo, R}$-invariant, then orthogonal complement
$\cM^\perp$ is $\bS_{\bo, R}^*$-invariant and we may use Theorem \ref{T:1.2g} to represent $\cM^\perp$ as
$\operatorname{Ran} \cO_{C, \bA}$ with $(C, \bA)$ a $\bo$-isometric pair with in addition $\bA$ strongly $\bo$-stable.  Furthermore
in this case the observability operator $\cO_{\bo, C, \bA}$ is isometric, the observability gramian $\cG_{\bo, C, \bA}$ is equal to the identity
$I_\cX$, and the additional hypothesis \eqref{extra-hy'} is valid in the trivial form $0 \preceq 0$.  Hence condition \eqref{extra-hy}
holds with equality and the procedure given in part (1) constructs the McCT-inner Beurling-Lax representer for $\cM$.
\end{proof}

\begin{remark}  \label{R:gen-ker-decom}  
Let us observe that the Stein equation \eqref{again21} can be equivalently rewritten as
$$
  \begin{bmatrix} A^* & ( \widetilde \cO_{\bgam, C, \bA})^* \end{bmatrix}
  \begin{bmatrix} \cG_{\bo, C, \bA} \otimes I_d & 0 \\ 0 & I_{\ell^2_\cY(\free) }\end{bmatrix}
  \begin{bmatrix}  A \\ \widetilde \cO_{\bgam, C, \bA} \end{bmatrix} = \cG_{\bo, C, \bA}.
$$
Then the same calculations as in the derivation of the kernel 
decomposition \eqref{ker-decom1}  in the proof of Theorem \ref{T:explicit2}
leads to the construction of a formal power series $\Theta(z)$ so that  
$$
k_{\bo}(z, \zeta) I_\cY - \Theta(z) (k_{\rm nc, Sz}(z, \zeta) I_\cU) \Theta(\zeta)^* =
C R_\bo(Z(z) A) \cG_{\bo, C, \bA}^{-1} R_\bo(Z(\zeta) A)^* C^*.
$$
This kernel decomposition however appears to have no application to proving results about Beurling-Lax representations
(unless of course $\cG_{\bo, C, \bA} = I_\cX$).
\end{remark}

\begin{remark}  \label{R:NC-BL}  Note that if $\cM$ is a contractively-included subspace of $H^2_{\bo,\cY}(\free)$ such that
its Brangesian complement satisfies the hypotheses of item (4) in Theorem \ref{T:2-1.2}, it follows that $\cM$ satisfies conditions
(2) and (3) in Theorem \ref{T:NC-BL}, i.e., $\cM$ is $\bS_{\bo, R}$-invariant and $(\bS_{\bo, R}|_\cM)^*$ is a row contraction.
In general it need not be the case (unlike the classical case for $d=1$) that $\cM$ contractively included in $H^2_{\bo, \cY}(\free)$
and invariant for $\bS_{\bo, R}$ implies that $\cM^{[\perp]}$ is $\bS_{\bo, R}^*$-invariant (unless $\cM$ is 
isometrically contained in $H^2_{\bo, \cY}(\free)$); see e.g., the end of the report \cite{BBCR}.
\end{remark}

There is a more flexible Beurling-Lax representation theorem in the same spirit but which appears to 
require an additional hypothesis (see Remark \ref{R:2proofs} below).

\begin{theorem} \label{T:NC-BL'}
Let $\bo$ and $\bo'$ be two admissible weights satisfying condition \eqref{1.8g}.
Then a Hilbert space $\cM$ is such that
\begin{enumerate}
\item $\cM$ is contractively included in $H^2_{\bo,\cY}(\free)$,
\item $\cM$ is ${\bf S}_{\bo,R}$-invariant,
\item the $d$-tuple $\; {\bf A} =(A_1,\ldots,A_d)$ where $A_j=(S_{\bo,R,j}|_{\cM})^*$ for
$j=1,\ldots,d$, is a strongly $\bo'$-stable $\bo'$-hypercontraction
\end{enumerate}
if and only  if there is a coefficient Hilbert space $\cU$ and a
contractive multiplier $\Theta$ from $H^2_{\bo',\cU}(\free)$ to $H^2_{\bo,\cY}(\free)$ so that
\begin{equation}
\cM = M_\Theta  H^2_{\bo',\cU}(\free)\quad \mbox{with lifted norm}\quad 
\| M_\Theta f \|_{\cM} = \| Q f \|_{H^2_{\bo',\cU}(\free)},
\label{Mform'}
\end{equation}
where $Q$ is the orthogonal projection onto $(\operatorname{Ker} M_{\Theta})^\perp$.
  
\smallskip
        
Moreover, if it is the case that $\cM$ is isometrically contained in $H^2_{\bo, \cY}(\free)$, then
$\Theta$ can be taken to be a {\rm McCT}-inner multiplier from $H^2_{\bo', \cU}(\free)$ to $H^2_{\bo, \cY}(\free)$.
        \end{theorem}
        
\begin{proof}  
\textbf{Sufficiency:}  Suppose that $\cM$ is of the form \eqref{Mform'} for a contractive multiplier from 
$H^2_{\bo', \cU}(\free)$ to $H^2_{\bo, \cY}(\free)$.  The inequality \eqref{ineq1} takes the form
$$
\| M_\Theta f \|_{H^2_{\omega, \cY}(\free) } \le \| Q f \|_{H^2_{\bo', \cU}(\free)} \le \| M_\Theta f \|_\cM
$$
and hence (1) holds.  Next, relations \eqref{9.2} take the form
\begin{equation}
Q S_{\bo', R, j} = Q S_{\bo', R, j} Q, \quad S^*_{\bo', R, j} Q = Q S^*_{\bo', R, j} Q\quad\text{for}\quad j =1,\dots,d,
\label{9.2'}
\end{equation}
and mimicking the succeeding computation (with $T_j$, $S_{\bo',R,j}$ and $H^2_{\bo',\cU}(\free)$ instead of 
$T_j^*$, $S_{1,R,j}$ and $H^2_{\bo',\cU}(\free)$, respectively) 
we get a more explicit formula for $A_j=(S_{\bo,R,j}|_{\cM})^*$:
\begin{equation}
A_j \colon M_\Theta f  \mapsto M_\Theta S_{\bo',R,j}^* Qf\quad\mbox{for}\quad j=1,\ldots,d. 
\label{again18}
\end{equation}
Making use of the same substitutions as in computation \eqref{again20} we conclude that $\bA^*$ is a row contraction, i.e., 
that the tuple $\bA$ is contractive in the sense of \eqref{4.29}.

\smallskip

Iterating the formula \eqref{again18} gives, on account of relations \eqref{9.2'},
$$
\bA^\alpha\colon M_\Theta f \mapsto M_\Theta 
(\bS^*_{\bo',R})^{\alpha} Q f \quad \text{for}\quad \alpha \in \free.
$$
Combining the latter with the definition of the lifted norm \eqref{M-norm}, we get
$$
\|\bA^{\alpha}\Theta f\|_{\cM}=\|{\bf S}_{\bo',R}^{*\alpha}Qf\|_{H^2_{\bo',\cU}(\free)}\quad \text{for}\quad \alpha \in \free.
$$
Making use of all these relations along with the definitions \eqref{1.11g}, \eqref{3.5ag}, we get
\begin{align*}
\left\langle (\Gamma_{\bo',\bA}[I_{\cM}])\Theta f, \, \Theta f\right\rangle_{\cM}&=
\big\langle \big(\Gamma_{\bo',{\bf S}_{\bo',R}^*}[I_{H^2_{\bo',\cU}(\free)}]\big)
Q f, \, Q f\big\rangle_{H^2_{\bo',\cU}(\free)},\\
\big\langle (\Gamma^{(k)}_{\bo',\bA}[I_{\cM}])\Theta f, \, \Theta
f\big\rangle_{\cM}&=\big\langle \big(\Gamma^{(k)}_{\bo',{\bf S}_{\bo',R}^*}[I_{H^2_{\bo',\cU}(\free)}]\big)
Q f, \, Q f\big\rangle_{H^2_{\bo',\cU}(\free)}
\end{align*}
for $k=1,2,\dots$, and also
\begin{align*}
&\sum_{\alpha\in\free:|\alpha|=k}\big\langle 
\bA^{*\alpha^\top}\Gamma^{(k)}_{\bo',\bA}[I_{\cM}]\bA^{\alpha}\Theta f, \, 
\Theta f\big\rangle_{\cM}\\
&\quad=\sum_{\alpha\in\free:|\alpha|=k}\big\langle\Gamma^{(k)}_{\bo',\bA}[I_{\cM}]\bA^{\alpha}\Theta f, \, 
]\bA^{\alpha}\Theta 
f\big\rangle_{\cM}\\  
&\quad=\sum_{\alpha\in\free:|\alpha|=k}\big\langle\Gamma^{(k)}_{\bo',\bA}[I_{\cM}]\Theta {\bf 
S}_{\bo',R}^{*\alpha}Qf, \, \Theta {\bf S}_{\bo',R}^{*\alpha}Qf\big\rangle_{\cM}\\
&\quad=\sum_{\alpha\in\free:|\alpha|=k}
\big\langle\big(\Gamma^{(k)}_{\bo',{\bf S}_{\bo',R}^*}[I_{H^2_{\bo',\cU}(\free)}]\big)
Q {\bf S}_{\bo',R}^{*\alpha}Qf, \; Q {\bf S}_{\bo,R}^{*\alpha}Qf\big\rangle_{H^2_{\bo',\cU}(\free)}\\
&\quad=\sum_{\alpha\in\free:|\alpha|=k}\big\langle
{\bf S}_{\bo',R}^{\alpha^\top}\big(\Gamma^{(k)}_{\bo',{\bf S}_{\bo',R}^*}[I_{H^2_{\bo',\cU}(\free)}]\big)
{\bf S}_{\bo',R}^{*\alpha}Qf, \; Qf\big\rangle_{H^2_{\bo',\cU}(\free)}.
\end{align*}
Since $S_{\bo',R}^*$ is an $\bo'$-strongly stable $\bo'$-hypercontraction on $H^2_{\bo',\cY}(\free)$
(by Proposition \ref{P:bo-model2}), we conclude from the latter computations that
$$
\left\langle (\Gamma_{\bo',\bA}[I_{\cM}])\Theta f, \, \Theta f\right\rangle_{\cM}\ge 0, \quad
\big\langle (\Gamma^{(k)}_{\bo',\bA}[I_{\cM}])\Theta f, \, \Theta
f\big\rangle_{\cM}\ge 0
$$
for all $f\in H^2_{\bo',\cU}(\free)$ and all $k\ge 1$, and that moreover,
$$
\lim_{k\to \infty}\sum_{\alpha\in\free:|\alpha|=k}\big\langle
\bA^{*\alpha^\top}\Gamma^{(k)}_{\bo',\bA}[I_{\cM}]\bA^{\alpha}\Theta f, \,
\Theta f\big\rangle_{\cM}=0.
$$
We conclude that $\bA$ is an $\bo'$-strongly stable $\bo'$-hypercontraction
on $\cM$ and thereby complete the verification of condition (3). The verification of conditions (1) and (2) 
proceeds  as in the proof of sufficiency in Theorem \ref{T:NC-BL}.

\smallskip

\textbf{Necessity.}  Suppose that $\cM$ is a subspace of $H^2_{\bo, \cY}(\free)$ satisfying conditions (1), (2), (3).
By condition (3), the operator $\Gamma_{\bo', \bA}[I_\cM]$ is positive semidefinite.
Define $C \colon \cM \to \cU:= \overline{ \operatorname{Ran}} (\Gamma_{\bo', \bA}[I_\cM])^{1/2}$ so that
$C^*C = \Gamma_{\bo', \bA}[I_\cM]$.  Then the pair $(C, \bA)$ is a $\bo'$-isometric output pair, and
the $\bo'$-observability operator $\cO_{\bo', C, \bA} \colon \cM \to H^2_{\bo', \cU}(\free)$ is isometric (by part (2) in Lemma 
\ref{L:7.6}) and satisfies the intertwining condition
$$
   S_{\bo', R, j}^* \cO_{\bo', C, \bA} = \cO_{\bo', C, \bA} A_j\quad \text{for}\quad  j = 1, \dots, d.
$$
Taking adjoints then gives
$$
\cO_{\bo', C, \bA}^* S_{\bo', R, j} = A_j^* \cO_{\bo', C, \bA}^* \quad \text{for}\quad  j = 1, \dots, d.
$$
The inclusion map $\iota\colon \cM\to  H^2_{\bo, \cY}(\free)$ is a contraction due to condition (1). Due
to condition (2) and the definition of $A_j$, we also have $\iota \circ A_j^* = S_{\bo, R, j}\circ \iota$. 
Therefore, the operator 
$$
X = \iota \circ\cO_{\bo', C, \bA}^* \colon H^2_{\bo', \cU}(\free) \to H^2_{\bo, \cY}(\free)
$$
is contractive (as $\iota$ is a contraction and $\cO_{\bo', C, \bA}$ is an isometry) and 
satisfies the intertwining relations 
$$
X S_{\bo', R, j}=\iota \circ\cO_{\bo', C, \bA}^*S_{\bo', R, j}=\iota \circ A_j^* \cO_{\bo', C, \bA}^*
=S_{\bo, R, j}\circ \iota\circ \cO_{\bo', C, \bA}^* = S_{\bo, R, j} X
 $$
for $j=1,\ldots,d$. By Proposition \ref{P:RR} there is a contractive multiplier $\Theta$ from 
$H^2_{\bo', \cU}(\free)$ to $H^2_{\bo, \cY}(\free)$ so that $X = M_\Theta$. In case $\cM$ is isometrically included in 
$H^2_{\bo, \cY}(\free)$, one can verify that $M_{\Theta}$ is a partial isometry so that  
$\Theta$ is McCT-inner just as in the proof of  Theorem \ref{T:NC-BL}.
\end{proof}

\begin{remark}  \label{R:2proofs}
 Note that the situation in Theorem \ref{T:NC-BL} is the special case of Theorem \ref{T:NC-BL'} 
where in Theorem \ref{T:NC-BL'} one specializes $\bo'$ to be $\bmu_1$ (i.e., $\omega'_j = 1$ for all $j\ge 0$).
The results differ somewhat in that in Theorem \ref{T:NC-BL} it was not necessary to assume that 
$\bA$ is $\bo'$-strongly stable, or equivalently for the case $\bo' = \mu_1$ as a consequence of Remark 
\ref{R:stabmu}, $\bA$ is strongly stable in the usual sense.  On the other hand, the proof of sufficiency in Theorem
\ref{T:NC-BL'} shows that in the end necessarily $\bA$ is strongly stable after all. In summary we can say that
specializing Theorem \ref{T:NC-BL'} to the case $\bo' = \bmu_1$ leads us to the weaker version
of Theorem \ref{T:NC-BL} where one adds to hypothesis (3) the (in the end unneeded) additional hypothesis
that ${\mathbf T}^*$ is strongly stable.

\smallskip

An observation concerning Theorem \ref{T:NC-BLisom} versus  the isometric-inclusion version of \ref{T:NC-BL'} is the following.  When  $\cM$ is contained in $H^2_{\bo, \cY}(\free)$ isometrically and $\bo' = \bmu_1$,
it is automatic that $\bA$ is a $\bo'$-contraction;  for the case of a general $\bo'$ this is not the case in general (see 
\cite[Remark 7.6]{BBIEOT}).

\smallskip

Finally let us  note that the classical de Branges result in \cite{dBR2} amounts to the case $d=1$, $\bo = \bmu_1$
in Theorem \ref{T:NC-BL}.
\end{remark}

\section{Representations with model space of the form $\bigoplus_{j=1}^n \cA_{j, \cU_j}(\free)$}  \label{S:BL3}

A Beurling-Lax representation as in Theorem \ref{T:NC-BL'} lacks the precision of the 
classical version in several 
respects:  (1) the operator $M_{G}$ is only a partial isometry rather than an isometry (the distinction between 
{\em McCT-inner} and {\em strictly inner}) with the consequence that in general $M_G$ has a huge kernel, and 
(2) the input coefficient space $\cU$ is almost always infinite dimensional. For the case of standard weighted 
Bergman-Fock spaces we will give a more structured version of the Beurling-Lax theorem
which should carry more information about the $\bS_{n,R}$-invariant subspace $\cM$ of $\cA_{n,\cY}(\free)$, 
at least under a strong additional assumption; we leave as an open question whether a version of the following result 
still holds without imposition of this extra assumption.
The extra assumption on the $\bS_{n,R}$-invariant subspace $\cM$ 
of $\cA_{n,\cY}(\free)$ is the following:

\medskip

$\bullet$ The output pair $(C, \bA)$ with $C = E|_{\cM^{\perp}}$ and 
$\bA =(S_{n,R,1}|_{\cM^{\perp}},\ldots,S_{n,R,d}\vert_{\cM^{\perp}})$
is exactly observable, i.e.,
\begin{equation} \label{assume}
    \cG_{1,C, \bA}:=\sum_{\alpha\in\free}\bA^{*\alpha^\top}C^*C\bA^{\alpha}
\succ 0 \text{ (strictly positive definite).}
\end{equation}
Since $\cG_{k,C, \bA}\succeq \cG_{1,C, \bA}$ for all $k>1$ (see the proof of Proposition \ref{P:identities}),
the assumption \eqref{assume} implies that $\cG_{k,C,\bA}$ is strictly positive definite for all $k\ge 1$.

\begin{theorem} \label{T:BL}
Let $\cM$ be a closed ${\bf S}_{n,R}$-invariant subspace of
$\cA_{n,\cY}(\free)$ satisfying the assumption 
\eqref{assume}. Then there exist Hilbert spaces $\cU_{1},\ldots,\cU_{d}$ and 
a McCT-inner multiplier
$F(z) = \begin{bmatrix} F_{1}(z) & \cdots & F_{n}(z) \end{bmatrix}$
from ${\displaystyle\bigoplus_{j=1}^{n} \cA_{j,\, \cU_{j}}(\free)}$ to $\cA_{n, \cY}(\free)$ so that 
$\; \cM = M_{F} \bigg( {\displaystyle\bigoplus_{j=1}^{n} \cA_{j,\, \cU_{j}}(\free)}\bigg)$.
        \end{theorem}

\begin{proof} We start as at the beginning of Chapter 5 by representing $\cM$ as 
$({\rm Ran}\cO_{n,C,\bA})^\perp$ for some $n$-isometric pair $(C,\bA)$ with $\bA$  
strongly stable, but with the additional assumption \eqref{assume} 
imposed.  

\smallskip

By Proposition \ref{P:2.4}, it suffices to produce
coefficient Hilbert spaces $\cU_{1}, \dots, \cU_{n}$ and power series 
$F_{j}$ in $\mathcal L(\cU_j,\cY)\langle\langle z\rangle\rangle$ so that
\begin{equation}   \label{sumker}
    k_{\cM}(z, \zeta) = \sum_{j=1}^{n} F_{j}(z)k_{{\rm nc }, j}(z,\zeta)F_{j}(\zeta)^{*},
 \end{equation} 
where $k_\cM$ is given in \eqref{kMa} and where $k_{{\rm nc }, j}(z,\zeta)$ is the 
noncommutative
Bergman kernel defined in \eqref{nc-n-Berg-ker}. The construction proceeds via an iterated
application of Theorem \ref{T:cm}. 

\smallskip

We first construct inner multipliers $\To_j(z)$ ($j=1,\ldots,n$)
between Fock spaces as follows. For $j=1$ we use the identity 
$$
{\mathcal G}_{1,C,\bA}-\sum_{j=1}^d A_j^*{\mathcal G}_{1,C,\bA}A_j=C^*C
$$
(i.e., the identity \eqref{4.10} for $k=1$) and find $\sbm{\Bo_1 \\ \Do_1}: \, \cU_n\to 
\sbm{\cX^d \\  \cY}$ such that 
$$
\begin{bmatrix} \Bo_1 \\ \Do_1 \end{bmatrix}\begin{bmatrix}
\Bo_1^{*} & \Do_1^{*}\end{bmatrix}=\begin{bmatrix} I_{\cX^d} & 0 \\0 &
I_{\cY} \end{bmatrix}-\begin{bmatrix} A \\ C
\end{bmatrix}{\mathcal G}_{1,C,\bA}^{-1}\begin{bmatrix} A^{*} & C^{*}\end{bmatrix}.
$$
 By Theorem \ref{T:cm}, the power series 
\begin{equation}
\To_1(z)=\Do_1+C(I-Z(z)A)^{-1}Z(z)\Bo_1
\label{11.4}
\end{equation}
is a strictly inner multiplier from  $H^{2}_{\cU_n}(\free)$ to $H^{2}_{\cY}(\free)$ and 
\begin{align}
&k_{\rm nc, Sz}(z,\zeta)I_{\cY}-
\To_1(z)k_{\rm nc, Sz}(z,\zeta)\To_1(\zeta)^*\notag\\
&=C(I-Z(z)A)^{-1}{\mathcal G}_{1,C,\bA}^{-1}(I-A^*Z(\zeta)^*)^{-1}C^*.
\label{11.5a}
\end{align}
For $j=2, \dots, n$ we do a similar construction based  on the  identity (see \eqref{4.10})
$$
{\mathcal G}_{j,C,\bA}-\sum_{\ell=1}^dA_\ell^*{\mathcal G}_{j,C,\bA}A_\ell={\mathcal G}_{j-1,C,\bA} \quad
(j=2,\ldots,n).
$$
We find operators
$\Bo_{j}\colon \, \cU_{n+1-j} \to \cX^d\;$ and 
$\; \Do_{j} \colon \, \cU_{n+1-j} \to \cY\; $ so that
$$
\begin{bmatrix}A & \Bo_j \\ {\mathcal G}_{j-1,C,\bA}^{\frac{1}{2}} &
\Do_j\end{bmatrix}\begin{bmatrix}{\mathcal G}_{j,C,\bA}^{-1} & 0 \\ 0 &
I_{\cU_{n+1-j}}\end{bmatrix}\begin{bmatrix}A^* & {\mathcal G}_{j-1,C,\bA}^{\frac{1}{2}} \\
\Bo_j^* & \Do_j^*\end{bmatrix}=
\begin{bmatrix}{\mathcal G}_{j,C,\bA}^{-1}\otimes I_d & 0 \\ 0 & I_{\cY}\end{bmatrix}
$$
and
\begin{equation}
\begin{bmatrix}A^* & {\mathcal G}_{j-1,C,\bA}^{\frac{1}{2}} \\ \Bo_j^* &
\Do_j^*\end{bmatrix}\begin{bmatrix}{\mathcal G}_{j,C,\bA}\otimes I_d & 0 \\ 0 &
I_{\cY}\end{bmatrix}\begin{bmatrix}A & \Bo_j \\ {\mathcal G}_{j-1,C,\bA}^{\frac{1}{2}} &
\Do_j\end{bmatrix}=\begin{bmatrix}{\mathcal G}_{j,C,\bA} & 0 \\ 0 &
I_{\cU_{n+1-j}}\end{bmatrix}.
\label{6.12}
\end{equation}
In fact, we claim that the latter equalities determine $\Bo_j$ and $\Do_j$ uniquely up
to a common isometric factor $W_j: \, \cU_{n+1-j}\to\cX^d$ on  the right:
\begin{align}
\Bo_j&=({\mathcal G}_{j,C,\bA}^{-1}\otimes I_d-A{\mathcal G}_{j,C,\bA}^{-1}A^*)^{\frac{1}{2}}W_j,
\quad W_j^*W_j=I_{\cU_{n+1-j}},
\label{6.13a}\\
\Do_j&=-{\mathcal G}_{j-1,C,\bA}^{-\frac{1}{2}}A^*({\mathcal 
G}_{j,C,\bA}\otimes I_d)\Bo_j.
\label{6.13}
\end{align}
To verify this claim, first note that comparison of (1,2)-block 
entries in \eqref{6.12} yields \eqref{6.13}. 
Equating the bottom corner blocks in \eqref{6.12} then gives
$$
\Bo_j^*({\mathcal G}_{j,C,\bA}\otimes I_d)\Bo_j+\Do_j^*\Do_j=I_{\cU_{n+1-j}},
$$
which in view of \eqref{6.13} can be written as 
\begin{equation}   \label{bB-id}
\Bo_j^*\left({\mathcal G}_{j,C,\bA}\otimes I_d+
({\mathcal G}_{j,C,\bA}\otimes I_d)A{\mathcal G}_{j-1,C,\bA}^{-1}A^*({\mathcal G}_{j,C,\bA}\otimes I_d)
\right)\Bo_j= I_{\cU_{n+1-j}}.
\end{equation}
We claim that \eqref{bB-id} can be simplified to
\begin{equation}   \label{bB-id'}
\Bo_j^*\left({\mathcal G}_{j,C,\bA}^{-1}\otimes I_d-A{\mathcal G}_{j,C,\bA}^{-1}A^*\right)^{-1}\Bo_j
=I_{\cU_{n+1-j}},
\end{equation}
from which \eqref{6.13a} will then follow. 

\smallskip

To see this latter claim, we may use the operator-valued version of the Sherman-Morrison formula: 
given any strictly positive definite $2 \times 2$-block operator matrix $\sbm{ X & Y \\ Y^{*} & Z}$ with 
$X$ and $Z$ invertible, then the Schur complement $Z - Y^{*} X^{-1} Y$ 
is invertible if and only if the other Schur complement $X - Y Z^{-1} Y^{*}$ is 
invertible, and then their inverses are connected via the  
Sherman-Morrison formula  (see e.g.\ \cite{GvL}):
\begin{equation} \label{ShMo}
 (Z - Y^{*} X^{-1} Y)^{-1} = Z^{-1} + Z^{-1} Y^{*} (X - Y Z^{-1} 
 Y^{*})^{-1} Y Z^{-1}.
\end{equation}
We apply the identity \eqref{ShMo} to the case where
$$
\begin{bmatrix} X & Y \\ Y^{*} & Z \end{bmatrix} =\begin{bmatrix} 
\cG_{j,C,\bA} & A^{*} \\ A & \cG_{j,C,\bA}^{-1} \otimes I_{d} 
\end{bmatrix}
$$
to conclude that
\begin{equation}   \label{conclude2}
\big( \cG_{j,C,\bA}^{-1} \otimes I_{d} - A \cG_{j,C,\bA}^{-1} A^{*} \big)^{-1} = 
\cG_{j,C,\bA} \otimes I_{d} + (\cG_{j,C,\bA} \otimes I_{d}) A 
\Delta^{-1} A^{*} (\cG_{j,C,\bA} \otimes I_{d})
\end{equation}
where
$$
  \Delta = \cG_{j,C,\bA} - A^{*} (\cG_{j,C,\bA} \otimes I_{d}) A.
$$
However, a consequence of the (1,1)-entry of the identity 
\eqref{6.12} is that $\Delta = \cG_{j-1,C,\bA}$ which is invertible 
as a consequence of assumption \eqref{assume}.  Substituting this expression for $\Delta$ into 
\eqref{conclude2} gives us the equivalence between \eqref{bB-id} and 
\eqref{bB-id'}, and hence also the verification of \eqref{6.13a}, as 
desired.

\smallskip

We now define the power series
\begin{equation}
\To_j(z)=\Do_j+{\mathcal G}_{j-1,C,\bA}^{\frac{1}{2}}(I-Z(z)A)^{-1}Z(z)\Bo_j
\label{6.14}
\end{equation}  
for $j=2,\ldots,n$, which are strictly inner multipliers from $H^2_{\cU_{n-j+1}}(\free)$ to 
$H^2_{\cX}(\free)$ and satisfy the identities
\begin{align}
&k_{\rm nc, Sz}(z,\zeta)I_{\cY}-
\To_j(z)k_{\rm nc, Sz}(z,\zeta)\To_j(\zeta)^*\notag\\
&={\mathcal G}_{j-1,C,\bA}^{\frac{1}{2}}
(I-Z(z)A)^{-1}{\mathcal G}_{j,C,\bA}^{-1}(I-A^*Z(\zeta)^*)^{-1}{\mathcal G}_{j-1,C,\bA}^{\frac{1}{2}}.
\label{6.15}   
\end{align}
We finally define the series $F_{\ell}\in\cL(\cU_{\ell},\cY)\langle\langle z\rangle\rangle$ via formulas
\begin{align}
  &  F_{\ell}(z) = C (I -Z(z)A)^{-(n-\ell)}\cG^{-\frac{1}{2}}_{n-\ell,C,\bA}
    \To_{n+1-\ell}(z) \quad\text{for} \quad\ell = 1, 2, \dots, n-1,
    \notag \\
    & F_{n}(z) = \To_{1}(z)=\Do_1+C(I-Z(z)A)^{-1}Z(z)\Bo_1, 
    \label{defFell}
 \end{align}
and claim that this choice of $F_{1}, \dots, F_{n}$ satisfies
 \eqref{sumker} (with coefficient spaces $\cU_{j} = \cX$ for $j=1,
 \dots, n-1$). Indeed, substituting equality \eqref{6.15} (for $j=n$) into the
formula \eqref{kMa} for $k_{\cM}$ and then making use of formula \eqref{defFell} for $F_1$,
we have
 \begin{align}
k_{\cM}(z, \zeta) & = k_{{\rm nc }, n}(z,\zeta)I_{\cY}
- C(I -Z(z)A)^{-n+1} \cG^{-\frac{1}{2}}_{n-1,C,\bA} \notag\\
&\qquad\qquad\qquad\qquad\cdot\big[ 
k_{\rm nc, Sz}(z,\zeta)I_{\cX}-
\To_n(z)k_{\rm nc, Sz}(z,\zeta)\To_n(\zeta)^*\big] \notag\\
    & \qquad\qquad\qquad\qquad  \cdot \cG^{-\frac{1}{2}}_{n-1,C,\bA} (I -  A^{*}
Z(\zeta)^*)^{-n+1} C^{*}  \notag\\
     & =  k_{{\rm nc }, n}(z,\zeta)I_{\cY} +
     F_{1}(z)k_{\rm nc, Sz}(z,\zeta) F_{1}(\zeta)^{*}\label{dec20}\\
& \quad- C(I - Z(z)A)^{-n+1}\cG^{-1}_{n-1,C,\bA}k_{\rm nc, Sz}(z,\zeta)
     (I - A^{*}Z(\zeta)^*)^{-n+1} C^{*}.\notag
 \end{align}
We next use equality \eqref{6.15} (for $j=n-1$) and formulas \eqref{ncsz},
\eqref{nc-n-Berg-ker} to get
 \begin{align*} 
& C(I - Z(z)A)^{-n+1}\cG^{-1}_{n-1,C,\bA}k_{\rm nc, Sz}(z,\zeta)
     (I -A^{*}Z(\zeta)^*)^{-n+1} C^{*}\\
& =\sum_{\alpha\in\free}\bzeta^{\alpha^\top}C(I - Z(z)A)^{-n+1}
\cG^{-1}_{n-1,C,\bA}
     (I -A^{*}Z(\zeta)^*)^{-n+1} C^{*}z^\alpha\\
& =\sum_{\alpha\in\free}\bzeta^{\alpha^\top} C(I - 
Z(z)A)^{-n+2}\cG^{-\frac{1}{2}}_{n-2,C,\bA}\\
&\qquad\qquad \cdot \big[
k_{\rm nc, Sz}(z,\zeta)I_{\cX}-
\To_{n-1}(z)k_{\rm nc, Sz}(z,\zeta)\To_{n-1}(\zeta)^*\big] \notag\\   
&\qquad\qquad \cdot\cG^{-\frac{1}{2}}_{n-2,C,\bA} (I -  A^{*}
Z(\zeta)^*)^{-n+2} C^{*}z^\alpha\\
&=C(I - Z(z)A)^{-n+1}\cG^{-1}_{n-1,C,\bA}k_{{\rm nc},2}(z,\zeta)
     (I -A^{*}Z(\zeta)^*)^{-n+1} C^{*}\\
&\quad -F_2(z)k_{{\rm nc},2}(z,\zeta)F_2(\zeta)^*.
\end{align*}  
Substitution of this into \eqref{dec20} gives
 \begin{align*} 
k_{\cM}(z, \zeta) =& k_{{\rm nc }, n}(z,\zeta)I_{\cY} +
     F_{1}(z)k_{\rm nc, Sz}(z,\zeta) F_{1}(\zeta)^{*}+F_2(z)k_{{\rm nc},2}(z,\zeta)F_2(\zeta)^*\\
& \quad- C(I - Z(z)A)^{-n+2}\cG^{-1}_{n-2,C,\bA}k_{{\rm nc}, 2}(z,\zeta)
(I - A^{*}Z(\zeta)^*)^{-n+2} C^{*}.
 \end{align*}
 One can now use an inductive argument to show that in general,
 \begin{align}
k_{\cM}(z, \zeta) =&k_{{\rm nc }, n}(z,\zeta)I_{\cY}
    + \sum_{\ell=1}^{k} F_{\ell}(z)  k_{{\rm nc }, \ell}(z,\zeta)F_{\ell}(\zeta)^{*} \notag \\
&     - C (I - Z(z)A)^{-n + k}\cG^{-1}_{n-k, C,\bA}k_{{\rm nc }, k}(z,\zeta)
     (I - A^*Z(\zeta)^*)^{-n+k} C^{*}
      \label{id-ell}
 \end{align} 
 for $k = 1,2, \dots, n-1$.  For the final step, we use \eqref{11.5a} 
 to see that the last term in \eqref{id-ell} with $k=n-1$ is given by
 \begin{align}
& C(I - Z(z)A)^{-1}\cG^{-1}_{1,C,\bA}k_{{\rm nc }, n-1}(z,\zeta)
(I - A^*Z(\zeta)^*)^{-1} C^{*}\notag\\
&=\sum_{\alpha \in\free}\mu_{n-1,|\alpha|}^{-1}\bzeta^{\alpha^\top}
 C(I - Z(z)A)^{-1}\cG^{-1}_{1,C,\bA}(I - A^*Z(\zeta)^*)^{-1} 
 C^{*}z^\alpha\notag\\
&= \sum_{\alpha \in\free}\mu_{n-1,|\alpha|}^{-1}\bzeta^{\alpha^\top}\big[
k_{\rm nc, Sz}(z,\zeta)I_{\cY}-\To_{1}(z)k_{\rm nc, Sz}(z,\zeta) 
\To_{1}(\zeta)^{*}\big]z^\alpha
\notag\\
   & = k_{{\rm nc }, n}(z,\zeta)I_{\cY}-F_{n}(z)k_{{\rm nc }, n}(z,\zeta)F_{n}(\zeta)^{*},
\label{dec21} 
  \end{align}
where the first equality follows from the formula \eqref{nc-n-Berg-ker}, the second one follows
from \eqref{11.5a} and the last equality follows from \eqref{ncsz1}.
  Plugging \eqref{dec21} into the $k = n-1$ case of \eqref{id-ell} leaves us with
  \eqref{sumker} as wanted.
 \end{proof}
 
\begin{remark} \label{R:9.2}
We remark that the entry $F_n(z)=\To_1(z)$ in \eqref{11.4} is a strictly inner multiplier from  
$H^{2}_{\cU_n}(\free)$ to $H^{2}_{\cY}(\free)$. Therefore, if $n=1$, Theorem \ref{T:BL} amounts to 
Theorem \ref{T:bl4a}.
\end{remark}
To conclude this section, we present an alternate formula for $\To_{j}$ (and hence also for
 $F_{\ell}$) which may prove useful in applications.
  \begin{proposition} \label{L:6.4}
     The function $\To_{j}$ in \eqref{6.14} for $j=2, \dots, n$ can
     alternatively be given by
        \begin{align}  \label{6.16}
\To_j(z)&={\mathcal G}_{j-1,C,\bA}^{\frac{1}{2}}(I-Z(z)A)^{-1}{\mathcal G}_{j,C,\bA}^{-1}(Z(z)-A^*)\notag\\
&\qquad \times({\mathcal G}_{j,C,\bA}^{-1}\otimes I_d-A{\mathcal G}_{j,C,\bA}^{-1}A^*)^{-\frac{1}{2}}W_j.
\end{align}
Hence $F_{\ell}(z)$ can alternatively be given for  $\ell = 1, 2, \dots, n-1$ by
\begin{align}
    F_{\ell}(z) & = C (I - Z(z)A)^{-(n-\ell)}
    \cG^{-\frac{1}{2}}_{n-\ell,C,\bA} \To_{n+1-\ell}(z)  \notag \\
    & = C (I - Z(z)A)^{-(n+1-\ell)} \cG^{-1}_{n+1-\ell,C,\bA}(Z(z) - A^{*})\notag\\
&\qquad\times ({\mathcal G}_{n+1-\ell,C,\bA}^{-1}\otimes
I_d-A{\mathcal G}_{n+1-\ell,C,\bA}^{-1}A^*)^{-\frac{1}{2}}W_{n+1-\ell}.
  \label{Fell-new}
 \end{align}
    \end{proposition}
\begin{proof}
In the proof, we  shorten notation ${\mathcal G}_{j,C,\bA}$ to  ${\mathcal G}_j$.
We substitute \eqref{6.13} into \eqref{6.14} to get
\begin{align}
\To_j(z)=&-{\mathcal G}_{j-1}^{-\frac{1}{2}}A^*({\mathcal
G}_{j}\otimes I_d)\Bo_j+{\mathcal G}_{j-1}^{\frac{1}{2}}(I-Z(z)A)^{-1}Z(z)\Bo_j
\notag\\
=&{\mathcal G}_{j-1}^{\frac{1}{2}}(I-Z(z)A)^{-1}\left[
-(I-Z(z)A){\mathcal G}_{j-1}^{-1}A^*({\mathcal G}_{j}\otimes I_d)+Z(z)\right]\Bo_j\label{6.17}\\
=&{\mathcal G}_{j-1}^{\frac{1}{2}}(I-Z(z)A)^{-1}\left[Z(z)(I+A{\mathcal G}_{j-1}^{-1}A^*({\mathcal
G}_{j}\otimes I_d))-{\mathcal G}_{j-1}^{-1}A^*({\mathcal G}_{j}\otimes I_d)\right]\Bo_j.\notag
\end{align}
To further simplify the latter expression, we make use of equality \eqref{4.10} in the form
$$
{\mathcal G}_{j}-A^*({\mathcal G}_{j}\otimes I_d)A={\mathcal G}_{j-1}.
$$
We first observe that
\begin{align*}
&{\mathcal G}_{j}^{-1}A^*({\mathcal G}_{j}\otimes I_d)\left(I+A{\mathcal G}_{j-1}^{-1}A^*({\mathcal
G}_{j}\otimes I_d)\right)\\
&={\mathcal G}_{j}^{-1}A^*({\mathcal G}_{j}\otimes I_d)+{\mathcal G}_{j}^{-1}A^*({\mathcal G}_{j}\otimes
I_d)A{\mathcal
G}_{j-1}^{-1}A^*({\mathcal G}_{j}\otimes I_d)\\
&={\mathcal G}_{j}^{-1}A^*({\mathcal G}_{j}\otimes I_d)+{\mathcal G}_{j}^{-1}\left({\mathcal G}_{j}-
{\mathcal G}_{j-1}\right){\mathcal G}_{j-1}^{-1}A^*({\mathcal
G}_{j}\otimes I_d) \\
&={\mathcal G}_{j-1}^{-1}A^*({\mathcal G}_{j}\otimes I_d),
\end{align*}
which can be written equivalently as
$$
    \cG_{j-1}^{-1} A^{*} (\cG_{j} \otimes I_{d}) \left(I + A
    \cG_{j-1}^{-1} A^{*} (\cG_{j} \otimes I_{d}) \right)^{-1} =
    \cG_{j}^{-1} A^{*} (\cG_{j} \otimes I_{d}),
$$
and, being combined with \eqref{6.17} gives
\begin{align*}
\To_j(z)=&{\mathcal G}_{j-1}^{\frac{1}{2}}(I-Z(z)A)^{-1}\left[Z(z)-{\mathcal G}_{j}^{-1}A^*
({\mathcal G}_{j}\otimes I_d)\right](I+A{\mathcal G}_{j-1}^{-1}A^*({\mathcal G}_{j}\otimes I_d))\Bo_j
\notag \\
=&{\mathcal G}_{j-1}^{\frac{1}{2}}(I-Z(z)A)^{-1}{\mathcal G}_{j}^{-1}(Z(z)-A^*)
\notag \\
&\quad\times({\mathcal
G}_{j}\otimes I_d+({\mathcal G}_{j}\otimes I_d)A {\mathcal G}_{j-1}^{-1}A^*({\mathcal G}_{j}\otimes
I_d))\Bo_j\notag\\
=&{\mathcal G}_{j-1}^{\frac{1}{2}}(I-Z(z)A)^{-1}{\mathcal G}_{j}^{-1}(Z(z)-A^*)({\mathcal G}_{j}^{-1}\otimes I_d-A{\mathcal G}_{j}^{-1}A^*)^{-1}\Bo_j.
\end{align*}
Substituting \eqref{6.13a} into the latter formula implies \eqref{6.16}. Formula \eqref{Fell-new}
is a straightforward consequence of \eqref{6.16}.
\end{proof}

\chapter[Non-orthogonal Beurling-Lax representations]{Non-orthogonal Beurling-Lax representations
based on wandering subspaces}\label{S:BL-quasi-wandering}

Nonorthogonal representations of a closed shift-invariant subspace $\cM$ of $\cA_{n,\cY}$
based on the wandering subspace $\cE=\cM\ominus {\bf S}_{n}\cM$ appears in \cite{ars} (for $n=2$)
and \cite{sh1} (for $n=3$).
Another version of a non-orthogonal representation of $\cM$ is that of Izuchi-Izuchi-Izuchi \cite{iziziz} (see also \cite{chen}) is
based on the notion of quasi-wandering subspace $\mathcal Q=P_{\cM} S_n\vert_{\cM^\perp}$.
Fleshing out of this approach for the $\bo$-setting is the topic of Section \ref{S:BL-quasi-wanderinga}.

\smallskip

With any ${\bf S}_{\bo,R}$-invariant space $\cM$ isometrically included in $H^2_{\bo,\cY}(\free)$, one can associate 
the {\em wandering subspace} (the notion goes back to \cite{wold,halmos})
\begin{equation}   \label{m3}
\cE=\cM\ominus\bigg(\bigoplus_{j=1}^d S_{\bo,R,j}\cM\bigg)
\end{equation}
and the {\em quasi-wandering} subspace (introduced more recently in \cite{iziziz})
\begin{equation}   \label{m3a}
\cQ=P_{\cM}\bigg(\bigoplus_{j=1}^d S_{\bo,R,j}\cM^\perp\bigg),
\end{equation}
where $P_{\cM}$ denotes the orthogonal projection of $H^2_{\bo,\cY}(\free)$ onto $\cM$ and where 
$\cM^\perp:=H^2_{\bo,\cY}(\free)\ominus\cM$. 
In the case of the Fock space $H^2_{\cY}(\free)$ (i.e., when $\bo=\bmu_1$) these two subspaces coincide and moreover, 
\begin{equation}   \label{m3b}
\cE=\cQ=G\cdot\cU:=\left\{Gu: \, u\in\cU\right\},
\end{equation}
where $G$ is a strictly inner multiplier appearing in the Beurling-Lax representation $\cM=G H^2_{\cU}(\free)$ of $\cM$ 
(see Theorem \ref{T:bl4a}). Furthermore, since $S_{1,R,1},\ldots,S_{1,R,d}$ are isometries (by formula
\eqref{jan11} with $\bo=\bmu_1$) with mutually orthogonal ranges, the subspaces $\bS_{1,R}^\alpha\cM$ and 
$\bS_{1,R}^\beta\cM$ are mutually orthogonal for all $\alpha\neq \beta$ in $\free$ which eventually leads to the 
Wold-type representation 
\begin{equation}
\cM = \bigoplus_{\alpha\in\free}\bS_{1,R}^\alpha\cE=\bigoplus_{\alpha\in\free}\bS_{1,R}^\alpha\cQ.
\label{m3c}
\end{equation}
In the case of a general admissible weight $\bo$, the operators $S_{\bo,R,1},\ldots,S_{\bo,R,d}$ are contractions 
(rather than isometries) with still mutually orthogonal ranges. Therefore, the formulas \eqref{m3} and \eqref{m3a}
still make sense although (even in the case of the Bergman-Fock spaces $\cA_{n,\cY}(\free)$ for $n>1$) they define 
different subspaces $\cE$ and $\cQ$ of a closed $\bS_{\bo,R}$-invariant subspace $\cM$ of $H^2_{\bo,\cY}(\free)$.

\smallskip

In general, the quasi-wandering subspace $\cQ$ and its shifts $\bS_{\bo,R}^\alpha \cQ$ are not orthogonal to each other.
As we will see later, the wandering subspace $\cE$ is orthogonal to all its shifts $\bS_{\bo,R}^\alpha \cE$
($\alpha\neq\emptyset$). However, it is not true anymore that the subspaces $\bS_{\bo,R}^\alpha\cE$ and $\bS_{\bo,R}^\beta\cE$ 
are orthogonal for $\alpha\neq \beta$. Thus, the best one can hope for is to recover a closed $\bS_{\bo,R}$-invariant subspace 
$\cM\subset H^2_{\bo,\cY}(\free)$ from its wondering subspace $\cE$ or the quasi-wandering subspace $\cQ$ via the closed 
linear span of the respective (not-orthogonal) shifted subspaces $\bS_{\bo,R}^\alpha \cE$ and $\bS_{\bo,R}^\alpha \cQ$.
In the rest of this chapter, we will see to which extent this is the case. A secondary question arising from representation 
formulas \eqref{m3b} is whether it is possible to represent (isometrically or at least as sets) the subspaces \eqref{m3} and 
\eqref{m3a} in the form $\cE=\Theta\cdot \cU$ and $\cQ=F\cdot \cU$ for some Hilbert coefficient space $\cU$ and power series $\Theta$ 
and $F$, will be also discussed below.

\section[Beurling-Lax quasi-wandering subspace representations]
{Beurling-Lax representations based on quasi-wandering subspaces}\label{S:BL-quasi-wanderinga}  
Given a closed  ${\bf S}_{\bo,R}$-invariant subspace $\cM$ of $H^2_{\bo,\cY}(\free)$, let $(C,\bA)$ (with $C\in\cL(\cX,\cY)$)
be an $\bo$-isometric pair such that $\cM=({\rm Ran} \, \cO_{\bo,C,\bA})^{\perp}$. Let us consider the power series 
\begin{equation}
F(z)=CR_{\bo}(Z(z)A)(Z(z)-A^*).
\label{13.3g}
\end{equation}
By \eqref{obsreal}, \eqref{ncobsophardy}, the formula for $F{\bf x}$ (${\bf x}\in\cX^d$) can be written more explicitly as
\begin{equation}
F(z){\bf x}
=\sum_{\alpha\in\free} \sum_{j=1}^{d}\omega_{|\alpha|}^{-1}
C\bA^{\alpha}z^\alpha(z_jx_j-A_j^*x_j) \quad ({\bf x}=\sbm{ x_{1} \\ \vdots \\ x_{d}}\in\cX^d),
\label{13.3h}
\end{equation}
which also can be equivalently written in terms of operators $\cO_{\bo,C,\bA}$ and $S_{\bo,R,j}$ as
\begin{equation}
 F{\bf x}+\cO_{\bo,C,\bA}
\sum_{j=1}^{d}A_j^*x_j=\sum_{j=1}^{d}S_{\bo,R,j}\cO_{\bo,C,\bA}x_j.
\label{apr8}
\end{equation}
Due to the intertwining relations \eqref{4.8aga} and since $\cO_{\bo, C, \bA}$ is an isometry from
$\cX$ into $H^{2}_{\bo, \cY}(\free)$, we have for any $x, x_j\in\cX$ and $j=1,\ldots,d$,
\begin{align*}
&\left\langle (S_{\bo,R,j}\cO_{\bo,C,\bA}-\cO_{\bo,C,\bA}A_j^*)x_j, \; \cO_{\bo,C,\bA}x
\right\rangle_{H^{2}_{\bo, \cY}(\free)}\\
&=\left\langle \cO_{\bo,C,\bA}x_j, \; S_{\bo,R,j}^*\cO_{\bo,C,\bA}x
\right\rangle_{H^{2}_{\bo, \cY}(\free)}-
\left\langle \cO_{\bo,C,\bA}A_j^*x_j, \; \cO_{\bo,C,\bA}x
\right\rangle_{H^{2}_{\bo, \cY}(\free)}\\
&=\left\langle \cO_{\bo,C,\bA}x_j, \; \cO_{\bo,C,\bA}A_jx
\right\rangle_{H^{2}_{\bo, \cY}(\free)}-
\left\langle \cO_{\bo,C,\bA}A_j^*x_j, \; \cO_{\bo,C,\bA}x
\right\rangle_{H^{2}_{\bo, \cY}(\free)}\\
&=\left\langle x_j, \; A_jx\right\rangle_{\cX}-
\left\langle A_j^*x_j, \; x \right\rangle_{\cX}=0.
\end{align*}
The latter equalities imply that the $H^{2}_{\bo, \cY}(\free)$-inner product of
the power series \eqref{13.3h} against the vector $\cO_{\bo,C,\bA}x$ equals zero.
Therefore, $F{\bf x}\perp {\rm Ran} \cO_{\bo, C, \bA}=\cM^\perp$ for any $\bx \in \cX^{d}$,
meaning that $F{\bf x}\in \cM$ for any $\bx \in \cX^{d}$. Therefore, the sum on the left side of 
\eqref{apr8} is orthogonal in $H^{2}_{\bo, \cY}(\free)$-metric (as the first term belongs to $\cM$ and the second
is in ${\rm Ran} \cO_{\bo, C, \bA}=\cM^\perp$). Therefore, 
$$
F{\bf x}=P_{\cM}\bigg(\sum_{j=1}^{d}S_{\bo,R,j}\cO_{\bo,C,\bA}x_j\bigg),
$$
and since the expression in the parentheses presents the generic element of the direct sum 
$\bigoplus_{j=1}^d S_{\bo,R,j}\cM^\perp$,
we conclude from \eqref{m3a} that 
\begin{equation}
F(z)\cX^d=P_{\cM}\bigg(\bigoplus_{j=1}^d S_{\bo,R,j}\cM^\perp\bigg)=\cQ
\label{13.8}
\end{equation}
as sets. The following result is the quasi-wandering version of the Beurling-Lax theorem; the proof is 
adapted from the commutative versions \cite{iziziz,chen}. 

\begin{theorem} \label{T:13.1g}
Let us assume that the weight $\bo$ satisfies conditions \eqref{18.2}, \eqref{1.8g}.
Let $\cM$ be a closed ${\bf S}_{\bo,R}$-invariant subspace of $H^2_{\bo,\cY}(\free)$
containing no nontrivial reducing subspaces for ${\bf S}_{\bo,R}$ i.e., $\cM$ is such that
\begin{equation}
\cM\supset H^2_{\bo}(\free)\otimes y
\label{13.2}
\end{equation}
for no nonzero vector $y\in\cY$. Then 
\begin{equation}
\cM=\bigvee_{\alpha \in\free}{\bf S}^\alpha_{\bo,R}\cQ,
\label{13.3gg}
\end{equation}
where $\cQ$ is the quasi-wandering subspace of $\cM$ defined as in \eqref{m3a}.
\end{theorem}
\begin{proof} By definition \eqref{m3a}, $\cQ\subset\cM$. Since $\cM$ is  a closed ${\bf S}_{\bo,R}$-invariant subspace of 
$H^2_{\bo,\cY}(\free)$, the inclusion $\bigvee_{\alpha \in\free}{\bf S}^\alpha_{\bo,R}\cQ\subset\cM$ follows (regardless condition
\eqref{13.2}). 
To show the reverse containment, we take $\cQ$ in the form \eqref{13.8} with $F(z)$ defined by the formula \eqref{13.3g}
in terms of an $\bo$-isometric pair $(C,\bA)$ such that $\cM=({\rm Ran} \, \cO_{\bo,C,\bA})^{\perp}$, and we suppose that there is 
a nonzero $f\in\cM$ which is orthogonal (in the metric of $H^2_{\bo,\cY}(\free)$) to ${\bf S}^\alpha_{\bo,R}\cQ$ for all $\alpha \in\free$:
\begin{equation}
\langle f, \, \cO_{\bo,C,\bA}x\rangle_{H^2_{\bo,\cY}(\free)}=
0=\langle f, \, {\bf S}^\alpha_{\bo,R}F{\bf x}\rangle_{H^2_{\bo,\cY}(\free)}
\label{13.4}
\end{equation}
for all $x\in\cX$, ${\bf x}\in\cX^d$ and $\alpha \in \free$. Thus $f$ is orthogonal to ${\rm Ran} \, \cO_{\bo, C,
\bA}$ as well as to $\bigvee_{\alpha \in \free} \bS^{\alpha}_{n,R} F
\cX^{d}$. 

\smallskip

It follows from \eqref{apr8} that $f$ is orthogonal to $\sum_{j=1}^{d} S_{\bo,R,j}\cO_{\bo,C, \bA}x_{j}$
for arbitrary $x_{j} \in \cX$ ($j=1, \dots, d$) and furthermore, since 
the vectors $x_j$ can be chosen arbitrarily, it follows that $f$ is orthogonal to 
$S_{\bo,R,j}\cO_{\bo,C, \bA}\cX$ for all $j=1,\ldots,d$. Therefore, 
$S_{\bo,R,j}^*f$ is orthogonal to $\cO_{\bo,C, \bA}\cX=\cM^\perp$, that is,
$$
S_{\bo,R,j}^*f\in \cM \quad\mbox{for}\quad j=1,\ldots,d.
$$
Now it follows that conditions \eqref{13.4} hold for the element $S_{\bo,R,j}^*f$ (rather
than $f$ itself) for any fixed $j\in\{1,\ldots,d\}$. Repeating the preceding arguments,
we conclude that $S_{\bo,R,j}^*S_{\bo,R,i}^*f$ belongs to $\cM$ for any
$i,j\in\{1,\ldots,d\}$. The induction argument verifies that
\begin{equation}
{\bf S}_{\bo,R}^{* \alpha} f\in\cM\quad\mbox{for all}\quad \alpha \in\free.
\label{13.6g}
\end{equation}
Since $f\not\equiv 0$, there is a $\beta \in\free$ such that
$f_\beta \neq 0$ ($f_\beta \in\cY$). Recalling the evaluation operator 
$E: \, H^2_{\bo,\cY}(\free)\to \cY$ defined by $Ef=f_\emptyset$ whose adjoint $E^*$ is just the 
inclusion of $\cY$ into $H^2_{\bo, \cY}(\free)$, and invoking formula \eqref{esbo}, we get
$$
E^*E{\bf S}_{\bo,R}^{*\beta}f=\big({\bf S}_{\bo,R}^{*\beta}f\big)_\emptyset=\omega_{|\beta|}f_\beta:=
y\neq 0.
$$
By the first operator equality from \eqref{3.1e} applied to the function ${\bf S}_{\bo,R}^{*\beta}f$
we have
\begin{equation}
\Gamma_{\bo,{\bf S}_{\bo,R}^*}[I_{H^2_{\bo,\cY}(\free)}]{\bf S}_{\bo,R}^{*\beta}f
=E^*E{\bf S}_{\bo,R}^{*\beta}f=y\neq 0.
\label{3.1ea}
\end{equation}
Due to \eqref{13.6g} and since $\cM$ is ${\bf S}_{\bo,R}$-invariant, it follows that 
${\bf S}_{\bo,R}^{\alpha^\top}{\bf S}_{\bo,R}^{*\alpha}{\bf S}_{\bo,R}^{*\beta}f$ belongs to $\cM$ 
for any 
$\alpha\in\free$ as well as the power series 
\begin{equation}
g_N=\sum_{\alpha\in\free:|\alpha|\le N}c_{|\alpha|}{\bf S}_{\bo,R}^{\alpha^\top}
{\bf S}_{\bo,R}^{*\alpha}{\bf 
S}_{\bo,R}^{*\beta}f\quad\mbox{for all}\quad N\ge 1.
\label{3.1eb}
\end{equation}
Since ${\bf S}_{\bo,R}$ is a row contraction, it follows that for each $N\ge 1$,
\begin{align*}
\sum_{\alpha\in\free:|\alpha|\le N}c_{|\alpha|}{\bf S}_{\bo,R}^{\alpha^\top}{\bf S}_{\bo,R}^{*\alpha}&
\preceq \sum_{\alpha\in\free:|\alpha|\le N}|c_{|\alpha|}|{\bf S}_{\bo,R}^{\alpha^\top}
{\bf S}_{\bo,R}^{*\alpha}\\
&=\sum_{j=0}^N |c_{j}|\sum_{\alpha\in\free:|\alpha|=j}{\bf S}_{\bo,R}^{\alpha^\top}
{\bf S}_{\bo,R}^{*\alpha}
\preceq \sum_{j=0}^N|c_{j}|I_{H^2_{\bo,\cY}(\free)},
\end{align*}
and therefore, 
$$
\bigg\|\sum_{\alpha\in\free:|\alpha|\le N}c_{|\alpha|}{\bf S}_{\bo,R}^{\alpha^\top}
{\bf S}_{\bo,R}^{*\alpha}\bigg\|\le
\sum_{j=0}^\infty|c_{j}|=\|R_\bo^{-1}\|_{W^+}<\infty,
$$
due to the assumption \eqref{1.8g}. Thus, the sequence $\{g_N\}_{N\ge 1}\subset\cM$ defined in
\eqref{3.1eb} is uniformly 
bounded (in the metric of $H^2_{\bo,\cY}(\free)$) and admits a subsequential weak limit 
(which is in fact the strong limit)
$$
\sum_{\alpha\in\free}c_{|\alpha|}{\bf S}_{\bo,R}^{\alpha^\top}{\bf S}_{\bo,R}^{*\alpha}{\bf
S}_{\bo,R}^{*\beta}f=\Gamma_{\bo,{\bf 
S}_{\bo,R}^*}[I_{H^2_{\bo,\cY}(\free)}]{\bf S}_{\bo,R}^{*\beta}f
$$
which therefore, also belongs to $\cM$. We now conclude from \eqref{3.1ea} that $y$ belongs to
$\cM$. Then  $H^2_{\bo}(\free)\otimes y$ is a subspace of $\cM$ and is reducing for ${\bf 
S}_{\bo,R}$ which is a contradiction.
\end{proof}

\begin{remark}   \label{R:quasinorm}
Let us consider the gramian-type operator
$$
\Upsilon_{\bo,C,\bA}=\cO_{\bo,C,\bA}^*S_{\bo,R,j}^*S_{\bo,R,j}\cO_{\bo,C,\bA}=
\sum_{\alpha\in\free}\frac{\omega_{|\alpha|+1}}{\omega_{|\alpha|}^2}\bA^{*\alpha^\top}C^*C\bA^\alpha 
$$
which is independent of the choice of $j\in\{1,\ldots,d\}$. Since $S_{\bo,R,j}$ is a contraction, it is clear 
that $\Upsilon_{\bo,C,\bA}\preceq \cG_{\bo,C,\bA}$ and so $\Upsilon_{\bo,C,\bA}$ is a positive semidefinite contraction
if the pair $(C,\bA)$ is $\bo$-isomeric. Since the sum on the left side of \eqref{apr8} is orthogonal, we have
\begin{align*}
\|F{\bf x}\|^2_{H^2_{\bo,\cY}(\free)}&=
\big\|\sum_{j=1}^{d}S_{\bo,R,j}\cO_{\bo,C,\bA}x_j\big\|^2_{H^2_{\bo,\cY}(\free)}-
\|\cO_{\bo,C,\bA}A^*{\bf x}\|^2_{H^2_{\bo,\cY}(\free)}\\
&=\sum_{j=1}^d\|S_{\bo,R,j}\cO_{\bo,C,\bA}x_j\|_{H^2_{\bo,\cY}(\free)}^2-\|A^*{\bf x}\|_{\cX}^2\\
&=\big\langle (\Upsilon_{\bo,C,\bA}\otimes I_d-AA^*){\bf x}, \, {\bf x}\big\rangle_{\cX}.
\end{align*}
Therefore, the operator $\Delta_{\bo,C,\bA}:=\Upsilon_{\bo,C,\bA}\otimes I_d-AA^*$ is positive semidefinite, 
and in representation \eqref{13.8} we can take $F$ of the form 
$$
F(z)=CR_{\bo}(Z(z)A)(Z(z)-A^*)B
$$
where $B$ is any operator from an auxiliary Hilbert space $\cU$ onto $({\rm Ker}(\Delta_{\bo,C,\bA})^\perp$.
Furthermore, $F\cdot\cU=\cQ$ isometrically (contractively) if and only if $B^*\Delta_{\bo,C,\bA}B=I_{\cU}$
(respectively, $B^*\Delta_{\bo,C,\bA}B\preceq I_{\cU}$).
\end{remark}

\begin{remark}   \label{R:F1-quai-wandering}
In case $\bo = \bmu_n$ it is possible to make a connection with the power series $F=F_1$ from the 
representation \eqref{sumker}. We suppose that $(C,\bA)$ is an $n$-isometric pair
representing the ${\bf S}_{n,R}$-invariant  subspace $\cM\subset\cA_{n,\cY}(\free)$ via formula
$\cM=({\rm Ran} \, \cO_{n,C,\bA})^{\perp}$. We again impose the assumption \eqref{assume}.
Thus, $\cG_{n,C,\bA}=I_\cX$ and according to
the formula \eqref{6.13a} (for $j=n$), $\Bo_n=(I-AA^*)^{\frac{1}{2}}W_n$. Combining the latter 
equalities with \eqref{Fell-new} (for $\ell=1$) gives 
\begin{equation}
F(z)=C(I-Z(z)A)^{-n}(Z(z)-A^*)(I-AA^*)^{-\frac{1}{2}}.
\label{13.1}  
\end{equation}
As it follows from \eqref{sumker}, $F$ is a contractive multiplier from $H^2_{\cX^d}(\free)$ to 
$\cA_{n,\cY}(\free)$.
\end{remark}

\begin{remark} \label{R:13.2}
In case $n=1$, the power series \eqref{13.1} is a strictly inner multiplier 
from $H^2_{\cX^d}(\free)$ to 
$H^2_{\cY}(\free)$, since 
\begin{align*}
F(z)=&C(I-Z(z)A)^{-1}(Z(z)-A^*)(I-AA^*)^{-\frac{1}{2}}\\
=&-CA^*(I-AA^*)^{-\frac{1}{2}}\\
&+C(I-Z(z)A)^{-1}\left[Z(z)-A^*+(I-Z(z)A)A^*\right](I-AA^*)^{-\frac{1}{2}}\\
=&-CA^*(I-AA^*)^{-\frac{1}{2}}+C(I-Z(z)A)^{-1}Z(z)(I-AA^*)^{\frac{1}{2}}
\end{align*}
and the connecting operator $\sbm{A & (I-AA^*)^{\frac{1}{2}} \\ C & -CA^*(I-AA^*)^{-\frac{1}{2}}}$
is isometric. Then the representation \eqref{13.3gg} amounts to $\cM=F\cdot H^2_{\cX^d}(\free)$
and hence, Theorem \ref{T:13.1g} reduces to Theorem \ref{T:bl4a}.
\end{remark} 

\begin{remark} \label{R:13.9}
If we substitute formulas \eqref{6.13a}, \eqref{6.13} (for $j=n$) and $\cG_{n,C,\bA}=I_{\cX}$,
 $\cG_{n-1,C,\bA}=I_{\cX}-A^*A$ into \eqref{6.14}, we get
\begin{align*}
\To_n(z)&=-(I-A^*A)^{-\frac{1}{2}}A^*(I-AA^*)^{\frac{1}{2}}W_n \\
&\quad+
(I-A^*A)^{\frac{1}{2}}(I-Z(z)A)^{-1}Z(z)(I-AA^*)^{\frac{1}{2}}W_n\\
&=-A^*W_n+(I-A^*A)^{\frac{1}{2}}(I-Z(z)A)^{-1}Z(z)(I-AA^*)^{\frac{1}{2}}W_n,
\end{align*}
i.e., $\To_n$ is the characteristic function of the {\em row-contraction} $A$. To get $F_n$, we multiply 
$\To_n$ by $C(I-Z(z)A)^{-n+1}$ on the left. If $n=1$, we multiply 
$F_{n}$ just by $C$.  Note that in this case the final formula makes 
sense even without the imposition of the assumption \eqref{assume}.
\end{remark}

\section[Beurling-Lax  wandering subspace representations]
{Non-orthogonal Beurling-Lax representations based on wandering subspaces}
We now take a look at the wandering subspace $\cE$ (see \eqref{m3}) of a closed ${\bf S}_{\bo,R}$-invariant 
space $\cM$ isometrically included into $H^2_{\bo,\cY}(\free)$. As we will see later, the isometric representation
$\cE=\Theta\cdot\cU$ indeed exists and gives rise to the (essentially unique) $H^2_{\bo,\cY}(\free)$-Bergman inner multiplier
$\Theta$. In this subsection we will focus on the analog of the Beurling-Lax representation \eqref{m3c} which
cannot be orthogonal as the subspaces $\bS_{\bo,R}^\alpha\cE$ and $\bS_{\bo,R}^\beta\cE$ (for nonempty
$\alpha\neq \beta$) are not orthogonal in general. Thus, the best one can hope for is to recover a 
$\bS_{\bo,R}$-invariant subspace $\cM$ from its wandering subspace $\cE$ via the closed linear span 
$\cM=\bigvee_{\alpha\in\free}\bS_{\bo,R}^\alpha\cE$.

\smallskip

In the single-variable Bergman space $\cA_{2,\cY}$ this hope was realized in the seminal work
of Aleman, Richter and Sundberg \cite{ars}. Later elaborations due to Shimorin \cite{sh1, sh2} 
showed that the same result holds in the space $\cA_{3,\cY}$ but in general, not in $\cA_{n,\cY}$
for $n>3$. In \cite{iziziz1}, the result from \cite{ars} was recaptured via a substantial 
simplification of Shimorin's approach. In this section, we adapt the 
Izuchis' approach \cite{iziziz1}
to the noncommutative setting of $\cA_{2,\cY}(\free)$. The main result of the section is  
Theorem \ref{T:14.3} below. We start with some needed preliminaries.
\begin{theorem}
Let ${\bf T}=(T_1,\ldots, T_d)\in\cL(\cH)^d$ be a $*$-strongly stable $d$-tuple:
\begin{equation}
\lim_{N\to\infty}\sum_{\alpha\in\free:|\alpha|=N}\|{\bf T}^{*\alpha}h\|^2=0\quad\mbox{for all}\quad 
h\in\cH.
\label{14.0}
\end{equation}
Assume that the operators $T_1,\ldots,T_d$ are left-invertible (that is, 
there is $c>0$ such that 
$\|T_jh\|\ge c_j \|h\|$ for all $h\in\cH$ and $j=1,\ldots,d$) and have orthogonal ranges:
\begin{equation}
\langle T_i h, \, T_jh^\prime\rangle_{\cH}=0 \quad\mbox{for all}\quad h, h^\prime\in\cH\quad\mbox{and}\quad
i\neq j.
\label{14.00}
\end{equation}
Let us also assume that 
\begin{equation}
\|T_jh\|^2+\sum_{i=1}^d \|T_i^*T_j^*T_jh\|^2\le 2\|T_j^*T_jh\|^2
\quad\mbox{for all}\quad h\in\cH
\label{14.1}  
\end{equation}
and $j\in\{1,\ldots,d\}$. Then  $\cH=\bigvee_{\alpha\in\free}{\bf T}^\alpha\mathcal E$ where 
$\mathcal 
E=\cH\ominus\left(\bigoplus_{j=1}^d 
T_j\cH\right)$.
\label{T:14.1}
\end{theorem}

\begin{proof}
Since $T_j$ is bounded below, $T_j\cH$ is a closed subspace of $\cH$. Since $T_j\cH$ is orthogonal to 
$T_i\cH$ for $j\neq i$, the formula for $\mathcal E$ makes sense. Furthermore, the operator 
$(T_j^*T_j)^{-1}$
is bounded and  letting $h=(T_j^*T_j)^{-1}y$ in \eqref{14.1} we conclude that 
$$
\langle (T_j^*T_j)^{-1}y, \, y\rangle_{\cH}+\sum_{i=1}^d \|T_i^*y\|^2_{\cH}\le 2\|y\|_{\cH}^2
$$
for all $y\in\cH$, which can be written in the operator form as 
\begin{equation}
(T_j^*T_j)^{-1}+\sum_{i=1}^d T_iT_i^* \preceq 2 I_{\cH}.
\label{jan13}
\end{equation}
Therefore, $(T_j^*T_j)^{-1}+T_jT_j^* \preceq 2 I_{\cH}$, which implies (see \cite[p. 444]{iziziz1} 
for the proof)
$(T_j^*T_j)^{-1}\succeq I_{\cH}$. Substituting the latter inequality into \eqref{jan13} we conclude that 
${\bf T}=(T_1,\ldots, T_d)\in\cL(\cH)^d$ is a row-contraction:
\begin{equation}
\sum_{j=1}^d\|T_j^*h\|^2\le \|h\|^2\quad\mbox{for all}\quad h\in\cH.
\label{jan13a}
\end{equation}
Let us assume that $\bigvee_{\alpha\in\free}{\bf 
T}^{\alpha}\mathcal E$ is properly contained in $\cH$. Then there is $h\in\cH$ that is orthogonal
to ${\bf T}^{\alpha}\mathcal E$ for each $\alpha\in\free$. For this element, ${\bf T}^{*\alpha}h$ is 
orthogonal to $\mathcal E$ for each $\alpha\in\free$. Hence,
$$
{\bf T}^{*\alpha}h\in\cH\ominus \mathcal E=\bigoplus_{j=1}^d T_j\cH,
$$
so that for every fixed $\alpha\in\free$, there exist $g_{\alpha 1},\ldots,g_{\alpha d}\in\cH$ such 
that 
\begin{equation}
\label{14.3}  
{\bf T}^{*\alpha}h=\sum_{\ell=1}^d T_\ell g_{\alpha\ell}.
\end{equation}
Due to the orthogonality condition \eqref{14.00}, 
$$
\|{\bf T}^{*\alpha}h\|^2=\sum_{j=1}^d\|T_j g_{\alpha j}\|^2.
$$
On the other hand, combining \eqref{14.3} with \eqref{14.00} gives
$$
T_j^*{\bf T}^{*\alpha}h=\sum_{\ell=1}^d T_j^* T_\ell g_{\alpha\ell}=T_j^*T_j 
g_{\alpha j}\quad\mbox{for}\quad j=1,\ldots,d.
$$
Therefore, for each fixed $\alpha\in\free$,
\begin{align}
&\|{\bf T}^{*\alpha}h\|^2-2\sum_{j=1}^d\|T_j^*{\bf T}^{*\alpha}h\|^2+\sum_{i,j=1}^d\|T_i^*T_j^*{\bf 
T}^{*\alpha}h\|^2\notag\\
&=\sum_{j=1}^d\|T_j g_{\alpha j}\|^2-2\sum_{j=1}^d\|T_j^*T_j 
g_{\alpha j}\|^2+\sum_{i,j=1}^d\|T_i^*T_j^*T_j 
g_{\alpha j}\|^2\notag\\
&=\sum_{j=1}^d\bigg(\|T_j g_{\alpha j}\|^2-2\|T_j^*T_j g_{\alpha j}\|^2+\sum_{i=1}^d\|T_i^*T_j^*T_j 
g_{\alpha j}\|^2\bigg)\le 0,\label{14.5}
\end{align}   
where the last inequality follows from \eqref{14.1} (applied to $g_{\alpha j}$ instead of $h$). If 
we let
$$
R_N=\sum_{\alpha\in\free:|\alpha|=N}\|{\bf T}^{*\alpha}h\|^2,
$$
then \eqref{14.0} tells us that $R_N$ tends to zero as $N\to \infty$. Making use of the inequality
\eqref{jan13a} (applied to ${\bf T}^{*\alpha}h$ rather than $h$) we see that
this convergence is monotone:
\begin{align*}
R_{N+1}=\sum_{\alpha\in\free:|\alpha|=N+1}\|{\bf T}^{*\alpha}h\|^2&=\sum_{j=1}^d 
\sum_{\alpha\in\free:|\alpha|=N}\|T_j^*{\bf T}^{*\alpha}h\|^2\\
&=\sum_{\alpha\in\free:|\alpha|=N}\sum_{j=1}^d \|T_j^*{\bf T}^{*\alpha}h\|^2\\
&\le \sum_{\alpha\in\free:|\alpha|=N}\|{\bf T}^{*\alpha}h\|^2=R_N.
\end{align*}
Summing up inequalities \eqref{14.5} over all elements $\alpha\in\free$ with $|\alpha|=N$ we get
\begin{align*}
R_N-2R_{N+1}+R_{N+2} \le 0
\end{align*}  
which together with the preceding inequality implies 
$$
R_{N+1}-R_{N+2} \ge R_N-R_{N+1} \ge 0.
$$
Since $R_N\to 0$ as $N\to\infty$, it follows that $R_{N}=R_{N+1}$ for all $N\ge 0$.  In particular,
$R_N=R_0=\|{\bf T}^{*\emptyset}h\|^2=\|h\|^2\to 0$ as $N\to\infty$ so that $\|h\|=0$.
\end{proof}

\begin{theorem}
The shift $d$-tuple ${\bf S}_{n,R}$ satisfies all assumptions of Theorem \ref{T:14.1} but
\eqref{14.1}. The tuple ${\bf S}_{2,R}$ satisfies equalities
\begin{align}
\sum_{i=1}^d S_{2,R,i}S_{2,R,i}^*S_{2,R,j}^*S_{2,R,j}f&=2S_{2,R,j}^*S_{2,R,j}f-f,\label{14.6a}\\
\|S_{2,R,j}f\|^2+\sum_{i=1}^d \|S_{2,R,i}^*S_{2,R,j}^*S_{2,R,j}f\|^2&= 2\|S_{2,R,j}^*S_{2,R,j}f\|^2
\label{14.6}
\end{align}
for all $f\in\cA_{2,\cY}(\free)$ and $j=1,\ldots,d$.
\label{T:14.2}
\end{theorem}

\begin{proof}
Strong stability of ${\bf S}^*_{n,R}$
was established in Proposition \ref{P:bo-model}. Orthogonality of ${\rm Ran} \, S_{n,R,j}$ and 
${\rm Ran} \, S_{n,R,i}$ for $i\neq j$ is evident. Boundedness below follows from the estimate
$$
\|S_{n,R,j}f\|^2=\sum_{\alpha\in\free}\mu_{n,|\alpha|+1}\|f_\alpha\|^2\ge 
\bigg(\inf_{\alpha\in\free}\frac{\mu_{n,|\alpha|+1}}{\mu_{n,|\alpha|}}\bigg)\cdot
\sum_{\alpha\in\free}\mu_{n,|\alpha|}\|f_\alpha\|^2
=\frac{1}{n}\|f\|^2.
$$
To verify \eqref{14.6}, take $f(z)=\sum_{\alpha\in\free}f_\alpha z^\alpha\in\cA_{n,\cY}(\free)$ and 
observe that by \eqref{jan11},
\begin{equation}
S_{n,R,j}^*S_{n,R,j}f=\sum_{\alpha\in\free}\frac{\mu_{n,|\alpha|+1}}{\mu_{n,|\alpha|}} 
f_{\alpha}z^\alpha,
\label{14.7}
\end{equation}
so that 
$$
\|S_{n,R,j}^*S_{n,R,j}f\|^2=\sum_{\alpha \in \free} \mu_{n,|\alpha|}\left\|
\frac{\mu_{n,|\alpha|+1}}{\mu_{n,|\alpha|}} f_{\alpha}\right\|^2=\sum_{\alpha \in 
\free}\frac{\mu_{n,|\alpha|+1}^2}{\mu_{n,|\alpha|}}
\|f_\alpha\|^2.
$$
From the formulas \eqref{bo-shift} specialized to the present setting of $\bo=\bmu_n$, we have
$$
S_{n,R,i}S_{n,R,i}^*f=\sum_{\alpha\in\free}\frac{\mu_{n,|\alpha|+1}}{\mu_{n,\alpha}}f_{\alpha i}z^{\alpha i}
$$
for $i=1,\ldots,d$, and therefore,
$$
\sum_{i=1}^d S_{n,R,i}S_{n,R,i}^*f
=\sum_{i=1}^d \sum_{\alpha\in\free}\frac{\mu_{n,|\alpha|+1}}{\mu_{n,|\alpha|}}f_{\alpha i}z^{\alpha i}
=\sum_{\alpha\in\free:\alpha\ne \emptyset}\frac{\mu_{n,|\alpha|}}{\mu_{n,|\alpha|-1}}f_{\alpha}z^{\alpha},
$$
which being combined with \eqref{14.7}, gives
\begin{align}
\sum_{i=1}^d 
S_{n,R,i}S_{n,R,i}^*S_{n,R,j}^*S_{n,R,j}f&=\sum_{\alpha\in\free:\alpha\ne \emptyset}
\frac{\mu_{n,|\alpha|}}{\mu_{n,|\alpha|-1}}\cdot 
\frac{\mu_{n,|\alpha|+1}}{\mu_{n,|\alpha|}}f_\alpha z^\alpha\notag\\
&=\sum_{\alpha\neq \emptyset}
\frac{\mu_{n,|\alpha|+1}}{\mu_{n,|\alpha|-1}}f_\alpha z^\alpha.\label{jan12}
\end{align}
Making use of the latter equality and \eqref{14.7} we get
\begin{align}
\sum_{i=1}^d \|S_{n,R,i}^*S_{n,R,j}^*S_{n,R,j}f\|^2&=
\sum_{i=1}^d\langle S_{n,R,i}S_{n,R,i}^*S_{n,R,j}^*S_{n,R,j}f, \, S_{n,R,j}^*S_{n,R,j}f\rangle\notag\\
&=\sum_{\alpha\neq \emptyset}
\frac{\mu^2_{n,|\alpha|+1}}{\mu_{n,|\alpha|-1}}\|f_\alpha\|^2.\label{14.9}
\end{align}
We finally observe that for each $j\in\{1,\ldots,d\}$,
$$
\|S_{n,R,j}f\|^2=\sum_{\alpha\in\free}\mu_{n,|\alpha|+1}\|f_\alpha\|^2
$$
which together with \eqref{14.7} and \eqref{14.9} brings us to
\begin{align}
&\|S_{n,R,j}f\|^2-2\|S_{n,R,j}^*S_{n,R,j}f\|^2+\sum_{i=1}^d \|S_{n,R,i}^*S_{n,R,j}^*S_{n,R,j}f\|^2
\label{jan14e'}\\
&=\sum_{\alpha\neq 
\emptyset}\mu_{n,|\alpha|+1}^2\left(\mu_{n,|\alpha|+1}^{-1}-2\mu_{n,|\alpha|}^{-1}+\mu_{n,|\alpha|-1}^{-1}
\right)\|f_\alpha\|^2  + \left( \mu_{n,1} - 2 \mu_{n,1}^{2} \right) \|f_{\emptyset} \|^{2}. \notag
\end{align}
Making use of the binomial identity
$$
\mu_{n,|\alpha|+1}^{-1}-2\mu_{n,|\alpha|}^{-1}+\mu_{n,|\alpha|-1}^{-1}=\mu_{n-2,|\alpha|+1}^{-1}\quad (\alpha\neq\emptyset)
$$
which is easily verified using the definition \eqref{1.6pre} (or by equating the coefficients in the power-series identity 
$(1-\lambda)^{-n}(1-\lambda)^2=(1-\lambda)^{-(n-2)}$) and taking into account that $\mu_{n,1}=\frac{1}{n}$ we write \eqref{jan14e'}
as
\begin{align*}
&\|S_{n,R,j}f\|^2-2\|S_{n,R,j}^*S_{n,R,j}f\|^2+\sum_{i=1}^d \|S_{n,R,i}^*S_{n,R,j}^*S_{n,R,j}f\|^2    \notag \\
 & = \frac{n-2}{n^2}\|f_\emptyset\|^2+ \sum_{\alpha\neq \emptyset}\frac{ \mu_{n, |\alpha| +1}^2}{\mu_{n-2,|\alpha|+1}}
\| f_{\alpha} \|^{2}  \notag \\
&=\frac{n-2}{n^2}\|f_\emptyset\|^2+\sum_{\alpha\neq 
\emptyset}\frac{(n-1)(n-2)\mu_{n,|\alpha|+1}}{(n+|\alpha|)(n+|\alpha|-1)}\|f_\alpha\|^2.
\end{align*}
\noindent
Letting $n=2$ in the latter equality we get \eqref{14.6}. Letting $n=2$ in \eqref{jan12} we get
$$
\sum_{i=1}^dS_{2,R,i}S_{2,R,i}^*S_{2,R,j}^*S_{2,R,j}f=\sum_{\alpha\neq \emptyset}
\frac{|\alpha|}{|\alpha|+2}f_\alpha 
z^\alpha=\sum_{\alpha\in\free}\frac{|\alpha|}{|\alpha|+2}f_\alpha z^\alpha,
$$
while \eqref{14.7} (for $n=2$) implies 
$$
(2S_{2,R,j}^*S_{2,R,j}-I)f=\sum_{\alpha\in\free}\left(\frac{2|\alpha|+2}{|\alpha|+2}-1\right) 
f_{\alpha}z^\alpha
=\sum_{\alpha\in\free}\frac{|\alpha|}{|\alpha|+2}f_\alpha z^\alpha.
$$
The two latter equalities imply \eqref{14.6a}.
\end{proof}
Now let $\cH=\cM$ be an ${\bf S}_{2,R}$-invariant closed subspace of $\cA_{2,\cY}(\free)$
and let ${\bf T}$ be the restriction of ${\bf S}_{2,R}$ to $\cM$:
\begin{equation}
{\bf T}=(T_1,\ldots,T_d),\quad\mbox{where}\quad 
T_j=S_{2,R,j}\vert_{\cM}\quad\mbox{for}\quad j=1,\ldots,d.
\label{jan12d}
\end{equation}
By Theorem \ref{T:14.2}, ${\bf T}$ is a $*$-strongly stable row contraction and the operators
$T_1,\ldots,T_d$ are bounded below and their ranges are pair-wise orthogonal. Therefore,
${\bf T}$ meets all the assumptions of Theorem \ref{T:14.1} but, perhaps, \eqref{14.1}.
Let us show that it does satisfy \eqref{14.1}. To this end, we first observe that 
$T_j^*=P_{\cM}S_{2,R,j}^*$. We next observe that for any $f\in\cM$,
\begin{align}
S_{2,R,j}^*S_{2,R,j}f&=(P_\cM+P_{\cM^\perp})S_{2,R,j}^*S_{2,R,j}f 
\notag \\
&=T_j^*T_jf+P_{\cM^\perp}S_{2,R,j}^*S_{2,R,j}f,\label{jan12a} \\
S_{2,R,i}^*S_{2,R,j}^*S_{2,R,j}f&=(P_\cM+P_{\cM^\perp})S_{2,R,i}^*(P_\cM+P_{\cM^\perp})S_{2,R,j}^*
S_{2,R,j}f \notag \\ 
&=P_\cM S_{2,R,i}^*P_\cM 
S_{2,R,j}^*S_{2,R,j}f+P_{\cM^\perp}S_{2,R,i}^*S_{2,R,j}^*S_{2,R,j}f\notag\\
&=T_i^*T_j^*T_jf+P_{\cM^\perp}S_{2,R,i}^*S_{2,R,j}^*S_{2,R,j}f,\label{jan12b}
\end{align}
where we used the equality $P_\cM S_{2,R,i}^*\vert_{\cM^{\perp}}=0$ to get \eqref{jan12b}.
Since the terms on the right sides of \eqref{jan12a} and \eqref{jan12b} are orthogonal and since 
$\cM$ is included in $\cA_{2,\cY}(\free)$ isometrically, we have
\begin{align*} 
\|S_{2,R,j}^*S_{2,R,j}f\|^2_{\cA_{2,\cY}(\free)}=&\|T_j^*T_jf\|^2_{\cM}+\|P_{\cM^\perp}S_{2,R,j}^*
S_{2,R,j}f
\|^2_{\cA_{2,\cY}(\free)},\\
\|S_{2,R,i}^*S_{2,R,j}^*S_{2,R,j}f\|^2_{\cA_{2,\cY}(\free)}=&\|T_i^*T_j^*T_jf\|^2_{\cM}+
\|P_{\cM^\perp}S_{2,R,i}^*S_{2,R,j}^*S_{2,R,j}f\|^2_{\cA_{2,\cY}(\free)}.
\end{align*}
Substituting the two latter equalities into \eqref{14.6} gives
\begin{align}
&\|T_jf\|_{\cM}^2+\sum_{i=1}^d \|T_i^*T_j^*T_jf\|^2_{\cM}-2 \|T_j^*T_jf\|^2_{\cM}\notag\\
&=2\|P_{\cM^\perp}S_{2,R,j}^*S_{2,R,j}f
\|^2_{\cA_{2,\cY}(\free)}-\sum_{i=1}^d\|P_{\cM^\perp}S_{2,R,i}^*S_{2,R,j}^*
S_{2,R,j}f\|^2_{\cA_{2,\cY}(\free)}.
\label{jan12c}
\end{align}
In the next calculation, the first equality follows from \eqref{14.6a}, the second one holds since 
$f\in\cM$
and since $P_{\cM^\perp}S_{2,R,j}\vert\cM=0$, the third inequality holds since we drop the projection,
the fourth equality holds since the ranges of $S_{2,R,j}$'s are orthogonal, and the last inequality 
holds since 
$S_{2,R,i}$ is a contraction:
\begin{align}
&2\|P_{\cM^\perp}S_{2,R,j}^*S_{2,R,j}f
\|^2_{\cA_{2,\cY}(\free)}\notag\\
&=\frac{1}{2}\|P_{\cM^\perp}(I+
\sum_{i=1}^d S_{2,R,i}S_{2,R,i}^*S_{2,R,j}^*S_{2,R,j})f\|^2_{\cA_{2,\cY}(\free)}\notag\\
&=\frac{1}{2}\|P_{\cM^\perp}\sum_{i=1}^d
S_{2,R,i}P_{\cM^\perp}S_{2,R,i}^*S_{2,R,j}^*S_{2,R,j}f\|^2_{\cA_{2,\cY}(\free)}\notag\\
&\le \frac{1}{2}\|\sum_{i=1}^d
S_{2,R,i}P_{\cM^\perp}S_{2,R,i}^*S_{2,R,j}^*S_{2,R,j}f\|^2_{\cA_{2,\cY}(\free)}\notag\\
&=\frac{1}{2}\sum_{i=1}^d \|S_{2,R,i}P_{\cM^\perp}S_{2,R,i}^*S_{2,R,j}^*
S_{2,R,j}f\|^2_{\cA_{2,\cY}(\free)}\notag\\
&\le \frac{1}{2}\sum_{i=1}^d \|P_{\cM^\perp}S_{2,R,i}^*S_{2,R,j}^*S_{2,R,j}f\|^2_{\cA_{2,\cY}(\free)}.\notag
\end{align}
Combining this result with \eqref{jan12c} gives us
\begin{align*}
&\|T_jf\|_{\cM}^2+\sum_{i=1}^d \|T_i^*T_j^*T_jf\|^2_{\cM}-2  \|T_j^*T_jf\|^2_{\cM} \\
& \quad \le - \frac{1}{2} \sum_{i=1}^d \|P_{\cM^\perp}S_{2,R,i}^*S_{2,R,j}^*
S_{2,R,j}f\|^2_{\cA_{2,\cY}(\free)}
\le 0.
\end{align*}
We conclude that the tuple \eqref{jan12d} satisfies \eqref{14.1}, that is, all assumptions of Theorem 
\ref{T:14.1}. 
We arrive at the following version of the Beurling-Lax theorem.

\begin{theorem}
Let $\cM$ be an ${\bf S}_{2,R}$-invariant closed subspace of $\cA_{2,\cY}(\free)$. Then 
$$
\cM=\bigvee_{\alpha\in\free}{\bf S}_{2,R}^\alpha\mathcal E,\quad\mbox{where}\quad \mathcal
E=\cM\ominus\bigg(\bigoplus_{j=1}^d S_{2,R,j}\cM\bigg).
$$
\label{T:14.3}
\end{theorem}

\chapter[Orthogonal Beurling-Lax representations]
{Orthogonal Beurling-Lax representations based on wandering subspaces}  \label{S:BLorthog}
In this chapter we present the most comprehensive (in our opinion) and precise version of the Beurling-Lax theorem
for weighted Hardy-Fock spaces based on the notion of a Bergman-inner family and relying on realization formulas for 
such a family. The realization results are then applied to study expansive multiplier properties of Bergman-inner multipliers.

\section{Transfer functions $\Theta_{\bo, \bU_\beta}$ and metric constraints }
\label{S:metric}

In this section we take a closer look at the transfer functions $\Theta_{\bo, \bU_\beta}$  
introduced in Section \ref{S:Hardy} by the realization formula \eqref{thetareal}. 
Let us suppose now that $(C, \bA)$ is an $\bo$-output stable pair and that we construct the formal
power series $\Theta_{\bo, \bU_\beta}$ according to  \eqref{thetareal}, where the operators $B_{j, \beta}$
and $D_\beta$ are still to be determined:
\begin{equation}
\Theta_{\bo, \bU_\beta}(z) = \omega_{|\beta|}^{-1} D_\beta + \sum_{j=1}^d \sum_{\beta' \in \free}
\omega^{-1}_{|\beta| + |\beta'| +1} C \bA^{\beta'} B_{j, \beta} z^{\beta' j}.
\label{jul15}
\end{equation}
In any case, $\Theta_{\bo, \bU_\beta}$ induces the bounded operator $M_{\Theta_{\bo,\bU_\beta}}:\cU_\beta\to H^2_{\bo, \cY}(\free)$. 
We next impose some additional metric relations on the connection matrix \eqref{coll},
specifically one or more of the relations
\begin{align}
\sum_{j=1}^d A_j^*\Gr_{\bo,|\beta|+1,C,\bA}B_{j,\beta}+\omega_{|\beta|}^{-1}\cdot C^*D_\beta &=0,
\label{jul13}\\ 
\sum_{j=1}^d B_{j,\beta}^*\Gr_{\bo,|\beta|+1,C,\bA}B_{j,\beta}+\omega_{\bo,|\beta|}^{-1}\cdot D_\beta^*D_\beta&\preceq
I_{\cU_\beta},\label{jul14}\\
\sum_{j=1}^d B_{j,\beta}^*\Gr_{\bo,|\beta|+1,C,\bA}B_{j,\beta}
+\omega_{|\beta|}^{-1}\cdot D_\beta^*D_\beta&= I_{\cU_\beta},\label{jul14a}
\end{align}
and show how these lead to boundedness and orthogonality properties
for the associated multiplication operator $M_{\Theta_{\bo,\bU_\beta}}$.
Observe that the above relations can be written in terms of the column operators $A$ and 
$\widehat{B}_\beta$ in \eqref{coll} as 
\begin{align*}
A^*(\Gr_{\bo,|\beta|+1,C,\bA}\otimes 
I_d)\widehat{B}_{\beta}+\omega_{\bo,|\beta|}^{-1}\cdot C^*D_\beta &=0,\notag\\
\widehat{B}_{\beta}^*(\Gr_{\bo,|\beta|+1,C,\bA}\otimes 
I_d)\widehat{B}_{\beta}+\omega_{\bo,|\beta|}^{-1}\cdot 
D_\beta^*D_\beta & \preceq I_{\cU_\beta},\notag\\
\widehat{B}_{\beta}^*(\Gr_{\bo,|\beta|+1,C,\bA}\otimes 
I_d)\widehat{B}_{\beta}+\omega_{\bo,|\beta|}^{-1}\cdot 
D_\beta^*D_\beta &=I_{\cU_\beta},\notag
\end{align*}
respectively. Due to equality \eqref{4.34}, it turns out that relations \eqref{jul13} 
and \eqref{jul14} are equivalent to the matrix inequality
\begin{equation}   \label{contr}
  \begin{bmatrix} A^{*} & C^{*} \\ \widehat{B}_{\beta}^{*} & D_{\beta}^{*}
  \end{bmatrix} \begin{bmatrix} \Gr_{\bo,|\beta|+1,C,\bA}\otimes I_d & 0 \\ 0 & 
\omega_{|\beta|}^{-1}\cdot  I_{\cY}
\end{bmatrix}   \begin{bmatrix} A & \widehat{B}_{\beta} \\ C & 
D_{\beta} \end{bmatrix}
 \preceq \begin{bmatrix} \Gr_{\bo,|\beta|,C,\bA} & 0 \\ 0 & I_{\cU_{\beta}} \end{bmatrix},
    \end{equation}
while the equalities \eqref{jul13} and \eqref{jul14a} are equivalent to the matrix equality
\begin{equation}   \label{isom}
  \begin{bmatrix} A^{*} & C^{*} \\ \widehat{B}_{\beta}^{*} & D_{\beta}^{*}
  \end{bmatrix} \begin{bmatrix} \Gr_{\bo,|\beta|+1,C,\bA}\otimes I_d & 0 \\ 0 & 
\omega_{|\beta|}^{-1}\cdot  I_{\cY}
\end{bmatrix}   \begin{bmatrix} A & \widehat{B}_{\beta} \\ C & 
D_{\beta} \end{bmatrix}
 = \begin{bmatrix} \Gr_{\bo,|\beta|,C,\bA} & 0 \\ 0 & I_{\cU_{\beta}} \end{bmatrix}.
    \end{equation}
The two latter conditions are of metric nature; they express the
contractivity or isometric property of the  connection operator 
$ {\bf U}_\beta = \left[ \begin{smallmatrix} A & \widehat{B}_{\beta} \\ C & D_\beta \end{smallmatrix} \right]$
with respect to certain weights. 

\begin{lemma}   \label{L:5.1}
Let $(C,\bA)$ be an $\bo$-output stable pair and let $\Theta_{\bo, \bU_\beta}$ be defined
as in \eqref{thetareal}  for some $\beta \in\free$ and operators 
$B_{1,\beta},\ldots,B_{d,\beta}\in \cL(\cU_\beta, \cX)$ and $D_\beta \in \cL(\cU_\beta,\cY)$.

\smallskip   

$(1)$  If equality \eqref{jul13} holds, then
\begin{itemize}
\item[(a)] $\cO_{\bo,C,\bA}x$ is orthogonal to ${\bf 
S}_{\bo,R}^{\beta^\top}\Theta_{\bo, \bU_\beta}u$ for all 
$\beta \in\free$, $x\in\cX$ and $u\in\cU_\beta$.

\item[(b)] ${\bf S}_{\bo,R}^{\beta^\top}\Theta_{\bo, \bU_\beta}u$ is  orthogonal to
${\bf S}_{\bo,R}^{\beta^{\prime \top}}\Theta_{\bo, \bU_\beta'}u'$ for all
$u \in \cU_\beta$, $u' \in \cU_{\beta'}$ and all $\beta \ne \beta'$ in $\free$. 
 
\item[(c)] ${\bf S}_{\bo,R}^{(\gamma \beta)^\top}\Theta_{\bo, \bU_\beta}u$ and ${\bf 
S}_{\bo,R}^{(\beta \gamma)^\top}\Theta_{\bo, \bU_\beta}u$ are both orthogonal to 
${\bf S}_{\bo,R}^{\beta^\top}\Theta_{\bo, \bU_\beta}u^\prime$ for all 
$\gamma\neq\emptyset$ and $\beta$ in $\free$ and for any $u,u^\prime\in\cU_\beta$.
\end{itemize}

$(2)$ Moreover:
\begin{itemize}
    \item [(a)] If inequality \eqref{jul14} holds, then the 
    operator ${\bf S}_{\bo,R}^{\beta^\top} M_{\Theta_{\bo, \bU_\beta}}$ is a contraction from
    $\cU_{\beta}$ into $H^2_{\bo,\cY}(\free)$.

    \item[(b)] If both \eqref{jul13} and \eqref{jul14} hold, i.e., if
    \eqref{contr} holds, then the operator
    ${\bf S}_{\bo,R}^{\beta^\top}M_{\Theta_{\bo, \bU_\beta}}$ is a contraction from the Fock space
    $H^{2}_{\cU_\beta}(\free)$ into $H^2_{\bo,\cY}(\free)$.
    \end{itemize}

\smallskip

$(3)$ Similarly:
\begin{itemize}
    \item[(a)]  If equality \eqref{jul14a} holds, then the operator 
    ${\bf S}_{\bo,R}^{\beta^\top}M_{\Theta_{\bo, \bU_\beta}}$ is an isometry 
    from $\cU_{\beta}$ into $H^2_{\bo,\cY}(\free)$.

    \item[(b)]  If \eqref{jul13} and \eqref{jul14a} hold, i.e., if
    \eqref{isom} holds, then for every $f\in H^2_{\cU_\beta}(\free)$,
\begin{align}
&\|{\bf S}_{\bo,R}^{\beta^\top}\Theta_{\bo,  \bU_\beta}f\|_{H^2_{\bo,\cY}(\free)}^2=
\|f\|^2_{H^2_{\cU_\beta}(\free)} \label{SThetank-isom}\\
&\qquad\qquad-\sum_{\gamma\in\free}\sum_{j=1}^d
\left\|(I-S_{\bo,R,j}^*S_{\bo,R,j})^{\frac{1}{2}}{\bf
S}_{\bo,R}^{\beta^\top}\Theta_{\bo, \bU_\beta}({\bf S}_{1,R}^{*})^{j\gamma^\top}f
\right\|^2_{H^2_{\bo,\cY}(\free)}.\notag
\end{align}
    \item[(c)] If \eqref{isom} holds, then
\begin{align}
&\omega_{|\beta|}^{-1}I_{\cU_{\beta}} - 
\Theta_{\bo, \bU_\beta}(z)^*\Theta_{\bo, \bU_\beta}(\zeta) \notag\\ 
&=\omega_{|\beta|}\widehat{B}_{\beta}^{*}R_{\bo,|\beta|}(AZ(z))^{*}
(\Gr_{\bo,|\beta|+1,C,\bA}\otimes I_d)R_{\bo,|\beta|}(AZ(\zeta))
\widehat{B}_{\beta}\label{jul18}\\
&\quad-\omega_{|\beta|}\widehat{B}_{\beta}^{*}R_{\bo,|\beta|+1}(AZ(z))^{*}
Z(z)^*\Gr_{\bo,|\beta|,C,\bA}Z(\zeta)
R_{\bo,|\beta|+1}(AZ(\zeta))\widehat{B}_{\beta}.\notag
\end{align}
\end{itemize}  
\end{lemma}

\begin{remark}  \label{R:content}
In the terminology of Section 3.3, the content of parts 
(2b) and (3a) of Lemma \ref{L:5.1} can be phrased as follows: {\em if 
\eqref{isom} holds, then the formal power series $\Theta_{\bo, \bU_\beta}(z) 
z^{\beta}$ is $H^2_{\bo,\cY}(\free)$-Bergman-inner.}    
\end{remark}

\begin{proof}[Proof of (1):] From the power series expansion \eqref{jul15} we have
\begin{equation}
{\bf S}_{\bo,R}^{\beta^\top}\Theta_{\bo,\bU_\beta}(z)=\omega_{\bo,|\beta|}^{-1}D_\beta z^\beta +
\sum_{j=1}^d \sum_{\beta^\prime\in\free}      
\omega_{|\beta^{\prime}|+|\beta|+1}^{-1}
C{\bf A}^{\beta^{\prime}}B_{j,\beta}z^{\beta^{\prime}j \beta}.
\label{jul15a}
\end{equation}
We then make use of expansions \eqref{18.2aaa}, \eqref{jul15a} and the definition of 
the inner product in $H^2_{\bo,\cY}(\free)$ to get
\begin{align*}
&\left\langle {\bf S}_{\bo,R}^{\beta^\top}\Theta_{\bo, \bU_\beta} u, \, \cO_{\bo,C,\bA}x
\right\rangle_{H^2_{\bo,\cY}(\free)}\\
&=\omega_{|\beta|}^{-1}\cdot\left\langle D_\beta  u, \, 
C\bA^{\beta}x\right\rangle_{\cY}  +\sum_{j=1}^d
\sum_{\beta^\prime\in\free}\omega_{|\beta^{\prime}|+|\beta|+1}^{-1}\cdot 
\left\langle C\bA^{\beta^\prime}B_{j,\beta}u, \, C\bA^{\beta^\prime j \beta}x
\right\rangle_{\cY}  \\
&=\bigg\langle \bigg(\omega_{|\beta|}^{-1}C^*D_\beta +   
\sum_{j=1}^d A_j^*\bigg(\sum_{\beta^\prime\in\free}
\omega_{|\beta^{\prime}|+|\beta|+1}^{-1} 
\bA^{*\beta^{\prime\top}}C^*C\bA^{\beta^\prime}\bigg)
B_{j,\beta}\bigg)u,\, \bA^{\beta}x\bigg\rangle_{\cX}  \\
&=\bigg\langle \bigg(\omega_{|\beta|}^{-1}C^*D_\beta +\sum_{j=1}^d
A_j^*\Gr_{\bo,|\beta|+1,C,\bA}B_{j,\beta}\bigg) u,\, \bA^{\beta}x\bigg\rangle_{\cX}=0 
\end{align*}
where the fourth and the fifth equality follow from \eqref{4.32} and 
\eqref{jul13}, respectively. The latter computation verifies part (1a).

\smallskip

The verification of (1b) goes through several cases.

\smallskip

If $|\beta| = |\beta'|$ and $\beta \ne \beta'$,  then any word of the form $\alpha \beta$ is distinct
from any work of the form $\alpha' \beta$, where $\alpha$ and $\alpha'$ are independently arbitrary in 
$\free$.  Note that the power series representations of $\bS^{\beta^\top}_{\bo,R}\Theta_{\bo, \bU_\beta} u$ and 
$\bS^{\beta^{\prime \top}}_{\bo,R} \Theta_{\bo, \bU_\beta'} u'$ have the form
$$
\bS^{\beta^\top} \Theta_{\bo, \bU_\beta} u= 
\sum_{\alpha \in \free} [ \Theta_\bo, \bU_\beta]_\alpha  u z^{\alpha \beta}, \quad
\bS^{\beta^{\prime \top}} \Theta_{\bo, \bu_\beta'} u'= 
\sum_{\alpha' \in \free} [ \Theta_{\bo, \bU_\beta'}]_{\alpha'}  u' z^{\alpha' \beta},
$$
so the orthogonality $\bS^{\beta^\top}_{\bo,R} \Theta_{\bo, \bU_\beta} u \perp \bS^{\beta^{\prime \top}}_{\bo,R} \Theta_{\bo, \bU_\beta'} u'$
holds due to the pairwise orthogonality of the monomials in the respective power series expansions.
Hence we now need only consider the case where $|\beta| \ne | \beta'|$.

\smallskip

Without loss of generality we may assume that $|\beta| > |\beta'|$, so $\beta = \delta \beta''$ where $\delta \ne \emptyset$
and $|\beta''| = |\beta'|$.  In case $\beta' \ne \beta''$, then every word of the form $\alpha \beta = \alpha \delta \beta''$
is distinct from any word of the form $\alpha' \beta'$ for independently arbitrary words $\alpha$ and $\alpha'$ in $\free$,
since the respective right tails of length $|\beta''| = |\beta'|$ disagree.  Again we have pairwise orthogonality of the
monomials in the respective power series expansions for $\bS^{\beta^\top}_{\bo,R} \Theta_{\bo, \bU_\beta} u$ and
$\bS^{\beta^{\prime \top}}_{\bo,R} \Theta_{\bo, \bU_\beta'} u'$, so the desired orthogonality 
 $\bS^{\beta^\top}_{\bo,R} \Theta_{\bo, \bU_\beta} u \perp \bS^{\beta^{\prime \top}}_{\bo,R} \Theta_{\bo, \bU_{\beta'}} u'$ again holds.
 
\smallskip

 It remains only to consider the case where $\beta$ has the form $\beta = \delta \beta'$ where $\delta \ne \emptyset$.
In this case the pairwise orthogonal of the monomials in the respective power series expansion fails and we must
do a detailed calculation of the inner product which is supposed to be zero. Toward this end, we use the formula
\eqref{jul15a} to see that
\begin{align*}
\big(\bS_{\bo,R}^{\beta^{\prime \top}} \Theta_{\bo, \bU_{\beta'}} u' \big)(z) &=
\omega^{-1}_{|\beta'|} D_{\beta'}u' z^{\beta'} +
\sum_{j'=1}^d \sum_{\gamma' \in \free} \omega^{-1}_{|\gamma'| + |\beta'| + 1}
C \bA^{\gamma'} B_{j',\beta'}  z^{\gamma' j' \beta'} u',\\
\big( \bS^{\beta^\top}_{\bo, R} \Theta_{\bo, \bU_\beta} u \big)(z) & =
 \big( \bS^{(\delta \beta^{\prime})^\top}_{\bo, R} \Theta_{\bo, \bU_{\delta \beta'}} u \big)(z) \\
 &= \omega_{|\delta| + |\beta'|}^{-1} D_{\delta \beta'}u z^{\delta \beta'} + \sum_{j=1}^d \sum_{\gamma \in \free}
 \omega^{-1}_{|\gamma| + |\delta| + |\beta'| +1} C \bA^\gamma B_{j, \delta \beta'} z^{\gamma j \delta \beta'}u.
\end{align*}
It is convenient to write the word $\delta$ in the form $\delta = \widetilde \delta \widetilde j$
where $\widetilde \delta$ is a (possibly empty) word and $\widetilde{j}$ is the right-most letter in the nonempty word $\delta$.
Using the latter representations we can now compute
\begin{align*}
& \big\langle \bS^{\beta^\top}_{\bo,R} \Theta_{\bo, \bU_\beta} u, \,
\bS^{\beta^{\prime \top}}_{\bo,R} \Theta_{\bo, \bU_{\beta'}} u' \big\rangle_{H^2_{\bo, \cY}(\free)} \\
& =\omega^{-1}_{|\delta| + |\beta'|} \big\langle D_{\delta \beta'} u,  \, C \bA^{\widetilde \delta} 
B_{\widetilde j, \beta'} u' \big\rangle_{\cY} \\
& \quad \quad \quad \quad 
+ \sum_{j=1}^d \sum_{\gamma \in \free} \omega^{-1}_{|\gamma| + |\delta| + |\beta'| + 1}
\big\langle C \bA^{\gamma } B_{j, \delta \beta'} u, \,
C \bA^{\gamma j \widetilde \delta}  B_{\widetilde j, \beta'} u' \big\rangle_{\cY}   \\
& =  \big\langle\omega^{-1}_{|\delta| + |\beta'|}C^* D_{\delta \beta'}u +
\sum_{j=1}^d A_j^* \bigg( \sum_{\gamma \in \free} \omega^{-1}_{|\gamma| + |\delta| + \beta'| + 1}
\bA^{* \gamma^\top} C^* C \bA^\gamma \bigg) B_{j,  \delta \beta'}u, \,
\bA^{\widetilde \delta} B_{\widetilde j, \beta'} u' \big\rangle \\
& =  \big\langle\big(\omega^{-1}_{|\delta| + |\beta'|}C^* D_{\delta \beta'} +
\sum_{j=1}^d A_j^* \Gr_{\bo, |\delta| +|\beta'| + 1, C, \bA}B_{j,  \delta \beta'}\big)u, \,
\bA^{\widetilde \delta} B_{\widetilde j, \beta'} u' \big\rangle_{\cX}=0, 
\end{align*}
where the two last steps follow respectively from the definition \eqref{4.32} of $\Gr_{\bo,k, C, \bA}$
and equality \eqref{jul13} with $\beta = \delta \beta'$.  This completes the verification of (1b).

\smallskip

Verification of part $(1c)$ is quite similar: let us denote
by $\widetilde{j}\in\{1,\ldots,d\}$ the rightmost letter in the given $\gamma \neq\emptyset$ 
so that $\gamma =\widetilde{\gamma}\widetilde{j}$. We then see from \eqref{jul15a} and a similar 
expansion for ${\bf S}_{\bo,R}^{(\gamma \beta)^\top}\Theta_{\bo,\bU_\beta}$ that
\begin{align*}
&\big\langle {\bf S}_{\bo,R}^{(\gamma \beta)^\top}\Theta_{\bo,\beta} u, 
\, {\bf S}_{\bo,R}^{\beta^\top}\Theta_{\bo, \bU_\beta} u^\prime\big\rangle_{H^2_{\bo,\cY}(\free)}\\
& =  \omega_{|\beta| + |\gamma|}\cdot \big\langle \omega_{|\beta|}^{-1} D_\beta   u, \,  
\omega_{|\beta| + |\gamma|}^{-1} 
C\bA^{\widetilde{\gamma}}B_{\widetilde{j}, \beta}u^\prime\big\rangle_{\cY} \\
&\quad  +\sum_{j=1}^d\sum_{\beta^\prime\in\free} \omega_{|\beta'| + 1 
+ |\widetilde \gamma|}\cdot
\big\langle \omega_{|\beta^{\prime}|+|\beta|+1}^{-1} 
C\bA^{\beta^\prime}B_{j,\beta} u, \, 
\omega_{|\beta'| + 1 + |\widetilde \gamma|}^{-1} C\bA^{\beta^\prime j\widetilde{\gamma}
}B_{\widetilde{j},\beta} u^\prime\big\rangle_{\cY}  \\
 & =   \omega_{|\beta|}^{-1} \big\langle D_\beta u, \,  
C\bA^{\widetilde{\gamma}}B_{\widetilde{j},\beta} u^\prime\big\rangle_{\cY}
+\sum_{j=1}^d\sum_{\beta^\prime\in\free}\omega_{|\beta^{\prime}|+|\beta|+1}^{-1}
\big\langle C\bA^{\beta^\prime}B_{j,\beta}u, \, C\bA^{\beta^\prime 
j\widetilde{\gamma}}B_{\widetilde{j},\beta} u^\prime\big\rangle_{\cY}  \\
&=\big\langle \big(\omega_{|\beta|}^{-1}C^*D_\beta +
\sum_{j=1}^d A_j^*\bigg(\sum_{\beta^\prime\in\free}
\omega_{|\beta^{\prime}|+|\beta|+1}^{-1} 
\bA^{*\beta^{\prime\top}}C^*C\bA^{\beta^\prime}\bigg)
B_{j,\beta}\big)u,\, 
\bA^{\widetilde{\gamma}}B_{\widetilde{j},\beta} u^\prime\big\rangle_{\cX}  \\
&=\big\langle \big(\omega_{|\beta|}^{-1}C^*D_\beta +\sum_{j=1}^d
A_j^*\Gr_{\bo,|\beta|+1,C,\bA}B_{j,\beta}\big)u,\,  
\bA^{\widetilde{\gamma}}B_{\widetilde{j},\beta} u^\prime\big\rangle_{\cX}=0,
\end{align*}
where the two last equalities follow again from \eqref{4.32} and \eqref{jul13}. To 
prove the orthogonality of ${\bf S}_{\bo,R}^{(\beta \gamma)^\top}\Theta_{\bo, \bU_\beta} u$ and 
${\bf S}_{\bo,R}^{\beta^\top}\Theta_{\bo,  \bU_\beta} u^\prime$, we first observe that 
if $\beta$ does not divide $\gamma$ on the right, that is, $\gamma$ 
is not of the form $\gamma=\gamma^\prime \beta$, then every monomial in  
${\bf S}_{\bo,R}^{(\beta \gamma)^\top}\Theta_{\bo,\bU_\beta}u$ is orthogonal to every monomial in
${\bf S}_{\bo,R}^{\beta^\top}\Theta_{\bo, \bU_\beta}u'$ and therefore, the desired orthogonality 
follows. 
On the other hand, if
$\gamma=\widetilde{\gamma}\widetilde{j}\beta$ for some 
$\widetilde{\gamma}\in\free$ (the case $\widetilde{\gamma}=\emptyset$ is not excluded), then the 
calculation similar to the previous one gives
\begin{align*}
&\big\langle {\bf S}_{\bo,R}^{(\beta \gamma)^\top}\Theta_{\bo,\bU_\beta} u, 
\, {\bf S}_{\bo,R}^{\beta^\top}\Theta_{\bo,\bU_\beta}
u^\prime\big\rangle_{H^2_{\bo,\cY}(\free)}\\
&=\omega_{|\beta|}^{-1}\cdot\big\langle D_\beta u, \,
C\bA^{\beta \widetilde{\gamma}}B_{\widetilde{j}\beta} u^\prime\big\rangle_{\cY}
+\sum_{j=1}^d\sum_{\beta^\prime\in\free}\omega_{|\beta^{\prime}|+|\beta|+1}^{-1}
\big\langle C\bA^{\beta^\prime}B_{j,\beta}u, \, C\bA^{\beta^\prime 
j\beta \widetilde{\gamma}}B_{\widetilde{j},\beta} u^\prime\big\rangle_{\cY}  \\
&=\big\langle \omega_{|\beta|}^{-1} C^*D_\beta u, \,
\bA^{\beta \widetilde{\gamma}}B_{\widetilde{j}\beta} u^\prime\big\rangle_{\cY}\\
&\qquad+\bigg\langle 
\sum_{j=1}^d A_j^*\bigg(\sum_{\beta^\prime\in\free}
\omega_{|\beta^{\prime}|+|\beta|+1}^{-1} 
\bA^{*\beta^{\prime\top}}C^*C\bA^{\beta^\prime}\bigg)
B_{j,\beta}u,\, \bA^{\beta\widetilde{\gamma}}B_{\widetilde{j},\beta} u^\prime\bigg\rangle_{\cX}  \\
&=\big\langle \big(\omega_{|\beta|}^{-1}C^*D_\beta +\sum_{j=1}^d
A_j^*\Gr_{\bo,|\beta|+1,C,\bA}B_{j,\beta}\big)u,\,
\bA^{\beta \widetilde{\gamma}}B_{\widetilde{j},\beta}u^\prime\big\rangle_{\cY}=0
\end{align*}
and completes the proof of part (1c). 
\end{proof}
\begin{proof}[Proof of (2):] By \eqref{jul15a} 
we have 
\begin{equation}
\| {\bf S}_{\bo,R}^{\beta^\top}\Theta_{\bo,\bU_\beta}u\|^2_{H^2_{\bo,\cY}(\free)}
=\omega_{|\beta|}^{-1}\cdot \| D_\beta u\|^2_{\cY}+\sum_{j=1}^d
\sum_{\beta^\prime\in\free}\omega_{|\beta^{\prime}|+|\beta|+1}^{-1}\big\|
C\bA^{\beta^\prime}B_{j,\beta}u\big\|^2_{\cY}.\label{jul19}
\end{equation}
Since due to \eqref{4.32}, the second term on the right side can be written as 
\begin{align*}
&\sum_{j=1}^d\sum_{\beta^\prime\in\free}\omega_{|\beta^{\prime}|+|\beta|+1}^{-1}\big\langle
C\bA^{\beta^\prime}B_{j,\beta}u, \, C\bA^{\beta^\prime}B_{j,\beta}u\big\rangle_{\cX}\\
&=\bigg\langle \sum_{j=1}^d B_{j,\beta}^*\bigg(
\sum_{\beta^\prime\in\free}\omega_{|\beta^{\prime}|+|\beta|+1}^{-1}
\bA^{*\beta^{\prime\top}}C^*C\bA^{\beta^\prime}\bigg)B_{j,\beta}u,
\, u\bigg\rangle_{\cU}\\
&=\bigg\langle \sum_{j=1}^d 
B_{j,\beta}^*\Gr_{\bo,|\beta|+1,C,\bA}B_{j,\beta}u,\, u\bigg\rangle_{\cU},
\end{align*}
and since, according to \eqref{jul14},
\begin{equation}
\omega_{|\beta|}^{-1}\cdot \| D_\beta u\|^2_{\cY}+ 
\sum_{j=1}^d\left\langle\Gr_{\bo,|\beta|+1,C,\bA}B_{j,\beta}u, \;
B_{j,\beta}u\right\rangle_{\cX} \le \|u\|^2_{\cU_\beta}
\label{jul20}  
\end{equation}
for all $u\in\cU_\beta$, it now follows from \eqref{jul19} that 
$$
\| {\bf S}_{\bo,R}^{\beta^\top}\Theta_{\bo,\bU_\beta}u\|^2_{H^2_{\bo,\cY}(\free)}\le 
\|u\|^2_{\cU_\beta}.
$$
Thus, ${\bf S}_{\bo,R}^{\beta^\top} M_{\Theta_{\bo, \bU_\beta}}$  is a 
contraction from $\cU_\beta$ to $H^2_{\bo,\cY}(\free)$ and part (2a) follows.
By statements (1c) and (2a) of the lemma, Theorem \ref{T:12.2} applies to the series 
${\bf S}_{\bo,R}^{\beta^\top}\Theta_{\bo, \bU_\beta}$ which is therefore, a contractive multiplier 
from $H^2_{\cU_\beta}$ to $H^2_{\bo,\cY}(\free)$. This completes the proof of (2b) of the lemma.
\end{proof}

\begin{proof}[Proof of (3):]  In case \eqref{jul14a} holds,  then
    \eqref{jul19} and \eqref{jul20} hold with equalities and part (a) of (3) follows.  If 
also \eqref{jul13} holds, then Theorem \ref{T:12.2} applies to the series 
$\Theta={\bf S}_{\bo,R}^{\beta^\top}\Theta_{\bo, \bU_\beta}$ and equality 
\eqref{SThetank-isom} follows from \eqref{march13}.

\smallskip

    It remains to verify the formula \eqref{jul18} under assumption
    \eqref{isom} which is equivalent to the system of equations \eqref{jul13}, 
\eqref{jul14a} and \eqref{4.34}. We use these relations to compute
\begin{align}
& \omega_{|\beta|}^{-1}I_{\cU_{\beta}}- 
\Theta_{\bo, \bU_\beta}(z)^{*}\Theta_{\bo, \bU_\beta}(\zeta)\notag\\
&=\omega_{|\beta|}^{-1}I_{\cU_{\beta}}-\left[ \omega_{|\beta|}^{-1}D_{\beta}^{*} + \widehat{B}_{\beta}^{*}Z(z)^*
    R_{\bo,|\beta|+1}(Z(z)A)^{*} C^{*} \right]\notag\\
&\qquad\qquad\qquad\times    \left[\omega_{|\beta|}^{-1}D_{\beta}+C 
R_{\bo,|\beta|+1}(Z(\zeta) A)Z(\zeta) \widehat{B}_{\beta}\right]  \notag\\
&   = \omega_{|\beta|}^{-1}I_{\cU_{\beta}}-\omega_{|\beta|}^{-2}D_{\beta}^{*}D_{\beta} 
- \omega_{|\beta|}^{-1}\widehat{B}_{\beta}^{*}Z(z)^* 
R_{\bo,|\beta|+1}(Z(z)A)^{*} C^{*}D_\beta \notag\\
&\quad-\omega_{|\beta|}^{-1}D_\beta^*C R_{\bo,|\beta|+1}(Z(\zeta) 
A)Z(\zeta) \widehat{B}_{\beta}\notag\\
    & \quad  - \widehat{B}_{\beta}^{*}Z(z)^*
    R_{\bo,|\beta|+1}(Z(z)A)^{*} C^{*}C R_{\bo,|\beta|+1}(Z(\zeta)A) 
    Z(\zeta)\widehat{B}_{\beta}\notag\\
&=\omega_{|\beta|}\widehat{B}_{\beta}^{*}\left(\omega_{|\beta|}^{-1}I_{\cX^d}+Z(z)^* 
R_{\bo,|\beta|+1}(Z(z)A)^{*}A^*\right)(\Gr_{\bo,|\beta|+1,C,\bA}\otimes I_d)\notag\\
&\qquad\qquad\qquad\times\left(\omega_{|\beta|}^{-1}I_{\cX^d}+AR_{\bo,|\beta|+1}
(Z(\zeta) A) Z(\zeta)\right)\widehat{B}_{\beta}\label{lov4}\\
&\quad -\omega_{|\beta|}\widehat{B}_{\beta}^{*}Z(z)^*R_{\bo,|\beta|+1}(Z(z)A)^{*}
\Gr_{\bo,|\beta|,C,\bA}R_{\bo,|\beta|+1}(Z(\zeta) A)Z(\zeta)  \widehat{B}_{\beta}.\notag
\end{align}
We now observe that 
$$
R_{\bo,|\beta|+1}(Z(\zeta) A)Z(\zeta)=Z(\zeta)R_{\bo,|\beta|+1}(AZ(\zeta)),
$$
which together with \eqref{1.9pre} implies
\begin{align*}
\omega_{|\beta|}^{-1}I_{\cX^d}+AR_{\bo,|\beta|+1}(Z(\zeta) A) Z(\zeta)&=
\omega_{|\beta|}^{-1}I_{\cX^d}+AZ(\zeta)R_{\bo,|\beta|+1}(AZ(\zeta))\\
&=R_{\bo,|\beta|}(AZ(\zeta)).
\end{align*}
Substituting the two latter formulas into \eqref{lov4} verifies formula  
\eqref{jul18}.
\end{proof}
 The following result is an immediate consequence of Lemma \ref{L:5.1}.

\begin{corollary}   \label{C:5.3}
Let us assume that the pair $(C,\bA)$ is $\bo$-output stable and that relations \eqref{jul13},
\eqref{jul14} hold for all $\beta \in\free$. Then the representation \eqref{1.36prehardy} of the
function $\widehat{y}$ is orthogonal in the metric of $H^2_{\bo,\cY}(\free)$ and
\begin{align}
\|\widehat y\|^2_{H^2_{\bo,\cY}(\free)}&=\|\cO_{\bo,C,\bA}x\|_{H^2_{\bo,\cY}(\free)}^2+
\sum_{\beta\in\free}\|{\bf S}_{\bo,R}^{\beta^\top}\Theta_{\bo, \bU_\beta} 
u_\beta \|_{H^2_{\bo,\cY}(\free)}^2\notag\\
&\le\|\cG_{\bo,C,\bA}^{\frac{1}{2}}x\|^2_{\cX}+\sum_{\beta 
\in\free}\|u_\beta\|^2_{\cU_\beta}. \label{jul17}
\end{align}
If relations \eqref{jul14} hold with equalities for all $\beta \in\free$, then equality holds in
\eqref{jul17}.
\end{corollary}

Let us now impose the additional hypothesis that $(C, \bA)$ is exactly $\bo$-observable, so that we can make use
of the invertibility of all the shifted gramians $\Gr_{\bo, k, C, \bA}$ guaranteed to us by
Proposition \ref{P:shift-gram-inv}. Then the inequality \eqref{contr} can
equivalently be expressed as $\| \Xi \| \le 1$ where the operator 
$\Xi \colon \sbm{\cX \\ \cU_{\beta}} \to \sbm{\cX \\ \cY}$ is given by
\begin{equation}   \label{pr7}
 \Xi= \begin{bmatrix} \Gr_{\bo,|\beta|+1,C,\bA}^{\frac{1}{2}}\otimes I_d & 0 \\ 0
 & \omega_{|\beta|}^{-\frac{1}{2}}\cdot I_{\cY} \end{bmatrix}
 \begin{bmatrix} A & {\widehat B}_{\beta} \\ C & D_{\beta} \end{bmatrix}
     \begin{bmatrix} \Gr_{\bo,|\beta|,C,\bA}^{-\frac{1}{2}} & 0 \\ 0 &
         I_{\cU_{\beta}} \end{bmatrix}.
 \end{equation} 
Another equivalent condition is that $\| \Xi^{*} \| \le 1$ which in
turn can be expressed as   
$$
    \begin{bmatrix} A & {\widehat B}_{\beta} \\ C & D_{\beta} \end{bmatrix}
        \begin{bmatrix} \Gr_{\bo,|\beta|,C,\bA}^{-1} & 0 \\ 0 & 
	    I_{\cU_\beta} \end{bmatrix}
        \begin{bmatrix} A^{*} & C^{*} \\ {\widehat B}_{\beta}^{*} & 
	    D_{\beta}^{*}
        \end{bmatrix}\preceq  \begin{bmatrix} \Gr_{\bo,|\beta|+1,C,\bA}^{-1}\otimes I_d & 0 \\ 0 &
        \omega_{|\beta|} I_{\cY} \end{bmatrix}.
$$
 Note that equality \eqref{isom} means that the operator $\Xi$ is isometric.
Of  particular interest  is the case where $\Xi$ is coisometric, that is, where
the connection matrix ${\bf U}_\beta = \left[ \begin{smallmatrix} 
A & {\widehat B}_{\beta}  \\ C & D_\beta \end{smallmatrix} \right]$
is coisometric with respect to the weights indicated below:
\begin{equation}   \label{wghtcoisom}
    \begin{bmatrix} A & {\widehat B}_{\beta}  \\ C & D_{\beta} \end{bmatrix}   
        \begin{bmatrix} \Gr_{\bo,|\beta|,C,\bA}^{-1} & 0 \\ 0 & 
	    I_{\cU_\beta} \end{bmatrix}
        \begin{bmatrix} A^{*} & C^{*} \\ {\widehat B}_{\beta}^{*} & 
	    D_{\beta}^{*}
        \end{bmatrix} = \begin{bmatrix} \Gr_{\bo,|\beta|+1,C,\bA}^{-1}\otimes I_d & 0 \\ 0 &
        \omega_{|\beta|} I_{\cY} \end{bmatrix}.
\end{equation}

\begin{lemma}   \label{L:frakcoisom}
Let $(C,\bA)$ be an exactly $\bo$-observable $\bo$-output stable pair and 
let $\Theta_{\bo, \bU_\beta}$ be defined as in \eqref{jul15}  for some operators $B_{j, \beta} \in 
\cL(\cU_\beta, \cX)$ and $D_\beta \in \cL(\cU_\beta, \cY)$
subject to equality \eqref{wghtcoisom}. Then
\begin{align}
&\omega_{|\beta|}^{-1}I_{\cY} - \Theta_{\bo, \bU_\beta}(z)\Theta_{\bo, \bU_\beta}(\zeta)^{*}\notag\\
& = CR_{\bo,|\beta|}(Z(z)A)\Gr_{\bo,|\beta|,C,\bA}^{-1}R_{\bo,|\beta|}(Z(\zeta) A)^*C^*\label{frakkerid}\\
&\quad -CR_{\bo,|\beta|+1}(Z(z)A)Z(z) (\Gr_{\bo,|\beta|+1,C,\bA}^{-1}\otimes I_d)
Z(\zeta)^* R_{\bo,|\beta|+1}(Z(\zeta) A)^*C^*.
\notag
\end{align}
\end{lemma}

\begin{proof}The proof parallels the verification of the identity
    \eqref{jul18} done above.  The weighted-coisometry condition \eqref{wghtcoisom}
    gives us the set of equations
 \begin{align}  
A\Gr_{\bo,|\beta|,C,\bA}^{-1} A^{*} + {\widehat B}_{\beta} {\widehat B}_{\beta}^{*} &= 
\Gr_{\bo,|\beta|+1,C,\bA}^{-1}\otimes I_d, \notag\\ 
C\Gr_{\bo,|\beta|,C,\bA}^{-1} A^{*} + D_{\beta}{\widehat B}_{\beta}^{*}&= 0,  \label{relations1}\\
C \Gr_{\bo,|\beta|,C,\bA}^{-1} C^{*} + D_{\beta} D_{\beta}^{*}& = 
\omega_{|\beta|} I_{\cY}.\notag
 \end{align}
We then  compute (using relations \eqref{relations1} at the third step below)
 \begin{align*}
    & \omega_{|\beta|}^{-1}I_{\cY} - \Theta_{\bo, \bU_\beta}(z)\Theta_{\bo, \bU_\beta}(\zeta)^{*}\\
& =  \omega_{|\beta|}^{-1}I_{\cY}- \left[\omega_{|\beta|}^{-1}D_\beta +    
     CR_{\bo,|\beta|+1}(Z(z)A)Z(z){\widehat B}_{\beta} \right]\\
&\qquad\qquad\qquad\times  
 \left[ \omega_{|\beta|}^{-1}D_{\beta}^{*} +{\widehat B}_{\beta}^{*}Z(\zeta)^* R_{\bo,|\beta|+1}
(Z(\zeta) A)^*C^*\right] \\
& = \omega_{|\beta|}^{-1}I_{\cY}-\omega_{|\beta|}^{-2}D_\beta 
D_\beta^*- \omega_{|\beta|}^{-1} CR_{\bo,|\beta|+1}(Z(z)A)Z(z){\widehat B}_{\beta}D_{\beta}^{*}\\
&\quad-\omega_{|\beta|}^{-1}D_\beta{\widehat B}_{\beta}^{*}Z(\zeta)^* R_{\bo,|\beta|+1}(Z(\zeta) A)^*C^*\\
     & \quad - CR_{\bo,|\beta|+1}(Z(z)A)Z(z){\widehat 
     B}_{\beta}{\widehat B}_{\beta}^{*}Z(\zeta)^* 
R_{\bo,|\beta|+1}(Z(\zeta) A)^*C^*\\
     & = \omega_{|\beta|}^{-2}C\Gr_{\bo,|\beta|,C,\bA}^{-1}C^* + 
     \omega_{|\beta|}^{-1}CR_{\bo,|\beta|+1}(Z(z)A)Z(z)A\Gr_{\bo,|\beta|,C,\bA}^{-1}C^*\\
 & \quad +  \omega_{|\beta|}^{-1}C\Gr_{\bo,|\beta|,C,\bA}^{-1}A^{*}Z(\zeta)^* 
 R_{\bo,|\beta|+1}(Z(\zeta) A)^*C^* \\
     & \quad  - CR_{\bo,|\beta|+1}(Z(z)A)Z(z)\left[ (\Gr_{\bo,|\beta|+1,C,\bA}^{-1}\otimes I_d) - A
\Gr_{\bo,|\beta|,C,\bA}^{-1}  A^{*}\right]\\
&\qquad\qquad\qquad\qquad\times Z(\zeta)^* R_{\bo,|\beta|+1}(Z(\zeta) A)^*C^*\\
&=C\left(\omega_{|\beta|}^{-1}I_{\cX}+R_{\bo,|\beta|+1}(Z(z)A)Z(z)A\right)\Gr_{\bo,|\beta|,C,\bA}^{-1}\\
&\qquad\qquad\times
\left(\omega_{|\beta|}^{-1}I_{\cX}+A^*Z(\zeta)^*R_{\bo,|\beta|+1}(Z(\zeta) A)^*\right)C^*\\
&\quad-CR_{\bo,|\beta|+1}(Z(z)A)Z(z) (\Gr_{\bo,|\beta|+1,C,\bA}^{-1}\otimes I_d) 
Z(\zeta)^* R_{\bo,|\beta|+1}(Z(\zeta) A)^*C^*
     \end{align*}
which implies \eqref{frakkerid}, due to equality
$$
\omega_{|\beta|}^{-1}I_{\cX}+R_{\bo,|\beta|+1}(Z(z)A)Z(z)A=R_{\bo,|\beta|}(Z(z)A),
$$
which in turn, is a consequence of \eqref{1.6g'}.
\end{proof}
Since equality \eqref{wghtcoisom} implies inequality \eqref{contr}, it follows that under
assumption of Lemma \ref{L:frakcoisom}, all the conclusions of parts (1) and (2) in Lemma \ref{L:5.1} 
are true. To have all conclusions true, we need the operator \eqref{pr7} to be unitary.
The next result amounts to a more structured version of Proposition \ref{P:elementary'}.
 
 \begin{lemma}   \label{L:5.6}
Suppose that we are given an exactly $\bo$-observable
$\bo$-output-stable pair $(C,\bA)$ with $\bA=(A_1,\ldots,A_d)\in\cL(\cX)^d$ and $C\in\cL(\cX,\cY)$. Then 
for every $\beta \in\free$, there exist operators $B_{1,\beta},\ldots 
B_{d,\beta}\in\cL(\cU_\beta, \cX)$
and $D_\beta \in \cL(\cU_\beta, \cY)$ such that equalities \eqref{wghtcoisom} and \eqref{isom} hold
with $A$ and $\widehat{B}_\beta$ defined as in \eqref{coll}.
Explicitly, such  $\widehat{B}_{\beta}$ and $D_{\beta}$ are uniquely
 determined up to a common unitary right factor by solving the Cholesky factorization problem:
  \begin{equation}  \label{pr6}
\begin{bmatrix}\widehat{B}_{\beta} \\ 
    D_\beta \end{bmatrix}\begin{bmatrix}\widehat{B}_{\beta}^* & D_\beta^*\end{bmatrix}=
\begin{bmatrix}\Gr_{\bo,|\beta|+1,C,\bA}^{-1}\otimes I_d & 0 \\ 0 & 
    \omega_{|\beta|} I_{\cY}\end{bmatrix}-
\begin{bmatrix}A \\ C\end{bmatrix}\Gr_{\bo,|\beta|,C,\bA}^{-1}\begin{bmatrix}A^* & C^*\end{bmatrix}
\end{equation}
subject to the additional constraint that the coefficient space
$\cU_{\beta}$ be chosen so that $\left[ \begin{smallmatrix} 
\widehat{B}_{\beta} \\ D_{\beta} \end{smallmatrix} \right] \colon \cU_{\beta} \to \cX^d \oplus \cY$ is
injective. 
\end{lemma}

\begin{proof}
By Proposition \ref{P:wghtSteinid}, the weighted Stein
identity \eqref{4.34} holds for each $k \ge 0$.
Since $(C,\bA)$ is exactly $\bo$-observable, the gramian $\Gr_{\bo,|\beta|,C,\bA}$ is strictly
positive definite and then it follows from \eqref{4.34} that the operator
$$
U: = \begin{bmatrix}(\Gr_{\bo,|\beta|+1,C,\bA}^{\frac{1}{2}}\otimes I_d)A 
    \Gr_{\bo,|\beta|,C,\bA}^{-\frac{1}{2}} \\
\omega_{|\beta|}^{-\frac{1}{2}}C\Gr_{\bo,|\beta|,C,\bA}^{-\frac{1}{2}}\end{bmatrix} \colon  
\cX\to \cX^d\oplus\cY
$$
is an isometry.  We then find an injective $V \colon \cU_\beta \to \sbm{ \cX^d \\ \cY }$
such that $V V^* = I_{\cX^d \oplus \cY} - U U^*$ as in Proposition \ref{P:elementary'}.
This translates to $V$ having the form
$$
  V = \begin{bmatrix} \Gr^{\frac{1}{2}}_{\bo, |\beta| + 1, C, \bA} \otimes I_d & 0 \\ 0 & \mu^{-\frac{1}{2}}_{\bo, |\beta|} 
  I_\cY \end{bmatrix}  \begin{bmatrix} \widehat B_\beta \\ D_\beta \end{bmatrix}
$$
where $\sbm{ \widehat B_\beta \\ D_\beta }$ is a solution
of the factorization problem \eqref{pr6}.   Uniqueness of the solution 
$\sbm{ \widehat B_\beta \\ D_\beta }$ of the factorization problem \eqref{pr6}
corresponds exactly to the uniqueness of the partial isometry $V$ in the factorization problem
$V V^* = I_{\cX^d \oplus \cY} - U U^*$.  The fact that $\begin{bmatrix} U & V \end{bmatrix}$ is unitary means
that both equalities \eqref{wghtcoisom} and \eqref{isom} hold.
\end{proof}
\begin{remark}   \label{R:4.10}
 The operators $\widehat{B}_{\beta}$ and $D_\beta$ depend on 
$|\beta|$ rather than on $\beta$; 
thus, in the construction above we can always take 
$\cU_\beta=\cU_{\beta^\prime}$ and also 
$\widehat{B}_{\beta}=\widehat{B}_{\beta^\prime}$ and 
$D_\beta=D_{\beta^\prime}$ whenever $|\beta|=|\beta^\prime|$. 
\end{remark}

\section{Beurling-Lax representations based on Bergman-inner families}  \label{S:BL7}
In this section we present our most elaborate version of the Beurling-Lax 
theorem. 
Arguably this version, while more complicated than the previous 
 versions which we have discussed and the classical case, is the most compelling in that
the representer is determined up to a unitary change of basis on the 
input-space sequence, closer to the classical case where the representer 
is determined up to a unitary change of basis on the input space.
There results a unitary invariant for a shift invariant subspace.

\smallskip

In the classical case ($d=1$ and $n=1$), if $\cM$ is a
subspace of the Hardy space $H_{\cY}^{2}$  invariant for the shift 
operator $S \colon f(z) \mapsto z f(z)$, the Beurling-Lax representer 
$\Theta$ for $\cM$ can be constructed by choosing the coefficient 
space $\cU$ to have the same dimension as the {\em wandering 
subspace} $\cE: = \cM \ominus z \cM$ for $\cM$, and letting $\Theta$ 
be any unitary identification map of $\cU$ to $\cE$ (so $\cE = \Theta 
\cU$) and then extending $\Theta$ to $M_{\Theta} \colon H^{2}_{\cU} 
\to H^{2}_{\cY}$ by demanding shift-invariance:  $M_{\Theta} S = S 
M_{\Theta}$.  The subspace $\cM_{\beta}$ given below is the time-varying 
multidimensional analog of the wandering subspace, as explained in 
the following lemma.

\begin{lemma} \label{L:Mdecom}
    Suppose that $\cM \subset H^2_{\bo,\cY}(\free)$ is an ${\mathbf 
    S}_{\bo,R}$-invariant subspace.  For each word $\beta \in \free$, 
    define the subspace
\begin{equation}    \label{Mv}
    \cM_\beta=
\bS_{\bo, R}^{\beta^{\top}} \cM \ominus \bigg(\bigoplus_{j=1}^d 
\bS_{\bo,R}^{\beta^{\top}}\bS_{\bo,R,j} \cM  \bigg).
\end{equation}
Then $\cM$ has the orthogonal direct-sum decomposition
\begin{equation}  \label{Mdecom}
 \cM = \bigoplus_{\beta\in\free} \cM_\beta.
\end{equation}
More generally, for any $\alpha \in \free$, we have the orthogonal 
direct-sum decomposition
\begin{equation}  \label{Mdecom'}
    \bS_{\bo,R}^{\alpha^{\top}} \cM = \bigoplus_{\beta \in \free} 
    \cM_{\beta \alpha}.
\end{equation}
 \end{lemma}
\begin{proof}

    We first show that 
\begin{equation}
\cM_{\beta} \perp \cM_{\gamma} \quad\mbox{if}\quad \beta \ne \gamma,
\label{march1}
\end{equation}
i.e., that $f \perp g$ in $H^2_{\bo, \cY}(\free)$ for any $f \in \cM_{\beta}$ and $g \in \cM_{\gamma}$.
    Without loss of generality we take $|\gamma| \ge |\beta|$.  We then write
    $\gamma =\delta \beta' $ with $|\beta'| = |\beta|$ and $\delta \in \free$ 
    (possibly equal to $\emptyset$).  
    
\smallskip
\noindent
{\bf Case 1:} Suppose that $\beta' \ne \beta$.
    Then any $f \in \bS_{\bo,R}^{\beta^{\top}} \cM$, as an element of $\bS_{\bo,R}^{\beta^{\top}} 
    H^2_{\bo, \cY}(\free)$, has the form $f(z) = \sum_{\alpha \in 
    \free} f_{\alpha \beta } z^{ \alpha \beta}$.  Similarly, any 
     $g \in \bS_{\bo,R}^{\gamma^{\top}} \cM$ has the form 
$$
g(z) =    \sum_{\alpha' \in \free} 
    g_{\alpha' \gamma} z^{\alpha' \gamma} = \sum_{\alpha' \in \free} 
    g_{\alpha' \delta \beta^{\prime}} z^{\alpha' \delta \beta^{\prime}}.
$$  
    As $\beta$ and $\beta'$ are different and of the 
    same length, every word of the form $\alpha \beta$ is 
    distinct from every word of the form $\alpha' \delta \beta^{\prime}$
and since the monomial subspaces $z^{\alpha} \cY$ satisfy
 orthogonality relations
 $z^{\alpha} y \perp z^{\alpha'} y'$ for all $y,y' \in \cY$ and $\alpha \ne \alpha'$ in $\free$,
we have $f_{\alpha\beta} z^{\alpha \beta} \perp g_{\alpha' \delta \beta^{\prime}}z^{\alpha' \delta \beta^{\prime}}$ in 
$H^2_{\bo,\cY}(\free)$ and hence, $f \perp g$ in this case.
    
\smallskip
\noindent
{\bf Case 2:} Suppose that $\beta' = \beta$, so $\gamma  = \delta \beta $.  
    The assumption that $\beta \ne \gamma$ 
    implies that $\delta \ne \emptyset$.  Then
 $$
g \in \cM_{\gamma} \subset {\mathbf S}_{\bo,R}^{\gamma^{\top}} \cM = 
{\mathbf S}_{\bo,R}^{\beta^{\top}} {\mathbf S}_{\bo,R}^{\delta^{\top}} \cM \subset \bigoplus_{j=1}^{d} 
{\mathbf S}_{\bo,R}^{\beta^{\top}} S_{\bo,R,j} \cM
$$
(note that ${\mathbf S}_{\bo,R}^{\beta^{\top}} S_{\bo,R,j} \cM \perp {\mathbf 
S}_{\bo,R}^{\beta^{\top}} S_{\bo,R,i}\cM$ for $i\ne j$ by Case 1).  But $f \in \cM_{\beta}$, and
$\cM_{\beta}$ is orthogonal to  $\bigoplus_{j=1}^{d} {\mathbf S}_{\bo,R}^{\beta^{\top}} S_{\bo,R,j} \cM$ 
by definition \eqref{Mv}.  We conclude that $f \perp g$ in this case as well.  This completes the verification 
of \eqref{march1}.

\smallskip

As \eqref{Mdecom} is the special case of \eqref{Mdecom'} with $\alpha 
= \emptyset$, it remains only to prove \eqref{Mdecom'}.  First note 
that 
$$
  \cM_{\beta \alpha} = \bS_{\bo,R}^{\alpha^{\top}} \bS_{\bo,R}^{\beta^{\top}} \cM \ominus 
  \bigg( \bigoplus_{j=1}^{d} \bS_{\bo,R}^{\alpha^{\top}} \bS_{\bo,R}^{\beta^{\top}} 
  S_{\bo,R,j} \cM \bigg) \subset \bS_{\bo,R}^{\alpha^{\top}} \cM,
 $$
 and it follows that 
 $$
   \bS^{\alpha^{\top}}_{\bo,R}\cM \supset \bigoplus_{\beta \in \free} \cM_{\beta \alpha}.
 $$
 To show the reverse inclusion, it suffices to show that $h = 0$ 
 whenever $h \in \bS_{\bo,R}^{\alpha^{\top}} \cM$ is orthogonal to 
 $\cM_{\beta \alpha}$ for all $\beta \in \free$.
Suppose therefore that $h \in \bS_{\bo,R}^{\alpha^{\top}}\cM$ is orthogonal to 
$\cM_{\beta \alpha}$ for all $\beta \in \free$.  In particular, the 
special case $\beta = \emptyset$ gives us that $h \perp 
\cM_{\alpha} : = \bS_{\bo,R}^{\alpha^{\top}}\cM \ominus \bigoplus_{j=1}^{d} 
\bS_{\bo,R}^{\alpha^{\top}} S_{\bo,R,j} \cM$; thus
\begin{equation}
  h \in \operatorname{closure} \bigoplus_{j=1}^{d} \bS^{\alpha^{\top}}_{\bo,R} S_{\bo,R,j} \cM 
  + (\bS^{\alpha^{\top}}_{\bo,R}\cM)^{\perp}.
\label{again41}
\end{equation}
Since $h \in \bS_{\bo,R}^{\alpha^{\top}}\cM$ and $\bigoplus_{j=1}^{d} 
\bS_{\bo,R}^{\alpha^{\top}} S_{\bo,R,j} \cM \subset 
\bS_{\bo,R}^{\alpha^{\top}}\cM$, it follows from \eqref{again41} (by taking the orthogonal projection of
$H^2_{\bo, \cY}(\free)$ onto $\bS_{\bo,R}^{\alpha^{\top}} \cM$) that
\begin{equation}   \label{N=1}
  h \in \bigoplus_{j=1}^{d} \bS^{\alpha^{\top}}_{\bo,R} S_{\bo,R,j} \cM =
  \bigoplus_{\beta \in \free \colon |\beta| = 1} 
  \bS_{\bo,R}^{\alpha^{\top}} \bS_{\bo,R}^{\beta^{\top}} \cM.
\end{equation}

Inductively suppose that
\begin{equation}   \label{Ngen}
  h \in \bigoplus_{\beta \in \free \colon |\beta| = N} 
  \bS_{\bo,R}^{\alpha^{\top}} \bS_{\bo,R}^{\beta^{\top}} \cM.
\end{equation}
Choose any $\beta_{0} \in \free$ with $|\beta_{0}| = N$.  Then the 
condition $h \perp \cM_{\beta_{0} \alpha}$ leads to
\begin{equation}
 h \in  \operatorname{closure}\, \bigoplus_{j=1}^{d} \bS_{\bo,R}^{\alpha^{\top}} 
 \bS_{\bo,R}^{\beta_{0}^{\top}} S_{\bo,R,j} \cM + \big( 
 \bS_{\bo,R}^{\alpha^{\top}} \bS_{\bo,R}^{\beta_{0}^{\top}} \cM  \big)^{\perp}.
\label{again42}
\end{equation}
Combining the assumption \eqref{Ngen} with the inclusion
$$
\bigoplus_{j=1}^{d} \bS^{\alpha^{\top}}_{\bo,R} 
\bS^{\beta_{0}^{\top}}_{\bo,R} S_{\bo,R,j} \cM \subset
\bigoplus_{\beta \in \free \colon |\beta| = N} 
\bS^{\alpha^{\top}}_{\bo,R} \bS^{\beta^{\top}}_{\bo,R} \cM
$$
we conclude from \eqref{again42} (by taking the orthogonal projection of
$H^2_{\bo, \cY}(\free)$ onto $\bigoplus_{\beta \in \free \colon |\beta| = N}
  \bS_{\bo,R}^{\alpha^{\top}} \bS_{\bo,R}^{\beta^{\top}} \cM$) that 
$$
h \in \bigg( \bigoplus_{j=1}^{d} \bS_{\bo,R}^{\alpha^{\top}} 
\bS_{\bo,R}^{\beta_{0}^{\top}} S_{\bo,R,j} \cM \bigg) \oplus
\bigg( \bigoplus_{ \beta \in \free \colon |\beta| = N, \, \beta \ne 
\beta_{0}} \bS_{\bo,R}^{\alpha^{\top}} \bS_{\bo,R}^{\beta^{\top}} \cM 
\bigg).
$$
We now use that $\beta_{0}$ is an arbitrary element of $\free$ with 
$|\beta_0| = N$.  Intersecting the right-hand side of the latter formula 
over all such $\beta_{0}$ leads to
$$
 h \in \bigoplus_{\beta \in \free \colon |\beta| = N+1} 
 \bS_{\bo,R}^{\alpha^{\top}} \bS_{\bo,R}^{\beta^{\top}} \cM.
$$
Recalling that we established \eqref{N=1} and that the present 
argument started with the assumption \eqref{Ngen}, we conclude by the 
induction principle that \eqref{Ngen} holds for all 
$N=1,2,\dots$.  Thus
$$
 h \in \bigcap_{N=0}^{\infty}\bigoplus_{\beta \in \free \colon 
 |\beta| = N} \bS_{\bo,R}^{\alpha^{\top}} \bS_{\bo,R}^{\beta^{\top}} \cM
 \subset  \bS_{\bo,R}^{\alpha^{\top}} \bigg(
 \bigcap_{N=0}^{\infty} \bigoplus_{\beta \in \free \colon 
 |\beta| = N}  \bS_{\bo,R}^{\beta^{\top}} 
 H^2_{\bo,\cY}(\free) \bigg) = \{0\}
$$
and the orthogonal decomposition \eqref{Mdecom'} is verified.
\end{proof}

In the discussion to follow, given our $\bS_{\bo, R}$-invariant subspace $\cM$ (isometrically included
in $H^2_{\bo, \cY}(\free)$, we fix a choice of $\bo$-isometric output pair $(C, \bA)$ with $\bA$
strongly stable so that $\cM^\perp = \operatorname{Ran} \cO_{\bo,C, \bA}$ and so that $\cM$ has
 reproducing kernel $k_\cM$ defined as in \eqref{kMa}:
\begin{equation}
\label{kM}
k_{\cM}(z, \zeta)=k_{\rm nc, \bo}(z,\zeta)I_{\cY}-CR_\bo(Z(z)A)R_\bo(Z(\zeta)A)^*C^*.
    \end{equation}
We remind the reader that a canonical choice of  $(C, \bA)$ is the model output pair
$$
C = E|_{\cM^\perp}, \quad \bA = \bS_{\bo, R}^*|_{\cM^\perp}
$$
With any such choice, the observability operator $\cO_{C, \bA} \colon \cM \to H^2_{\bo, \cY}(\free)$
is an isometric inclusion and the gramian operator $\cG_{\bo, C, \bA}$ is the identity operator on $\cM$.
Furthermore, as a consequence of Proposition \ref{P:shift-gram-inv}, we are assured that all the 
shifted gramians $\Gr_{\bo, k, C, \bA}$ are bounded and boundedly invertible.
To get the reproducing kernel for ${\bf S}_{\bo,R}^{\beta^{\top}} \cM$ consistent with the
metric of $H^2_{\bo,\cY}(\free)$,  we start with the following characterization of 
the space  $({\bf S}_{\bo,R}^{\beta^{\top}}\cM)^{\perp}$ in terms of the shifted 
observability operator ${\Ob}_{\bo,|\beta|,C,\bA}$ introduced in \eqref{4.31}.

\begin{proposition}   \label{P:SMperp}
    The space $({\bf S}_{\bo,R}^{\beta^{\top}} \cM)^{\perp}$ is characterized as
 \begin{equation}   \label{SMperp}
     \big({\bf S}_{\bo,R}^{\beta^{\top}}\cM\big)^{\perp} = 
\big({\bf S}_{\bo,R}^{\beta^{\top}} H^2_{\bo,\cY}(\free)\big)^\perp\bigoplus
    {\bf S}_{\bo,R}^{\beta^{\top}} \operatorname{Ran} {\Ob}_{\bo,|\beta|,C,\bA}.
 \end{equation}
In particular,
\begin{equation}  \label{SMperp1}
    \big(S_{\bo,R,k} \cM\big)^{\perp} = \big(S_{\bo,R,k} H^2_{\bo,\cY}(\free)\big)^{\perp} 
    \bigoplus S_{\bo,R,k} \operatorname{Ran} \Ob_{\bo,1,C, \bA}
\end{equation}
for $k=1,\ldots,d$. Furthermore,
\begin{equation}   \label{SMperp2}
    \bigg( \bigoplus_{j=1}^{d} S_{\bo,R,j} \cM \bigg)^{\perp} = \cY 
    \bigoplus\bigg( \bigoplus_{k=1}^{d} S_{\bo,R,k} {\rm Ran}\, 
    \Ob_{\bo,1,C,\bA} \bigg).
 \end{equation}
 \end{proposition}

\begin{proof} We wish to characterize all functions $f(z)={\displaystyle
\sum_{\alpha\in\free}f_{\alpha} z^{\alpha}}$ which are orthogonal to 
${\bf S}_{\bo,R}^{\beta^{\top}} \cM$ in $H^2_{\bo,\cY}(\free)$. We may write 
$$
f(z)  = p(z) + \widetilde f(z)z^{\beta}\quad\text{where}\quad 
p\in \left({\bf S}_{\bo,R}^{\beta^{\top}} H^2_{\bo,\cY}(\free)\right)^\perp, \quad
\widetilde f \in H^2_{\bo,\cY}(\free).
$$
Since ${\bf S}_{\bo,R}^{\beta^{\top}} \cM \subset {\bf 
S}_{\bo,R}^{\beta^{\top}}H^2_{\bo,\cY}(\free)$, it follows
that $\big({\bf S}_{\bo,R}^{\beta^{\top}}H^2_{\bo,\cY}(\free)\big)^\perp\subset\big({\bf 
S}_{\bo,R}^{\beta^{\top}} \cM\big)^\perp$, so it suffices to characterize which functions of the
form $\widetilde f(z)z^{\beta}$ are orthogonal to ${\bf 
S}_{\bo,R}^{\beta^{\top}} \cM$. To this end, observe that
${\bf S}_{\bo,R}^{\beta^{\top}}\widetilde{f}$ is orthogonal to ${\bf 
S}_{\bo,R}^{\beta^{\top}} \cM$ if and only if the power series
${\bf S}^{*\beta}_{\bo,R}{\bf S}_{\bo,R}^{\beta^{\top}}\widetilde{f}$ belongs to $\cM^\perp= 
\operatorname{Ran}\cO_{\bo,C,\bA}$. Since by \eqref{again43},
$$
{\bf S}^{*\beta}_{\bo,R}{\bf S}_{\bo,R}^{\beta^{\top}}: \; \sum_{\alpha\in\free} 
\widetilde{f}_\alpha z^\alpha\to \sum_{\alpha\in\free}
\frac{\omega_{|\alpha|+|\beta|}}{\omega_{|\alpha|}}\, \widetilde{f}_\alpha z^\alpha,
$$
we thus conclude that ${\bf S}_{\bo,R}^{\beta^{\top}}\widetilde{f}$ is orthogonal to 
${\bf S}_{\bo,R}^{\beta^{\top}} \cM$ if and only if there exists a vector $x\in\cX$ such that
$$
\sum_{\alpha\in\free}
\frac{\omega_{|\alpha|+|\beta|}}{\omega_{|\alpha|}}
\, \widetilde{f}_\alpha z^\alpha=(\cO_{\bo,C,\bA}x)(z)=
\sum_{\alpha\in\free}\big(\omega_{|\alpha|}^{-1}\cdot C\bA^\alpha x\big)z^\alpha.
$$
Equating the corresponding Taylor coefficients in the latter equality gives
$$
\widetilde{f}_\alpha=\omega_{|\alpha|+|\beta|}^{-1}\cdot C\bA^\alpha x
$$
for all $\alpha\in\free$, and therefore,
$$
\widetilde{f}(z)=\sum_{\alpha\in\free}\widetilde{f}_\alpha z^\alpha
=\sum_{\alpha\in\free}\big(\omega_{|\alpha|+|\beta|}^{-1}\cdot
C\bA^\alpha x\big)z^\alpha={\Ob}_{\bo,|\beta|,C,\bA}x,
$$
by \eqref{4.31}. Thus, $\widetilde{f}\in\operatorname{Ran} {\Ob}_{\bo,|\beta|,C,\bA}$.
As the analysis is necessary and sufficient, the formula \eqref{SMperp} follows.
Note that formula \eqref{SMperp1} is just the special case of 
\eqref{SMperp} where the word $\beta$ is taken to consist of the 
single letter $k$.

\smallskip

To verify \eqref{SMperp2}, we note first that 
\begin{equation}
 \bigg( \bigoplus_{j=1}^{d} S_{\bo,R,j} \cM \bigg)^{\perp} =
 \bigcap_{j=1}^{d} \big(S_{\bo,R,j} \cM\big)^{\perp}.
\label{march2a}
\end{equation}
We next introduce the subspaces
$$
\cN_j:=S_{\bo,R,j} H^2_{\bo,\cY}(\free)\ominus
S_{\bo,R,j} {\rm Ran}\, \Ob_{\bo,1,C, \bA}
$$
and observe the equalities
\begin{align*}
\big(S_{\bo,R,k} H^2_{\bo,\cY}(\free)\big)^{\perp}&=
\cY \bigoplus \bigg(\bigoplus_{j \colon
  j \ne k} S_{\bo,R,j} H^2_{\bo,\cY}(\free) \bigg)\\
&=\cY\bigoplus \bigg(\bigoplus_{j \colon
  j \ne k} S_{\bo,R,j}{\rm Ran}\, \Ob_{\bo,1,C, \bA} \bigg)
\bigoplus\bigg(\bigoplus_{j \colon j \ne k} \cN_j\bigg)
\end{align*}
which, when substituted into the formula \eqref{SMperp1} for $\big(S_{\bo,R,k} 
\cM\big)^{\perp}$, give
\begin{equation}
\big(S_{\bo,R,k} \cM\big)^{\perp}=\cY\bigoplus \bigg(\bigoplus_{j=1}^d S_{\bo,R,j}{\rm 
Ran}\, \Ob_{\bo,1,C, \bA} \bigg)
\bigoplus\bigg(\bigoplus_{j \colon j \ne k} \cN_j\bigg).
\label{march2}
\end{equation}
Since $\cN_j\perp\cN_i$ for $i\neq j$, it follows that 
\begin{equation}
\bigcap_{k=1}^d \bigg(\bigoplus_{j \colon j \ne k} \cN_j\bigg)=\{0\}.
\label{march2b}
\end{equation}
Making use of representation \eqref{march2} and taking into account \eqref{march2b}
we get
$$
\bigcap_{k=1}^{d} (S_{\bo,R,k}\cM)^{\perp} = \cY \bigoplus \bigg(
\bigoplus_{k=1}^{d} S_{\bo,R,k} {\rm Ran}\, \Ob_{\bo,1,C,\bA}\bigg)
$$
and now the formula \eqref{SMperp2} follows from \eqref{march2a}.
    \end{proof}
It is convenient to introduce, in analogy with shifted observability operators and gramians, the shifted 
positive kernel
\begin{equation}
k_{{\rm nc},\bo,\beta}(z,\zeta):=\sum_{\alpha\in\free}\omega_{|\alpha|+|\beta|}^{-1}
 z^{\alpha \beta} \bzeta^{ (\alpha \beta)^{\top}}=k_{{\bf S}_{\bo,R}^{\beta^{\top}}H^2_{\bo}(\free)}
\quad\mbox{for}\quad \beta\in\free,
\label{SMa}
\end{equation}
which is, as the second equality in \eqref{SMa} indicates, the formal noncommutative reproducing 
kernel for the space ${\bf S}_{\bo,R}^{\beta^\top}H^2_{\bo}(\free)$ with metric inherited from 
$H^2_{\bo,\cY}(\free)$. It is clear that $k_{{\rm nc},\bo,\emptyset}$ is just $k_{{\rm nc},\bo}$.

\smallskip
    
With characterization \eqref{SMperp} and positive kernels \eqref{SMa} in hand, it is straightforward to derive the
kernel function for the space ${\bf S}_{\bo,R}^{\beta^{\top}} \cM$ with respect to the
metric inherited from $H^2_{\bo,\cY}(\free)$.

   \begin{proposition}  \label{P:6.2}  Let $\cM$ be a closed ${\bf S}_{\bo,R}$-invariant 
subspace of $H^2_{\bo,\cY}(\free)$ with reproducing kernel
$k_{\cM}$ given by \eqref{kM}. Then for every $\beta \in\free$, the formal  reproducing kernel functions 
for ${\bf S}_{\bo,R}^{\beta^{\top}} \cM$ 
and for $\cM_\beta$ (defined in \eqref{Mdecom}) in the metric of $H^2_{\bo,\cY}(\free)$ are given by
\begin{align}
k_{{\bf S}_{\bo,R}^{\beta^{\top}} \cM}(z,\zeta) = & k_{{\rm nc},\bo,\beta}(z,\zeta)I_{\cY}
-\boldsymbol{\mathfrak K}_{\beta}(z, \zeta),\label{kscm} \\
 k_{\cM_\beta}(z,\zeta)= & z^{\beta} \bzeta^{\beta^{\top}}
\omega_{|\beta|}^{-1}I_{\cY}-\boldsymbol{\mathfrak K}_{\beta}(z, \zeta)+
\sum_{j=1}^d \boldsymbol{\mathfrak K}_{j\beta}(z, \zeta),\label{kdif}
\end{align}
where $\boldsymbol{\mathfrak{K}}_{\beta}$ is the positive kernel defined as in \eqref{deffrakk}:
$$
\boldsymbol{\mathfrak{K}}_{\beta}(z,\zeta)=
CR_{\bo,|\beta|}(Z(z)A) \big(z^{\beta} \bzeta^{\beta^{\top}}
\Gr_{\bo,|\beta|,C,\bA}^{-1}\big)  R_{\bo,|\beta|}(Z(\zeta)A)^*C^* .
$$
\end{proposition}

\begin{proof} 
From the formula \eqref{SMperp} for $({\bf S}_{\bo,R}^{\beta^{\top}}\cM)^{\perp}$ we deduce that
in metric of $H^2_{\bo,\cY}(\free)$,
\begin{align*}
{\bf S}_{\bo,R}^{\beta^{\top}}\cM & =\big({\bf S}_{\bo,R}^{\beta^{\top}} H^2_{\bo,\cY}(\free)\big)\bigcap
   \big( {\bf S}_{\bo,R}^{\beta^{\top}} \operatorname{Ran} {\Ob}_{\bo,|\beta|,C,\bA}\big)^\perp
\notag  \\
    & = \big({\bf S}_{\bo,R}^{\beta^{\top}}H^2_{\bo,\cY}(\free)\big)\ominus 
    \big({\bf S}_{\bo,R}^{\beta^{\top}} \operatorname{Ran} {\Ob}_{\bo,|\beta|,C,\bA}\big).
\end{align*}
Therefore, the reproducing kernel for the subspace ${\bf S}_{\bo,R}^{\beta^{\top}}\cM$ is equal to the difference 
of reproducing kernels for the subspaces ${\bf S}_{\bo,R}^{\beta^{\top}}H^2_{\bo,\cY}(\free)$ and 
${\bf S}_{\bo,R}^{\beta^{\top}} \operatorname{Ran} {\Ob}_{\bo,|\beta|,C,\bA}$. But these kernels are equal
to $ k_{{\rm nc},\bo,\beta}(z,\zeta)I_{\cY}$ and $\boldsymbol{\mathfrak{K}}_{\beta}(z,\zeta)$ respectively, by \eqref{SMa} and 
Theorem \ref{T:RanOb=NFRKHS}. Hence \eqref{kscm} follows.
 
\smallskip

It remains to verify the formula \eqref{kdif}.  Toward this end, we first observe that
replacing $\beta$ by $j\beta$ in \eqref{SMa} gives
$$
k_{{\rm nc},\bo,(j\beta)}(z,\zeta)=\sum_{\alpha \in\free}\omega_{|\alpha|+|\beta|+1}^{-1}
z^{ \alpha j \beta} \bzeta^{(\alpha j \beta)^{\top}}\quad\mbox{for}\quad j=1,\ldots,d,
$$
and consequently,
\begin{align}
&k_{{\rm nc},\bo,\beta}(z,\zeta)-\sum_{j=1}^d k_{{\rm nc},\bo,(j\beta)}(z,\zeta)\label{jul1a}\\
&\quad =\sum_{\alpha\in\free}  \omega_{|\alpha|+|\beta|}^{-1} z^{ \alpha \beta }
\bzeta^{ (\alpha \beta)^{\top}}-\sum_{j=1}^d\sum_{\alpha\in\free} \omega_{|\alpha|+|\beta|+1}^{-1}
z^{ \alpha j \beta} \bzeta^{(\alpha j \beta)^{\top} }   \notag \\
&\quad =\sum_{\alpha\in\free}  \omega_{|\alpha|+|\beta|}^{-1} z^{ \alpha \beta }
\bzeta^{ (\alpha \beta)^{\top}}-\sum_{\alpha\in\free \colon \alpha \ne \emptyset}
\omega_{|\alpha|+|\beta|}^{-1}  z^{ \alpha \beta }
\bzeta^{ (\alpha \beta)^{\top}}= \omega_{|\beta|}^{-1}z^\beta \bzeta^{\beta^\top}.\notag
\end{align}
For the next computation we use \eqref{kscm}, \eqref{jul1a} and the orthogonal 
representation for $\cM_\beta$ in \eqref{Mv}, according to which
\begin{align*}
k_{\cM_\beta}(z,\zeta)   &  =  k_{{\bf S}_{\bo,R}^{\beta^{\top}} \cM}(z, \zeta)-
\sum_{j=1}^d k_{{\bf S}_{\bo,R}^{\beta^{\top} j}\cM} (z,  \zeta)\\
&=\big(k_{{\rm nc},\bo,\beta}(z,\zeta)-\sum_{j=1}^d k_{{\rm nc},\bo,(j\beta)}(z,\zeta)\big)I_{\cY}
-\boldsymbol{\mathfrak K}_{\beta}(z, \zeta)+\sum_{j=1}^d \boldsymbol{\mathfrak
K}_{ j \beta }(z, \zeta)  \\
& = \omega_{|\beta|}^{-1}z^\beta 
\bzeta^{\beta^\top}I_{\cY}-\boldsymbol{\mathfrak K}_{\beta}(z,\zeta)
+  \sum_{j=1}^d \boldsymbol{\mathfrak K}_{ j \beta }(z, \zeta).   
\end{align*}
This completes the proof of \eqref{kdif}.
\end{proof}

\begin{lemma}  \label{L:6.8}
Given $\beta \in\free$ and an exactly $\bo$-observable $\bo$-output stable pair
$(C,\bA)$, construct operators $\widehat{B}_\beta=\sbm{B_{1,\beta}\\ 
\vdots\\ B_{d,\beta}}\in \cL(\cU_\beta, \cX^d)$ and $D_\beta \in 
\cL(\cU_\beta, \cY)$  as in Lemma \ref{L:5.6}.  For each $\beta \in \free$ form the connection
$\bU_\beta = \sbm{ A & \widehat B_\beta \\ C & D_\beta }$ and let $\Theta_{\bo, \bU_\beta}$
be the associated transfer-function formal power series as in \eqref{jul15}.
Then the kernel \eqref{kdif} can be factored as
\begin{equation}
k_{\cM_\beta}(z,\zeta)=\Theta_{\bo, \bU_\beta}(z) \big(z^{\beta} 
\bzeta^{\beta^{\top}} I_{\cU_{\beta}}\big)\Theta_{\bo,\bU_\beta}(\zeta)^{*}.
\label{id6}
\end{equation}
Moreover, if we consider ${\mathbf S}_{1,R}^{\beta^{\top}} \cU_{\beta} =  
z^{\beta} \cU_{\beta}$ 
as a Hilbert space with norm 
$$
\|   z^{\beta} u \|_{ z^{\beta} \cU_{\beta}} = \| u \|_{\cU_{\beta}} 
\quad\text{for all}\quad u \in \cU_{\beta},
$$
lifted from $\cU_{\beta}$,
then the operator $M_{\Theta_{\bo, \bU_\beta}} \colon  z^{\beta} u \mapsto 
\Theta_{\bo, \bU_\beta}(z) z^{\beta} u $ is unitary from $z^{\beta} \cU_{\beta}$ 
onto $\cM_{\beta}$ for each $\beta \in  \free$.
\end{lemma}

\begin{proof}  By Lemma \ref{L:frakcoisom}, identity \eqref{frakkerid} holds. Multiplying
both parts of \eqref{frakkerid} by $z^{\beta}$ on the right and by 
$\bzeta^{\beta^{\top}}$ on the left, normalizing to a hereditary kernel, and then
combining the obtained equality with \eqref{kdif} easily leads to \eqref{id6}.

We note next that the factorization \eqref{id6} rewritten as
$$
  k_{\cM_{\beta}}(z, \zeta) = \Theta_{\bo, \bU_\beta}(z) z^{\beta} \cdot 
  \bzeta^{\beta^{\top}} \Theta_{\bo,\bU_\beta}(\zeta)^{*}
$$
amounts to a Kolmogorov decomposition for $k_{\cM_{\beta}}(z, \zeta)$, while 
the kernel $k_{z^{\beta} \cU_{\beta}}(z, \zeta)$ has the Kolmogorov 
decomposition
$$
  k_{z^{\beta} \cU_{\beta}}(z, \zeta) = (z^{\beta} 
  I_{\cU_{\beta}}) \cdot 
  (\bzeta^{\beta^{\top}} I_{\cU_{\beta}}).
$$
From the last statement in Theorem \ref{T:NFRKHS}, we can infer that 
$$
  M_{\Theta_{\bo, \bU_\beta}} \colon z^{\beta} u \mapsto \Theta_{\bo,\bU_ \beta}(z) \cdot  z^{\beta} u
 $$
 is a coisometry from $z^{\beta} \cU_{\beta}$ onto $\cM_{\beta}$.  Moreover, from 
 the identity \eqref{jul17} (specialized to the case where $u_{\beta'} = 
 0$ for $\beta' \ne \beta$ and $x= 0$), we see that $\widehat y(z): = 
 \Theta_{\bo, \bU_\beta}(z) u_{\beta} = 0$ forces $u_{\beta} = 0$, and hence also 
 $\Theta_{\bo,\bU_\beta}(z) \cdot z^{\beta } u_{\beta} = 0$ forces 
 $u_{\beta} = 0$, 
  so  $M_{\Theta_{\bo,\bU_\beta}}$ as an operator on $z^{\beta}\cU_{\beta}$ has no kernel.  
 Putting the pieces together, we see that $M_{\Theta_{\bo, \bU_\beta}} \colon 
 z^{\beta} \cU_{\beta} \to \cM_{\beta}$ is unitary as asserted.
\end{proof}
    
\begin{definition}  \label{D:innerfuncfam}
  Let us say that a family $\Theta_{\bo} = \{\Theta_{\bo,\beta}\}_{\beta \in \free}$ of formal 
power series (with $\Theta_{\bo, \beta} \in \cL(\cU_{\beta},\cY)\langle \langle z \rangle \rangle$ 
having operator-valued coefficients in $\cL(\cU_{\beta}, \cY)$  
  for auxiliary Hilbert spaces $\cU_{\beta}$ and $\cY$) is a 
  {\em $H^2_{\bo, \cY}(\free)$-Bergman-inner family} if, for each $\beta \in \free$, we have:
\begin{enumerate}
    \item The operator $M_{\Theta_{\bo,\beta}} \colon z^{\beta}
    \cU_{\beta} \to H^2_{\bo,\cY}(\free)$
    is  isometric,
    
    \item $M_{\Theta_{\bo,\beta}} \left( z^{\beta} \cU_{\beta} \right)$ is
    orthogonal to $M_{\Theta_{\bo, \gamma}} (z^{\gamma} \cU_{\gamma})$ for all 
    $\beta$ and $\gamma$ in $\free$ with $\beta \ne \gamma$,  and
    
 \item For each $\alpha \in \free$, 
 $$
 \bS_{\bo,R}^{\alpha^{\top}} \bigg( \bigoplus_{\beta \in \free} 
 \Theta_{\bo,\beta} z^{\beta} \cU_{\beta} \bigg) = \bigoplus_{\beta \in 
 \free} \Theta_{\bo,\beta \alpha} z^{\beta \alpha} \cU_{\beta \alpha}.
 $$
 \end{enumerate}
  \end{definition}
  We note that, for each $u_{\beta} \in \cU_{\beta}$,
  $M_{\Theta_{\bo, \beta}}(z^{\beta} u_{\beta}) = 
  \bS_{\bo,R}^{\beta^{\top}} M_{\Theta_{\bo, \beta}} u_{\beta} \in 
  \operatorname{Ran} \bS_{\bo,R}^{\beta^{\top}}$ and in general 
  $\operatorname{Ran} \bS_{\bo,R}^{\beta^{\top}} \perp 
  \operatorname{Ran} \bS_{\bo,R}^{\gamma^{\top}}$ if $|\beta| = 
  |\gamma|$ and $\beta \ne \gamma$ or if $\gamma = \delta \beta'$ with 
  $|\beta'| = |\beta|$ and $\beta' \ne \beta$, so the content of 
  condition (2) in Definition \ref{D:innerfuncfam} is that the 
  orthogonality condition holds for the special case where $\gamma = 
  \delta \beta$ and $\delta \ne \emptyset$.
  
 \smallskip 
 \noindent
 Given any collection $\{ \cU_{\beta}\}_{\beta \in \free}$ of coefficient 
  Hilbert spaces indexed by $\free$, we set
\begin{equation}
H^{2}_{\{\cU_{\beta}\}}(\free):= 
  {\displaystyle\bigoplus_{\beta \in\free}{\bf S}_{1,R}^{\beta^{\top}}\cU_{\beta}}
  = \bigoplus_{\beta \in \free} z^{\beta} \cU_{\beta}
\label{hbeta}
\end{equation}
  equal to the {\em time-varying Fock space} (compare with the standard 
  Fock space \eqref{Fock}) where every vector 
$$
{\bf u}=\bigoplus_{\beta \in\free} z^{\beta} u_{\beta} \in 
H^{2}_{\{\cU_{\beta}\}_{\beta\in\free}}(\free)
$$
is assigned the Fock-space norm $\| {\bf u}\|^{2} =
\sum_{\beta \in\free}\| u_{\beta}\|_{\cU_\beta}^{2}$ as in \eqref{Fock}, the difference here 
being that the coefficient space 
$\cU_{\beta}$ is allowed to depend on the index $\beta \in \free$. 

\smallskip

Given a $H^2_{\bo, \cY}(\free)$-Bergman-inner family $\{\Theta_{\bo,\beta}\}_{\beta\in\free}$, let us set
\begin{equation} \label{tvmultop}
\cM =  {\rm row}_{\beta \in \free} [\Theta_{\bo, \beta}] \cdot 
H^{2}_{\{ \cU_{\beta}\}}(\free) 
= \bigoplus_{\beta\in\free}\Theta_{\bo,\beta} {\bf S}_{1,R}^{\beta^{\top}} 
\cU_{\beta} \subset H^2_{\bo,\cY}(\free).
\end{equation}
Then conditions (1) and (2) in Definition \ref{D:innerfuncfam} imply that
the ``multiplication" operator
$M_{\Theta_\bo}\colon H^{2}_{\{\cU_{\beta}\}}(\free)\to H^2_{\bo,\cY}(\free)$ given by  
\begin{equation}  \label{MTheta}
M_{\Theta_{\bo}} = {\rm row}_{\beta \in \free}[\Theta_{\bo, \beta}]: \, 
{\bf u}=\bigoplus_{\beta\in\free} z^{\beta} u_{\beta} \mapsto \sum_{\beta\in\free} 
\Theta_{\bo,\beta}(z) z^{\beta}u_\beta
\end{equation}
maps $H^{2}_{\{\cU_{\beta}\}}(\free)$ isometrically onto the subspace $\cM$
given by \eqref{tvmultop}.
Furthermore, condition (3) in Definition \ref{D:innerfuncfam} implies 
that $\cM$ so defined is $\bS_{\bo,R}$-invariant. The next result is 
the converse: given any $\bS_{\bo,R}$-invariant closed subspace of $H^2_{\bo,\cY}(\free)$, 
there is a $H^2_{\bo, \cY}(\free)$-Bergman-inner family $\{ \Theta_{\bo, \beta}\}_{\beta \in \free}$ so that 
$\cM = M_{\Theta_{\bo}} H^{2}_{\{\cU_{\beta}\}}$, our next analog of the Beurling-Lax theorem for 
the freely noncommutative multivariable Hardy-Fock-space setting.

\begin{theorem}  \label{T:BL3}
    Let $\cM$ be a closed ${\bf S}_{\bo,R}$-invariant subspace of
    $H^2_{\bo,\cY}(\free)$. Define formal power series $\Theta_{\bo, 
    \beta} \in \cL(\cU_{\beta}, \cY)$ so that $M_{\Theta_{\bo, \beta}}$ 
    maps $ z^{\beta} \cU_{\beta}$ isometrically onto
    the subspace $\cM_{\beta}$ given by \eqref{Mv}.  Then $\Theta_{\bo} = 
    \{\Theta_{\bo,\beta}\}_{\beta\in\free}$ 
is a $H^2_{\bo, \cY}(\free)$-Bergman-inner family giving rise to a Beurling-Lax 
representation for the $\bS_{\bo,R}$-invariant subspace $\cM$ (using the notations \eqref{hbeta} and \eqref{MTheta}):
\begin{equation}   \label{BLrep2}
\cM=M_{\Theta_{\bo}}  H^{2}_{\{\cU_{\beta}\}}(\free).
\end{equation}
If $\Theta'_{\bo} = 
    \{\Theta'_{\bo,\beta}\}_{\beta \in \free}$ is another such $H^2_{\bo, \cY}(\free)$-Bergman-inner 
    family, then for each $\beta \in \free$ there is a unitary operator 
    $U_{\beta} \colon  \cU_{\beta}\to \cU'_{\beta}$ so that $\Theta'_{\bo,\beta}(z)U_{\beta} = 
    \Theta_{\bo,\beta}(z)$.

\smallskip

    Furthermore, one can construct such a $H^2_{\bo, \cY}(\free)$-Bergman-inner family 
$\{\Theta_{\bo,\beta}\}_{\beta \in \free}$ representing $\cM$ via  transfer-function 
realizations as follows:  
\begin{enumerate}
        \item Set $\cX =\cM^{\perp}$ and define $A \in \cL(\cX,\cX^d)$ and
        $C \in \cL(\cX, \cY)$ by
$$
  A = S_{\bo,R}^{*}|_{\cM^{\perp}},  \quad Cf = f_{\emptyset} \quad\text{for}\quad f \in
  \cM^{\perp}.
$$
\item Construct injective $\left[ \begin{smallmatrix} \widehat{B}_{\beta} \\ D_{\beta}
\end{smallmatrix} \right]$ by solving the Cholesky factorization
problem \eqref{pr6} in Lemma \ref{L:5.6}.

\item Define $\Theta_{\bo,\beta}(z)$ by
\begin{equation}   \label{innerfam-real}
\Theta_{\bo,\beta}(z) = \omega_{|\beta|}^{-1}D_{\beta}+ CR_{\bo,|\beta|+1}(Z(z)A) Z(z)\widehat{B}_{\beta}.
\end{equation}
\end{enumerate}
\end{theorem}

\begin{proof} By Lemma \ref{L:Mdecom}, we know that $\cM$ has the 
    orthogonal decomposition \eqref{Mdecom} with $\cM_{\beta}$ given 
    by \eqref{Mv}.  We define $\Theta_{\bo, \beta}$ so that 
    $M_{\Theta_{\bo, \beta}}$ maps $z^\beta\cU_{\beta}$ isometrically onto 
    $\cM_{\beta}$.  We thus have the orthogonal decomposition
    $$
     \cM = \bigoplus_{\beta \in \free} M_{\Theta_{\bo, \beta}} 
     z^{\beta} \cdot \cU_{\beta}.
 $$
 This leads to the the operator $M_{\Theta_{\bo}} \colon H^{2}_{\{ 
 \cU_{\beta}\}}(\free) \to H^2_{\bo, \cY}(\free)$ being a unitary map from 
 $H^{2}_{\{ \cU_{\beta}\}}$ onto $\cM$.  From the orthogonal 
 decomposition \eqref{Mdecom} for $\cM$, we see that any $\{\Theta_{\bo, \beta}\}$ 
constructed from $\cM$ in this way necessarily is a $H^2_{\bo, \cY}(\free)$-Bergman-inner family.

\smallskip
 
 The only constraint on the choice of $\Theta_{\bo, \beta}$ is that 
 $M_{\Theta_{\bo, \beta}}$ maps $z^{\beta} \cU_{\beta}$ isometrically onto the 
 subspace $\cM_{\beta}$.  Hence, any other choice $\{ \Theta'_{\bo, \beta}\}$ 
with respective coefficient spaces $\cU'_{\beta}$ 
 necessarily has the form $\Theta'_{\bo, \beta} z^{\beta} U =\Theta_{\bo, \beta} z^{\beta} $
 for a unitary operator $U \colon \cU_{\beta} \to \cU'_{\beta}$.

\smallskip
 
Suppose now that we are given the closed $\bS_{\bo,R}$-invariant subspace 
$\cM$ of $H^2_{\bo,\cY}(\free)$ and we construct  $\Theta_{\bo, \beta}$ according to the 
recipe (1), (2), (3) in the statement of the theorem.  By Theorem \ref{T:2-1.2} part (4), the 
pair $(C,\bA)$ is the canonical model $\bo$-isometric output pair for which the range 
${\rm Ran}\, \cO_{\bo,C,\bA}$ of the observability operator $\cO_{\bo,C,\bA}$ is exactly $\cM^{\perp}$.
As a result of Lemma \ref{L:5.6}, we see that the metric constraints \eqref{wghtcoisom} and 
\eqref{isom} hold.  By Lemma \ref{L:Mdecom}, the $\bS_{\bo,R}$-invariant subspace $\cM$ has the orthogonal 
decomposition \eqref{Mdecom} with $\cM_{\beta}$ as in \eqref{Mv}.
The kernel function $k_{\cM}(z, \zeta)$ for the subspace $\cM$ is given by \eqref{kMa} 
(also \eqref{kM}).  Then Proposition \ref{P:6.2} applies to tell us that the kernel function for $\cM_{\beta}$ is given 
by the formula \eqref{kdif}.  Then Lemma \ref{L:6.8} applies to tell us that $k_{\cM_{\beta}}(z, \zeta)$ 
has the factorization as in \eqref{id6}, and that the operator 
$M_{\Theta_{\bo,\beta}} \colon z^{\beta} u \mapsto \Theta_{\bo, \beta}(z) z^{\beta} u$ 
is unitary from $z^{\beta} \cU_{\beta}$ onto $\cM_{\beta}$ for each $\beta  \in \free$ as required.
\end{proof}

As a corollary we see that any $H^2_{\bo, \cY}(\free)$-Bergman-inner family 
$\{ \Theta_{\bo, \beta}\}_{\beta \in \free}$ has a realization as in \eqref{innerfam-real}.

\begin{corollary}  \label{C:innerfam-real}
suppose that $\{ \Theta_{\bo, \beta}\}_{\beta \in \free}$ is a 
$H^2_{\bo, \cY}(\free)$-Bergman-inner family.  Then there is a $\bo$-isometric output pair $(C, \bA)$ with
$\bA$ $\bo$-strongly stable embedded in a family of connection
 matrices $\left\{ \sbm{ A & \widehat B_{\alpha} \\ C & D_{\alpha}} 
 \colon \alpha \in \free \right\}$ satisfying the metric constraints 
 \eqref{isom} and \eqref{wghtcoisom} so that $\Theta_{\bo, \beta}(z)$ 
 has the transfer-function realization \eqref{innerfam-real}.
 \end{corollary}
 
 \begin{proof}  In the discussion preceding the statement of Theorem \ref{T:BL3}, 
 we saw that any $H^2_{\bo, \cY}(\free)$-Bergman-inner family $\{ \Theta_{\bo, \beta}\}_{\beta \in \free}$ 
 generates a $\bS_{n,R}$-invariant subspace $\cM$ of $H^2_{\bo, \cY}$ having 
 $\{ \Theta_{\bo, \beta}\}_{\beta \in \free}$ as its Beurling-Lax representer as 
 in \eqref{tvmultop}.  We may then use the results of Theorem \ref{T:BL3} to generate a 
 transfer-function realization \eqref{innerfam-real} for $\Theta_{\bo, \beta}$ meeting all 
 the desired requirements.
\end{proof}


Note that, by Definition \ref{D:defin}, the content of conditions (1) 
and (2) in Definition \ref{D:innerfuncfam} is that each $\Theta_{\bo, \beta}(z) \cdot z^{\beta}$ 
coming from a $H^2_{\bo, \cY}(\free)$-Bergman-inner family $\{ \Theta_{\bo, \beta} \}_{\beta \in \free}$ is itself 
$H^2_{\bo, \cY}(\free)$-Bergman inner.  
It is condition (3) in Definition  \ref{D:innerfuncfam} which specifies how all 
these $H^2_{\bo,\cY}(\free)$-Bergman-inner power series fit together to form a 
$H^2_{\bo, \cY}(\free)$-Bergman-inner family.  In particular, the first Bergman-inner power series 
$\Theta_{\bo, \emptyset}$ of the $H^2_{\bo, \cY}(\free)$-Bergman-inner family $\{ \Theta_{\bo, 
\beta}\}_{\beta \in \free}$ has a transfer function realization 
$$
\Theta(z) = D + C R_{\bo,1}(Z(z) A) Z(z) \widehat B
$$
with connection matrix 
\begin{equation}
{\bf U}=\begin{bmatrix}A & \widehat B \\ C & D \end{bmatrix}=
\sbm {A_{1} & B_{1} \\ \vdots & \vdots \\ A_{d} & B_{d} \\ C & D}
\label{coll'}
\end{equation}
satisfying the metric constraints of the form  \eqref{isom} and 
\eqref{wghtcoisom}. Our next result is a converse to these observations.

\begin{theorem}   \label{T:inner-real}
    Suppose that the formal power series $\Theta(z) \in \cL(\cU, 
    \cY)\langle \langle z \rangle \rangle$ is $H^2_{\bo, \cY}(\free)$-Bergman-inner. Then:

\smallskip

{\rm (1)}  There is a $H^2_{\bo, \cY}(\free)$-Bergman-inner family $\{ \Theta_{\bo, \beta}\}_{\beta \in \free}$
so that $\Theta = \Theta_{\bo, \emptyset}$. 

\smallskip

{\rm (2)} $\Theta$ has the form
\begin{equation}  \label{inner-real}
    \Theta(z) = D + \sum_{j=1}^{d} \sum_{\alpha \in \free}
    \omega_{|\alpha| +1}^{-1} C \bA^{\alpha} B_{j} z^{\alpha j}
\end{equation}
where $(C, \bA)$ is $\bo$-isometric with $\bA$ $\bo$-strongly stable, and where the connection matrix \eqref{coll'}
satisfies the metric constraints
\begin{align}
 \begin{bmatrix} A^{*} & C^{*} \\ B^{*} & D^{*} \end{bmatrix}
    \begin{bmatrix} \Gr_{\bo,1,C,\bA} \otimes I_{d} & 0 \\ 0 & I_{\cY}\end{bmatrix}
        \begin{bmatrix} A & B \\ C & D \end{bmatrix} &=
\begin{bmatrix} \cG_{\bo, C,  \bA} & 0 \\ 0 & I_{\cU} \end{bmatrix},
    \notag \\
 \begin{bmatrix} A & B \\ C & D \end{bmatrix} \begin{bmatrix}
\cG^{-1}_{\bo,C,\bA} & 0 \\ 0 & I_{\cU} \end{bmatrix}
\begin{bmatrix} A^{*} & C^{*} \\ B^{*} & D^{*} \end{bmatrix}& =
    \begin{bmatrix} \Gr^{-1}_{\bo,1,C,\bA} \otimes I_{d} & 0 \\ 0 &
        I_{\cY} \end{bmatrix}.  \label{inner-metric}
\end{align}
\end{theorem}

\begin{proof}  Suppose that $\Theta(z) \in \cL(\cU, \cY)\langle 
    \langle z \rangle \rangle$ is $H^2_{\bo, \cY}(\free)$-Bergman-inner.  
    Define $\cE \subset H^2_{\bo, \cY}(\free)$ by $\cE = M_{\Theta} 
    \cdot \cU$.
Then condition (1) in  Definition \ref{D:defin} tells us that 
$M_{\Theta}$ maps $\cU$ isometrically into $H^2_{\bo, \cY}(\free)$ and 
hence $\cE$ is a closed subspace of $H^2_{\bo, \cY}(\free)$.  We next 
set  
\begin{equation}
\cM = \bigvee_{\alpha \in \free} \bS_{\bo,R}^{\alpha} \cE.
\label{again38}
\end{equation}
Then $\cM$ is a closed $\bS_{\bo,R}$-invariant subspace of $H^2_{\bo,\cY}$.  By Theorem \ref{T:BL3}
it follows that $\cM$ has a representation of the form \eqref{BLrep2} in terms of a 
$H^2_{\bo, \cY}(\free)$-Bergman-inner family $\{ \Theta_{\bo, \beta}\}_{\beta \in \free}$.  
By the second condition in  Definition \ref{D:defin}, we see that, for 
$\beta$ a nonempty word in $\free$, the subspace $\cE =\Theta  \cdot \cU$ is orthogonal to 
the subspace
$$ 
\bS_{\bo,R}^{\beta} \cM =  \bS_{\bo,R}^{\beta} \bigg( 
\bigvee_{\alpha \in \free} \bS_{\bo,R}^{\alpha} \cE \bigg) = 
\bigvee_{\alpha \in \free} \bS_{\bo,R}^{\beta} \bS_{\bo,R}^{\alpha} \Theta \cdot \cU = 
\bigvee_{\alpha \in \free} \Theta \cdot z^{\alpha^\top \beta^\top} \cU.
$$
In particular, $\cE$ is orthogonal to $S_{\bo,R,k} \cM$ for $k=1,\ldots,d$ 
and hence 
\begin{equation}   \label{contain}
\cE\subseteq \cM \ominus \bigg(\bigoplus_{k=1}^{d} 
S_{\bo,R,k} \cM \bigg) =: \cM_{\emptyset}.
\end{equation}
To show that the containment \eqref{contain} actually holds with equality, we verify that the assumption that 
$f \in \cM_{\emptyset}$ is orthogonal to $\cE$ forces $f=0$. Indeed, 
if $f \in \cM_{\emptyset}$, then $f$ is orthogonal to $S_{\bo,R,k} \cM$ for $k=1,\ldots,d$ and 
hence, on account of \eqref{again38}, $f$ is orthogonal to $\bS_{\bo,R}^{\alpha}\cE$ for any nonempty 
$\alpha\in\free$.  If in addition, $f$ is orthogonal $\to\cE$, then, again by, \eqref{again38}, 
it is orthogonal to the whole $\cM$. Since 
$f\in\cM_\emptyset\subset \cM$, it follows that $f=0$,
and hence the containment \eqref{contain} holds with equality.
This means that without loss of generality we may take $\Theta_{\bo, \emptyset}$ 
to be $\Theta$ in the construction in Theorem \ref{T:inner-real}.  Consequently, $\Theta = 
\Theta_{\bo, \emptyset}$ has a realization as in \eqref{inner-real} and 
\eqref{inner-metric} by specialization of the general state-space 
formulas for the whole Bergman-inner family in the second part of 
Theorem \ref{T:BL3} to the case $\beta = \emptyset$.
\end{proof}

\begin{remark} \label{R:sysinterpret}
\textbf{System-theoretic interpretation of Theorem \ref{T:BL3}.}  
In systems theory language, the result of Theorem \ref{T:BL3} can be 
expressed as follows.  Associated with any time-varying noncommutative 
multidimensional linear system $\Sigma(\bU)$ of the form \eqref{1.31pre} is a well-defined 
input-output map
$$
  T_{\bU} \colon \{ u_{\beta} \}_{\beta \in \free} \to \{ 
  y_{\beta}\}_{\beta \in  \free},
$$
where the output string $\{y_{\beta}\}_{\beta \in \free}$ is determined 
from the input string $\{u_{\beta}\}_{\beta \in \free}$ ($u_{\beta} \in 
\cU_{\beta}$ for each $\beta \in \free$) by solving  the system equations 
\eqref{1.31pre}  recursively with initial condition $x(\emptyset)$ set 
equal to zero.  The formal noncommutative $Z$-transform of $T_{\bU}$ 
is the map $\widehat T_{\bU}$ defined by 
$$
  \widehat T_{\bU} \colon \sum_{\beta \in \free} u_{\beta} z^{\beta} \mapsto 
  \sum_{\beta \in \free} y_{\beta} z^{v}
$$
exactly when $T_{\bU} \colon \{u_{\beta}\}_{\beta \in \free} \mapsto \{ 
y_{\beta} \}_{\beta \in \free}$.  Let us say that a time-varying formal 
noncommutative multidimensional linear system $\Sigma_{\bU}$ as in 
\eqref{1.31pre} is {\em $\bo$-Bergman conservative} if the operators $A, 
\widehat B_{\beta}, C, D_{\beta}$ in the connection matrix \eqref{coll'}
satisfy the metric conditions \eqref{isom} and \eqref{wghtcoisom}.  
As a consequence of the identity \eqref{1.36prehardy} (with initial 
condition $x = \widehat x(\emptyset) = 0$), we see that the 
multiplication operator $M_{\Theta}$ has transfer function 
realization $M_{\Theta} = \widehat T_{\bU}$. 
\end{remark}

\section{Expansive multiplier property}  \label{S:exp-mult-wght}
In this section we take advantage of the realization formula \eqref{inner-real}
to discuss another property of $H^2_{\bo,\cY}(\free)$-inner multipliers
(which, however, occurs only for particular weights). The question goes back to \cite{heden1,
DKSS, DKSS1, borhed}, where it was shown that canonical divisors in  the Bergman space
$\cA_2$ (also in the non-Hilbert Bergman spaces $\cA^p_2$) are expansive multipliers in the sense
that $\|Gf\|_{A_2}\ge \|f\|_{A_2}$ for all $f\in A_2$. Canonical divisors form a subclass of
Bergman-inner functions for which the expansive multiplier property holds only in a weaker
version: for any Bergman-inner function $G$, the inequality $\|Gf\|_{\cA_2}\ge \|f\|_{\cA_2}$
holds for all $f\in H^2$ (but not for all $f\in \cA_2$, in genera)l. The latter weak form
of the expansive multiplier property holds for weighted Bergman spaces $\cA_3$ \cite{heden2}
but fails in $\cA_n$ for $n>3$.

\smallskip

In the present noncommutative setting, we define the expansive multiplier (in the weak form)
keeping in mind the Fock space $H^2(\free)$ as a suitable substitute of the Hardy space $H^2$ of
the unit disk.

\smallskip

Let us say that the space $H^2_{\bo,\cY}(\free)$ possesses the {\em expansive multiplier property}
if for any $\cL(\cU,\cY)$-valued $H^2_{\bo,\cY}(\free)$-inner multiplier $\Theta$,
\begin{equation}
\|\Theta f\|_{H^2_{\bo,\cY}(\free)}\ge \|f\|_{H^2_{\bo,\cU}(\free)}\quad \text{for all}\quad
f\in H^2_{\cU}(\free).
\label{18.5}
\end{equation}

\begin{remark}
If $\Theta\in\cL(\cU,\cY)\langle\langle z\rangle\rangle$ is $H^2_{\bo,\cY}(\free)$-inner, then for
every noncommutative polynomial $p(z) = p_{\emptyset} + \widetilde
p(z)$ where $\widetilde p(z) = \sum_{1 \le |\alpha| \le K} p_{\alpha}
z^{\alpha}$, we have
$$
\|\Theta p\|^2_{H^2_{\bo,\cY}(\free)}=\|\Theta p_\emptyset\|^2_{H^2_{\bo,\cY}(\free)}
+\|\Theta \widetilde{p}\|^2_{H^2_{\bo,\cY}(\free)}
=\|p_\emptyset\|^2_{\cU}+\|\Theta \widetilde{p}\|^2_{H^2_{\bo,\cY}(\free)}.
$$
Since $\; \|p\|^2_{H^2_{\bo,\cU}(\free)}=
\|p_{\emptyset}\|^2_{\cU}+\|\widetilde{p}\|^2_{H^2_{\bo,\cU}(\free)}, \;$
we conclude that the inequality
$$
\|\Theta p\|_{H^2_{\bo,\cY}(\free)}\ge \|p\|_{H^2_{\bo,\cU}(\free)}
$$
holds for all polynomials in $\cU \langle z \rangle$ if and only if it holds for all polynomials
$p$ with $p_\emptyset=0$.
\label{R:18.2}
\end{remark}

\begin{lemma} \label{L:18.3}
Let $\Theta$ be an $\cL(\cU,\cY)$-valued $H^2_{\bo,\cY}(\free)$-inner power series realized as
in Theorem \ref{T:inner-real}. Then for every polynomial
\begin{equation}
p(z)=\sum_{\alpha\in\free: 1\le|\alpha|\le k}p_{\alpha} z^\alpha,
\label{18.7}
\end{equation}
the following equality holds:
\begin{align}
&\|\Theta p\|^2_{H^2_{\bo,\cY}(\free)}-\|p\|^2_{H^2_{\bo,\cU}(\free)}\label{18.8}\\
&  =\sum_{v\in\free:|v|\ge 2}\sum_{\alpha \bi \alpha^\prime=v=\beta
\bj \beta^\prime} \left(\frac{\omega_{|v|}}{\omega_{|\alpha|+1}\omega_{|\beta|+1}}-
\frac{\omega_{{\rm max}(|\alpha^\prime|,|\beta^\prime|)}}
{\omega_{|v|-{\rm max}(|\alpha^\prime|,|\beta^\prime|)}\cdot
\omega_{||\alpha^\prime|-|\beta^\prime||}}\right)\notag\\
&\qquad\qquad\qquad\qquad \cdot \big\langle C\bA^{\alpha}B_\bi p_{\alpha^\prime},
\,  C\bA^{\beta}B_\bj p_{\beta^\prime}\big\rangle_{\cY}
\notag
\end{align}
where, according to \eqref{18.1},
$\|p\|^2_{H^2_{\bo,\cU}(\free)}={\displaystyle\sum_{\alpha\in\free: 1\le|\alpha|\le
k}\omega_{|\alpha|}\|p_\alpha\|^2_{\cU}}$.
\end{lemma}

\begin{proof}
Making use of \eqref{inner-real} and \eqref{18.7}, we write the power series expansion
\begin{align}
\Theta(z)p(z)=&D\sum_{j=1}^d p_jz_j+
\sum_{v\in\free: \, 2\le|v|\le k} \bigg(Dp_v+\sum_{\alpha \bi \alpha^\prime=v}\omega_{|\alpha|+1}^{-1}
C\bA^\alpha B_\bi p_{\alpha^\prime}\bigg)z^v\notag\\
&+\sum_{v\in\free: \, |v|> k} \bigg(\sum_{\alpha \bi \alpha^\prime=v}\omega_{|\alpha|+1}^{-1}
C\bA^\alpha B_\bi p_{\alpha^\prime}\bigg)z^v.\notag
\end{align}
By the definition \eqref{18.1} of the $H^2_{\bo,\cY}(\free)$-norm, we have
\begin{align}
\|\Theta p\|^2_{H^2_{\bo,\cY}(\free)}=&\sum_{i=1}^d
\omega_{1}\cdot \|Dp_i\|^2_{\cY}\notag\\
&+\sum_{v\in\free: \, 2\le|v|\le k}\omega_{|v|}\cdot \big\|
Dp_v+\sum_{\alpha \bi \alpha^\prime=v}\omega_{|\alpha|+1}^{-1}
C\bA^\alpha B_\bi p_{\alpha^\prime}\big\|^2_{\cY}\notag\\
&+\sum_{v\in\free: \, |v|> k}\omega_{|v|}\cdot \big\|
\sum_{\alpha \bi \alpha^\prime=v}\omega_{|\alpha|+1}^{-1}
C\bA^\alpha B_\bi p_{\alpha^\prime}\big\|^2_{\cY}\notag\\ 
&= {\rm I} + {\rm II} + {\rm III}, \label{18.9}
\end{align}
where
\begin{align}
{\rm I} & = \sum_{v \in \free \colon |v| \le k} \omega_{|v|}\cdot
\|Dp_v\|^2_{\cY}, \label{dI}\\
{\rm II} & = 2 {\rm Re} \bigg(\sum_{v\in\free: \, 2\le|v|\le k}\omega_{|v|}\cdot
\bigg\langle Dp_v, \, \sum_{\alpha \bi \alpha^\prime=v}\omega_{|\alpha|+1}^{-1}
C\bA^\alpha B_\bi p_{\alpha^\prime}\bigg\rangle_{\cY}\bigg), \label{dII}\\
{\rm III} & = \sum_{v\in\free: \, |v|\ge 2} \omega_{|v|}\cdot\bigg\|\sum_{\alpha
\bi \alpha^\prime=v}\omega_{|\alpha|+1}^{-1}
C\bA^\alpha B_\bi p_{\alpha^\prime}\bigg\|^2_{\cY}.\label{dIII}
\end{align}
To simplify the notation we set $p_{v} = 0$ for $|v|>k$ and for any two representations
\begin{equation}
v=\alpha \bi \alpha^\prime=\beta \bj \beta^\prime \quad (|\alpha^\prime|\le k, \; |\beta^\prime|\le k)
\label{18.15}
\end{equation}
of a word $v\in\free$,  we let
$$
d_{v,\alpha^\prime,\beta^\prime}=\frac{\omega_{|v|}}{\omega_{|\alpha|+1}\cdot
\omega_{|\beta|+1}}\cdot\big\langle C\bA^\alpha B_\bi p_{\alpha^\prime}, \,
C\bA^\beta B_\bj p_{\beta^\prime}\big\rangle_{\cY}.
$$
Observe that the elements $\alpha$, $\beta$, $\bi$ and $\bj$ in representations \eqref{18.15}
are completely determined by $v$, $\alpha^\prime$ and $\beta^\prime$. 
We also note that if $|\alpha^\prime|=|\beta^\prime|$, then it follows from \eqref{18.15} that
$\alpha^\prime=\beta^\prime$, $\alpha=\beta$ and $\bi=\bj$  (see
Remark \ref{R:determined} below).
In case $|\alpha^\prime|>|\beta^\prime|$,
it follows that
\begin{equation}
\alpha^\prime=\gamma \bj\beta^\prime\quad\mbox{and}\quad \beta=\alpha
\bi \gamma\quad\mbox{for
some}\quad \gamma\in\free.
\label{18.16}
\end{equation}
Using the latter notation we rewrite \eqref{dIII} as
\begin{equation}
{\rm III}=\sum_{v\in\free}\sum_{\alpha \bi \alpha^\prime=v}
d_{v,\alpha^\prime,\alpha^\prime}+2{\rm Re}
\bigg(\sum_{v\in\free}\sum_{{\scriptsize\begin{array}{c}\alpha \bi
\alpha^\prime=v=\beta \bj \beta^\prime\\
|\beta^\prime|<|\alpha^\prime|\le k\end{array}}}d_{v,\alpha^\prime,\beta^\prime}\bigg).
\label{18.17}
\end{equation}
For the first term on the right side, we have
\begin{align}
\sum_{v\in\free}\sum_{\alpha \bi \alpha^\prime=v}
d_{v,\alpha^\prime,\alpha^\prime}
&=\sum_{v\in\free}\sum_{\alpha \bi \alpha^\prime=v}
\frac{\omega_{|v|}}{\omega^2_{|\alpha|+1}}\cdot \left \|
C\bA^\alpha B_\bi p_{\alpha^\prime}\right\|^2_{\cY}\notag \\
&=\sum_{\alpha,\alpha^\prime\in\free}\sum_{\bi=1}^d
\frac{\omega_{|\alpha|+|\alpha^\prime|+1}}{\omega^2_{|\alpha|+1}}\cdot
\left \|C\bA^\alpha B_\bi p_{\alpha^\prime}\right\|^2_{\cY}.\label{18.18}
\end{align}
For the second term on the right side of \eqref{18.17}, we have, on account of
\eqref{18.16},
\begin{align}
&\sum_{v\in\free}\sum_{{\scriptsize\begin{array}{c}\alpha \bi \alpha^\prime=v=\beta
\bj \beta^\prime\\ |\beta^\prime|<|\alpha^\prime|\le
k\end{array}}}d_{v,\alpha^\prime,\beta^\prime}\label{18.19}\\
&=\sum_{{\scriptsize\begin{array}{c}
\alpha \bi \alpha^\prime=\beta \bj\beta^\prime,\\
|\beta^\prime|<|\alpha^\prime|\le k\end{array}}}
\frac{\omega_{|\alpha|+|\alpha^\prime|+1}}{\omega_{|\alpha|+1}\cdot
\omega_{|\beta|+1}}\cdot\big\langle C\bA^\alpha B_\bi p_{\alpha^\prime}, \,
C\bA^\beta B_\bj p_{\beta^\prime}\big\rangle_{\cY}\notag\\
&=\sum_{{\scriptsize\begin{array}{c}
\alpha^\prime=\gamma \bj\beta^\prime,\\
|\alpha^\prime|\le k\end{array}}}\sum_{j=1}^d
\frac{\omega_{|\alpha|+|\alpha^\prime|+1}}{\omega_{|\alpha|+1}\cdot
\omega_{|\alpha|+|\gamma|+2}}
\cdot \big\langle C\bA^\alpha B_\bi
p_{\alpha^\prime}, \,
C\bA^{\alpha \bi \gamma} B_\bj p_{\beta^\prime}\big\rangle_{\cY}.\notag
\end{align}
Substituting \eqref{18.18} and \eqref{18.19} into \eqref{18.17} gives
\begin{align}
&{\rm III}=\sum_{\alpha,\alpha^\prime\in\free}\sum_{\bi=1}^d
\frac{\omega_{|\alpha|+|\alpha^\prime|+1}}{\omega^2_{|\alpha|+1}}\cdot
\left \|C\bA^\alpha B_\bi p_{\alpha^\prime}\right\|^2_{\cY}\label{18.20}\\
&+2{\rm Re}\bigg(
\sum_{{\scriptsize\begin{array}{c}
\alpha^\prime=\gamma \bj \beta^\prime,\\
|\alpha^\prime|\le k\end{array}}}\sum_{\bi=1}^d
\frac{\omega_{|\alpha|+|\alpha^\prime|+1}}{\omega_{|\alpha|+1}\cdot
\omega_{|\alpha|+|\gamma|+2}}\cdot\big\langle C\bA^\alpha B_\bi
p_{\alpha^\prime}, \,
C\bA^{\alpha \bi \gamma} B_\bj p_{\beta^\prime}\big\rangle_{\cY}\bigg).\notag
\end{align}
By equality of bock (2,2)-entries in the first matrix equality in
\eqref{inner-metric} and on account of \eqref{4.32},
\begin{align*}
\left\|Dp_v\right\|^2_{\cY}=&\left\|p_v\right\|^2_{\cU}-\sum_{\bi=1}^d\langle
\Gr_{\bo,1,C,\bA}B_\bi p_v, \, B_\bi p_v\rangle_{\cX}\\
=&\left\|p_v\right\|^2_{\cU}-\sum_{\bi=1}^d\bigg\langle
\sum_{\alpha\in\free}\omega_{|\alpha|+1}^{-1}\bA^{*\alpha^\top}C^*C\bA^{\alpha}B_\bi p_v, \,
B_\bi p_v\bigg\rangle_{\cX}\\
=&\left\|p_v\right\|^2_{\cU}-\sum_{\bi=1}^d\sum_{\alpha\in\free}\omega_{|\alpha|+1}^{-1}
\cdot \big\|C\bA^{\alpha}B_\bi p_v\big\|^2_{\cX}
\end{align*}
and therefore (see \eqref{dI}),
\begin{align}
{\rm I} &= \sum_{v\in\free: \, 1\le |v|\le k} \omega_{|v|}\cdot \|p_v\|^2_{\cU}
-\sum_{v\in\free: \, 1\le |v|\le m}\sum_{\bi=1}^d\sum_{\alpha\in\free}
\frac{\omega_{|v|}}{\omega_{|\alpha|+1}}\cdot
\big\|C\bA^{\alpha}B_\bi p_v\big\|^2_{\cX}\notag\\
&= \|p\|^2_{H^2_{\bo,\cU}(\free)}
-\sum_{\alpha,\alpha^\prime\in\free}\sum_{\bi =1}^d
\frac{\omega_{|\alpha^\prime|}}{\omega_{|\alpha|+1}}\cdot
\big\|C\bA^{\alpha}B_\bi p_{\alpha^\prime}\big\|^2_{\cX}.
\label{18.21}
\end{align}
We next combine the equality of the (2,1)-block entries in the first
matrix equation in \eqref{inner-metric} with \eqref{4.32} to get
\begin{align*}
\big\langle Dp_{\alpha^\prime}, \, C\bA^\gamma B_\bi p_{\beta^\prime}\big\rangle_{\cY}=&
\bigg\langle p_{\alpha^\prime}, \, D^*C\bA^\gamma B_\bj
p_{\beta^\prime}\bigg\rangle_{\cU}\\
=&-\bigg\langle p_{\alpha^\prime}, \, \sum_{\bi=1}^d
B_\bi^*\Gr_{\bo,1,C,\bA}\bA^{\bi\gamma}B_\bj
p_{\beta^\prime}\bigg\rangle_{\cU}\\
=&-\bigg\langle p_{\alpha^\prime}, \,
\sum_{\bi=1}^d\sum_{\alpha\in\free}\omega_{|\alpha|+1}^{-1}B_\bi^*
\bA^{*\alpha^\top}C^*C\bA^{\alpha \bi \gamma}B_{\bj} p_{\beta^\prime}\bigg\rangle_{\cU}\\
=&-\sum_{\alpha\in\free}\sum_{\bi=1}^d\omega_{|\alpha|+1}^{-1}\cdot \big\langle
C\bA^{\alpha}B_\bi p_{\alpha^\prime}, \,  C\bA^{\alpha \bi
\gamma}B_\bj p_{\beta^\prime}\big\rangle_{\cY}
\end{align*}
which together with \eqref{dII} and \eqref{18.16} gives
\begin{align}
{\rm II}= & 2 {\rm Re} \bigg(\sum_{\alpha^\prime\in\free: \, |\alpha^\prime|\le
k}\omega_{|\alpha^\prime|}\cdot
\bigg\langle Dp_{\alpha^\prime}, \, \sum_{\alpha^\prime=\gamma \bj\beta^\prime}
\omega_{|\gamma|+1}^{-1}
C\bA^\gamma B_\bj p_{\beta^\prime}\bigg\rangle_{\cY}\bigg)\label{18.22}\\
=&-2 {\rm Re}
\bigg(\sum_{{\scriptsize\begin{array}{c}
\alpha^\prime=\gamma \bj \beta^\prime,\\
|\alpha^\prime|\le k\end{array}}}\sum_{\bi=1}^d
\frac{\omega_{|\alpha^\prime|}}{\omega_{|\alpha|+1}\cdot\omega_{|\gamma|+1}}
\cdot \big\langle C\bA^{\alpha}B_\bi p_{\alpha^\prime},
\,  C\bA^{\alpha \bi \gamma}B_\bj p_{\beta^\prime}\big\rangle_{\cY}\bigg).\notag
\end{align}
We now substitute \eqref{18.20}, \eqref{18.21} and \eqref{18.22} into \eqref{18.9}:
\begin{align}
&\|\Theta p\|^2_{H^2_{\bo,\cY}(\free)}-\|p\|^2_{H^2_{\bo,\cU}(\free)}\notag\\
&=\sum_{\alpha,\alpha^\prime\in\free:|\alpha^\prime|\le k}\sum_{\bi=1}^d
\bigg(\frac{\omega_{|\alpha|+|\alpha^\prime|+1}}{\omega^2_{|\alpha|+1}}
-\frac{\omega_{|\alpha^\prime|}}{\omega_{|\alpha|+1}}\bigg)\cdot
\left \|C\bA^\alpha B_\bi p_{\alpha^\prime}\right\|^2_{\cY}\notag\\
&\qquad +2 {\rm Re}
\bigg(\sum_{{\scriptsize\begin{array}{c}
\alpha^\prime=\gamma \bj \beta^\prime,\\
|\alpha^\prime|\le k\end{array}}}\sum_{\bi=1}^d\left(
\frac{\omega_{|\alpha|+|\alpha^\prime|+1}}{\omega_{|\alpha|+1}\cdot
\omega_{|\alpha|+|\gamma|+2}}-
\frac{\omega_{|\alpha^\prime|}}{\omega_{|\alpha|+1}\cdot\omega_{|\gamma|+1}}\right)\notag\\
&\qquad\qquad\qquad\qquad \cdot \big\langle C\bA^{\alpha}B_\bi p_{\alpha^\prime},
\,  C\bA^{\alpha j\gamma}B_\bj p_{\beta^\prime}\big\rangle_{\cY}\bigg).\notag
\end{align}
Making use of representations \eqref{18.15} and \eqref{18.16} we rewrite the latter equality as
\begin{align}
&\|\Theta p\|^2_{H^2_{\bo,\cY}(\free)}-\|p\|^2_{H^2_{\bo,\cU}(\free)}\notag\\
&=\sum_{v\in\free: |v|\ge 2}\sum_{\alpha \bi \alpha^\prime=v}
\bigg(\frac{\omega_{|v|}}{\omega^2_{|\alpha|+1}}
-\frac{\omega_{|\alpha^\prime|}}{\omega_{|\alpha|+1}}\bigg)\cdot
\left \|C\bA^\alpha B_\bi p_{\alpha^\prime}\right\|^2_{\cY}\notag\\
&\qquad +2 {\rm Re}
\bigg(\sum_{v\in\free:
|v|\ge 2}\sum_{{\scriptsize\begin{array}{c}\alpha \bi
\alpha^\prime=v=\beta \bj \beta^\prime\\
\alpha^\prime=\gamma \bj \beta^\prime,\\
|\alpha^\prime|\le k\end{array}}}\left(
\frac{\omega_{|v|}}{\omega_{|\alpha|+1}\cdot \omega_{|\beta|+1}}-
\frac{\omega_{|\alpha^\prime|}}{\omega_{|\alpha|+1} \cdot
\omega_{|\alpha^\prime|-|\beta^\prime|}}\right)\notag\\
&\qquad\qquad\qquad\qquad \cdot \big\langle C\bA^{\alpha}B_\bi p_{\alpha^\prime},
\,  C\bA^{\beta}B_\bj p_{\beta^\prime}\big\rangle_{\cY}\bigg),\notag
\end{align}
which can be written, since $\omega_0=1$, in a more compact form as in \eqref{18.8}.
\end{proof}

\begin{remark}  \label{R:determined}
There is an alternative way of writing the right-hand side of
\eqref{18.8} which will be more convenient for our purposes.  Fix a
word $v \in \free$ with $|v| \ge 2$ and consider a splitting of
$v$ as
\begin{equation}   \label{v-decom}
 v = \alpha \bi \alpha'
 \end{equation}
where $1 \le |\alpha'| \le |v| - 1$.
We note that once the length $i = |\alpha'|$ is specified then each
of the three factors $\alpha, \bi, \alpha'$ in the decomposition
\eqref{v-decom} is uniquely
determined.  We may therefore write $\alpha = \alpha_{v, i}$, $\bi =
\bi_{v, i}$, $\alpha' = \alpha'_{v, i}$. Similarly, in the
decomposition $v = \beta \bj \beta'$, we may write $\beta =
\alpha_{v, j}$, $\bj = \bi_{v,j}$, $\beta' = \alpha'_{v,j}$.
Then we also have
$$  |\alpha'_{v,i}| = i, \quad |\alpha_{v,i}| + 1 = |v| - i.
$$
Then the right-hand side of \eqref{18.8} has the form
\begin{align}
&  \sum_{v \in \free \colon |v| \ge 2} \sum_{i,j=1}^{|v|-1}
\left( \frac{ \omega_{|v|}}{\omega_{|v|-i} \, \, \omega_{|v| - j}} -
\frac{ \omega_{\max(i,j)}}{\omega_{|v| - \max(i,j)} \cdot \omega_{\max(||i|- |j||}} \right)
\notag \\
& \quad \quad \quad \quad
\cdot \langle C \bA^{\alpha_{v,i}} B_{\bi_{v,i}} p_{\alpha'_{v,i}},
C \bA^{\alpha_{v,j}} B_{\bi_{v,j}} p_{\alpha'_{v,j}} \rangle.
\label{rewrite}
\end{align}
\end{remark}

The following theorem presents sufficient conditions (in terms of the weight sequence $\bo$)
for the space $H^2_{\bo,\cY}(\free)$ to possess the expansive multiplier property.

\begin{theorem}
\label{T:18.3}
Let $\bo=\{\omega_j\}_{j\ge 0}$ be the weight sequence satisfying conditions \eqref{18.2}.
For every pair $(k,r)$   of positive integers $k<r$, define the symmetric
real
matrix
$$
M^{(k,r)}=\left[M^{(k,r)}_{ij}\right]_{i,j=1}^k
$$
with the entries $M^{(k,r)}_{ij}$ given by
\begin{equation}
M^{(k,r)}_{ij}=M^{(k,r)}_{ji}=1-\frac{\omega_j\omega_{r-i}}{\omega_r\omega_{j-i}}\quad\mbox{for}\quad
1\le i\le j\le k.
\label{18.23}
\end{equation}
If $\det M^{(k,r)}\ge 0$ for all $1\le k<r$, then $H^2_{\bo,\cY}(\free)$ possesses
the expansive multiplier property, i.e.,  condition
\eqref{18.5} holds for any $H^2_{\bo,\cY}(\free))$-inner power series $\Theta$.
\end{theorem}
\begin{proof} For fixed $0<k<r$, let
$\Omega$ be the $k\times k$ diagonal matrix with the weights
$\omega_{r-1}, \omega_{r-2},\ldots\omega_{r-k}$ on the main diagonal. It is readily checked that
$$
N^{(r,k)}:=\left[\frac{\omega_{r}}{\omega_{r-j}\omega_{r-i}}-
\frac{\omega_{{\rm max}(i,j)}}{\omega_{r-{\rm max}(i,j)}\cdot \omega_{|j-i|}}
\right]_{i,j=1}^k= \omega_r \Omega^{-1} M^{(r,k)}\Omega^{-1},
$$
and hence, the assumption $M^{(r,k)}\ge 0$ implies $N^{(r,k)}\ge 0$.
Note that the right-hand side of \eqref{18.8}, rewritten as in
\eqref{rewrite}, becomes
$$  
 \sum_{v \in \free \colon |v| \ge 2}  \sum_{i,j=1}^{|v|-1} [N^{(|v|,
 |v|-1}]_{i,j} \cdot \langle C \bA^{\alpha_{v,i}} B_{\bi_{v,i}} p_{\alpha'_{v,i}},
C \bA^{\alpha_{v,j}} B_{\bi_{v,j}} p_{\alpha'_{v,j}} \rangle.
$$
Since $N^{(r,k)}$ is positive semidefinite, we see immediately that
this expression is nonnegative for any choice of operators $A$, $B_j$,
$C$ and vectors $p_\alpha\in\cU$. Therefore, inequality \eqref{2.4a} holds for all
polynomials vanishing at the  origin and now Remark \ref{R:18.2} completes the proof.
\end{proof}
If $\omega_j=\mu_{2,j}=\frac{1}{j+1}$, then \eqref{18.23} takes the form
$M^{(k,r)}_{ij}=\frac{i(r-j)}{(j+1)(r-i+1)}$, and it was shown in \cite{BBkarm} that
$$
\det M^{(k,r)}=\frac{2^k(r-k)(r+1)^{k-1}(r-k+1)!}{r^2(k+1) (k+1)!\, (r-1)!}\quad\mbox{for all}\quad 2\le
k<r.
$$
Thus, $\det M^{(k,r)}>0$ for all $1\le k<r$ and thus,
the expansive multiplier property holds in the space $\cA_{2,\cY}(\free)$ by Theorem \ref{T:18.3}.

\smallskip

If $\omega_j=\mu_{3,j}=\frac{2}{(j+1)(j+2)}$, then \eqref{18.23} takes the form
$$
M^{(k,r)}_{ij}=M^{(k,r)}_{ji}=1-\frac{(r+1)(r+2)(j-i+1)(j-i+2)}{(j+1)(j+2)(r-i+1)(r-i+2)}.
$$
It was shown in \cite{BBkarm} that also in this case, $\det M^{(k,r)}>0$ for all $1\le k<r$.
By Theorem \ref{T:18.3}, we conclude that the Bergman-Fock space  $\cA_{3,\cY}(\free)$
also possesses the expansive multiplier property.

\section[Extremal solutions of interpolation problems]
{Bergman-inner multipliers as extremal solutions of interpolation problems}
\label{S:ncint}

The Bergman-inner function associated with a sequence of nonzero points $\{z_j\}$ in the unit disk 
${\mathbb D}$  appeared in \cite{heden1, DKSS} as a unit-norm element $f_{0}$ of the subspace 
$\cM$ of the Bergman space consisting of all functions $f$ with zero set containing $\{z_{j}\}$ 
which achieves the maximal possible value of $|f(0)|$. In this section we show 
$H^2_{\bo,\cY}(\free)$-Bergman-inner multipliers (see Definition \ref{D:defin})
appear as extremal solutions of certain operator-argument interpolation problems in the space 
$H^2_{\bo, \cL(\cU, \cY)}: = H^2_{\bo}(\free)\otimes\cL(\cU,\cY)$.

\smallskip

Being equipped with the $\cL(\cU)$-valued inner product
\begin{equation} \label{17.1}
[F,G]=\sum_{\alpha\in\free}\omega_{|\alpha|}G_\alpha^*F_\alpha
\quad\mbox{for}\quad F(z) = {\displaystyle\sum_{\alpha \in \free} F_{\alpha} z^{\alpha}},
\; \; G(z) = {\displaystyle\sum_{\alpha \in \free} G_{\alpha} z^{\alpha}},
\end{equation}
the space $H^2_{\bo}(\free) \otimes \cL(\cU,
\cY)$ becomes a $C^{*}$-module (sometimes also called a Hilbert
module) over the $C^{*}$-algebra $\cA = \cL(\cU)$ (see e.g.\ \cite{BlM}).
Note that here we take the inner product to be linear in the first
argument (as is standard for Hilbert spaces for the case where the
$C^{*}$-algebra $\cA$ is equal to the complex numbers ${\mathbb C}$
but is the reverse of the standard convention for $C^{*}$-modules).
It is readily seen from \eqref{17.1} and the definition \eqref{18.1} of the norm in $H^2_{\bo}(\free)$
that
\begin{equation}
[F,F]_{H^2_{\bo,\cL(\cU, \cY)}(\free)}=I_\cU\quad\Longleftrightarrow
\quad \|Fu\|_{H^2_{\bo,\cY}(\free)}=\|u\|_\cU\quad
\mbox{for all $u\in\cU$}.
\label{17.2}
\end{equation}
For an $\bo$-output-stable pair $(E, {\bf T})$ with ${\bf T}=(T_1,\ldots,T_d)$, we define a
{\em left-tangential functional calculus}
$F \to (E^{*}F)^{\wedge  L}(\bT^{*})$ on $H^2_{\bo,\cL(\cU,\cY)}(\free)$ by
          \begin{equation}
\label{2.1}
(E^{*} F)^{\wedge L}(\bT^{*}) = \sum_{ \alpha\in\free}
             \bT^{* \alpha^\top} E^{*} F_{\alpha}
        \end{equation}
for $F(z)\in H^2_{\bo,\cL(\cU,\cY)}(\free)$ as in \eqref{17.1}. The computation
        \begin{align*}
\bigg\langle \sum_{\alpha\in\free} \bT^{* \alpha^\top}E^{*}
F_{\alpha}u, \; x \bigg\rangle_{\cX} & =
\sum_{\alpha\in\free} \left\langle F_{\alpha}u, \; E\bT^{\alpha}x \right \rangle_{\cY} \\
& = \sum_{\alpha\in\free} \omega_{|\alpha|}
\left\langle F_{\alpha}u, \; \omega_{|\alpha|}^{-1}E\bT^{\alpha}x \right \rangle_{\cY}\\
&=\langle Fu, \; \mathcal {O}_{\bo,E,\bT} x\rangle_{H^2_{\bo,\cY}(\free)}
        \end{align*}
shows that the $\bo$-output-stability of the pair $(E, \bT)$ is exactly
what is needed to verify that the infinite series in the definition
\eqref{2.1} of $(E^{*}F)^{\wedge L}(\bT^{*})$ converges
in the weak topology on $\cX$. In fact, the left-tangential
evaluation with operator argument $f \to (E^{*}F)^{\wedge L}(\bT^{*})$
amounts to the adjoint of the $\bo$-observability operator:
$$
            (E^{*} F)^{\wedge L}(\bT^{*}) = \cO_{\bo,E, \bT}^{*}
            F \quad\text{for}\quad F \in H^2_{\bo,\cL(\cU,\cY)}(\free)
$$
and suggests the interpolation problem  with
operator argument {\bf OAP}$({\bf T}, E, N)$ whose data set consists
of a $d$-tuple ${\bf T}=(T_1,\ldots,T_d)$ and operators
$E\in{\mathcal L}(\cX,\cY)$ and $F\in{\mathcal L}(\cX,\cU)$ such that
the pair $(E,{\bT})$ is $\bo$-output stable. We assume in addition that $(E,{\bT})$
is exactly $\bo$-observable so that the gramian $\cG_{\bo,E, \bT}$ is strictly positive definite,

\medskip
\noindent
{\bf OAP$({\bf T}, E, N)$:} {\em
Given the data set $\{\bT, E, N\}$ as above,
find all $F\in H^2_{\bo,\cL(\cU,\cY)}(\free)$ such that
\begin{equation}
(E^*F)^{\wedge L}({\bT}^*):= \cO_{\bo,E, \bT}^{*}M_F\vert_{\cU}=N^*.
\label{2.4}
\end{equation}}
\begin{theorem}
All solutions $F\in H^2_{\bo,\cL(\cU,\cY)}(\free)$ of the problem \eqref{2.4}
are parametrized by the formula
\begin{equation}
F(z)=F_{\rm min}(z)+G(z),
\label{march9u}
\end{equation}
where
\begin{equation}
F_{\rm min}(z)=\sum_{\alpha\in\free}\omega_{|\alpha|}^{-1}E{\bf T}^\alpha \cG_{\bo,E,\bT}^{-1}N^*z^\alpha=
ER_{\bo}(Z(z)T)\cG_{\bo,E,\bT}^{-1}N^*
\label{march9a}
\end{equation}
and where $G(z) \in H^2_{\bo, \cL(\cU, \cY)}$ subject to $M_{G}|_{\cU}
\subset  ({\rm Ran} \, \cO_{\bo,E, \bT})^\perp$ is a free parameter.
Furthermore, the representation \eqref{march9a} is orthogonal with respect to the
inner product \eqref{17.1}. Therefore,
$$
[F,F]_{H^2_{\bo, \cL(\cU, \cY)}(\free)}=[F_{\rm min},F_{\rm
min}]_{H^2_{\bo, \cL(\cU, \cY)}(\free)}+[G,G]_{H^2_{\bo, \cL(\cU, \cY)}(\free)}
$$
so that $F_{\rm min}$ has the minimal $\cL(\cU)$-valued
self-inner-product (and hence also the minimum possible norm) among all solutions to the problem
\eqref{2.4}.
 \label{T:17.1}
\end{theorem}
\begin{proof}
We start with computations
\begin{align*}
&[F_{\rm min},F_{\rm min}]
=\sum_{\alpha\in\free}\omega_{|\alpha|}^{-1}
N\cG_{\bo,E,\bT}^{-1}{\bf T}^{*\alpha^\top}E^*E{\bf T}^\alpha \cG_{\bo,E,\bT}^{-1}N^*
=N\cG_{\bo,E,\bT}^{-1}N^*,\\
&(E^*F_{\rm min})^{\wedge L}({\bT}^*)=\sum_{\alpha\in\free}\omega_{|\alpha|}^{-1}
{\bf T}^{*\alpha^\top}E^*E{\bf T}^\alpha\cG_{\bo,E,\bT}^{-1}N^*=N^*
\end{align*}
showing that $F_{\rm min}(z)$ belongs to $H^2_{\bo, \cL(\cU, \cY)}(\free)$ and satisfies condition
\eqref{2.4}.
Next observe that $G\in H^2_{\bo, \cL((\cU, \cY)}(\free)$ satisfies the homogeneous condition
$\cO_{\bo,E, \bT}^{*}G(z)u$ $=0$ if and only if $Gu$ belongs to $({\rm Ran} \, \cO_{\bo,E, \bT})^\perp$.
Therefore,
\begin{equation}
(E^*G)^{\wedge L}({\bT}^*):= \cO_{\bo,E, \bT}^{*}M_G\vert_{\cU}=0 \; 
\Longleftrightarrow \; M_{G}|_{\cU} \subset ({\rm Ran} \, \cO_{\bo,E, \bT})^\perp,
\label{2.4a}
\end{equation}
and representation \eqref{march9u} follows. The representation is
orthogonal with respect to the inner product \eqref{17.1},
since $F_{\rm min}(z)u$ belongs to ${\rm Ran} \, \cO_{\bo,E, \bT}$ for any $u\in\cU$.
\end{proof}

A more detailed parametrization in \eqref{march9u} can be obtained by invoking any one of 
Beurling-Lax type theorems presented above describing the $\bS_{\bo,R}$-invariant subspace
$({\rm Ran} \, \cO_{\bo,E,\bT})^\perp$ of $H^2_{\bo,\cY}(\free)$.

\smallskip

We next consider a more structured interpolation problem in $H^2_{\bo, \cL(\cU, \cY)}(\free)$.
Given an $\bo$-isometric pair $(C,\bA)$ with $C\in\cL(\cX,\cY)$ and
$\bA=(A_1,\ldots,A_d)\in\cL(\cX)^d$, and given $D\in\cL(\cU,\cY)$, let
\begin{align}
& T_j=\begin{bmatrix}0 & 0\\ 0 & A_j\end{bmatrix} \in \cL(\cY \oplus
\cX) \text{ for } j=1,\ldots,d, \notag \\
& E=\begin{bmatrix}I_{\cY} & C\end{bmatrix} \in \cL(\cY \oplus \cX,
\cY), \quad
N=\begin{bmatrix}D^*& 0 \end{bmatrix} \in \cL(\cY \oplus \cX, \cU).
\label{march9c}
\end{align}
Then, for $F(z) = \sum_{\alpha \in\free} F_{\alpha} z^{\alpha} \in
H^2_{\bo, \cL(\cU, \cY)}(\free)$,
$$
(E^*F)^{\wedge L}({\bT}^*)=\begin{bmatrix} F_\emptyset \\ C^*F_\emptyset\end{bmatrix}+
\sum_{\alpha\in\free:|\alpha|>0}\begin{bmatrix}0 & 0\\ 0 & \bA^{*\alpha^\top}\end{bmatrix}
\begin{bmatrix}I_{\cY} \\ C^*\end{bmatrix} F_\alpha=
\begin{bmatrix} F_\emptyset \\ (C^*F)^{\wedge L}({\bA}^*)\end{bmatrix}
$$
and then the interpolation condition \eqref{2.4} amounts to
\begin{equation}
(C^*F)^{\wedge L}({\bA}^*)=0\quad\mbox{and}\quad F_\emptyset=D.
\label{march9e}
\end{equation}
Observe that by \eqref{march9c} and the power series representations \eqref{4.3} and \eqref{2.13pre},
\begin{equation}
\cG_{\bo,E,\bT}=\begin{bmatrix}I_{\cY} & C \\ C^* & \cG_{\bo,C,\bA}\end{bmatrix},\quad
ER_{\bo}(Z(z)T)=I_{\cY}+ CR_{\bo}(Z(z)A).
\label{march12}
\end{equation}
The assumption that the pair $(E, \bT)$ be exactly $\bo$-observable
means that $\cG_{\bo,E,\bT}$ be invertible in $\cL(\cX)$;  from the first formula in
\eqref{march12} and a Schur-complement calculation, we see that this
is equivalent to the condition that $\cG_{\bo,C,\bA}\succ C^{ *}C$ 
In particular, $\cG_{\bo,C,\bA}\succ 0$, i.e., $(C, \bA)$ is
exactly $\bo$-observable.  We next clarify when the converse direction holds.

\begin{proposition}   \label{P:GnETinv}
Let $(C, \bA)$ be an $\bo$-isometric output pair. Then $\cG_{\bo,C,\bA}\succ C^{*} C$ 
if and only if 
$(C, \bA)$ is exactly $\bo$-observable (so $\cG_{\bo,C,\bA}\succ 0)$
and ${\displaystyle\sum_{j=1}^{d} A_{j}^{*} A_{j}} \succ 0$.
\end{proposition}

\begin{proof} If $\cG_{\bo,C,\bA}\succ C^{*}C$, then $\Gr_{\bo,1,C,\bA} \succeq \cG_{\bo,C,\bA} \succ 0$,
by \eqref{grmonotone}.
From the identity \eqref{4.34} applied with $k=0$, we have
    $$
    \sum_{j=1}^{d} A_{j}^{*} \Gr_{\bo,1,C,\bA} A_{j} = \cG_{\bo,C,\bA} -
    C^{*}C \succ 0
    $$
and therefore,    
    $$
    \sum_{j=1}^{d} A_{j}^{*} A_{j} \succeq \frac{1}{\| \Gr_{\bo,1,C,\bA}
    \|} \sum_{j=1}^{d} A_{j}^{*} \Gr_{\bo,1,C,\bA} A_{j} \succ 0
    $$
which verifies the ``only if" part of the statement.

\smallskip

Conversely, if $\cG_{\bo,C,\bA}\succeq  \delta I_{\cX}$ for some $\delta > 0$ and 
$\sum_{j=1}^{d} A_{j}^{*} A_{j} \succ 0$, then  $\Gr_{\bo,1,C,\bA} \succeq \cG_{\bo,C,\bA}\succeq 
\delta I_{\cX}$, and hence,
  $$
\cG_{\bo,C,\bA} - C^{*}C = \sum_{j=1}^{d} A_{j}^{*} \Gr_{\bo,1,C, \bA}
A_{j} \succeq  \delta \cdot \sum_{j=1}^{d} A_{j}^{*} A_{j} \succ 0
$$
which completes the proof.
\end{proof}

We note that the condition $\sum_{j=1}^{d} A_{j}^{*} A_{j}\succ 0$ can be viewed as a version of the assumption that
none of the specified zero locations $\{z_{j} \colon j =
1,2,\dots\}$ for the shift-invariant subspace $\cM$ are at the origin
in the classical case mentioned above.

\smallskip

We now consider the following extremal problem:

\medskip

\noindent
\textbf{Extremal Interpolation Problem (EP):}
{\em Given an exactly $\bo$-observable pair $(C,\bA)$ with $\Gr_{\bo,1, C \bA}$ positive  definite,
find an injective operator $D\in\cL(\cU,\cY)$ acting
from an auxiliary Hilbert space $\cU$ into $\cY$ so that the minimal
norm solution $F_{\rm min}$ to the problem
\eqref{march9e} satisfies $[F_{\rm min},F_{\rm min}]=I_{\cU}$.}

\smallskip

The solution of the Extremal Interpolation Problem is as follows.

\begin{theorem}  \label{T:extint}  Suppose that $(C, \bA)$ is an
    exactly $\bo$-observable output pair. Then any solution $F_{\rm min}$ of the
   {\rm \textbf{EP}} is $H^2_{\bo,\cY}$-Bergman-inner.  Moreover,
   any such solution $F_{\rm min}$ is given by
\begin{equation}   \label{Fmin}
 F_{\rm min}(z) = D + C R_{\bo,1}(Z(z) A) Z(z) B
 \end{equation}
where
\begin{equation}   \label{Fminbd}
 B = - A (\cG_{\bo,C,A} - C^{*}C)^{-1} C^{*} D\quad\mbox{and}\quad
D = \big( I - C \cG_{\bo,C,\bA}^{-1}C^{*} \big)^{\frac{1}{2}} V,
\end{equation}
with $V$ be any isometry from an auxiliary Hilbert space $\cU$
onto $\overline{\operatorname{Ran}} (I - C \cG_{\bo,C,\bA}^{-1}
C^{*})^{\frac{1}{2}}$.
\end{theorem}

\begin{proof}
The solution set for the homogeneous problem \eqref{march9e} (i.e.,
with $D=0$) consists of $G \in H^2_{\bo, \cL(\cU, \cY)}$ with
$ \operatorname{Ran} M_{G}|_{\cU} \subset  (\operatorname{Ran} \cO_{\bo,E, \bT})^\perp$,
so by Theorem \ref{T:17.1}, the minimal norm solution
$F_{\rm min}$ of the non-homogeneous problem is such that
$M_{F}|_{\cU} \subset {\rm Ran} \, \cO_{\bo,E, \bT}\otimes \cU$.
In other words, for every $u\in\cU$, the series $F_{\rm min}(z)u$ belongs to
${\rm Ran} \, \cO_{\bo,E, \bT}$, and according to the second formula in \eqref{march12}, 
$F_{\rm min}(z)u$ is of the form
$$
F_{\rm min}(z)u=y+C(I-Z(z)A)^{-n}x
$$
for some $y\in\cY$ and $x\in\cX$. Since ${\rm Ran} \, \cO_{\bo,C, \bA}$ is $\bS_{\bo,R}^*$-invariant,
$\bS_{\bo,R}^{*\alpha^\top}F_{\rm min}u$ belongs to ${\rm Ran} \,\cO_{\bo,C, \bA}$
for any non-empty $\alpha\in\free$. On the other hand, due to the first (homogeneous) condition in \eqref{march9e},
$F_{\rm min}v$ belongs to $({\rm Ran} \, \cO_{\bo,C, \bA})^\perp$ for each $v\in\cU$. Therefore,
$\bS_{\bo,R}^{*\alpha^\top}F_{\rm min}u$ is orthogonal to $F_{\rm min}v$ or, equivalently, $F_{\rm min}u$
is orthogonal to $\bS_{\bo,R}^{\alpha}F_{\rm min}v$ for all $u,v\in\cU$ and non-empty $\alpha\in\free$.
Besides, $\|F_{\rm min}u\|_{H^2_{\bo,\cY}(\free)}=\|u\|_{\cU}$ for all $u\in\cU$, by \eqref{17.2}.
Then $F_{\rm min}$ is $H^2_{\bo, \cY}(\free)$-Bergman-inner, by  Definition \ref{D:defin}.

\smallskip

We now calculate the extremal series $F_{\rm min}$. By the first formula in \eqref{march12},
$$
\cG_{\bo,E,\bT}^{-1}=\begin{bmatrix}I_{\cY} & 0 \\ 0 & 0\end{bmatrix}+
\begin{bmatrix}-C \\ I_{\cX}\end{bmatrix}(\cG_{\bo,C,\bA}-C^*C)^{-1}\begin{bmatrix}-C^* &
    I_{\cX}\end{bmatrix}
$$
and substituting the latter equality into \eqref{march9a} gives
\begin{align*}
F_{\rm min}(z)&=\begin{bmatrix}I_{\cY} & C\end{bmatrix}\begin{bmatrix}I_{\cY} & 0 \\ 0 &
R_\bo(Z(z)A)\end{bmatrix}\cG_{\bo,E,\bT}^{-1}\begin{bmatrix}D\\ 0 \end{bmatrix}\notag\\
&=\begin{bmatrix}I_{\cY} & CR_\bo(Z(z)A)\end{bmatrix}\bigg(
\begin{bmatrix}D\\ 0 \end{bmatrix}+\begin{bmatrix}-C \\ I_{\cX}\end{bmatrix}(\cG_{\bo,C,\bA}-C^*C)^{-1}C^*D
\bigg)\notag\\
&=D+C(R_\bo(Z(z)A)-I_{\cX})(\cG_{\bo,C,\bA}-C^*C)^{-1}C^*D\\
&=D+CR_{\bo,1}(Z(z)A)(\cG_{\bo,C,\bA}-C^*C)^{-1}C^*D
\end{align*}
The latter formula can be written in the form \eqref{Fmin} with $B$ defined as in \eqref{Fminbd}
and $D$ to be yet determined. By formula \eqref{4.34} (for $k=0$),
\begin{equation}
A^*(\Gr_{\bo,1,C,\bA}\otimes I_d)A=\sum_{j=1}^d A_j^*\Gr_{\bo,1,C,\bA}A_j+C^*C=\cG_{\bo,C,\bA}.
\label{sab1}
\end{equation}
Then for $B$ of the form \eqref{Fminbd} we have 
\begin{align}
A^*(\Gr_{\bo,1,C,\bA}\otimes I_d)B&=-A^*(\Gr_{\bo,1,C,\bA}\otimes I_d)
A (\cG_{\bo,C,A} - C^{*}C)^{-1} C^{*} D\notag\\
&=- C^{*} D,\label{sab2}\\
B^*(\Gr_{\bo,1,C,\bA}\otimes I_d)B&=-D^*C(\cG_{\bo,C,\bA}-C^*C)^{-1}
A^*(\Gr_{\bo,1,C,\bA}\otimes I_d)B^*\notag\\
&=D^*C(\cG_{\bo,C,\bA}-C^*C)^{-1}C^*D.\label{sab3}
\end{align}
For $F_{\rm min}$ of the form \eqref{Fmin} with the entry $B$ given as in \eqref{Fminbd}, we have
\begin{align}
[F_{\rm min},F_{\rm min}]_{H^2_{\bo,\cL(\cU,\cY)}}=
&D^*D+\sum_{j=1}^d\sum_{\alpha\in\free}\omega_{|\alpha|+1}^{-1}
B_j^*\bA^{*\alpha^\top}C^*C\bA^{\alpha}B_j\notag\\
=&D^*D+B^*(\Gr_{\bo,1,C,\bA}\otimes I_d)B.\label{march10c}
\end{align}
Equalities \eqref{sab1}, \eqref{sab2}, \eqref{march10c} can be written in the matrix form 
$$
\begin{bmatrix} A^{*} & C^{*} \\ B^{*} & D^{*}\end{bmatrix}
\begin{bmatrix} \Gr_{\bo,1,C,\bA}\otimes I_d & 0 \\ 0 & I_{\cY}
\end{bmatrix}
\begin{bmatrix} A & B \\ C & D \end{bmatrix} =
\begin{bmatrix} \cG_{\bo,C,\bA} & 0 \\ 0 & [F_{\rm min},F_{\rm min}] \end{bmatrix}.
$$
Since $F_{\rm min}$ satisfies the condition $[F_{\rm min},F_{\rm min}]=I_{\cU}$, we conclude from the 
latter equality by Remark \ref{R:content} that $F_{\rm min}$ is $H^2_{\bo,\cY}(\free)$-Bergman-inner.
To find $D$ explicitly, we combine \eqref{sab3} and \eqref{march10c} and use the Sherman-Morrison formula as follows:
\begin{align*}
I_{\cY}=&D^*D+D^*C(\cG_{\bo,C,\bA}-C^*C)^{-1}C^*D\\
=&D^*\left(I+C(\cG_{\bo,C,\bA}-C^*C)^{-1}C^*\right)D
=D^*\big(I-C\cG_{\bo,C,\bA}^{-1}C^*\big)^{-1}D.
\end{align*}
We then conclude from the latter equality that $D$ is of the form as in \eqref{Fminbd}.
\end{proof}

\chapter{Model theory for $\bo$-hypercontractive operator $d$-tuples} 
\label{C:model}

Let us say that a commutative operator $d$-tuple $\bT = (T_1, \dots, T_d)$ is an {\em abstract $\bo$-shift}
if $\bT$ satisfies conditions \eqref{ol2} and $\cX_0 = \{0\}$ where $\cX_0$ is the space displayed in \eqref{ol4}.
Then, according to Theorem \ref{T:ol1}, any such $\bT$ is unitarily equivalent to the concrete $\bo$-shift
$\bS_{\bo, R}$ on $H^2_{\bo, \cY}(\free)$ for an appropriate coefficient Hilbert space $\cY$, i.e., $\bS_{\bo, R}$
on $H^2_\bo(\free)$ (with an appropriate multiplicity) serves as the functional model for any abstract $\bo$-shift
(up to unitary equivalence and some multiplicity).    In this chapter we will characterize abstract operator tuples
$\bT = (T_1, \dots, T_d)$ such that $\bA = \bT^* = (T_1^*, \dots, T_d^*)$ is unitarily equivalent to the
restriction of the adjoint shift-tuple $\bS_{\bo, R}^*$ on $H^2_{\bo, \cY}(\free)$ to an invariant subspace $\cN$,
where $\cN$ is isometrically included, or more generally only contractively included, in the ambient space
$H^2_{\bo, \cY}(\free)$.  We already have some results in this direction, namely Theorem \ref{T:beta-stablemodel}.
In the classical case one can go further by arriving at an explicitly computable characteristic function $\Theta_T$
which serves as a complete unitary invariant and from which one can define an explicit functional model for the operator $T$.
Our goal in this chapter is to enhance Theorems \ref{T:beta-stablemodel} by identifying 
a characteristic multiplier $\Theta_\bT$ which serves as a complete unitary invariant for the original operator tuple $\bT$,
just as in the Sz.-Nagy--Foias theory \cite{NF}.

\smallskip

In this chapter we  present two approaches to this problem for the class of $*$-$\bo$-hypercontractions $\bT$ which are
completely non-coisometric (c.n.u.)---corresponding to the case for $d=1$ where the contraction operator $T$ has no
coinvariant subspace $\cN$ such that $T^*|_\cN$ is isometric.  

\smallskip

The second approach, developed for the case $d=1$ in our earlier work \cite{BBIEOT, BBSzeged}, restricts the analysis
to the so-called pure case ($\bT^*$ is $\bo$-strongly stable) with the payoff that one gets a more explicit model theory
by using the Beurling-Lax representation via a Bergman-inner family $\{\Theta_\alpha \colon \alpha \in \free\}$
 for the $\bS_{\bo, R}$-invariant
subspace $\cM = \cN^\perp$, including an explicit formula for each member $\Theta_\alpha$ of the Bergman-inner
family in terms of the original operator $d$-tuple $\bT$.  This is the topic of Section \ref{S:model}.
We also show how to recover the original operator-tuple $\bT$ directly from the characteristic Bergman-inner family
and verify that the characteristic Bergman-inner family is a complete unitary invariant for $\bT$.

\section[Contractive multipliers as characteristic functions]{Model theory based on contractive-multiplier/McCT-inner 
multiplier as characteristic function}  \label{S:Popescu-model}

    Recall Definition \ref{D:7.3} of a {\em $\bo$-hypercontractive
    operator $d$-tuple $\bA = (A_{1}, \dots, A_{d})$}:  an operator
    $d$-tuple $\bA = (A_{1}, \dots, A_{d})$ on a Hilbert space $\cX$
    is {\em $\bo$-hypercontractive} if it is contractive (in the sense of \eqref{4.22})
and is subject to inequalities
$$
\Gamma^{(k)}_{\bo,\bA}[I_{\cX}]:=
-\sum_{\alpha\in\free}\bigg(\sum_{\ell=1}^{k}\frac{c_{|\alpha|+\ell}}{\omega_{k-\ell}}\bigg)
\, \bA^{*\alpha^\top}\bA^\alpha \succeq 0 \quad\mbox{for all}\quad k\ge 1,
$$ 
$$
\Gamma_{\bo,\bA}[I_\cX]=\sum_{\alpha\in\free}c_{|\alpha|}\bA^{*\alpha^\top}\bA^\alpha\succeq 0.
$$
For consistency with the Sz.-Nagy-Foias model theory, we consider an operator $d$-tuple 
$\bT = (T_{1}, \dots, T_{d})$ for which the adjoint tuple $\bA: = \bT^{*}: = (T_{1}^{*}, \dots, T_{d}^{*})$
is a $\bo$-hypercontraction.  We shall use the convention that non-bold $T$ denotes the operator
$$
T = \begin{bmatrix} T_{1} & \cdots & T_{d} \end{bmatrix} \colon\cX^{d} \to \cX
$$
    so that $A = T^{*}$ is the column operator
    $$
     A =  T^{*} = \begin{bmatrix} T_{1}^{*} \\ \vdots \\ T_{d}^{*} \end{bmatrix}
      \colon \cX \to \begin{bmatrix} \cX \\ \vdots \\ \cX \end{bmatrix}.
    $$

We consider an operator $d$-tuple $\bT = (T_1, \dots, T_d)$ $\cX$ such that $\bT^*$ is a $\bo$-c.n.c.~$\bo$-hypercontractive
operator tuple, i.e., we assume that $\bT$ is a {$*$-$\bo$-hypercontractive operator tuple.}  
Let us introduce the $\bo$-defect operator $D_ {\bo, \bT^*}$ defined by
\begin{equation}   \label{defect-boT*}
  D_{\bo, \bT^*} = \Gamma_{\bo, \bA}[I_\cX]^{\frac{1}{2}} \colon \cX \to \cD_{\bo, \bT^*}
\end{equation}
where we set 
$$
     \cD_{\bo, \bT^*} =\overline{ \operatorname{Ran}} \, D_{\bo, \bT^*}.
$$
Then by construction the pair $(D_{\bo, \bT^*}, \bT^*)$ is an $\bo$-isometric output pair.  We define the $\bo$-observability operator
$\cO_{\bo, D_{\bo, \bT^*}, \bT^*}$ as in \eqref{18.2aaa}.  We impose the condition that the output-pair $(D_{\bo, \bT^*}, \bT^*)$
be observable, i.e., that the associated observability operator $\cO_{\bo, D_{\bo, \bT^*}, \bT^*}$  have only trivial kernel.
In terms of the original $d$-tuple $\bT$, we say that $\bT$ is {\em $\bo$-completely non-coisometric} ({\em $\bo$-c.n.c.}\ for short).
As $\cO_{\bo, \cD_{\bo, \bT^*}}$ has trivial kernel, we can make $\cN: = \operatorname{Ran} \cO_{\bo, \cD_{\bo, \bT^*}}$ a Hilbert space
by assigning to it the lifted norm
$$
  \| \cO_{\bo, D_{\bo, \bT^*}, \bT^*} x \|_\cN = \| x \|_\cX.
$$
From Theorem \ref{T:2-1.2'} we see that $\cN = \cH(K_\cN)$ is in fact a 
NKRKHS with reproducing kernel $K_\cN$ given by 
$$
K_\cN(z, \zeta)  = D_{\bo, \bT^*} R_\bo(Z(z) \bT^*) R_\bo(Z(\zeta) \bT^*)^* D_{\bo, \bT^*},
$$
that $\cN$ is contained contractively in $H^2_{\bo, \cD_{\bo, \bT^*}}(\free)$, 
and moreover, $\cN$ is $\bS_{\bo, R}^*$-invariant with $\cO_{\bo, \cD_{\bT^*}, \bT^*}$ implementing a unitary equivalence between $\bT^*$ and 
$\bS_{\bo, R}^*|_\cN$. Let us make the following formal definition.

\begin{definition}  \label{D:char-func}
We shall say that the $\bo$-c.n.c.~$*$-$\bo$-hypercontractive tuple $\bT$ {\em admits a characteristic multiplier $\Theta_\bT$}
if the Brangesian complement $\cM = \cN^{[\perp]}$ of $\cN = \operatorname{Ran} \cO_{\bo, D_{\bT^*}, \bT^*}$ (equipped
with the lifted norm) admits a Beurling-Lax representation $\cM = \Theta \cdot H^2_{\cU}(\free)$ as in Theorem \ref{T:NC-BL}.
We then say that $\Theta$ is a {\em characteristic  multiplier} $\Theta_\bT$ for $\bT$.
\end{definition}

We then have the following characterization as to which $\bo$-c.n.c.~$*$-$\bo$-hypercon\-tractive tuples have a characteristic
function $\Theta_\bT$ in the sense of Definition \ref{D:char-func}.

\begin{theorem}  \label{T:char-func}  Let $\bT = (T_1, \dots, T_d)$ be a $\bo$-c.n.c.~$*$-$\bo$-hypercontractive tuple.
\begin{enumerate}
\item $\bT$ admits a characteristic multiplier if and only if 
\begin{align}
& I - \cO_{\bo, D_{\bo, \bT^*}, \bT^*} (\cO_{\bo, D_{\bo, \bT^*}, \bT^*})^*  \succeq   \notag \\
& \quad  
\sum_{j=1}^d S_{\bo, R, j}   ( I - \cO_{\bo, D_{\bo, \bT^*}, \bT^*} (\cO_{\bo, D_{\bo, \bT^*}, \bT^*})^*)  S_{\bo,R,j}^*.
\label{exist-char}
\end{align}
\item 
In particular, \eqref{exist-char} holds if it is the case that
\begin{align}
& I - (\cO_{\bo, D_{\bo, \bT^*}, \bT^*})^* \cO_{\bo, D_{\bo, \bT^*}, \bT^*}  \succeq \notag \\
& \quad \sum_{j=1}^d T_j (I - (\cO_{\bo, D_{\bo, \bT^*}, \bT^*})^* \cO_{\bo, D_{\bo, \bT^*}, \bT^*} ) T_j^*,
\label{suff-char}
\end{align}
in which case an explicit version of a characteristic multiplier $\Theta_\bT$ for $\bT$ can be constructed by implementing the algorithm
given in Theorem \ref{T:explicit2} with $\bA = \bT^*$, $C = D_{\bo, \bT^*}$.

\item In case $\cO_{\bo, D_{\bo, \bT^*}, \bT^*}$ is isometric (or equivalently, in case $\bT^*$ is $\bo$-strongly stable),
$\cN$ and $\cM = \cN^\perp$ are isometrically contained in $H^2_{\bo, \cD_{\bo, \bT^*}}(\free)$, 
$$
I-P_{\cM}= \cO_{\bo, D_{\bo, \bT^*}, \bT^*} \cO_{\bo, D_{\bo, D_{\bT^*}}, \bT^*}^* = P_\cN,
$$
and both conditions \eqref{exist-char} and \eqref{suff-char} are automatic.
\end{enumerate}
\end{theorem}

\begin{proof}  Statement (1) follows as an application of Theorem \ref{T:NC-BL} (see the last statement there) to the case where $\Pi =
I - \cO_{D_{\bo, \bT^*}, T^*} (\cO_{D_{\bo, \bT^*}, T^*} )^*$.  Statements (2) and (3) follows as an application
of Theorem \ref{T:explicit2} to the case $(C, \bA) = (D_{\bo, \bT^*}, \bT^*)$.
\end{proof}
 
 Not all contractive multipliers can be characteristic multipliers; there are contractive multipliers $\Theta$ form
 $H^2_\cU(\free)$ to $H^2_{\bo, \cY}(\free)$ so that the space 
$\cN = \cH^p(I - M_\Theta M_\Theta^*)$ is not $\bS_{\bo, R}^*$-invariant (see the end of \cite{BBCR}).  The following is a
 characterization (albeit not easily verifiable) of which contractive multipliers can arise as a characteristic multipliers
 for some $\bo$-c.n.c.~$*$-$\bo$-hypercontractive operator tuple $\bT$.
 
 \begin{theorem}  \label{T:Theta-char}
 Let $\Theta$ be a contractive multiplier form $H^2_\cU(\free)$ to $H^2_{\bo, \cY}(\free)$.  Then there is an
 $\bo$-c.n.c.~$*$-$\bo$-hypercontractive tuple $\bT = (T_1, \dots, T_d)$ and a unitary identification map
 $\iota \colon \cY \to \cD_{\bo, \bT^*}$ so that $\iota \cdot \Theta = \Theta_{\bT}$ if and only if the space
 $\cN = \cH^p(I - M_\Theta M_\Theta^*) \subset H^2_{\bo, \cY}(\free)$ satisfies the conditions of item (4) in
 Theorem \ref{T:2-1.2} in the stronger form where the second of the inequalities \eqref{ineqs} is required to hold with equality.
 Explicitly, we may take $\bT = ( S_{\bo, R}^*|_{\cH^p(I - M_\Theta M_\Theta^*)})^*$.

If $\Theta = \Theta_{\bT'}$ for an $\bo$-c.n.c.~$*$-$\bo$-hypercontractive tuple $\bT'$, then $\bT'$
is unitarily equivalent to the model $\bT$ as constructed in the previous paragraph.
\end{theorem}

\begin{proof}  If $\Theta = \Theta_\bT$, then by construction $\cN = \operatorname{Ran} \cO_{\bo, D_{\bT^*}, \bT^*}$
has the required form (with $\iota = I_{\cD_{\bo, \bT^*}}$).  Conversely, if $\cN$ satisfies the stronger version of the conditions
in item (4) of Theorem \ref{T:2-1.2}, then we may take $C = E|_\cN$, $\bT^* = \bS_{\bo, R}^*|_\cN$ and find a unitary 
identification map $\iota \colon \cY \to \cD_{\bo, \bT^*}$ so that $\iota C = D_{\bo, \bT^*}$.  Then
\begin{align*}
& D_{\bo, \bT^*} R_\bo(Z(z) T^*) R_\bo(Z(\zeta) T^*)^* D_{\bo, \bT^*} =
\iota C R_\bo(Z(z) T^*) R_\bo(Z(\zeta) T^*)^* D_{\bo, \bT^*} \iota^* \\
& = k_{\bo, {\rm nc}}(z, \zeta) I_{\cD_{\bo, \bT^*}} - \iota \Theta(z) (k_{\rm nc, Sz}(z, \zeta) I_\cU) \Theta(\zeta)^* \iota^*
\end{align*}
so $\iota \Theta(z)$ serves as a characteristic function for $\bT$.
\end{proof}

\begin{remark}  \label{R:zeroMFM}
In the classical setting, while any pure contractive analytic function arises as the characteristic function of a completely nonunitary
contraction operator $T$, those which arise as the characteristic function of a completely non-coisometric contraction operator
are characterized as those contractive analytic functions $\Theta(z) \colon \cU \to \cY$ for which $I - \Theta(\zeta)^* \Theta(\zeta)$ has
{\em zero maximal factorable minorant}, i.e., such that {\em the only $\cL(\cU,\mathcal F)$-valued analytic function $A(z)$ such that 
$A(\zeta)^* A(\zeta) \preceq I_\cU - \Theta(\zeta)^* \Theta(\zeta)$ for a.e. $\zeta \in {\mathbb T}$ is $A(z)\equiv 0$}. 
(see \cite[Theorem 6]{BallKriete}).  See \cite{Cuntz-scat} for a formulation of all this in the Fock space setting.  There does not
appear to be in results in this direction for the weighted Bergman space setting.
\end{remark}

\section[Bergman-inner family as characteristic multiplier-family]
{Model theory for $*$-$\bo$ strongly stable hypercontractions via characteristic Bergman-inner families}
\label{S:model}

We now specialize the discussion to the case where $\bT = (T_1, \dots, T_d)$ is such that $\bT^*$ is $\bo$-strongly stable as 
well as a $\bo$-hypercontractive operator tuple---we then say simply that {\em $\bT$ is a $*$-$\bo$ strongly stable, hypercontraction.}
The significance of the {\em $\bo$-strong stability} hypothesis is that then the observability operator $\cO_{\bo, D_{\bo, \bT^*}, \bT^*}$
is isometric (not just injective), so $\cN: = \operatorname{Ran} \cO_{\bo, D_{\bo, \bT^*}, \bT^*}$ with lifted norm is isometrically 
(not just contractively) contained in $H^2_{\bo, \cD_{\bo, \bT^*}}$ (see Theorem \ref{T:1.2g}).  Thus Theorem \ref{T:2-1.2'} specialized 
to this setup tells us that $\bT^*$ is unitarily equivalent
to $\bS^*_{\bo, R}|_\cN$ where $ \cN = \operatorname{Ran} \cO_{\bo, D_{\bT^*}, \bT^*} $
with lifted norm  is a $\bS_{\bo, R}^*$-invariant subspace isometrically contained in $H^2_{\bo, \cD_{\bo, \bT^*}}(\free)$.  
We define the {\em characteristic Bergman-inner family} $\{ \Theta_\beta\}_{\beta \in \free}$ to be the Bergman-inner Beurling-Lax
representer for the subspace $\cM$ having orthogonal complement equal to the range of the observability operator
$\cO_{\bo, D_{\bT^*}, \bT}^*$.  

\smallskip

In detail, define the shifted observability operator $\Ob_{\bo,k,D_{\bo, \bT^{*}}, \bT^{*}}$ and the shifted
    gramian $\Gr_{\bo,k,D_{\bo, \bT^{*}}, \bT^{*}}$ via formulas
    \eqref{4.31} and \eqref{4.32}; then Proposition
    \ref{P:wghtSteinid} gives us the validity of the weighted Stein
    identity
$$
        \sum_{j=1}^{d} T_{j} \Gr_{\bo,k+1,D_{\bo,\bT^{*}}, \bT^{*}}
        T_{j}^{*} + \omega_k^{-1}D_{\bo, \bT^{*}}^{*}D_{\bo, \bT^{*}} =
        \Gr_{\bo,k,D_{\bo,\bT^{*}}, \bT^{*}}.
$$
As explained in the proof of Lemma \ref{L:5.6}, it follows that the
operator
$$
X_{\bo,k}: = \begin{bmatrix} \Gr^{-1}_{\bo,k+1,D_{\bo,\bT^{*}}, \bT^{*}}
\otimes I_{d} & 0 \\ 0 & \omega_{k} I_{\cD_{\bo, \bT^{*}}} \end{bmatrix}
- \begin{bmatrix} T^{*} \\ D_{\bo, \bT^*} \end{bmatrix}
 \Gr^{-1}_{\bo,k,D_{\bo,\bT^{*}}, \bT^{*}} \begin{bmatrix} T & D_{\bo, \bT^*}
 \end{bmatrix}
$$
is positive semidefinite.
We let $\bD_{\bo,k, \bT}$ denote the positive semidefinite square root
of $X_{\bo,k}$, and let  $\bcD_{\bo,k, \bT}:=\overline{{\rm Ran}}\, \bD_{\bo,k, \bT}$. 
We then view $\bD_{\bo,k, \bT}$ as an operator from
$\bcD_{\bo,k, \bT}$ into$ \left[ \begin{smallmatrix} \cX \\ \cD_{\bo,
\bT^{*}} \end{smallmatrix} \right]$.  For any $\beta\in \free$,
we decompose $\bD_{\bo,|\beta|, \bT}$ as
\begin{equation}   \label{defectT}
  \bD_{\bo,|\beta|, \bT} = \begin{bmatrix} \widehat B_{\beta} \\ D_{\beta}
\end{bmatrix} \colon \bcD_{\bo, |\beta|, \bT} \to \begin{bmatrix} \cX \\
\cD_{\bo, \bT^{*}} \end{bmatrix}
\end{equation}
(so  $\left[ \begin{smallmatrix} \widehat B_{\beta} \\ D_{\beta}
\end{smallmatrix} \right]$ depends on $\beta$  only through $|\beta|$).
We finally introduce the family of connection matrices
$$
  \bU_{\beta} = \begin{bmatrix} T^{*} & \widehat B_{\beta} \\ D_{\bo, \bT^{*}}
  & D_{\beta} \end{bmatrix} \colon \begin{bmatrix} \cH \\ \bcD_{\bo, \beta,\bT} 
\end{bmatrix} \to \begin{bmatrix} \cH \\ \cD_{\bo, \bT^{*}}\end{bmatrix}
$$
and define the {\em characteristic Bergman-inner family} for $\bT$ by
$\Theta_{\bT} = \{ \Theta_{\bT, \beta} \}_{\beta \in \free}$ where
\begin{align}
\Theta_{\bT, \beta}(z)& =  \omega_{|\beta|}^{-1}D_{\beta} +
D_{\bo, \bT^{*}} R_{\bo, |\beta| + 1}(Z(z) T^{*}) Z(z) \widehat B_{\beta}
    \notag \\
  & \text{where } \begin{bmatrix} \widehat B_{\beta} \\ D_{\beta}
\end{bmatrix} =  \bD_{\bo,|\beta|, \bT} \text{ (see \eqref{defectT})}.
 \label{charfunc}
 \end{align}
The following statement summarizes the results from Chapter \ref{S:BLorthog} when
applied to the present model-theory context.

 \begin{theorem}  \label{T:charfuncT}
 Suppose that $\bT = (T_{1}, \dots, T_{d})$ is a $*$-$\bo$ strongly stable, hypercontractive operator-tuple
 on the Hilbert space $\cX$.
 Let $\Theta_{\bT} = \{ \Theta_{\bT, \beta}\}_{\beta \in \free}$ be the
 characteristic Bergman-inner family of $\bT$ as defined in \eqref{charfunc}.
 Then $\{\Theta_{\bT, \beta}\}$ is a $H^2_{\bo, \cY}(\free)$-Bergman-inner family and $\bT$ is 
 unitarily equivalent to the compressed shift operator-tuple $P_{\cM^{\perp}} \bS_{\bo,R}|_{\cM^{\perp}}$, where
 $$
 \cM = M_{\Theta_\bT} \cdot H^{2}_{\{\bcD_{\bo, |\beta|, \bT}\}_{\beta \in \free}}(\free)
 $$
 is the $\bS_{\bo,R}$-invariant subspace of $H^2_{\bo, \cD_{\bo, \bT^{*}}}(\free)$ 
associated with the $H^2_{\bo, \cY}(\free)$-Bergman-inner family $\{\Theta_{\bT, \beta}\}_{\beta \in \free}$.

 Furthermore, if $\bT'$ is another $*$-$\bo$ strongly stable, hypercontractive operator-tuple on a Hilbert space $\cX$ 
 having $\{\Theta_{\bT',\beta}\}_{\beta \in \free}$ as its  characteristic Bergman-inner family, 
then $\bT$ and $\bT'$ are unitarily equivalent if and
 only if the characteristic Bergman-inner families
 $\{\Theta_{\bT, \beta}\}_{\beta \in \free}$ and 
$\{ \Theta_{\bT',\beta}\}_{\beta \in \free}$ coincide in the following
 sense:  for each $k \in {\mathbb Z}_{+}$ there is a unitary operator
 ${\mathfrak v}_{k} \colon \bcD_{\bo, k, \bT } \to \bcD_{\bo,k, \bT' }$ and
 there is a unitary operator ${\mathfrak u} \colon \cD_{\bo, \bT^{*}} \to
 \cD_{\bo, \bT^{\prime *}}$ so that $\;
   {\mathfrak u} \Theta_{\bT, \beta}(z) = \Theta_{\bT', \beta}(z)
   {\mathfrak v}_{|\beta|}$.
 \end{theorem}

 \begin{proof} The characteristic Bergman-inner family $\Theta_{\bT} = \{\Theta_{\bT, \beta}\}_{\beta \in \free}$ for a $*$-$\bo$ strongly stable, hypercontractive operator tuple $\bT$ is by definition the $H^2_{\bo, \cY}(\free)$-Bergman-inner family 
constructed in Theorem \ref{T:BL3} for the case where the
$\bS_{\bo,R}$-invariant subspace $\cM$ is chosen with $\cM^{\perp} = \operatorname{Ran} \cO_{\bo, D_{\bT^{*}}, \bT^{*}}$. 
     
The fact that $\Theta_{\bT}$ is an $H^2_{\bo, \cY}(\free)$-Bergman-inner family is now a
consequence of Theorem \ref{T:BL3}.  The fact that $\bT$ is unitarily equivalent to the compressed shift operator-tuple
$P_{\cM^{\perp}} \bS_{\bo,R}|_{\cM^{\perp}}$ is a simple consequence of the intertwining relations 
$$ S^*_{\bo, R, j} \cO_{\bo, D_{\bo, \bT^*}, \bT^*}  = \cO_{\bo, D_{\bo, \bT^*}, \bT^*} T_j^*
$$
(see \eqref{4.8aga}) with the observability operator $\cO_{\bo, D_{\bo, \bT^*}, \bT^{*}}$
     implementing the unitary equivalence. 
     
\smallskip

 If $\bT$ and $\bT'$ are two $*$-$\bo$ strongly stable, hypercontractive operator tuples which are unitarily equivalent
 via a unitary operator $u \colon \cX \to \cX'$ (so $T'_j   = u T_j$ for $j=1, \dots, d$), then it is a matter of checking that the
 unitary operator $u$ can be passed through the definition of the defect operators $D_{\bo, \bT^*}$ \eqref{defect-boT*}
 and $\bD_{\bo, |\beta|, \bT}$ \eqref{defectT} to arrive at the coincidence of the respective characteristic
 Bergman-inner families.  The converse fact that coincidence of the respective Bergman-inner families     
 implies unitary equivalence of the respective model $\bo$-$*$ strongly stable, hypercontractive operator-tuples
 $\bT$ and $\bT'$ is straightforward.
  \end{proof}

Corollary \ref{C:innerfam-real} tells us that any $H^2_{\bo, \cY}(\free)$-Bergman-inner family 
$\Theta = \{ \Theta_{\beta}\}_{\beta \in \free}$ can
be realized so that $M_{\Theta} = {\rm row}_{\beta \in \free} [ M_{\Theta_{\beta}} ]
= \widehat T_{\bU}$ for a $\bo$-Bergman-conservative linear system $\bU$.
One way of constructing $\{\bU_{\beta} \}_{\beta \in \free}$ is via the
construction outlined in Theorem \ref{T:BL3} based on the
$\bS_{\bo,R}$-invariant subspace $\cM = M_{\Theta}H^{2}_{\{\cU_{\beta}\}}(\free)$.
A second (closely related) way is to view
$\{\Theta_{\beta}\}_{\beta \in \free}$ as the characteristic-function
family for the operator $d$-tuple $\bT = P_{\cM^{\perp}}
\bS_{\bo,R}|_{\cM^{\perp}}$ where $\cM$ is again taken equal to
$\cM = M_{\Theta} H^{2}_{\{\cU_{\beta}\}}(\free)$.
We now give a third, more direct way which expresses the operator
components $A, \widehat B_{\beta}, C, D_{\beta}$ more directly as
operators on function spaces without involving solving Cholesky
factorization problems.  Here we use the conventions that
\begin{equation}   \label{convention0}
\bS_{\bo,R} = (S_{\bo,R,1}, \dots, S_{\bo,R,d})
\end{equation}
is the shift operator $d$-tuple on $H^2_{\bo, \cY}(\free)$ for which
we have the functional calculus notation
$$
\bS_{\bo,R}^{\alpha^{\top}}  = S_{\bo,R,i_{1}} \cdots S_{\bo,R, i_{N}}\quad \text{if}\quad
\alpha = i_{N} \cdots i_{1} \in \free,
$$
while
\begin{equation}   \label{convention1}
S_{\bo,R} = \begin{bmatrix} S_{\bo,R,1} &
\cdots & S_{\bo,R,d} \end{bmatrix} \colon (H^2_{\bo, \cY}(\free))^{d} \to H^2_{\bo, \cY}(\free)
\end{equation}
and hence
\begin{equation}   \label{convention2}
S_{\bo,R}^{*} = \begin{bmatrix} S^*_{\bo,R,1} \\ \vdots \\ S^*_{\bo,R,d} 
\end{bmatrix} \colon H^2_{\bo, \cY}(\free) \to (H^2_{\bo, \cY}(\free))^{d}.
\end{equation}

\begin{theorem}  \label{T:model}
    Suppose that we are given a $H^2_{\bo, \cY}(\free)$-Bergman-inner family
    $\Theta = \{ \Theta_{\beta} \}_{\beta \in \free}$ as
    in Definition \ref{D:innerfuncfam} generating the
    $\bS_{\bo,R}$-invariant subspace
    $$
    \cM = M_{\Theta}  H^{2}_{\{\cU_{\beta}\}_{\beta \in \free}}(\free).
    $$
Then a Bergman conservative realization $\{\bU_{\beta}\}_{\beta \in\free}$ 
for $\Theta$ can be constructed as follows:
        take $\cX = \cM^{\perp} \subset H^2_{\bo, \cY}(\free)$ and set
   \begin{equation}   \label{canUv}
       \bU_{\beta} = \begin{bmatrix} S_{\bo,R}^{*}|_{\cM^{\perp}} &
       S_{\bo,R}^{*} \bS_{\bo,R}^{* \beta} \bS_{\bo,R}^{\beta^{\top}}
       M_{\Theta_{\beta}}|_{\cU_{\beta}}  \\
       E|_{\cM^{\perp}} & \omega_{|\beta|}
       [\Theta_{\beta}]_{\emptyset} \end{bmatrix} \colon \begin{bmatrix}
       \cM^{\perp} \\ \cU_{\beta} \end{bmatrix} \to \begin{bmatrix}
       (\cM^{\perp})^{d} \\ \cY \end{bmatrix}.
   \end{equation}
   In all these formulas the conventions \eqref{convention0},
   \eqref{convention1}, \eqref{convention2} apply.
    \end{theorem}

    \begin{proof}
Given the $H^2_{\bo, \cY}(\free)$-Bergman-inner family $\Theta = \{\Theta_{\beta}\}_{\beta \in \free}$, 
we set 
$$
\cM = \bigoplus_{\beta
\in \free} \bS_{\bo,R}^{\beta^{\top}} M_{\Theta_{\beta}} \cU_{\beta}.
$$
Then as a consequence of  analysis in the proof of Theorem
\ref{T:BL3}, we know that $\cM$ is $\bS_{\bo,R}$ invariant, and
due to property (3) in Definition \ref{D:innerfuncfam}, we have
\begin{align*}
&\bS_{\bo,R}^{\beta^{\top}} \cM \ominus
  \bigg(  \bigoplus_{j=1}^{d} \bS_{\bo,R}^{\beta^{\top}}S_{\bo,R,j} \cM\bigg)\\
&=  \bigg(\bigoplus_{\alpha \in \free} \Theta_{\alpha \beta}
    z^{\alpha \beta} \cU_{\alpha \beta} \bigg) \ominus
    \bigg( \bigoplus_{\alpha \in \free} \bigoplus_{j=1}^{d}
\Theta_{\alpha j\beta} z^{\alpha j \beta}\cU_{\alpha j \beta} \bigg)  \\
&= \bigg( \bigoplus_{\alpha \in \free} \Theta_{\alpha \beta}
z^{\alpha \beta}\cU_{\alpha \beta} \bigg) \ominus \bigg( \bigoplus_{\alpha \in \free
\colon \alpha \ne \emptyset}  \Theta_{\alpha \beta} z^{\alpha \beta}\cU_{\alpha \beta}\bigg)
= \Theta_{\beta} z^{\beta} \cU_{\beta}.
\end{align*}
Thus,
\begin{equation}  \label{Mbeta-form'}
 \bS_{\bo,R}^{\beta^{\top}} \cM \ominus \bigg(
 \bigoplus_{k=1}^{d}\bS_{\bo,R}^{\beta^{\top}} S_{\bo,R,k} \cM \bigg)
 = \bS_{\bo,R}^{\beta^{\top}} \Theta_{\beta} \cU_{\beta}.
\end{equation}
Let us set $A=S_{\bo,R}^{*}|_{\cM^{\perp}}$ and 
$\widehat B_{\beta}=S_{\bo,R}^{*} \bS_{\bo,R}^{* \beta} \bS_{\bo,R}^{\beta^{\top}} 
M_{\Theta_{\beta}}|_{\cU_{\beta}}$, i.e., 
$$
A=\begin{bmatrix}A_1 \\ \vdots \\ A_d\end{bmatrix}=
\begin{bmatrix}S_{\bo,R,1}^{*}|_{\cM^{\perp}} \\ \vdots \\ S_{\bo,R,d}^{*}|_{\cM^{\perp}}\end{bmatrix},\quad
\widehat B_{\beta}=\begin{bmatrix}B_{1,\beta} \\ \vdots \\ B_{d,\beta}\end{bmatrix}
=\begin{bmatrix}S_{\bo,R,1}^{*} \bS_{\bo,R}^{* \beta} \bS_{\bo,R}^{\beta^{\top}} M_{\Theta_{\beta}}|_{\cU_{\beta}}
 \\ \vdots \\ 
S_{\bo,R,d}^{*} \bS_{\bo,R}^{* \beta} \bS_{\bo,R}^{\beta^{\top}} M_{\Theta_{\beta}}|_{\cU_{\beta}}
\end{bmatrix},
$$
and let us write the connection matrix \eqref{canUv} as 
\begin{equation}  \label{Uv}
    \bU_{\beta} = \begin{bmatrix} A & \widehat B_{\beta} \\ C &
    D_{\beta} \end{bmatrix},\quad\mbox{where}\quad C=E|_{\cM^{\perp}},\quad
\quad D_\beta=\omega_{|\beta|}^{-1}[\Theta_{\beta}]_{\emptyset}.
\end{equation}
\smallskip
\noindent
\textbf{Step 1:} {\em The connection matrix $\bU_{\beta} = \sbm{A & \widehat B_{\beta} \\ C &
    D_{\beta}}$ maps $\sbm{\cM^{\perp} \\ \cU_{\beta}}$ into $\sbm{(\cM^{\perp})^{d}\\ \cY}$.}

\smallskip
\noindent
{\em Proof:} Since $\cM$ is ${\bf S}_{\bo,R}$-invariant, its orthogonal complement $\cM^\perp$ is 
${\bf S}_{\bo,R}^*$-invariant and hence $A_j: \cM^\perp\to\cM^\perp$ for $j=1,\ldots,d$. As 
$C: \, \cM^\perp\to\cY$ and $D_\beta: \, \cU_\beta\to\cY$ by construction, it remains to show that
$B_{j,\beta}$ maps $\cU_{\beta}$  into $\cM^{\perp}$ for $j=1,\ldots,d$.

\smallskip

To this end, we apply formula \eqref{again43} to the power series
representation for $\Theta_{\beta}$:
$$
  \Theta_{\beta}(z) = \sum_{\alpha \in \free} [\Theta_{\beta}]_{\alpha} z^{\alpha}
$$
to conclude that for any $u\in\cU_\beta$,
\begin{equation}   \label{SvSvT}
\left(\bS_{\bo,R}^{*\beta} \bS_{\bo,R}^{\beta^{\top}} M_{\Theta} u
\right)(z) = \sum_{\alpha \in \free} \frac{ \omega_{|\alpha| +
|\beta|}}{\omega_{|\alpha|}} [\Theta_{\beta}]_{\alpha}u z^{\alpha},
\end{equation}
from which we get
\begin{equation}   \label{Bvk}
\left(B_{j,\beta}u\right)(z) = \left(S_{\bo,R,j}^{*} \bS_{\bo,R}^{* \beta}
\bS_{\bo,R}^{\beta^{\top}}M_{\Theta_{\beta}}u  \right)(z) =
\sum_{\alpha \in \free} \frac{ \omega_{|\alpha| + |\beta| +1}}{ \omega_{|\alpha| }}
[\Theta_{\beta}]_{\alpha j}  u z^{\alpha}.
\end{equation}
For any $f\in\cM$, we then have 
\begin{align*}
    \langle f, \, B_{j,\beta} u \rangle_{H^2_{\bo, \cY}(\free)} &
    = \langle f,  \, S_{\bo,R,j}^{*} \bS_{\bo,R}^{* \beta} \bS_{\bo,R}^{\beta^{\top}}
    \Theta_{\beta} \cdot u \rangle_{H^2_{\bo, \cY}(\free)}  \\
    & = \langle \bS_{\bo,R}^{\beta^{\top}j}f, \, \bS_{\bo,R}^{\beta^{\top}}
    \Theta_{\beta}\cdot u \rangle_{H^2_{\bo,\cY}(\free)} = 0,
\end{align*}
by \eqref{Mbeta-form'}. Thus, for any $u\in\cU_\beta$, the vector $B_{j,\beta} u$ is orthogonal to any 
$f\in\cM$ and hence, $B_{j,\beta} \colon \cU_{\beta} \to \cM^{\perp}$ as required.

\medskip
\noindent
\textbf{Step 2:} {\em $\bU_{\beta}$ satisfies the weighted isometry
property \eqref{isom}.}

\smallskip
\noindent
{\em Proof:} To verify the property \eqref{isom}, it suffices to
verify the three pieces:
\begin{align}
  &  A^{*} \left( \Gr_{\bo,|\beta|+1, C, \bA} \otimes I_{d} \right) A
  +  \omega_{|\beta|}^{-1} C^{*} C =
    \Gr_{\bo,|\beta|, C, \bA}, \label{i} \\
& A^{*} \left( \Gr_{\bo, |\beta| + 1, C, \bA} \otimes I_{d} \right) \widehat
B_{\beta} + \omega_{|\beta|}^{-1} C^{*} D_{\beta} = 0, \label{ii} \\
& \widehat B_{\beta}^{*} \left( \Gr_{\bo, |\beta| + 1, C, \bA} \otimes I_{d}
\right) \widehat B_{\beta} + \omega_{|\beta|}^{-1} D_{\beta}^{*} D_{\beta} =
I_{\cU_{\beta}}.  \label{iii}
\end{align}
We first note that the identity \eqref{i} is a consequence of the general identity
\eqref{4.34}.  Note next that \eqref{ii} is equivalent to the validity of
\begin{equation}   \label{ii'}
    \bigg\langle \begin{bmatrix} \left(\Gr_{\bo,|\beta| + 1, C,\bA} \otimes
    I_{d}\right) A \\  C \end{bmatrix} f, \, \begin{bmatrix} \widehat
    B_{\beta} \\ \omega_{|\beta|}^{-1} D_{\beta} \end{bmatrix} u \bigg
    \rangle_{H^2_{\bo,\cY}(\free) \oplus \cY} = 0
\end{equation}
for all $f \in \cM^{\perp}$ and $u \in \cU_{\beta}$.
Due to the facts that $\bA = \bS_{\bo,R}^{*}|_{\cM^{\perp}}$ and $\cM^{\perp}$ is
$\bS_{\bo,R}^{*}$-invariant,  we can rewrite the left-hand side of \eqref{ii'} as
\begin{align}
     &\big  \langle \big( \Gr_{\bo, |\beta|+1, E, \bS_{\bo,R}^{*}} \otimes I_{d} \big)
    S_{\bo,R}^{*} f, S_{\bo,R}^{*} \bS_{\bo,R}^{* \beta} \bS_{\bo,R}^{\beta^{\top}}
    M_{\Theta_{\beta}} u \big\rangle_{H^2_{\bo,\cY}(\free)} + \langle Ef,
    [\Theta_{\beta}]_{\emptyset} u \rangle_{\cY}  \notag \\
    & = \langle S_{\bo,R} \big(\Gr_{\bo, |\beta| + 1, E, \bS_{\bo,R}^{*}} \otimes I_{d}
    \big) S_{\bo,R}^{*} f, \, \bS_{\bo,R}^{*\beta} \bS_{\bo,R}^{\beta^{\top}}
    M_{\Theta_{\beta}} u  \rangle_{H^2_{\bo,\cY}(\free)} +
\langle  E f,  [\Theta_{\beta}]_{\emptyset} u \rangle_{\cY}  \notag \\
& = \langle ( \Gr_{\bo, |\beta|, E, \bS_{\bo,R}^{*}} - \omega_{|\beta|}^{-1}
  E^{*} E) f, \, \bS_{\bo,R}^{*\beta} \bS_{\bo,R}^{\beta^{\top}}
    M_{\Theta_{\beta}} u  \rangle_{H^2_{\bo,\cY}(\free)} +
\langle  E f,  \, [\Theta_{\beta}]_{\emptyset} u \rangle_{\cY}  \notag \\
& = \langle \Gr_{\bo, |\beta|, E, \bS_{\bo,R}^{*}} f, \bS_{\bo,R}^{*\beta}
\bS_{\bo,R}^{\beta^{\top}}
    M_{\Theta_{\beta}} u  \rangle_{H^2_{\bo,\cY}(\free)}  \notag \\
    & \quad  \quad -
    \langle \omega_{|\beta|}^{-1} Ef, E \bS_{\bo,R}^{*\beta} \bS_{\bo,R}^{\beta^{\top}}
    M_{\Theta_{\beta}} u  \rangle_{H^2_{\bo,\cY}(\free)}
     + \langle  E f,  \, [\Theta_{\beta}]_{\emptyset} u \rangle_{\cY}
    \label{ii''}
\end{align}
where we made use of the general identity \eqref{4.34} applied with
$\bA = \bS_{\bo,R}^{*}$ and $C = E$:
$$
  S_{\bo,R} \left( \Gr_{\bo, |\beta| + 1, E, \bS_{\bo,R}} \otimes I_{d} \right)
  S_{\bo,R}^{*} = \Gr_{\bo, |\beta|, E, \bS_{\bo,R}^{*}} - \omega_{|\beta|}^{-1}
  E^{*} E.
$$
From the formula \eqref{SvSvT} we read off that
\begin{equation}  \label{Ubeta21}
E \bS_{\bo,R}^{*\beta} \bS_{\bo,R}^{\beta^{\top}}
M_{\Theta_{\beta}} u  = \omega_{|\beta|}  [\Theta_{\beta}]_{\emptyset}
\end{equation}
and hence the last two terms in \eqref{ii''} cancel.
Thus verification of the identity \eqref{ii} collapses to
verification that the first term in the last expression in \eqref{ii''} is zero, i.e., that
\begin{equation}
\big\langle \Gr_{\bo, |\beta|, E, \bS_{\bo,R}^{*}} f, \, \bS_{\bo,R}^{*\beta}
    \bS_{\bo,R}^{\beta^{\top}}M_{\Theta_{\beta}} u  \big\rangle_{H^2_{\bo,\cY}(\free)}=0.
\label{ii'''}
\end{equation}
Making use of the second equality in \eqref{gramss}, we can write \eqref{ii'''} as 
\begin{align*}
&\langle \Gr_{\bo, |\beta|, E, \bS_{\bo,R}^{*}} f, \, \bS_{\bo,R}^{*\beta}
    \bS_{\bo,R}^{\beta^{\top}}M_{\Theta_{\beta}} u  \rangle_{H^2_{\bo,\cY}(\free)}\\
&=\langle \bS_{\bo,R}^{*\beta}
    \bS_{\bo,R}^{\beta^{\top}}\Gr_{\bo, |\beta|, E, \bS_{\bo,R}^{*}} f, \, M_{\Theta_{\beta}} u  \rangle_{H^2_{\bo,\cY}(\free)}
=\langle f, \, M_{\Theta_{\beta}} u  \rangle_{H^2_{\bo,\cY}(\free)}.
\end{align*}
As $f \in \cM^\perp$ and $M_{\Theta_\beta} \colon \cU_\beta \to \cM_\beta$, we now can conclude that
\eqref{ii'''} holds, and hence the verification of \eqref{ii} is complete.

\smallskip

As for  \eqref{iii}, note that  an equivalent condition is
\begin{equation}  \label{iii'}
   \big\langle \big( \Gr_{\bo, |\beta|+1, E, \bS_{\bo,R}^{*}} \otimes I_{d}
    \big) \widehat B_{\beta} u, \widehat B_{\beta} u
   \big\rangle_{H^2_{\bo,\cY}(\free)}  + \omega_{|\beta|}^{-1} \cdot \|
   D_{\beta} u\|_{\cY}^{2} = \| u \|_{\cU}^{2}
\end{equation}
for all $u \in \cU_{\beta}$.  From \eqref{Bvk} and \eqref{grer} we see
that the first term on the left side is
\begin{align*}
& \left \langle \left( \Gr_{\bo, |\beta|+1, E, \bS_{\bo,R}^{*}} \otimes I_{d}
 \right) \widehat B_{\beta} u, \widehat B_{\beta} u
   \right \rangle_{H^2_{\bo,\cY}(\free)}   \\
 &  = \sum_{j=1}^{d} \bigg\langle \sum_{\alpha \in \free}
 \frac{ \omega_{|\alpha|}}{\omega_{|\alpha| + |\beta| +1}} \cdot
 \frac{\omega_{|\alpha| + |\beta| + 1}}{\omega_{|\alpha|}} [
 \Theta_{\beta}]_{\alpha j} u \, z^{\alpha}, \,
 \sum_{\alpha \in \free}
 \frac{\omega_{|\alpha| + |\beta| + 1}}{\omega_{|\alpha|}}
 [ \Theta_{\beta}]_{\alpha j} u \, z^{\alpha} \bigg\rangle_{H^2_{\bo, \cY}(\free)} \\
 & = \sum_{j=1}^{d} \sum_{\alpha \in \free} \omega_{|\alpha| + |\beta| +1} 
\| [\Theta_{\beta}]_{\alpha j} u \|^{2}_{\cY}\\
&  = \sum_{\alpha \in \free \colon \alpha \ne \emptyset}
 \omega_{|\alpha| + |\beta|} \| [\Theta_{\beta}]_{\alpha} u \|^{2}_{\cY}=
\| {\bf S}_{\bo,R}^{\beta^\top}\Theta_{\beta}u \|^{2}_{H^2_{\bo, \cY}(\free)}- 
\omega_{|\beta|} \| [\Theta_{\beta}]_{\emptyset} u \|_{\cY}^{2}.
 \end{align*}
Combining the latter computation with equalities
$$
\| {\bf S}_{\bo,R}^{\beta^\top}\Theta_{\beta}u \|^{2}_{H^2_{\bo, \cY}(\free)}=\|u\|^2_{\cU_\beta}\quad
\mbox{and}\quad \omega_{|\beta|} \| [\Theta_{\beta}]_{\emptyset} u \|_{\cY}^{2}
=\omega_{|\beta|}^{-1} \| D_{\beta} u \|^{2}_{\cY},
$$
that hold by property (1) of an $H^2_{\bo, \cY}(\free)$-Bergman-inner family (Definition \ref{D:innerfuncfam}) and 
the definition of $D_\beta$ in \eqref{Uv}, respectively, we arrive at \eqref{iii'} as wanted. 

\medskip
\noindent
\textbf{Step 3:} {\em $\bU_{\beta}$ satisfies the weighted coisometry
property \eqref{wghtcoisom}.}

\smallskip
\noindent
{\em Proof:} To verify the weighted coisometry property \eqref{wghtcoisom}, we
note that, in view of the validity of the weighted isometry property \eqref{isom}
already checked, it suffices to show that the connection matrix
$\bU_{\beta}$ \eqref{Uv} maps $\cM^{\perp} \oplus \cU_{\beta}$ onto
$\cM^{\perp} \oplus \cY$.   To show that $\bU_{\beta}$ is onto, we
suppose that 
a vector $\sbm{ g \\ y } \in \sbm{ (\cM^{\perp})^{d} \\ \cY }$ is
orthogonal to $\bU_{\beta} \sbm{ \cM^{\perp} \\ \cU_{\beta} }$ in the
$\sbm{ \Gr_{\bo, |\beta| + 1, C, \bA} \otimes I_{d} & 0 \\ 0 &
I_{\cY}}$-metric, i.e., we suppose that vectors $g\in(\cM^{\perp})^{d}$ and $y\in\cY$ are such that  
\begin{equation}   \label{know0}
\left\langle \sbm{\Gr_{\bo, |\beta| + 1, C, \bA} \otimes I_{d} & 0 \\ 0 &
I_{\cY}}\bU_{\beta}\begin{bmatrix}f \\ u\end{bmatrix}, \, \begin{bmatrix}g \\ y\end{bmatrix}
\right\rangle_{(\cM^{\perp})^{d} \oplus \cY} =0
\end{equation}
for all $f\in\cM^\perp, \; u\in\cU_\beta$, and we wish to show that this forces $g=0$ and $y=0$.
From the identity \eqref{Ubeta21} we see that the connection
matrix  $\bU_{\beta}$ \eqref{canUv} has the factorization
$$
 \bU_{\beta} = \begin{bmatrix} S_{\bo,R}^{*} \\ E \end{bmatrix}
 \begin{bmatrix} I_{\cM^{\perp}} & \bS_{\bo,R}^{* \beta}
     \bS_{\bo,R}^{\beta^{\top}} M_{\Theta_{\beta}} \end{bmatrix} \colon
     \begin{bmatrix} \cM^{\perp} \\ \cU_{\beta} \end{bmatrix} \to
         \begin{bmatrix} (\cM^{\perp})^{d} \\ \cY \end{bmatrix}.
$$
Substituting this factorization into \eqref{know0} we get
$$
\left\langle \begin{bmatrix}\big(\Gr_{\bo, |\beta| + 1, E,
\bS_{\bo,R}^{*}}\otimes I_{d} \big) S_{\bo,R}^{*} \\ E \end{bmatrix}( f + \bS_{\bo,R}^{*
\beta} \bS_{\bo,R}^{\beta^{\top}} M_{\Theta_{\beta}} u), \,
\begin{bmatrix} g \\ y \end{bmatrix}
    \right\rangle_{(\cM^{\perp})^{d} \oplus \cY} = 0.
$$
Breaking out $g \in (\cM^{\perp})^{d}$ into its components $g_{j} \in
\cM^{\perp}$:  $g = \sbm{ g_{1} \\ \vdots \\ g_{d} }$, rewrite the last equality
in the form
\begin{equation}   \label{know1a}
\big \langle f + \bS_{\bo,R}^{* \beta} \bS_{\bo,R}^{\beta^{\top}}
 M_{\Theta_{\beta}} u, \, \sum_{j=1}^{d} S_{\bo,R,j} \Gr_{\bo, |\beta| +
 1, E, \bS_{\bo,R}^{*}} g_{j} + E^{*} y \big\rangle_{H^2_{\bo, \cY}(\free)} =
 0
 \end{equation}
 for all $f \in \cM^{\perp}$ and $u \in \cU_{\beta}$.  For ease of
 notation let us set
 \begin{equation}  \label{hgy}
   h_{g,y}: =  \sum_{j=1}^{d} S_{\bo,R,j} \Gr_{\bo, |\beta| +
 1, E, \bS_{\bo,R}^{*}} g_{j} + E^{*} y.
 \end{equation}
Setting $u = 0\in\cU_{\beta}$ and varying $f$
 arbitrarily in $\cM^{\perp}$ we conclude from \eqref{know1a} that
 \begin{equation}  \label{conclude1}
     h_{g,y} \in \cM.
 \end{equation}
If we set $f =0$ and vary $u$ arbitrarily in $\cU_{\beta}$ in
\eqref{know1a}, we get
 $$
     \langle \bS_{\bo,R}^{\beta^{\top}} M_{\Theta_{\beta}} u,
     \,  \bS_{\bo,R}^{\beta^{\top}} h_{g,y}
     \rangle_{H^2_{\bo,\cY}(\free)} = 0\quad\mbox{for all}\quad u \in
     \cU_{\beta}.
$$
When we combine this with \eqref{conclude1} we arrive at
 \begin{equation} \label{know2}
     \bS_{\bo,R}^{\beta^{\top}} h_{g,y} \in \bS_{\bo,R}^{\beta^{\top}} \cM
     \cap (\bS_{\bo,R}^{\beta^{\top}} M_{\Theta_{\beta}}
     \cU_{\beta})^{\perp}.
     \end{equation}
As a consequence of the identity \eqref{Mbeta-form'}, we know that
$$
\bS_{\bo,R}^{\beta^{\top}} \cM
     \cap (\bS_{\bo,R}^{\beta^{\top}} M_{\Theta_{\beta}}
     \cU_{\beta})^{\perp} =  \bigoplus_{j=1}^{d}
     \bS_{\bo,R}^{\beta^{\top}} S_{\bo,R,j} \cM
$$
and hence condition \eqref{know2} gives us
\begin{equation}   \label{know2a}
    \bS_{\bo,R}^{\beta^{\top}} h_{g,y} \in \bigoplus_{j=1}^{d}
     \bS_{\bo,R}^{\beta^{\top}} S_{\bo,R,j} \cM.
\end{equation}
Applying the operator $\bS_{\bo,R}^{\beta^{\top}}$ to both parts of 
\eqref{hgy} and taking into account that 
$\Gr_{\bo,k,E,\bS_{\bo,R}^{*}} = \Ob_{\bo,k,E, \bS_{\bo,R}^{*}}$ (by 
explicit formulas \eqref{grer}), we get
\begin{equation}   \label{Sbetahgy}
  \bS_{\bo,R}^{\beta^{\top}} h_{g,y} = \sum_{j=1}^{d}
  \bS_{\bo,R}^{\beta^{\top}} S_{\bo,R,j} \Ob_{\bo, |\beta| +1, E,
  \bS_{\bo,R}^{*}} g_{j} + \bS_{\bo,R}^{\beta^{\top}} E^{*} y.
\end{equation}
From \eqref{SMperp} we know that
\begin{equation}  \label{know2'}
\big( \bS_{\bo,R}^{(j\beta)^{\top}} \cM \big)^{\perp} =
\big( \bS_{\bo,R}^{(j\beta)^{\top}} H^2_{\bo,\cY}(\free)
\big)^{\perp} \bigoplus \bS_{\bo,R}^{(j\beta)^{\top}} {\rm
Ran}\, \Ob_{\bo,|\beta|+1, E, \bS_{\bo,R}^{*}}.
\end{equation}
Let us note that the $j$-th term in the sum on the right hand side of
\eqref{Sbetahgy} satisfies
$$
    \bS^{\beta^{\top}}_{\bo,R} S_{\bo,R,j} \Ob_{\bo, |\beta|+1, E,
    \bS^{*}_{\bo,R}} g_{j} \in \bS_{\bo, R}^{(j\beta)^{\top}}
    \operatorname{Ran} \Ob_{\bo, |\beta| + 1, E, \bS_{\bo,R}^{*}}
$$
and hence from \eqref{know2'} we conclude that
\begin{equation} \label{note1}
     \bS^{\beta^{\top}}_{\bo,R} S_{\bo,R,j} \Ob_{\bo, |\beta|+1, E,
    \bS^{*}_{\bo,R}} g_{j} \in \big( \bS_{\bo,R}^{(j \cdot \beta)^{\top}} \cM \big)^{\perp}\quad
\text{for each}\quad j = 1, \dots, d.
\end{equation}
But for $j \ne k$, $\operatorname{Ran} \bS_{\bo,R}^{(j\beta)^{\top} } \perp \operatorname{Ran} 
\bS_{\bo,R}^{(k \beta)^{\top}}$ and hence \eqref{note1} actually implies
\begin{equation}   \label{note1'}
     \bS^{\beta^{\top}}_{\bo,R} S_{\bo,R,j} \Ob_{\bo, |\beta|+1, E,
    \bS^{*}_{\bo,R}} g_{j} \in \bigg( \bigoplus_{k=1}^{d} \bS_{\bo,R}^{(k\beta)^{\top}}
    \cM \bigg)^{\perp}\; \text{for each}\;  j = 1, \dots, d.
\end{equation}
Note next that $\bS_{\bo,R}^{\beta^{\top}} E^{*}y$,
the last term on the right hand side of \eqref{Sbetahgy},
satisfies the orthogonality property
$$
\langle \bS_{\bo,R}^{\beta^{\top}} E^{*} y, \, \bS_{\bo,R}^{(j\beta)^{\top}} 
f \rangle_{H^2_{\bo, \cY}(\free)} =
\langle \bS_{\bo,R}^{\beta^{\top}} E^{*} y, \, \bS_{\bo,R}^{\beta^{\top}}
S_{\bo,R,j} f \rangle_{H^2_{\bo, \cY}(\free)} = 0
$$
for all $f\in H^2_{\bo,\cY}(\free)$ and for all $ j = 1, \dots, d$; in other words,
$$ 
\bS_{\bo,R}^{\beta^{\top}} E^{*} y \in \bigg( \bigoplus_{j=1}^{d}
\bS_{\bo,R}^{(j\beta)^{\top}} H^2_{\bo,\cY}(\free)\bigg)^{\perp}.
$$
In particular, we have
\begin{equation} \label{note1''}
    \bS_{\bo,R}^{\beta^{\top}} E^{*} y \in \bigg( \bigoplus_{j=1}^{d}
\bS_{\bo,R}^{(j\beta)^{\top}} \cM \bigg)^{\perp}.
\end{equation}
Adding up \eqref{note1'} over all $j=1, \dots, d$
together with \eqref{note1''} and recalling \eqref{Sbetahgy} and
\eqref{know2a}, we are left with
$$
  \bS_{\bo,R}^{\beta^{\top}} h_{g,y} \in \bigg(\bigoplus_{j=1}^{d}
  \bS_{\bo,R}^{\beta^{\top}} S_{\bo,R,j} \cM \bigg) \bigcap
  \bigg(\bigoplus_{j=1}^{d}
  \bS_{\bo,R}^{\beta^{\top}} S_{\bo,R,j} \cM \bigg)^{\perp},
$$
and hence $\bS_{\bo,R}^{\beta^{\top}}h_{g,y} = 0$.  As $\bS_{\bo,R}^{\beta^{\top}}$ is injective,
it then follows that $h_{g,y}= 0$. Note that the sum in the definition \eqref{hgy} is orthogonal, and hence
$y= 0$ and
$$
 S_{\bo,R,j} \Gr_{\bo, |\beta| +1, E, \bS_{\bo,R}^{*}} g_{j} = 0 \quad\text{for each}\quad j = 1, \dots, d.
$$
As $S_{\bo,R,j} \Gr_{\bo, |\beta| +1, E, \bS_{\bo,R}^{*}}$ is injective, we
conclude that each $g_{j} = 0$ as well, and the coisometric property
of $\bU_{\beta}$ now follows.

\smallskip

\noindent
\textbf{Step 4:  Verification of the transfer-function realization
formula \eqref{thetahardy} for $\Theta_{\beta}$.}
It remains now only to check that we recover $\Theta_{\beta}$ via the
realization formula \eqref{thetahardy} (or equivalently the
representation \eqref{jul15}) explicitly in terms of power-series
coefficients:
$$
\Theta_{\beta}(z) =  \omega_{|\beta|}^{-1} D_{\beta} +
\sum_{j=1}^{d} \sum_{\beta' \in \free} \omega^{-1}_{|\beta'| + |\beta| +1} C
\bA^{\beta'} B_{j,\beta} z^{\beta' j }.
$$
We compute, using \eqref{Bvk} and \eqref{esbo},
\begin{align*}
    C \bA^{\beta'}B_{j,\beta} & = E \bS_{\bo,R}^{* \beta'} \bigg( \sum_{\alpha \in
    \free} \frac{\omega_{|\alpha| + |\beta| +1}}{\omega_{|\alpha|}} [
    \Theta_{\beta}]_{\alpha j} z^{\alpha} \bigg) \\
&=\omega_{|\beta'|}\cdot \left[\sum_{\alpha \in
    \free} \frac{\omega_{|\alpha| + |\beta| +1}}{\omega_{|\alpha|}} [
    \Theta_{\beta}]_{\alpha j} z^{\alpha} \right]_{\beta'}\\
&=\omega_{|\beta'|}\cdot \frac{\omega_{|\beta'| + |\beta| +1}}{\omega_{|\beta'|}}\cdot
\left[\Theta_{\beta}\right]_{[\beta' j}= \omega_{|\beta'| + |\beta| +1}  \left[\Theta_{\beta}\right]_{\beta' j}.
\end{align*}
We combine this with the identity $D_{\beta} = \omega_{|\beta|}[\Theta_\beta]_{\emptyset}$ to get
\begin{align*}
   & \omega_{|\beta|}^{-1} D_{\beta} + \sum_{j=0}^{d} \sum_{\beta' \in \free}
    \omega_{|\beta'| + |\beta| +1}^{-1} C \bA^{\beta'} B_{j,\beta}
    z^{\beta' k} \\
    &  \quad = [\Theta_\beta]_{\emptyset} + \sum_{j=0}^{d} \sum_{\beta' \in
    \free} [\Theta_{\beta}]_{\beta'j} z^{\beta' j} = \Theta_{\beta}(z)
\end{align*}
as wanted.
\end{proof}

\section{Model theory for $n$-hypercontractions}   \label{S:modeln}

    Recall Definition \ref{D:4.1} of an {\em $n$-hypercontractive 
    operator $d$-tuple $\bA = (A_{1}, \dots, A_{d})$}:  an operator 
    $d$-tuple $\bA = (A_{1}, \dots, A_{d})$ on a Hilbert space $\cX$ 
    is {\em $n$-hypercontractive} if the defect operators 
    $$
    \Gamma_{m,\bA}[I_{\cX}] = (I - B_{\bA})^{m}[I_{\cX}] = \sum_{\alpha \in \free 
    \colon  |\alpha| \le m} (-1)^{|\alpha|} \bcs{ m \\ |\alpha| } 
    \bA^{*\alpha^{\top}} \bA^{\alpha}
    $$
    are all positive semidefinite for $1\le m \le n$.  As in the general $\bo$-setting,
we consider an  operator $d$-tuple $\bT = (T_{1}, \dots, T_{d})$ for which the 
    adjoint tuple $\bA: = \bT^{*}: = (T_{1}^{*}, \dots, T_{d}^{*})$ 
    is an $n$-hypercontraction.    Note that while it is the case that $\Gamma_{m, \bA} = \Gamma_{\bmu_m, \bA}$,
    the definition of $\bmu_n$-hypercontraction requires that the shifted defect operators
    $$
    \Gamma_{\bmu_n, \bA}^{(k)}[I_\cX] = \frac{R_{\bmu_n, k}}{R_{\bmu_n}}(B_A)
    $$
    be positive-semidefinite for all $k \ge 0$.  However the linking identity \eqref{Gamma-bmun}
    enables one to show that these two classes are identical.  One can check that in all the formulas
    in the preceding sections of this Chapter, while shifts of observability gramian operators are ubiquitous, shifted
    defect operators ($\Gamma_{\bmu_k}$ versus $\Gamma^{(k)}$ with $k \ge 1$) do not occur.  Furthermore,
    by Remark \ref{R:stabmu}, in case $\bo = \bmu_n$, $\bmu_n$-strong stability may be replaced by ordinary strong
    stability.
   Thus all the results
    concerning model theory for $*$-$\bo$-hypercontractive operator tuples apply {\em mutatis mutandis}
    for the particular case of $n$-hypercontractions by simply specializing the weight $\bo$ to the case
    $\bo = \bmu_n$, and with {\em $\bo$-strong stability} replaced by the standard notion of {\em strong stability}.

\chapter[Regular formal power series]{Hardy-Fock spaces built from a regular formal power series}
\label{S:change}
\section{Introduction}   \label{S:change-intro}
We let $p$ be a \textbf{regular noncommutative formal power series};  by this we mean
that $p$ has a formal power series  representation
\begin{equation}  \label{apr12}
p(z) = \sum_{\alpha \in \free} p_{\alpha} z^{\alpha}
\end{equation}
with scalar coefficients $p_{\alpha} \in {\mathbb C}$ such that
\begin{equation} \label{apr12-1}
p_{\emptyset} = 0, \quad  p_{\alpha} > 0 \; \text{ if } \; |\alpha| =1, \quad
p_{\alpha} \ge 0 \; \text{ for all } \; \alpha \in \free.
\end{equation} 
We assume that this series has  a positive radius of convergence, i.e., there is a number 
$\rho > 0$ such that, whenever  $\bA = (A_{1}, \dots, A_{d})$ is a Hilbert-space operator
$d$-tuple such that 
$\| \begin{bmatrix} A_{1}, \dots, A_{d} \end{bmatrix}\| < \rho$, then the series
$$
p(\bA): =\sum_{k=1}^{\infty} \sum_{\alpha \colon |\alpha| = k} p_{\alpha}\bA^{\alpha}
$$
converges in the strong operator topology (see \cite{PopescuJFA2006} for additional background on
free holomorphic functions defined on multivariable operator balls  via formal power series in freely noncommutative indeterminates as
we have here).  We also apply this
functional calculus to operator $d$-tuples in $\cL(\cL(\cX))^{d}$, in
particular to the operator $d$-tuple $B_{\bA} = (B_{A_1}, \dots, B_{A_d})$ where $B_{A_j }\colon X \mapsto A_j^* X A_j$
(not to be confused with the Chapter \ref{S:Stein} notation \eqref{4.4} appropriate for the case where $p_\alpha = p_{|\alpha|}$):  
a standing assumption throughout this
chapter is that
\begin{equation}   \label{bA-assume}
\| \begin{bmatrix} B_{A_1} & \cdots & B_{A_d} \end{bmatrix}  \| = \| \begin{bmatrix} A_1&\cdots & A_d \end{bmatrix} \| < \rho
\end{equation}
so that the series  
$$
p(B_{\bA}): \; X\mapsto \sum_{\alpha \in \free} p_\alpha \bA^{*\alpha^\top}X\bA^\alpha
$$
converges in the norm topology of $\cL(\cL(\cX))$.
Returning to viewing $p$ as a formal power series \eqref{apr12}, for any $n \in
{\mathbb Z}_{+}$ we may form the formal power series $(1- p)^{n}$ given by
\begin{equation}
(1-p)^{n}(z)=\sum_{k=0}^{n} (-1)^{k} \bcs{n \\ k }
\bigg(\sum_{\alpha \ne \emptyset} p_{\alpha} z^{\alpha}\bigg)^{k}
=\sum_{\alpha \in \free} c_{p,n; \alpha} z^{\alpha}\label{march27}
\end{equation}
where
$$
  c_{p,n; \alpha} = \sum_{k=0}^{n}  (-1)^{k} \bcs{n \\ k}
\sum_{\beta_{1}, \cdots, \beta_{k} \ne \emptyset \colon \beta_{1} \cdots \beta_{k}
= \alpha} p_{\beta_{1}} \cdots p_{\beta_{k}}.
$$
Applying the functional calculus described above, given an operator
$d$-tuple $\bA$ with the assumption
\eqref{bA-assume} in force,  we may form the operator
$(1 - p)^{n}(B_{\bA}) \in \cL(\cL(\cX))$:
\begin{equation}
\Gamma_{p,n;\bA}[X] := (1 - p)^{n}(B_{\bA})[X]
= \sum_{\alpha \in \free} c_{p,n; \alpha} \bA^{* \alpha^{\top}} X\bA^{\alpha}. 
\label{march26}
\end{equation}
\begin{definition}
We shall say that the operator-tuple $\bA = (A_{1}, \dots, A_{d})$ is
\textit{$(p,n)$-contractive} if, in addition to  \eqref{bA-assume},  it is the case that
$\Gamma_{p,n;\bA}[I_{\cX}] \succeq 0$.  We shall say that $\bA$ is \textit{$(p,n)$-hypercontractive} if
$$
\Gamma_{p,k;\bA}[I_{\cX}] \succeq 0\quad\mbox{for}\quad 1 \le k \le n.
$$
Below, many of the results of Chapters 3--6 concerning $n$-hypercontractions (which are
$(p,n)$-hypercontractions for the choice of $p(z) = z_{1} + \cdots + z_{d}$) will 
be extended to the general $(p,n)$-setting for the general $p(z)$ subject to \eqref{apr12}, \eqref{apr12-1}.
\label{D:10.1}
\end{definition}

\smallskip

In addition to the auxiliary power series \eqref{march27}, we  shall have need of 
the $(p,n)$-analog of the resolvent-power function $R_n$ \eqref{1.8pre}:
\begin{align}
R_{p,n}(z) := R_{n}(p(z)) &=  (1 - p(z))^{-n} \notag\\
&=    \sum_{j=0}^{\infty} \bcs{ n+j-1 \\ j } \bigg( \sum_{\alpha \ne
    \emptyset} p_{\alpha} z^{ \alpha}\bigg)^{j}
     := \sum_{\alpha \in \free} \omega^{-1}_{p,n; \alpha} z^{\alpha},
    \label{Rpn}
\end{align}
where $\omega_{p,n; \emptyset} = 1$ and 
\begin{equation}   \label{omegapnalpha}
\omega^{-1}_{p,n; \alpha} := \sum_{j=0}^{|\alpha|} \bcs{n+j-1 \\ j}
\sum_{\beta_{1},\dots, \beta_{j} \ne \emptyset \colon \beta_{1}\cdots
\beta_{j} = \alpha} p_{\beta_{1}} \cdots p_{\beta_{j}} \; \text{ for } |\alpha| \ge 1.
\end{equation}
Due to assumptions \eqref{apr12-1}, all the terms on the right side of \eqref{omegapnalpha} are nonnegative and 
at least one term
(corresponding to $j=|\alpha|$) is positive. Thus, the reciprocal notation on the left side of \eqref{apr12-1} makes sense and it 
is clear that $\omega_{p,n; \alpha}>0$ for all $\alpha\in\free$. We next introduce the shifted counterparts of \eqref{Rpn}
\begin{equation}
R_{p,n;\beta}(z):= \sum_{\alpha \in \free} \omega_{p,n; \alpha\beta}^{-1} z^{\alpha}
\quad\mbox{for all}\quad \beta\in\free
\label{march20}
\end{equation}
(the $(p,n)$-analog of \eqref{1.8pre}) and observe that for any $\beta\in\free$,
\begin{equation}
R_{p,n;\beta}(z)=\omega_{p,n; \beta}^{-1} +\sum_{j=1}^d R_{p,n;j\beta}(z)z_j
\label{march20a}
\end{equation}
which can be seen as the noncommutative $(p,n)$-analog of the relation \eqref{1.9pre}. 

\smallskip

Given an operator $d$-tuple $\bA = (A_{1}, \dots, A_{d})$ and $z =
(z_{1}, \dots, z_{d})$ our set of free noncommutative indeterminates,
we set $z \bA$ equal to the $d$-tuple of monomials with operator
coefficients
$$
  z  \bA  = (z_{1} A_{1}, \dots, z_{d} A_{d})
$$
and then define $R_{p,n}(z  \bA)$ and $R_{p,n,\beta}(z  \bA)$ via formal power series
\begin{equation}
R_{p,n}(z\bA)=\sum_{\alpha \in\free} \omega_{p,n;\alpha}^{-1} {\bf
A}^\alpha z^{\alpha},\qquad R_{p,n;\beta}(z  \bA)=\sum_{\alpha \in\free} \omega_{p,n;\alpha\beta}^{-1} 
{\bf A}^\alpha z^{\alpha}.
\label{2.14pre-p}
\end{equation}
Note that the assumption that $p_{\alpha} \ge 0$ for all $\alpha \in
\free$ and $p_{\alpha} > 0$ when $|\alpha| =1$ guarantees that the
term $j=|\alpha|$ in the sum \eqref{omegapnalpha} is strictly
positive while all other terms are nonnegative; hence $\omega_{p,n;
\alpha}^{-1} > 0$, justifying the inverse notation.  We now have
arrived at a weight function $\bo_{p,n}: = \{ \omega_{p,n; \alpha}
\}_{\alpha \in \free}$ (where $\omega_{p,n; \alpha} = \left(\omega_{p,n;
\alpha}^{-1}\right)^{-1}$) of positive real numbers.
The $(p,n)$-analog of \eqref{1.31pre} is the weighted system
$$
\Sigma_{\{\bU_\alpha\}, p, n} \colon
\left\{ \begin{array}{cl}
x(1 \alpha) &=
{\displaystyle\frac{\omega_{p,n;\alpha}}{\omega_{p,n;1\alpha}}}A_{1}x(\alpha)+
{\displaystyle\frac{1}{\omega_{p,n;1\alpha}}}B_{1,\alpha} u(\alpha)\\
\vdots & \qquad\qquad\vdots \qquad\qquad\qquad\vdots\\
x(d \alpha) &=
{\displaystyle\frac{\omega_{p,n;\alpha}}{\omega_{p,n;d\alpha}}}A_{d}x(\alpha)+
{\displaystyle\frac{1}{\omega_{p,n;d\alpha}}}B_{d,\alpha} u(\alpha)\vspace{2mm}\\
y(\alpha) &  =  Cx(\alpha)+\omega_{p,n;\alpha}^{-1} D_\alpha u(\alpha).\end{array}\right.
$$
Upon running the latter system with the fixed initial condition $x_\emptyset=x$ we get
$$
y(\alpha)=\omega_{p,n;\alpha}^{-1}\cdot\bigg(C\bA^\alpha x+D_\alpha u(\alpha)+
\sum_{\alpha^{\prime\prime}j\alpha^\prime=\alpha}
C{\bf A}^{\alpha^{\prime\prime}}B_{j,\alpha^\prime}u(\alpha^\prime)\bigg).
$$
Applying the noncommutative $Z$-transform \eqref{1.25pre} to the latter formula and taking into account
\eqref{2.14pre-p} we have
\begin{align}
\widehat y(z)=&\sum_{\alpha\in\free} \omega_{p,n;\alpha}^{-1}
\bigg(C\bA^\alpha x +D_\alpha u(\alpha)+
\sum_{\alpha^{\prime\prime}j\alpha^\prime=\alpha}
C{\bA}^{\alpha^{\prime\prime}}B_{j,\alpha^\prime}u(\alpha^\prime)\bigg)z^\alpha \notag\\
=&CR_{p,n}(z\bA)x+\sum_{\alpha\in\free} \omega_{p,n;\alpha}^{-1}D_\alpha z^\alpha u(\alpha)\notag\\
&+\sum_{\alpha^\prime\in\free}\sum_{j=1}^d
\bigg(\bigg(\sum_{\alpha^{\prime\prime}\in\free}\omega_{p,n;\alpha^{\prime\prime} j\alpha^{\prime}}^{-1}
C{\bA}^{\alpha^{\prime\prime}}
z^{\alpha^{\prime\prime}}\bigg)z_jB_{j,\alpha^\prime}\bigg)z^{\alpha^\prime}u(\alpha^\prime)\notag\\
=&CR_{p,n}(z\bA)x+\sum_{\alpha\in\free}
\bigg(\omega_{p,n;\alpha}^{-1} D_\alpha+C\sum_{j=1}^d R_{p,n;j\alpha}(z\bA)z_jB_{j,\alpha}\bigg)
z^{\alpha}u(\alpha)
\notag\\
=&\cO_{p,n;C,{\bf A}}x+\sum_{\alpha\in\free}\Theta_{p,n,\bU_\alpha}(z)z^{\alpha}u(\alpha),
\label{1.36pre-p}
\end{align}
where the first term represents the $(p,n)$-observability operator
\begin{equation}
\cO_{p,n; C, \bA} \colon x \mapsto C R_{p,n}( z \bA) x
= \sum_{\alpha\in\free} (\omega_{p,n; \alpha}^{-1} C\bA^{\alpha} x) z^{\alpha}
\label{march21}
\end{equation}
associated with the output pair $(C,{\bf A})$ and where
\begin{align}
\Theta_{p,n,\bU_\alpha}(z)&=\omega_{p,n;\alpha}^{-1} D_\alpha+C\sum_{j=1}^d 
R_{p,n;j\alpha}( z \bA)z_jB_{j,\alpha}\notag\\
&=\omega_{p,n;\alpha}^{-1} D_\alpha+\sum_{\gamma\in\free}\sum_{j=1}^d 
\omega_{p,n;\gamma j\alpha}^{-1} C\bA^{\gamma}B_{j,\alpha}z^{\gamma j}
\label{1.37pre-p}
\end{align}
is a family of transfer functions. Observe that except for the input-to-state operator $\widehat B_{\alpha}$,
the transfer function $\Theta_{\bo,\bU_\alpha}$ in \eqref{thetahardy} depends on $|\alpha|$ rather than $\alpha$.
In contrast to this basic case ($p(z) = z_{1} + \cdots + z_{d}$), the dependence of 
$\Theta_{p,n,\bU_\alpha}$ in \eqref{1.37pre-p} on $\alpha$ in general
is more substantial. We point out, however, that the formula \eqref{1.37pre-p} can be written
in the form resembling the basic case \eqref{1.37pre} as follows:
\begin{equation}
\Theta_{p,n, \bU_\beta}(z)=\omega_{p,n;\beta}^{-1}D_{\beta}
+C\widehat{R}_{p,n;\beta}(z\bA)\widehat{Z}(z)\widehat{B}_\beta,
\label{apr22}
\end{equation}
where $\widehat{B}_\beta= \sbm{ B_{1,\beta} \\ \vdots \\ B_{d,
\beta} }\in\cL(\cU_\beta,\cX^d)$ is defined as in \eqref{coll} and where
\begin{equation}
\widehat{R}_{p,n;\beta}(z)=\begin{bmatrix}R_{p,n;1\beta}(z) &\ldots & 
R_{p,n;d\beta}(z)\end{bmatrix},\quad
\widehat{Z}(z)=\begin{bmatrix}z_1 I_{\cX}&& 0 \\ &\ddots& \\ 0&& z_d I_{\cX}\end{bmatrix}.
\label{march20ab}
\end{equation}

\section{Contractive multipliers from $H^{2}_{\widetilde p, \cU}(\free)$ to 
$H^{2}_{\bo_{p,n}, \cY}(\free)$}  \label{S:con-mult-p}

Starting with a power series \eqref{apr12} with nonnegative coefficients
and composing it with $R_n$ in \eqref{Rpn} we came up with the weight function 
$\bo_{p,n}: = \{ \omega_{p,n; \alpha}\}_{\alpha \in \free}$. The corresponding
Hardy-Fock space $H^{2}_{\bo_{p,n},\cY}(\free)$ is now defined as in \eqref{18.1}:
\begin{equation}
H^2_{\bo_{p,n},\cY}(\free) = \bigg\{\sum_{\alpha \in \free}
f_{\alpha}z^{\alpha}
\in \cY\langle\langle z\rangle\rangle\colon \;
\sum_{\alpha \in \free} \omega_{p,n;\alpha}\cdot \|f_{\alpha}\|_{\cY}^{2} <\infty \bigg\},
\label{18.1p}
\end{equation}
and it follows, as in \eqref{kbo}, that $H^{2}_{\bo_{p,n}, \cY}(\free)$ is a NFRKS with
noncommutative reproducing kernel
\begin{equation}   \label{kbo-p}
    k_{\bo_{p,n}}(z, \zeta) = \sum_{\alpha \in \free}
    \omega_{p,n; \alpha}^{-1} z^{\alpha}
    \overline{\zeta}^{\alpha^{\top}}.
\end{equation}
 However, it is no longer the case that the sequence $\bo = \bo_{p,n}$
 satisfies the conditions \eqref{18.2}.  We therefore must develop
 some of the results from Chapter  \ref{C:contrmult} from scratch.

\smallskip

To get an analog of Theorem \ref{T:12.2} for the $(p,n)$-setting, 
we need an adjustment of the Hardy-Fock space $H^{2}_{\cU}(\free)$ 
serving as the input space for $M_{\Theta}$ in the statement of the 
theorem. Given $p(z)$ as in \eqref{apr12}, we let $\widetilde p$ be the linear part of 
$p$, or, more precisely, define $\widetilde p$ according to the 
following recipe:
$$
p(z) = \widetilde p(z) + \sum_{\alpha \in \free \colon |\alpha| \ge 
2} p_{\alpha} z^{\alpha}, \quad \text{so}\quad \widetilde p(z) = \sum_{j=1}^{d} 
p_{j} z_{j}.
$$
Then the generalized Hardy space $H^{2}_{\bo_{\widetilde{p},1}, \cU}(\free)$ is the adjustment of the Hardy-Fock space
$H^2_{\cU}(\free)$ which we seek.
Note that if $\widetilde p(z) = z_{1} + \cdots + z_{d}$ (i.e., 
$p_{\alpha} = 1$ for each word $\alpha$ with $|\alpha | = 1$), then $H^{2}_{\bo_{\widetilde{p},1}, \cU}(\free)$ is just the same space as 
Hardy-Fock space $H^{2}_{\cU}(\free)$ and we have the redundant notation
$k_{\rm nc, Sz} = k_{\bo_{\widetilde{p},1}}=k_{\bf 1}$
for the noncommutative Szeg\H{o} kernel 
(where $\bf 1$ is the weight sequence consisting of all $1$'s).

\smallskip

For the case of a general $p$, for each $\alpha=i_1i_2\ldots i_N\in\free$, we let 
\begin{equation}
d_{\widetilde{p},\alpha}=d_{\widetilde{p},i_1i_2\ldots i_N}:=p_{i_1}p_{i_2}\cdots p_{i_N}\quad\mbox
{and}\quad
d_{\widetilde{p},\emptyset}=1.
\label{jan4bxa}
\end{equation}
Then the kernel
\begin{equation}
k_{\bo_{\widetilde{p},1}}(z,\zeta)=\sum_{\alpha\in\free}d_{\widetilde{p},\alpha}z^\alpha\bzeta^{\alpha^\top}
\label{jan4bx}
\end{equation}
is the  reproducing kernel for the NFRKHS $H^2_{\bo_{\widetilde{p},1},\cU}(\free)$,  i.e., 
$\omega_{\widetilde p, 1; \alpha}^{-1} = d_{\widetilde p, \alpha}$, and hence also
\begin{equation}   \label{jan4bx'}
\| \sum_{\alpha \in \free} f_\alpha z^\alpha \|^2_{H^2_{\bo_{\widetilde{p}, 1}, \cY}(\free)} =
\sum_{\alpha \in \free} d_{\widetilde p, \alpha}^{-1} \| f_\alpha \|^2.
\end{equation}

The spaces $H^2_{\bo_{\widetilde{p},1}, \cZ}(\free)$ have many properties in common with the Fock space
$H^2_\cZ(\free)$.  Let us adapt Definition \ref{D:mult} to the $H^2_{\bo_{\widetilde{p},1}, \cZ}(\free)$-setting; in particular
we say that a multiplier $\Theta$ from $H^2_{\bo_{\widetilde{p},1}, \cU}(\free)$ to $H^2_{\bo_{\widetilde{p},1}, \cY}(\free)$ is {\em strictly inner}
if the multiplication operator $M_\Theta \colon H^2_{\bo_{\widetilde{p},1}, \cU}(\free) \to H^2_{\bo_{\widetilde{p},1}, \cY}(\free)$ is
isometric.  Then we have the following analog of Lemma \ref{L:in}.

\begin{lemma}
Let $F$ be a contractive multiplier from $H^2_{\bo_{\widetilde{p},1},\cU}(\free)$ to 
$H^2_{\bo_{\widetilde{p},1},\cY}(\free)$.
Then $F$ is a strictly inner if and only if $\|Fu\|_{H^2_{\widetilde{p},\cY}(\free)}=
\|u\|_{\cU}$ for all
$u\in\cU$.
\label{L:in-p}
\end{lemma}

\begin{proof}
The proof of the non-trivial "if" part follows the lines of the proof of Lemma \ref{L:in}. 
Note that the shift operators 
$S_{\bo_{\widetilde{p},1},R,j}: \, f(z)\mapsto f(z)z_j$ satisfy
$$
S_{\bo_{\widetilde{p},1},R,j}^*S_{\bo_{\widetilde{p},1},R,j}=p_j^{-1} I_{H^2_{\bo_{\widetilde{p},1},\cY}(\free)}
\quad(j=1,\ldots,d)
$$
and have mutually orthogonal ranges. Therefore,
$$
\|\bS_{\widetilde{p},R}^{\alpha} M_{F}u\|_{H^2_{\widetilde{p},\cY}(\free)}=
\frac{\|M_{F} u\|_{H^2_{\widetilde{p},\cY}(\free)}}{\sqrt{d_{\widetilde{p},\alpha}}}
=\frac{\|u\|_\cU}{\sqrt{d_{\widetilde{p},\alpha}}}
\quad\mbox{for all}\quad u\in\cU, \; \; \alpha\in\free.
$$
Therefore, for any $u,v\in\cU$ and $\alpha,\beta\in\free$, we have
\begin{align*}
& \|\bS_{\bo_{\widetilde{p},1},R}^\alpha M_{F}u+
\bS_{\bo_{\widetilde{p},1},R}^\beta M_{F}v\|_{H^2_{\bo_{\widetilde{p},1},\cY}(\free)}^2  \\
& \quad =\frac{\|u\|^2_{\cU}}{d_{\widetilde{p},\alpha}}+
\frac{\|v\|^2_{\cU}}{d_{\widetilde{p},\beta}}+2{\rm Re} \big\langle 
\bS_{\bo_{\widetilde{p},1},R}^\alpha M_{F}u, \,
\bS_{\bo_{\widetilde{p},1}  ,R}^\beta M_{F}v\big\rangle_{H^2_{\bo_{\widetilde{p},1},\cY}(\free)}.
\end{align*}
On the other hand, as $M_{F}: \;
H^2_{\bo_{\widetilde{p},1},\cU}(\free)\to H^2_{\bo_{\widetilde{p},1},\cY}(\free)$ is a contraction, we have 
\begin{align*}
\|\bS_{\bo_{\widetilde{p},1},R}^\alpha M_{F}u+
\bS_{\bo_{\widetilde{p},1},R}^\beta M_{F}v\|_{H^2_{\bo_{\widetilde{p},1},\cY}(\free)}^2
&\le \|\bS_{\bo_{\widetilde{p},1},R}^\alpha  u+\bS_{\bo_{\widetilde{p},1},R}^\beta v\|_{H^2_{\bo_{\widetilde{p},1},\cU}(\free)}^2 \\
& =\frac{\|u\|^2_{\cU}}{d_{\widetilde{p},\alpha}}+
\frac{\|v\|^2_{\cU}}{d_{\widetilde{p},\beta}}
\end{align*}
for all $\alpha\neq \beta$. 
From the two latter relations, we conclude 
(as in the proof of Lemma \ref{L:in}) that
$$
\big\langle \bS_{\bo_{\widetilde{p},1},R}^\alpha M_{F}u, \, 
\bS_{\bo_{\widetilde{p},1},R}^\beta M_{F}v\big\rangle_{H^2_{\bo_{\widetilde{p},1},\cY}(\free)}= 0,
$$
and therefore, $\|Fh\|_{H^2_{\bo_{\widetilde{p},1},\cY}(\free)}=\|h\|_{H^2_{\bo_{\widetilde{p},1},\cU}(\free)}$ 
for every $\cU$-valued ``polynomial"
$h$ and therefore, for any $h\in H^2_{\bo_{\widetilde{p},1},\cU}(\free)$. 
Therefore, $F$ is a strictly inner multiplier from $H^2_{\bo_{\widetilde{p},1},\cU}(\free)$ to 
$H^2_{\bo_{\widetilde{p},1},\cY}(\free)$.
\end{proof}

We present next the $\bo_{\widetilde{p},1}$-analog of Theorems \ref{T:NC1} and  \ref{T:cm}.  To reduce this more general result
to the setting already covered in Chapter \ref{S:con-mult}, we use the change of variable
$$  
z = (z_1, \dots, z_d)  \mapsto \iota z := (p_1^{-\frac{1}{2}} z_1, \dots, p_d^{-\frac{1}{2}} z_d)
$$
together with the induced map on formal power series
\begin{equation}  \label{change-of-variable}
 f(z) = \sum_{\alpha \in \free} f_\alpha z^\alpha \mapsto (\iota f)(z) : = f(\iota z) =
 \sum_{\alpha \in \free} f_\alpha \, (\iota z)^\alpha = \sum_{\alpha \in \free} d_{\widetilde p, \alpha}^{-\frac{1}{2}} f_\alpha z^\alpha.
\end{equation}
Then using \eqref{jan4bx'} it is easily checked that $\iota$ is unitary as a map from $H^2_{\bo_{\widetilde p, 1}, \cU}(\free)$ onto
the Fock space $H^2_\cU(\free)$:
$$ 
 \| \iota f \|^2_{H^2_\cU(\free)} = \| f \|^2_{H^2_{\bo_{\widetilde p, 1}, \cU}(\free)}\quad \text{for all}\quad f \in H^2_{\bo_{\widetilde p, 1}, \cU}(\free),
$$
while in terms of kernel functions we have the identity
\begin{equation}   \label{ker-identity}
  k_{\bo_{\widetilde p, 1}} (\iota z, \iota \zeta) = k_{\rm nc, Sz}(z, \zeta).
\end{equation}
\begin{theorem}  \label{T:cm-p}  Define $\widetilde P$ and $Z_{\widetilde p}(z)$ as
\begin{equation}
\widetilde{P}=\begin{bmatrix} p_1 & & 0\\ &  \ddots & & \\ 0& & p_d\end{bmatrix},\quad
Z_{\widetilde{p}}(z)=\begin{bmatrix} p_1z_1 & \cdots & p_1z_1\end{bmatrix}\otimes I_{\cX}
\label{jan39}
\end{equation}
(where the Hilbert space $\cX$ is determined by the context).

\medskip
{\rm (1)} Any contractive multiplier $G$ from $H^2_{\bo_{\widetilde p, 1}, \cU}(\free)$ to $H^2_{\bo_{\widetilde p, 1}, \cY}(\free)$
admits a $\widetilde p$-unitary realization in the following sense:
there exists a Hilbert space $\cX$ and a connection matrix $\bU = \sbm{ A & B \\ C & D}$ of the form \eqref{1.20} and subject to constraints
\begin{equation}
\bU^*\begin{bmatrix} \widetilde P \otimes I_\cX  & 0 \\ 0 & I_\cY \end{bmatrix}\bU =\begin{bmatrix} I_\cX & 0 \\ 0 & I_\cU \end{bmatrix}, \quad
\bU\bU^*=\begin{bmatrix} \widetilde P^{-1}  \otimes I_\cX & 0 \\ 0 & I_\cY \end{bmatrix},
\label{punitary}
\end{equation}
so that
\begin{equation}   \label{ptilde-real}
G(z) = D + \sum_{j=1}^d  d_{\widetilde p, \alpha} C \bA^\alpha B_j z^{\alpha j} = D + C (I - Z_{\widetilde p}(z) A )^{-1} Z_{\widetilde p}(z) B.
\end{equation}
Conversely, if $G\in\cL(\cU,\cY)\langle\langle z\rangle\rangle$ has a realization as in \eqref{ptilde-real} with 
the connection matrix $\bU = \sbm{ A & B \\ C & D}$, such that
$\bU^*\sbm{\widetilde P \otimes I_\cX  & 0 \\ 0 & I_\cY}\bU \preceq  \sbm{I_\cX & 0 \\ 0 & I_{\cU}}$,
then $G$ is a contractive multiplier from $H^2_{\bo_{\widetilde p, 1}, \cU}(\free)$ to $H^2_{\bo_{\widetilde p, 1}, \cY}(\free)$.

\smallskip

{\rm (2)} Given a tuple $\bA=(A_1,\ldots,A_d)\in\cL(\cX)^d$ and $C\in\cL(\cX,\cY)$,
let $H\in\cL(\cX)$ be a strictly positive definite operator such that
\begin{equation}
H-\sum_{j=1}^dp_j A_j^*HA_j  \succeq C^*C.
\label{jan38}
\end{equation}
Let $A$ be defined as in \eqref{1.23pre} and let $\sbm{B \\ D}: \, \cU\to \sbm{\cX^d \\ \cY}$ be a solution 
of the Cholesky factorization problem
\begin{equation}
\begin{bmatrix} B \\ D \end{bmatrix}\begin{bmatrix}
B^{*} & D^{*}\end{bmatrix}=\begin{bmatrix} ( \widetilde{P} \otimes H)^{-1} & 0 \\0 &
I_{\cY} \end{bmatrix}-\begin{bmatrix} A \\ C
\end{bmatrix}H^{-1}\begin{bmatrix} A^{*} & C^{*}\end{bmatrix}.
\label{1039}
\end{equation}

\begin{enumerate}
\item[(i)]
Then the power series
$$
G(z)=D+C(I-Z_{\widetilde{p}}(z)A)^{-1}Z_{\widetilde{p}}(z)B
$$
is a contractive multiplier from $H^2_{\bo_{\widetilde{p},1},\cU}(\free)$ to $H^2_{\bo_{\widetilde{p},1},\cY}(\free)$.
Moreover, 
\begin{align*}
& k_{\bo_{\widetilde{p},1}}(z,\zeta)I_{\cY}-G(z)(k_{\bo_{\widetilde{p},1}}(z,\zeta) I_\cU)G(\zeta)^*    \\
& \quad =C(I-Z_{\widetilde{p}}(z)A)^{-1}H^{-1}(I-A^*Z_{\widetilde{p}}(\zeta)^*)^{-1}C^*.
\end{align*}

\item[(ii)]
If \eqref{jan38} holds with equality and $\bA$ is $\widetilde{p}$-strongly stable in the sense that  
$$
\lim_{N \to \infty} \widetilde{p}(B_{\bA})^{N}[I_{\cX}] := \lim_{N \to \infty} \sum_{\alpha\in\free:|\alpha|=N}
d_{\widetilde{p},\alpha}\bA^{*\alpha^\top}\bA^\alpha= 0 
$$
in the strong operator topology, then $G$ is McCT-inner. 

\item[(iii)]
If \eqref{jan38} holds with equality, $\bA$ is strongly stable, and the solution $\sbm{ B \\ D }$ of \eqref{1039}
is injective, then $G$ is a strictly inner multiplier.
\end{enumerate}
Conversely, any contractive,  McCT-inner and strictly inner multiplier from $H^2_{\bo_{\widetilde{p},1},\cU}(\free)$ to 
$H^2_{\bo_{\widetilde{p},1},\cY}(\free)$ arises in the way described in parts (1), (2), (3) above.
\end{theorem}

\begin{proof}  
By Proposition \ref{P:2.3}, $G$ being a contractive multiplier from $H^2_{\bo_{\widetilde p, 1}, \cU}(\free)$ to $H^2_{\bo_{\widetilde p, 1}, \cY}(\free)$
can be expressed in terms of the positivity of the associated formal kernel
$$
 K^G_{\widetilde p}(z, \zeta):= k_{\bo_{\widetilde p, 1}}(z, \zeta) I_\cY - G(z) (k_{\bo_{\widetilde p, 1}}(z, \zeta) I_\cU) G(\zeta)^*.
$$
Replace $z$ with $\iota z$ and $\zeta$ with $\iota \zeta$ to see that
$$
K^G_{\widetilde p}(\iota z, \iota \zeta) =
k_{\rm nc, Sz}(z, \zeta) I_\cY - G(\iota z) (k_{\rm nc, Sz}(z, \zeta) I_\cU) G(\iota \zeta)^*
$$
is also a positive formal kernel.  Again by Proposition \ref{P:2.3}, this in turn  means that the formal power series 
$G(\iota z)$ is a contractive multiplier from $H^2_\cU(\free)$ into $H^2_\cY(\free)$.  
By Theorem \ref{T:NC1} there is a unitary connection matrix as in \eqref{1.20}  which we write in the form
 $\widetilde{\bU} = \sbm{ \widetilde A & \widetilde B \\ C & D}$ so that
\begin{equation}   \label{real-1}
G(\iota z) = D + C (I - Z(z) \widetilde A)^{-1}  Z(z) \widetilde B.
\end{equation}
If we let $A : = (\widetilde P^{-\frac{1}{2}} \otimes I_\cX) \widetilde A$ and $B : = \widetilde P^{-\frac{1}{2}} B$, then the operator
$$
\bU := \begin{bmatrix} A & B \\ C & D \end{bmatrix}=\begin{bmatrix} \widetilde P^{-\frac{1}{2}} \otimes I_\cX)& 0 \\ 0 & I_{\cU}\end{bmatrix}\widetilde{\bU}
$$
meets the constraints \eqref{punitary}, since $\widetilde{\bU}$ is unitary. Furthermore, we see from \eqref{jan39} and \eqref{1.23pre} that
$Z_{\widetilde p}(z)=Z(z)(\widetilde P \otimes I_\cX)$ and hence,
\begin{align*}
Z(\iota^{-1} z) \widetilde A &= Z(z) (\widetilde P^{\frac{1}{2}} \otimes I_\cX) \widetilde A = Z_{\widetilde p}(z)
(\widetilde P^{-\frac{1}{2}} \otimes I_\cX) \widetilde A) =Z_{\widetilde p}(z)  A,\\
Z(\iota^{-1} z) \widetilde B&= Z(z) (\widetilde P^{\frac{1}{2}} \otimes I_\cX) \widetilde B =
Z_{\widetilde p}(z) (\widetilde P^{-\frac{1}{2}} \otimes I_\cX)\widetilde B =Z_{\widetilde P}(z) B.
\end{align*}
Replacing $z$ with $\iota^{-1} z$ in \eqref{real-1} and making use of the latter equalities we get
$$
  G(z) = D + C (I - Z(\iota^{-1} z) \widetilde A)^{-1} Z(\iota^{-1} z) \widetilde B=
D + C (I - Z_{\widetilde p}(z)  A)^{-1} Z_{\widetilde p}(z)  B.
$$
Thus, $G(z)$ admits a realization \eqref{ptilde-real} with the connection matrix $\widetilde \bU$ subject to conditions \eqref{punitary}.
In this way part (1) in Theorem \ref{T:cm-p} reduces to the content of Theorem \ref{T:NC1}.
Furthermore,  replacement of z by $\iota z$ and $\zeta$ by $\iota \zeta$ in 
the kernel identity \eqref{ker-identity} (with $\widetilde A$ in place of $A$ and $\widetilde B$ in place of $B$)
implies that replacement of z by $\iota z$ and $\zeta$ by $\iota \zeta$
leads to the identity
$$
k_{\rm nc, Sz}(z, \zeta) I_\cY - G(\iota z) (k_{\rm Sz, nc}(z, \zeta) I_\cU) G(\iota \zeta)^* = 
C (I - Z(z) A)^{-1} (I - A^* Z(\zeta))^{-1} C^*.
$$
Continuing along these lines, one can reduce all the assertions of the second statement (2) of Theorem \ref{T:cm-p} to the 
corresponding assertions in Theorem \ref{T:cm}.
\end{proof}

An application of the lurking isometry argument (or more simply of the change-of-variable  map \eqref{change-of-variable}
to reduce the statement to that of Theorem \ref{T:Leech})  leads us to the Leech theorem in the present setting.

\begin{theorem}  \label{T:Leech-p}
Given power series $G\in\cL(\cY,\cX)\langle\langle z\rangle\rangle$ and $F\in\cL(\cU,\cX)\langle\langle
z\rangle\rangle$, the formal kernel
$$
K_{G,F}(z, \zeta): = G(z)(k_{\bo_{\widetilde{p},1}}(z,\zeta)\otimes 
I_{\cY})G(\zeta)^*-F(z)(k_{\bo_{\widetilde{p},1}}(z,\zeta)\otimes I_{\cU})F(\zeta)^*
$$
is positive if and only if there exists a contractive multiplier $S$ from
$H^2_{\bo_{\widetilde{p},1},\cU}(\free)$ to $H^2_{\bo_{\widetilde{p},1},\cY}(\free)$ such that $F(z)=G(z)S(z)$.
\end{theorem}
Now we will establish the $p$-version of Theorem \ref{T:3.8}. Given $p(z)$ as in 
\eqref{apr12}--\eqref{apr12-1} and the associated weight
function \eqref{omegapnalpha}, let us define another non-negative weight ${\bgam}_{p,n}$ 
associated with $p$, namely, ${\bgam}_{p,n}=\{\gamma_{p,n;\alpha}\}_{\alpha\in\free}$,
where
\begin{equation}  \label{apr30c}
\gamma_{p,n;\alpha j}^{-1}=\omega_{p,n;\alpha j}^{-1}-\omega_{p,n;\alpha}^{-1}p_j\quad (\alpha\in\free, \;
j\in\{1,\ldots,d\}).
\end{equation}
If $n=1$, then we let $\gamma_{p,0;\alpha}=1$ if $\alpha=\emptyset$ and $\gamma_{p,0;\alpha}=0$, 
otherwise. The positivity of ${\bgam}_{p,n}$ is clear from the formula 
$$
\omega_{p,n;\alpha j}^{-1}-\omega_{p,n;\alpha}^{-1}p_j=
\omega_{p,n-1;\alpha j}^{-1}+
\sum_{\alpha^\prime\alpha^{\prime\prime}=\alpha: \alpha^{\prime\prime}\neq \emptyset}
\omega_{p,n;\alpha^{\prime}}^{-1}p_{\alpha^{\prime\prime}j}>0,
$$
which in turn, is verified by equating the coefficients of $z^{\alpha j}$ in the identity
$$
( 1 - p(z))^{-n}- (1 - p(z))^{-n+1} = (1 - p(z))^{-n} p(z)
$$
combined with  \eqref{Rpn} and \eqref{omegapnalpha} as follows:
$$
\omega_{p,n;\alpha j}^{-1}-\omega_{p,n-1;\alpha j}^{-1}=\sum_{\alpha^\prime\alpha^{\prime\prime}=\alpha}
\omega_{p,n;\alpha^\prime}^{-1}p_{\alpha^{\prime\prime}j} 
=\omega_{p,n;\alpha}^{-1}p_j+\sum_{\alpha^\prime\alpha^{\prime\prime}=\alpha: \alpha^{\prime\prime}\neq \emptyset}
p_{\alpha^{\prime\prime}j}.
$$
Let us define the {\em weighted $Z$-transform} $\bPs_{\bgam_{p,n}}$ by
\begin{equation} \label{jan4v}
 \bPs_{\bgam_{p,n}} = {\rm Row}_{\alpha \in \free}[\gamma_{p,n;\alpha}^{-\frac{1}{2}}
z^{\alpha}] \colon  \{f_\alpha \}_{\alpha \in\free}\mapsto \sum_{\alpha
\in\free}\gamma_{p,n;\alpha}^{-\frac{1}{2}}f_\alpha z^\alpha
\end{equation}
and let us identify $\bPs_{\gamma_{p,n}}$ with operator-valued power series
\begin{equation}
\bPs_{\bgam_{p,n}}(z) = \sum_{\alpha \in \free} \left({\rm Row}_{\beta \in \free}
[\delta_{\alpha, \beta} (\gamma_{p,n;\alpha}^{-\frac{1}{2}} \otimes
I_{\cY})] \right) z^{\alpha}
\label{jan4bv}
\end{equation}
where $\delta_{\alpha, \beta}$ is the Kronecker symbol. 
Then we have the following analog of item (2) in Theorems \ref{T:3.8}.

\begin{theorem}  \label{T:dopmay}
Let $M_{\bPs_{\bgam_{p,n}}}$ be the multiplication operator associated with the operator-valued function $\bPs_{\bgam_{p,n}}$
\eqref{jan4v}.  Then:
\begin{enumerate}
\item $M_{\bPs_{\bgam_{p,n}}}|_{\ell^2_\cY(\free)}$ is an isometry from the space of constant functions $\ell^2_\cY(\free)$ 
in $H^2_{\ell^2_\cY(\free)}(\free)$ into $H^2_{\bgam_{p,n}, \cY}(\free)$.

\item   $M_{\bPs_{\bgam_{p,n}}}$ is a coisometry from   $H^2_{\bo_{\widetilde{p},1}, \ell^2_\cY(\free)}(\free)$
to $H^2_{\bo_{p,n},\cY}(\free)$.

\item
A given formal power series $F(z) \in \cL(\cU, \cY)\langle \langle z \rangle \rangle$ is a contractive multiplier from 
$H^2_{\bo_{\widetilde{p},1}, \cU}(\free)$ to $H^2_{\bo_{p,n}, \cY}(\free)$ if and only if it is of the form 
\begin{equation}   \label{jan4d'-p}
F(z) = \bPs_{\bgam_{p,n}}(z) S(z)
\end{equation}
for some contractive multiplier $S(z)$ from $H^2_{\bo_{{\widetilde p},1}, \cU}(\free)$ to $H^2_{\bo_{\widetilde{p},1}, \ell^2_\cY(\free)}(\free)$.
\end{enumerate}
\end{theorem}

\begin{proof} 
For $\by = \{ y_\alpha\}_{\alpha \in \free}$ in $\ell^2_\cY(\free)$, note that
$$
\| \bPs_{\bgam_{p,n}}(z) \by \|^2_{H^2_{\bgam_{p,n}, \cY}(\free)} = 
\sum_{\alpha\in\free}\gamma_{p,n; \alpha} \| \gamma^{-\frac{1}{2}}_{p,n; \alpha} y_\alpha \|^2
= \| \by \|^2_{\ell^2_\cY(\free)}.
$$
This verifies statement (1).
To verify statement (2), by Proposition \ref{P:2.3}, it suffices to verify the identity (the analog of identity \eqref{jan4c})
\begin{equation}
k_{\bo_{p,n}}(z,\zeta) I_{\cY}=\bPs_{\bgam_{p,n}}(z)(k_{\bo_{\widetilde{p},1}}(z,\zeta)
 I_{\ell^2_{\cY}(\free)})\bPs_{\bgam_{p,n}}(\zeta)^*.
\label{may1}
\end{equation}
Indeed, since $\omega^{-1}_{p,n; \emptyset}=1$, we can write \eqref{kbo-p} as 
$$
k_{\bo_{p,n}}(z, \zeta)=1+\sum_{j=1}^d \sum_{\alpha\in\free}\omega^{-1}_{p,n; \alpha j}z^{\alpha j}
\overline{\zeta}^{j\alpha^\top}
$$
and then, upon taking into account the formula \eqref{apr30c} for $\gamma^{-1}_{p,n;\alpha}$ we see that
\begin{align*}
k_{\bo_{p,n}}(z, \zeta) - \sum_{j=1}^d p_j \overline{\zeta}_j k_{\bo_{p,n}}(z, \zeta) z_j 
& = 1+\sum_{j=1}^d \sum_{\alpha\in\free}
\big(\omega^{-1}_{p,n; \alpha j}- p_j \omega^{-1}_{p,n; \alpha} \big) z^{\alpha j}\overline{\zeta}^{j\alpha^\top}  \\
& =1+ \sum_{j=1}^d \sum_{\alpha\in\free}\gamma^{-1}_{p,n; \alpha j} z^{\alpha j} \overline{\zeta}^{j\alpha^\top}
=\sum_{\alpha\in\free}\gamma^{-1}_{p,n; \alpha} z^{\alpha} \overline{\zeta}^{\alpha^\top}
   \end{align*}
which is to say that
$$
  k_{\bo_{p,n}}(z, \zeta) = \sum_{j=1}^d p_j \overline{\zeta_j} k_{\bo_{p,n}}(z, \zeta) z_j + k_{\bgam_{p,n}}(z, \zeta).
$$
Iteration of this last identity and recalling the definition of $d_{\widetilde{p}, \alpha}$ \eqref{jan4bxa} then gives
$$
k_{\bo_{p,n}}(z, \zeta) = \sum_{\alpha \colon |\alpha| = N+1} d_{\widetilde p, \alpha} \overline{\zeta}^{\alpha^\top} k_{\bo_{p,n}}(z,\zeta) z^\alpha
   + \sum_{\alpha \colon |\alpha| \le N} d_{\widetilde p,\alpha} \overline{\zeta}^{\alpha^\top} k_{\bgam_{p,n}}(z, \zeta) z^\alpha
$$
for $N=1,2,\dots$.  Taking the limit as $N \to \infty$ then leads us to
\begin{align*}
k_{\bo_{p,n}}(z, \zeta) & = \sum_{\alpha} d_{\widetilde{p}, \alpha} \overline{\zeta}^{\alpha^\top} k_{\bgam_{p,n}} (z, \zeta) z^\alpha \\
& = \sum_{\alpha\in\free} d_{\widetilde p, \alpha} \overline{\zeta}^{\alpha^\top} \bPs_{\bgam_{p,n}}(z) \bPs_{\bgam_{p,n}}(\zeta)^* z^\alpha \\
&  = \bPs_\bgam(z) \bigg( \sum_{\alpha\in\free} d_{\widetilde{p}, \alpha} z^\alpha \overline{\zeta}^{\alpha^\top} \bigg) \bPs_{\bgam_{p,n}}(\zeta)^* \\
& =  \bPs_{\bgam_{p,n}}(z) \big( k_{\bo_{\widetilde{p},1}}(z, \zeta) \otimes I_{\ell^2_\cY(\free)} \big) \bPs_{\bgam_{p,n}}(\zeta)^*
\end{align*}
where we used \eqref{jan4bx} for the last step, and \eqref{may1} is thus verified.

\smallskip

The proof of (3) parallels the proof of statement (3) in Theorem \ref{T:3.8}.  Let $F(z) \in
\cL(\cU, \cY)\langle \langle z \rangle \rangle$ be a contractive multiplier from $H^2_{\bo_{\widetilde{p},1}, \cU}(\free)$ to 
$H^2_{\bo_{p,n}}(\free)$.  By Proposition \ref{P:2.3}, this amounts to the statement that the kernel
$$
  K_{\bo_{p,n}}^F(z, \zeta) = k_{\bo_{p,n}}(z, \zeta) \otimes I_\cY -
  F(z) \big(k_{\bo_{\widetilde{p},1}}(z, \zeta) \otimes I_{\ell^2_\cY(\free)} \big)F(\zeta)^*
$$
be a positive formal kernel.  On the other hand, by the Leech theorem (Theorem \ref{T:Leech-p}), the existence 
of the factorization \eqref{jan4d'-p} with $S(z)$ a contractive multiplier from $H^2_{\bo_{{\widetilde p},1}, \cU}(\free)$
to $H^2_{\bo_{\widetilde p,1}, \ell^2_\cY(\free)}(\free)$  is equivalent to the positivity of the kernel
$$
K_{\bo_{\widetilde{p}, 1}}^{\bPs_{p,n}, F}(z, \zeta) = \bPs_{\bgam_{p,n}}(z) \big( k_{\bo_{\widetilde{p},1}}(z, \zeta) \otimes I_{\ell^2_\cY(\free)} \big)
\bPs_{\bgam_{p,n}}(\zeta)^* - F(z) (k_{\bo_{\widetilde{p},1}}(z, \zeta) \otimes I_\cU) F(\zeta)^*.
$$
But as a consequence of the identity \eqref{may1}, the kernel $K_{\bo_{\widetilde{p},1}}^{\bPs_{p,n}, F}$ is the same as the kernel
$K_{\bo_{p,n}}^F$ and (3) now follows.
\end{proof}

In parallel with Definitions \ref{D:defin},  we  now introduce the following definition of {\em (p,n)-Bergman-inner}.
\begin{definition}   \label{D:8in}
    We say that a power series $\Theta \in \cL(\cU, \cY)\langle
    \langle z \rangle \rangle$ is {\em $(p,n)$-Bergman inner} if
    \begin{enumerate}
    \item[(i)] $\|M_\Theta u \|_{H^{2}_{\bo_{p,n}, \cY}(\free)}  = \| u \|_\cU \text{ for all } u \in \cU$, and
    \item[(ii)] $M_{\Theta}u  \perp {\bf S}_{\bo_{p,n},R}^\alpha M_{\Theta}v$ in $H^2_{\bo_{p,n}}(\free)$  for all $u,v \in \cU$ and
    $\emptyset \ne \alpha \in \free$.
   \end{enumerate}
 \end{definition}
 
The next two theorems  (the $(p,n)$-analogs of Theorems \ref{T:12.2} and \ref{T:bergin}) show that 
$(p,n)$-Bergman-inner power series can be alternatively defined as contractive multipliers from 
$H^{2}_{\bo_{\widetilde{p},1}, \cU}(\free)$ to $H^{2}_{\bo_{p,n}, \cY}(\free)$ that act isometrically on constants.

\begin{theorem}  \label{T:12.2-p}
    Let $\Theta \in \cL(\cU, \cY)\langle \langle z \rangle \rangle$
    be such that
\begin{enumerate}
    \item $\| M_{\Theta} u \|_{H^{2}_{\bo_{p,n}, \cY}(\free)} \le \|
    u \|_{\cU}$ for all $u \in \cU$, and
    \item $M_{\Theta} u \perp {\bf S}_{\bo_{p,n},R}^\alpha M_{\Theta} v$ for all $u,v
    \in \cU$ and all nonempty $\alpha \in \free$.
\end{enumerate}
Then $\Theta$ is a contractive multiplier from $H^{2}_{\bo_{\widetilde{p},1},
\cU}(\free)$ to $H^{2}_{\bo_{p,n}, \cY}(\free)$ and
\begin{align}
\|M_{\Theta} f\|_{H^2_{\bo_{p,n},\cY}(\free)}^2 \le &\|f\|^2_{H^2_{\bo_{\widetilde{p},1}, \cU}(\free)}\label{march13-g}\\
&-
\sum_{\alpha \in\free}\sum_{j=1}^d d_{\widetilde p, j\alpha}
\left\|\mathfrak D_j
M_{\Theta}{\bf S}_{\bo_{\widetilde{p},1},R}^{*})^{j \alpha^\top}f
\right\|^2_{H^2_{\bo_{p,n},\cY}(\free)}\notag
\end{align}
where $d_{\widetilde p, j\alpha}$ is defined in \eqref{jan4bxa} and where we let 
$$
\mathfrak D_j=(I- p_j S_{\bo_{p,n},R,j}^*S_{\bo_{p,n},R,j})^{\frac{1}{2}}
$$
for short.
Moreover, if $\| M_{\Theta} u \|_{H^{2}_{\bo_{p,n},\cY}(\free)} = \|
u \|_{\cU}$ for all $u \in \cU$, then equality holds in \eqref{march13-g}
and $\Theta$ is $(p,n)$-Bergman inner.
\end{theorem}

\begin{proof}
For any $\cU$-valued polynomial $f$ as in \eqref{jul21} and the right backward shift
tuple  ${\bf S}_{\widetilde{p},R}^*$ on $H^2_{\widetilde{p},\cU}(\free)$, we have
\begin{equation}   \label{actual}
S_{\widetilde p, R, j}^* f =  p_j^{-1} \sum_{\alpha \in \free \colon |\alpha| < m}  f_{\alpha j} z^\alpha.
\end{equation}
Due to assumptions of the theorem (the same as in Theorem \ref{T:12.2}), and since the ranges of
$S_{\bo_{p,n},R,i}$ and $S_{\bo_{p,n},R,j}$ are orthogonal whenever $i\neq j$, we can repeat the calculation
\eqref{nov10}) adjusting the formula \eqref{actual} instead of \eqref{jul21a} as follows:
\begin{align}
\|M_{\Theta} f\|^2
\le & \|f_{\emptyset}\|^2+
\bigg\|\sum_{j=1}^d\sum_{\alpha\in\free:|\alpha|\le m-1}
 {\bf S}_{\bo_{p,n},R}^{j\alpha^\top}M_{\Theta} f_{\alpha j} \bigg\|^2\label{nov10-g}\\
 = & \| f_{\emptyset} \|^2 + \sum_{j=1}^d \bigg\| S_{\bo_{p,n},j} M_\Theta \bigg(
\sum_{\alpha\in\free:|\alpha| \le m-1} f_{\alpha j} z^\alpha\bigg)  \bigg\|^2   \notag   \\
 =&\|f_{\emptyset}\|^2+\sum_{j=1}^d\left\|
S_{\bo_{p,n},R,j}M_{\Theta} ( p_j S_{\bo_{\widetilde{p},1},R,j}^*f) \right\|^2   \text{ (by \eqref{actual}) } \notag\\
 =&\|f_{\emptyset}\|^2+\sum_{j=1}^d\big\| p_j^{1/2}
S_{\bo_{p,n},R,j}M_{\Theta} ( p_j^{1/2} S_{\bo_{\widetilde{p},1},R,j}^*f) \big\|^2    \notag \\
 = &  \| f_{\emptyset} \|^2   + \sum_{j=1}^d \bigg( \big\| M_\Theta (p_j^{1/2} S_{\bo_{\widetilde{p},1}, R, j}^* f) \big\|^2 \notag\\
& \quad \qquad  -   \big\| \big( I - p_j S^*_{\bo_{p,n}, R,j} S_{\bo_{p,n}, R, j} \big)^{1/2}
M_\Theta \big( p_j^{1/2} S_{\bo_{\widetilde{p},1}, R, j}^* f \big) \big\|^2 \bigg)\notag \\
=&\| f_{\emptyset} \|^2   + \sum_{j=1}^d \left( \big\| M_\Theta (p_j^{1/2} S_{\bo_{\widetilde{p},1}, R, j}^* f) \big\|^2 
-\big\| \mathfrak D_jM_\Theta \big( p_j^{1/2} S_{\bo_{\widetilde{p},1}, R, j}^* f \big) \big\|^2 \right).  \notag
\end{align}
Replace $f$ by $d_{\widetilde p, \alpha}^{1/2} {\bf S}_{\bo_{\widetilde{p},1},R}^{*\alpha^\top}f$ in \eqref{nov10-g} and take into account that
$({\bf S}_{\bo_{\widetilde{p},1},R}^{*\alpha^\top}f)_{\emptyset}= d_{\widetilde p, \alpha}^{-1} f_{\alpha}$ (as a consequence of \eqref{actual})
 to arrive at
 \begin{align}
\|M_{\Theta}d_{\widetilde p, \alpha}^{1/2} ({\bf S}_{\bo_{\widetilde{p},1},R}^*)^{\alpha^\top}f\|^2
& \le d_{\widetilde p, \alpha}^{-1}  \|f_\alpha\|^2
+\sum_{j=1}^d\left\|M_{\Theta} (d_{\widetilde p, j \alpha}^{1/2}) ({\bf S}_{\bo_{\widetilde{p},1},R}^{*})^{j \alpha^\top}f \right\|^2 \label{8.3'-g}\\
&\qquad\qquad -\sum_{j=1}^d\left\|\mathfrak D_jM_{\Theta} ( d_{\widetilde p, j \alpha}^{1/2}) 
({\bf S}_{\bo_{\widetilde{p}1},R}^{*})^{j \alpha^\top}f \right\|^2.
\notag
\end{align}
Iterating the inequality \eqref{nov10-g}   and using \eqref{8.3'-g} then gives, for any $m' =1,2,\dots$,
\begin{align}
\|M_{\Theta} f\|^2 &\le \sum_{|\alpha| < m'}  d_{\widetilde p, \alpha}^{-1} \|f_\alpha \|^2_{\cU}
+ \sum_{|\alpha| = m'} \left\| M_\Theta (d_{\widetilde p, \alpha}^{1/2}) (\bS_{\bo_{\widetilde{p},1}, R}^{*})^{ \alpha^\top} f \right\|^2
\notag \\
&  \qquad -\sum_{|\alpha|< m'}\sum_{j=1}^d
\left\|\mathfrak D_j
M_{\Theta}(d_{\widetilde p, j \alpha}^{1/2}) ({\bf S}_{\bo_{\widetilde{p},1},R}^{*})^{j \alpha^\top}f
\right\|^2.  \notag
\end{align}
Since $f_\alpha = 0$ once $|\alpha| \ge m$ it follows that $(\bS_{\bo_{\widetilde{p},1}, R}^*)^\alpha f = 0$ for
any $\alpha$ with $|\alpha| \ge m$ and we see that, once $m' \ge m$, this last inequality collapses to
\begin{align}
&\|M_{\Theta} f\|^2
\le \|f\|^2
- \sum_{|\alpha|< m' }\sum_{j=1}^d
\left\|\mathfrak D_j
M_{\Theta}(d_{\widetilde p, j \alpha}^{1/2}) ({\bf S}_{\bo_{\widetilde{p},1},R}^{*})^{j \alpha^\top} f
\right\|^2.
\notag
\end{align}
Letting $m' \to \infty$ in the last inequality then gives us \eqref{march13-g} 
for any $\cU$-valued noncommutative polynomial $f$.
We then get the result for a general $f$ in $H^{2}_{\bo_{\widetilde{p},1},\cU}(\free)$ by
approximating $f$ by finite truncations of its power series representation.

\smallskip

If equality $\|M_{\Theta} u\|_{H^2_{\bo_{p,n},\cY}(\free)}=\|u\|_{\cU}$ holds for
all $u\in\cU$, then we have equalities throughout \eqref{nov10-g}, \eqref{8.3'-g} and therefore,
in \eqref{march13-g} as well.   Moreover as a consequence of \eqref{march13-g} we have that
$\Theta$ is a contractive multiplier from $H^2_{\bo_{\widetilde{p},1}, \cU}(\free)$ to $H^2_{\bo_{p,n}, \cY}(\free)$.
As we are also assuming the equality $\|M_{\Theta} u\|_{H^2_{\bo_{p,n},\cY}(\free)}=\|u\|_{\cU}$ for
all $u \in \cU$, it follows from Definition \ref{D:8in} that $\Theta$ is $(p,n)$-Bergman inner.
\end{proof}

\begin{theorem}  \label{T:berginp}
Let $\Theta\in\cL(\cU,\cY)\langle\langle z\rangle\rangle$ be a contractive multiplier from 
$H^2_{\bo_{\widetilde{p},1},\cU}(\free)$ to
$H^2_{\bo_{p,n}, \cY}(\free)$ which is isometric on constants: $\| M_\Theta u \|_{H^2_{\bo_{p,n}, \cY}(\free)} = \| u \|_\cU$ for all $u \in \cU$.
Then $\Theta$ is $(p,n)$-Bergman inner. Furthermore, there is a unique contractive 
multiplier $S$ from $H^2_{\bo_{\widetilde{p},1},\cU}(\free)$ to
$H^2_{\ell^{2}_{\widetilde{p},\cY}(\free)}(\free)$ such that
\begin{equation}
\Theta(z) =\bPs_{\bgam_{p,n}}(z)S(z)
\label{march9p}
\end{equation}
where $\bPs_{\bgam_{p,n}}(z)$ is defined as in \eqref{apr30c} and \eqref{jan4bv}. 
Moreover, this unique $S$ is a strictly inner multiplier
from $H^2_{\bo_{\widetilde{p},1}, \cU}(\free)$ to $H^2_{\bo_{\widetilde{p},1},\ell^{2}_{\cY}(\free)}(\free)$.
\end{theorem}

\begin{proof}
Since $\Theta$ is a contractive multiplier from $H^2_{\bo_{\widetilde{p},1},\cU}(\free)$ to
$H^2_{\bo_{p,n},\cY}(\free)$, it is (by Theorem \ref{T:dopmay} (3)) of the form \eqref{march9p}
for some contractive multiplier $S(z)$ from $H^2_{\bo_{\widetilde{p},1},\cU}(\free)$ to $H^2_{\bo_{\widetilde{p},1},\ell^{2}_{\cY}(\free)}(\free)$
By assumption we have
$$
\|u\|_{\cU}= \|M_{\Theta} u\|_{H^2_{\bo_{p,n},\cY}(\free)}=\|M_{\bPs_{\bgam}} M_{S} u\|_{H^2_{\bo_{p,n},\cY}(\free)}
$$
for all $u\in\cU$. To repeat verbatim the arguments from the proof of Theorem \ref{T:bergin}
regarding uniqueness and strict-inner property of $S$ as well as the 
orthogonality relations $\Theta u\perp {\bf S}_{\bo_{p,n}}\Theta v$
(required to conclude that $\Theta$ is $(p,n)$-Bergman inner) we need the operators 
$M_{S} \colon \cU \to H^2_{\bo_{\widetilde{p},1},\ell^{2}_{\cY}(\free)}(\free)$ and 
$\; M_{\bPs_{\bgam}} \colon H^2_{\bo_{\widetilde{p},1},\ell^{2}_{\cY}(\free)}(\free)\to  H^2_{\bo_{p,n},\cY}(\free)$ 
to be respectively contractive and coisometric. But we have this by part (2) in Theorem \ref{T:dopmay} and 
the fact that $S(z)$ is a contractive multiplier from $H^2_{\bo_{\widetilde{p},1},\cU}(\free)$ to
$H^2_{\bo_{\widetilde{p},1},\ell^{2}_{\cY}(\free)}(\free)$.
\end{proof}

\section{Output stability, Stein equations and inequalities}  \label{S:obs-gram-p}

For convenience of future reference, let us collect the following definitions.
\begin{definition}  \label{D:def-pn}  Given $C \in \cL(\cX, \cY)$ and $\bA = (A_1, \dots, A_d) \in \cL(\cX)^d$, we say that:

\smallskip

{\rm (1)} $(C,\bA)$  is {\em $(p,n)$-output stable}
if the $(p,n)$-observability operator $\cO_{p,n; C, \bA}$ defined in \eqref{march21} is
bounded from $\cX$ into $H^{2}_{\bo_{p,n}, \cY}(\free)$. When this is the case, it makes sense
to introduce the {\em $(p,n)$-observability gramian}
$$
\cG_{p,n; C, \bA}: = \cO_{p,n;C,\bA}^{*} \cO_{p,n;C,\bA} \in
\cL(\cX),
$$
whose representation in terms of a strongly convergent series
\begin{equation}
\cG_{p,n;C, \bA} = R_{p,n}(B_{\bA})[C^{*}C] = \sum_{\alpha \in
\free} \omega_{p,n; \alpha}^{-1} \bA^{* \alpha^{\top}} C^{*} C\bA^{\alpha}
\label{rest5}
\end{equation}
follows from \eqref{Rpn} and \eqref{18.1p} and suggests to let $\cG_{p,0;C, \bA} = C^{*}C$.

In addition we say that the output pair $C, \bA)$ is {\em $(p,n)$-observable} if the observability operator
$\cO_{p,n, C, \bA}$ has trivial kernel.  If moreover, the observability operator
$\cO_{p,n, C, \bA}$ is bounded below (equivalently, the gramian $\cG_{p,n, C, \bA}$ is strictly positive definite), we say that
$(C, \bA)$ is {\em exactly observable}.

\smallskip

{\rm (2)} $\bA = (A_1, \dots, A_d)$ is {\em $p$-strongly stable} if it holds that
 \begin{equation} \label{bAstable-p}
     \lim_{N \to \infty} p(B_{\bA})^{N}[I_{\cX}] = 0 \; \text{ in the
     strong operator topology.}
 \end{equation}
Since $p(B_{\bA})^{N}[I_{\cX}]\succeq 0$ for all $N\ge 0$, the convergence
 in the strong operator topology is equivalent to convergence in the
 weak operator topology in \eqref{bAstable-p}.

\smallskip

{\rm (3)} $(C, \bA)$ is {\em $(p,n)$-contractive} if $\bA$ is $(p,n)$-hypercontractive (as in Definition \ref{D:10.1}) and in addition 
$\Gamma_{p,n;\bA}[I_{\cX}] \succeq C^{*}C$, where the $(p,n)$-defect operator $\Gamma_{p,n; \bA}$ is given by \eqref{march26}.

\smallskip

{\rm (4)} $(C, \bA)$ is {\em $(p,n)$-isometric} if $\bA$ is $(p,n)$-hypercontractive and in addition 
$\Gamma_{p,n; \bA}[I_{\cX}]=C^{*}C$.
\end{definition}

 \begin{proposition}  \label{P:identities-p}
Let $\Gamma_{p,n; \bA}$ be the operatorial map defined in \eqref{march26}
(defined as long as $\| B_\bA \| < \rho$).
Then for all $H \in \cL(\cX)$ and any integers $k \ge
1$ and $N \ge 0$,
\begin{align}
  &  \Gamma_{p,k; \bA}[H] =  \Gamma_{p,k-1; \bA}[H] - p(B_{\bA})[
    \Gamma_{p,k-1; \bA}[H]],  \label{4.8-p} \\
& H = \bigg( \sum_{j=0}^{N} \bcs{ k+j-1 \\ j } p^{j}(B_{\bA}) \Gamma_{p,k;\bA}
 + \sum_{j=1}^{k} \bcs{ N+k \\ N+j } p(B_{\bA})^{N+j}
\Gamma_{p,k-j;\bA} \bigg)[H]. \label{4.9-p}
 \end{align}
 Furthermore, if $(C, \bA)$ is $(p,n)$-output stable, then it is also
 $(p,k)$-output stable, and the
  equalities 
 \begin{equation}
\Gamma_{p,k; \bA}[\cG_{p,n;, C, \bA}] = \cG_{p,n-k; C, \bA}
  \;  \text{ for } \;  k=0,1, \dots, n. \label{4.10-p}
 \end{equation}
 hold, as well as the chain of inequalities:
 \begin{equation}   \label{chain-ineq-p}
 C^* C \preceq \cG_{p,1;C, \bA} \preceq \cdots \preceq \cG_{p,n; C, \bA}.
 \end{equation}
   \end{proposition}
   
   \begin{proof}
   Formula \eqref{4.8-p} follows from substitution of the operator argument $H$ into the identity
$$
  (1 - p)^{k} = (1 - p)^{k-1} - p (1 - p)^{k-1}.
$$
By substituting $p(B_{\bA})$ for $B_{\bA}$ and $(1-p)^{k}(B_{\bA})$ 
for $(I - B_{\bA})^{k}$ in the derivation of \eqref{4.11'} we arrive 
at the following $p$-analog of \eqref{4.11'}:
$$
I_{\cL(\cX)} = \sum_{j=0}^{N} \bcs{ k+j -1 \\ j } p^{j}(B_{\bA}) (1 - 
p)^{k}(B_{\bA}) + \sum_{j=1}^{k} \bcs{ N+k \\ N+j } p^{N+j}(B_{\bA}) 
(1 - p)^{k-j}(B_{\bA}).
$$
Applying this operator identity to $H$ leads to \eqref{4.9-p}.

\smallskip

To verify \eqref{4.10-p}, we shall make use of the following extension of 
\eqref{GenPrin} which we shall call the \textbf{Free Noncommutative Mertens Theorem}:

\smallskip
\noindent
{\em If the series ${\displaystyle\sum_{\alpha\in\free}{\bf a}_\alpha}$ converges absolutely and if $\; {\displaystyle\sum_{\alpha\in\free}{\bf a}_\alpha}={\bf a}\; $
and $\; {\displaystyle\sum_{\alpha\in\free}{\bf b}_\alpha}={\bf b}$, then }
\begin{equation}
\sum_{\alpha\in\free}\bigg(\sum_{\beta, \gamma \in \free \colon \beta \gamma = \alpha}{\bf a}_\beta {\bf b}_{\gamma}\bigg)=
\bigg(\sum_{\alpha\in\free}{\bf a}_\alpha\bigg)\cdot
\bigg(\sum_{\alpha\in\free}{\bf b}_\alpha\bigg) ={\bf a}\cdot{\bf b}.
\label{ncGenPrin}
\end{equation}
The proof follows by an adaptation of the proof of \cite[Theorem 3.50]{Rudin}, the details of which we leave to the reader for lack of space.

\smallskip

We apply this Mertens theorem 
as follows.  As $(C, \bA)$ is $(p,n)$-output stable, the series representation for
$\cG_{p,n; C, \bA} = (1 -p)^{-n}(B_\bA)[C^*C]$ is convergent.  By the assumption that $\| \begin{bmatrix} A_1 & \cdots & A_d \end{bmatrix} \|
< \rho$, the series representation for $\Gamma_{p,k; \bA} = (1 - p)^k(B_\bA)$ is absolutely convergent.  Hence by the free noncommutative
general principle \eqref{ncGenPrin}, we can compute
\begin{align*}
\Gamma_{p,k,\bA}[\cG_{p,n; C, \bA}] & = (1 - p)^k(B_\bA)[ (1-p)^{-n} (B_\bA)[C^* C]]  \\
& = \left( (1 - p)^k \cdot (1 - p)^{-n} \right) (B_\bA) [C^* C]  \\
& = (1 - p)^{-(n-k)}(B_\bA) [C^* C] =: \cG_{p,n; C, \bA}.
\end{align*}
In particular it falls out that $(C, \bA)$ being $(p,n)$-output stable implies that $(C, \bA)$ is $(p,k)$-output stable for $1 \le k \le n$.
The special case $k=1$ in \eqref{4.10-p} tells as that
\begin{align*}
\cG_{p,n-1; C, \bA} & = \Gamma_{p,1; \bA}[\cG_{p,n; C, \bA}]  \\
& = (1 - p)(B_\bA)[\cG_{p,n;C, \bA}] \\
& = \cG_{p,n; C, \bA} - p(B_\bA)[\cG_{p,n; C, \bA}] \preceq \cG_{p,n;C, \bA}
\end{align*}
since $p(B_\bA)$ is a positive map.  Iteration of this argument then gives us the chain of inequalities \eqref{chain-ineq-p}.  As in
Proposition \ref{P:identities} for the special case $p(z) = z_1 + \cdots + z_d$,  the fact that $(p,n)$-output stability implies $(p,k)$-output
stability for $k = 0,1,\dots, n-1$ can also be seen as a consequence of the chain of inequalities \eqref{chain-ineq-p}.  Finally one can also
see these inequalities as a consequence of verifying directly from the formula \eqref{omegapnalpha} the coefficient inequalities 
$\omega^{-1}_{p,k-1,\alpha} \le  \omega^{-1}_{p, k, \alpha}$ for $k=1, \dots, n$.
\end{proof}
\noindent
We next verify that contractivity implies hyper-contractivity for the $(p,n)$-setting.

\begin{lemma}
Let  us assume that operators $H$ and $A_j$ in $\cL(\cX)$ are such that
\begin{equation}
H \succeq p(B_{\bA})[H]\succeq
0\quad\mbox{and}\quad\Gamma_{p,n;\bA}[H]\succeq 0
\label{4.12-p}
\end{equation}
for some integer $n\ge 3$. Then
$\Gamma_{p,k;\bA}[H]\succeq 0$ for all $k=1,\ldots,n-1$.
\label{L:squeeze-p}
\end{lemma}

\begin{proof}
Since the map $X\mapsto p(B_{\bA})[X]$ is positive, iterating the first condition in \eqref{4.12-p}
gives
\begin{equation}
H\succeq p^j(B_{\bA})[H]\quad\mbox{for all}\quad j\ge 0.
\label{4.13-p}
\end{equation}
Then we introduce Hermitian operators
$$
S_{m,k}:=p^k(1-p)^m(B_{\bA})[H]\quad\mbox{for $k\in{\mathbb Z}_+\; $ and $\; m=0,1,\ldots,n$}.
$$
and verify inequalities \eqref{4.16} for these operators. Indeed, by \eqref{4.13-p},
\begin{align}
S_{m,k}& = p^k(1-p)^m(B_{\bA})[H]\notag\\
& = \sum_{j=0}^m   (-1)^j    \bcs{m \\ j}p^{k+j}(B_{\bA})[H]
\preceq \sum_{j=0}^m \bcs{m \\ j}H=2^m \cdot H,\label{4.17-pa}
\end{align}
thus proving the right inequality in \eqref{4.16}. The left inequality follows
by a similar argument using lower estimates. We next observe that due to \eqref{4.8-p} and the second inequality in
\eqref{4.12-p},
$$
\Gamma_{p,n-1;\bA}[H] \succeq p(B_{\bA})[\Gamma_{p,n-1;\bA}[H]].
$$
Since the map $X\mapsto p(B_{\bA})[X]$ is positive, we then have
$$
p^j(B_{\bA})[\Gamma_{p,n-1;\bA}[H]]\succeq p^{j+1}(B_{\bA})[\Gamma_{p,n-1;\bA}[H]],
$$
which, on account of \eqref{4.17-p} can be written as
\begin{equation}
S_{n-1,j}\succeq S_{n-1,j+1}.
\label{4.17-p}
\end{equation}
Also, it follows from definitions \eqref{4.17-p} that for any  $N\ge 1$,
\begin{align}
\sum_{j=0}^N S_{n-1,j}&=
\sum_{j=0}^N p^j(1-p)^{n-1}(B_\bA)[H]\notag\\
&=(1-p^{N+1})(1-p)^{n-2}(B_\bA)[H]=S_{n-2,0}-S_{n-2,N+1}.\notag
\end{align}
Using the same arguments as in the proof of Lemma \ref{L:squeeze} we conclude that
the operator $S_{n-1,0}=(1-p)^{n-1}(B_{\bA})[H]=\Gamma_{p,n-1;\bA}[H]$ is positive semidefinite.
We then obtain recursively that $\Gamma_{p,k;\bA}[H]\succeq 0$ for all $k=1,\ldots, n-1$.
\end{proof}

Note that for $H=I_{\cX}$, the left condition in \eqref{4.12-p} amounts to $I_{\cX}\succeq p(B_{\bA}[I_{\cX}])$ or
equivalently, to 
$$
(1-p)(B_{\bA}[I_{\cX}])=\Gamma_{p,1;\bA}[I_{\cX}] \succeq 0.
$$ 
Hence, specializing Lemma \ref{L:squeeze-p} to the case where $H=I_{\cX}$ and recalling Definition \ref{D:10.1}
we arrive at the following result.

\begin{corollary} \label{C:10.2}
If the operator tuple $\bA=(A_1,\ldots,A_d)$ is $(p,1)$-contractive and $(p,n)$-contractive for some $n\ge 3$, then it is 
also $(p,n)$-hypercontractive.
\end{corollary}

We next present the $(p,n)$-analogue of Theorem \ref{T:7.2} and Theorem \ref{T:2-1.1}.

  \begin{theorem}  \label{T:2-1.1-p}
      Let $C \in \cL(\cX, \cY)$ and $\bA = (A_{1}, \dots, A_{d}) \in
      \cL(\cX)^{d}$.  Then:
  \begin{enumerate}
      \item The pair $(C, \bA)$ is $(p,n)$-output stable if and only
      if there exists an operator $H \in \cL(\cX)$ satisfying the
      inequalities
      \begin{equation} \label{4.20-p}
          H \succeq p(B_{\bA})[H] \succeq 0 \quad\text{and}\quad \Gamma_{p,n;
          \bA}[H] \succeq C^{*}C.
\end{equation}

\item If $(C, \bA)$ is $(p,n)$-output stable, then the observability
gramian $H = \cG_{p,n;C,\bA}$ satisfies
\begin{equation} \label{4.21-p}
    H \succeq p(B_{\bA})[H] \succeq 0 \quad\text{and}\quad \Gamma_{p,n;
    \bA}[H] = C^{*}C
\end{equation}
and is the minimal positive semidefinite solution of the system \eqref{4.20-p}.

\item
There is a unique solution $H$ of the system \eqref{4.21-p} with $H =
\cG_{p,n; C, \bA}$ if $\bA$ is $p$-strongly stable.  Moreover, in
case $\bA$ is $p$-contractive in the sense that $p(B_{\bA})[I_{\cX}]
\preceq I_{\cX}$, then the solution of the $(p,n)$-Stein equation in
\eqref{4.21-p} is unique if and only if $\bA$ is $p$-strongly stable.
 \end{enumerate}
  \end{theorem}
  
  Note that  \eqref{4.20-p} and \eqref{4.21-p} holding with $H = I_\cX$ is just the definition of
  the output pair $(C, \bA)$ being $(p,n)$-contractive or $(p,n)$-isometric, respectively
  (see Definition \ref{D:def-pn}).

  \begin{proof}
If $(C,\bA)$ is $(p,n)$-output stable, then the series in \eqref{rest5} converges in the strong operator topology to 
the operator $H=\cG_{p,n; C, \bA}\succeq 0$. By formulas \eqref{march26} and \eqref{4.10-p} (with $k=1$), we have
\begin{align*}
\cG_{p,n; C, \bA}-p(B_{\bA})[\cG_{p,n; C, \bA}]&=(1-p)(B_{\bA})[\cG_{p,n; C, \bA}]\\
&=\Gamma_{p,1;\bA}[\cG_{p,n; C, \bA}]=\cG_{p,n-1; C, \bA}\succeq 0.
\end{align*}
By formula  \eqref{4.10-p} (for $k=n$), we have 
$$
\Gamma_{p,n;\bA}[\cG_{p,n; C, \bA}] =\cG_{p,0; C, \bA}=C^*C.
$$
Thus, the operator $H=\cG_{p,n; C, \bA}\succeq 0$ satisfies relations \eqref{4.21-p} and hence also inequalities \eqref{4.20-p}.

\smallskip

Suppose now that $H$ is any solution of the inequalities \eqref{4.20-p}.
Let us note that the analog of \eqref{4.9}
(based on identity \eqref{4.9-p} and Lemma \ref{L:squeeze-p} rather than
identity \eqref{4.9} and Lemma \ref{L:squeeze}) is
  \begin{align}
   &   \sum_{j=0}^{N} \bcs{ n+j-1 \\ j} p^{j}(B_{\bA})[C^{*}C]
       \preceq \sum_{j=0}^{N} \bcs{n+j-1 \\ j} p^{j}(B_{\bA})
      \Gamma_{p,n;\bA}[H] \notag \\
      &  \quad = H - \sum_{j=1}^{n} \bcs{ N+n \\ N+j} p^{N+j}(B_{\bA})
      [\Gamma_{p,n-j;\bA}[H] ] \preceq H. \label{4.24-p}
 \end{align}
 Thus the nondecreasing operator sequence of positive semidefinite
 operators
 $$
  S_{N} = \sum_{j=0}^{N} \bcs{ n+j-1 \\ j} p(B_{\bA})^{j}[C^{*}C], \quad
  N=1,2, \dots
 $$
 is bounded above (by $H$)
and hence converges strongly to a positive semidefinite operator. From the computation
 \eqref{Rpn} we see that the limit of this series is exactly
 $R_{p,n}(\bA)[C^{*}C] = \cG_{p,n; C, \bA}$. Passing to the limit in
 \eqref{4.24-p} as $N \to \infty$ now gives $\cG_{p,n;C, \bA} \preceq H$. 
In particular, $\cG_{p,n;C, \bA}$ is bounded (since $H$ is) and
 therefore the pair $(C, \bA)$ is $(p,n)$-output stable. Besides, as $H$ is an arbitrary positive semidefinite 
solution to the system \eqref{4.20-p}, the latter inequality tells us that $\cG_{p,n;C, \bA}$ is the minimal 
such solution. This completes the proof of parts (1) and (2).

\smallskip

Now suppose that $\bA$ is $p$-strongly stable. and that $H$ solves
the system \eqref{4.20-p}.  We shall show that necessarily $H =
 \cG_{p,n; C, \bA}$. We first observe that if positive semidefinite operators 
$P,Q\in\cL(\cX)$ satisfy the Stein equation 
\begin{equation}
   \Gamma_{p,1; \bA}[P]: = P - p(B_{\bA})[P] = Q,
\label{last1}
\end{equation}
then $P$ is uniquely recovered from \eqref{last1} via 
the strongly convergent series
\begin{equation}
    P = \sum_{j=0}^{\infty} p^j(B_{\bA})[Q].
\label{last2}
\end{equation}
Indeed, iterating \eqref{last1} gives 
\begin{equation}
P=\sum_{j=0}^N p^j(B_\bA)[Q]+p^{N+1}(B_\bA)[P],
\label{last3}
\end{equation}
and since all the terms in the latter equality are positive semidefinite, 
the strong convergence of the series on the right side of \eqref{last2} follows.
The $p$-strong stability of $\bA$ guarantees that $p^{N+1}(B_\bA)[P]$ tends to zero (strongly) as $N\to\infty$,
and then, upon letting $N\to\infty$ in \eqref{last3}, we arrive at \eqref{last2}.
In terms of the operator $B_\bA$, the latter uniqueness means that 
that the operator $I_{\cL(\cX)} - p(B_\bA)$ is invertible with inverse given by
$$
  ( I_{\cL(\cX)} - p(B_\bA))^{-1} \colon Q \mapsto \sum_{j=0}^\infty p^j(B_{\bA})[Q].
$$
Taking the $n$-the power of the latter operator and making use of the formula \eqref{Rpn} we get
$$
  (I - p(B_\bA))^{-n} [Q]=\sum_{j=0}^{\infty} \bcs{ n+j-1 \\ j } p^j(B_\bA)[Q]=
\sum_{\alpha \in \free} \omega^{-1}_{p,n; \alpha}\bA^{*\alpha^\top}Q\bA^\alpha 
$$
Next note that the Stein equation in \eqref{4.21-p} can be viewed as the equation
$$
   (I - p(B_\bA))^n[H] = C^* C.
 $$
 Hence 
$$
H = (I - B_\bA)^{-n}[C^*C]=\sum_{\alpha \in \free} \omega^{-1}_{p,n; \alpha}\bA^{*\alpha^\top}Q\bA^\alpha
$$
necessarily is the unique solution.  As $\cG_{p,n: C, \bA}$ is given by the same formula  \eqref{rest5}
 we conclude that $\cG_{p,n; C, \bA}$ is the unique solution.

\smallskip

It remains to  show that if $\bA$ is a $p$-contraction and the system \eqref{4.21-p}  admits a unique solution 
(which necessarily is $H={\mathcal G}_{p,n;C,\bA}$), then the tuple $\bA$ is $p$-strongly stable. As in the 
proof of Theorem \ref{T:2-1.1}, we prove the contrapositive: {\em if $\bA$ is not $p$-strongly stable, then 
the solution of \eqref{4.21-p} is not unique.}

\smallskip

Due to assumption $p(B_{\bA})[I_{\cX}]\preceq I_{\cX}$, the sequence of operators
\begin{equation}
\Delta_{N} = p^N(B_{\bA})[I_{\cX}]=
\sum_{\beta_{1},\dots, \beta_{N} \ne \emptyset \colon \beta_{1}\cdots
\beta_{N} = \alpha} p_{\beta_{1}} \cdots p_{\beta_{N}} 
{\mathbf A}^{*\alpha^{\top}}{\mathbf A}^{\alpha},\qquad N=1,2, \dots
\label{last5}
\end{equation}
is decreasing and therefore has a strong limit $\Delta={\displaystyle\lim_{N \to\infty} \Delta_{N}}\succeq 0$.
Since $\bA$ is assumed not to be $p$-strongly stable, this limit  $\Delta$ is not zero. 
Let us assume for a moment that 
\begin{equation}
p(B_{\bA})[\Delta]=\Delta.
\label{last6}
\end{equation}
Then it follows from \eqref{march26} and \eqref{last5} that 
$$
\Gamma_{p,n;\bA}[\Delta]=(1-p)^{n}(B_{\bA})[\Delta]=\sum_{j=0}^{n} (-1)^{j} \bcs{n \\ j}p^j(B_\bA)[\Delta]
=\sum_{j=0}^k(-1)^j\bcs{k \\ j}\Delta=0,
$$
and therefore, the operator $H={\mathcal G}_{n,C,\bA}+\Delta$
(as well as ${\mathcal G}_{n,C,\bA}$) satisfies the system \eqref{4.21-p}
which therefore has more than one positive-semidefinite solution.

\smallskip

It remains to verify \eqref{last6}. To this end, we first observe that since $\Delta\preceq \Delta_N$ for all 
$N$ and since all the coefficients of $p$ are non-negative,
$$
\sum_{|\alpha|\le k}p_\alpha {\mathbf A}^{*\alpha^{\top}}\Delta{\mathbf A}^{\alpha}\preceq 
p(B_{\bA})[\Delta_N]=\Delta_{N+1}. 
$$
Letting $k,N\to\infty$ in the latter inequality we conclude that $p(B_{\bA})[\Delta]$ is well-defined and satisfies 
\begin{equation}
p(B_{\bA})[\Delta]\preceq \Delta.
\label{last9}
\end{equation}
We now fix a nonzero $x\in\cX$ and $\varepsilon>0$. Since the series representing 
$p(B_\bA)[I_{\cX}]\preceq I_{\cX}$ converges, we can find $k\ge 0$ such that 
$$
\sum_{\alpha\in\free:|\alpha|>k}p_\alpha\big\|{\mathbf A}^{\alpha}x\big\|^2_{\cX}<\varepsilon.
$$
Then for any operator $Q\in\cL(\cX)$ such that $0\preceq Q\preceq I_{\cX}$, we also have 
\begin{equation}
\sum_{\alpha\in\free:|\alpha|> k}p_\alpha\cdot \big\langle Q\bA^\alpha x,\, \bA^\alpha x\big\rangle_{\cX}\le
\sum_{\alpha\in\free:|\alpha|>k}p_\alpha\big\|{\mathbf A}^{\alpha}x\big\|^2_{\cX}<\varepsilon.
\label{last7}
\end{equation}
Since the sequence $\Delta_N$ converges to $\Delta$ (weakly or strongly), for every $\alpha$, we can find $n_\alpha$ such that 
$$
p_{\alpha}\cdot \big\langle(\Delta_n-\Delta)\bA^\alpha x,\, \bA^\alpha x\big\rangle_{\cX}<d^{-k}\varepsilon\quad \mbox{for all}\quad n>n_\alpha. 
$$
Then for any $n>\max(n_\alpha: \, |\alpha|\le k)$, we have 
\begin{equation}
\sum_{\alpha\in\free:|\alpha|\le k}p_\alpha\cdot \big\langle(\Delta_n-\Delta)\bA^\alpha x,\, \bA^\alpha x\big\rangle_{\cX}<\varepsilon.
\label{last8}
\end{equation}
Combining \eqref{last7} (with $Q=\Delta_n-\Delta$) and \eqref{last8} gives, for the same $n$ as above,
\begin{align*}
\big\langle p(B_\bA)[\Delta_n-\Delta]x, \, x\big\rangle_{\cX}&=
\sum_{\alpha\in\free:|\alpha|\le k}p_\alpha\cdot \big\langle(\Delta_n-\Delta)\bA^\alpha x,\, \bA^\alpha x\big\rangle_{\cX}\\
&\quad +\sum_{\alpha\in\free:|\alpha|> k}p_\alpha\cdot \big\langle(\Delta_n-\Delta)\bA^\alpha x,\, \bA^\alpha x\big\rangle_{\cX}<2\varepsilon.
\end{align*}
Therefore, for any fixed $x\in\cX$ and any $\varepsilon>0$, there is $n$ such that
\begin{align*}
\langle \Delta x, \, x\rangle\le \langle \Delta_{n+1} x, \, x\rangle&=\big\langle p(B_\bA)[\Delta_n]x, \, x\big\rangle\\
&=\big\langle p(B_\bA)[\Delta]x, \, x\big\rangle+
\big\langle p(B_\bA)[\Delta_n-\Delta]x, \, x\big\rangle\\
&\le \big\langle p(B_\bA)[\Delta]x, \, x\big\rangle+2\varepsilon.
\end{align*}
Letting $\varepsilon\to 0^+$ in the latter inequality we get $\langle \Delta x, \, x\rangle\le \langle p(B_\bA)[\Delta]x, \, x\big\rangle$.
By polarization, $\Delta\le p(B_\bA)[\Delta]$, which together with \eqref{last9} implies \eqref{last6}, thus completing the proof of the 
theorem.
 \end{proof}

The following analog of Theorem \ref{T:2-1.1} is a direct 
consequence of Definition \ref{D:def-pn} (part (3)) applied to the statements 
in Theorem \ref{T:2-1.1-p} with $H = I_\cX$.
 \begin{proposition}  \label{P:2-1.1-p}
{\rm (1)} Suppose that $(C, \bA)$ is a
         $(p,n)$-contractive pair.  Then $(C, \bA)$ is
         $(p,n)$-output-stable with $\cG_{p,n;C,\bA} \preceq I_{\cX}$
         and the gramian $\cG_{p,n;C, \bA}$ is the unique positive
         semidefinite solution of the system \eqref{4.21-p} if and
         only if $\bA$ is $p$-strongly stable.

\smallskip

{\rm (2)} Suppose that $(C, \bA)$ is a $(p,n)$-isometric pair.
         Then $(C, \bA)$ is $(p,n)$-output-stable.  Moreover $H =
         I_{\cX}$ is the unique solution of the system \eqref{4.21-p}
         if and only if $\bA$ is $p$-strongly stable.  In this case
         $\cO_{p,n;C, \bA}$ is isometric and hence also $(C, \bA)$ is
         exactly $(p,n)$-observable.
\end{proposition}

\section{The $\bo_{p,n}$-shift model operator tuple $\bS_{\bo_{p,n},R}$ } \label{S:bo-model}

 We define the shift operator tuple
$\bS_{\bo_{p,n},R}$ on $H^{2}_{\bo_{p,n},\cY}(\free)$
as in \eqref{18.2b} and observe that it is not a row-contraction, in general. However, it is 
a $(p,1)$-contraction (see Proposition \ref{P:4.1-p} below), i.e., 
$\sum_{\alpha\in\free}p_\alpha {\mathbf S}^{\alpha^\top}_{\bo_{p,n},R}{\mathbf S}^{*\alpha}_{\bo_{p,n},R}\preceq I$. 
In particular, $p_j S_{\bo_{p,n},R,j}S_{\bo_{p,n},R,j}^*\preceq I$ where $p_{j}$ is the coefficient of 
$z_{j}$ in the 
formal power series representation \eqref{apr12}), and therefore,
$$
\| S_{\bo_{p,n},R,j} \| \le p_{j}^{-\frac{1}{2}} < \infty\quad\mbox{for}\quad 1 \le j \le d.
$$
The formula for $S_{\bo_{p,n},R,j}^{*}$ is again given by 
\eqref{Sboj*} 
\begin{equation}
S_{\bo_{p,n},R,j}^*: \; \sum_{\gamma \in\free}f_\gamma z^\gamma \mapsto
        \sum_{\gamma \in\free}\frac{\omega_{p,n;\gamma j}}{\omega_{p,n;\gamma}} \,
        f_{\gamma j}z^\gamma\quad\mbox{for}\quad j=1,\ldots,d
\label{5.2p}
\end{equation}
(with $\bo_{p,n}$ in place of $\bo$) while its iteration gives 
\begin{equation}
(\bS_{\bo_{p,n},R}^*)^{\alpha}: \; \sum_{\gamma \in\free}f_\gamma z^\gamma \mapsto
        \sum_{\gamma \in\free}\frac{\omega_{p,n;\gamma\alpha}}{\omega_{p,n;\gamma}} \,
        f_{\gamma\alpha}z^{\gamma}\quad\mbox{for all}\quad\alpha\in\free.
\label{SRLv-p}
\end{equation}

\begin{proposition}  \label{P:4.1-p}
The operator tuple  ${\mathbf 
S}^*_{\bo_{p,n},R}=(S_{\bo_{p,n},R,1}^{*},\ldots,S_{\bo_{p,n},R,d}^{*})$ is a 
$p$-strongly stable
$(p,n)$-hypercontraction. Furthermore, the model pair $(E,{\mathbf S}^*_{\bo_{p,n},R})$
is $(p,n)$-isometric, where $E$ is the operator defined by
$Ef=f_{\emptyset}$ for $f\in H^2_{\bo_{p,n},\cY}(\free)$. 
\end{proposition}
\begin{proof}
Since $p_\beta\ge 0$ for all $\beta\in\free$, we have for any $\alpha,\gamma\in\free$ and 
$N\le |\alpha|$,
\begin{align*}
&\sum_{i=0}^{|\alpha|+|\gamma|}\bcs{n+i-1 \\ n-1}\sum_{\phi_{1},
    \dots, \phi_{i} \ne \emptyset \colon \phi_{1} \cdots \phi_{i}
    = \gamma\alpha}p_{\phi_{1}} \cdots p_{\phi_{i}}\\
&\ge \sum_{j=0}^{|\gamma|}\bcs{n+N+j-1 \\ n-1}\sum_{{\scriptsize\begin{array}{c}
\phi_{1},\dots, \phi_{j} \ne \emptyset \colon \phi_{1} \cdots \phi_{j}= \gamma,\\
\beta_{1},\dots,\beta_N\ne \emptyset \colon \beta_{1}\cdots\beta_N=\alpha\end{array}}}
p_{\phi_{1}} \cdots p_{\phi_{j}}p_{\beta_{1}} \cdots p_{\beta_{N}}\\
&=\bigg(\sum_{j=0}^{|\gamma|}\bcs{n+N+j-1 \\ n-1}\sum_{{\scriptsize\begin{array}{c}
\phi_{1},\dots, \phi_{j} \ne \emptyset \colon\\
 \phi_{1} \cdots \phi_{j}= \gamma\end{array}}}p_{\phi_{1}} \cdots p_{\phi_{j}}
\bigg)\cdot \sum_{{\scriptsize\begin{array}{c}
\beta_{1},\dots,\beta_N\ne \emptyset \colon \\
\beta_{1}\cdots\beta_N=\alpha\end{array}}}
p_{\beta_{1}} \cdots p_{\beta_{N}}\\
&\ge \bigg(\sum_{j=0}^{|\gamma|}\bcs{n+j-1 \\ n-1}\sum_{{\scriptsize\begin{array}{c}
\phi_{1},\dots, \phi_{j} \ne \emptyset \colon\\ 
 \phi_{1} \cdots \phi_{j}= \gamma\end{array}}}p_{\phi_{1}} \cdots p_{\phi_{j}}
\bigg)\cdot \sum_{{\scriptsize\begin{array}{c}
\beta_{1},\dots,\beta_N\ne \emptyset \colon \\
\beta_{1}\cdots\beta_N=\alpha\end{array}}} 
p_{\beta_{1}} \cdots p_{\beta_{N}},
\end{align*}
which, upon making use of notation \eqref{omegapnalpha}, can be written as 
\begin{equation}
\omega^{-1}_{p,n;\gamma\alpha}\ge \omega^{-1}_{p,n;\gamma}\cdot \sum_{
\beta_{1},\dots,\beta_N\ne \emptyset \colon 
\beta_{1}\cdots\beta_N=\alpha}p_{\beta_{1}} \cdots p_{\beta_{N}};
\label{march24}
\end{equation}
Note that this condition may be viewed as the more general 
$p$-version of the inequality $1 \le \frac{\omega_{j}}{\omega_{j+1}}$ 
appearing in assumption \eqref{18.2} for the weighted case discussed 
in Chapter \ref{S:Hardy}. 

\smallskip

For an arbitrary element $f\in H^2_{\bo_{p,n},\cY}(\free)$, 
the series  representing $\|f\|^2_{H^2_{\bo_{p,n},\cY}(\free)}$ converges and hence,
\begin{equation}
\label{5.7-p}
\lim_{N \to \infty} \sum_{\gamma \in \free \colon |\gamma|\ge
N}\omega_{p,n;\gamma}\|f_\gamma\|^2_{\cY}=0,\quad\mbox{if}\quad 
f(z) = {\displaystyle\sum_{\gamma \in \free} f_{\gamma} z^{\gamma}}.
\end{equation}
Making use of \eqref{SRLv-p}, \eqref{18.1p} and \eqref{march24}, we have
\begin{align}
&\left\langle p^N(B_{{\mathbf S}^*_{\bo_{p,n},R}})[I_{H^2_{\bo_{p,n},\cY}(\free)}]
f, \; f\right\rangle_{H^2_{\bo_{p,n},\cY}(\free)}\notag\\
&=\sum_{\alpha\in\free}\bigg(\sum_{\beta_{1},
    \dots, \beta_{N} \ne \emptyset \colon \beta_{1} \cdots \beta_{N}
    = \alpha} p_{\beta_{1}} \cdots p_{\beta_{N}} \bigg)
\|S_{\bo_{p,n},R}^{*\alpha}f\|^2_{H^2_{\bo_{p,n},\cY}(\free)}\notag\\
&=\sum_{\alpha,\gamma\in\free}\bigg(\sum_{\beta_{1},
    \dots, \beta_{N} \ne \emptyset \colon \beta_{1} \cdots \beta_{N}
    = \alpha}p_{\beta_{1}} \cdots p_{\beta_{N}} \bigg)
\frac{\omega^2_{p,n;\gamma\alpha}}{\omega_{p,n;\gamma}}\|f_{\gamma\alpha}\|^2_{\cY}\notag\\
&\le \sum_{\alpha,\gamma\in\free: |\alpha|\ge N}\omega_{p,n;\gamma\alpha}\|f_{\gamma\alpha}\|^2_{\cY}
\le \sum_{\beta\in\free \colon |\beta|\ge N}\omega_{p,n;\beta}\|f_{\beta}\|^2_{\cY}.\notag
\end{align}
Combining the latter estimate with \eqref{5.7-p} we conclude
that $p^N(B_{{\mathbf S}^*_{\bo_{p,n},R}})[I_{H^2_{\bo_{p,n},\cY}(\free)}]$ tends to
zero in the strong operator topology meaning that the operator tuple
${\mathbf S}^*_{\bo_{p,n},R}$ is $p$-strongly stable.

\smallskip

We next make use of \eqref{march26} and \eqref{SRLv-p} to compute for the same 
$f$ as in \eqref{5.7-p},
\begin{align}                       
&\left\langle \Gamma_{p,k;{\mathbf S}^*_{\bo_{p,n},R}}[I_{H^2_{\bo_{p,n},\cY}(\free)}]f, \,
f\right\rangle_{H^2_{\bo_{p,n},\cY}(\free)}\notag\\
&=\sum_{\alpha \in\free}c_{p,k;\alpha}
\| {\mathbf  S}_{\bo_{p,n},R}^{* \alpha} f\|^{2}_{H^2_{\bo_{p,n},\cY}(\free)}\notag\\
&=\sum_{\alpha \in\free}c_{p,k;\alpha}\sum_{\gamma\in\free}
\frac{\omega^2_{p,n;\gamma\alpha}}{\omega_{p,n;\gamma}}\|f_{\gamma\alpha}\|^2_{\cY}\notag\\
&=\sum_{\gamma\in\free}\bigg(\sum_{\alpha,\beta\in\free\colon\beta\alpha=\gamma}
c_{p,k;\alpha}\omega_{p,n;\beta}^{-1}\bigg)\omega^2_{p,n;\gamma}\|f_{\gamma}\|^2_{\cY}.\label{march30}
\end{align}
for all $f \in H^2_{\bo_{p,n},\cY}(\free)$ and $k=1,\ldots,n$.\footnote{If $f$ is a polynomial, then the 
sums in \eqref{march30} are finite, which justifies rearrangements used in that calculation. 
Then the general case follows by approximation arguments.}

\smallskip
 
Equating the coefficients of $z^{\gamma}$ in the power series identity
$$
(1-p(z))^{-n}(1-p(z))^k=(1-p(z))^{-(n-k)}
$$
for a fixed $k\in\{1,\ldots,n\}$, we conclude, on account of 
\eqref{march27}, \eqref{Rpn} and \eqref{18.1p} that 
$$
\sum_{\alpha,\beta\in\free:\beta\alpha=\gamma}\omega_{p,n;\beta}^{-1}c_{p,k;\alpha}
=\omega_{p,n-k;\gamma}^{-1}\quad\mbox{for}\quad k=0,\ldots,n-1,
$$
and that 
$$
\sum_{\alpha,\beta\in\free:\beta\alpha=\gamma}\omega_{p,n;\beta}^{-1}c_{p,n;\alpha}
 = \left\{\begin{array}{ccl} 1 &\mbox{if} & |\gamma|=0,\\
0 &\mbox{if} & |\gamma|>0.\end{array}\right.
$$
Combining the two latter formulas with \eqref{march30} leads us to equalities
$$
\left\langle \Gamma_{p,k;{\mathbf S}^*_{\bo_{p,n},R}}[I_{H^2_{\bo_{p,n},\cY}(\free)}]f, \,
f\right\rangle_{H^2_{\bo_{p,n},\cY}(\free)}=
\sum_{\gamma\in\free}\frac{\omega^2_{p,n;\gamma}}{\omega_{p,n-k;\gamma}}
\|f_{\gamma}\|^2_{\cY}
$$
for all $k=1,\ldots,n-1$, and also
\begin{equation}
\left\langle \Gamma_{p,n;{\mathbf S}^*_{\bo_{p,n},R}}[I_{H^2_{\bo_{p,n},\cY}(\free)}]f, \,
f\right\rangle_{H^2_{\bo_{p,n},\cY}(\free)}=\|f_\emptyset\|_{\cY}^2=\|Ef\|^2_{\cY}.
\label{march31}
\end{equation}
Since the latter relations hold for all $f\in H^2_{\bo_{p,n},\cY}(\free)$, it follows that 
$$
\Gamma_{p,k;{\mathbf S}^*_{\bo_{p,n},R}}[I_{H^2_{\bo_{p,n},\cY}(\free)}]\succeq 0\quad\mbox{for all}
\quad k=1,\ldots,n,
$$ 
and in addition, 
$\Gamma_{p,n;{\mathbf S}^*_{\bo_{p,n},R}}[I_{H^2_{\bo_{p,n},\cY}(\free)}]=E^*E$. Hence, the 
tuple ${\mathbf S}^*_{\bo_{p,n},R}$ is $(p,n)$-hypercontractive and the pair 
$(E,{\mathbf S}^*_{\bo_{p,n},R})$ is $(p,n)$-isometric.
\end{proof}

\begin{remark}  \label{R:GpnES=I}  The results of Proposition \ref{P:4.1-p} tell us that $\bS_{p,n;R}^*$ is $p$-strongly stable
and that the output-pair $(E, \bS_{p,n;R}^*)$ is $(p,n)$-isometric.  Hence, as a consequence of part (2) of Proposition \ref{P:2-1.1-p}
it follows that $\cO_{p,n; E, \bS_{p,n;R}^*}$ is isometric, i.e., 
$$
(\cO_{p,n; E, \bS_{p,n;R}})^*   \cO_{p,n; E, \bS_{p,n;R}} = I_{H^2_{\bo_{p,n}, \cY}(\free)}.
$$
In fact, as a simple consequence of the formula \eqref{SRLv-p} for
$(\bS_{\bo_{p,n}, R})^\alpha$ and the definition \eqref{march21} of $\cO_{p,n; C, \bA}$ (applied with $(C, \bA) = (E, \bS_{p,n; R}^*)$,
one can see that  in fact already 
$$
\cO_{p,n; E, \bS_{n,p; R}^*} = I_{H^2_{\bo_{p,n}, \cY}(\free)}.
$$
\end{remark}

Specifying Proposition \ref{P:4.1-p} to the linear case we arrive at the following result.

\begin{corollary}
For $\widetilde{p}(z)=p_1z_1 + \cdots + p_dz_d$ with $p_1,\ldots,p_d>0$, the operator tuple  
${\mathbf S}^*_{\bo_{\widetilde{p},1},R}$ is a $\widetilde{p}$-strongly stable
$(\widetilde{p},1)$-contraction on $H^2_{\bo_{\widetilde{p},1},\cY}(\free)$. Furthermore, the model pair $(E,{\mathbf S}^*_{\widetilde{p},R})$
is $(\widetilde{p},1)$-isometric.
\label{C:10.11}
\end{corollary}

\section{Observability operator range spaces in $H^2_{\bo_{p,n}, \cY}(\free)$}  \label{S:obsoprange-p}

Although it is possible to work out a $(p,n)$-analog of Theorem 
\ref{T:2-1.2}, for the sake of simplicity we present here only an abridged version of
its simplified version (the  $(p,n)$-analog of Theorem \ref{T:2-1.2'})  along with the $(p,n)$-version
of the converse statement  (part (4)) in Theorem \ref{T:2-1.2} as well as Theorem \ref{T:1.2g}.

\begin{theorem}  \label{T:2-1.2'-p}
Let $(C, {\mathbf A})$ be a $(p,n)$-contractive pair with state space ${\mathcal X}$ and output
         space $\cY$.  Then:
\smallskip

\noindent
 {\rm (1)} $(C, {\mathbf A})$ is $(p,n)$-output-stable and the intertwining relation
\begin{equation}
\label{4.8a-p}
S_{\bo_{p,n},R,j}^{*} {\mathcal O}_{p,n;C, {\mathbf A}}x ={\mathcal O}_{p,n;C, {\mathbf A}} A_{j} x\quad (x\in\cX)
\end{equation}
holds for all $j=1,\ldots,d$. Hence $\operatorname{Ran}\cO_{p,n; C,{\mathbf A}}$ is ${\bf S}_{\bo_{p,n},R}^*$-invariant.

\smallskip

\noindent
{\rm (2)} The operator ${\mathcal O}_{p,n;C,{\mathbf A}}$ is a
         contraction from ${\mathcal X}$ into $H^2_{\bo_{p,n},\cY}(\free)$.
         Moreover, ${\mathcal O}_{p,n;C, {\mathbf A}}$ is isometric if
         and only if $(C, {\mathbf A})$ is a $(p,n)$-isometric pair and ${\mathbf
         A}$ is $p$-strongly stable.

\smallskip

\noindent
{\rm (3)}  If the linear manifold ${\mathcal M}:= \operatorname{Ran}
         {\mathcal O}_{p,n;C,{\mathbf A}}$ is given the lifted norm
         \begin{equation}  \label{lifted'-p}
           \| {\mathcal O}_{p,n;C, {\mathbf A}}x\|_{{\mathcal M}} =
           \| Q x \|_{{\mathcal X}}
         \end{equation}
         where $Q$ is the orthogonal projection of ${\mathcal X}$ onto
         $(\operatorname{Ker}{\mathcal O}_{p,n;C, {\mathbf A}})^{\perp}$, then:
   \begin{enumerate}
   \item[(a)]
         ${\mathcal O}_{p,n;C, {\mathbf A}}$ is a coisometry from
         ${\mathcal X}$ onto ${\mathcal M}$ and implements a unitary equivalence between
         $\bS_{p,n; R}^*|_\cM$ and $Q \bA|_{\operatorname{Ran} Q}$.
         
         \item[(b)]
 ${\mathcal M}$ is contained contractively in $H^2_{\bo_{p,n},\cY}(\free)$ 
 $\cM$ is isometrically equal to the FNRKHS ${\mathcal H}(K_{p,n;C,{\mathbf A}})$ with
         reproducing kernel $K_{p,n;C,{\mathbf A}}(z,\zeta)$ given by
\begin{equation}   \label{apr1a}
K_{p,n;C, {\mathbf A}}(z,\zeta) = CR_{p,n}(Z(z) \bA) R_{p,n}(Z(\zeta) \bA)^*C^*.  
\end{equation}

\item[(c)] The pair $(E|_\cM, \bS_{n,p;R}^*|_\cM)$ is a $(p,n)$-contractive output pair, or explicitly
 (in view of Corollary \ref{C:10.2}),  ${\bf S}^*_{\bo_{p,n},R}|_\cM$ is $(p,1)$-contractive on $\cM$, i.e., 
\begin{align}
&  \langle \Gamma_{p,1; \bS^*_{p,n;R}|_\cM}[I_\cM] f, f \rangle_\cM 
= \|f \|^2_\cM -\sum_{\alpha\in\free}p_{\alpha}\cdot \|{\bf S}_{\bo_{p,n},R}^{*\alpha}f\|^2_{\cM} \ge 0,  \text{ or }  \notag \\
&  \langle  (\Gamma_{p,n; \bS_{p,n; R}^*|_\cM}[I_\cM] - E^* E) f, f \rangle_\cM =
\sum_{\alpha\in\free} c_{p,n;\alpha}\cdot
\|{\bf S}_{\bo_{p,n},R}^{*\alpha}f \|^{2}_{\cM} -  \|f_{\emptyset}\|^{2}_{\cY} \ge 0 
\label{dif-quot'-p}
\end{align}
for all  $f \in \cM$.
 Moreover, \eqref{dif-quot'-p} holds with equality
         if and only the orthogonal projection $Q$ of ${\mathcal X}$ onto
         $(\operatorname{Ker} {\mathcal O}_{p,n,C, {\mathbf A}})^{\perp}$
         is subject to relations
 $$
Q\succeq \sum_{\alpha\in\free}p_\alpha \bA^{*\alpha^\top}Q\bA^\alpha\quad\mbox{and}\quad
\Gamma_{p,n;\bA}[Q]=C^*C.
$$
         
 \item[(d)] 
In particular, if $(C, {\mathbf A})$ is observable, then $\cO_{p,n; C, \bA}$ implements a unitary equivalence between $\bA$ and
$\bS_{p,n; R}^*|_\cM$.  Furthermore,
 \eqref{dif-quot'-p} holds with
equality if and only if $(C, {\mathbf A})$ is a $(p,n)$-isometric pair.

\item[(e)]  If $\bA$ is $p$-strongly stable and $(C, \bA)$ is $(p,n)$-isometric, then $\cO_{p,n;  C, \bA}$ is isometric 
(and hence in particular $(C, \bA)$ is $(p,n)$-observable).
Hence $\cM$ is contained in $H^2_{p,n; \cY}(\free)$ isometrically and $\cO_{p,n; C, \bA}$ implements a unitary equivalence between
$\bA$ and $\bS_{p,n; R}^*|_{\cM}$, where $\cM$ is the orthogonal complement of a subspace $\cN = \cM^\perp$ isometrically included
in $H^2_{p,n, \cY}(\free)$ which is invariant under the forward shift tuple $\bS_{p,n; R}$.
\end{enumerate}

\smallskip

\noindent
{\rm (4)}
Conversely, let $\cM$ be a Hilbert space contractively included in $H^2_{\bo_{n,p}, \cY}(\free)$ (not necessarily isometrically or even
contractively)  such that
\begin{enumerate}
\item[(i)] $\cM$ is invariant under the backward shift $\bS_{\bo_{n,p}, R}^*$,

\item[(ii)] $(E|_\cM, \bS_{\bo_{n,p}, R}^*|_\cM)$ is a $(p,n)$-contractive output pair.
\end{enumerate}
Then it follows that $\cM$ is contractively included in $H^2_{\bo_{n,p}, \cY}(\free)$ and there exists an $(n,p)$-contractive pair $(C, \bA)$
so that
$ \cM =  \cH(K_{p,n; C, \bA}$ with $K_{p,n; C, \bA}$ as in \eqref{apr1a}.

If $\cM$ is isometrically included in $H^2_{p,n; \cY}(\free)$ and invariant under $\bS_{p,n; R}^*$, then $\cM$ has the form
$\cM = \operatorname{Ran} \cO_{p,n; C, \bA}$ with $(C, \bA)$ a $(p,n)$-isometric output pair and $\bA$ $p$-strongly stable.
In fact one can take $(C, \bA) = (E|\cM, \bS_{p,n; R}^*|_\cM)$.
\end{theorem}
         
\begin{proof} 
With regard to statement (1),  making use of the power series expansion \eqref{march21} and of
formula \eqref{5.2p} to get for any fixed $x\in\cX$,
            \begin{eqnarray}
S_{\bo_{p,n},R,j}^*\cO_{p,n,C,\bA}x&=&S_{\bo_{p,n},R,j}^*\bigg(\sum_{\gamma\in\free}
\left(\omega_{p,n;\gamma}^{-1}C\bA^\gamma x\right) \, z^\gamma \bigg)\nonumber\\
&=&\sum_{\gamma\in\free}\bigg(\frac{\omega_{p,n;\gamma j}}{\omega_{p,n;\gamma}}\cdot
\omega_{p,n;\gamma j}^{-1}C\bA^{\gamma j} x\bigg) \, z^\gamma \nonumber\\
            &=&\sum_{\gamma \in\free}\omega_{p,n;\gamma}^{-1}(C\bA^{\gamma j}x)z^\gamma=
           {\cO}_{p,n;C, \bA}A_j x,\nonumber
            \end{eqnarray}
and \eqref{4.8a-p} follows. Statement (2) is a direct consequence of Proposition \ref{P:2-1.1-p}.

Statement (3a) is a direct consequence of the intertwining relation \eqref{4.8a-p} and the definition of the $\cM$-norm \eqref{lifted'-p}.

Statement (3b) is a consequence of the definition of the norm \eqref{lifted'-p}; the calculation of the associated kernel \eqref{apr1a}
follows from Proposition \ref{P:principle}.

Statement (3c) is a consequence of the assumed condition that $(C, \bA)$ is a $(p,n)$-contractive pair and the definition \eqref{lifted'} of the
$\cM$-norm, analogous to the proof of (3b) in Theorem \ref{T:2-1.2}.

Statement (3d) is just the specialization (3a) to the case where $Q = I_\cX$.

Statement (3e) is a consequence of statement (2) in Proposition \ref{P:2-1.1-p}.

Statement (4) can be seen as a consequence of the previous statements combined with the observation of Remark \ref{R:GpnES=I} 
that $\cO_{E, \bS_{p,n;R}^*} = I_{H^2_{\bo_{p,n}, \cY}(\free)}$.
\end{proof}

\section{Beurling-Lax theorems: $(p,n)$-versions}
 \label{S:BL-p}
 
In this section we present Beurling-Lax representation theorems for 
shift-invariant subspaces of the space $H^{2}_{\bo_{p,n}, \cY}(\free)$ which can be 
considered as $(p,n)$-versions of the theorems presented in Sections \ref{S:NC-appl} and 
\ref{S:BL-quasi-wanderinga}. We start with the Beurling-Lax representation for contractively
included shift-invariant subspaces.

\begin{theorem} \label{T:NC-BL-p}
Let $\widetilde{p}$ be the linear part of a regular noncommutative formal power series $p$ as 
in \eqref{apr12}, \eqref{apr12-1}. A Hilbert space $\cM$ is such that
\begin{enumerate}
\item $\cM$ is contractively included in $H^2_{\bo_{p,n},\cY}(\free)$,
\item $\cM$ is ${\bf S}_{\bo_{p,n},R}$-invariant,
\item the $d$-tuple $\; {\bf T} =(T_1,\ldots,T_d)$ where $T_j=S_{\bo_{p,n},R,j}|_{\cM}$ for
$j=1,\ldots,d$, is $(\widetilde{p},1)$-contractive
\end{enumerate}
if and only  if there is a coefficient Hilbert space $\cU$ and a
contractive multiplier $\Theta$ from $H^2_{\bo_{\widetilde{p},1},\cU}(\free)$ to
$H^2_{\bo_{p,n},\cY}(\free)$ so that
\begin{equation}
\cM = \Theta \cdot H^2_{\bo_{\widetilde{p},1},\cU}(\free),\quad\mbox{with lifted norm}\quad 
\| \Theta \cdot f \|_{\cM} = \| Q f \|_{H^2_{\bo_{\widetilde{p},1},\cU}(\free)}
\label{Mform-p}
\end{equation}
where $Q$ is the orthogonal projection onto $(\operatorname{Ker} M_{\Theta})^\perp$.
          
If $\cM$ is represented as a pullback space $\cM = \cH^p(\Pi)$ for some positive semidefinite operator $\Pi \in \cL(H^2_{\bo_{p,n}, \cY}(\free))$,
then conditions (1), (2), (3) above can be expressed more succinctly simply as
\begin{equation}   \label{pullback-ineq-p}
  0 \preceq \Pi \preceq I_{H^2_{\bo_{p,n}, \cY}(\free)}, \quad \Pi - \sum_{j=1}^d p_j S_{\bo_{p,n},R} \Pi S_{\bo_{p,n}, R}^* \succeq 0.
\end{equation}
\end{theorem}
        
\begin{proof} If $\Theta$ is a contractive multiplier
from $H^2_{\bo_{\widetilde{p},1},\cU}(\free)$ to $H^2_{\bo_{p,n},\cY}(\free)$ and
$\cM$ is of the form \eqref{Mform-p}, then it follows as in \eqref{ineq1} that
$\|\Theta f\|_{H^2_{\bo_{p,n},\cY}(\free)}\le \|\Theta f\|_{\cM}$ which verifies property
(1). Property (2) follows from the intertwining equalities 
$$
S_{\bo_{p,n},R,j}M_{\Theta}= M_{\Theta}S_{\bo_{\widetilde{p},1},R,j}
$$ 
where $M_{\Theta}$ is multiplication by $\Theta$ on the left.
Furthermore, the formulas \eqref{9.2} and \eqref{9.4} still hold true (with ${\bf S}_{\widetilde{p},R}$ instead of
${\bf S}_{1,R}$). Since $Q$ is a projection and 
the tuple ${\bf S}_{\widetilde{p},R}$  is $(\widetilde{p},1)$-contractive on $H^2_{\widetilde{p},\cU}(\free)$ (by Corollary \ref{C:10.11}),
we can mimic the computation \eqref{again20} to get
$$
\sum_{j=1}^d p_j\| T_j^* M_\Theta f \|^2_\cM = \sum_{j=1}^d p_j\| Q S^*_{\widetilde{p},R,j} Q f \|^2_{H^2_{\widetilde{p},\cU}(\free)}
 \le \| Q f \|^2_{H^2_{\widetilde{p},\cU}(\free)}= \| M_\Theta f \|^2_\cM,
$$
which shows that $\bT$ is $(\widetilde{p},1)$-contractive on $\cM$ and completes the proof of sufficiency.

\smallskip

Conversely, let us assume that the Hilbert space $\cM$ satisfies conditions (1), (2), (3), 
i.e. (on account of Theorem \ref{T:charNFRKHS}), $\cM$ is a NFRKHS
contractively included in $H^2_{\bo_{p,n},\cY}(\free)$,
which is invariant under the right coordinate multipliers $T_j=S_{\bo, R, j}|_\cM$
($1\le j\le d$) and moreover, the tuple $\bT = (T_1, \dots, T_d)$ is $(\widetilde{p},1)$-contractive:
$$
\Gamma_{\widetilde{p},1,{\bf T}^*}[I_{\cM}]\succeq 0.
$$
Let $k_{\cM}(z,\zeta)$ denote the reproducing kernel for $\cM$. Computations similar to those in the proof of 
Proposition \ref{P:referee} show that the inequality 
$$
\left\langle \big( \Gamma_{\widetilde{p},1,{\bf T}^*}[I_{\cM}] \otimes I_{{\mathbb C} \langle \langle
\overline{\zeta} \rangle \rangle} \big)
k_\cM(\cdot, \zeta) y, \, k_\cM(\cdot, z) y' \right\rangle_{\cM \langle \langle
\overline{\zeta} \rangle \rangle \times \cM\langle \langle\overline{z} \rangle \rangle}\ge 0
$$
holding for all $y,y'\in\cY$ is equivalent to the kernel
\begin{equation}   \label{kernelL-p}
L(z,\zeta):=k_\cM(z, \zeta)-\sum_{j=1}^d p_j \bzeta_j k_\cM(z,  \zeta) z_j
\end{equation}
  being positive, that is, to having a Kolmogorov decomposition $L(z, \zeta) = G(z) G(\zeta)^{*}$
  for some $G\in \cL(\cU, \cY)\langle\langle z\rangle\rangle$.
  Write \eqref{kernelL-p} in the form
  $$
  k_{\cM}(z, \zeta) = G(z) G(\zeta)^{*} + \sum_{j=1}^d p_j \bzeta_j k_\cM(z,  \zeta) z_j
  $$
and then iterate the latter identity while making use of the notation \eqref{jan4bxa} to arrive at 
$$
k_{\cM}(z, \zeta)= \sum_{\alpha \colon |\alpha| \le N} d_{\widetilde p, \alpha} \overline{\zeta}^{\alpha^\top} G(z) G(\zeta)^* z^\alpha
+ \sum_{\alpha \colon |\alpha| = N+1} d_{\widetilde p, \alpha} \overline{\zeta}^{\alpha^\top} k_\cM(z, \zeta) z^\alpha
$$
for all $N=0,1,2,\dots$.  Taking the limit as $N \to \infty$ then gives us
\begin{align*}
k_{\cM}(z, \zeta) &=  \sum_{\alpha \in \free} d_{\widetilde p, \alpha} \overline{\zeta}^{\alpha^\top} G(z) G(\zeta)^* z^\alpha \\
&= G(z) \bigg( \sum_{\alpha \in \free} d_{\widetilde p,\alpha} z^\alpha \overline{\zeta}^{\alpha^\top} \bigg) G(\zeta)^*
=G(z)k_{\widetilde{p}}(z,\zeta)G(\zeta)^{*}.
\end{align*}
 By part (2) in Proposition \ref{P:2.3}. the latter identity tells us that $G$ is a coisometric multiplier from 
$H^2_{\widetilde{p},\cU}(\free)$ onto $\cM$. The rest of the proof is the same as in Theorem \ref{T:NC-BL}. We consider 
$\Theta = G$ as a multiplier from $H^2_{\bo_{\widetilde{p},1},\cU}(\free)$ into $H^2_{\bo_{p,n}, \cY}(\free)$. Since the inclusion map 
$\iota \colon \cM \to H^2_{\bo_{p,n}, \cY}(\free)$ is a contraction (by condition (1)), we have 
$\|M_\Theta\|=\|\iota M_G\|\le 1$, i.e.,  $\Theta$ is a contractive
 multiplier from $H^2_{\bo_{\widetilde p, 1},\cU}(\free)$ to $H^2_{\bo_{p,n}, \cY}(\free)$. Moreover, as 
 $M_G H^2_{\bo_{\widetilde{p},1},\cU}(\free) = \cM$, we have
$\Theta H^2_{\bo_{\widetilde{p},1},\cU}(\free) = \cM$. The fact that $M_G \colon H^2_{\bo_{\widetilde p, 1},\cU}(\free) \to \cM$ is a 
coisometry can be interpreted as saying that the $\cM$-norm is given by \eqref{Mform-p}.

\smallskip

Finally, the reformulation of conditions (1), (2), (3) in the form \eqref{pullback-ineq-p} follows via a straightforward $\widetilde p$-adaptation
of the argument in Theorem \ref{T:NC-BL}.
\end{proof}

We next derive from Theorem \ref{T:NC-BL-p} the $(p,n)$-analog of Theorem \ref{T:NC-BLisom} on isometrically included shift invariant 
subspaces $\cM$ of $H^2_{\bo_{p,n}, \cY}(\free)$.

\begin{theorem} \label{T:NC-BLisom-p}
A Hilbert space $\cM$ is isometrically included in $H^2_{\bo_{p,n},\cY}(\free)$ and is ${\bf S}_{\bo_{p,n},R}$-invariant if and only if
there is a Hilbert space $\cU$ and a McCT-inner multiplier $\Theta$ from $H^2_{\bo_{\widetilde{p},1},\cU}(\free)$ to $H^2_{\bo_{p,n},\cY}(\free)$ 
so that
\begin{equation}
\cM = \Theta \cdot H^2_{\bo_{\widetilde{p},1},\cU}(\free).
\label{Mform''-p}
\end{equation}
 \end{theorem}
 
\begin{proof}  As the "if" direction is immediate, it suffices to consider only the "only if" direction.  As $\cM$ is assumed to be isometrically
included in $H^2_{\bo, \cY}(\free)$,  a Beurling-Lax representation of the form \eqref{Mform''-p} with $\Theta$ a contractive multiplier forces $\Theta$
to actually be a McCT-inner multiplier. Thus it remains only to verify condition (3) in Theorem \ref{T:NC-BL-p}.
Toward this end we note that
$$
T^*_j=\left(S_{\bo_{p,n},R,j}\vert_\cM\right)^*=P_\cM S_{\bo_{p,n}, R, j}^*\vert_{\cM}\quad\mbox{for}\quad j=1,\ldots,d,
$$
for the adjoints of $T_j=S_{\bo_{p,n},R,j}|_{\cM}$ (which holds true if $\cM$ is a closed subspace of 
$H^2_{\bo_{p,n},\cY}(\free)$) and use the fact that the tuple ${\bf S}_{\bo_{p,n},R}$ is $(p,1)$-contractive to compute
\begin{align*}
 \sum_{j=1}^d p_j\| T_j^* f \|^2_\cM & = \sum_{j=1}^d  p_j\| P_\cM (S_{\bo, R, j})^* f \|^2_{H^2_{\bo, \cY}(\free)}\\
&  \le \sum_{j=1}^d p_j\| (S_{\bo, R, j})^* f \|^2_{H^2_{\bo, \cY}(\free)}\\
&\le \sum_{\alpha\in\free} p_\alpha\| (S_{\bo, R, j})^{*\alpha} f \|^2_{H^2_{\bo, \cY}(\free)}
\le \| f \|^2_{H^2_{\bo, \cY}(\free)}
 \end{align*}
which confirms that condition (3) in the statement of Theorem \ref{T:NC-BL} holds and we indeed have \eqref{Mform''}.
Furthermore, the fact that the operator $M_G$ in the proof of Theorem \ref{T:NC-BL-p} is a coisometry translates to
 $M_\Theta$ is a partial isometry, i.e., $\Theta$ is a McCT-inner multiplier from $H^2_{\bo_{\widetilde{p},1},\cU}(\free)$ 
to $H^2_{\bo_{p,n}, \cY}(\free)$.
\end{proof}

To get more explicit formulas for the  McCT-inner multiplier $\Theta$ in Theorem \ref{T:NC-BLisom-p}, we use the weight 
$\bgam_{p,n}=\{\gamma_{p,n;\alpha}\}$ constructed from $p(z)$ via formulas \eqref{apr30c}, the time-domain 
$\bgam$-observability operator $\widehat\cO_{\bgam_{p,n};C,\bA}\colon \cX\to \ell^2_{\cY}(\free)$ defined by
\begin{equation}
\widehat{\cO}_{\bgam_{p,n};C,\bA}: \; x\mapsto
\{\gamma_{p,n;\alpha}^{-\frac{1}{2}}C\bA^\alpha x\}_{\alpha\in\free},
\label{12.2-p}
\end{equation}
and the $\bgam$-gramian
\begin{equation}
\cG_{\bgam_{p,n},C,\bA}=\widehat{\cO}_{\bgam_{p,n};C,\bA}^*\widehat{\cO}_{\bgam_{p,n};C,\bA}=
\sum_{\alpha\in\free}\gamma_{p,n;\alpha}^{-1}\bA^{*\alpha^\top}C^*C\bA^\alpha. 
\label{12.2ag-p}
\end{equation}
Straightforward power series computation based on formulas \eqref{march21}, \eqref{12.2ag-p} and \eqref{apr30c} verifies
the Stein equation
$$
\cG_{p,n,C,\bA}-\sum_{j=1}^d p_j A_j^*\cG_{p,n,C,\bA}A_j=\cG_{\bgam_{p,n},C,\bA}.
$$
We next apply the operator-valued power series \eqref{jan4bv} to the operator $\widehat{\cO}_{\bgam_{p,n};C,\bA}$
and show that, in terms of notation from  \eqref{march21} and \eqref{jan39}, 
\begin{equation}
\bPs_{\bgam_{p,n}}(z)\widehat{\cO}_{\bgam_{p,n};C,\bA}=CR_{p,n}(z\bA)(I-Z_{\widetilde{p}}(z)A).
\label{again22-p}
\end{equation}
Indeed, from the  explicit formulas 
\eqref{jan4bv}, \eqref{12.2-p}, we have, on account of \eqref{apr30c},
\begin{align*}
\bPs_{\bgam_{p,n}}(z)\widehat{\cO}_{\bgam_{p,n};C,\bA}&=\sum_{\alpha\in\free}\gamma_{p,n;\alpha}^{-1}C\bA^\alpha z^\alpha\\
&=C+\sum_{j=1}^d\sum_{\alpha\in\free}\gamma_{p,n;\alpha j}^{-1}C\bA^{\alpha j} z^{\alpha j}\\
&=C+\sum_{j=1}^d\sum_{\alpha\in\free}(\omega_{p,n;\alpha j}^{-1}-\omega_{p,n;\alpha}^{-1}p_j)C\bA^{\alpha j}z^{\alpha j},
\end{align*}
from which \eqref{again22-p} follows, since
\begin{align*}
C+\sum_{j=1}^d\sum_{\alpha\in\free}\omega_{p,n;\alpha j}^{-1}C\bA^{\alpha j}z^{\alpha j}&=
\sum_{\alpha\in\free}\omega_{p,n;\alpha}^{-1}C\bA^{\alpha}z^{\alpha}=CR_{p,n}(z\bA),\\
\sum_{j=1}^d\sum_{\alpha\in\free}\omega_{p,n;\alpha}^{-1}p_jC\bA^{\alpha j}z^{\alpha j}
&=\bigg(\sum_{\alpha\in\free}\omega_{p,n;\alpha}^{-1}C\bA^\alpha z^\alpha\bigg) \bigg(\sum_{j=1}^d p_jA_jz_j \bigg) \\
&=CR_{p,n}(z\bA)Z_{\widetilde{p}}(z)A.
\end{align*}

We now present the $(p,n)$-version of Theorem \ref{T:explicit2}.

\begin{theorem}  \label{T:explicit2-p}  
{\rm (1)} Let us assume that a Hilbert space $\cM$ is contractively contained in $H^2_{\bo_{p,n}, \cY}(\free)$ and that its Brangesian complement 
$\cM^{[\perp]}$ satisfies conditions (i) and (ii) in part (4) of Theorem \ref{T:2-1.2'-p}, and hence has a representation a 
$$
   \cM^{[\perp]} = \operatorname{Ran} \cO_{n,p; C, \bA}
$$
for some $(p,n)$-contractive output pair $(C, \bA)$.  Let us impose the additional hypothesis that $(C, \bA)$ satisfies the inequality
\begin{equation}   \label{extra-hy'-p}
 \sum_{j=1}^d p_j A_j^* (I - \cG_{n,p; C, \bA}) A_j \preceq I - \cG_{p,n, C, \bA}.
 \end{equation}
 Then the related inequality 
 \begin{equation}   \label{extra-hy-p}
 \begin{bmatrix} A^* & \widehat \cO_{\bgam, C, \bA}^* \end{bmatrix} \begin{bmatrix} P \otimes I_\cX  & 0 \\ 0 & I_\cY \end{bmatrix}
 \begin{bmatrix}  A \\ \widehat \cO_{\bgam_{p,n}; C, \bA} \end{bmatrix} \preceq I_\cX
 \end{equation}
 holds and as a consequence of  \eqref{fact} we may choose a solution
 \begin{equation}   \label{bd-p}
 \begin{bmatrix} B \\ D \end{bmatrix} = \begin{bmatrix} B \\ \operatorname{col}_{\alpha \in \free} [D_\alpha] \end{bmatrix} \colon \cU
 \to \begin{bmatrix} \cX^d \\ \ell^2_\cY(\free) \end{bmatrix}
 \end{equation}
 of the Cholesky factorization problem
 \begin{equation} \label{Cho-big1-p}
 \begin{bmatrix} B \\ D \end{bmatrix} \begin{bmatrix} B^* & D^* \end{bmatrix} =
 \begin{bmatrix} \widetilde P^{-1} \otimes I_\cX & 0 \\ 0 & I_\cY \end{bmatrix} - 
 \begin{bmatrix} A \\ \widehat \cO_{\bgam_{p,n}; C, \bA} \end{bmatrix} 
 \begin{bmatrix} A^* & (\widehat \cO_{\bgam_{p,n}; C, \bA})^* \end{bmatrix}.
 \end{equation}
 Define formal power series
 \begin{equation}   \label{Theta1-p}
 {\mathfrak D}(z) = \sum_{\alpha \in \free} \gamma_{p,n; \alpha}^{-\frac{1}{2}} D_\alpha z^\alpha, \quad
 \Theta(z) = {\mathfrak D}(z) + CR_{p,n}( z \bA) Z_{\widetilde p}(z) B.
 \end{equation}
 Then $\Theta$ is a Beurling-Lax representer for $\cM$, i.e.,  $\Theta$ is a contractive multiplier from $H^2_{\bo_{\widetilde p, 1}, \cU}(\free)$
 to $H^2_{\bo_{p,n}, \cY}(\free)$ such that $\cM = \Theta \cdot H^2_{\bo_{\widetilde p, 1}, \cU}(\free)$.

\smallskip

{\rm (2)} Suppose that $\cM$ is isometrically-included subspace of $H^2_{\bo_{p,n}, \cY}(\free)$ which is $\bS_{p,n; R}$-invariant with
  $\bS_{p,n;R}^*$-invariant orthogonal complement $\cM^\perp$ represented as $\cM^\perp = \operatorname{Ran} \cO_{p,n; C, \bA}$
  with $\bA$ being $p$-strongly stable and $(C, \bA)$ a $(p,n)$-isometric output pair  as in the last part of item (4) in Theorem \ref{T:2-1.2'-p}.  
Then the additional hypothesis
\eqref{extra-hy'-p} in part (1) above is automatic and a McCT-inner Beurling-Lax representer $\Theta$ for $\cM$ can be constructed explicitly via the procedure given by \eqref{bd-p}, \eqref{Cho-big1-p}, \eqref{Theta1-p}.
\end{theorem}

\begin{proof}  The proof  of statement (1) follows the outline of the proof of statement (1) in Theorem \ref{T:explicit2} with suitable adjustments, 
specifically:  
\begin{enumerate}
\item[(i)] Use the fact that, by item (3b) in Theorem \ref{T:2-1.2'-p}, 
\begin{equation}   \label{kerkmp-p}
K_{\cM^{[\perp]}}(z, \zeta) = C R_{p,n}(z \bA) R_{p,n} (\zeta \bA)^* C^*,
\end{equation}
and the argument as in the first step in the proof of Theorem \ref{T:explicit2} that we are only required to produce a formal power series
$\Theta(z)$ so that the following kernel decomposition holds: 
\begin{equation}   \label{ker-decom1-p}
  k_{\bo_{p,n}}(z, \zeta) I_\cY - \Theta(z) (k_{\widetilde p, 1}(z, \zeta) I_\cU) \Theta(\zeta)^* = C R_{p,n}(z \bA) R_{p,n} (\zeta \bA)^* C^*.
\end{equation}
\item[(ii)] Verify that \eqref{extra-hy'-p} implies \eqref{extra-hy-p} (in fact they are equivalent) in view of the factorization \eqref{12.2ag-p}.

\item[(iii)]  With $C_0 = \widehat \cO_{\bgam_{p,n}, C, \bA}$ and $\sbm{ B \\ D}$ constructed as in \eqref{bd-p}, introduce
$$
   \Theta_0(z) = D + C_0 (I - Z_{\widetilde p}(z) A)^{-1} Z_{\widetilde p}(z) B.
$$

\item[(iv)] Use the definition of ${\mathfrak D}(z)$ and  the identity \eqref{again22-p} to verify that 
\begin{equation}  \label{formula-p}
   \bPs_{\bgam_{p,n}}(z) \Theta_0(z) = \Theta(z)
 \end{equation}
 where $\Theta(z) $ is as in \eqref{Theta1-p}.
 Use Theorem \ref{T:cm-p} (in place of Theorem \ref{T:cm}) to verify that $\Theta_0$ satisfies (in place of \eqref{ker-decom0}) the kernel decomposition
 $$
 k_{\bo_{\widetilde p, 1}}(z, \zeta) I_{\ell^2_\cY(\free)} - \Theta_0(z) (k_{\bo_{\widetilde p, 1}}(z, \zeta) I_\cU) \Theta_0(\zeta)^*= 
 C_0 (I - Z_{\widetilde p}(z) A)^{-1} (I - A^* Z_{\widetilde p}(\zeta)^*)^{-1} C_0^*.
 $$
\item[(v)] Multiply this last identity on the left by $\bPs_{\bgam_{p,n}}(z)$ and on the right by $\bPs_{\bgam_{p,n}}(\zeta)^*$, use identities
\eqref{may1} (in place of \eqref{factor}), \eqref{formula-p} (in place of \eqref{formula}), and \eqref{again22-p} (in place of \eqref{bPs-Obs})
to arrive at the kernel decomposition \eqref{ker-decom1-p} as required.
\end{enumerate}

As for statement (2), note that,  as a consequence of part (3e) of Theorem \ref{T:2-1.2'-p}, $\bA$ being $p$-strongly stable and
$(C, \bA)$ being $(p,n)$-isometric implies that $\cO_{p,n, C, \bA}$ is isometric, i.e., that
$\cG_{p,n,C, \bA} = I_\cX$. Then condition \eqref{extra-hy'-p} holds trivially (in the form $o \preceq 0$), and hence all the analysis of part (1) applies.  
As $\cM$ is isometrically included in $H^2_{\bo_{-,n}, \cY}(\free)$, the resulting contractive-multiplier Beurling-Lax representer $\Theta$ is in fact
McCT-inner.
\end{proof}

We now present the $(p,n)$-version of Theorem \ref{T:NC-BL'}.
\begin{theorem} \label{T:NC-BL-pp}
Let $p$ and $q$ be two noncommutative regular power series and let $n$ and $m$ be two positive integers.
A Hilbert space $\cM$ is such that
\begin{enumerate}
\item $\cM$ is contractively included in $H^2_{\bo_{p,n},\cY}(\free)$,
\item $\cM$ is ${\bf S}_{\bo_{p,n},R}$-invariant,
\item the $d$-tuple $\; {\bf A} =(A_1,\ldots,A_d)$ where $A_j=(S_{\bo_{p,n},R,j}|_{\cM})^*$ for
$j=1,\ldots,d$, is a $q$-strongly stable $(q,m)$-hypercontraction
\end{enumerate}
if and only  if there is a coefficient Hilbert space $\cU$ and a
contractive multiplier $\Theta$ from $H^2_{\bo_{q,m},\cU}(\free)$ to 
$H^2_{\bo_{p,n},\cY}(\free)$ so that
\begin{equation}
\cM = \Theta \cdot H^2_{\bo_{q,m},\cU}(\free)\quad\mbox{with lifted norm}\quad
\| \Theta \cdot f \|_{\cM} = \| Q f \|_{H^2_{\bo_{q,m},\cU}(\free)}
\label{Mform-pp}
\end{equation}
        where $Q$ is the orthogonal projection onto $(\operatorname{Ker}
          M_{\Theta})^\perp$.
        \end{theorem}
	
\begin{proof}[Sketch of the proof] 
If $\Theta$ is a contractive multiplier 
from $H^2_{\bo_{q,m},\cU}(\free)$ to $H^2_{\bo_{p,n},\cY}(\free)$ and 
$\cM$ is of the form \eqref{Mform-pp}, then it follows as in the proof of Theorem \ref{T:NC-BL'} 
that $\|\Theta f\|_{H^2_{\bo_{p,n},\cY}(\free)}\le \|\Theta f\|_{\cM}$ which verifies property
(1). Property (2) follows from the intertwining equalities $S_{\bo_{p,n},R,j}
M_{\Theta}= M_{\Theta}S_{\bo_{q,m},R,j}$.
Furthermore, the formulas \eqref{9.2'} and \eqref{again18} still hold true (with ${\bf S}_{\bo_{q,m},R}$
instead of ${\bf S}_{\bo',R}$) and therefore,
$$
\| {\bf A}^\alpha\Theta f\|^2_{\cM}=\|{\bf S}_{\bo_{p,n},R}^{*\alpha}
Qf\|^2_{H^2_{\bo_{q,m},\cU}(\free)}\quad\mbox{for all}\quad \alpha\in\free, \, 
f\in H^2_{\bo_{q,m},\cU}(\free).
$$
Since ${\bf S}_{\bo_{q,m},R}^*$ is a $q$-strongly stable $(q,m)$-hypercontraction on 
$H^2_{\bo_{q,m}}(\free)$, we have
\begin{align*}
&\left\langle q^N(B_{\bA})[I_{\cM}]\Theta f, \; \Theta f\right\rangle_{\cM}\notag\\
&=\sum_{\alpha\in\free}\bigg(\sum_{\beta_{1},
    \dots, \beta_{N} \ne \emptyset \colon \beta_{1} \cdots \beta_{N}
    = \alpha} q_{\beta_{1}} \cdots q_{\beta_{N}} \bigg)
\|\bA^\alpha \Theta f\|^2_{\cM}\notag\\
&=\sum_{\alpha\in\free}\bigg(\sum_{\beta_{1},
    \dots, \beta_{N} \ne \emptyset \colon \beta_{1} \cdots \beta_{N}
    = \alpha} q_{\beta_{1}} \cdots q_{\beta_{N}} \bigg)
\|{\bf S}_{\bo_{q,m},R}^{*\alpha}Qf\|^2_{H^2_{\bo_{q,m},\cU}(\free)}\notag\\
&=\left\langle q^N(B_{{\mathbf S}^*_{\bo_{q,m},R}})[I_{H^2_{\bo_{q,m},\cU}(\free)}]
Qf, \; Qf\right\rangle_{H^2_{\bo_{q,m},\cY}(\free)}\to 0 \; \; (\mbox{as $N\to\infty$})
\end{align*}
and 
\begin{align*}
\left\langle \Gamma_{q,k;\bA}[I_{\cM}]\Theta f, \, \Theta f\right\rangle_{\cM}
&=\sum_{\alpha \in\free}c_{q,k;\alpha}\|\bA^{\alpha} \Theta f\|^{2}_{\cM}\notag\\
&=\sum_{\alpha \in\free}c_{q,k;\alpha}
\| {\mathbf  S}_{\bo_{q,m},R}^{* \alpha} Qf\|^{2}_{H^2_{\bo_{q,m},\cU}(\free)}\notag\\
&=\left\langle \Gamma_{q,k;\bA}[I_{H^2_{\bo_{q,m},\cY}(\free)}]Q f, \,
Qf\right\rangle_{H^2_{\bo_{q,m},\cY}(\free)}\ge 0
\end{align*}
for $k=1,\ldots,m$, which shows that ${\bf A}$ is a $q$-strongly stable $(q,m)$-hypercontraction on
$\cM$ and therefore completes the proof of sufficiency. 

\smallskip

For the necessity part, given a Hilbert space $\cM$ subject to conditions (1), (2), (3), we start
with the $(q,m)$-hypercontractive tuple $\; {\bf A}=(A_1,\ldots,A_d)$ where $A_j=
(S_{\bo_{q,m},R,j}|_{\cM})^*$ and factor the  positive semidefinite operator 
$\Gamma_{q,m,\bA}[I_{\cM}]$ as 
$\Gamma_{q,m,\bA}[I_{\cM}]=C^*C$ for an appropriately chosen operator  $C \colon \cM \to \cU$ from 
$\cM$ into a  Hilbert space $\cU$ subject to condition 
$\operatorname{dim} \cU = \operatorname{rank} \Gamma_{q,m,\bA}[I_{\cM}]$.
Then the pair $(C, \bA)$ is $(q,m)$-isometric and, since $\bA$ is $q$-strongly stable by hypothesis (3),
it follows from  part (2) of Theorem \ref{T:2-1.2'-p} that the observability operator $\cO_{q,m;C, \bA}$
is an isometry from $\cM$ into $H^2_{\bo_{q,m},\cU}(\free)$. Since the inclusion map 
$\iota \colon \cM \to H^2_{\bo_{p,n},\cY}(\free)$ is a contraction by hypothesis (1), the operator 
$$
G = \iota \circ \cO_{q,m;C, \bA}^{*} \colon H^2_{\bo_{q,m},\cU}(\free)\to 
H^2_{\bo_{p,n},\cY}(\free)
$$ 
is a contraction. Making use of intertwining relations \eqref{4.8a-p} and taking into account that 
$\iota \circ A_j^*=S_{\bo_{p,n},R,j}\circ \iota$, we conclude as in the proof of 
Theorem \ref{T:NC-BL'} that 
$$
G S_{\bo_{q,m},R,j}=S_{\bo_{p,n},R,j} G\quad\mbox{for}\quad j=1,\ldots,d.
$$
By the adaptation of Proposition \ref{P:RR} to the setting of weighted spaces of the form 
$H^{2}_{\bo_{p,n},\cY}(\free)$,  it follows that
$G$ is a multiplication operator, i.e., there is a contractive multiplier $\Theta$ so that 
$G = M_{\Theta}$.

\smallskip

Since $\cO_{q,m;C, \bA} \colon \cM \to H^2_{\bo_{q,m},\cU}(\free)$ is an isometry, 
it follows that $\operatorname{Ran} \cO_{q,m;C,\bA}^{*} = \cM$ and also that $\cM = \Theta \cdot
H^2_{\bo_{q,m},\cU}(\free)$ with $\cM$-norm given by \eqref{Mform-pp}.  This completes the proof.
\end{proof}
The next result shows that the quasi-wandering-subspace version of the Beurling-Lax theorem
extends verbatim  to the $(p,n)$-setting.

\begin{theorem} \label{T:13.1p}
Let $\cM$ be a closed ${\bf S}_{\bo_{p,n},R}$-invariant subspace of $H^2_{\bo_{p,n},\cY}(\free)$
containing no nontrivial reducing subspaces for ${\bf S}_{\bo_{p,n},R}$. Then 
\begin{equation}
\cM=\bigvee_{\alpha \in\free}{\bf S}^\alpha_{\bo_{p,n},R}\cQ, \quad\mbox{where}\quad 
\cQ=P_{\cM}\bigg(\bigoplus_{j=1}^d S_{\bo_{p,n},R,j}\cM^\perp\bigg).
\label{13.3gp}
\end{equation}
\end{theorem}
\begin{proof} 
The proof follows the lines of that of Theorem \ref{T:13.1g}.
We let $(C,\bA)$ be an $(p,n)$-isometric pair such that $\cM=({\rm Ran} \, \cO_{p,n;C,\bA})^{\perp}$, and then 
observe that the quasi-wandering subspace $\cQ$ in \eqref{13.3gp} is equal (as a set) to
\begin{equation}
\cQ=F(z)\cX^d,\quad\mbox{where}\quad F(z)=CR_{p,n}(z\bA)(Z(z)-A^*).
\label{apr8ap}
\end{equation}
In more details, for a fixed $\bx\in\cX^d$, we represent $F(z)\bx$ as in \eqref{13.3h} 
(with $\omega_{p,n;\alpha}$ instead of $\omega_{|\alpha|}$), and then write this representation in terms of  
operators $\cO_{\bo_{p,n},C,\bA}$ and $S_{\bo_{p,n},R,j}$ as
\begin{equation}
 F{\bf x}+\cO_{p,n;,C,\bA}\sum_{j=1}^{d}A_j^*x_j=\sum_{j=1}^{d}S_{\bo_{p,n},R,j}\cO_{p,n;C,\bA}x_j.
\label{apr8p}
\end{equation}
Making use of intertwining relations \eqref{4.8a-p}, it is not hard to verify that the two terms
on the right side of \eqref{apr8p} are orthogonal in $H^2_{\bo_{p,n},\cY}(\free)$-metric, from which 
\eqref{apr8ap} follows, since $\cM^\perp={\rm Ran} \, \cO_{p,n;C,\bA}$.

\smallskip

Then, assuming that there is a nonzero $f\in\cM= ({\rm Ran} \cO_{p,n; C, \bA})^\perp$ which is orthogonal to
${\bf S}_{\bo_{p,n},R}^\alpha F\cX^d$ for all $\alpha\in\free$, we use the decomposition
\eqref{apr8} (adjusted to the present setting) to conclude that
${\bf S}_{\bo_{p,n},R}^{* \alpha} f$ belongs to $\cM$ for all $\alpha \in\free$. Since $\cM$ is
${\bf S}_{\bo_{n,p},R}$-invariant, it now follows that
\begin{equation}
{\bf S}_{\bo_{p,n},R}^{\gamma}{\bf S}_{\bo_{p,n},R}^{* \alpha} f\in\cM\quad\mbox{for all}\quad
\alpha,\gamma\in\free.
\label{13.3gpa}
\end{equation}
Since ${\bf S}^{*}_{\bo_{p,n},R}$ is a $(p,1)$-contraction, we have (as in \eqref{4.13-p})
$p^j(B_{{\bf S}^{*}_{\bo_{p,n},R}})[I]\preceq I$ for all $j\ge 0$. 
It then follows, as in the computation \eqref{4.17-pa}, that
\begin{align*}
\sum_{\alpha\in\free:|\alpha|\le N}c_{p,n;\alpha}{\bf S}_{\bo_{p,n},R}^{\alpha^\top}
{\bf S}_{\bo_{p,n},R}^{*\alpha}
&\preceq (1+p)^n(B_{{\bf S}^{*}_{\bo_{p,n},R}})[I_{H^2_{\bo_{p,n},\cY}(\free)}]\\
&=\sum_{j=0}^n \bcs{n \\ j}p^j(B_{{\bf S}^{*}_{\bo_{p,n},R}})[I_{H^2_{\bo_{p,n},\cY}(\free)}]\\
&\preceq  \sum_{j=0}^n \bcs{n \\ j} I_{H^2_{\bo_{p,n},\cY}(\free)}=2^n\cdot 
I_{H^2_{\bo_{p,n},\cY}(\free)}
\end{align*}
and hence,
\begin{equation}
\bigg\|\sum_{\alpha\in\free:|\alpha|\le N}c_{p,n;\alpha}
{\bf S}_{\bo_{p,n},R}^{\alpha^\top}{\bf S}_{\bo_{p,n},R}^{*\alpha}\bigg\|\le 2^n\quad\mbox{for each} 
\quad N\ge 1.
\label{13.3gpb}
\end{equation}
Since $f\not\equiv 0$, 
there is a $\beta \in\free$ such that $f_\beta \neq 0$. If $E: \, f\to f_\emptyset$ is the 
evaluation operator 
on $H^2_{\bo_{p,n},\cY}(\free)$, then its adjoint $E^*$ is the inclusion of $\cY$ into 
$H^2_{\bo_{p,n},\cY}(\free)$,
and from the formula \eqref{SRLv-p} we have
$$
E^*E{\bf S}_{\bo_{p,n},R}^{*\beta}f=\big({\bf S}_{\bo_{p,n},R}^{*\beta}f\big)_\emptyset=\omega_{p,n;\beta}
f_\beta:=y\neq 0.
$$
By the equality \eqref{march31} applied to the power series ${\bf S}_{\bo_{p,n},R}^{*\beta}f$ we have
\begin{equation}
\Gamma_{p,n,{\bf S}_{\bo_{p,n},R}^*}[I_{H^2_{\bo_{p,n},\cY}(\free)}]{\bf S}_{\bo_{p,n},R}^{*\beta}f
=E^*E{\bf S}_{\bo_{p,n},R}^{*\beta}f=y\neq 0.
\label{3.1eaa}
\end{equation}
Due to \eqref{13.3gpa}, the power series
$$
g_N=\sum_{\alpha\in\free:|\alpha|\le N}c_{p,n;\alpha}{\bf S}_{\bo,R}^{\alpha^\top}
{\bf S}_{\bo,R}^{*\alpha}{\bf
S}_{\bo,R}^{*\beta}f
$$
belongs to $\cM$ for all $N\ge 1$, and the sequence $\{g_N\}_{N\ge 1}\subset\cM$ is uniformly
bounded in metric of $H^2_{\bo_{p,n},\cY}(\free)$, due to \eqref{13.3gpb}. 
Therefore, it admits a subsequential weak limit 
$$
\sum_{\alpha\in\free}c_{p,n;\alpha}{\bf S}_{\bo_{p,n},R}^{\alpha^\top}
{\bf S}_{\bo_{p,n},R}^{*\alpha}{\bf
S}_{\bo_{p,n},R}^{*\beta}f=\Gamma_{\bo_{p,n},{\bf
S}_{\bo_{p,n},R}^*}[I_{H^2_{\bo_{p,n},\cY}(\free)}]{\bf S}_{\bo_{p,n},R}^{*\beta}f
$$
which belongs to $\cM$ and which is equal to $y$, by \eqref{3.1eaa}. Then  
$H^2_{\bo_{p,n}}(\free)\otimes y$ 
is a subspace of $\cM$ and is reducing for ${\bf S}_{\bo_{p,n},R}$ which is a contradiction.
\end{proof}

\section{$H^2_{\bo_{p,n},\cY}(\free)$-Bergman-inner families}\label{S:binnerp}

We next present the $(p,n)$-analog of the most elaborate version of the 
Beurling-Lax theorem presented in Theorem \ref{T:BL3}. To get such an analog we need to make the extra assumption
that the right shift operators $S_{\bo_{p,n},R,j}$ on $H^2_{\bo_{p,n},\cY}(\free)$ are left invertible for 
$j=1,\ldots,d$. As follows from \eqref{5.2p},
$$
S^*_{\bo_{p,n},R,j}S_{\bo_{p,n},R,j}: \;  \; \sum_{\alpha \in\free}f_\gamma z^\alpha \mapsto
        \sum_{\alpha \in\free}\frac{\omega_{p,n;\alpha j}}{\omega_{p,n;\alpha}} \,
        f_{\alpha j}z^\alpha\quad (j=1,\ldots,d)
$$       
and hence our assumption can be expressed in terms of the numbers \eqref{omegapnalpha} as follows.

\smallskip

$\bullet$ The regular power series $p$ is such that the associated positive numbers $\omega_{p,n;\alpha}$ are subject to
inequalities
\begin{equation}
\frac{\omega_{p,n;\alpha}}{\omega_{p,n;\alpha j}}\le M\quad\mbox{for some $\; M>0\; $ and 
 all $\; \alpha\in\free\; $ and $\; j=1,\ldots,d$}.
\label{rest}
\end{equation}
The latter assumption is not very restrictive -- it holds true if the coefficients $p_\alpha$ \eqref{apr12} are 
uniformly bounded (in particular, if $p(z)$ is a noncommutative polynomial) and even more generally, if 
$p_{\alpha}$ does not grow too fast with respect to $|\alpha|$.
\begin{proposition}
Let $p(z)$ be a regular noncommutative power series as in \eqref{apr12}, \eqref{apr12-1} and let us assume that 
$p_{\alpha j}\le \varphi p_{\alpha}$ for some $\varphi>0$ and all $\alpha\in\free$ and $j=1,\ldots,d$. Then 
\eqref{rest} holds.
\label{P:rest}
\end{proposition}
\begin{proof} For a fixed $\alpha\in\free$ and $j\in\{1,\ldots,d\}$, we have by \eqref{omegapnalpha},
$$
\omega^{-1}_{p,n; \alpha j} = \sum_{\ell=0}^{|\alpha|+1} \bcs{n+\ell-1 \\ n-1}
\sum_{\beta_{1},\dots, \beta_{\ell} \ne \emptyset \colon \beta_{1}\cdots
\beta_{\ell} = \alpha j} p_{\beta_{1}} \cdots p_{\beta_{\ell}}. 
$$
For each fixed $\ell\ge 0$, we consider separately the terms $p_{\beta_{1}} \cdots p_{\beta_{\ell}}$
with $\beta_\ell=j$ and with $\beta_\ell=\beta'_\ell j$ for some $\beta^\prime_\ell\neq\emptyset$.
We therefore have $\omega^{-1}_{p,n; \alpha j}=\Pi_1+\Pi_2$, where
\begin{align}
\Pi_1&=\sum_{\ell=1}^{|\alpha|+1} \bcs{n+\ell-1 \\ n-1}\sum_{\beta_{1},\dots, \beta_{\ell-1} \ne \emptyset \colon
\beta_{1}\cdots \beta_{\ell-1} = \alpha} p_{\beta_{1}} \cdots p_{\beta_{\ell-1}}p_j,\label{rest3}\\
\Pi_2&=\sum_{\ell=0}^{|\alpha|} \bcs{n+\ell-1 \\ n-1}\sum_{\beta_{1},\dots, \beta_{\ell} \ne \emptyset \colon
\beta_{1}\cdots\beta_{\ell} = \alpha} p_{\beta_{1}} \cdots p_{\beta_{\ell-1}}p_{\beta_{\ell}j}.\label{rest4}
\end{align}
Note that in \eqref{rest4} we changed notation $\beta'_\ell$ back to $\beta_\ell$. 
Since $p_{\beta_{\ell}j}\le \varphi p_{\beta_{\ell}}$ for all $\beta\in\free$ (by the assumption), we have
from \eqref{rest4} and \eqref{omegapnalpha}
$$
\Pi_2\le \varphi \sum_{\ell=0}^{|\alpha|} \bcs{n+\ell-1 \\ n-1}\sum_{\beta_{1},\dots, \beta_{\ell} \ne \emptyset \colon
\beta_{1}\cdots\beta_{\ell} = \alpha} p_{\beta_{1}} \cdots p_{\beta_{\ell-1}}p_{\beta_{\ell}}=\varphi 
\omega^{-1}_{p,n; \alpha}
$$
Shifting the index $\ell$ in \eqref{rest3} and taking into account that 
$\bcs{n+\ell \\ n-1}\le n\bcs{n+\ell-1 \\ n-1}$ for all $\ell\ge 0$, we get
\begin{align*}
\Pi_1&=p_j\sum_{\ell=0}^{|\alpha|} \bcs{n+\ell \\ n-1}
\sum_{\beta_{1}\cdots \beta_{\ell} = \alpha} p_{\beta_{1}} \cdots p_{\beta_{\ell}}\\
&\le np_j\sum_{\ell=0}^{|\alpha|} \bcs{n+\ell-1 \\ n-1}
\sum_{\beta_{1}\cdots \beta_{\ell} = \alpha} p_{\beta_{1}} \cdots p_{\beta_{\ell}}=np_j\omega^{-1}_{p,n; \alpha}.
\end{align*}
From the two last inequalities, $\omega^{-1}_{p,n; \alpha j}\le \omega^{-1}_{p,n; \alpha}(\varphi+np_j)$ and 
therefore \eqref{rest} holds with $M=\varphi+n\max\{p_1,\ldots,p_d\}$.
\end{proof}
To consider the family of transfer functions \eqref{1.37pre-p} under suitable metric conditions imposed on
the system matrices \eqref{coll}, we first need to introduce the shifted $(p,n)$-observability 
operators and gramians which can be done upon making use of the power series \eqref{march20} as follows.
Given an output pair $(C, \bA)$ with $C\in\cL(\cX,\cY)$ and $\bA\in\cL(\cX)^d$ we formally define 
for each $\beta\in\free$,
\begin{align}
 \Ob_{p,n,\beta; C, \bA} &\colon x \mapsto C R_{p,n;\beta}(z \bA) =
   \sum_{\alpha\in\free} (\omega_{p,n; \alpha\beta}^{-1} C\bA^{\alpha} x)
   z^{\alpha},   \label{4.31-p} \\
  \Gr_{p,n,\beta; C, \bA} &= R_{p,n;\beta}(B_\bA)[C^*C]= \sum_{\alpha\in\free}\omega_{p,n;
  \alpha\beta}^{-1} \bA^{* \alpha^{\top}} C^{*} C \bA^{\alpha}.\label{4.32-p}
\end{align}
Letting $\beta=\emptyset$  in \eqref{4.31-p} and \eqref{4.32-p} we conclude from \eqref{march21} and \eqref{rest5}
$$
\Ob_{p,n,\emptyset; C, \bA} = \cO_{p,n; C, \bA}\quad \text{and}\quad
\Gr_{p,n,\emptyset; C, \bA} = \cG_{p,n; C, \bA}.
$$
The $(p,n)$-output stability of the pair $(C,\bA)$ is exactly what is needed for convergence of 
the series in \eqref{rest5}. In general, this condition does not guarantee the convergence 
in \eqref{4.32-p} for $\beta\neq \emptyset$. Now the assumption \eqref{rest} comes into play.
 \begin{lemma}
Let us assume that a regular noncommutative power series $p$ satisfies the assumption \eqref{rest} and 
let the pair $(C, \bA)$ as above be $(p,n)$-output stable. Then
for each $\beta\in\free$, the formulas \eqref{4.31-p} and \eqref{4.32-p}
define bounded operators $\Ob_{p,n,\beta: C, \bA}\in\cL(\cX,H^2_{\bo_{p,n},\cY}(\free))$ and 
$\Gr_{p,n,\beta; C, \bA}\in\cL(\cX)$.

\smallskip

If in addition, the pair  $(C, \bA)$ is exactly $(p,n)$-observable, then 
\begin{enumerate}
\item The operator $\Gr_{p,n,\beta; C, \bA}$ is strictly positive definite.
\item The space ${\mathbf S}^{\beta^{\top}}_{\bo_{p,n}, R}\operatorname{Ran} \Ob_{p,n,\beta;C,\bA}$
(with $H^2_{\bo_{p,n}, \cY}(\free)$-inner product) is a NFRKHS
   with reproducing kernel given by
   \begin{equation}   \label{deffrakk-p}
\boldsymbol{\mathfrak{K}}_{p,n;\beta}(z,\zeta)=
CR_{p,n;\beta}(z\bA) \big(z^{\beta} \bzeta^{\beta^{\top}}
\Gr_{p,n,\beta;C,\bA}^{-1}\big)  R_{p,n,\beta}(\zeta \bA)^*C^* .
\end{equation}
\end{enumerate}
\label{L:sign}
 \end{lemma}
\begin{proof}
By \eqref{rest}, $\frac{\omega_{p,n;\alpha}}{\omega_{p,n;\alpha\beta}}\le M^{|\beta|}$ for all $\alpha\in\free$.
Then we have from \eqref{4.32-p} and \eqref{rest5},
$$
\Gr_{p,n,\beta; C, \bA}=\sum_{\alpha\in\free}\frac{\omega_{p,n;\alpha}}{\omega_{p,n;\alpha\beta}}\cdot
\omega_{p,n;\alpha}^{-1} \bA^{* \alpha^{\top}} C^{*} C \bA^{\alpha}
\le M^{|\beta|}\cdot\cG_{p,n;C,\bA}
$$
Since the pair $(C,\bA)$ is $(p,n)$-output-stable, the gramian $\cG_{p,n;C,\bA}$ is bounded and hence,
${\Gr}_{p,n;\beta,C,\bA}$ is bounded as well. A similar computation based on formulas \eqref{4.32-p} and \eqref{march21}
verifies the inequality 
$$
\|\Ob_{p,n,\beta; C, \bA}x\|_{H^2_{\bo_{p,n}, \cY}(\free)}\le M^{|\beta|}\cdot 
\|\cO_{p,n;\bA,C}x\|_{H^2_{\bo_{p,n}, \cY}(\free)}
\quad\mbox{for all}\quad x\in\cX
$$
which implies that the operator $\Ob_{p,n,\beta; C, \bA}: \cX\to H^2_{\bo_{p,n}, \cY}(\free)$ is bounded. On the other hand,
by virtue of \eqref{march24} with $\beta=i_1\ldots i_N$,
$$
\omega^{-1}_{p,n;\alpha\beta}\ge \omega^{-1}_{p,n;\alpha}\cdot \varepsilon_{p,\beta},\quad\mbox{where}\quad
\varepsilon_{p,\beta}=p_{i_1}\ldots p_{i_N}>0
$$
which on account of \eqref{4.32-p} and \eqref{rest5}, implies
$$
\Gr_{p,n,\beta; C, \bA}=\sum_{\alpha\in\free}\frac{\omega_{p,n;\alpha}}{\omega_{p,n;\alpha\beta}}\cdot
\omega_{p,n;\alpha}^{-1}\bA^{* \alpha^{\top}} C^{*} C \bA^{\alpha}\succeq \varepsilon_{p,\beta}^{-1}\cdot\cG_{p,n;C,\bA}.
$$
If the pair $(C,\bA)$ is exactly $(p,n)$-observable, then $\cG_{p,n;C,\bA}\succ 0$ and hence,
${\Gr}_{p,n;\beta,C,\bA}$ is strictly positive definite as well. This proves part (1) in the the last statement.
To prove part (2), we write for a fixed $x\in\cX$,
$$
{\bf S}_{\bo_{p,n},R}^{\beta^\top}{\Ob}_{p,n,\beta;C,\bA}x=
\sum_{\alpha\in\free}\omega_{p,n;\alpha\beta}^{-1}(C\bA^{\alpha}x) \, z^{\alpha\beta}
$$
and then invoke the definition of the inner product in $H^2_{\bo_{p,n},\cY}(\free)$ to get
\begin{align*}
\|{\bf S}_{\bo,R}^{\beta^\top}{\Ob}_{p,n,\beta;,C,\bA}x\|^2_{H^2_{\bo_{p,n},\cY}(\free)}&=
\sum_{\alpha\in\free}\omega_{p,n;\alpha\beta}^{-1}
\big\langle C\bA^{\alpha}x, \, C\bA^{\alpha}x\big\rangle_{\cX}\notag\\
&=\bigg\langle \sum_{\alpha\in\free} \omega_{p,n;\alpha\beta}^{-1}
\bA^{*\alpha^{\top}}C^*C\bA^{\alpha}x, \, x \bigg\rangle_{\cX}
\end{align*}
which can be written, by the definition \eqref{4.32-p}, as 
$$
\|{\bf S}_{\bo,R}^{\alpha\top}\Ob_{p,n,\beta;C,\bA} x\|^2_{H^2_{\bo,\cY}(\free)}=\left\langle
\Gr_{p,n,\beta;C,\bA}x, \, x\right\rangle_{\cX}.
$$
Now a direct application of Proposition \ref{P:principle}
shows that the formal reproducing kernel for the space 
${\mathbf S}^{\beta^{\top}}_{\bo_{p,n}, R}\operatorname{Ran} \Ob_{p,n,\beta;C,\bA}$
is given by
$$
( S^{\beta^{\top}}_{\bo_{p,n}, R} \Ob_{p,n,\beta; C, \bA})(z)
\Gr_{p,n,\beta;C,\bA}^{-1} \big(( S^{\beta^{\top}}_{\bo_{p,n},R}
\Ob_{p,n,\beta;C,A})(\zeta) \big)^{*},
$$
which agrees exactly with
$\boldsymbol{\mathfrak K}_{p,n;\beta}(z, \zeta)$ given by \eqref{deffrakk-p}.
\end{proof}
A power-series computation based on the formula \eqref{4.32-p} and similar formulas for 
$\Gr_{p,n,j\beta; C, \bA}$ ($j=1,\ldots,d$) confirms the weighted Stein identity
 \begin{equation} \label{4.34-p}
\sum_{j=1}^d A_j^*\Gr_{p,n,j\beta; C, \bA}A_j + \omega_{p,n;\beta}^{-1}\cdot
     C^{*}C = \Gr_{p,n,\beta;C, \bA}\quad\mbox{for all}\quad  \beta\in\free,
 \end{equation}
which can be written in a more compact form as 
$$
A^*\widehat{\Gr}_{p,n,\beta; C, \bA}A+\omega_{p,n;\beta}^{-1}\cdot
     C^{*}C = \Gr_{p,n,\beta;C, \bA},
$$
where we have set
\begin{equation}   \label{blockgram}
\widehat{\Gr}_{p,n,\beta; C, \bA}=\begin{bmatrix}
\Gr_{p,n,1\beta; C, \bA} &&0\\  &\ddots &\\ 0 && \Gr_{p,n,d\beta; C, \bA} \end{bmatrix}\quad
\mbox{for all}\quad \beta\in\free.
\end{equation}
We now introduce the $(p,n)$-analog of the metric relation \eqref{isom}:
\begin{equation}   \label{isom-p}
  \begin{bmatrix} A^{*} & C^{*} \\ \widehat{B}_{\beta}^{*} & D_{\beta}^{*}
  \end{bmatrix} \begin{bmatrix} \widehat{\Gr}_{p,n,\beta;C,\bA} & 0 \\ 0 &
\omega_{p,n;\beta}^{-1}\cdot  I_{\cY}
\end{bmatrix}   \begin{bmatrix} A & \widehat{B}_{\beta} \\ C &
D_{\beta} \end{bmatrix}
 = \begin{bmatrix} \Gr_{p,n,\beta;C,\bA} & 0 \\ 0 & I_{\cU_{\beta}} \end{bmatrix}
    \end{equation}
and note that equality \eqref{4.34-p} corresponds to the (1,1)-entry of \eqref{isom-p}.
\begin{remark}
We next remark that if we are given only a $(p,n)$-output stable pair $(C, \bA)$ 
which is exactly $(p,n)$-observable, then by solving a suitable Cholesky factorization problem
(i.e., following the construction in Lemma \ref{L:5.6}) we can construct operators
$B_{1,\beta},\ldots,B_{d,\beta}\in \cL(\cU_\beta, \cX)$
and $D_\beta \in \cL(\cU_\beta,\cY)$  satisfying not only equality
\eqref{isom-p} but also the weighted coisometry condition
\begin{equation}   \label{wghtcoisom-p}
    \begin{bmatrix} A & {\widehat B}_{\beta}  \\ C & D_{\beta} \end{bmatrix}
        \begin{bmatrix} \Gr_{p,n,\beta;C,\bA}^{-1} & 0 \\ 0 &
            I_{\cU_\beta} \end{bmatrix}
        \begin{bmatrix} A^{*} & C^{*} \\ {\widehat B}_{\beta}^{*} &
            D_{\beta}^{*}
        \end{bmatrix} = \begin{bmatrix} \widehat{\Gr}_{p,n,\beta;C,\bA}^{-1} & 0 \\ 0 &
        \omega_{p,n;\beta} I_{\cY} \end{bmatrix}.
\end{equation}
\label{R:rest}
\end{remark}

\begin{lemma}   \label{L:5.1-p}
Let $(C,\bA)$ be a $(p,n)$-output stable pair and let $\Theta_{p,n, \bU_\beta}$ be defined
as in \eqref{1.37pre-p}  for some $\beta\in\free$ and some operators
$B_{1,\beta},\ldots,B_{d,\beta}\in \cL(\cU_\beta, \cX)$ and $D_\beta
\in \cL(\cU_\beta,\cY)$ subject to equality \eqref{isom-p}. Then $\Theta_{p,n, \bU_\beta}(z)z^\beta$ is 
$(p,n)$-Bergman inner. Moreover,
\begin{enumerate}
\item $\cO_{p,n;C,\bA}x$ is orthogonal to ${\bf S}_{\bo_{p,n},R}^{\beta^\top}\Theta_{p,n, \bU_\beta}u$ 
for all $x\in\cX$ and $u\in\cU_\beta$.

\item ${\bf S}_{\bo_{p,n},R}^{(\gamma \beta)^\top}\Theta_{p,n,\bU_\beta}u$ and 
${\bf S}_{\bo_{p,n},R}^{(\beta \gamma)^\top}\Theta_{p,n, \bU_\beta}u$ are both orthogonal to
${\bf S}_{\bo_{p,n},R}^{\beta^\top}\Theta_{p,n, \bU_\beta}u^\prime$ for all
$\gamma\neq\emptyset$ and for any $u,u^\prime\in\cU_\beta$.

\item With notation as in \eqref{apr22}, the following power-series identity holds:
\begin{align}
&\omega_{p,n;\beta}^{-1}I_{\cU_{\beta}}-
\Theta_{p,n,\bU_\beta}(z)^{*}\Theta_{p,n, \bU_\beta}(\zeta)\notag\\
&=\omega_{p,n;\beta} \widehat{B}_{\beta}^{*}
\big(\omega_{p,n;\beta}^{-1}I+\widehat{R}_{p,n;\beta}(z\bA)^*A^*\big)
\widehat{\Gr}_{p,n,\beta;C,\bA}\big(\omega_{p,n;\beta}^{-1}I+A\widehat{R}_{p,n;\beta}(\zeta\bA)
\big)\widehat{B}_{\beta}\notag\\
&\quad -\omega_{p,n;\beta}\widehat{B}_{\beta}^{*}
\widehat{R}_{p,n;\beta}(z\bA)^*\Gr_{p,n,\beta;C,\bA}
\widehat{R}_{p,n;\beta}(\zeta\bA)\widehat{B}_{\beta}.\label{jul18-p}
\end{align}
\end{enumerate}
\end{lemma}

\begin{proof} To lighten the notation in the proof, we write simply $\Theta_{p,n,\beta}$ rather than $\Theta_{p,n, \bU_\beta}$.
Statements (1) and (2) rely on the equality 
\begin{equation}   \label{isom-p12}
\omega_{p,n;\beta}^{-1}\cdot C^*D_\beta +\sum_{j=1}^d
A_j^*\Gr_{p,n,j\beta,C,\bA}B_{j,\beta} = 0
    \end{equation}
which occurs as the (1,2)-entry in \eqref{isom-p}. Making use of 
expansions \eqref{march21}, \eqref{1.37pre-p} and the definition of the inner product in
$H^2_{\bo_{p,n},\cY}(\free)$ we get
\begin{align*}
&\big\langle {\bf S}_{\bo_{p,n},R}^{\beta^\top}\Theta_{p,n,\beta} u, \, \cO_{p,n;C,\bA}x
\big\rangle_{H^2_{\bo_{p,n},\cY}(\free)}\\
&=\omega_{p,n;\beta}^{-1}\cdot\left\langle D_\beta  u, \,
C\bA^{\beta}x\right\rangle_{\cY}  +\sum_{j=1}^d
\sum_{\alpha\in\free}\omega_{p,n;\alpha j\beta}^{-1}\cdot
\left\langle C\bA^{\alpha}B_{j,\beta}u, \, C\bA^{\alpha j \beta}x
\right\rangle_{\cY}  \\
&=\bigg\langle \bigg(\omega_{p,n;\beta}^{-1}\cdot C^*D_\beta +
\sum_{j=1}^d A_j^*\bigg(\sum_{\alpha\in\free}
\omega_{p,n;\alpha j\beta}^{-1}
\bA^{*\alpha^{\top}}C^*C\bA^{\alpha}\bigg)
B_{j,\beta}\bigg)u,\, \bA^{\beta}x\bigg\rangle_{\cX}  \\
&=\bigg\langle \bigg(\omega_{p,n;\beta}^{-1}\cdot C^*D_\beta +\sum_{j=1}^d
A_j^*\Gr_{p,n,j\beta,C,\bA}B_{j,\beta}\bigg) u,\, \bA^{\beta}x\bigg\rangle_{\cX}=0
\end{align*}
thus proving statement (1). Letting $\widetilde{j}\in\{1,\ldots,d\}$ to denote the rightmost 
letter in the given $\gamma \neq\emptyset$ (i.e., $\gamma =\widetilde{\gamma}\widetilde{j}$)
we compute (as in the proof of Lemma \ref{L:5.1}),
\begin{align*}
&\big\langle {\bf S}_{\bo_{p,n},R}^{(\gamma \beta)^\top}\Theta_{n,p;\beta} u,
\; {\bf S}_{\bo_{p,n},R}^{\beta^\top}\Theta_{p,n,\beta} 
u^\prime\big\rangle_{H^2_{\bo_{p,n},\cY}(\free)}\\
& =\omega_{p,n;\beta\gamma}\cdot \big\langle \omega_{p,n;\beta}^{-1} D_\beta   u, \;
\omega_{p,n;\beta\gamma}^{-1}
C\bA^{\widetilde{\gamma}}B_{\widetilde{j}, \beta}u^\prime\big\rangle_{\cY} \\
&\quad  +\sum_{j=1}^d\sum_{\alpha\in\free} \omega_{p,n, \alpha j\gamma}\cdot
\big\langle \omega_{p,n,\alpha j\beta}^{-1}
C\bA^{\alpha}B_{j,\beta} u, \;
\omega_{p,n, \alpha j\gamma}^{-1} C\bA^{\alpha j\widetilde{\gamma}
}B_{\widetilde{j},\beta} u^\prime\big\rangle_{\cY}  \\
 & = \omega_{p,n;\beta}^{-1}\cdot \big\langle D_\beta u, \;
C\bA^{\widetilde{\gamma}}B_{\widetilde{j},\beta} u^\prime\big\rangle_{\cY}
+\sum_{j=1}^d\sum_{\alpha\in\free}\omega_{p,n;\alpha j\beta}^{-1}\cdot
\big\langle C\bA^{\alpha}B_{j,\beta}u, \; C\bA^{\alpha j\widetilde{\gamma}}
B_{\widetilde{j},\beta} u^\prime\big\rangle_{\cY}  \\
&=\bigg\langle \bigg(\omega_{p,n;\beta}^{-1}C^*D_\beta +
\sum_{j=1}^d A_j^*\bigg(\sum_{\alpha\in\free}
\mu_{p,n;\alpha j\beta}^{-1}
\bA^{*\alpha^{\top}}C^*C\bA^{\alpha}\bigg) B_{j,\beta}\bigg)u, \;
\bA^{\widetilde{\gamma}}B_{\widetilde{j},\beta} u^\prime\bigg\rangle_{\cX}  \\
&=\bigg\langle \bigg(\omega_{p,n;\beta}^{-1} \cdot C^*D_\beta +\sum_{j=1}^d
A_j^*\Gr_{p,n,j\beta;C,\bA}B_{j,\beta}\bigg)u,\,
\bA^{\widetilde{\gamma}}B_{\widetilde{j},\beta} u^\prime\bigg\rangle_{\cX}=0,
\end{align*}
where the last equality follows from \eqref{isom-p12}. Thus, 
${\bf S}_{\bo_{p,n},R}^{(\gamma \beta)^\top}\Theta_{p,n,\beta}u$ is orthogonal 
to ${\bf S}_{\bo_{p,n},R}^{\beta^\top}\Theta_{p,n,\beta}u^\prime$. Orthogonality
of ${\bf S}_{\bo_{p,n},R}^{(\gamma \beta)^\top}\Theta_{p,n,\beta}u$ to 
${\bf S}_{\bo_{p,n},R}^{\beta^\top}\Theta_{p,n,\beta}u^\prime$ is verified in much the same way.
Yet another inner-product computation,
\begin{align*}
&\| {\bf S}_{\bo_{p,n},R}^{\beta^\top}\Theta_{p,n,\beta}u\|^2_{H^2_{\bo_{p,n},\cY}(\free)}\\
&=\omega_{p,n;\beta}^{-1}\cdot \| D_\beta u\|^2_{\cY}+\sum_{j=1}^d
\sum_{\alpha\in\free}\omega_{p,n;\alpha j\beta}^{-1}\big\|
C\bA^{\alpha}B_{j,\beta}u\big\|^2_{\cX}\\
&=\omega_{p,n;\beta}^{-1}\cdot \| D_\beta u\|^2_{\cY}+
\bigg\langle \sum_{j=1}^d B_{j,\beta}^*\bigg(
\sum_{\alpha\in\free}\omega_{p,n;\alpha j\beta}^{-1}
\bA^{*\alpha^{\top}}C^*C\bA^{\alpha}\bigg)B_{j,\beta}u, \; u\bigg\rangle_{\cU}\\
&=\bigg\langle \bigg(\omega_{p,n;\beta}^{-1}\cdot D_\beta^*D_\beta
+\sum_{j=1}^d B_{j,\beta}^*\Gr_{p,n,j\beta;C,\bA}B_{j,\beta}\bigg)u, \; u\bigg\rangle_{\cU}=\|u\|^2_{\cU}
\end{align*}
(where the last equality is justified by the equality of $(2,2)$-entries in \eqref{isom-p}) shows that 
the operator ${\bf S}_{\bo_{p,n},R}^{\beta^\top}M_{\Theta_{p,n,\beta}}$ is an isometry
from $\cU_{\beta}$ into $H^2_{\bo_{p,n},\cY}(\free)$. Combining this fact with part (2) of the lemma, we conclude 
by Definition \ref{D:8in}, that $\Theta_{p,n,\beta}(z)z^\beta$ is $(p,n)$-Bergman inner.

\smallskip

Finally, substituting the representation \eqref{apr22} into the left hand side of \eqref{jul18-p}
and making use of \eqref{isom-p} we get
\begin{align*}
&\omega_{p,n;\beta}^{-1}I_{\cU_{\beta}} -
    \big( \omega_{p,n;\beta}^{-1}D_{\beta}^{*} + \widehat{B}_\beta^*\widehat{Z}(z)^*
    \widehat{R}_{p,n;\beta}(z\bA)^*C^*
\big)\big(\omega_{p,n;\beta}^{-1}D_{\beta}+ C\widehat{R}_{p,n;\beta}(\zeta\bA)\widehat{Z}(\zeta)
\widehat{B}_\beta\big)  \notag\\
&   = \omega_{p,n;\beta}^{-1}I_{\cU_{\beta}}-\omega_{p,n;\beta}^{-2}D_{\beta}^{*}D_{\beta}
- \omega_{p,n;\beta}^{-1}\widehat{B}_\beta^*\widehat{Z}(z)^*\widehat{R}_{p,n;\beta}(z\bA)^*C^*D_\beta
\notag\\
&\quad-\omega_{p,n;\beta}^{-1}D_\beta^*C \widehat{R}_{p,n;\beta}(\zeta\bA)\widehat{Z}(\zeta)
\widehat{B}_\beta
- \widehat{B}_\beta^*\widehat{Z}(z)^*\widehat{R}_{p,n;\beta}(z\bA)^*C^*C
\widehat{R}_{p,n;\beta}(\zeta\bA)
\widehat{Z}(\zeta)\widehat{B}_\beta\notag\\
    &  = \omega_{p,n;\beta}^{-1} \widehat{B}_{\beta}^{*}\widehat\Gr_{p,n,\beta;C,\bA}
    \widehat{B}_{\beta}
+\widehat{B}_\beta^*\widehat{Z}(z)^*\widehat{R}_{p,n;\beta}(z\bA)^* A^{*}\widehat\Gr_{p,n,\beta;C,\bA}
\widehat{B}_{\beta}\notag\\
&\quad + \widehat{B}_{\beta}^{*} \widehat\Gr_{p,n,\beta;C,\bA} A
\widehat{R}_{p,n;\beta}(\zeta\bA) \widehat{Z}(\zeta)\widehat{B}_\beta\notag\\
     & \quad - \omega_{p,n;\beta} \widehat{B}_\beta^*\widehat{Z}(z)^*
     \widehat{R}_{p,n;\beta}(z\bA)^*
(\Gr_{p,n,\beta;C,\bA}-A^*\widehat\Gr_{p,n,\beta;C,\bA}A)\widehat{R}_{p,n;\beta}(\zeta\bA)
\widehat{Z}(\zeta)\widehat{B}_\beta.
\end{align*}
which is the same as the expression on the right side of \eqref{jul18-p}.
\end{proof}

As a corollary we obtained the following analogue of Corollary \ref{C:5.3}

\begin{corollary}  \label{C:5.3-p}  Assume that $(C, \bA)$ and $B_{1, \beta}, \dots, B_{d, \beta} \in \cL(\cU_\beta, \cX)$, $D_\beta \in \cL(\cU_\beta, \cY)$
are as in the hypotheses of Lemma \ref{L:5.1-p} with associated family of transfer functions $\Theta_{p,n \bU_\beta}$ as in \eqref{1.37pre-p}.
Then the representation \eqref{1.36pre-p} is orthogonal in the metric of $H^2_{\bo_{p,n}, \cY}(\cF)$ and we have
\begin{align*}
\| \widehat y \|^2_{H^2_{\bo_{p,n}}(\free)} & = \| \cO_{\bo_{p,n}, C, \bA} x \|^2_{H^2_{\bo_{p,n}, \cY}(\free)} +
\| \bS_{\bo_{p,n}, R} \Theta_{\bo_{p,n}, \bU_\beta} u_\beta \|^2_{H^2_{\bo_{p,n}, \cY}}  \\
& = \| \cG^{\frac{1}{2}}_{\bo_{p,n}, c, \bA} x \|^2_\cX + \sum_{\beta \in \free} \| u_\beta \|^2_{\cU_\beta}.
\end{align*}
\end{corollary}

We now get back to closed ${\bf S}_{\bo_{p,n},R}$-invariant subspaces of $H^2_{\bo_{p,n},\cY}(\free)$.
A careful inspection of the proof of Lemma \ref{L:Mdecom} shows that the same proof goes through for any 
closed ${\bf T}$-invariant subspace of a Hilbert space $\cH$ for any operator tuple ${\bf T}=(T_1,\ldots,T_d)$
of bounded operators with orthogonal ranges. Since the right coordinate-variable multipliers 
$\bS_{\bo_{p,n},R,1},\ldots,\bS_{\bo_{p,n},R,d}$ on $H^2_{\bo_{p,n},\cY}(\free)$ have mutually orthogonal ranges,
we have, in particular, the following result.  

\begin{lemma} \label{L:Mdecom-p}
If a closed subspace $\cM \subset H^2_{\bo_{p,n},\cY}(\free)$ is ${\mathbf S}_{\bo_{p,n},R}$-invariant, then 
\begin{equation}
 \cM = \bigoplus_{\beta\in\free} \cM_\beta,\quad\mbox{where}\quad 
\cM_\beta:=\bS_{\bo_{p,n}, R}^{\beta^{\top}} \cM \ominus \bigg(\bigoplus_{j=1}^d
\bS_{\bo_{p,n},R}^{\beta^{\top}}\bS_{\bo_{p,n},R,j} \cM  \bigg).
\label{Mdecom-p}
\end{equation}
Furthermore, for any $\alpha \in \free$, we have the orthogonal direct-sum decomposition
$$
    \bS_{\bo_{p,n},R}^{\alpha^{\top}} \cM = \bigoplus_{\beta \in \free}
    \cM_{\beta \alpha}.
$$
 \end{lemma}
 
Given a  ${\bf S}_{\bo_{p,n},R}$-invariant closed subspace $\cM$ of $H^2_{\bo_{p,n},\cY}(\free)$, we 
let $(C,\bA)$ be a $(p,n)$-isometric pair such that $\cM^{\perp} =\operatorname{Ran} \cO_{p,n; C,\bA}$
(by Theorem \ref{T:2-1.2'-p}, we can let $(C,\bA) = (E|_{\cM}, {\bf S}_{\bo_{p,n},R}^{*}|_{\cM})$). 
Thus, $\cM^{\perp}$ is the NFRKHS with reproducing kernel \eqref{kerkmp-p}. 
Since  $k_{{\rm nc},\bo_{p,n}}(z,\zeta)$ \eqref{kbo-p} is the reproducing kernel for
$H^2_{\bo_{p,n}}(\free)$, it  follows that $\cM$ has reproducing kernel
\begin{equation}
\label{kMa-p}
k_{\cM}(z, \zeta) =k_{{\rm nc},\bo_{p,n}}(z,\zeta)I_{\cY}- C R_{p,n}(z\bA)\cG_{p,n;C,\bA}^{-1}R_{p,n}(\zeta\bA)^*C^*.
 \end{equation}
The next step is to characterize the spaces ${\bf S}_{\bo_{p,n},R}^{\beta^{\top}}\cM$ (or their orthogonal complements)
in terms of the chosen pair $(C,\bA)$.
\begin{proposition}   \label{P:SMperp-p}
    The space $({\bf S}_{\bo_{n,p},R}^{\beta^{\top}} \cM)^{\perp}$ is characterized as
 \begin{equation}   \label{SMperp-p}
     \big({\bf S}_{\bo_{p,n},R}^{\beta^{\top}}\cM\big)^{\perp} =
\big({\bf S}_{\bo_{n.p},R}^{\beta^{\top}} H^2_{\bo_{n,p},\cY}(\free)\big)^\perp\bigoplus
    {\bf S}_{\bo_{n,p},R}^{\beta^{\top}} \operatorname{Ran} {\Ob}_{p,n,\beta;C,\bA}
 \end{equation}
where the shifted gramian $\operatorname{Ran} {\Ob}_{p,n,\beta;C,\bA}$ is defined as in \eqref{4.31-p}.
 \end{proposition}
\begin{proof}
Following the strategy of the proof of Proposition \ref{P:SMperp}, we write a power series 
$f\in H^2_{\bo_{p,n},\cY}(\free)$ as 
$$
f(z)=p(z)+\widetilde{f}(z)z^\beta\quad\mbox{with}\quad p\in  
\left({\bf S}_{\bo_{p,n},R}^{\beta^{\top}} H^2_{\bo_{p,n},\cY}(\free)\right)^\perp\subset {\bf
S}_{\bo_{p,n},R}^{\beta^{\top}}H^2_{\bo_{p,n},\cY}(\free)
$$
and then characterize power series $\widetilde{f}\in H^2_{\bo_{p,n},\cY}(\free)$ such that 
${\bf S}_{\bo_{p,n},R}^{\beta^{\top}}\widetilde f$ is orthogonal to ${\bf
S}_{\bo_{p,n},R}^{\beta^{\top}} \cM$, or equivalently, such that  
${\bf S}^{*\beta}_{\bo_{p,n},R}{\bf S}_{\bo_{p,n},R}^{\beta^{\top}}\widetilde{f}$ is orthogonal to $\cM$, that is,
belongs to $\cM^\perp= \operatorname{Ran}\cO_{p,n;C,\bA}$. Since by \eqref{SRLv-p},
$$
{\bf S}^{*\beta}_{\bo_{p,n},R}{\bf S}_{\bo_{p,n},R}^{\beta^{\top}}: \; \sum_{\alpha\in\free}
\widetilde{f}_\alpha z^\alpha\to \sum_{\alpha\in\free}
\frac{\omega_{p,n;\alpha\beta}}{\omega_{p,n;\alpha}}\, \widetilde{f}_\alpha z^\alpha,
$$
we conclude that ${\bf S}_{\bo_{p,n},R}^{\beta^{\top}}\widetilde{f}$ is orthogonal to
${\bf S}_{\bo_{p,n},R}^{\beta^{\top}} \cM$ if and only if 
$$
\sum_{\alpha\in\free}
\frac{\omega_{p,n;\alpha\beta}}{\omega_{p,n;\alpha}}
\, \widetilde{f}_\alpha z^\alpha=(\cO_{p,n;,C,\bA}x)(z)=
\sum_{\alpha\in\free}\big(\omega_{p,n;\alpha}^{-1}\cdot C\bA^\alpha x\big)z^\alpha.
$$
for some vector $x\in\cX$. Equating the corresponding Taylor coefficients in the latter equality gives
$\widetilde{f}_\alpha=\omega_{p,n;\alpha\beta}^{-1}\cdot C\bA^\alpha x$ for all $\alpha\in\free$, and therefore,
$$
\widetilde{f}(z)=\sum_{\alpha\in\free}\widetilde{f}_\alpha z^\alpha
=\sum_{\alpha\in\free}\big(\omega_{p,n;\alpha\beta}^{-1}\cdot
C\bA^\alpha x\big)z^\alpha={\Ob}_{p,n,,\beta;C,\bA}x,
$$
by \eqref{4.31}. Thus, $\widetilde{f}\in\operatorname{Ran} {\Ob}_{p,n,\beta;C,\bA}$.
As the analysis is necessary and sufficient, the formula \eqref{SMperp-p} follows.
\end{proof}
   \begin{lemma}  \label{L:6.2-p}  Let $\cM$ be a closed ${\bf S}_{\bo_{p,n},R}$-invariant
subspace of $H^2_{\bo_{p,n},\cY}(\free)$ with reproducing kernel
$k_{\cM}$ given by \eqref{kMa-p}. Then for every $\beta \in\free$, the formal  reproducing kernel 
for the space $\cM_\beta$ (defined in \eqref{Mdecom-p}) in the metric of $H^2_{\bo_{p,n},\cY}(\free)$ is given by
\begin{equation}
 k_{\cM_\beta}(z,\zeta)= z^{\beta} \bzeta^{\beta^{\top}}
\omega_{p,n;\beta}^{-1}I_{\cY}-\boldsymbol{\mathfrak K}_{p,n;\beta}(z, \zeta)+
\sum_{j=1}^d \boldsymbol{\mathfrak K}_{p,n;j\beta}(z, \zeta),\label{kdif-p}
\end{equation}
where $\boldsymbol{\mathfrak{K}}_{p,n;\beta}$ is the kernel defined as in \eqref{deffrakk-p}.
\end{lemma}
\begin{proof}
Taking orthogonal complements in the equality \eqref{SMperp-p} gives
\begin{equation}
{\bf S}_{\bo_{p,n},R}^{\beta^{\top}}\cM  
= \big({\bf S}_{\bo_{p,n},R}^{\beta^{\top}}H^2_{\bo_{p,n},\cY}(\free)\big)\ominus
    \big({\bf S}_{\bo_{p,n},R}^{\beta^{\top}} \operatorname{Ran} {\Ob}_{p,n,\beta;C,\bA}\big).
    \label{SM-p}
\end{equation}
Since for any $\beta\in\free$, the reproducing kernel for the space 
${\bf S}_{\bo_{p,n},R}^{\beta^{\top}}H^2_{\bo_{p,n}}(\free)$
equals
\begin{equation}
k_{{\rm nc},p,n;\beta}(z,\zeta):=\sum_{\alpha\in\free}\omega_{p,n;\alpha\beta}^{-1}
 z^{\alpha \beta} \bzeta^{ (\alpha \beta)^{\top}}=k_{{\bf S}_{\bo_{p,n},R}^{\beta^{\top}}H^2_{\bo_{p,n}}(\free)}(z,\zeta),
\label{SMa-p}
\end{equation}
and since the reproducing kernel for the space 
${\bf S}_{\bo_{p,n},R}^{\beta^{\top}} \operatorname{Ran} {\Ob}_{p,n,\beta;C,\bA}$ equals 
$\boldsymbol{\mathfrak{K}}_{p,n;\beta}$ (by Lemma \ref{L:sign}), we conclude from the formula
\eqref{SM-p} that the the reproducing kernel for the space ${\bf S}_{\bo_{p,n},R}^{\beta^{\top}}\cM$
(in the metric of $H^2_{\bo_{p,n}}(\free)$ equals
\begin{equation}
k_{{\bf S}_{\bo_{p,n},R}^{\beta^{\top}} \cM}(z,\zeta) = k_{{\rm nc},p,n;\beta}(z,\zeta)I_{\cY}
-\boldsymbol{\mathfrak K}_{n,p;\beta}(z, \zeta).\label{kscm-p} 
\end{equation}
Replacing $\beta$ by $j\beta$ in \eqref{kscm-p} gives
\begin{equation}
k_{{\bf S}_{\bo_{p,n},R}^{\beta^{\top}}S_{\bo_{p,n},R,j}\cM}(z,\zeta) = k_{{\rm nc},p,n;j\beta}(z,\zeta)I_{\cY}
-\boldsymbol{\mathfrak K}_{n,p;j\beta}(z, \zeta).\label{kscm-pp}
\end{equation}
Using the reproducing kernels \eqref{kscm-p} and \eqref{kscm-pp}, we conclude from the orthogonal representation
\eqref{Mdecom-p} for $\cM_\beta$ that the reproducing kernel for $\cM_\beta$ equals
\begin{align*}
k_{\cM_\beta}(z,\zeta)&=k_{{\bf S}_{\bo_{p,n},R}^{\beta^{\top}} \cM}(z,\zeta)-\sum_{j=1}^d
k_{{\bf S}_{\bo_{p,n},R}^{\beta^{\top}}S_{\bo_{p,n},R,j}\cM}(z,\zeta)\\
&=k_{{\rm nc},p,n;\beta}(z,\zeta)I_{\cY}
-\boldsymbol{\mathfrak K}_{n,p;\beta}(z, \zeta)-\sum_{j=1}^d 
\big(k_{{\rm nc},p,n;j\beta}(z,\zeta)I_{\cY}
-\boldsymbol{\mathfrak K}_{n,p;j\beta}(z, \zeta)\big).
\end{align*}
We finally observe from \eqref{SMa-p} that
\begin{align*}
&k_{{\rm nc},p,n;\beta}(z,\zeta)-\sum_{j=1}^d k_{{\rm nc},p,n;(j\beta)}(z,\zeta)\\
&\quad =\sum_{\alpha\in\free}  \omega_{p,n;\alpha \beta}^{-1} z^{ \alpha \beta }
\bzeta^{ (\alpha \beta)^{\top}}-\sum_{j=1}^d\sum_{\alpha\in\free} \omega_{p,n;\alpha j\beta}^{-1}
z^{ \alpha j \beta} \bzeta^{(\alpha j \beta)^{\top} }   \notag \\
&\quad =\sum_{\alpha\in\free}  \omega_{p,n;\alpha\beta}^{-1} z^{ \alpha \beta }
\bzeta^{ (\alpha \beta)^{\top}}-\sum_{\alpha\in\free \colon \alpha \ne \emptyset}
\omega_{p,n;\alpha\beta|}^{-1}  z^{ \alpha \beta }
\bzeta^{ (\alpha \beta)^{\top}}= \omega_{\beta}^{-1}z^\beta \bzeta^{\beta^\top}.\notag
\end{align*}
Combining the two latter equalities completes the proof of \eqref{kdif-p}.
\end{proof}
We next factor the kernel \eqref{kdif-p} as follows.
\begin{lemma}
Given an exactly $(p,n)$-observable and $(p,n)$-output-stable pair $(C,\bA)$ with $C\in\cL(\cX,\cY)$ and 
$\bA=(A_1,\ldots,A_d)\in\cL(\cX)^d$, let 
$\sbm{\widehat B_\beta \\ D_\beta}: \, \cU_\beta\to \cX^d\oplus\cY$ (with $\widehat B_\beta=\sbm{B_{1,\beta}\\
\vdots \\B_{d,\beta}}$) be an injective solution to the 
Cholesky factorization problem:
  \begin{equation}  \label{pr6-p}
\begin{bmatrix}\widehat{B}_{\beta} \\
    D_\beta \end{bmatrix}\begin{bmatrix}\widehat{B}_{\beta}^* & D_\beta^*\end{bmatrix}=
\begin{bmatrix} \widehat{\Gr}_{p,n,\beta;C,\bA}^{-1} & 0 \\ 0 &
        \omega_{p,n;\beta} I_{\cY} \end{bmatrix}-
\begin{bmatrix} A \\ C\end{bmatrix} \Gr_{p,n,\beta;C,\bA}^{-1} 
        \begin{bmatrix} A^{*} & C^{*}\end{bmatrix},
\end{equation}
(where $\Gr_{p,n,\beta;C,\bA}$ and $\widehat{\Gr}_{p,n,\beta;C,\bA}$ are given by \eqref{4.32-p} and 
\eqref{blockgram}) and let $\Theta_{p,n, \bU_\beta}(z)$ be defined as in \eqref{apr22}. Then 
\begin{enumerate}
\item The equality \eqref{isom-p} holds and hence, $\Theta_{p,n,\bU_\beta}(z)$ is $(p,n)$-Bergman inner.
\item Then the kernel \eqref{kdif-p} can be factored as
\begin{equation}
k_{\cM_\beta}(z,\zeta)=\Theta_{n,p,\bU_\beta}(z) \big(z^{\beta}
\bzeta^{\beta^{\top}} I_{\cU_{\beta}}\big)\Theta_{n,p, \bU_\beta}(\zeta)^{*}.
\label{id6-p}
\end{equation}
\item The operator $M_{\Theta_{n,p, \bU_\beta}} \colon  z^{\beta} u \mapsto
\Theta_{n,p, \bU_\beta}(z) z^{\beta} u $ is unitary from $z^{\beta} \cU_{\beta}$
(considered as a Hilbert space with lifted norm 
$\|   z^{\beta} u \|_{ z^{\beta} \cU_{\beta}} = \| u \|_{\cU_{\beta}}$)
onto $\cM_{\beta}$.
\end{enumerate}
\label{L:rest}
\end{lemma}

\begin{proof}
The existence of an injective solution to the factorization problem \eqref{pr6-p} and verification of the 
equality \eqref{isom-p} goes through by the arguments used in the proof of Lemma \ref{L:5.6}. 
The multiplier $\Theta_{p,n, \bU_\beta}$ is $(p,n)$-Bergman inner, by Lemma \ref{L:5.1-p}. Making use 
of the formula \eqref{apr22} and taking into account equalities of the corresponding 
blocks in \eqref{wghtcoisom-p}, we get
 \begin{align}
&\omega_{p,n;\beta}^{-1}I_{\cY} - \Theta_{p,n, \bU_\beta}(z)\Theta_{p,n, \bU_\beta}(\zeta)^{*}\notag\\
&= \omega_{p,n;\beta}^{-1}I_{\cY}-
\big(\omega_{p,n;\beta}^{-1}D_{\beta}+ C\widehat{R}_{p,n;\beta}(z\bA)\widehat{Z}(z)\widehat{B}_\beta\big)
\big(\omega_{p,n;\beta}^{-1}D_{\beta}^{*}+\widehat{B}_\beta^*\widehat{Z}(\zeta)^*\widehat{R}_{p,n;\beta}
(\zeta\bA)^*C^*\big) \notag\\
& = \omega_{p,n;\beta}^{-1}I_{\cY}-\omega_{p,n;\beta}^{-2}D_\beta
D_\beta^*- \omega_{p,n;\beta}^{-1}C\widehat{R}_{p,n;\beta}(z\bA)\widehat{Z}(z)\widehat{B}_\beta
D_{\beta}^{*}\notag\\
&\quad-\omega_{p,n;\beta}^{-1}D_\beta\widehat{B}_\beta^*\widehat{Z}(\zeta)^*
\widehat{R}_{p,n;\beta}(\zeta\bA)^*C^*
     - C\widehat{R}_{p,n;\beta}(z\bA)\widehat{Z}(z)\widehat{B}_\beta
\widehat{B}_\beta^*\widehat{Z}(\zeta)^*\widehat{R}_{p,n;\beta}(\zeta\bA)^*C^*\notag\\
     & = \omega_{p,n;\beta}^{-2}C\Gr_{p,n,\beta;C,\bA}^{-1}C^* +\omega_{p,n;\beta}^{-1}
C\widehat{R}_{p,n;\beta}(z\bA)\widehat{Z}(z)A\Gr_{p,n,\beta;C,\bA}^{-1}C^*\notag\\
 & \quad + \omega_{p,n;\beta}^{-1}C\Gr_{p,n,\beta;C,\bA}^{-1}A^*\widehat{Z}(\zeta)^*
 \widehat{R}_{p,n;\beta}(\zeta\bA)^*C^*\notag\\
     & \quad  - C\widehat{R}_{p,n;\beta}(z\bA)\widehat{Z}(z)
\left[ \widehat{\Gr}_{p,n,\beta;C,\bA}^{-1} - A
\Gr_{p,n,\beta;C,\bA}^{-1}  A^{*}\right]\widehat{Z}(\zeta)^*
\widehat{R}_{p,n;\beta}(\zeta\bA)^*C^*\notag\\
&=C\big(\omega_{p,n;\beta}^{-1}I_{\cX}+\widehat{R}_{p,n;\beta}(z\bA)\widehat{Z}(z)A\big)
\Gr_{p,n,\beta;C,\bA}^{-1}
\big(\omega_{p,n;\beta}^{-1}I_{\cX}+A^*\widehat{Z}(\zeta)^*
\widehat{R}_{p,n;\beta}(\zeta\bA)^* \big) C^* \notag\\
&\quad-C\widehat{R}_{p,n;\beta}(z\bA)\widehat{Z}(z)
\widehat{\Gr}_{p,n,\beta;C,\bA}^{-1}
\widehat{Z}(\zeta)^*\widehat{R}_{p,n;\beta}(\zeta\bA)^*C^*\notag\\
&=CR_{p,n;\beta}(z\bA)\Gr_{p,n,\beta;C,\bA}^{-1}R_{p,n;\beta}(\zeta\bA)^*C^*\notag\\
&\qquad-C\widehat{R}_{p,n;\beta}(z\bA)\widehat{Z}(z)
\widehat{\Gr}_{p,n,\beta;C,\bA}^{-1}
\widehat{Z}(\zeta)^*\widehat{R}_{p,n;\beta}(\zeta\bA)^*C^*,\label{rest8}
     \end{align}
where for the last step we used the equality
$$
\omega_{p,n;\beta}^{-1}I+\widehat{R}_{p,n;\beta}(z\bA)\widehat{Z}(z)A
=\omega_{p,n;\beta}^{-1}I_{\cX}+\sum_{j=1}^d R_{p,n;j\beta}(z\bA)z_jA_j=R_{p,n;\beta}(z\bA),
$$   
which follows from \eqref{march20a} by definitions \eqref{march20ab}. We next observe from
definitions \eqref{blockgram} and again \eqref{march20ab} that
\begin{align*}
&\widehat{R}_{p,n;\beta}(z\bA)\widehat{Z}(z)
\widehat{\Gr}_{p,n,\beta;C,\bA}^{-1}
\widehat{Z}(\zeta)^*\widehat{R}_{p,n;\beta}(\zeta\bA)^*\\
&\qquad=\sum_{j=1}^d R_{p,n;j\beta}(z\bA)z_j{\Gr}_{p,n,j\beta;C,\bA}^{-1}\overline{\zeta}_j R_{p,n;j\beta}(\zeta\bA)^*
\end{align*}
Substituting the latter equality into \eqref{rest8} and then multiplying both sides in \eqref{rest8} by
$z^\beta$ on the right and by $\overline{\zeta}^{\beta^\top}$ on the left, we get
\begin{align*}
&\omega_{p,n;\beta}^{-1}z^{\beta}
\bzeta^{\beta^{\top}}I_{\cY} - \Theta_{n,p, \bU_\beta}(z) \big(z^{\beta}
\bzeta^{\beta^{\top}} I_{\cU_{\beta}}\big)\Theta_{n,p, \bU_\beta}(\zeta)^{*}\\
&=CR_{p,n;\beta}(z\bA) \big(z^{\beta} \bzeta^{\beta^{\top}}
\Gr_{p,n,\beta;C,\bA}^{-1}\big)  R_{p,n,\beta}(\zeta \bA)^*C^* \\
&\qquad-
\sum_{j=1}^d R_{p,n;j\beta}(z\bA) \big(z^{\beta j} \bzeta^{(\beta j)^{\top}}{\Gr}_{p,n,j\beta;C,\bA}^{-1}\big)
R_{p,n;j\beta}(\zeta\bA)^*\\
&=\boldsymbol{\mathfrak K}_{p,n;\beta}(z, \zeta)-\sum_{j=1}^d \boldsymbol{\mathfrak K}_{p,n;j\beta}(z, \zeta),
\end{align*}
where we used the definition \eqref{deffrakk-p} for the last step. Comparing the latter formula with \eqref{kdif-p}
we conclude \eqref{id6-p}. The last statement of the lemma is justified exactly as in Lemma \ref{L:6.8}.
\end{proof}

\section[$*$-$(p,n)$-hypercontractive model theory]{Operator model theory for c.n.c. $*$-$(p,n)$-hypercontractive operator tuples $\bT$}  \label{S:model-pn}

In this brief final section we consider a model theory for $*$-$(p,n)$-hypercontractive operator tuples analogous to that developed in
Chapter 8 for the case of $*$-$\bo$-hypercontractive tuples.  As in Chapter 8 there are two flavors: (i) model theory with
characteristic function equal to a contractive multiplier, and (ii) model theory based on Bergman-inner families.

\subsection{ $(p,n)$-model theory based on contractive multipliers}
We now consider operator $d$-tuples $\bT = (T_1, \dots, T_d)$ on a Hilbert space $\cX$ such that the operator $d$-tuple
$\bA = \bT^* = (T_1^*, \dots, T_d^*)$ is a $(p,n)$-hypercontractive, or equivalently, so that $\bT^*$ is $(p,1)$- and $(p,n)$-contractive
(see Definition \ref{D:10.1} and Corollary \ref{C:10.2}).  Thus in particular $\Gamma_{p,n; \bT^*}[I_\cX] \succeq 0$.  We define the $(p,n)$-defect
operator $D_{p,n, \bT^*}$ of $\bT^*$ and the $(p,n)$-defect space $\cD_{p,n, \bT^*}$ by
\begin{equation}  \label{pn-defect}
D_{p,n, \bT^*} = \Gamma_{p,n}[I_\cX]^{\frac{1}{2}}, \quad \cD_{p,n \bT^*} = \overline{\operatorname{Ran}} \, D_{p,n \bT^*} \subset \cX.
\end{equation}
Then by construction the pair $(D_{p,n, \bT^*}, \bT^*)$ is a $(p,n)$-isometric output pair (see Definition \ref{D:def-pn}).  
We next form the $(p,n)$-observability operator $ \cO_{p,n; D_{p,n, \bT^*}, \bT^*}$ as in \eqref{march21}.  As a consequence of
part (2) of Theorem \ref{T:2-1.2'-p}, $\cO_{p,n; D_{p,n \bT^*}, \bT^*}$ maps $\cX$ contractively into $H^2_{\bo_{p,n},  \cD_{p,n, \bT^*}}(\free)$,
so in particular the output pair $(D_{p,n, \bT^*}, \bT^*)$ is also $(p,n)$ stable in the sense of the definition given at the beginning of
Section \ref{S:obs-gram-p}.

We shall assume in addition that $\bT$ is {\em $(p,n)$-completely noncoisometric} (or {\em $(p,n)$-c.n.c.}\ for short), by which we mean that the
observability operator $\cO_{p,n, D_{n,p, \bT^*}, \bT^*}$ has trivial kernel.  Thus we may define a norm on the space
$\cN: = \operatorname{Ran} \cO_{\bo, D_{\bo, \bT^*}, \bT^*}$ as the {\em lifted norm} defined by
$$
  \| \cO_{\bo, D_{\bo, \bT^*}, \bT^*} x \|_\cN = \| x \|_\cX.
$$
Then as a consequence of Theorem \ref{T:2-1.2'-p}, $\cN$ has the structure of a NFRKHS $\cN = \cH(K_\cN)$ with kernel $K_\cN$ given by
$$
  K_\cN(z, \zeta) = D_{p,n, \bT^*} R_{p,n}(Z(z) T^*) R_{p,n}(Z(\zeta) T^*)^* D_{p,n, \bT^*}.
$$
$\cN$ is $\bS_{\bo_{p,n}, R}^*$-invariant, and the observability operator $\cO_{\bo, D_{\bo, \bT^*}, \bT^*}$ implements a unitary equivalence
between $\bT^*$ and $\bS_{\bo_{p,n}, R}^*|_\cN$.

Let us make the following formal definition.

\begin{definition}  \label{D:char-func-pn}
We shall say that the $(p,n)$-c.n.c.~$*$-$(p,n)$-hypercontractive tuple $\bT$ {\em admits a characteristic multiplier $\Theta_\bT$}
if the Brangesian complement $\cM = \cN^{[\perp]}$ of $\cN = \operatorname{Ran} \cO_{p,n;, D_{\bT^*}, \bT^*}$ (equipped
with the lifted norm) admits a Beurling-Lax representation $\cM = \Theta \cdot H^2_{\cU}(\free)$ as in Theorem \ref{T:NC-BL-p}.
We then say that $\Theta$ is a {\em characteristic  multiplier} $\Theta_\bT$ for $\bT$.
\end{definition}

With this definition in hand it is possible to formulate $(p,n)$-analogues of Theorem \ref{T:char-func} and \ref{T:Theta-char}.  We leave
the details to the interested reader.

\subsection{$(p,n)$-model theory based on Bergman-inner families}
This flavor of model theory for the moment handles only the class of $*$-$(p,n)$-hypercontractive tuples $\bT = (T_1, \dots, T_d)$
with $\bT^*$ $p$-strongly stable.  In this case we define the $(p,n)$-defect operator $D_{p,n, \bT^*}$ as in \eqref{pn-defect}
and form the observability operator $\cO_{p,n, D_{p,n, \bT^*}, \bT^*}$.  As a consequence of part (3e) in Theorem \ref{T:2-1.2'-p}
$\cO_{p,n, D_{p,n, \bT^*}, \bT^*}$ is isometric and $\cN; = \operatorname{Ran} \cO_{p,n, D_{p,n, \bT^*}, \bT^*}$
sits isometrically in $H^2_{\bo_{p,n}, \cD_{p,n, \bT^*}}(\free)$ as a closed $\bS_{p,n,R}^*$-invariant subspace.
Then $\cM : = \cN^\perp$ is $\bS_{p,n,R}$-invariant and, assuming that $(p,n)$ satisfy the mild restriction \eqref{rest}, 
 we can represent $\cM$ via a $(p,n)$-Bergman-inner family $\{ \Theta_{\bT, \beta}   \colon \beta \in \free \}$ as in Section \ref{S:binnerp}.
 This $(p,n)$-Bergman-inner family we declare to be the {\em characteristic Bergman-inner family} for $\bT$ since it serves as a complete
 unitary invariant (up to natural identifications) for the original $p$-strongly stable, $*$-$(p,n)$-hypercontractive $d$-tuple $\bT$.  With the 
 Beurling-Lax representation theorems of Section \ref{S:BL-p} replacing those of Section \ref{S:BL7}, we get a direct analog of Theorem
 \ref{T:charfuncT}.  Furthermore one can check that the structures are sufficiently parallel that one can derive a $(p,n)$-analog of the realization
 formula \eqref{canUv} for a given $(p,n)$-Bergman-inner family $\{ \Theta_\beta \colon \beta \in \free\}$.


\begin{thebibliography}{99}

\bibitem{Agler1982}  J.\ Agler, {\em The Arveson extension theorem and coanalytic models},  Integral Equations and Operator Theory
\textbf{5} (1982), 608--631.

\bibitem{aglerhyper}
J.~Agler, {\em Hypercontractions and subnormality}, J. Operator Theory {\bf 13} (1985), no. 2,
203--217.



\bibitem{AS1} J.~Agler and M.~Stankus, $m$-isometric transformations of Hilbert space I, {\em Integral Equ.\ Oper.\ Theory} \textbf{21}, 383--429 (1996).

\bibitem{AS2} J.~Agler and M.~Stankus, $m$-isometric transformations of Hilbert space II, {\em Integral Equ.\ Oper.\ Theory} \textbf{23},
1--48, (1995).

\bibitem{AS3}  J.~Agler and M.~Stankus, $m$-isometric transformations of Hilbert space III, {\em Integral Equ.\ Oper.\ Theory} \textbf{24},
379--421 (1996).


\bibitem{ars}
A.~Aleman, S.~Richter and C.~Sundberg, {\em Beurling's theorem for the Bergman space}, Acta
Math. {\bf 177} (1996), no. 2, 275--310.

\bibitem{AMc1} J.\ Agler and J.E.\ McCarthy, {\em Global holomorphic 
functions in several non-commuting variables}, Canad.\ J.\ Math. 
\textbf{67} (2) (2015), 241--285.

\bibitem{AMc2} J.\ Agler and J.E.\ McCarthy, {\em Pick interpolation for 
free holomorphic functions}, Amer.\ J.\ Math. \textbf{137} (6) 
(2015), 1685--1701.


\bibitem{AEM} C.-G.~Ambrozie, M.~Engli\v s and V.~M\"uller,
\emph{Operator tuples and analytic models over general domains in
${\mathbb C}^{n}$}, J.~Operator Theory \textbf{47} (2002), 287--302.

\bibitem{ArazyEnglis} J.~Arazy and M.~Engli\v s, \emph{Analytic
models for commuting operator tuples on bounded symmetric domains},
Trans.~Amer.~Math.~Soc. \textbf{355} (2003), no. 2, 837--864.

\bibitem{ArvesonIII} W.~Arveson, \emph{Subalgebras of $C^{*}$ algebras III:
Multivariable operator theory}, Acta Math. \textbf{181}
(1998), 159--228.

\bibitem{Athavale} A.~Athavale, {\em Model theory on the unit ball in 
${\mathbb C}^{m}$}, J.~Oper.~Theory \textbf{27} (1992), 347--358.


\bibitem{BB11}
J.A.~Ball and V.~Bolotnikov, {\em Canonical transfer-function realization for Schur 
multipliers on the Drury-Arveson space and models for commuting row contractions}, Indiana Univ. 
Math. J. {\bf 61} (2012), no. 2, 665--716.


\bibitem{BBIEOT} J.A.~Ball and V.~Bolotnikov, {\em Weighted Bergman spaces:
shift-invariant subspaces and input/state/output linear systems}, Integral Equations 
Operator Theory {\bf 76} (2013), no. 1, 301--356.

\bibitem{BBSzeged} J.A.~Ball and V.~Bolotnikov, {\em Weighted Hardy 
spaces: shift invariant and coinvariant subspaces, linear systems and 
operator model theory}, Acta Sci.~Math.~(Szeged) \textbf{79} (2013), 
623--686.

\bibitem{BBCR} 
J.A.~Ball and V.~Bolotnikov, {\em A Beurling type theorem in weighted Bergman spaces}, C. R. Math. 
Acad. Sci. Paris 351 (2013), no. 11-12, 433--436.


\bibitem{BBkarm} J.A.~Ball and V.~Bolotnikov, {\em On the expansive property of inner functions in weighted
Hardy spaces}, Contemporary Mathematics {\bf 653}, in:  Complex 
Analysis and Dynamical Systems VI Part 2, pp. 47--61, Contemp.\ Math. 
\textbf{667}, Israel Math.\ Conf. Proc., Amer.\ Math.\ Soc., 
Providence, 2016.

\bibitem{BBPAMS2017} J.A.~Ball and V.~Bolotnikov, {\em Contractive 
multipliers from Hardy space to weighted Hardy space}, Proc.\ 
Amer.\ Math.\ Soc. {\bf 145} (2017), no. 6, 2411--2425.

\bibitem{BB-HOT} J.A.\ Ball and V.\ Bolotnikov,  {\em de Branges-Rovnyak spaces: basics and theory}, in Operator Theory Vol. 1 (ed.\ D.\ Alpay)
pp.\ 631--679, Springer, 2015.

\bibitem{BBF1} J.A.~Ball, V.~Bolotnikov, and Q.~Fang, {\em Multivariable
backward-shift-invariant subspaces and observability operators},
Multidimens.~Syst.~Signal Process.~\textbf{18} (2007) no.~4, 191--248.



\bibitem{BBF3} J.A.~Ball, V.~Bolotnikov, and Q.~Fang, {\em 
Schur-class multipliers on the Fock space: de Branges-Rovnyak 
reproducing kernel spaces and transfer-function realizations}, in: 
Operator Theory, Structured Matrices, and Dilations: Tiberiu 
Constantinescu Memorial Volume (Ed.~M.Bakonyi, A.~Gheondea, 
M.~Putinar, and J.~Rovnyak), pp. 85--114, Theta Series in Advances 
Mathematics, Theta, Bucharest, 2007.



\bibitem{BGM1} J.A.~Ball, G.~Groenewald and T.~Malakorn,
\emph{Structured noncommutative multidimensional linear system},
SIAM J. Control Optim. \textbf{44} (2005), no. 4, 1474--1528.

\bibitem{BGM2} J.A.~Ball, G.~Groenewald and T.~Malakorn, \emph{Conservative
structured noncommutative multidimensional linear systems}, in:
 The State Space Method: Generalizations and Applications
(Ed. D. Alpay and I. Gohberg), pp. 179--223, \textbf{OT 161},
Birkh\"auser, Basel, 2006.

\bibitem{BallKriete} J.A.\ Ball and T.L.\ Kriete, {\em OPerator-valued Nevanlinna-Pick kernels and the functional models for contraction operators},
Integral Equations and Operator Theory \textbf{10} (1987), 17--61

\bibitem{BMV1} J.A.\ Ball, G.\ Marx, and V.\ Vinnikov, {\em 
Noncommutative reproducing kernel Hilbert spaces}, J.\ Funct.\ Anal. 
\textbf{271} (2016), 1844--1920.


\bibitem{PapaBear}  J.A.~Ball, D.S.~Kaliuzhnyi-Verbovetskyi, 
C.~Sadosky, and V.~Vinnikov, {\em Scattering systems with several  
evolutions and formal reproducing kernel Hilbert spaces}, 
Complex Anal. Oper. Theory {\bf 9} (2015), no. 4, 827--931.


\bibitem{NFRKHS} J.A.~Ball and V.~Vinnikov, \emph{Formal reproducing kernel
Hilbert spaces: the commutative and noncommutative settings}, in:
Reproducing Kernel Spaces and Applications (Ed.~D.~Alpay), pp.
77--134, \textbf{OT 143}, Birkh\"auser, Basel, 2003.

\bibitem{Cuntz-scat} J.A.~Ball and V.~Vinnikov, \emph{Lax-Phillips
scattering and conservative linear systems: a Cuntz-algebra
multidimensional setting},  Mem. Amer. Math. Soc.  {\bf 178} (2005), no. 837.

\bibitem{BEKS}  M.\ Bhattacharjee, J.\ Eschmeier, D.K.\ Keshari, and J.\ Sarkar, {\em Dilations,
wandering subspaces, and inner functions}, Linear Algebra and its Applications \textbf{523} (2017), 263--280.

\bibitem{BES} T.~Bhattacharyya, J.~Eschmeier and J.~Sarkar,
\emph{Characteristic function of a pure commuting contractive tuple},
{\em Integral Equations and Operator Theory} {\bf 53} (2005), no. 1, 23--32.

\bibitem{BES2} T.~Bhattacharyya, J.~Eschmeier and J.~Sarkar,
{\em On C.N.C.\ commuting contractive tuples}, Proc.\ Ind.\ Acad.\ Sci.\ Math.\ Sci.\ \textbf{116} (2006) no.\ 3, 299--316.


\bibitem{BS} T.~Bhattacharyya and J.~Sarkar,
\emph{Characteristic function for polynomially contractive
commuting tuples}, J.~Math.~Anal.~Appl. \textbf{321} No.~1 (2006),
242--259.

\bibitem{BlM}  D.P.\ Blecher and C.\ le Merdy, {\em Operator Algebras 
and Their Modules:  An Operator Space Approach}, London Mathematical 
Society Monographs New Series \textbf{30}, Oxford University Press, 
2004.


\bibitem{borhed}
A.~Borichev and  H.~Hedenmalm, {\em
Harmonic functions of maximal growth: invertibility and cyclicity in Bergman spaces},
J. Amer. Math. Soc. {\bf 10} (1997), no. 4, 761--796.



\bibitem{dBR1}
{L. de} Branges and J.~Rovnyak,
Canonical models in quantum scattering theory,
in: {\em Perturbation Theory and its
Applications in Quantum Mechanics} (C.~Wilcox, ed.) pp. 295--392,
{Holt, Rinehart and Winston, New York}, 1966.

\bibitem{dBR2}
L.~de Branges and J.~Rovnyak,
{\em Square summable power series},
Holt, Rinehart and Winston, New York, 1966.


\bibitem{Bunce} J.W.~Bunce, {\em Models for $n$-tuples of 
noncommuting operators}, J.~Funct.~Anal.~\textbf{57} (1984) no.~ 1, 
21--30.

\bibitem{Burb} F.~Beatrous and J.~Burbea, {\em Reproducing kernels and interpolation of 
holomorphic functions}, in: {\em Complex analysis, functional analysis and approximation theory}, 
pp. 25--46, North-Holland Math. Stud. {\bf 125}, North-Holland, Amsterdam, 1986.**



\bibitem{chen}
Y.~Chen, {\em Quasi-wandering subspaces in a class of reproducing analytic
Hilbert spaces}, Proc.
Amer. Math. Soc. {\bf 140} (2012), 4235--4242.

\bibitem{CV1993} R.E.\ Curto and F.H.\ Vasilescu, {\em Standard operator models in the polydisk},
Indiana Univ.\ Math.\ J. \textbf{42} (1993) no.\ 3, 791--810.

\bibitem{CV1995}   R.E.\ Curto and F.H.\ Vasilescu, {\em Standard operator models in the polydisk II},
Indiana Univ.\ Math.\ J. \textbf{44} (1995) no.\ 3, 727--746,



\bibitem{DP} K.~Davidson and D.~Pitts, {\em Nevanlinna-Pick interpolation
for non-commutative analytic Toeplitz algebras}, Integral
Equations and Operator Theory \textbf{31} (1998), no. 3, 321--337.

\bibitem{Davidson-Bordeaux}  K.R.~Davidson, {\em Free semigroup algebras:
a survey}, in: Systems, Approximation, Singular Integral
Operators, and Related Topics (ed. A.A.~Borichev and N.K.~Nikolski),
pp. 209--240, \textbf{OT 129}, Birkh\"auser, Basel, 2001.

\bibitem{douglas}
R.G.~Douglas, {\em On majorization, factorization, and range inclusion of operators on Hilbert 
space}, Proc. Amer. Math. Soc. {\bf 17} (1966), 413--415. **

\bibitem{DMS} R.G.\ Douglas, G.\ Misra, and J.\ Sarkar, {\em 
Contractive Hilbert modules and their dilations}, Israel J.\ Math. 
\textbf{187} (2012), 141--165.

\bibitem{DS2004} P.\ Duren and A.\ Schuster, {\em Bergman Spaces}, Mathematical Surveys and Monographs
\textbf{100}, Amer.\ Math.\ Soc., Providence (2004).

\bibitem{DulPag} 
 G.E.~Dullerud and F.~Paganini, {\em A Course in Robust Control
  Theory: A Convex Approach}, Texts in Applied Mathematics Vol.\ {\bf 36},
  Springer-Verlag, New York, 2000.

\bibitem{DKSS}
P.~ Duren, D.~Khavinson, H.S.~ Shapiro and  C.~Sundberg, {\em Contractive zero-divisors in
Bergman spaces}, Pacific J. Math. {\bf 157} (1993), no. 1, 37--56.

\bibitem{DKSS1}
P.~ Duren, D.~Khavinson, H.S.~ Shapiro and  C.~Sundberg, {\em
Invariant subspaces in Bergman spaces and
the biharmonic equation}, Michigan Math. J. {\bf 41} (1994), no. 2, 247--259.

\bibitem{Esch2018}  J.\ Eschmeier, {\em Bergman inner functions and $m$-hypercontractions}, J.\ Funct.\ Anal.\ \textbf{275} (2018) no.\ 1, 73--102.

\bibitem{EL} J.\ Eschmeier and S.\ Langend\"orfer, {\em Multivariable Bergman shifts and Wold decompositions},
Integral Equations and Operator Theory \textbf{90} (2018) no.\ 5, Art.\ 56, 17 pp.

\bibitem{Frazho} A.E.~Frazho, {\em Models for noncommuting operators},
J.~Funct.~Anal.~\textbf{48} (1982) no.~1, 1--11.

\bibitem{FtHK} A.E.\ Frazho, S.\ ter Horst, and M.A.\ Kaashoek, {\em 
All solutions to an operator Nevanlinna-Pick interpolation problem}, 
preprint, 2017.

\bibitem{gleason}
A.M.~Gleason, \emph{Finitely generated ideals in Banach algebras},
J. Math. Mech. {\bf 13} (1964), 125--132.

\bibitem{gisolof}
O.~Giselsson and A.~Olofsson, {\em On some Bergman shift operators}, 
Complex Anal. Oper. Theory {\bf 6} (2012), no. 4, 829--842. 

\bibitem{GvL} G.H.\ Golub and C.F.\ Van Loan, {\em Matrix 
Computations}, Johns Hopkins Series in Mathematical Sciences 
\textbf{3}, Johns Hopkins University Press, Baltimore, 1983.

\bibitem{halmos}
P.~R.~Halmos, {\em Shifts on Hilbert spaces,} J. Reine Angew. Math. 
{\bf 208} (1961), 102--112.

\bibitem{heden1}
H.~Hedenmalm, {\em A factorization theorem for square area-integrable analytic functions}, J.
Reine
Angew. Math. {\bf 422} (1991), 45--68.

\bibitem{heden2}
H.~Hedenmalm, {\em A factoring theorem for a weighted Bergman space}, Algebra i Analiz {\bf 4}
(1992),  no. 1, 167--176; translation in St. Petersburg Math. J. {\bf 4} (1993), no. 1, 163--174.

\bibitem{HKZ2000} H.\ Hedenmalm, B.\ Korenblum, and K.\ Zhu, {\em Theory of Bergman Spaces}, Graduate Texts in Mathematics \textbf{199},
Springer, Berlin (2000).

\bibitem{hedenzhu}
H.~Hedenmalm, K.~Zhu, {\em On the failure of optimal factorization for certain weighted Bergman
spaces},
Complex Variables Theory Appl. {\bf 1}9 (1992), no. 3, 165--176.


\bibitem{iziziz1}
K.~J.~Izuchi, K.~H.~Izuchi and Y.~Izuchi, {\em 
Wandering subspaces and the Beurling type Theorem I},
Arch. Math. {\bf 95} (2010), no. 5, 439--446.


\bibitem{iziziz}
K.~J.~Izuchi, K.~H.~Izuchi and Y.~Izuchi, {\em Quasi-wandering subspaces in
the Bergman space}, Integral Equations Operator Theory {\bf 67} (2010) 151--161.

\bibitem{K-VV}
D.\ Kaliuzhnyi-Verbovetskyi and V.\ Vinnikov, {\em Foundations of 
Free Noncommutative Function Theory}, Mathematical Surveys and 
Monographs Vol.\ \textbf{199}, Amer.\ Math.\ Soc., Providence, 2014.

\bibitem{Leech} R.B.\ Leech, {\em Factorization of analytic functions and 
operator inequalities},  Integral Equations and Operator Theory 
\textbf{78} (2014) no.\ 1, 71--73.

\bibitem{MCT} S.~McCullough and T.T.~Trent, {\em Invariance subspaces and
Nevanlinna-Pick kernels}, J.~Funct.~Math.~Anal.~\textbf{178} (2000),
no.1, 226--249.

 

\bibitem{MS99} P.~S.~Muhly and B.~Solel, {\em Tensor algebras over 
$C^{*}$-correspondences: representations, dilations, and 
$C^{*}$-envelopes}, J.~Funct.~Anal.~\textbf{158} (1998) no.~2, 
389--457.

\bibitem{MS05} P.~S.~Muhly and B.~Solel, {\em Canonical models for 
representations of Hardy algebras}, Integral Equations Operator 
Theory \textbf{53} (2005) no.~3, 411--452. 

\bibitem{MS08}  P.~S.~Muhly and B.~Solel, {\em Schur class functions and automorphisms of Hardy algebras},
Doc.\ Math.\ \textbf{13} (2008), 365--411.

\bibitem{MS16}  P.~S.~Muhly and B.~Solel, {\em Matricial function theory and weighted shifts},
Integral Equations and Operator Theory \textbf{84} (2016) no.\ 4, 501--553.



\bibitem{muller}
V.~M\"uller, {\em Models for operators using weighted shifts}, J. Operator Theory {\bf 20}
(1988), no. 1, 3--20.

\bibitem{MV} V.~M\"uller and F.H.~Vasilescu, {\em Standard models for some
commuting multioperators},  Proc.~Amer.~Math.~Soc.~\textbf{117}
No.~4 (1993), 979--989.

\bibitem{oljfa} A.~Olofsson, {\em A characteristic operator function
for the class of $n$-hypercontractions}, J.~Funct.~Anal.~\textbf{236}
(2006),  517--545.

\bibitem{olieot}
A.~Olofsson, {\em An operator-valued Berezin transform and the class of
$n$-hypercontractions}, Integral Equations Operator Theory  {\bf 58} (2007), no. 4,
503--549.

\bibitem{olaa}
A.~Olofsson, {\em Operator-valued Bergman inner functions as transfer functions},
Algebra i Analiz  {\bf 19} (2007),  no. 4, 146--173.

\bibitem{olofw}
A.~Olofsson and A.~Wennman, {\em Operator identities for standard weighted Bergman shift and Toeplitz operators},
J. Operator Theory {\bf 70} (2013), no.2, 451--475.

\bibitem{Pau}  V.\ Paulsen, {\em Completely Bounded Maps and Operator 
Algebras}, Cambridge Studies in Advanced Mathematics \textbf{78}, 
Cambridge University Press, Cambridge, 2002.

\bibitem{PopescuNF0} G.~Popescu, \emph{Models for infinite sequences
of noncommuting operators}, Acta Sci.~Math. (Szeged) \textbf{53}
(1989), 355--368.

\bibitem{PopescuNF1} G.~Popescu, \emph{Isometric dilations for
infinite sequences of noncommuting operators},
Trans.~Amer.~Math.~Soc. \textbf{316} (1989), 523--536.

\bibitem{PopescuNF2} G.~Popescu, \emph{Characteristic functions for
infinite sequences of noncommuting operators}, J.~Operator Theory
\textbf{22} (1989) no. 1, 51--71.

\bibitem{PopescuBL} G.~Popescu, \emph{Multi-analytic operators and some
factorization theorems},  Indiana~U.~Math.~J. \textbf{38} (1989), no.3,
693--710.


\bibitem{Popescu-multi} G.~Popescu, \emph{Multi-analytic operators on Fock
spaces}, Math.~Ann. \textbf{303} (1995), 31--46.



\bibitem{PopescuJFA2006}  G.\ Popescu, {\em Free holomorphic functions on the unit ball
of $B(\cH)^n$}, J.\ Funct.\ Anal.\ \textbf{241} (2006), 268--333.

\bibitem{Popescu-var1} G.\ Popescu, \emph{Operator theory on noncommutative
varieties}, Indiana U.~Math.~J.~\textbf{55} (2006) no.~2, 389--442.

\bibitem{Popescu-var2} G.\ Popescu, \emph{Operator theory on
noncommutative varieties II}, Proc.~Amer.~Math.~Soc. \textbf{135} 
(2007) no.!~7, 2151--2164. 

 \bibitem{PopescuJFA2008} G.\ Popescu, {\em Noncommutative Berezin 
    transforms and multivariable operator model theory}, J.\ Funct.\ 
    Anal.\ \textbf{254} (2008), 1003-1057.
    
\bibitem{PopescuMemoir2010} G.\ Popescu, {\em Operator Theory on 
Noncommutative Domains}, Memoirs of the American Mathematical Society 
Number \textbf{964} (2010).

\bibitem{PopescuJFA2013} G.\ Popescu, {\em Berezin transforms on noncommutative varieties in polydomains},
J.\ Funct.\  Anal.\ \textbf{265} (2013) no.\ 10, 2500--2552.

\bibitem{PopescuTAMS2016} G.\ Popescu, {\em Berezin transforms on noncommutative polydomains}, Trans.\
Amer.\ Math.\ Soc.\ \textbf{368} (2016) no.\ 6, 4375--4416.

\bibitem{PopescuMANN} G.\ Popescu, {\em Invariant subspaces and operator model theory on noncommutative varieties},
Math.\ Ann.\ \textbf{372} (2018), no.\ 1-2, 611--650.

\bibitem{Pott} S.\ Pott, \emph{Standard models under polynomial
positivity conditions}, J.~Operator Theory \textbf{41} (1999), 365--389.

\bibitem{rosenrov}
M.~Rosenblum and J.~Rovnyak, {\em Hardy classes and operator theory},
Oxford Mathematical Monographs. The Clarendon Press, Oxford University 
Press, New York, 1985.

\bibitem{Rudin} W.\ Rudin, {\em Principles of Mathematical Analysis, Third Edition}, International Series in Pure and Applied Mathematics,
McGraw-Hill,  1976.

\bibitem{Sarason-deB} D.\ Sarason, {\em Sub-Hardy Hilbert Spaces in the Unit Disk,} University of Arkansas Lecture Notes in
Mathematical Sciences \textbf{10}, Wiley, New York, 1994.

\bibitem{sarkar}
J.~Sarkar, {\em An invariant subspace theorem and invariant subspaces of analytic reproducing kernel Hilbert 
spaces I}, J. Operator Theory {\bf 73} (2015), no. 2, 433--441.

\bibitem{sarkarII}
J.~Sarkar, {\em An invariant subspace theorem and invariant subspaces of analytic reproducing kernel Hilbert 
spaces II},  Complex Anal.\ Oper.\ Theory (2016) no.\ 10, 769--782.

\bibitem{sh1}
S.~Shimorin, {\em Wold-type decompositions and wandering subspaces for operators close to
isometries}, J. Reine Angew. Math. {\bf 531} (2001), 147--189.

\bibitem{sh2}
S.~Shimorin, {\em On Beurling-type theorems in weighted $\ell^{2}$ and Bergman spaces},
Proc. Amer. Math. Soc. {\bf 131} (2003), no. 6, 1777--1787.


\bibitem{NF}  B.~Sz.-Nagy, C.~Foias, H.~Bercovici, and L.~K\'ercy, {\em Harmonic Analysis
of Operators on Hilbert Space}, Second edition, Revised and enlarged edition, Universitext, Springer, New York,
2010.

\bibitem{Timotin} D.~Timotin, {\em Regular dilations and models for 
multicontractions}, Indiana Univ.~Math.~J.~\textbf{47} (1998) no.~2, 
671--684.

\bibitem{vuk}
D.~Vukoti\'c, {\em Linear extremal problems for Bergman spaces}, 
Exposition. Math. {\bf 14} (1996), no. 4, 313--352.

\bibitem{wold}
H.~Wold, {\em A study in analysis of stationary time series}. Almquist und Wiksell, Uppsala, 1938. 

\end{thebibliography}
\end{document}